\documentclass[oneside,american]{amsbook}
\usepackage[T1]{fontenc}
\usepackage[latin9]{inputenc}
\usepackage{babel}
\usepackage{mathrsfs}
\usepackage{amsthm}
\usepackage{amsbsy}
\usepackage{amssymb}
\usepackage{esint}
\usepackage[all]{xy}
\usepackage{nomencl}
\providecommand{\printnomenclature}{\printglossary}
\providecommand{\makenomenclature}{\makeglossary}
\makenomenclature
\usepackage[unicode=true,pdfusetitle,
 bookmarks=true,bookmarksnumbered=false,bookmarksopen=false,
 breaklinks=false,pdfborder={0 0 1},backref=false,colorlinks=false]
 {hyperref}

\makeatletter

\providecommand{\tabularnewline}{\\}

\numberwithin{section}{chapter}
\numberwithin{equation}{section}
\numberwithin{figure}{section}
\theoremstyle{plain}
\newtheorem{thm}{\protect\theoremname}
  \theoremstyle{remark}
  \newtheorem{rem}[thm]{\protect\remarkname}
  \theoremstyle{plain}
  \newtheorem{prop}[thm]{\protect\propositionname}
  \theoremstyle{plain}
  \newtheorem{lem}[thm]{\protect\lemmaname}
  \theoremstyle{definition}
  \newtheorem{defn}[thm]{\protect\definitionname}
  \theoremstyle{plain}
  \newtheorem{cor}[thm]{\protect\corollaryname}
\newenvironment{lyxlist}[1]
{\begin{list}{}
{\settowidth{\labelwidth}{#1}
 \setlength{\leftmargin}{\labelwidth}
 \addtolength{\leftmargin}{\labelsep}
 }}
{\end{list}}
  \theoremstyle{plain}
  \newtheorem*{thm*}{\protect\theoremname}
  \theoremstyle{definition}
  \newtheorem{example}[thm]{\protect\examplename}

\newcommand{\xyR}[1]{\xydef@\xymatrixrowsep@{#1}}
\newcommand{\xyC}[1]{\xydef@\xymatrixcolsep@{#1}}

\newcommand{\Hom}{\mathrm{Hom}}
\newcommand{\Mor}{\mathrm{Mor}}

\newcommand{\End}{\mathrm{End}}
\newcommand{\Aut}{\mathrm{Aut}}
\newcommand{\Int}{\mathrm{Int}}
\newcommand{\Lie}{\mathrm{Lie}}

\newcommand{\Gal}{\mathrm{Gal}}

\newcommand{\Gr}{\mathrm{Gr}}
\newcommand{\Fi}{\mathrm{Fil}}

\newcommand{\Spec}{\mathrm{Spec}}
\newcommand{\Sym}{\mathrm{Sym}}

\newcommand{\rank}{\mathrm{rank}}

\newcommand{\Res}{\mathrm{Res}}

\newcommand{\loc}{\mathrm{loc}}

\newcommand{\can}{\mathrm{can}}

\newcommand{\Rep}{\mathsf{Rep}}
\newcommand{\Vect}{\mathsf{Vect}}

\newcommand{\Gra}{\mathsf{Gr}}
\newcommand{\Fil}{\mathsf{Fil}}

\newcommand{\Sch}{\mathsf{Sch}}
\newcommand{\Set}{\mathsf{Set}}
\newcommand{\Group}{\mathsf{Group}}
\newcommand{\Ring}{\mathsf{Ring}}
\newcommand{\QCoh}{\mathsf{QCoh}}
\newcommand{\LF}{\mathsf{LF}}
\newcommand{\Norm}{\mathsf{Norm}}

\newcommand{\HV}{\mathsf{HV}}

\makeatother

  \providecommand{\corollaryname}{Corollary}
  \providecommand{\definitionname}{Definition}
  \providecommand{\examplename}{Example}
  \providecommand{\lemmaname}{Lemma}
  \providecommand{\propositionname}{Proposition}
  \providecommand{\remarkname}{Remark}
  \providecommand{\theoremname}{Theorem}
\providecommand{\theoremname}{Theorem}

\begin{document}

\title{Filtrations and Buildings}

\author{Christophe Cornut}



\date{September, 2016}

\dedicatory{En hommage à Alexander Grothendieck}
\begin{abstract}
We construct and study a scheme theoretical version of the Tits vectorial
building, relate it to filtrations on fiber functors, and use them
to clarify various constructions pertaining to affine Bruhat-Tits
buildings, for which we also provide a Tannakian description.
\end{abstract}
\maketitle
\tableofcontents{}

\chapter{Introduction}

Spherical, affine and vectorial buildings are covered by apartments
which are respectively spheres, affine and vector spaces. They interact
with each other as spheres, affine and vector spaces do. 

The \emph{combinatorial} Tits building of a reductive group $G$ over
a field $K$ reflects the incidence relations between the parabolic
subgroups of $G$. The \emph{spherical} Tits building is the geometric
realization of the combinatorial one, obtained by gluing spheres along
common spherical sectors. Both buildings were defined by Tits in~\cite{Ti74},
they only depend upon the adjoint group of $G$, and they are not
functorial in $G$. The \emph{vectorial }Tits building was defined
by Rousseau in~\cite{Ro09}, and it does depend functorially upon
$G$. For a semi-simple group $G$, it may by defined as the cone
on the spherical building of $G$, obtained by gluing vector spaces
along common sectors. For a torus $G$, it is the group of $K$-rational
cocharacters of $G$, tensored with $\mathbb{R}$. This vectorial
Tits building is the unifying theme of our somewhat eclectic paper. 

When $K$ is a non-archimedean local field, there is also an \emph{affine
}building attached to $G$: the (extended) Bruhat-Tits building of
$G$, as defined in \cite{BrTi72,BrTi84}. It is obtained by gluing
affine spaces along common alcoves and it reflects the incidence relations
between bounded open subgroups of $G$. The combinatorial Tits building
of $G$ encodes the geometry of this affine Bruhat-Tits building at
infinity, and the combinatorial Tits buildings of the maximal reductive
quotients of the special fibers of various (non-necessarily reductive)
integral models of $G$ similarly encode the local geometry of the
affine Bruhat-Tits building of $G$. 

For classical groups, which come equipped with a standard faithful
representation, a global construction of the Bruhat-Tits building
is given in~\cite{BrTi84b,BrTi87}, generalizing the pioneering work
of Goldman and Iwahori in \cite{GoIw63}, which dealt with the case
of a general linear group and served as a model for the development
of the whole Bruhat-Tits theory. For these classical groups, the Bruhat-Tits
building of $G$ is cut out from the space of all non-archimedean
$K$-norms on the standard representation. My initial intention was
to expand such a global construction to arbitrary reductive groups,
and to clarify and canonify the aforementioned relations between the
Bruhat-Tits building of $G$ and the Tits buildings of various related
groups. 

Let me try to explain how I came to be interested in these questions.
Grassmannians, flag manifolds and their affine counterparts are essentially
orbits of $G$ acting on its related buildings, and they show up in
many branches of mathematics. In particular, integral $p$-adic Hodge
theory makes extensive use of filtrations and lattices, which are
respectively parametrized by flag manifolds and affine Grassmannians,
or by the corresponding larger ambient spaces: the vectorial Tits
building and the affine Bruhat-Tits building. Working on classical
results in $p$-adic Hodge theory \cite{CoNi16,Co13,Co15}, I eventually
realized that a natural transitive action of the former on the latter
explained many features of the interplay between filtrations and lattices
in this area of mathematics. For $G=GL_{n}$, this circle of ideas
was already more or less implicit in various works, for instance in
Laffaille \cite{Laf80} or Fontaine and Rapoport \cite{FoRa05}. But
for more general groups, sound foundations seemed to be lacking or
scattered in the existing literature: I needed a flexible and extensive
dictionary connecting building-theoretical, geometric and metric notions
and tools to properties of the relevant objects in linear algebra,
filtrations and lattices. 

For instance the aforementioned action itself was most definitely
well-known to authors working in metric geometry. Here, the cone $\mathcal{C}(\partial X)$
on the visual boundary $\partial X$ of a $CAT(0)$-metric space $(X,d)$
(see \cite{BrHa99} for these notions) acts on $X$ by non-expanding
maps as follows: an element of the cone is a pair $(\xi,\ell)$ where
$\xi\in\partial X$ is an asymptotic class of unit speed geodesic
rays in $X$ and $\ell\geq0$ is some length (or speed); it acts on
a point $x$ of $X$ by moving it at distance $\ell$ along the unique
geodesic ray in $\xi$ emanating from $x$. Taking $X$ to be the
Bruhat-Tits building of $G$ over $K$, it was also known since Bruhat
and Tits that $X$ can be equipped with a non-canonical metric $d$
which turns it into a $CAT(0)$-space, and for which the visual boundary
$\partial X$ is a realization of the spherical building of $G$ over
$K$ -- this geometric statement was initially encapsulated in the
group-theoretical notion of double Tits systems, see \cite[5.1.33]{BrTi72}
and \cite[11.10]{Ro09}. But there were no clear-cut identifications
between filtrations for $G$ and the cone $\mathcal{C}(\partial X)$,
or between lattices (and norms) for $G$ and its Bruhat-Tits building
$X$; moreover, having to rely on the artificial choice of a metric
to define the action furthermore blurred its naturality. This paper
provides proper definitions of these notions, all the required canonical
identifications, an explicit formula for the action of filtrations
(which are elements of the vectorial Tits building) on norms (which
are elements of the Bruhat-Tits building), along with many properties
of these objects and constructions.

There is a general Grothendieckian recipe to pass from $G=GL_{n}$
to more general algebraic groups: replace the natural standard representation
by the entire category of all algebraic representations. This has
become a very common method in $p$-adic Hodge theory \cite{Ko85,DaOrRa10},
and as far as filtrations are concerned, it was already implemented
for reductive groups over arbitrary fields in the foundational work
on Tannakian categories, Saavedra Rivano's thesis~\cite{SaRi72}.
For norms and lattices, it was completed more recently by Haines's
student Wilson~\cite{Wi10} for split reductive groups over complete
discrete valuation fields. This Tannakian formalism also has many
advantages: it has build-in functorialities, it works for arbitrary
affine groups over arbitrary base schemes, it provides a conceptual
framework for many algebraic constructions, and it gives rise to various
interesting representable sheaves. 

In chapter~$2$, we thus actually start with a reductive group $G$
over an arbitrary base scheme $S$. For a totally ordered commutative
group $\Gamma=(\Gamma,+,\leq)$, we introduce there our fundamental
$G$-equivariant cartesian diagram of $S$-schemes 
\[
\xymatrix{\mathbb{G}^{\Gamma}(G)\ar[r]^{\Fi}\ar[d]_{F} & \mathbb{F}^{\Gamma}(G)\ar[r]^{t}\ar[d]_{F} & \mathbb{C}^{\Gamma}(G)\ar[d]_{F}\\
\mathbb{OPP}(G)\ar[r]^{p_{1}} & \mathbb{P}(G)\ar[r]^{t} & \mathbb{O}(G)
}
\]
where $\mathbb{P}(G)$ and $\mathbb{OPP}(G)$ are respectively the
$S$-schemes of parabolic subgroups $P$ of $G$ and pairs of opposed
parabolic subgroups $(P,P')$ of $G$, $\mathbb{O}(G)$ is the $S$-scheme
of $G$-orbits in $\mathbb{P}(G)$ or $\mathbb{OPP}(G)$, $\mathbb{G}^{\Gamma}(G)=\underline{\Hom}(\mathbb{D}_{S}(\Gamma),G)$
where $\mathbb{D}_{S}(\Gamma)$ is the diagonalized multiplicative
group over $S$ with character group $\Gamma$, while $\mathbb{F}^{\Gamma}(G)$
and $\mathbb{C}^{\Gamma}(G)$ are suitable quotients of $\mathbb{G}^{\Gamma}(G)$.
The \emph{facet} morphisms $F$ are surjective and locally constant
in the étale topology on their base, the $p_{1}$ and $\Fi$ morphisms
are affine smooth surjective with geometrically connected fibers and
the \emph{type} morphisms $t$ are projective smooth surjective with
geometrically connected fibers. Since $\mathbb{O}(G)$ is finite étale
over $S$, all of the above schemes are smooth, separated and surjective
over $S$. We also equip $\mathbb{C}^{\Gamma}(G)$ and $\mathbb{O}(G)$
with $S$-\emph{monoid structures}, and the facet map $F:\mathbb{C}^{\Gamma}(G)\rightarrow\mathbb{O}(G)$
is compatible with them. We finally define two related partial orders
on the $S$-monoid $\mathbb{C}^{\Gamma}(G)$, the \emph{weak and strong
dominance orders}. 

For $\Gamma=\mathbb{Z}$, $\mathbb{D}_{S}(\Gamma)=\mathbb{G}_{m,S}$\nomenclature[G_m]{$\mathbb{G}_m$}{Multiplicative group over $\Spec (\mathbb{Z})$, $\mathbb{G}_m (R)=R^\times$.}\nomenclature[G_a]{$\mathbb{G}_a$}{Additive group over $\Spec (\mathbb{Z})$, $\mathbb{G}_a (R)=R$.}
and $\mathbb{G}^{\Gamma}(G)$ is the $S$-scheme of cocharacters of
$G$, whose conjugacy classes are classified by $\mathbb{C}^{\Gamma}(G)$.
For $\Gamma=\mathbb{R}$, $\mathbb{F}^{\Gamma}(G)$ is a scheme theoretical
version of the Tits vectorial building defined by Rousseau in \cite{Ro09}
and $\mathbb{C}^{\Gamma}(G)$ is a scheme theoretical version of a
closed Weyl chamber. More general $\Gamma$'s, for instance the valuation
groups of valuation fields of height $>1$ may also be useful in connection
with recent developements in $p$-adic geometry. 

In chapter~$3$, we show that $\mathbb{G}^{\Gamma}(G)$ and $\mathbb{F}^{\Gamma}(G)$
represent functors respectively related to $\Gamma$-graduations and
$\Gamma$-filtrations on a variety of fiber functors. The main difficulty
here is to show that the $\Gamma$-filtrations split fpqc-locally
on the base scheme. For $\Gamma=\mathbb{Z}$, this was essentially
established in the thesis of Saavedra Rivano \cite{SaRi72}, at least
when $S$ is the spectrum of a field. We strictly follow Saavedra's
proof (which he attributes to Deligne), adding a considerable amount
of details and some patch when needed. We advise our reader to read
both texts side by side, only switching to ours when he feels uncomfortable
with the necessary generalizations of Saavedra's arguments%
\footnote{For $\Gamma=\mathbb{Z}$, Ziegler recently established the fpqc-splitting
of $\mathbb{Z}$-filtrations on fiber functors on arbitrary Tannakian
categories \cite{Zi12}, thereby proving a conjecture which was left
open after Saavedra's thesis. In particular, the $\mathbb{Z}$-filtrations
we consider have fpqc-splittings even when $G$ is not reductive,
but defined over a field. In the reductive case, the final arguments
in Ziegler's proof simplify those of Saavedra's, but rely more on
the Saavedra-Deligne theorem that all fiber functors on Tannakian
categories are fpqc-locally isomorphic~\cite{De90}. According to
D.~Schäppi, it follows from his own work \cite{Sc12,Sc13} and Lurie's
note on Tannaka duality that the same result holds for any $\otimes$-functor
$\Rep^{fp}(G)(S)\rightarrow\QCoh(T)$ where: $S$ is affine, $T$
is an $S$-scheme, $G$ is affine flat over $S$, $\Rep^{fp}(G)(S)$
is the $\otimes$-category of algebraic representations of $G$ on
finitely presented $\mathcal{O}_{S}$-modules, and $G$ has the resolution
property: any finitely presented algebraic representation of $G$
is covered by another one which is locally free. It then seems likely
that Ziegler's proof could yield a common generalization of his result
($\Gamma=\mathbb{Z}$, $G$ affine over a field) and ours ($\Gamma$
and $S$ arbitrary, but $G$ reductive) on the existence of fpqc-splittings
of $\Gamma$-filtrations, using a hefty dose of the stack formalism.
We have chosen to stick to the constructive, down-to-earth original
proof of Saavedra/Deligne -- and to reductive groups as well.%
}. Various constructions of chapter $2$ have counterparts in this
Tannakian framework, which are reviewed in section~\ref{sec:ConsequencesTan}.
In particular, we show that the first line of our fundamental diagram
is functorial in the reductive group $G$ over $S$. The weak dominance
order on $\mathbb{C}^{\Gamma}(G)$ is compatible with this functoriality,
but we would like to already emphasize here that the monoid structure
is not. 

In chapter $4$, we study the sections of our schemes over a local
ring $\mathcal{O}$. We first equip $\mathbf{F}^{\Gamma}(G)=\mathbb{F}^{\Gamma}(G)(\mathcal{O})$
with a collection of \emph{apartments} $\mathbf{F}^{\Gamma}(S)$ indexed
by the maximal split subtori $S$ of $G$, and with the collection
of \emph{facets} $F^{-1}(P)$ indexed by the parabolic subgroups $P$
of $G$. The key properties of the resulting combinatorial structure
are well-known when $\mathcal{O}$ is a field and $\Gamma=\mathbb{R}$,
in which case $\mathbf{F}^{\Gamma}(G)$ is the Tits vectorial building,
but most of them carry over to this more general situation, thanks
to the wonderful last chapter of SGA3~\cite{SGA3.3r}. We describe
the behavior of these auxiliary structures under specialization (when
$\mathcal{O}$ is Henselian) or generization (when $\mathcal{O}$
is a valuation ring). When $\Gamma$ is a subring of $\mathbb{R}$,
we also attach to every finite free faithful representation $\tau$
of $G$ a partially defined \emph{scalar product} on $\mathbf{F}^{\Gamma}(G)$
and the corresponding \emph{distance }and\emph{ angle} functions,
and we study their basic properties. When $\mathcal{O}$ is a field,
a theorem of Borel and Tits \cite{BrTi72} implies that these functions
are defined everywhere, and one thus retrieves the aforementioned
non-canonical distances on the vectorial Tits building $\mathbf{F}(G)=\mathbf{F}^{\mathbb{R}}(G)$.

Over a field $K$ and with $\Gamma=\mathbb{R}$, we next define a
notion of \emph{affine $\mathbf{F}(G)$-spaces}, which interact with
the vectorial Tits building $\mathbf{F}(G)$ as affine spaces do with
their underlying vector space. Strongly influenced by the formalism
set up by Rousseau in \cite{Ro09} and Parreau in \cite{Pa10}, we
introduce various axioms that these spaces may satisfy, leading to
the more restricted class of \emph{affine $\mathbf{F}(G)$-buildings}.
Most of the abstract definitions of affine buildings that have already
been proposed \cite{KlLe97,Pa99,Ro09} also add a euclidean metric
into the picture, and involve a covering atlas of charts, which are
isometries from a given fixed euclidean affine space onto subsets
of the building (its apartments) subject to various conditions. Our
definition also involves a covering by apartments, but their affine
structure is inherited from a globally defined $G(K)$-equivariant
transitive operation $(x,\mathcal{F})\mapsto x+\mathcal{F}$ of the
vectorial building $\mathbf{F}(G)$ on the given affine $\mathbf{F}(G)$-space.
It is therefore essentially a boundary-based formalism for affine
buildings, as opposed to the more usual apartment-based formalism. 

Eventhough our affine $\mathbf{F}(G)$-buildings have no fixed metric,
they are equipped with a canonical metrizable topology and a canonical
vector valued convex distance $\mathbf{d}$, taking values in $\mathbf{C}(G)=t(\mathbf{F}(G))$.
The choice of a faithful representation $\tau$ of $G$ eventually
equips them with a convex distance $d_{\tau}=\left\Vert \mathbf{d}\right\Vert _{\tau}$
in the usual sense, for which they often become $CAT(0)$-metric spaces
as defined in \cite{BrHa99}. But the finer and canonical vectorial
distance $\mathbf{d}$ really is a key feature of our buildings: it
retrieves and generalizes many classical invariants in various set-up
(such as types of filtrations and relative positions of lattices or
quadrics), and its formal properties imply most, if not all, of the
known inequalities among these invariants. 

Of course $\mathbf{F}(G)$ is itself an affine $\mathbf{F}(G)$-building,
with a distinguished point. When $K$ is equipped with a non-trivial,
non-archimedean absolute value, we show in chapter~$6$ that the
(extended) affine building $\mathbf{B}^{e}(G)$ constructed by Bruhat
and Tits \cite{BrTi72,BrTi84} is canonically equipped with a structure
of affine $\mathbf{F}(G)$-building in our sense. This is our precise
formalization of the ``combinatorial'' assertion that the visual
boundary of the Bruhat-Tits building is a geometric realization of
the combinatorial Tits building. This being done, we may fix a base
point $\circ_{G}$ in $\mathbf{B}^{e}(G)$ and try to describe the
whole building as a quotient of $\mathbf{F}(G)$ using the surjective
map $\mathbf{F}(G)\ni\mathcal{F}\mapsto\circ_{G}+\mathcal{F}\in\mathbf{B}^{e}(G)$.
We do this in the last section, assuming that our base point $\circ_{G}$
is hyperspecial, i.e.~corresponds to a reductive group $G$ over
the valuation ring $\mathcal{O}$ of $K$, which we also assume to
be Henselian. Note that the existence of an hyperspecial point amounts
to an assumption on $G_{K}$ \cite[2.4]{Ti79}. 

More precisely, we first define a space of $K$-norms on the fiber
functor 
\[
\omega_{G}^{\circ}:\Rep^{\circ}(G)(\mathcal{O})\rightarrow\Vect(K)
\]
where $\Rep^{\circ}(G)(\mathcal{O})$ is the category of algebraic
representations of $G$ on finite free $\mathcal{O}$-modules. This
space is equipped with a $G(K)$-action, an explicit $G(K)$-equivariant
operation of $\mathbf{F}(G_{K})$ and a base point $\alpha_{G}$ fixed
by $G(\mathcal{O})$. We show that the map $\circ_{G}+\mathcal{F}\mapsto\alpha_{G}+\mathcal{F}$
is well-defined, injective, $G(K)$-equivariant and compatible with
the operations of $\mathbf{F}(G_{K})$. It thus defines an isomorphism
$\boldsymbol{\alpha}$ of affine $\mathbf{F}(G_{K})$-buildings from
$\mathbf{B}^{e}(G_{K})$ to a set $\mathbf{B}(\omega_{G}^{\circ},K)=\alpha_{G}+\mathbf{F}(G_{K})$
of $K$-norms on $\omega_{G}^{\circ}$.

This Tannakian description of the extended Bruhat-Tits building immediately
implies that the assignment $G\mapsto\mathbf{B}^{e}(G_{K})$ is functorial
in the reductive group $G$ over $\mathcal{O}$. Such a functoriality
was already established by Landvogt~\cite{La00}, with fewer assumptions
on $G_{K}$ but more assumptions on $K$. It also suggests a possible
definition of Bruhat-Tits buildings for reductive groups over valuation
rings of height greater than $1$, as well as a similar Tannakian
description of symmetric spaces (in the archimedean case, see \ref{sub:ExampSymSpac}).
It is related to previous constructions as follows. 

Our canonical isomorphism $\boldsymbol{\alpha}:\mathbf{B}^{e}(G_{K})\rightarrow\mathbf{B}(\omega_{G}^{\circ},K)$
assigns to a point $x$ in $\mathbf{B}^{e}(G_{K})$ and to any algebraic
representation $\tau$ of $G$ on a flat $\mathcal{O}$-module $V(\tau)$
a $K$-norm $\boldsymbol{\alpha}(x)(\tau)$ on $V_{K}(\tau)=V(\tau)\otimes K$.
For the adjoint representation $\tau_{\mathrm{ad}}$ of $G$ on $\mathfrak{g}=\Lie(G)$,
the adjoint-regular and regular representations $\rho_{\mathrm{adj}}$
and $\rho_{\mathrm{reg}}$ of $G$ on $\mathcal{A}(G)=\Gamma(G,\mathcal{O}_{G})$,
we obtain respectively: a $K$-norm $\alpha_{\mathrm{ad}}(x)$ on
$\mathfrak{g}_{K}=\Lie(G_{K})$ whose closed balls give the Moy-Prasad
filtration of $x$ on $\mathfrak{g}_{K}$ \cite{MoPr94}, the $K$-norm
$\alpha_{\mathrm{adj}}(x)$ in $G_{K}^{\mathrm{an}}$ constructed
in~\cite{ReThWe10}, and an embedding $x\mapsto\alpha_{\mathrm{reg}}(x)$
of the extended Bruhat-Tits building in the analytic Berkovich space
$G_{K}^{\mathrm{an}}$ attached to $G_{K}$. Our isomorphism $\boldsymbol{\alpha}$
also induces an explicit $G(\mathcal{O})$-equivariant identification
between the ``tangent space'' of $\mathbf{B}^{e}(G_{K})$ at $\circ_{G}$
and the vectorial Tits building $\mathbf{F}(G_{k})$ of the special
fiber $G_{k}$ of $G$ over the residue field $k$ of $\mathcal{O}$,
as expected from \cite[4.6.35-45]{BrTi84}. 

In Wilson's Tannakian formalism for Bruhat-Tits buildings \cite{Wi10},
alcoves and their parahorics played the leading role. His Moy-Prasad
filtrations are the lattice chains of closed balls of our norms. We
owe to his work the essential shape of our formalism, if not the very
idea that such a formalism was indeed possible: we were first naively
looking for a base-point free description of the Bruhat-Tits buildings.
His point of view is more adapted to the study of the simplicial structure
of these buildings, but only covers split groups over discrete valuation
rings. Our approach covers unramified groups over fields equipped
with a Henselian absolute value. It lacks an intrinsic description
of $(1)$ the equivalence relation on $\mathbf{F}(G_{K})$ defined
by $\mathcal{F}\sim\mathcal{F}'\iff\circ_{G}+\mathcal{F}=\circ_{G}+\mathcal{F}'$,
and of $(2)$ the image $\mathbf{B}(\omega_{G}^{\circ},K)=\alpha_{G}+\mathbf{F}(G_{K})$
of $\boldsymbol{\alpha}$ in the larger space of all $K$-norms on
the fiber functor $\omega_{G}^{\circ}$. \nomenclature[e]{$e$}{$2.71828182846...$}\nomenclature[pi]{$\pi$}{$3.14159265359...$}

Finally, we would like to mention that some of our results should
extend to more general fiber functors, using the Schäppi/Lurie generalization
of Deligne's theorem as mentioned in the previous footnote. 

~

\thanks{This work grew out of a question by J-F.~Dat and many discussions
with D.~Mauger on buildings and cocharacters. I am very grateful
to G.~Rousseau and A.~Parreau, who always had answers to my numerous
questions. Apart from the emphasis on the boundary, most of the definitions
and results of chapter~$5$ are either taken from his survey~\cite{Ro09}
or from her preprint~\cite{Pa10}. P.~Deligne kindly provided the
patch at the very end of the proof of the splitting theorem, dealing
with groups of type $G2$ in characteristic $2$, and M.~Hils the
proof of lemma~\ref{lem:ExistGoodExtensionValued}.}\maketitle

\chapter{The group theoretical formalism\label{Chapter:GroupThForm}}

For a reductive group scheme $G$ over an arbitrary base scheme $S$,
we will define and study a cartesian diagram of smooth and separated
schemes over $S$, 
\[
\xyC{2pc}\xymatrix{\mathbb{G}^{\Gamma}(G)\ar@{->>}[r]^{\Fi}\ar[d]_{F} & \mathbb{F}^{\Gamma}(G)\ar@{->>}[r]^{t}\ar[d]_{F} & \mathbb{C}^{\Gamma}(G)\ar[d]_{F}\\
\mathbb{OPP}(G)\ar@{->>}[r]^{p_{1}} & \mathbb{P}(G)\ar@{->>}[r]^{t} & \mathbb{O}(G)
}
\]
Our main background reference for this chapter is SGA3 \cite{SGA3.1r,SGA3.2,SGA3.3r}.

\section{$\Gamma$-graduations on smooth affine groups}
\begin{thm}
\label{thm:RepGr}Let $H$ and $G$ be group schemes over a base scheme
$S$, with $H$ of multiplicative type and quasi-isotrivial, $G$
smooth and affine. Then the functor 
\[
\underline{\Hom}_{S-\Group}(H,G):\left(\Sch/S\right)^{\circ}\rightarrow\Set,\qquad T\mapsto\Hom_{T-\Group}(H_{T},G_{T})
\]
is representable by a smooth and separated scheme over $S$.\nomenclature[Hom]{$\underline{\Hom}$}{Sheafified version of Hom.}\nomenclature[Group_S]{$S-\Group$}{Category of group schemes over $S$.}\nomenclature[Sch]{$\Sch$}{Category of schemes.}\nomenclature[Sch_S]{$\Sch /S$}{Category of schemes over $S$.}\nomenclature[Set]{$\Set$}{Category of sets.}\nomenclature[X_T]{$X_T$}{Pull-back or base change of some $X$ over $S$ through $T \rightarrow S$.}\nomenclature[Co]{$\mathsf{C} ^\circ $}{Category opposed to $\mathsf{C}$.}\end{thm}
\begin{rem}
When $H$ is of finite type, it is quasi-isotrivial by \cite[X 4.5]{SGA3.2}.
The theorem is then due to Grothendieck, see \cite[XI 4.2]{SGA3.2}.
The proof given there relies on the density theorem of \cite[IX 4.7]{SGA3.2},
definitely a special feature of finite type multiplicative groups.
When $H$ is trivial, we may still reduce the proof of the above theorem
to the finite type case, as explained in remark~\ref{Rem:AltProof}
below. For the general case, we have to find another road through
SGA3, passing through \cite[X 5.6]{SGA3.2} which has no finite type
assumption on $H$ but requires $H$ and $G$ to be of multiplicative
type and quasi-isotrivial:\end{rem}
\begin{prop}
\label{prop:RepGspCaseTorus}Let $H$ and $G$ be group schemes of
multiplicative type over $S$, with $H$ quasi-isotrivial and $G$
of finite type. Then $\underline{\Hom}_{S-\Group}(H,G)$ is representable
by a quasi-isotrivial twisted constant group scheme $X$ over $S$. \end{prop}
\begin{proof}
This is \cite[X 5.6]{SGA3.2}, since $G$ is also quasi-isotrivial
by \cite[X 4.5]{SGA3.2}.\end{proof}
\begin{lem}
\label{lem:StructQuasIsoTwistedSch}Let $X$ be a quasi-isotrivial
twisted constant scheme over $S$. Then $X$ is separated and étale
over $S$, satisfies the valuative criterion of properness, and:
\begin{enumerate}
\item If $S$ is irreducible and geometrically unibranch with generic point
$\eta$, then 
\[
X={\textstyle \coprod_{\lambda\in X_{\eta}}}X(\lambda)\quad\mbox{with}\quad X(\lambda)=\overline{\{\lambda\}}\mbox{ open and closed in }X,
\]
each $X(\lambda)$ is a connected finite étale cover of $S$ and $\Gamma(X/S)=\Gamma(X_{\eta}/\eta)$. 
\item If $S$ is local henselian with closed point $s$, then 
\[
X={\textstyle \coprod}_{x\in X_{s}}X(x)\quad\mbox{with}\quad X(x)=\Spec\,\mathcal{O}_{X,x}\mbox{ open and closed in }X,
\]
each $X(x)$ is a connected finite étale cover of $S$, and $\Gamma(X/S)=\Gamma(X_{s}/s)$.\nomenclature[Gamma(X/S)]{$\Gamma (X/S)$}{Sections of a morphism $X \rightarrow S$.}
\end{enumerate}
\end{lem}
\begin{proof}
The morphism $X\rightarrow S$ is separated by \cite[2.7.1]{EGA4.2}
and étale by \cite[17.7.3]{EGA4.4}. Since valuation rings are normal
integral domains, thus irreducible and geometrically unibranch, it
remains to establish $(1)$ and $(2)$. 

Suppose first that $S$ is irreducible and geometrically unibranch
with generic point $\eta$. Then by \cite[18.10.7]{EGA4.4} applied
to $X\rightarrow S$, 
\[
X={\textstyle \coprod_{\lambda\in X_{\eta}}}X(\lambda)\quad\mbox{with}\quad X(\lambda)=\overline{\{\lambda\}}\mbox{ open and closed in }X,
\]
thus $X(\lambda)$ is already étale over $S$. Fix an étale covering
$\{S_{i}\rightarrow S\}$ trivializing $X$, so that $X\times_{S}S_{i}=Q_{i,S_{i}}$
for some set $Q_{i}$. Using \cite[18.10.7]{EGA4.4} again, we may
assume that each $S_{i}$ is connected, in which case we obtain decompositions
\[
Q_{i}={\textstyle \coprod_{\lambda\in X_{\eta}}}Q_{i}(\lambda)\quad\mbox{with}\quad X(\lambda)\times_{S}S_{i}=Q_{i}(\lambda)_{S_{i}}.
\]
Since the generic fiber $\lambda\rightarrow\eta$ of $X(\lambda)\rightarrow S$
is finite of degree $n(\lambda)=[k(\lambda):k(\eta)]$, each $Q_{i}(\lambda)$
is a finite subset of $Q_{i}$ of order $n(\lambda)$, therefore $X(\lambda)\times_{S}S_{i}$
is finite over $S_{i}$ and $X(\lambda)$ is finite over $S$ by \cite[2.7.1]{EGA4.2}.
Being finite and étale over the connected $S$, $X(\lambda)$ is a
finite étale cover of $S$. Being irreducible, it is also connected.
By \cite[17.4.9]{EGA4.4}, the map which sends a section $g$ of $X\rightarrow S$
to its image $g(S)$ identifies $\Gamma(X/S)$ with the set of connected
components $X(\lambda)$ of $X$ for which $X(\lambda)\rightarrow S$
is an isomorphism, i.e.~such that $n(\lambda)=1$. Therefore $\Gamma(X/S)=\Gamma(X_{\eta}/\eta)$. 

Suppose next that $S$ is local henselian with closed point $s$.
Since $X\rightarrow S$ is quasi-finite at every $x\in X_{s}$ by
\cite[17.6.1]{EGA4.4}, it follows from \cite[18.5.11.c]{EGA4.4}
that 
\[
X\supset X'={\textstyle \coprod_{x\in X_{s}}}X(x)\quad\mbox{with}\quad X(x)=\Spec\,\mathcal{O}_{X,x}\,\mbox{open and closed in }X,
\]
and $X(x)$ is finite and étale over $S$. By assumption, there is
a surjective étale morphism $S_{0}\rightarrow S$ trivializing $X$,
so that $X\times_{S}S_{0}=Q_{S_{0}}$ for some set $Q$. Using \cite[18.5.11.c]{EGA4.4}
again, we may assume that $S_{0}$ is a local scheme, finite and étale
over $S$, say with closed point $s_{0}$ lying above $s$. Since
$X'\times_{S}S_{0}$ is open in $X\times_{S}S_{0}$ and contains its
special fiber $X_{s_{0}}$, we have $X'\times_{S}S_{0}=X\times_{S}S_{0}$,
thus actually $X'=X$ by \cite[2.7.1]{EGA4.2}. Finally $\Gamma(X/S)=\Gamma(X_{s}/s)$
by \cite[18.5.12]{EGA4.4}. \end{proof}
\begin{lem}
\label{lem:ImageMorphTor}Let $f:H\rightarrow G$ be a morphism of
group schemes over $S$, with $H$ of multiplicative type and $G$
separated of finite presentation. Then there is a unique closed multiplicative
subgroup $Q$ of $G$ such that $f$ factors through a faithfully
flat morphism $f':H\rightarrow Q$. Moreover $f'$ is also uniquely
determined by $f$.\end{lem}
\begin{proof}
Everything being local for the fpqc topology, we may assume that $S$
is affine and $H=\mathbb{D}_{S}(M)$\nomenclature[DSM]{$\mathbb{D}_S(M)$}{Diagonalizable group scheme over $S$ with character group $M$.}
for some abstract commutative group $M$. Then $M=\underrightarrow{\lim}M'$
where $M'$ runs through the filtered set $\mathcal{F}(M)$ of finitely
generated subgroups of $M$, thus also $\mathbb{D}_{S}(M)=\underleftarrow{\lim}\mathbb{D}_{S}(M')$.
Since $\mathbb{D}_{S}(M')$ is affine for all $M'$ and $G\rightarrow S$
is locally of finite presentation, it follows from \cite[8.13.1]{EGA4.3}
that $f$ factors through $f_{1}:\mathbb{D}_{S}(M')\rightarrow G$
for some $M'\in\mathcal{F}(M)$. Applying \cite[IX 6.8]{SGA3.2} to
$f_{1}$ yields a closed multiplicative subgroup $Q$ of $G$ such
that $f_{1}$ factors through a faithfully flat (and affine) morphism
$f'_{1}:\mathbb{D}_{S}(M')\rightarrow Q$, whose composite with the
faithfully flat (and affine) morphism $\mathbb{D}_{S}(M)\rightarrow\mathbb{D}_{S}(M')$
is the desired factorization. Since $Q$ is then also the image of
$f$ in the category of fpqc sheaves on $\Sch/S$, it is already unique
as a subsheaf of $G$. Since $Q\rightarrow G$ is a monomorphism,
also $f'$ is unique.\end{proof}
\begin{defn}
We call $Q$ the image of $f$ and denote it by $Q=\mathrm{im}(f)$. \end{defn}
\begin{lem}
\label{lem:CentrSmooth}Let $f:H\rightarrow G$ be a morphism of group
schemes over $S$, with $H$ of multiplicative type and $G$ smooth
and affine. Then the centralizer of $f$ is equal to the centralizer
of its image, and is representable by a closed smooth subgroup of
$G$. \end{lem}
\begin{proof}
Let $f=\iota\circ f'$ be the factorization of the previous lemma.
Since $f'$ is faithfully flat (and quasi-compact, being a morphism
between affine $S$-schemes, therefore even affine), it is an epimorphism
in the category of schemes. It then follows from the definitions in
\cite[VIB §6]{SGA3.1r} that the centralizers of $f$, $\iota$ and
$\mathrm{im}(f)$ are equal. By \cite[XI 5.3]{SGA3.2}, the centralizer
of $\iota$ is a closed smooth subgroup of $G$. \end{proof}
\begin{lem}
\label{lem:LocusSurj}Let $f:H\rightarrow Q$ be a morphism of group
schemes of multiplicative type over $S$, with $Q$ of finite type.
Define $U=\{s\in S:f_{s}\mbox{ is faithfully flat}\}$. Then $U$
is open and closed in $S$ and $f_{U}:H_{U}\rightarrow Q_{U}$ is
faithfully flat. \end{lem}
\begin{proof}
Let $I$ be the image of $f$. Then $U$ is the set of points $s\in S$
where $I_{s}=Q_{s}$. Now apply \cite[IX 2.9]{SGA3.2} to $I\hookrightarrow Q$.
\end{proof}
\noindent We may now prove theorem~\ref{thm:RepGr}. Define presheaves
$A,B,C$ on $\Sch/S$ by 
\begin{eqnarray*}
C(S') & = & \left\{ \mbox{multiplicative subgroups \ensuremath{Q}}\mbox{ of }G_{S'}\right\} ,\\
B(S') & = & \left\{ (Q,f'):Q\in C(S')\mbox{ and }f:H_{S'}\rightarrow Q\mbox{ is a morphism}\right\} ,\\
A(S') & = & \left\{ (Q,f')\in B(S')\,\mbox{with }f'\mbox{ faithfully flat}\right\} .
\end{eqnarray*}
Then $C$ is representable, smooth and separated by \cite[XI 4.1]{SGA3.2},
$B\rightarrow C$ is relatively representable by étale and separated
morphisms by proposition~\ref{prop:RepGspCaseTorus} and lemma~\ref{lem:StructQuasIsoTwistedSch},
$A\rightarrow B$ is relatively representable by open and closed immersions
by lemma~\ref{lem:LocusSurj} and finally $A$ is isomorphic to $\underline{\Hom}_{S-\Group}(H,G)$
by lemma~\ref{lem:ImageMorphTor}, which is therefore indeed representable
by a smooth and separated scheme over $S$.
\begin{defn}
For an abstract commutative group $\Gamma=(\Gamma,+)$ and a smooth
and affine group scheme $G$ over $S$, we set \nomenclature[G^Gamma(G)]{$\mathbb{G}^\Gamma (G)$}{Scheme of $\Gamma$-graduations on $G$, defined page \nomrefpage}$\mathbb{G}^{\Gamma}(G)=\underline{\Hom}_{S-\Group}(\mathbb{D}_{S}(\Gamma),G)$.
Thus 
\[
\mathbb{G}^{\Gamma}(G):(\Sch/S)^{\circ}\rightarrow\Set
\]
is representable by a smooth and separated scheme over $S$. \end{defn}
\begin{prop}
\label{Pro:StructImagG}Let $f:\mathbb{D}_{S}(\Gamma)\rightarrow G$
be a morphism of group schemes over $S$, with $G$ separated and
of finite presentation. Then for each $s$ in $S$, 
\[
\Gamma(s)=\left\{ \gamma\in\Gamma:\gamma\mbox{ is trivial on }\ker(f_{s})\right\} 
\]
belongs to the set $\mathcal{F}(\Gamma)$ of finitely generated subgroups
of $\Gamma$. For each $\Lambda\in\mathcal{F}(\Gamma)$,
\[
S(\Lambda)=\left\{ s\in S:\Gamma(s)=\Lambda\right\} 
\]
is open and closed in $S$, and finally 
\[
\ker(f)_{S(\Lambda)}=\mathbb{D}_{S(\Lambda)}(\Gamma/\Lambda)\quad\mbox{and}\quad\mathrm{im}(f)_{S(\Lambda)}=\mathbb{D}_{S(\Lambda)}(\Lambda).
\]
\end{prop}
\begin{proof}
We may assume that $S$ is affine and $G$ is of multiplicative type
(using lemma~\ref{lem:ImageMorphTor} for the latter). Since $\mathbb{D}_{S}(\Gamma)=\underleftarrow{\lim}\mathbb{D}_{S}(\Lambda)$,
it follows again from \cite[8.13.1]{EGA4.3} that there is some $\Lambda$
in $\mathcal{F}(\Gamma)$ such that $f$ factors through $g:\mathbb{D}_{S}(\Lambda)\rightarrow G$,
i.e.~$\mathbb{D}_{S}(\Gamma/\Lambda)\subset\ker(f)$. But then $\Gamma(s)\subset\Lambda$
for every $s\in S$, which proves the first claim. Applying now \cite[IX 2.11 (i)]{SGA3.2}
to $g$ gives a finite partition of $S$ into open and closed subsets
$S_{i}$, together with a collection of distinct subgroups $\Lambda_{i}$
of $\Lambda$ such that $\ker(g)_{S_{i}}=\mathbb{D}_{S_{i}}(\Lambda/\Lambda_{i})$
and $\mathrm{im}(g)_{S_{i}}\simeq\mathbb{D}_{S_{i}}(\Lambda_{i})$.
But then $\ker(f)_{S_{i}}=\mathbb{D}_{S_{i}}(\Gamma/\Lambda_{i})$,
$\mathrm{im}(f)_{S_{i}}\simeq\mathbb{D}_{S_{i}}(\Lambda_{i})$ and
$S_{i}=S(\Lambda_{i})$, which proves the remaining claims.\end{proof}
\begin{cor}
\label{cor:imageofcocharIssubtorus}If $\Gamma$ is torsion free,
$\mathrm{im}(f)$ is a locally trivial subtorus of $G$. \end{cor}
\begin{rem}
\label{Rem:AltProof}The above proposition suggests another proof
of theorem~\ref{thm:RepGr} when $H=\mathbb{D}_{S}(\Gamma)$. It
shows indeed that the Zariski sheaf $\mathbb{G}^{\Gamma}(G)$ is the
disjoint union of relatively open and closed subsheaves $\mathbb{G}^{\Gamma}(G)(\Lambda)$,
indexed by $\Lambda\in\mathcal{F}(\Gamma)$. Moreover, $\mathbb{G}^{\Gamma}(G)(\Lambda)$
is isomorphic to the subsheaf $\mathbb{G}^{\Lambda}(G)(\Lambda)$
of $\mathbb{G}^{\Lambda}(G)$, which is representable by a smooth
and separated scheme over $S$ by \cite[XI 4.2]{SGA3.2}. 
\end{rem}

\section{\label{sub:FiltrationsGroup}$\Gamma$-filtrations on reductive groups}

Let $S$ be a scheme, $G$ a reductive group over $S$, $\mathfrak{g}=\Lie(G)$\nomenclature[Lie(G)]{$\Lie (G)$}{Lie algebra of $G$.}
its Lie algebra. Let $\Gamma=(\Gamma,+,\leq)$ be a non-trivial totally
ordered commutative group.

\subsection{~}

Recall from \cite[XXVI 3.5 ]{SGA3.3r} that the sheaf\nomenclature[P(G)]{$\mathbb{P}(G)$}{Scheme of parabolic subgroups of $G$, defined page  \nomrefpage}
\[
\mathbb{P}(G):(\Sch/S)^{\circ}\rightarrow\Set
\]
whose sections over an $S$-scheme $T$ are given by 
\[
\mathbb{P}(G)(T)=\left\{ \mbox{parabolic subgroups }P\mbox{ of }G_{T}\right\} 
\]
is representable, smooth and projective over $S$, with Stein factorization
\[
\mathbb{P}(G)\stackrel{t}{\longrightarrow}\mathbb{O}(G)\rightarrow S
\]
where \nomenclature[O(G)]{$\mathbb{O}(G)$}{Scheme of types of parabolic subgroups of $G$, defined page  \nomrefpage}$\mathbb{O}(G)$
is the $S$-scheme of open and closed subschemes of the Dynkin $S$-scheme
\nomenclature[DYN(G)]{$\mathbb{DYN}(G)$}{Dynkin scheme of $G$, defined page  \nomrefpage}$\mathbb{DYN}(G)$
of the reductive group $G/S$, see \cite[XXIV 3.3]{SGA3.3r}. Both
$\mathbb{DYN}(G)$ and $\mathbb{O}(G)$ are twisted constant finite
schemes over $S$, thus finite étale over $S$ by \cite[2.7.1.xv]{EGA4.2}
and \cite[17.7.3]{EGA4.4}, and $\mathbb{O}(G)$ is actually a finite
étale cover of $S$. The morphism \nomenclature[t]{$t$}{Type morphism $t:\mathbb{P}(G) \rightarrow \mathbb{O}(G)$, defined page  \nomrefpage}$t$
is smooth, projective, with non-empty geometrically connected fibers;
it classifies the parabolic subgroups of $G$ in the following sense:
two parabolic subgroups $P_{1}$ and $P_{2}$ of $G$ are conjugated
locally in the fpqc topology on $S$ if and only if $t(P_{1})=t(P_{2})$.

\subsection{~}

For a parabolic subgroup $P$ of $G$ with unipotent radical $U$,
we denote by \nomenclature[Ro(P)]{$\overline{R}(P)$}{Radical of $P/U$, where $U$ is the unipotent radical of $P$.}$\overline{R}(P)$
the radical of $P/U$ \cite[XXII 4.3.6]{SGA3.3r}. For the universal
parabolic subgroup \nomenclature[P_u]{$P_u$}{Universal parabolic subgroup of $G_{\mathbb{P}(G)}$.}$P_{u}$
of $G_{\mathbb{P}(G)}$, we obtain a $\mathbb{P}(G)$-torus \nomenclature[R_P(G)]{$R_{\mathbb{P}(G)}$}{Radical $\overline{R}(P_u)$ of $P_u/U_u$, a torus over $\mathbb{P}(G)$.}$R_{\mathbb{P}(G)}=\overline{R}(P_{u})$.
We claim that it descends canonically to an $\mathbb{O}(G)$-torus
\nomenclature[R_O(G)]{$R_{\mathbb{O}(G)}$}{A torus over $\mathbb{O}(G)$ defined on  page  \nomrefpage}$R_{\mathbb{O}(G)}$
over $\mathbb{O}(G)$. Since $t$ is faithfully flat and quasi-compact,
it is a morphism of effective descent for affine group schemes by
\cite[VIII 2.1]{SGA1r}, thus also for tori by definition \cite[IX 1.3]{SGA3.2}.
Our claim now follows from:
\begin{lem}
\label{lem:CanoDescData}There exists a canonical descent datum on
$R_{\mathbb{P}(G)}$ with respect to $t$.\end{lem}
\begin{proof}
We have to show that for any $T\rightarrow S$ and any pair of parabolic
subgroups $P_{1}$ and $P_{2}$ of $G_{T}$ such that $t(P_{1})=t(P_{2})$,
there exists a canonical isomorphism $\overline{R}(P_{1})\simeq\overline{R}(P_{2})$.
Let $M_{i}=P_{i}/U_{i}$ be the maximal reductive quotient of $P_{i}$,
so that $R_{i}=\overline{R}(P_{i})$ is the radical of $M_{i}$. We
may assume that $T=S$ and, by a descent argument, that $P_{2}=\Int(g)(P_{1})$
for some $g\in G(S)$. Then $\Int(g)$ induces isomorphisms $P_{1}\rightarrow P_{2}$,
$M_{1}\rightarrow M_{2}$ and $R_{1}\rightarrow R_{2}$. Since $g$
is well-defined up to right multiplication by an element of $P_{1}(S)$
thanks to \cite[XXVI 1.2]{SGA3.3r}, $M_{1}\rightarrow M_{2}$ is
well-defined up to an inner automorphism of $M_{1}$ and $R_{1}\rightarrow R_{2}$
does not depend upon any choice: this is our canonical isomorphism.
\end{proof}

\subsection{~}

By \cite[XXVI 4.3.4 and 4.3.5]{SGA3.3r}, the formula\nomenclature[OPP(G)]{$\mathbb{OPP}(G)$}{Scheme of pairs of opposed parabolic subgroups of $G$,  page \nomrefpage}
\[
\mathbb{OPP}(G)(T)=\left\{ (P_{1},P_{2})\mbox{ pair of opposed parabolic subgroups of }G_{T}\right\} 
\]
defines an open subscheme $\mathbb{OPP}(G)$ of $\mathbb{P}(G)^{2}$
and the two projections
\[
p_{1},p_{2}:\mathbb{OPP}(G)\rightarrow\mathbb{P}(G)
\]
are isomorphic $U_{u}$-torsors, thus affine smooth surjective morphisms
with geometrically connected fibers. Here \nomenclature[U_u]{$U_u$}{Unipotent radical of  $P_u$.}$U_{u}$
is the unipotent radical of the universal parabolic subgroup $P_{u}$
of $G_{\mathbb{P}(G)}$, it acts by conjugation on the fibers, and
the isomorphism is the involution \nomenclature[iota]{$\iota$}{Opposition involution of  $\mathbb{OPP}(G)$ or $\mathbb{O}(G)$, defined page \nomrefpage}$\iota(P_{1},P_{2})=(P_{2},P_{1})$
of the $S$-scheme $\mathbb{OPP}(G)$. We denote by $(P_{u}^{1},P_{u}^{2})=(p_{1}^{\ast}P_{u},p_{2}^{\ast}P_{u})$
the universal pair of opposed parabolic subgroups of $G_{\mathbb{OPP}(G)}$,
by $U_{u}^{i}=p_{i}^{\ast}U_{u}$ the unipotent radical of $P_{u}^{i}$,
and by \nomenclature[R_OPP(G)]{$R_{\mathbb{OPP}(G)}$}{Radical of the universal Levi subgroup of $G_{\mathbb{OPP}(G)}$.}$R_{\mathbb{OPP}(G)}$
the radical of the corresponding universal Levi subgroup $L_{u}=P_{u}^{1}\cap P_{u}^{2}$
of $G_{\mathbb{OPP}(G)}$. Thus 
\[
L_{u}\simeq P_{u}^{i}/U_{u}^{i}\simeq p_{i}^{\ast}(P_{u}/U_{u})\quad\mbox{and}\quad R_{\mathbb{OPP}(G)}\simeq p_{i}^{\ast}R_{\mathbb{P}(G)}.
\]
We also denote by $\iota$ the opposition involution on $\mathbb{O}(G)$,
see~\cite[XXVI 4.3.1]{SGA3.3r}. Thus
\[
t\circ p_{2}=t\circ p_{1}\circ\iota=\iota\circ t\circ p_{1}.
\]

\subsection{~}

The $S$-scheme $\mathbb{G}^{\Gamma}(R_{\mathbb{OPP}(G)})$ represents
the functor mapping $T\rightarrow S$ to the set of triples $(P_{1},P_{2},f)$
where $(P_{1},P_{2})$ is a pair of opposed parabolic subgroups of
$G_{T}$ with Levi subgroup $L=P_{1}\cap P_{2}$, and $f:\mathbb{D}_{T}(\Gamma)\rightarrow L$
is a central morphism. The next proposition uses the total ordering
on $\Gamma=(\Gamma,+,\leq)$ to define a section
\[
\mathbb{G}^{\Gamma}(G)\hookrightarrow\mathbb{G}^{\Gamma}(R_{\mathbb{OPP}(G)}),\qquad f\mapsto(P_{f},P_{\iota f},f)
\]
of the obvious forgetful morphism of $S$-schemes 
\[
\mathbb{G}^{\Gamma}(R_{\mathbb{OPP}(G)})\rightarrow\mathbb{G}^{\Gamma}(G),\qquad(P_{1},P_{2},f)\mapsto f.
\]

\begin{prop}
\label{prop:DefPGroupCase}Let $f:\mathbb{D}_{S}(\Gamma)\rightarrow G$
be a morphism and write $\mathfrak{g}=\oplus_{\gamma}\mathfrak{g}_{\gamma}$
for the corresponding weight decomposition of $\mathrm{ad}\circ f:\mathbb{D}_{S}(\Gamma)\rightarrow GL_{S}(\mathfrak{g})$.
There exists a unique parabolic subgroup $P_{f}$ of $G$ containing
the centralizer $L_{f}$ of $f$ such that 
\[
\Lie(P_{f})=\oplus_{\gamma\geq0}\mathfrak{g}_{\gamma}.
\]
Moreover $L_{f}$ is a Levi subgroup of $P_{f}$, thus $P_{f}=U_{f}\rtimes L_{f}$
where $U_{f}$ is the unipotent radical of $P_{f}$. For $\iota f=f^{-1}$,
$P_{\iota f}$ is opposed to $P_{f}$, $L_{f}=P_{f}\cap P_{\iota f}$
and 
\[
\begin{array}{c}
\Lie(P_{f})=\oplus_{\gamma\geq0}\mathfrak{g}_{\gamma}\\
\Lie(U_{f})=\oplus_{\gamma>0}\mathfrak{g}_{\gamma}
\end{array}\quad\begin{array}{c}
\Lie(P_{\iota f})=\oplus_{\gamma\leq0}\mathfrak{g}_{\gamma}\\
\Lie(U_{\iota f})=\oplus_{\gamma<0}\mathfrak{g}_{\gamma}
\end{array}\quad\mbox{and}\quad\Lie(L_{f})=\mathfrak{g}_{0}.
\]
\end{prop}
\begin{proof}
Let $Q$ be the image of $f$. Then $L_{f}$ is the centralizer of
$Q$ by lemma~\ref{lem:CentrSmooth} and $Q$ is a locally trivial
subtorus of $G$ by proposition~\ref{Pro:StructImagG} (since $\Gamma$
is torsion free). We may assume that $Q$ is trivial, i.e.~$Q\simeq\mathbb{D}_{S}(\Lambda)$
for some finitely generated subgroup $\Lambda$ of $\Gamma$. The
proposition then follows from~\cite[XXVI 6.1]{SGA3.3r}.\end{proof}
\begin{prop}
\label{prop:GtoG(ROPP)isClopen}The morphism $\mathbb{G}^{\Gamma}(G)\rightarrow\mathbb{G}^{\Gamma}(R_{\mathbb{OPP}(G)})$
is an open and closed immersion, and $\mathbb{G}^{\Gamma}(G)\rightarrow\mathbb{OPP}(G)$
is a quasi-isotrivial twisted constant morphism. \end{prop}
\begin{proof}
The second assertion follows from the first one by Grothendieck's
proposition~\ref{prop:RepGspCaseTorus}. Given a section $(P_{1},P_{2},f)$
of $\mathbb{G}^{\Gamma}(R_{\mathbb{OPP}(G)})$ over some $S$-scheme
$T$, we have to show that the condition $(P_{1},P_{2})=(P_{f},P_{\iota f})$
is representable by an open and closed subscheme of $T$. It is plainly
representable by the inverse image of the diagonal of $\mathbb{OPP}(G)$
under the $S$-morphism $T\rightarrow\mathbb{OPP}(G)^{2}$ defined
by our two pairs $(P_{1},P_{2})$ and $(P_{f},P_{\iota f})$, which
is a closed subscheme of $T$ since $\mathbb{OPP}(G)$ is separated
over $S$. On the other hand, since the Levi subgroup $L=P_{1}\cap P_{2}$
of $G$ is contained in $L_{f}=P_{f}\cap P_{\iota f}$, our condition
$(P_{1},P_{2})=(P_{f},P_{\iota f})$ is equivalent to $(\mathfrak{p}_{1},\mathfrak{p}_{2})=(\oplus_{\gamma\geq0}\mathfrak{g}_{\gamma},\oplus_{\gamma\leq0}\mathfrak{g}_{\gamma})$
where $\mathfrak{p}_{i}=\Lie(P_{i})$: this last claim is local in
the fpqc topology on $T$, we may thus assume that $L$ contains a
maximal torus of $G$ and then apply \cite[XXII 5.3.5]{SGA3.3r}.
Now write $\mathfrak{u}_{i}=\oplus\mathfrak{u}_{i,\gamma}$ for the
weight decomposition of the Lie algebra of the unipotent radical of
$P_{i}$, and set $\mathfrak{u}_{i}^{\pm}=\oplus_{\pm\gamma\geq0}\mathfrak{u}_{i,\gamma}$.
Then our Lie algebra condition is equivalent to the vanishing of the
locally free sheaf $\mathfrak{\mathfrak{u}}_{1}^{-}\oplus\mathfrak{u}_{2}^{+}$,
and it is therefore representable by the open complement of its support.\end{proof}
\begin{rem}
This gives yet another proof of theorem~\ref{thm:RepGr} (using Grothendieck's
proposition~\ref{prop:RepGspCaseTorus}), when $G$ is reductive
and $\Gamma$ torsion free (using \cite{Le42} to construct a total
order $\leq$ on $\Gamma$).
\end{rem}

\subsection{~\label{sub:DefFondDiagPrel}}

The cartesian diagram (in the fibered category of tori over schemes):
\[
\xyC{2pc}\xymatrix{R_{\mathbb{OPP}(G)}\ar@{->>}[r]^{p_{1}}\ar[d] & R_{\mathbb{P}(G)}\ar@{->>}[r]^{t}\ar[d] & R_{\mathbb{O}(G)}\ar[d]\\
\mathbb{OPP}(G)\ar@{->>}[r]^{p_{1}} & \mathbb{P}(G)\ar@{->>}[r]^{t} & \mathbb{O}(G)
}
\]
induces an analogous cartesian diagram (in the fibered category of
quasi-isotrivial twisted constant group schemes over schemes):
\[
\xyC{2pc}\xymatrix{\mathbb{G}^{\Gamma}(R_{\mathbb{OPP}(G)})\ar@{->>}[r]^{\Fi}\ar[d]_{F} & \mathbb{G}^{\Gamma}(R_{\mathbb{P}(G)})\ar@{->>}[r]^{t}\ar[d]_{F} & \mathbb{G}^{\Gamma}(R_{\mathbb{O}(G)})\ar[d]_{F}\\
\mathbb{OPP}(G)\ar@{->>}[r]^{p_{1}} & \mathbb{P}(G)\ar@{->>}[r]^{t} & \mathbb{O}(G)
}
\]
which is given on $T$-valued points by the following formulas: 
\[
\xyC{2pc}\xymatrix{(P_{1},P_{2},f)\ar@{|->}[r]^{\Fi}\ar@{|->}[d]_{F} & (P_{1},\overline{f})\ar@{|->}[r]^{t}\ar@{|->}[d]_{F} & (t(P_{1}),\overline{f})\ar@{|->}[d]_{F}\\
(P_{1},P_{2})\ar@{|->}[r]^{p_{1}} & P_{1}\ar@{|->}[r]^{t} & t(P_{1})
}
\]
Here $\overline{f}:\mathbb{D}_{T}(\Gamma)\rightarrow\overline{R}(P_{1})$
is defined by the diagram 
\[
\xymatrix{\mathbb{D}_{T}(\Gamma)\ar[r]^{f}\ar[dr]_{\overline{f}} & R(L)\ar@{^{(}->}[r]\ar[d]^{\simeq} & L\ar[d]^{\simeq}\\
 & \overline{R}(P_{1})\ar@{^{(}->}[r] & P_{1}/U_{1}
}
\]
where $L=P_{1}\cap P_{2}$ and $U_{1}$ is the unipotent radical of
$P_{1}$.
\begin{lem}
The open and closed subscheme $\mathbb{G}^{\Gamma}(G)$ of $\mathbb{G}^{\Gamma}(R_{\mathbb{OPP}(G)})$
is saturated with respect to $\mathbb{G}^{\Gamma}(R_{\mathbb{OPP}(G)})\rightarrow\mathbb{G}^{\Gamma}(R_{\mathbb{P}(G)})$
and $\mathbb{G}^{\Gamma}(R_{\mathbb{OPP}(G)})\rightarrow\mathbb{G}^{\Gamma}(R_{\mathbb{O}(G)})$. \end{lem}
\begin{proof}
It is sufficient to establish that it is saturated with respect to
the second map. We have to show: for an $S$-scheme $T$, a morphism
$f:\mathbb{D}_{T}(\Gamma)\rightarrow G_{T}$, a pair of opposed parabolic
subgroups $(P_{1},P_{2})$ of $G_{T}$ with Levi $L=P_{1}\cap P_{2}$,
and a central morphism $h:\mathbb{D}_{T}(\Gamma)\rightarrow L$, if
$(P_{f},P_{\iota f},f)$ and $(P_{1},P_{2},h)$ have the same image
in $\mathbb{G}^{\Gamma}(R_{\mathbb{O}(G)})(T)$, then $(P_{1},P_{2})=(P_{h},P_{\iota h})$.
This is local in the fpqc topology on $T$. Since $t(P_{f})=t(P_{1})$
by assumption, we may assume that there is a $g\in G(T)$ such that
\nomenclature[Int(g)]{$\Int(g)$}{Inner automorphism $h \mapsto g h g^{-1}$.}$\Int(g)(P_{f},P_{\iota f})=(P_{1},P_{2})$
by \cite[4.3.4 iii]{SGA3.3r}. But then also $\Int(g)\circ f=h$ (by
assumption), thus $(P_{1},P_{2})=\Int(g)(P_{f},P_{\iota f})=(P_{h},P_{\iota h})$. 
\end{proof}

\subsection{~\label{sub:DefFondDiag}}

By an elementary case of fpqc descent (along $p_{1}$ and $t$), we
thus obtain a cartesian diagram of open and closed embeddings of smooth
$S$-schemes, \nomenclature[Fil]{$\Fi$}{Morphism $\Fi : \mathbb{G}^\Gamma (G) \rightarrow \mathbb{F}^\Gamma (G)$ defined on page \nomrefpage}\nomenclature[t_]{$t$}{Type morphism $t : \mathbb{F}^\Gamma (G) \rightarrow \mathbb{C}^\Gamma (G)$, defined page \nomrefpage}\nomenclature[F^\Gamma(G)]{$\mathbb{F}^\Gamma (G)$}{Scheme of $\Gamma$ filtrations on $G$,  page  \nomrefpage}\nomenclature[C^\Gamma(G)]{$\mathbb{C}^\Gamma (G)$}{Scheme of types of $\Gamma$-graduations of $\Gamma$-filtrations on $G$, page  \nomrefpage}\nomenclature[F]{$F$}{Facet morphism on $\mathbb{G}^\Gamma (G)$, $\mathbb{F}^\Gamma (G))$ or $\mathbb{C}^\Gamma (G)$, page  \nomrefpage}
\[
\xyC{2pc}\xymatrix{\mathbb{G}^{\Gamma}(G)\ar@{->>}[r]^{\Fi}\ar@{^{(}->}[d] & \mathbb{F}^{\Gamma}(G)\ar@{->>}[r]^{t}\ar@{^{(}->}[d] & \mathbb{C}^{\Gamma}(G)\ar@{^{(}->}[d]\\
\mathbb{G}^{\Gamma}(R_{\mathbb{OPP}(G)})\ar@{->>}[r]^{\Fi} & \mathbb{G}^{\Gamma}(R_{\mathbb{P}(G)})\ar@{->>}[r]^{t} & \mathbb{G}^{\Gamma}(R_{\mathbb{O}(G)})
}
\]
which in turns gives our fundamental cartesian diagram of smooth $S$-schemes
\[
\xyC{2pc}\xymatrix{\mathbb{G}^{\Gamma}(G)\ar@{->>}[r]^{\Fi}\ar[d]_{F} & \mathbb{F}^{\Gamma}(G)\ar@{->>}[r]^{t}\ar[d]_{F} & \mathbb{C}^{\Gamma}(G)\ar[d]_{F}\\
\mathbb{OPP}(G)\ar@{->>}[r]^{p_{1}} & \mathbb{P}(G)\ar@{->>}[r]^{t} & \mathbb{O}(G)
}
\]
The $S$-group scheme $G$ acts on both diagrams by conjugation and
their last column are the quotients of the other two columns by the
action of $G$ in the category of fpqc sheaves on $S$. The morphism
$\Fi:\mathbb{G}^{\Gamma}(G)\rightarrow\mathbb{F}^{\Gamma}(G)$ is
a $U_{\mathbb{F}^{\Gamma}(G)}$-torsor, where $U_{\mathbb{F}^{\Gamma}(G)}$
is the unipotent radical of the pull-back $P_{\mathbb{F}^{\Gamma}(G)}$
of the universal parabolic subgroup $P_{u}$ of $G_{\mathbb{P}(G)}$.
In particular, it is affine smooth surjective with geometrically connected
fibers. The morphism $t:\mathbb{F}^{\Gamma}(G)\rightarrow\mathbb{C}^{\Gamma}(G)$
is projective smooth surjective with geometrically connected fibers.
The three \emph{facet }morphisms $F$ are quasi-isotrivial twisted
constant (i.e.~locally constant in the étale topology on their base),
in particular they are separated and étale by lemma~\ref{lem:StructQuasIsoTwistedSch}.
We will see in due time that they are also surjective (\ref{sub:DefFacetsClosedOpen}).
Since $\mathbb{O}(G)$ is a finite étale cover of $S$, everyone is
smooth, surjective and separated over $S$. We denote by 
\[
0:S\rightarrow\mathbb{G}^{\Gamma}(G),\quad0:S\rightarrow\mathbb{F}^{\Gamma}(G)\quad\mbox{and}\quad0:S\rightarrow\mathbb{C}^{\Gamma}(G)
\]
the element of $\mathbb{G}^{\Gamma}(G)(S)$ corresponding to the trivial
morphism $\mathbb{D}_{S}(\Gamma)\rightarrow G$ or its images in $\mathbb{F}^{\Gamma}(G)(S)$
or $\mathbb{C}^{\Gamma}(G)(S)$. They respectively map to $(G,G)\in\mathbb{OPP}(G)(S)$,
$G\in\mathbb{P}(G)$ and $\mathbb{DYN}(G)\in\mathbb{O}(G)$. Being
sections of separated $S$-schemes, these $0$-sections are closed
immersions and the last one is also open. 

If $S$ is irreducible and geometrically unibranch or local henselian,
then so are the connected components of $\mathbb{O}(G)$ by \cite[18.10.1 and 18.5.10]{EGA4.4}.
Over each of them, the facet map $F:\mathbb{C}^{\Gamma}(G)\rightarrow\mathbb{O}(G)$
is then merely an infinite disjoint union of connected finite étale
covers, see lemma~\ref{lem:StructQuasIsoTwistedSch}. The same decomposition
then also holds for its pull-backs over $\mathbb{P}(G)$ or $\mathbb{OPP}(G)$.

\subsection{~\label{sub:DescParEqConj}}

For an $S$-scheme $T$ and morphisms $x,y:\mathbb{D}_{T}(\Gamma)\rightarrow G_{T}$,
we have
\begin{eqnarray*}
\Fi(x)=\Fi(y)\quad\mbox{in}\quad\mathbb{F}^{\Gamma}(G)(T) & \iff & \exists p\in P_{x}(T):\mbox{ }\Int(p)(x)=y\\
 & \iff & \exists u\in U_{x}(T):\mbox{ }\Int(u)(x)=y
\end{eqnarray*}
and then such a $u$ is unique. This equivalence relation is known
as the Par-equivalence and denoted by \nomenclature[sim_Par]{$\sim _{\mathrm{Par}}$}{Par-equivalence on $\mathbb{G}^\Gamma (G)$, defined page \nomrefpage}$x\sim_{\mathrm{Par}}y$.
If $T$ is an (absolutely) affine scheme, then 
\[
\mathbb{F}^{\Gamma}(G)(T)=\mathbb{G}^{\Gamma}(G)(T)/\sim_{\mathrm{Par}}
\]
by \cite[XXVI 2.2]{SGA3.3r}. On the other hand,
\[
\begin{array}{rl}
 & t\circ\Fi(x)=t\circ\Fi(y)\quad\mbox{in}\quad\mathbb{C}^{\Gamma}(G)(T)\\
\iff & \mbox{fpqc locally on }T,\,\exists g\in G(T):\mbox{ }\Int(g)(x)=y.
\end{array}
\]
If $T$ is semi-local, then by \cite[XXVI 5.2]{SGA3.3r}, 
\[
\begin{array}{rl}
 & t\circ\Fi(x)=t\circ\Fi(y)\quad\mbox{in}\quad\mathbb{C}^{\Gamma}(G)(T)\\
\iff & \exists g\in G(T):\mbox{ }\Int(g)(x)=y.
\end{array}
\]

\subsection{~\label{sub:NotationsPFandFbar}}

For an $S$-scheme $T$ and $\mathcal{F}$ in $\mathbb{F}^{\Gamma}(G)(T)$,
we denote by \nomenclature[P_F]{$P_\mathcal{F}$}{Parabolic subgroup of $G$ fixing $\mathcal{F}$, page \nomrefpage}\nomenclature[Foverlined]{$\overline{\mathcal{F}}$}{Morphism $\mathbb{D}(\Gamma) \rightarrow \overline{R}(P_\mathcal{F})$ attached to a $\Gamma$-filtration $\mathcal{F}$, page \nomrefpage}$(P_{\mathcal{F}},\overline{\mathcal{F}})$
the image of $\mathcal{F}$ in $\mathbb{G}^{\Gamma}(R_{\mathbb{P}(G)})(T)$.
Thus $P_{\mathcal{F}}=F(\mathcal{F})$ is a parabolic subgroup of
$G_{T}$, equal to the stabilizer of $\mathcal{F}$ in $G_{T}$ by
\cite[XXVI 1.2]{SGA3.3r} and $\overline{\mathcal{F}}:\mathbb{D}_{T}(\Gamma)\rightarrow\overline{R}(P_{\mathcal{F}})$
is a morphism of tori over $T$. We write \nomenclature[U_F]{$U_{\mathcal{F}}$}{Unipotent radical of $P_{\mathcal{F}}$.}$U_{\mathcal{F}}=R^{u}P_{\mathcal{F}}$
for the unipotent radical of $P_{\mathcal{F}}$, so that $\overline{R}(P_{\mathcal{F}})$
is the radical of $P_{\mathcal{F}}/U_{\mathcal{F}}$. If $L$ is a
Levi subgroup of $P_{\mathcal{F}}=U_{\mathcal{F}}\rtimes L$ and $f:\mathbb{D}_{T}(\Gamma)\rightarrow L$
is the corresponding central morphism lifting $\overline{\mathcal{F}}$,
then 
\[
\mathcal{F}=\Fi(f)\quad\mbox{and}\quad L_{f}=L.
\]
The inversion $f\mapsto f^{-1}$ yields compatible involutions on
$\mathbb{G}^{\Gamma}(G)$ and $\mathbb{C}^{\Gamma}(G)$, which we
shall both denote by $\iota$. By proposition~\ref{prop:DefPGroupCase},
they are also compatible with the eponymous involutions on $\mathbb{OPP}(G)$
and $\mathbb{O}(G)$: \nomenclature[iota]{$\iota$}{Opposition involution of  $\mathbb{G}^\Gamma (G)$ or $\mathbb{C}^\Gamma (G)$, defined page \nomrefpage}
\[
F\circ\iota=\iota\circ F\quad\mbox{on}\quad\mathbb{G}^{\Gamma}(G)\mbox{ or }\mathbb{C}^{\Gamma}(G).
\]

\subsection{Functoriality}

The formation of our fundamental diagram
\[
\xyC{2pc}\xymatrix{\mathbb{G}^{\Gamma}(G)\ar@{->>}[r]^{\Fi}\ar[d]_{F} & \mathbb{F}^{\Gamma}(G)\ar@{->>}[r]^{t}\ar[d]_{F} & \mathbb{C}^{\Gamma}(G)\ar[d]_{F}\\
\mathbb{OPP}(G)\ar@{->>}[r]^{p_{1}} & \mathbb{P}(G)\ar@{->>}[r]^{t} & \mathbb{O}(G)
}
\]
is plainly compatible with base change on $S$. We will see later
on (corollary~\ref{cor:FunctFirstLine}) that the first line is covariantly
functorial in $\Gamma$ and $G$. This is obvious for $\mathbb{G}^{\Gamma}(G)$
and easy for $\mathbb{C}^{\Gamma}(G)=G\backslash\mathbb{G}^{\Gamma}(G)$,
but not so for $\mathbb{F}^{\Gamma}(G)$: the $\Gamma$-functoriality
of $\mathbb{G}^{\Gamma}(G)$ is not compatible with the facet maps,
and the second line of our diagram is simply not functorial in $G$.
To showcase the first (bad) behavior, note that we will eventually
have an action of the set $\End(\Gamma,+,\leq)$ of non-decreasing
homomorphisms of $\Gamma$ by morphisms of $S$-schemes on the first
line, simply denoted by $(\lambda,x)\mapsto\lambda\cdot x$. Then
$x\mapsto0\cdot x$ is nothing but the structural morphism of the
$S$-scheme $\mathbb{G}^{\Gamma}(G)$, $\mathbb{F}^{\Gamma}(G)$ or
$\mathbb{C}^{\Gamma}(G)$, followed by the corresponding $0$-section.
Thus $F(0\cdot x)=\mathbb{DYN}(G)$ in $\mathbb{O}(G)$ for any $x\in\mathbb{C}^{\Gamma}(G)$.
However, for a monomorphism $\gamma:(\Gamma_{1},+,\leq)\hookrightarrow(\Gamma_{2},+,\leq)$,
the induced morphisms $\gamma$ in the commutative diagram 
\[
\xyC{2pc}\xymatrix{\mathbb{G}^{\Gamma_{1}}(G)\ar@{->>}[r]^{\Fi}\ar@{_{(}->}[d]_{\gamma} & \mathbb{F}^{\Gamma_{1}}(G)\ar@{->>}[r]^{t}\ar@{_{(}->}[d]_{\gamma} & \mathbb{C}^{\Gamma_{1}}(G)\ar@{_{(}->}[d]_{\gamma}\\
\mathbb{G}^{\Gamma_{2}}(G)\ar@{->>}[r]^{\Fi} & \mathbb{F}^{\Gamma_{2}}(G)\ar@{->>}[r]^{t} & \mathbb{C}^{\Gamma_{2}}(G)
}
\]
are open and closed immersions which commute with the facet maps:
this follows from proposition~\ref{Pro:StructImagG} and \ref{prop:DefPGroupCase},
given the construction of our fundamental diagram.

\subsection{$\mathbb{C}^{\Gamma}(G)$ is a commutative monoid\label{sub:CMonoid}}

There is natural structure of commutative monoid on the $S$-scheme
$\mathbb{O}(G)$, given by the intersection morphism
\[
\cap:\mathbb{O}(G)\times_{S}\mathbb{O}(G)\rightarrow\mathbb{O}(G)\qquad(a,b)\mapsto a\cap b
\]
Let $\mathbb{O}'(G)$ be the open and closed subscheme of $\mathbb{O}(G)\times_{S}\mathbb{O}(G)$
on which $a\cap b=a$, i.e. $a\subset b$. Let $p_{1}$ and $p_{2}:\mathbb{O}'(G)\rightarrow\mathbb{O}(G)$
be the two projections. We claim:
\begin{lem}
There exists a canonical morphism $p_{2}^{\ast}R_{\mathbb{O}(G)}\rightarrow p_{1}^{\ast}R_{\mathbb{O}(G)}$.\end{lem}
\begin{proof}
Let~$\mathbb{P}'(G)$ be the inverse image of $\mathbb{O}'(G)$ in
$\mathbb{P}(G)\times_{S}\mathbb{P}(G)$, and denote by $q_{1}$ and
$q_{2}:\mathbb{P}'(G)\rightarrow\mathbb{P}(G)$ the two projections.
Then $q_{i}^{\ast}(R_{\mathbb{P}(G)})=(p_{i}^{\ast}R_{\mathbb{O}(G)})_{\mathbb{P}'(G)}$
for $i\in\{1,2\}$. We have to show that there is a canonical morphism
\[
q_{2}^{\ast}R_{\mathbb{P}(G)}\rightarrow q_{1}^{\ast}R_{\mathbb{P}(G)}
\]
compatible with the descent data on both sides. This boils down to:
for any $S'\rightarrow S$ and $(P_{1},P_{2})\in\mathbb{P}'(G)(S')$,
there exists a canonical morphism $\overline{R}(P_{2})\rightarrow\overline{R}(P_{1})$.
We may assume that $S'=S$. Since $t(P_{1})\subset t(P_{2})$, there
exists by \cite[XXVI 3.8]{SGA3.3r} a unique parabolic subgroup $P'_{2}$
of $G$, containing $P_{1}$, such that $t(P_{2})=t(P'_{2})$. Using
the canonical isomorphism $\overline{R}(P'_{2})\simeq\overline{R}(P_{2})$,
we may thus assume that $P'_{2}=P_{2}$,~i.e. $P_{1}\subset P_{2}$.
Let $U_{i}$ be the unipotent radical of $P_{i}$, so that $U_{2}\subset U_{1}$
is a normal subgroup of $P_{1}$. Then $P_{1}/U_{2}$ is a parabolic
subgroup of $P_{2}/U_{2}$ with maximal reductive quotient $P_{1}/U_{1}$,
which reduces us further to the case where $G=P_{2}$. Then $P_{1}$
contains the radical \nomenclature[R(G)]{$R(G)$}{Radical of $G$.}\nomenclature[R^u(G)]{$R^u (P)$}{Unipotent radical of $P$.}$\overline{R}(G)=R(G)$
of $G$, and $P_{1}\rightarrow P_{1}/U_{1}$ maps $R(G)$ to the radical
$\overline{R}(P_{1})$ of $P_{1}/U_{1}$. This yields our canonical
morphism $\overline{R}(P_{2})\rightarrow\overline{R}(P_{1})$. 
\end{proof}
\noindent Pulling back the above morphism through 
\begin{eqnarray*}
\mathbb{O}(G)\times_{S}\mathbb{O}(G) & \rightarrow & \mathbb{O}'(G)\\
(a,b) & \mapsto & (a\cap b,b)
\end{eqnarray*}
we obtain a morphism $p_{2}^{\ast}R_{\mathbb{O}(G)}\rightarrow(\cap)^{\ast}R_{\mathbb{O}(G)}$
of tori over $\mathbb{O}(G)\times_{S}\mathbb{O}(G)$. By symmetry,
there is also a morphism $p_{1}^{\ast}R_{\mathbb{O}(G)}\rightarrow(\cap)^{\ast}R_{\mathbb{O}(G)}$.
The product of these two yields a morphism in the fibered category
of tori over $\Sch/S$,
\[
\xyR{1.5pc}\xymatrix{R_{\mathbb{O}(G)}\times_{S}R_{\mathbb{O}(G)}\ar[r]\ar[d] & R_{\mathbb{O}(G)}\times_{\mathbb{O}(G)}R_{\mathbb{O}(G)}\ar[d]\\
\mathbb{O}(G)\times_{S}\mathbb{O}(G)\ar[r]\sp(0.6){\cap} & \mathbb{O}(G)
}
\]
Composing it with the multiplication map on the $\mathbb{O}(G)$-torus
$R_{\mathbb{O}(G)}$, we obtain yet another such morphism, namely
\[
\xymatrix{R_{\mathbb{O}(G)}\times_{S}R_{\mathbb{O}(G)}\ar[r]\ar[d] & R_{\mathbb{O}(G)}\ar[d]\\
\mathbb{O}(G)\times_{S}\mathbb{O}(G)\ar[r]\sp(0.63){\cap} & \mathbb{O}(G)
}
\]
Applying now the $\mathbb{G}^{\Gamma}(-)$ construction to the latter
diagram yields a morphism 
\[
\xymatrix{\mathbb{G}^{\Gamma}(R_{\mathbb{O}(G)})\times_{S}\mathbb{G}^{\Gamma}(R_{\mathbb{O}(G)})\ar[r]\ar[d] & \mathbb{G}^{\Gamma}(R_{\mathbb{O}(G)})\ar[d]\\
\mathbb{O}(G)\times_{S}\mathbb{O}(G)\ar[r]\sp(0.6){\cap} & \mathbb{O}(G)
}
\]
in the fibered category of commutative group schemes over $\Sch/S$.
The top map of this diagram defines a commutative monoid structure
on the $S$-scheme $\mathbb{G}^{\Gamma}(R_{\mathbb{O}(G)})$. By construction,
the structural morphism $\mathbb{G}^{\Gamma}(R_{\mathbb{O}(G)})\rightarrow\mathbb{O}(G)$
is compatible with the monoid structures on both sides.
\begin{lem}
\label{lem:CsubMonoide}The $S$-scheme $\mathbb{C}^{\Gamma}(G)$
is a submonoid of $\mathbb{G}^{\Gamma}(R_{\mathbb{O}(G)})$.\end{lem}
\begin{proof}
Using additive notations, we have to show that for $S'\rightarrow S$
and $c_{1},c_{2}$ in $\mathbb{C}^{\Gamma}(G)(S')$, there exists
an fpqc cover $S''\rightarrow S'$ and an element $f\in\mathbb{G}^{\Gamma}(G)(S'')$
such that $c_{1}+c_{2}=t\circ\Fi(f)$ in $\mathbb{G}^{\Gamma}(R_{\mathbb{O}(G)})$.
We may assume that $S'=S$ and $c_{i}=t\circ\Fi(f_{i})$ for some
$f_{i}:\mathbb{D}_{S}(\Gamma)\rightarrow G$. Using \cite[XXVI 1.14 and XXIV 1.5]{SGA3.3r},
we may also assume that there is an épinglage $(G,T,\Delta,\cdots)$
which is adapted to $P_{1}=P_{f_{1}}$ and $P_{2}=P_{f_{2}}$. Then
by \cite[XXVI 1.6 and 1.8]{SGA3.3r}, we may assume that $L_{1}=L_{f_{1}}$
and $L_{2}=L_{f_{2}}$ both contain the maximal torus $T$ of $G$,
so that both $f_{1}$ and $f_{2}$ factor through $T$. Let $f=f_{1}+f_{2}:\mathbb{D}_{S}(\Gamma)\rightarrow T\hookrightarrow G$
and $P=P_{f}$. We claim that $c_{1}+c_{2}=t\circ\Fi(f)$. By \cite[XXVI 3.2]{SGA3.3r},
$t(P_{i})=\Delta(P_{i})_{S}$ where $\Delta(P_{i})\subset\Delta\subset\Hom(T,\mathbb{G}_{m,S})$
is the set of simple roots occurring in $\Lie(L_{i})$, i.e.~$\Delta(P_{i})=\left\{ \alpha\in\Delta:\alpha\circ f_{i}=0\in\Gamma\right\} $.
By construction, $\alpha\circ f_{i}\geq0$ in $\Gamma$ for every
$\alpha\in\Delta$, thus also $\alpha\circ f=\alpha\circ f_{1}+\alpha\circ f_{2}\geq0$
in $\Gamma$ for every $\alpha\in\Delta$, with $\alpha\circ f=0$
if and only if $\alpha\circ f_{1}=0=\alpha\circ f_{2}$. It follows
that our épinglage is also adapted to $P$, with $\Delta(P)=\Delta(P_{1})\cap\Delta(P_{2})$,
i.e. $t(P)=t(P_{1})\cap t(P_{2})$ in $\mathbb{O}(G)(S)$. The inclusion
$P\subset P_{i}$ induces the canonical morphism $\mathrm{can}_{i}:\overline{R}(P_{i})\rightarrow\overline{R}(P)$
and one checks easily that $\overline{f}=\mathrm{can}_{1}\circ\overline{f}_{1}+\mathrm{can}_{2}\circ\overline{f}_{2}$.
Thus by definition, 
\[
c_{1}+c_{2}=t(P,\mathrm{can}_{1}\circ\overline{f}_{1}+\mathrm{can}_{2}\circ\overline{f}_{2})=t(P,\overline{f})=t\circ\Fi(f)
\]
as was to be shown.\end{proof}
\begin{rem}
The $0$-section of $\mathbb{C}^{\Gamma}(G)$ is the identity element
of its monoid structure. The latter is compatible with functoriality
in $\Gamma$, but not with functoriality in $G$: if $H$ is a subtorus
of $G$ and $f$ is a section of $\mathbb{G}^{\Gamma}(H)=\mathbb{F}^{\Gamma}(H)=\mathbb{C}^{\Gamma}(H)$,
then $f+\iota f$ is trivial in $\mathbb{C}^{\Gamma}(H)$, but not
necessarily in $\mathbb{C}^{\Gamma}(G)$. 
\end{rem}

\subsection{The split case\label{sub:C(G)_splitcase}}

Suppose that $(G,T,M,R)$ is a split reductive group over $S$ \cite[XXII 1.13]{SGA3.3r}:
$G$ is a reductive group over $S$, $M$ is a finite free $\mathbb{Z}$-module,
$T\subset G$ is a maximal subtorus of $G$ equipped with an isomorphism
$T\simeq\mathbb{D}_{S}(M)$, $R\subset M$ is a set of roots of $T$
in $G$ and for each $\alpha\in R$, the corresponding quasi-coherent
sub-sheaf $\mathfrak{g}_{\alpha}$ of $\mathfrak{g}=\Lie(G)$ is a
free $\mathcal{O}_{S}$-module (of rank $1$). Let 
\[
\mathscr{R}=(M,R,M^{\ast},R^{\ast})
\]
be the induced (reduced) root system \cite[XXII 1.14]{SGA3.3r} with
Weyl group $W=W(\mathscr{R})$ in $\Aut(M)$. Let $W_{G}(T)=N_{G}(T)/Z_{G}(T)$
be the Weyl group of $T$ in $G$, a constant group scheme over $S$
identified with $W_{S}$ through its action on $T$ \cite[XXII 3.4]{SGA3.3r}.
The composition of the isomorphism of group schemes over $S$ 
\[
\left(\Hom\left(M,\Gamma\right)\right)_{S}\simeq\Hom_{S-\Group}\left(\mathbb{D}_{S}(\Gamma),\mathbb{D}_{S}(M)\right)\simeq\mathbb{G}^{\Gamma}(T)
\]
from \cite[VIII 1.5]{SGA3.1r} with the morphism of $S$-schemes 
\[
\mathbb{G}^{\Gamma}(T)\hookrightarrow\mathbb{G}^{\Gamma}(G)\stackrel{\Fi}{\longrightarrow}\mathbb{F}^{\Gamma}(G)\stackrel{t}{\longrightarrow}\mathbb{C}^{\Gamma}(G)
\]
thus factors through a morphism of étale $S$-schemes, 
\[
\left(W\backslash\Hom\left(M,\Gamma\right)\right)_{S}\rightarrow\mathbb{C}^{\Gamma}(G).
\]
We claim that the latter is an isomorphism. Since both sides are étale
over $S$, it is sufficient to establish that for any geometric point
$\Spec(k)\rightarrow S$, the induced map 
\[
W\backslash\Hom\left(M,\Gamma\right)\rightarrow\mathbb{C}^{\Gamma}(G)(k)=G(k)\backslash\Hom_{k-\Group}\left(\mathbb{D}_{k}(\Gamma),G_{k}\right)
\]
is a bijection. Any $f:\mathbb{D}_{k}(\Gamma)\rightarrow G_{k}$ factors
through a maximal torus $T'$ of $G_{k}$ by corollary~\ref{cor:imageofcocharIssubtorus},
and $T'=\Int(g)(T_{k})$ for some $g\in G(k)$ by~\cite[XII 6.6.a]{SGA3.2}:
our map is surjective. For $\varphi,\varphi':M\rightarrow\Gamma$
giving $f,f':\mathbb{D}_{k}(\Gamma)\rightarrow T_{k}$ and $g\in G(k)$
such that $\Int(g)\circ f=f'$, $\Int(g)(T_{k})$ and $T_{k}$ are
maximal tori of $L_{f'}$, thus $\Int(hg)(T_{k})=T_{k}$ for some
$h\in L_{f'}(k)$; but then $n=hg\in N_{G}(T)(k)$ and $\Int(n)\circ f=\Int(h)\circ f'=f'$,
thus $\varphi'=w\varphi$ where $w$ is the image of $n$ in $W=W_{G}(T)(k)$:
our map is injective.

Fix a system of positive roots $R_{+}\subset R$ \cite[XXI 3.2.1]{SGA3.3r},
which corresponds to a Borel subgroup $B$ of $G$ by \cite[XXII 5.5.1]{SGA3.3r}.
By lemma~\ref{lem:DominantIsFundDom} below, the submonoid 
\[
\Hom^{+}(M,\Gamma)=\left\{ f\in\Hom(M,\Gamma):\forall\alpha\in R_{+},\, f(\alpha)\geq0\right\} 
\]
is a fundamental domain for the action of $W$ on $\Hom(M,\Gamma)$.
The isomorphism
\[
\left(\Hom^{+}\left(M,\Gamma\right)\right)_{S}\simeq\left(W\backslash\Hom\left(M,\Gamma\right)\right)_{S}\simeq\mathbb{C}^{\Gamma}(G)
\]
is then easily seen to be compatible with the $S$-monoid structures.

\subsection{$\mathbb{C}^{\Gamma}(G)$ is a partially ordered commutative monoid\label{sub:C(G)_dominorder}}

A partial order $\leq$ on an $S$-scheme $X$ is a subscheme $\mathcal{R}=\mathcal{R}(\leq)$
of $X\times_{S}X$ such that for every $S$-scheme $Y$, the subset
$\mathcal{R}(Y)$ of $X(Y)\times X(Y)$ defines a partial order (also
denoted by $\leq$) on $X(Y)$. We say that the partial order is open
(resp.~closed) if $\mathcal{R}\hookrightarrow X\times_{S}X$ is an
open (resp.~closed) immersion. A partial order on an $S$-monoid
$(X,\cdot)$ is a partial order on the underlying $S$-scheme such
that for any $S$-scheme $Y$ and $f_{1},f_{2},g$ in $X(Y)$, $f_{1}\leq f_{2}$
implies $f_{1}\cdot g\leq f_{2}\cdot g$ and $g\cdot f_{1}\leq g\cdot f_{2}$. 

If $\mathscr{R}=(M,R,M^{\ast},R^{\ast})$ is a (not necessarily reduced)
root system and $R_{+}\subset R$ is a system of positive roots, the
weak ($\leq$) and strong ($\preceq$) partial orders on the abstract
monoid $\Hom^{+}(M,\Gamma)$ defined in section~\ref{sub:AppendixRoots}
below induce open and closed partial orders on the constant $S$-monoid
$\left(\Hom^{+}\left(M,\Gamma\right)\right)_{S}$. If $\mathscr{R}=\mathscr{R}(G,T,M,R)$
is the root system of a split reductive group $(G,T,M,R)$, we thus
obtain open and closed partial orders on the $S$-monoid $\mathbb{C}^{\Gamma}(G)$.
These partial orders then do not depend upon the chosen auxiliary
data $(T,M,R;R^{+})$: this may be checked on geometric points, where
all such data are indeed conjugated. Since every reductive group $G$
over $S$ is locally splittable in the étale topology on $S$ \cite[XXII 2.3]{SGA3.3r},
we finally obtain by étale descent: the $S$-monoid $\mathbb{C}^{\Gamma}(G)$
is canonically equipped with weak ($\leq$)\nomenclature[<=]{$\leq$}{Weak dominance partial order on $\mathbb {C}^\Gamma (G)$, page \nomrefpage}
and strong $(\preceq$)\nomenclature[<=s]{$\preceq$}{Strong dominance partial order on $\mathbb {C}^\Gamma (G)$, page \nomrefpage}
partial orders, both open and closed.

The weak and strong partial orders are functorial in $\Gamma$, and
coincide if $\Gamma$ is divisible. We will see later on that the
weak partial order is also functorial in $G$.

\subsection{Behavior under isogenies\label{sub:Isogenies}}

Suppose that the (torsion free) commutative group $\Gamma$ is (uniquely)
divisible, i.e.~that it is a $\mathbb{Q}$-vector space. 
\begin{prop}
The fundamental cartesian diagram 
\[
\xyC{2pc}\xymatrix{\mathbb{G}^{\Gamma}(G)\ar[r]^{\Fi}\ar[d]^{F} & \mathbb{F}^{\Gamma}(G)\ar[r]^{t}\ar[d]^{F} & \mathbb{C}^{\Gamma}(G)\ar[d]^{F}\\
\mathbb{OPP}(G)\ar[r]^{p_{1}} & \mathbb{P}(G)\ar[r]^{t} & \mathbb{O}(G)
}
\]
is invariant under central isogenies.\end{prop}
\begin{proof}
The bottom line only depends upon the adjoint group \nomenclature[Z(G)]{$Z(G)$}{Center of $G$.}\nomenclature[G^ad]{$G^{\mathrm{ad}}$}{Adjoint group $G^{\mathrm{ad}} = G/Z(G)$ of $G$.}$G^{\mathrm{ad}}=G/Z(G)$:
this is true for $\mathbb{O}(G)$ because $\mathbb{DYN}(G)=\mathbb{DYN}(G^{\mathrm{ad}})$
by definition of the Dynkin $S$-scheme \cite[XXIV 3.3]{SGA3.3r}
in view of \cite[XXII 4.3.7]{SGA3.3r}, and the maps $P\mapsto P/Z(G)$
and $P^{\mathrm{ad}}\mapsto\mathrm{ad}^{-1}(P^{\mathrm{ad}})$ (where
\nomenclature[ad]{$\mathrm{ad}$}{Morphism $\mathrm{ad}:G \rightarrow G^\mathrm{ad}$.}$\mathrm{ad}:G\rightarrow G^{\mathrm{ad}}$
is the quotient map) induce mutually inverse bijections between parabolic
subgroups of $G$ and parabolic subgroups of $G^{\mathrm{ad}}$, which
are compatible with the type maps and with opposition. For the top
line, let $f:G_{1}\rightarrow G_{2}$ be a central isogeny \cite[XXII 4.2.9]{SGA3.3r}.
We first claim that composition with $f$ yields an isomorphism $\mathbb{G}^{\Gamma}(G_{1})\rightarrow\mathbb{G}^{\Gamma}(G_{2})$:
for split tori, this immediately follows from \cite[VIII 1.5]{SGA3.1r}
and our assumption on $\Gamma$; for tori, our claim is local in the
fpqc topology on $S$ by \cite[2.7.1]{EGA4.2}, which reduces us to
the previous case; for arbitrary reductive groups, use lemma~\ref{lem:ImageMorphTor}
and \cite[XVII 7.1.1]{SGA3.2}. If now $P_{1}$ is a parabolic subgroup
of $G_{1}$ with image $P_{2}$ in $G_{2}$, then $f$ induces an
isogeny $\overline{R}(P_{1})\rightarrow\overline{R}(P_{2})$. Thus
$f$ yields an isogeny $R_{f}:R_{\mathbb{P}(G_{1})}\rightarrow R_{\mathbb{P}(G_{2})}$
of tori over $\mathbb{P}(G_{1})\simeq\mathbb{P}(G_{2})$. The induced
isomorphism $\mathbb{G}^{\Gamma}(R_{\mathbb{P}(G_{1})})\simeq\mathbb{G}^{\Gamma}(R_{\mathbb{P}(G_{2})})$
is compatible with the morphisms $\mathbb{G}^{\Gamma}(G_{i})\rightarrow\mathbb{G}^{\Gamma}(R_{\mathbb{P}(G_{i})})$,
therefore also $\mathbb{F}^{\Gamma}(G_{1})\simeq\mathbb{F}^{\Gamma}(G_{2})$.
Since $R_{f}$ is also compatible with the canonical descent data
of lemma~\ref{lem:CanoDescData}, it descends to an isogeny $R_{f}:R_{\mathbb{O}(G_{1})}\rightarrow R_{\mathbb{O}(G_{2})}$
of tori over $\mathbb{O}(G_{1})\simeq\mathbb{O}(G_{2})$. The induced
isomorphism $\mathbb{G}^{\Gamma}(R_{\mathbb{O}(G_{1})})\simeq\mathbb{G}^{\Gamma}(R_{\mathbb{O}(G_{2})})$
is again compatible with the morphisms $\mathbb{F}^{\Gamma}(G_{i})\rightarrow\mathbb{G}^{\Gamma}(R_{\mathbb{O}(G_{i})})$,
therefore also $\mathbb{C}^{\Gamma}(G_{1})\simeq\mathbb{C}^{\Gamma}(G_{2})$.
\end{proof}
\noindent Plainly, the above diagrams are also compatible with products.
Considering the canonical diagram of central isogenies \cite[XXII 4.3 \& 6.2]{SGA3.3r}
\[
R(G)\times G^{\mathrm{der}}\rightarrow G\rightarrow G^{\mathrm{ab}}\times G^{\mathrm{ss}}\rightarrow G^{\mathrm{ab}}\times G^{\mathrm{ad}}
\]
where $R(G)$ is the radical of $G$, \nomenclature[G^der]{$G^\mathrm{der}$}{Derived group of $G$.}$G^{\mathrm{der}}$
its derived group, \nomenclature[G^ab]{$G^\mathrm{ab}$}{Abelianization $G^\mathrm{ab} = G/G^\mathrm{der}$  of $G$.}$G^{\mathrm{ab}}=G/G^{\mathrm{der}}$
its coradical, \nomenclature[G^ss]{$G^\mathrm{ss}$}{Semi-simplification $G^\mathrm{ss} = G/R(G)$  of $G$.}$G^{\mathrm{ss}}=G/R(G)$
its semi-simplification and $G^{\mathrm{ad}}=G/Z(G)$ its adjoint
group, we obtain compatible canonical decompositions\nomenclature[G^Gamma(G)^r]{$\mathbb{G}^\Gamma (G)^r$}{Reduced part of $\mathbb{G}^\Gamma (G)$, page \nomrefpage}\nomenclature[F^Gamma(G)^r]{$\mathbb{F}^\Gamma (G)^r$}{Reduced part of $\mathbb{F}^\Gamma (G)$, page \nomrefpage}\nomenclature[C^Gamma(G)^r]{$\mathbb{C}^\Gamma (G)^r$}{Reduced part of $\mathbb{C}^\Gamma (G)$, page \nomrefpage}\nomenclature[G^Gamma(G)^c]{$\mathbb{G}^\Gamma (G)^c$}{Central part of $\mathbb{G}^\Gamma (G)$, page \nomrefpage}\nomenclature[F^Gamma(G)^c]{$\mathbb{F}^\Gamma (G)^c$}{Central part of $\mathbb{F}^\Gamma (G)$, page \nomrefpage}\nomenclature[C^Gamma(G)^c]{$\mathbb{C}^\Gamma (G)^c$}{Central part of $\mathbb{C}^\Gamma (G)$, page \nomrefpage}
\[
\begin{array}{rcccc}
\mathbb{G}^{\Gamma}(G) & = & \mathbb{G}^{\Gamma}(G)^{r} & \times & \mathbb{G}^{\Gamma}(G)^{c}\\
\mathbb{F}^{\Gamma}(G) & = & \mathbb{F}^{\Gamma}(G)^{r} & \times & \mathbb{F}^{\Gamma}(G)^{c}\\
\mathbb{C}^{\Gamma}(G) & = & \mathbb{C}^{\Gamma}(G)^{r} & \times & \mathbb{C}^{\Gamma}(G)^{c}
\end{array}
\]
with $\mathbb{G}^{\Gamma}(G)^{c}=\mathbb{F}^{\Gamma}(G)^{c}=\mathbb{C}^{\Gamma}(G)^{c}=\mathbb{G}^{\Gamma}(R(G))=\mathbb{G}^{\Gamma}(G^{\mathrm{ab}})=\mathbb{G}^{\Gamma}(Z(G))$
and 
\[
\begin{array}{rccclcc}
\mathbb{G}^{\Gamma}(G)^{r} & = & \mathbb{G}^{\Gamma}(G^{\mathrm{der}}) & = & \mathbb{G}^{\Gamma}(G^{\mathrm{ss}}) & = & \mathbb{G}^{\Gamma}(G^{\mathrm{ad}}),\\
\mathbb{F}^{\Gamma}(G)^{r} & = & \mathbb{F}^{\Gamma}(G^{\mathrm{der}}) & = & \mathbb{F}^{\Gamma}(G^{\mathrm{ss}}) & = & \mathbb{F}^{\Gamma}(G^{\mathrm{ad}}),\\
\mathbb{C}^{\Gamma}(G)^{r} & = & \mathbb{C}^{\Gamma}(G^{\mathrm{der}}) & = & \mathbb{C}^{\Gamma}(G^{\mathrm{ss}}) & = & \mathbb{C}^{\Gamma}(G^{\mathrm{ad}}).
\end{array}
\]
The decomposition of $\mathbb{C}^{\Gamma}(G)$ is compatible with
the partially ordered (weak=strong) monoid structures: for $x=(x^{r},x^{c})$
and $y=(y^{r},y^{c})$ in $\mathbb{C}^{\Gamma}(G)=\mathbb{C}^{\Gamma}(G)^{r}\times\mathbb{C}^{\Gamma}(G)$,
\[
x+y=(x^{r}+y^{r},x^{c}+y^{c})\quad\mbox{and}\quad x\leq y\iff\left(x^{r}\leq y^{r}\mbox{ and }x^{c}=y^{c}\right).
\]
This is easily checked by reduction to the split case, cf.~section~\ref{sub:WeylConeIsProductAbstract}
below.

\section{Relative positions of $\Gamma$-filtrations\label{sub:RelativePositions}}

Let $G$ be a reductive group over $S$.

\subsection{Standard positions}

Recall that two parabolic subgroups $P_{1}$ and $P_{2}$ of $G$
are said to be in standard (relative) position if and only if they
satisfy the equivalent conditions of \cite[XXVI 4.5.1]{SGA3.3r},
in particular: $(i)$ $P_{1}\cap P_{2}$ is smooth over $S$, or $(ii)$
$P_{1}\cap P_{2}$ is a subgroup of type $(R)$ of $G$, or $(iv)$
$P_{1}\cap P_{2}$ contains, locally on $S$ for the Zariski topology,
a maximal torus of $G$. Then all such maximal tori are, locally on
$S$ for the étale topology, conjugated in $P_{1}\cap P_{2}$ \cite[XII 7.1]{SGA3.2}.
In any case, $P_{1}\cap P_{2}$ has geometrically connected fibers
\cite[4.5]{BoTi65}. For an $S$-scheme $T$, we set\nomenclature[STD(G)]{$\mathbb{STD}(G)$}{Scheme of pairs of parabolic subgroups of $G$ in standard relative position, defined page \nomrefpage}\nomenclature[TSTD(G)]{$\mathbb{TSDT}(G)$}{Scheme of types of pairs of parabolic subgroups of $G$ in standard relative position, defined page \nomrefpage}\nomenclature[t_2]{$t_2$}{Type morphism $t_2: \mathbb{STD}(G) \rightarrow \mathbb{TSTD}(G)$, defined page \nomrefpage}
\[
\mathbb{STD}(G)(T)=\left\{ (P_{1},P_{2})\in\mathbb{P}(G)^{2}(T):\mbox{ }P_{1}\mbox{ and }P_{2}\mbox{ are in standard position}\right\} .
\]
By~\cite[XXVI 4.5.3]{SGA3.3r}, this defines a representable subsheaf
of $\mathbb{P}(G)^{2}$ with Stein factorization
\[
\mathbb{STD}(G)\stackrel{t_{2}}{\longrightarrow}\mathbb{TSTD}(G)\longrightarrow S
\]
fitting in a commutative (but not cartesian) diagram
\[
\xymatrix{\mathbb{STD}(G)\ar@{^{(}->}[d]\ar[r]^{t_{2}} & \mathbb{TSTD}(G)\ar[d]_{q}\\
\mathbb{P}(G)^{2}\ar[r]^{t^{2}} & \mathbb{O}(G)^{2}
}
\]
where $q$ is a finite étale surjective morphism while $t_{2}$ is
a smooth, surjective, finitely presented morphism with geometrically
connected fibers which is a quotient of $\mathbb{STD}(G)$ by the
diagonal action of $G$ in the category of fpqc sheaves on $S$. By~\cite[XXVI 4.2.5 \& 4.4.3]{SGA3.3r},
the morphism $q$ has two canonical sections\nomenclature[tr]{$tr$}{Transverse section $tr:\mathbb{O}(G) \rightarrow \mathbb{TSTD}(G)$, defined page \nomrefpage}\nomenclature[os]{$os$}{Osculatory section $os:\mathbb{O}(G) \rightarrow \mathbb{TSTD}(G)$, defined page \nomrefpage}
\[
tr,os:\mathbb{O}(G)^{2}\rightarrow\mathbb{TSDT}(G)
\]
corresponding respectively to the transverse and osculatory (standard)
positions. By~\cite[XXVI 4.2.4]{SGA3.3r}, $t_{2}^{-1}(\mbox{im}(tr))$
is a relatively dense open $S$-subscheme \nomenclature[GEN(G)]{$\mathbb{GEN}(G)$}{Scheme of pairs of parabolic subgroups of $G$ in generic relative position, page \nomrefpage}$\mathbb{GEN}(G)$
of $\mathbb{P}(G)^{2}$. Pulling back everything through the surjective
étale facet morphism $F^{2}:\mathbb{C}^{\Gamma}(G)^{2}\rightarrow\mathbb{O}(G)^{2}$,
we thus obtain a commutative diagram\nomenclature[STD^Gamma(G)]{$\mathbb{STD}^\Gamma (G)$}{Scheme of pairs of $\Gamma$-filtrations on $G$ in standard relative position, defined page \nomrefpage}\nomenclature[TSTD^Gamma(G)]{$\mathbb{TSDT}^\Gamma (G)$}{Scheme of types of pairs of $\Gamma$-filtrations on $G$ in standard relative position, defined page \nomrefpage}
\[
\xymatrix{\mathbb{STD}^{\Gamma}(G)\ar@{^{(}->}[d]\ar[r]^{t_{2}} & \mathbb{TSTD}^{\Gamma}(G)\ar[d]_{q}\\
\mathbb{F}^{\Gamma}(G)^{2}\ar[r]^{t^{2}} & \mathbb{C}^{\Gamma}(G)^{2}\ar@/_{.5pc}/[u]_{tr,os}
}
\]
where $t_{2}$ \nomenclature[t_2_]{$t_2$}{Type morphism $t_2: \mathbb{STD}^\Gamma (G) \rightarrow \mathbb{TSTD}^\Gamma (G)$, defined page \nomrefpage}and
$q$ still have the properties listed above, together with a relatively
dense open $S$-subscheme $\mathbb{GEN}^{\Gamma}(G)$ of $\mathbb{F}^{\Gamma}(G)^{2}$\nomenclature[GEN^Gamma(G)]{$\mathbb{GEN}^\Gamma (G)$}{Scheme of pairs of $\Gamma$-filtrations on $G$ in generic relative position, page \nomrefpage}.
For a scheme $Z$ over $\mathbb{P}(G)^{2}$, we set\nomenclature[STD(Z)]{$\mathbb{STD}(Z)$}{Pull-back of $\mathbb{STD}(G) \hookrightarrow \mathbb{P}(G)^2$ through  $Z \rightarrow \mathbb{P}(G)^2$.}
\[
\mathbb{STD}(Z)=Z\times_{\mathbb{P}(G)^{2}}\mathbb{STD}(G).
\]
For instance, $\mathbb{STD}^{\Gamma}(G)=\mathbb{STD}\left(\mathbb{F}^{\Gamma}(G)\times_{S}\mathbb{F}^{\Gamma}(G)\right)$. 
\begin{rem}
\label{Rk:StdOverField}The monomorphisms $\mathbb{STD}(G)\hookrightarrow\mathbb{P}(G)^{2}$
and $\mathbb{STD}^{\Gamma}(G)\hookrightarrow\mathbb{F}^{\Gamma}(G)^{2}$
are surjective. More precisely, for any $S$-scheme $T=\Spec(k)$
with $k$ a field, 
\[
\mathbb{STD}(G)(k)=\mathbb{P}(G)^{2}(k)\quad\mbox{and}\quad\mathbb{STD}^{\Gamma}(G)(k)=\mathbb{F}^{\Gamma}(G)^{2}(k)
\]
by Bruhat's theorem \cite[XXVI 4.1.1]{SGA3.3r}.
\end{rem}

\subsection{The addition map on $\Gamma$-filtrations\label{sub:AdditionMap}}
\begin{prop}
There is an $S$-morphism 
\[
+:\mathbb{STD}^{\Gamma}(G)\rightarrow\mathbb{F}^{\Gamma}(G),\qquad(\mathcal{F},\mathcal{G})\mapsto\mathcal{F}+\mathcal{G}
\]
 such that for every $S$-scheme $T$, $(\mathcal{F},\mathcal{G})\in\mathbb{STD}^{\Gamma}(G)(T)$
and $\mathcal{H}\in\mathbb{F}^{\Gamma}(G)(T)$, 
\[
\mathcal{F}+\mathcal{G}=\mathcal{G}+\mathcal{F}\quad\mbox{and}\quad\mathcal{H}+0=0+\mathcal{H}=\mathcal{H}\quad\mbox{in}\quad\mathbb{F}^{\Gamma}(G)(T).
\]
\end{prop}
\begin{proof}
Since $(P_{\mathcal{F}},P_{\mathcal{G}})\in\mathbb{STD}(G)(T)$, there
is, locally on $T$ for the Zariski topology, a maximal torus $H$
of $G_{T}$ inside $P_{\mathcal{F}}\cap P_{\mathcal{G}}$ \cite[XXVI 4.5.1]{SGA3.3r}.
Let $L_{\mathcal{F}}$ and $L_{\mathcal{G}}$ be the Levi subgroups
of $P_{\mathcal{F}}$ and $P_{\mathcal{G}}$ containing $H$ \cite[XXVI 1.6]{SGA3.3r},
let $f:\mathbb{D}_{T}(\Gamma)\rightarrow L_{\mathcal{F}}$ and $g:\mathbb{D}_{T}(\Gamma)\rightarrow L_{\mathcal{G}}$
be the corresponding central morphisms lifting $\overline{\mathcal{F}}$
and $\overline{\mathcal{G}}$. Then $f$ and $g$ both factor through
$H$, and their product $f+g$ in the commutative group $H$ is a
group homomorphism $\mathbb{D}_{T}(\Gamma)\rightarrow G_{T}$. We
claim that $\mathcal{F}+\mathcal{G}=\Fi(f+g)$ does not depend upon
the choice of the maximal torus $H$ -- the whole construction is
then indeed local in the Zariski topology on $T$ as well as functorial
in the $S$-scheme $T$, and the resulting $S$-morphism $+:\mathbb{STD}^{\Gamma}(G)\rightarrow\mathbb{F}^{\Gamma}(G)$
obviously has the required properties. Let thus $H'$ be another maximal
torus of $G_{T}$ inside $K=P_{\mathcal{F}}\cap P_{\mathcal{G}}$,
giving rise to $f'$, $g'$ and $f'+g':\mathbb{D}_{T}(\Gamma)\rightarrow H'\subset G$.
Then, locally on $T$ for the étale topology, there is a $k\in K(T)$
such that $H'=\Int(k)(H)$ by \cite[XII 7.1]{SGA3.2}, in which case
also $f'=\Int(k)\circ f$, $g'=\Int(k)\circ g$ and $f'+g'=\Int(k)\circ(f+g)$.
It is thus sufficient to establish that $K\subset P_{f+g}$. This
second claim is again local in the étale topology on $T$, which reduces
us further to the following case: $(G,H,M,R)$ is a split group over
$S=T$ (i.e.~$H\simeq\mathbb{D}_{S}(M)$ and $R\subset M$ is the
set of roots of $H$ in $\Lie(G)$) with $f$ and $g$ respectively
induced by morphisms $f^{\sharp}$ and $g^{\sharp}:M\rightarrow\Gamma$,
so that $f+g$ is induced by $(f+g)^{\sharp}=f^{\sharp}+g^{\sharp}$.
For a closed subset $R'$ of $R$, we denote by $H_{R'}\supset H$
the corresponding subgroup of $G$ of type $(R)$, as in \cite[XXII 5.4]{SGA3.3r}
(thus $H=H_{\emptyset}$ and $G=H_{R}$). Then $P_{\mathcal{F}}=H_{R(f)}$,
$P_{\mathcal{G}}=H_{R(g)}$ and $P_{f+g}=H_{R(f+g)}$ where $R(h)=\{\alpha\in R:h^{\sharp}(\alpha)\geq0\}$
by definition of these parabolic subgroups of $G$. Thus $K=H_{R(f)\cap R(g)}$
is contained in $H_{R(f+g)}=P_{f+g}$ by \cite[XXII 5.4.5]{SGA3.3r}. \end{proof}
\begin{prop}
\label{prop:CompAdditionType}For any $S$-scheme $T$ and $(\mathcal{F},\mathcal{G})\in\mathbb{STD}^{\Gamma}(G)(T)$,
\[
t(\mathcal{F}+\mathcal{G})\preceq t(\mathcal{F})+t(\mathcal{G})\quad\mbox{in}\quad\mathbb{C}^{\Gamma}(G)(T)
\]
with equality if $\mathcal{F}$ and $\mathcal{G}$ are in osculatory
position.\end{prop}
\begin{proof}
We may assume that $T=s$ is a geometric point, with $\mathcal{F}$
and $\mathcal{G}$ lifting to morphisms $f,g:\mathbb{D}_{s}(\Gamma)\rightarrow H$
for some maximal (split) subtorus $H\simeq\mathbb{D}_{s}(M)$ of $G_{s}$,
corresponding to morphisms $f^{\sharp},g^{\sharp}:M\rightarrow\Gamma$
as above. Let $R\subset M$ be the roots of $H$ in $\Lie(G_{s})$.
By~\cite[XXI 3.3.6]{SGA3.3r}, there is a system of positive roots
$R_{+}\subset R$ such that $(f^{\sharp}+g^{\sharp})(R_{+})\subset\Gamma_{+}$,
i.e.~$f^{\sharp}+g^{\sharp}\in\Hom^{+}(M,\Gamma)$ in the notations
of section~\ref{sub:AppendixRoots}. Thus if $\vartheta:\Hom(M,\Gamma)\twoheadrightarrow\Hom^{+}(M,\Gamma)$
is the retraction from lemma~\ref{lem:DominantIsFundDom}, then $\vartheta(f^{\sharp}+g^{\sharp})=f^{\sharp}+g^{\sharp}$
and $f^{\sharp}\preceq\vartheta(f^{\sharp})$, $g^{\sharp}\preceq\vartheta(g^{\sharp})$
in $\Hom(M,\Gamma)$ by lemma~\ref{lem:CarDominantByDomOrder}, therefore
$\vartheta(f^{\sharp}+g^{\sharp})\preceq\vartheta(f^{\sharp})+\vartheta(g^{\sharp})$
in $\Hom^{+}(M,\Gamma)$, i.e.~$t(\mathcal{F}+\mathcal{G})\preceq t(\mathcal{F})+t(\mathcal{G})$
in $\mathbb{C}^{\Gamma}(G)(s)$ by definition. If $\mathcal{F}$ and
$\mathcal{G}$ are in osculatory position, there is a Borel subgroup
$B$ of $G_{s}$ inside $P_{\mathcal{F}}\cap P_{\mathcal{G}}$ \cite[XXVI 4.4.1]{SGA3.3r}.
We may then take $H$ inside $B$ and for $R_{+}$, the roots of $H$
in $\Lie(B)$, so that $f^{\sharp}$ and $g^{\sharp}$ already belong
to $\Hom^{+}(M,\Gamma)$, $\vartheta(f^{\sharp}+g^{\sharp})=\vartheta(f^{\sharp})+\vartheta(g^{\sharp})$
and indeed $t(\mathcal{F}+\mathcal{G})=t(\mathcal{F})+t(\mathcal{G})$.
\end{proof}

\subsection{~\label{sub:FunctLevi}}

We record here a special case of the functoriality of $\mathbb{F}^{\Gamma}(-)$.
\begin{prop}
\label{prop:FunctLevi}Let $L$ be a Levi subgroup of a parabolic
subgroup $P$ of $G$ with unipotent radical $U$. Then: $(1)$ for
$f:\mathbb{D}_{S}(\Gamma)\rightarrow L$ inducing $g:\mathbb{D}_{S}(\Gamma)\rightarrow G$
and $h:\mathbb{D}_{S}(\Gamma)\rightarrow P/U$, the parabolic subgroups
$P$ and $P_{g}$ of $G$ are in standard relative position, $K=P\cap P_{g}$
is a smooth subgroup scheme of $G$, $K\cdot U$ is a parabolic subgroup
of $G$ with Levi $L_{f}$, $K\cap L=(K\cdot U)\cap L=P_{f}$ and
$K\cdot U/U=P_{h}$ in $P/U$. $(2)$ There is a unique morphism $\iota:\mathbb{F}^{\Gamma}(L)\rightarrow\mathbb{F}^{\Gamma}(G)$
such that the diagram
\[
\xymatrix{\mathbb{G}^{\Gamma}(L)\ar@{^{(}->}[r]\ar[d]_{\Fi} & \mathbb{G}^{\Gamma}(G)\ar[d]^{\Fi}\\
\mathbb{F}^{\Gamma}(L)\ar[r]^{\iota} & \mathbb{F}^{\Gamma}(G)
}
\]
is commutative.\end{prop}
\begin{proof}
Everything in $(1)$ is local for the fpqc topology on $S$. We may
thus assume that \nomenclature[Z_G(x)]{$Z_G(x)$}{Centralizer of $x$ in $G$.}$L_{f}=Z_{L}(f)$
and $G$ are split with respect to a maximal torus $H$ of $G$ contained
in $L_{f}$ with $H$ trivial, i.e.~$H=\mathbb{D}_{S}(M)$ for some
finitely generated abelian group $M$~\cite[XXII 2.3]{SGA3.3r}.
Then $H\subset L_{f}=L\cap L_{g}\subset P\cap P_{g}$, thus $P$ and
$P_{g}$ are in standard relative position, $K=P\cap P_{g}$ is a
smooth subgroup of $G$ and $K\cdot U$ is a parabolic subgroup of
$G$ by~\cite[XXVI 4.5.1]{SGA3.3r} and its proof. More precisely,
let $R\subset M$ be the roots of $H$ in $\Lie(G)$, so that $R=R_{L}\coprod R_{U}\coprod-R_{U}$
where $R_{L}$ and $R_{U}$ are respectively the roots of $H$ in
$\Lie(L)$ and $\Lie(U)$. For $X$ in $\{\emptyset,L,U\}$ let $R_{X}=R_{X}^{0}\coprod R_{X}^{+}\coprod R_{X}^{-}$
be the decomposition of $R_{X}$ induced by $g$, i.e. 
\[
R_{X}^{\pm}=\{\alpha\in R_{X}:\pm\alpha\circ g>0\mbox{ in }\Gamma\}\quad\mbox{and}\quad R_{X}^{0}=\{\alpha\in R_{X}:\alpha\circ g=0\mbox{ in }\Gamma\}.
\]
This yields a decomposition of $R$ in nine pieces, as shown in the
following table:

\bigskip{}
\hfill %
\begin{tabular}{c|ccc}
$ $ & $L_{g}$ & $U_{g}$ & $U_{\iota g}$\tabularnewline
\hline 
$L$ & $R_{L}^{0}$ & $R_{L}^{+}$ & $R_{L}^{-}=-R_{L}^{+}$\tabularnewline
$U$ & $R_{U}^{0}$ & $R_{U}^{+}$ & $R_{U}^{-}$\tabularnewline
$ $ & $-R_{U}^{0}$ & $-R_{U}^{-}$ & $-R_{U}^{+}$\tabularnewline
\end{tabular}\hfill ~

\bigskip{}
\noindent For a closed subset $R'$ of $R$, let $H(R')$ be the
subgroup scheme of $G$ of type $(R)$ which is determined by $R'$,
see~\cite[XXII 5.4.2-7]{SGA3.3r}. Thus $L=H(R_{L})$, $P=H(R_{L}\cup R_{U})$,
$L_{g}=H(R^{0})$, $P_{g}=H(R^{0}\cup R^{+})$, $L_{f}=H(R_{L}^{0})$
and $P_{f}=H(R_{L}^{0}\cup R_{L}^{+})$ while
\begin{eqnarray*}
K & = & H(R_{L}^{0}\cup R_{L}^{+}\cup R_{U}^{0}\cup R_{U}^{+})\\
\mbox{and}\quad K\cdot U & = & H(R_{L}^{0}\cup R_{L}^{+}\cup R_{U}^{0}\cup R_{U}^{+}\cup R_{U}^{-}).
\end{eqnarray*}
By~\cite[XXVI 6.1]{SGA3.3r}, $L_{f}=H(R_{L}^{0})$ is a Levi subgroup
of $K\cdot U$. By \cite[XXII 5.4.5]{SGA3.3r}, $P_{f}\subset K$,
thus $P_{f}\subset K\cap L\subset(K\cdot U)\cap L$. But $(K\cdot U)\cap L$
is a parabolic subgroup of $L$ with Levi $L_{f}$, thus $P_{f}=K\cap L=(K\cdot U)\cap L$
and $P_{f}\cdot U=K\cdot U$ by repeated applications of \cite[XXVI 1.20]{SGA3.3r}.
Finally, $P_{f}$ maps to $P_{h}=P_{f}\cdot U/U=K\cdot U/U$ under
the isomorphism $L\simeq P/U$, which finishes the proof of $(1)$.
Then $(2)$ easily follows: if $(f,f')$ induce $(g,g')$ and $\Fi(f)=\Fi(f')$,
then $f'=\Int(p)\circ f$ for some $p\in P_{f}(S)$, thus also $g'=\Int(p)\circ g$
and $\Fi(g')=\Fi(g)$ since $P_{f}=L\cap P_{g}\subset P_{g}$. 
\end{proof}

\subsection{~\label{sub:defofGrMap}}

Let $G'=P_{u}/U_{u}$ where $P_{u}\subset G_{\mathbb{P}(G)}$ is the
universal parabolic subgroup with unipotent radical $U_{u}=R^{u}(P_{u})$.
Thus $G'$ is a reductive group over $\mathbb{P}(G)$ and 
\[
\mathbb{F}^{\Gamma}(G')(T)=\left\{ (P,\mathcal{F}):P\in\mathbb{P}(G)(T),\,\mathcal{F}\in\mathbb{F}^{\Gamma}(P/U)(T),\mbox{ }U=R^{u}(P)\right\} 
\]
for any $S$-scheme $T$. 
\begin{prop}
\label{prop:ExistCanoMapGrP}There is a canonical morphism of schemes
over $\mathbb{P}(G)$,\nomenclature[Gr_P(F)]{$\Gr_P (\mathcal{F})$}{$\Gamma$-filtration on $P/U$ induced by a $\Gamma$-filtration $\mathcal{F}$ on $G$ in standard relative position with a parabolic subgroup $P$ of $G$, page \nomrefpage}
\[
\mathbb{STD}\left(\mathbb{P}(G)\times_{S}\mathbb{F}^{\Gamma}(G)\right)\rightarrow\mathbb{F}^{\Gamma}(G')\qquad(P,\mathcal{F})\mapsto(P,\Gr_{P}(\mathcal{F})).
\]
\end{prop}
\begin{proof}
Start with $(P,\mathcal{F})\in\mathbb{STD}\left(\mathbb{P}(G)\times_{S}\mathbb{F}^{\Gamma}(G)\right)(T)$
and put $K=P\cap P_{\mathcal{F}}$. Then $K$ is a smooth subgroup
scheme of $G_{T}$ which contains, locally on $T$ for the Zariski
topology, a maximal subtorus $H$ of $G_{T}$ \cite[XXVI 4.5.1]{SGA3.3r}.
Let $L$ and $L_{\mathcal{F}}$ be the Levi subgroups of $P$ and
$P_{\mathcal{F}}$ containing $H$ \cite[XXVI 1.6]{SGA3.3r}. Let
$f:\mathbb{D}_{T}(\Gamma)\rightarrow L_{\mathcal{F}}$ be the central
morphism lifting $\overline{\mathcal{F}}:\mathbb{D}_{T}(\Gamma)\rightarrow\overline{R}(P_{\mathcal{F}})$,
so that $\mathcal{F}=\Fi(f)$ and $L_{f}=L_{\mathcal{F}}$. Then $f$
factors through the maximal subtorus $H$ of $L_{\mathcal{F}}$, which
is also a maximal subtorus of $L$. Let $h:\mathbb{D}_{T}(\Gamma)\rightarrow P/U$
be the induced morphism. By the previous proposition, $P_{h}=K\cdot U/U$,
thus $K$ fixes $\Fi(h)\in\mathbb{F}^{\Gamma}(P/U)(T)$. If $H'$
is another maximal subtorus of $G$ contained in $K$, then, locally
on $T$ for the étale topology, $H'=\Int(k)(H)$ for some $k\in K(T)$
by \cite[XII 7.1]{SGA3.2}. But then $L'=\Int(k)(L)$, $L'_{\mathcal{F}}=\Int(k)(L_{\mathcal{F}})$,
$f'=\Int(k)\circ f$ and $h'=\Int(k)\circ h$ are the objects associated
to $H'$ as above, thus $\Fi(h')=k\cdot\Fi(h)=\Fi(h)$ since $K$
fixes $\Fi(h)$. It follows that $\Gr_{P}(\mathcal{F})=\Fi(h)$ does
not depend upon the choice of $H$, and also that the whole construction
is indeed local in the Zariski topology on $T$. \end{proof}
\begin{rem}
\label{rk:sectionOfGrP}The pull-back of this morphism through $p_{1}:\mathbb{OPP}(G)\rightarrow\mathbb{P}(G)$
has a canonical section: for an $S$-scheme $T$, the latter is given
by the formula
\[
(P_{1},P_{2},\mathcal{F})\mapsto(P_{1},P_{2},\iota(\mathcal{F}_{L}))
\]
where $(P_{1},P_{2})=(U_{1}\rtimes L,U_{2}\rtimes L)$ is a pair of
opposed parabolic subgroups of $G_{T}$ with common Levi subgroup
$L=P_{1}\cap P_{2}$, $\mathcal{F}$ is an element of $\mathbb{F}^{\Gamma}(P_{1}/U_{1})(T)$,
$\mathcal{F}_{L}$ is its unique lift in $\mathbb{F}^{\Gamma}(L)(T)$,
and $\iota:\mathbb{F}^{\Gamma}(L)\rightarrow\mathbb{F}^{\Gamma}(G_{T})$
is the morphism of proposition~\ref{prop:FunctLevi} (thus indeed
$P_{1}$ and $P_{\iota(\mathcal{F}_{L})}$ are in standard relative
position).
\end{rem}

\section{\label{sub:AppendixRoots}Interlude on the dominance partial orders}

Let $\Gamma=(\Gamma,+,\leq)$ be a non-trivial totally ordered commutative
group. We set\nomenclature[Gamma_+]{$\Gamma _+$}{$\Gamma_{+}=\{\gamma\in\Gamma:\gamma\geq0\}$.}
\[
\Gamma_{+}=\{\gamma\in\Gamma:\gamma\geq0\}.
\]

\subsection{~}

Let $\mathscr{R}=(M,R,M^{\ast},R^{\ast})$ be a root system \cite[XXI 1.1.1]{SGA3.3r}
with Weyl group $W=W(\mathscr{R})$ \cite[XXI 1.1.8]{SGA3.3r}. Fix
a system of positive roots $R_{+}\subset R$ \cite[XXI 3.2.1]{SGA3.3r}
and let $\Delta\subset R_{+}$ be the corresponding simple roots \cite[XXI 3.2.8]{SGA3.3r}.
Then
\begin{lem}
\label{lem:DominantIsFundDom}The submonoid of dominant morphisms
in $\Hom(M,\Gamma)$,\nomenclature[Hom^+(M,Gamma)]{$\Hom ^+ (M,\Gamma)$}{Dominant morphisms in $\Hom (M, \Gamma)$.}
\begin{eqnarray*}
\Hom^{+}(M,\Gamma) & = & \left\{ f\in\Hom(M,\Gamma):\forall\alpha\in R_{+},\, f(\alpha)\geq0\right\} \\
 & = & \left\{ f\in\Hom(M,\Gamma):\forall\alpha\in\Delta,\, f(\alpha)\geq0\right\} 
\end{eqnarray*}
is a fundamental domain for the action of $W$ on $\Hom(M,\Gamma)$.\end{lem}
\begin{proof}
For any morphism $f:M\rightarrow\Gamma$, define 
\begin{eqnarray*}
R_{f\geq0} & = & \left\{ \alpha\in R:f(\alpha)\geq0\right\} ,\\
R_{f>0} & = & \left\{ \alpha\in R:f(\alpha)>0\right\} ,\\
R_{f=0} & = & \left\{ \alpha\in R:f(\alpha)=0\right\} .
\end{eqnarray*}
Thus $R_{f=0}$ is closed and symmetric, $R=R_{f>0}\coprod R_{f=0}\coprod-R_{f>0}$
and 
\[
f\mbox{ is dominant}\iff R_{+}\subset R_{f\geq0}\iff R_{f>0}\subset R_{+}.
\]
By \cite[XXI 3.3.6]{SGA3.3r}, there exists $w\in W$ such that 
\[
R_{+}\subset wR_{f\geq0}=R_{wf\geq0},
\]
therefore $wf$ is dominant. If $f$ and $wf$ are dominant, then
\[
R_{+}=R_{f>0}\coprod R_{+}^{1}\quad\mbox{and}\quad w^{-1}R_{+}=R_{f>0}\coprod R_{+}^{2}
\]
where $R_{+}^{1}=R_{+}\cap R_{f=0}$ and $R_{+}^{2}=w^{-1}R_{+}\cap R_{f=0}$
are systems of positive roots in the closed symmetric subset $R_{f=0}$
of $R$. Thus by \cite[XXI 3.4.1 and 3.3.7]{SGA3.3r}, there is an
$w_{0}$ in the Weyl group $W_{f}\subset W$ of $R_{f=0}$ such that
$w_{0}R_{+}^{2}=R_{+}^{1}$. Now $W_{f}$ is spanned by the reflections
$\left\{ s_{\alpha}:\alpha\in R_{f=0}\right\} $ and for any $m\in M$
and $\alpha\in R_{f=0}$, 
\[
(s_{\alpha}f)(m)=f(s_{\alpha}m)=f\left(m-\left\langle m,\alpha^{\ast}\right\rangle \alpha\right)=f(m)-\left\langle m,\alpha^{\ast}\right\rangle f(\alpha)=f(m),
\]
thus $w_{0}$ fixes $f$, stabilizes $R_{f>0}$ and maps $w^{-1}R_{+}$
to $R_{+}$. But then $w=w_{0}$ by \cite[XXI 5.4]{SGA3.3r} hence
$wf=w_{0}f=f$, which proves the lemma.
\end{proof}

\subsection{~\label{sub:MdisFundDomain}}

Applying the lemma to the dual root system $\mathscr{R}^{\ast}=(M^{\ast},R^{\ast},M,R)$
with $\Gamma=\mathbb{Z}$, we obtain the well known fact that the
cone of dominant weights\nomenclature[M_d]{$M_d$}{Dominant weights.}
\begin{eqnarray*}
M_{d} & = & \left\{ m\in M:\forall\alpha\in R_{+},\,\left\langle m,\alpha^{\ast}\right\rangle \geq0\right\} \\
 & = & \left\{ m\in M:\forall\delta\in\Delta,\,\left\langle m,\delta^{\ast}\right\rangle \geq0\right\} 
\end{eqnarray*}
is a fundamental domain for the action of $W$ on $M$.

\subsection{~}

The coroot cone and coroot lattice defined by\nomenclature[Gamma R_+^*]{$\Gamma _+ R _+ ^\ast$}{$\Gamma_+$-cone spanned by the positive coroots, page \nomrefpage}\nomenclature[Gamma R^*]{$\Gamma  R ^\ast$}{$\Gamma$-subgroup spanned by the coroots, page \nomrefpage}
\begin{eqnarray*}
\Gamma_{+}R_{+}^{\ast} & = & \left\{ m\mapsto{\textstyle \sum}_{\alpha\in R_{+}}\left\langle m,\alpha^{\ast}\right\rangle \gamma_{\alpha}:\forall\alpha\in R_{+},\,\gamma_{\alpha}\in\Gamma_{+}\right\} \\
\Gamma R^{\ast} & = & \left\{ m\mapsto{\textstyle \sum}_{\alpha\in R}\left\langle m,\alpha^{\ast}\right\rangle \gamma_{\alpha}:\forall\alpha\in R,\,\gamma_{\alpha}\in\Gamma\right\} 
\end{eqnarray*}
are a submonoid and a subgroup of $\Hom(M,\Gamma)$, and so are their
saturations\nomenclature[Gamma R_+^*_sat]{$(\Gamma _+ R _+ ^\ast)_{sat}$}{Saturation of $\Gamma_+ R_+ ^\ast$, page \nomrefpage}\nomenclature[Gamma R^*_sat]{$(\Gamma  R  ^\ast)_{sat}$}{Saturation of $\Gamma R^\ast$, page \nomrefpage}
\begin{eqnarray*}
\left(\Gamma_{+}R_{+}^{\ast}\right)_{sat} & = & \left\{ f\in\Hom(M,\Gamma):\exists n\in\mathbb{N}^{\times}\mbox{\,\ such that }nf\in\Gamma_{+}R_{+}^{\ast}\right\} \\
\left(\Gamma R^{\ast}\right)_{sat} & = & \left\{ f\in\Hom(M,\Gamma):\exists n\in\mathbb{N}^{\times}\mbox{ such that }nf\in\Gamma R^{\ast}\right\} 
\end{eqnarray*}
in $\Hom(M,\Gamma)$. Inside $\Hom(M,\Gamma\otimes\mathbb{Q})$, any
$f\in(\Gamma R^{\ast})_{sat}$ can be written as
\[
f(-)={\textstyle \sum}_{\delta\in\Delta}\left\langle -,\mathrm{ind}(\delta^{\ast})\right\rangle \gamma_{\delta}
\]
for a unique $(\gamma_{\delta})\in(\Gamma\otimes\mathbb{Q})^{\Delta}$,
where $\mathrm{ind}(\delta^{\ast})$ is the simple coroot corresponding
to $\delta\in\Delta$, namely $\mathrm{ind}(\delta^{\ast})=\delta^{\ast}$
if $2\delta\notin R$ and $\mathrm{ind}(\delta^{\ast})=\frac{1}{2}\delta^{\ast}=(2\delta)^{\ast}$
otherwise. Then
\begin{eqnarray*}
f\in(\Gamma_{+}R_{+}^{\ast})_{sat} & \iff & \exists n\in\mathbb{N}^{\times}\mbox{ such that }n(\gamma_{\delta})\in\Gamma_{+}^{\Delta},\\
f\in\Gamma R^{\ast} & \iff & (\gamma_{\delta})\in\Gamma^{\Delta},\\
f\in\Gamma_{+}R_{+}^{\ast} & \iff & (\gamma_{\delta})\in\Gamma_{+}^{\Delta}.
\end{eqnarray*}
In particular, $(\Gamma_{+}R_{+}^{\ast})_{sat}\cap-(\Gamma_{+}R_{+}^{\ast})_{sat}=\{0\}$.
Moreover by duality, 
\[
\left(\Gamma_{+}R_{+}^{\ast}\right)_{sat}=\left\{ f\in\Hom(M,\Gamma):\forall m\in M_{d},\, f(m)\geq0\right\} .
\]

\subsection{~}

The weak dominance partial order $\leq$ on $\Hom(M,\Gamma)$ is defined
by
\begin{eqnarray*}
f_{1}\leq f_{2} & \iff & \forall m\in M_{d}:\quad f_{1}(m)\leq f_{2}(m),\\
 & \iff & f_{2}-f_{1}\in\left(\Gamma_{+}R_{+}^{\ast}\right)_{sat}.
\end{eqnarray*}
The strong dominance partial order $\preceq$ on $\Hom(M,\Gamma)$
is defined by
\[
f_{1}\preceq f_{2}\iff f_{2}-f_{1}\in\Gamma_{+}R_{+}^{\ast}.
\]
They are both compatible with the addition map: for $f_{1},f_{2},g_{1},g_{2}\in\Hom(M,\Gamma)$,
\begin{eqnarray*}
\left(f_{1}\leq g_{1}\mbox{ and }f_{2}\leq g_{2}\right) & \Longrightarrow & f_{1}+f_{2}\leq g_{1}+g_{2},\\
\left(f_{1}\preceq g_{1}\mbox{ and }f_{2}\preceq g_{2}\right) & \Longrightarrow & f_{1}+f_{2}\preceq g_{1}+g_{2}.
\end{eqnarray*}
They are related as follows: for any $f_{1},f_{2}\in\Hom(M,\Gamma)$,
we have 
\begin{eqnarray*}
f_{1}\preceq f_{2} & \iff & f_{1}\leq f_{2}\quad\mbox{and}\quad\pi(f_{1})=\pi(f_{2})
\end{eqnarray*}
where $\pi:\Hom(M,\Gamma)\twoheadrightarrow\Hom(M,\Gamma)/\Gamma R^{\ast}$
is the projection. Note that 
\[
f_{1}\leq f_{2}\Longrightarrow\pi(f_{2}-f_{1})\in(\Gamma R^{\ast})_{sat}/\Gamma R^{\ast}.
\]
In particular since $(\Gamma R^{\ast})_{sat}/\Gamma R^{\ast}$ is
torsion, 
\[
f_{1}\leq f_{2}\iff\exists n\in\mathbb{N}^{\times}:\quad nf_{1}\preceq nf_{2}.
\]

\subsection{~}

Since $WM_{d}=M$, both partial orders restrict to the identity on
the fixed point set of $W$ in $\Hom(M,\Gamma)$: for any $W$-invariant
$f_{1},f_{2}\in\Hom(M,\Gamma)$,
\[
\xymatrix{f_{1}\preceq f_{2}\ar@{=>}[r] & f_{1}\leq f_{2}\ar@{=>}[r] & \forall m\in M_{d}:\, f_{1}(m)\leq f_{2}(m)\ar@{=>}[d]\\
f_{1}=f_{2}\ar@{=>}[u] &  & \forall m\in M:\, f_{1}(m)\leq f_{2}(m)\ar@{=>}[ll]
}
\]

\subsection{~}

These partial orders yield the following characterization of $\Hom^{+}(M,\Gamma)$: 
\begin{lem}
\label{lem:CarDominantByDomOrder}The projection $\pi$ is $W$-invariant
and for every $f\in\Hom(M,\Gamma)$, 
\begin{eqnarray*}
f\in\Hom^{+}(M,\Gamma) & \iff & \forall w\in W:\quad wf\leq f,\\
 & \iff & \forall w\in W:\quad wf\preceq f.
\end{eqnarray*}
In particular, $\Hom(M,\Gamma)^{W}\subset\Hom^{+}(M,\Gamma)$.\end{lem}
\begin{proof}
For any $f\in\Hom(M,\Gamma)$ and $\alpha\in R$, 
\[
f-s_{\alpha}f=\left\langle -,\alpha^{\ast}\right\rangle f(\alpha)\quad\mbox{in}\quad\Hom(M,\Gamma).
\]
Thus $\pi$ is $W$-invariant, $wf\leq f\iff wf\preceq f$ for $w\in W$
and
\begin{eqnarray*}
f\in\Hom^{+}(M,\Gamma) & \iff & \forall\alpha\in R_{+}:\quad s_{\alpha}f\preceq f.
\end{eqnarray*}
It remains to establish that
\begin{eqnarray*}
\forall\alpha\in R_{+}:\quad s_{\alpha}f\preceq f & \Longrightarrow & \forall w\in W:\quad wf\preceq f
\end{eqnarray*}
and we argue by induction on the length $\ell(w)$ of $w$ in the
coxeter group $(W,(s_{\alpha})_{\alpha\in\Delta})$. If $\ell(w)>1$,
then $w=w's_{\alpha}$ for some $\alpha\in\Delta$, $w'\in W$ with
$\ell(w')<\ell(w)$. Thus
\[
f-wf=\left(f-w'f\right)+w'\left(f-s_{\alpha}f\right)=\left(f-w'f\right)+\left\langle -,w'\alpha^{\ast}\right\rangle f(\alpha).
\]
Now $f-w'f\in\Gamma_{+}R_{+}^{\ast}$ by induction, $f(\alpha)\geq0$
by assumption and $w'\alpha=-w\alpha\in R^{+}$ by \cite[VI, \S 1, $\mathrm{n}^\circ 1.6$, Corollaire 2]{BoLie46},
therefore $f-wf\in\Gamma_{+}R_{+}^{\ast}$, i.e.~$wf\preceq f$. 
\end{proof}

\subsection{~\label{sub:minimalset}}

If $\Gamma$ is (uniquely) divisible, the weak and strong dominance
order coincide. Moreover, for any $f\in\Hom^{+}(M,\Gamma)$, lemma~\ref{lem:CarDominantByDomOrder}
implies that $f^{\flat}\leq f$ where 
\[
f^{\flat}={\textstyle \frac{1}{\sharp Wf}\sum_{f'\in Wf}}f'\in\Hom(M,\Gamma)^{W}\subset\Hom^{+}(M,\Gamma).
\]
Thus $\Hom(M,\Gamma)^{W}$ is then precisely the set of minimal elements
in $\Hom^{+}(M,\Gamma)$. For $\Gamma=\mathbb{Z}$ and $\mathscr{R}$
reduced, semi-simple and adjoint (i.e.~$\mathbb{Z}R=M$), the strong
dominance order on $\Hom^{+}(M,\mathbb{Z})$ is studied in \cite{St98}.
Its minimal elements are the linear forms $f:M\rightarrow\mathbb{Z}$
such that $f(R_{+})\in\{0,1\}$.

\subsection{~\label{sub:carDomCoweightsByWeyl}}

Applying lemma~\ref{lem:CarDominantByDomOrder} to the dual root
system $\mathscr{R}^{\ast}$ with $\Gamma=\mathbb{Z}$, we obtain:
$(1)$ Let $M'$ be the kernel of the coinvariant map $M\twoheadrightarrow M_{W}$.
Then $M'\subset\mathbb{Z}R$. Since also $\alpha\equiv-\alpha$ in
$M_{W}$ for every $\alpha\in R$, actually $2\mathbb{Z}R\subset M'\subset\mathbb{Z}R$,
therefore $M'$ and $\mathbb{Z}R$ have the same saturation in $M$.
And: $(2)$ For every $m\in M$, 
\[
m\in M_{d}\iff\forall w\in W:\quad m-wm\in\mathbb{N}R_{+}\quad(\mbox{or: }(\mathbb{N}R_{+})_{sat}).
\]
Returning to the original root system $\mathscr{R}$ and the general
$\Gamma$, we thus find: 
\[
\Hom(M,\Gamma)^{W}=\Hom(M/M',\Gamma)=\Hom(M/\mathbb{Z}R,\Gamma)=\Hom(M/(\mathbb{Z}R)_{sat},\Gamma)
\]
and for every $f\in\Hom^{+}(M,\Gamma)$ and $m\in M_{d}$, 
\[
f(m)=\max f(Wm)\quad\mbox{in}\quad\Gamma.
\]

\subsection{~}

This last property yields the following characterisation of the restriction
of the weak order to the cone $\Hom^{+}(M,\Gamma)$. Any morphism
$f:M\rightarrow\Gamma$ induces a ring homomorphism $f:\mathbb{Z}[M]\rightarrow\mathbb{Z}[\Gamma]$.
For $x\in\mathbb{Z}[\Gamma]$, we denote by $\max(x)\in\Gamma$ the
largest element in the finite support of $x$ if $x\neq0$, and set
$\max(0)=0$. Then:
\begin{lem}
\label{lem:CarDominantOrderByMax}For $f_{1},f_{2}\in\Hom^{+}(M,\Gamma)$,
\begin{eqnarray*}
f_{1}\leq f_{2} & \iff & \forall x\in\mathbb{N}[M]^{W},:\quad\max\left(f_{1}(x)\right)\leq\max\left(f_{2}(x)\right).
\end{eqnarray*}
\end{lem}
\begin{proof}
Since $M_{d}$ is a fundamental domain for the action of $W$ on $M$,
\[
\mathbb{N}[M]^{W}=\left\{ {\textstyle x=\sum_{m\in M_{d}}x_{m}e_{m}:x_{m}\in\mathbb{N},\,\{x_{m}\neq0\}\mbox{\,\ finite}}\right\} 
\]
where $e_{m}=\sum_{m'\in Wm}m'$. For any $f:M\rightarrow\Gamma$
and $x\in\mathbb{N}[M]^{W}$ with $x\neq0$, $f(x)$ is also nonzero
with support $\cup_{m\in M_{d},x_{m}\neq0}f(Wm)$. Thus if $f$ is
moreover dominant, 
\[
\max\left(f(x)\right)=\max\left\{ f(m):m\in M_{d},\, x_{m}\neq0\right\} .
\]
The lemma easily follows. 
\end{proof}

\subsection{~\label{sub:WeylConeIsProductAbstract}}

Let $\mathscr{R}_{ss}=(M_{ss},R_{ss},M_{ss}^{\ast},R_{ss}^{\ast})$
be the semi-simplification of $\mathscr{R}$, as defined in~\cite[XXI 6.5]{SGA3.3r}.
Thus $M_{ss}=\mathbb{Z}R_{sat}$, $R_{ss}=R$, $M_{ss}^{\ast}$ is
the dual of $M_{ss}$ and $R_{ss}^{\ast}$ is the image of $R^{\ast}$
under the transpose map $M^{\ast}\twoheadrightarrow M_{ss}^{\ast}$.
The restriction map $f\mapsto f_{ss}=f\vert M_{ss}$ yields an epimorphism
$\Hom(M,\Gamma)\twoheadrightarrow\Hom(M_{ss},\Gamma)$ with kernel
$\Hom(M,\Gamma)^{W}$, inducing epimorphism of monoids $\Hom^{+}(M,\Gamma)\twoheadrightarrow\Hom^{+}(M_{ss},\Gamma)$
and $\Gamma_{+}R_{+}^{\ast}\twoheadrightarrow\Gamma_{+}R_{ss,+}^{\ast}$
-- the former is therefore also compatible with the weak and strong
partial orders. If $\Gamma$ is divisible, the average map $f\mapsto f^{\flat}$
of section~\ref{sub:minimalset} gives a retraction of $\Hom(M,\Gamma)^{W}\hookrightarrow\Hom(M,\Gamma)$,
and it follows that $f\mapsto(f_{ss},f^{\flat})$ yields an isomorphism
of partially ordered monoids 
\[
\Hom^{+}(M,\Gamma)\simeq\Hom^{+}(M_{ss},\Gamma)\times\Hom(M,\Gamma)^{W}.
\]
The (weak=strong) partial order on the product is then given by 
\[
(f_{ss},f^{\flat})\leq(g_{ss},g^{\flat})\iff f_{ss}\leq g_{ss}\quad\mbox{and}\quad f^{\flat}=g^{\flat}.
\]

\subsection{~\label{sub:CaractDomByConvEnv}}

If $\Gamma=\mathbb{R}$, then $\Hom^{+}(M,\mathbb{R})$ is a closed
cone in the finite dimensional $\mathbb{R}$-vector space $\Hom(M,\mathbb{R})$.
The (weak=strong) dominance partial order then has the following intrinsic
characterisation: for every $f_{1},f_{2}\in\Hom^{+}(M,\mathbb{R})$,
\[
f_{1}\leq f_{2}\quad\iff\quad f_{1}\mbox{\,\ lies in the convex hull of }W\cdot f_{2}.
\]
Indeed, suppose first that $f_{1}\leq f_{2}$. If $f_{1}$ does not
belong to the convex hull of $W\cdot f_{2}$, there is a linear form
$F$ on $\Hom(M,\mathbb{R})$ such that $F(f_{1})>F(wf_{2})$ for
every $w\in W$, which means that there is an $x$ in $M\otimes\mathbb{R}$
such that $f_{1}(x)>f_{2}(wx)$ for every $w\in W$. Since $M\otimes\mathbb{Q}$
is dense in $\mathbb{R}$, we may assume that $x\in M\otimes\mathbb{Q}$,
and then rescaling that actually $x\in M$. Let $y=wx$ be the unique
element in $Wx\cap M_{d}$~(\ref{sub:MdisFundDomain}). Then $f_{1}(y)\geq f_{1}(x)$
since $f_{1}\in\Hom^{+}(M,\mathbb{R})$ (\ref{sub:carDomCoweightsByWeyl})
and $f_{1}(x)>f_{2}(y)$ by construction, thus $f_{1}(y)>f_{2}(y)$
with $y\in M_{d}$, a contradiction. Suppose conversely that 
\[
f_{1}={\textstyle \sum_{w\in W}}\lambda_{w}wf_{2}\mbox{ in }\Hom(M,\mathbb{R})\mbox{ with }\lambda_{w}\in[0,1],\,{\textstyle \sum}_{w\in W}\lambda_{w}=1.
\]
Since $f_{2}\in\Hom^{+}(M,\mathbb{R})$, $wf_{2}\leq f_{2}$ for every
$w\in W$ by lemma~\ref{lem:CarDominantByDomOrder}, thus 
\[
f_{1}={\textstyle \sum_{w\in W}}\lambda_{w}wf_{2}\leq{\textstyle \sum_{w\in W}}\lambda_{w}f_{2}=f_{2}\mbox{ in }\Hom(M,\mathbb{R}).
\]

\subsection{~}

The partial orders on $W\backslash\Hom(M,\Gamma)$ which are induced
by the restriction of $\preceq$ and $\leq$ to the fundamental domain
$\Hom^{+}(M,\Gamma)\simeq W\backslash\Hom(M,\Gamma)$ of lemma~\ref{lem:DominantIsFundDom}
do not depend upon the chosen system of positive roots $R^{+}$ --
indeed, all such systems are conjugated under $W$. The weak order
even does not depend upon the root system giving rise to $W$: any
orbit $[f]\in W\backslash\Hom(M,\Gamma)$ yields a well-defined function
$[f]:\mathbb{Z}[M]^{W}\rightarrow\mathbb{Z}[\Gamma]$, and for every
$[f_{1}],[f_{2}]\in W\backslash\Hom(M,\Gamma)$, 
\[
[f_{1}]\leq[f_{2}]\quad\iff\quad\forall x\in\mathbb{N}[M]^{W}:\quad\max[f_{1}](x)\leq\max[f_{2}](x)
\]
by lemma~\ref{lem:CarDominantOrderByMax}. The strong order moreover
depends upon $\Gamma R^{\ast}$:
\[
[f_{1}]\preceq[f_{2}]\quad\iff\quad[f_{1}]\leq[f_{2}]\quad\mbox{and}\quad\pi[f_{1}]=\pi[f_{2}]
\]
where $\pi:W\backslash\Hom(M,\Gamma)\twoheadrightarrow\Hom(M,\Gamma)/\Gamma R^{\ast}$
is the projection from lemma~\ref{lem:CarDominantByDomOrder}.

\subsection{~\label{sub:startcompabsoluterelativedomorder}}

We record here a technical result comparing the partial orders attached
to, respectively, the relative and absolute root systems of a reductive
group $G$ over a field $k$. Let $S$ a maximal split torus in $G$,
$T$ a maximal torus in the centralizer $Z_{G}(S)$ of $S$, $k^{s}$
a separable closure of $k$, $\Gal_{k}=\Gal(k^{s}/k)$. Denote by
\[
\overline{\mathscr{R}}=\mathscr{R}\left(G_{k^{s}},T_{k^{s}}\right)=\left(\overline{M},\overline{R},\overline{M}^{\ast},\overline{R}^{\ast}\right)\quad\mbox{and}\quad\mathscr{R}=\mathscr{R}(G,S)=\left(M,R,M^{\ast},R^{\ast}\right)
\]
the absolute and relative root systems \cite[XXVI 7.12]{SGA3.3r},
with Weyl groups
\[
\overline{W}=W(\overline{\mathscr{R}})=W(G_{k^{s}},T_{k^{s}})\quad\mbox{and}\quad W=W(\mathscr{R})=W(G,S).
\]
We will also consider the subgroups $\overline{W}_{S}^{0}\subset\overline{W}_{S}\subset\overline{W}$
defined by
\[
\mathcal{N}_{G}(T)\cap\mathcal{Z}_{G}(S)/\mathcal{Z}_{G}(T)\subset\mathcal{N}_{G}(T)\cap\mathcal{N}_{G}(S)/\mathcal{Z}_{G}(T)\subset\mathcal{N}_{G}(T)/\mathcal{Z}_{G}(T).
\]
The embedding $S\hookrightarrow T$ induces a pair of dual morphisms
\[
\xyC{4pc}\xymatrix{\overline{M}\times\overline{M}^{\ast}\ar@<-3ex>@{->>}[d]_{\mathrm{res}}\ar[r]\sp(0.57){\left\langle -,-\right\rangle } & \mathbb{Z}\ar@{=}[d]\\
M\times M^{\ast}\ar@<-2ex>@{_{(}->}[u]_{\mathrm{res}^{\ast}}\ar[r]\sp(0.57){\left\langle -,-\right\rangle } & \mathbb{Z}
}
\]
with $R\subset\mathrm{res}(\overline{R})\subset R\cup\{0\}$. Set
$\overline{R}(0)=\left\{ \overline{\alpha}\in\overline{R}:\mathrm{res}(\overline{\alpha})=0\right\} $,
so that 
\[
\overline{\mathscr{R}}(0)=\mathscr{R}\left(Z_{G}(S)_{k^{s}},T_{k^{s}}\right)=\left(\overline{M},\overline{R}(0),\overline{M}^{\ast},\overline{R}(0)^{\ast}\right)
\]
is the absolute root system of $Z_{G}(S)$, with Weyl group 
\[
\overline{W}_{S}^{0}=W\left(\overline{\mathscr{R}}(0)\right)=W\left(Z_{G}(S)_{k^{s}},T_{k^{s}}\right).
\]
By \cite[5.5]{BoTi65}, the natural map $\overline{W}_{S}\rightarrow W$
identifies $W$ with $\overline{W}_{S}/\overline{W}_{S}^{0}$. Since
the restriction $\mathrm{res}:\overline{M}\twoheadrightarrow M$ is
equivariant with respect to $\overline{W}_{S}\twoheadrightarrow W$,
it induces a map 
\[
W\backslash\Hom\left(M,\Gamma\right)\hookrightarrow\overline{W}_{S}\backslash\Hom\left(\overline{M},\Gamma\right)\twoheadrightarrow\overline{W}\backslash\Hom\left(\overline{M},\Gamma\right).
\]

\begin{prop}
\label{prop:compdomorders}The map $W\backslash\Hom\left(M,\Gamma\right)\rightarrow\overline{W}\backslash\Hom\left(\overline{M},\Gamma\right)$
is injective. Moreover for any $f,g\in W\backslash\Hom\left(M,\Gamma\right)$
with image $\overline{f},\overline{g}\in\overline{W}\backslash\Hom\left(\overline{M},\Gamma\right)$,
\[
f\preceq g\Longrightarrow\overline{f}\preceq\overline{g}\Longrightarrow\overline{f}\leq\overline{g}\iff f\leq g.
\]
\end{prop}
\begin{rem}
We do not know if $\overline{f}\preceq\overline{g}$ implies $f\preceq g$.
The proof given below relates this to the following question. Recall
that $G$ is simply connected if and only if $\mathbb{Z}\overline{R}^{\ast}=\overline{M}^{\ast}$.
Is it true that then also $\mathbb{Z}R^{\ast}=M^{\ast}$? The dual
question has a positive answer: $G$ is adjoint if and only if $\mathbb{Z}\overline{R}=\overline{M}$,
in which case also $\mathbb{Z}R=M$. 
\end{rem}

\subsection{~}

Let $G_{der}\subset G$ be the derived group of $G$ \cite[XXII 6.2]{SGA3.3r},
$\pi:G_{sc}\twoheadrightarrow G_{der}$ the simply connected cover
of $G_{der}$ \cite[A.4.11]{CoGaPr10}. Then $T_{der}=T\cap G_{der}$
is a maximal torus in $G_{der}$ \cite[XXII 6.2.8]{SGA3.3r} and $T_{sc}=\pi^{-1}(T_{der})$
is a maximal torus in $G_{sc}$ \cite[XVII 7.1.1]{SGA3.2}. Let $S_{der}\subset T_{der}$
and $S_{sc}\subset T_{sc}$ be their maximal split subtori and denote
by $\mathscr{R}_{der}$, $\mathscr{R}_{sc}$, $\overline{\mathscr{R}}_{der}$
and $\overline{\mathscr{R}}_{sc}$ the corresponding relative and
absolute root systems. By~\cite[XXII 6.2.7]{SGA3.3r} and the definition
of $G_{sc}$, the morphisms $T_{sc}\twoheadrightarrow T_{der}\hookrightarrow T$
induce compatible bijections $\overline{R}\simeq\overline{R}_{der}\simeq\overline{R}_{sc}$
and $\overline{R}_{sc}^{\ast}\simeq\overline{R}_{der}^{\ast}\simeq\overline{R}^{\ast}$.
They also induce morphisms $S_{sc}\twoheadrightarrow S_{der}\hookrightarrow S$,
and $S=S_{der}\cdot R(G)_{sp}$ since $T=T_{der}\cdot R(G)$ \cite[XXII 6.2.8]{SGA3.3r},
where $R(G)_{sp}$ is the maximal split subtorus of the radical $R(G)$
of $G$. Since $R$, $R_{der}$ and $R_{sc}$ are the nonzero restrictions
of the elements of $\overline{R}$, $\overline{R}_{der}$ and $\overline{R}_{sc}$
to respectively $S$, $S_{der}$ and $S_{sc}$, it follows that our
morphisms also induce bijections $R\simeq R_{der}\simeq R_{sc}$.
Finally, the morphisms $G_{sc}\twoheadrightarrow G_{der}\hookrightarrow G$
induce embeddings $W_{sc}\hookrightarrow W_{der}\hookrightarrow W$
between the Weyl groups of the maximal tori $S_{sc}\subset G_{sc}$,
$S_{der}\subset G_{der}$ and $S\subset G$. It then follows from
the unicity of the relative coroots, or from their actual construction
in \cite[XXVI 7.4]{SGA3.3r}, that $S_{sc}\twoheadrightarrow S_{der}\hookrightarrow S$
also induces compatible bijections $R_{sc}^{\ast}\simeq R_{der}^{\ast}\simeq R^{\ast}$
(and the Weyl groups maps are bijective). Since composition with $S_{sc}\hookrightarrow T_{sc}$
maps $\mathbb{Z}R_{sc}^{\ast}$ into $X_{\ast}(T_{sc})=\mathbb{Z}\overline{R}_{sc}^{\ast}$,
we obtain 
\[
\mathrm{res}^{\ast}(\mathbb{Z}R^{\ast})\subset\mathbb{Z}\overline{R}^{\ast}.
\]

\subsection{~}

Fix a minimal parabolic subgroup $Z_{G}(S)\subset P\subset G$, a
Borel subgroup $T_{k^{s}}\subset B\subset P_{k^{s}}\subset G_{k^{s}}$,
let $\Delta\subset R_{+}\subset R$ and $\overline{\Delta}\subset\overline{R}_{+}\subset\overline{R}$
be the corresponding simple and positive roots, $\overline{\Delta}(0)\subset\overline{R}(0)_{+}\subset\overline{R}(0)$
the simple and positive roots attached to $Z_{G}(S)_{k^{s}}\cap B$,
so that $\overline{\Delta}(0)=\overline{\Delta}\cap\overline{R}(0)$,
$\overline{R}(0)_{+}=\overline{R}_{+}\cap\overline{R}(0)$, 
\[
R_{+}\subset\mathrm{res}\left(\overline{R}_{+}\right)\subset R_{+}\cup\{0\}\quad\mbox{and}\quad\Delta\subset\mathrm{res}\left(\overline{\Delta}\right)\subset\Delta\cup\{0\}.
\]
In particular, the morphism $\mathrm{res}_{\Gamma}^{\ast}:\Hom\left(M,\Gamma\right)\hookrightarrow\Hom\left(\overline{M},\Gamma\right)$
maps $\Hom^{+}\left(M,\Gamma\right)$ to $\Hom^{+}\left(\overline{M},\Gamma\right)$.
The first assertion of Proposition~\ref{prop:compdomorders} thus
follows from Lemma~\ref{lem:DominantIsFundDom}. For the remaining
claims, we have to establish the following inclusions: 
\[
\Gamma_{+}R_{+}^{\ast}\subset\left(\mathrm{res}_{\Gamma}^{\ast}\right)^{-1}\left(\Gamma_{+}\overline{R}_{+}^{\ast}\right)\subset\left(\mathrm{res}_{\Gamma}^{\ast}\right)^{-1}\left(\left(\Gamma_{+}\overline{R}_{+}^{\ast}\right)_{sat}\right)=\left(\Gamma_{+}R_{+}^{\ast}\right)_{sat}.
\]
We may assume that $\Gamma=\mathbb{Z}$, in which case $\mathrm{res}_{\Gamma}^{\ast}=\mathrm{res}^{\ast}:M^{\ast}\hookrightarrow\overline{M}^{\ast}$
and we want:
\[
\mathbb{N}R_{+}^{\ast}\subset\left(\mathrm{res}^{\ast}\right)^{-1}\left(\mathbb{N}\overline{R}_{+}^{\ast}\right)\subset\left(\mathrm{res}^{\ast}\right)^{-1}\left(\left(\mathbb{N}\overline{R}_{+}^{\ast}\right)_{sat}\right)=\left(\mathbb{N}R_{+}^{\ast}\right)_{sat}.
\]
The central inclusion is obvious. Since we already know that $\mathrm{res}^{\ast}(\mathbb{Z}R^{\ast})\subset\mathbb{Z}\overline{R}^{\ast}$
and 
\[
\mathbb{N}R_{+}^{\ast}=\mathbb{Z}R^{\ast}\cap\left(\mathbb{N}R_{+}^{\ast}\right)_{sat},\qquad\mathbb{N}\overline{R}_{+}^{\ast}=\mathbb{Z}\overline{R}^{\ast}\cap\left(\mathbb{N}\overline{R}_{+}^{\ast}\right)_{sat}
\]
it only remains to establish the following lemma.
\begin{lem}
\label{lem:comprelativeandabsolutecoroot}With notations as above,
\[
(\mathrm{res}^{\ast})^{-1}\left((\mathbb{N}\overline{R}_{+}^{\ast})_{sat}\right)=(\mathbb{N}R_{+}^{\ast})_{sat}\quad\mbox{and}\quad\left(\mathrm{res}\left(\overline{M}_{d}\right)\right)_{sat}=M_{d}.
\]

\end{lem}
Note that the second formula follows from the first one by duality.

\subsection{~}

The dual and bidual cones of the coroot cones $\mathbb{N}R_{+}^{\ast}$
and $\mathbb{N}\overline{R}_{+}^{\ast}$ are respectively equal to
the cones of dominant weights $M_{d}$ and $\overline{M}_{d}$, and
to their own saturations $(\mathbb{N}R_{+}^{\ast})_{sat}$ and $(\mathbb{N}\overline{R}_{+}^{\ast})_{sat}$.
By \ref{sub:carDomCoweightsByWeyl}, the restriction map $\mathrm{res}:\overline{M}\twoheadrightarrow M$
sends $\overline{M}_{d}$ into $M_{d}$. Indeed for $m=\mathrm{res}(\overline{m})$
with $\overline{m}\in\overline{M}_{d}$ and any $w\in W$ lifting
to $\overline{w}\in\overline{W}_{S}$, $m-wm=\mathrm{res}(\overline{m}-\overline{w}\overline{m})$
belongs to $\mathbb{N}R_{+}$ since $\overline{m}-\overline{w}\overline{m}$
belongs to $\mathbb{N}\overline{R}_{+}$. Passing to the bidual cones,
we thus obtain the easiest inclusion: 
\[
\mathrm{res}^{\ast}\left(\mathbb{N}R_{+}^{\ast}\right)_{sat}\subset\left(\mathbb{N}\overline{R}_{+}^{\ast}\right)_{sat}.
\]

\subsection{~\label{sub:DescTwistedActionRelRoots}}

For the opposite inclusion, we will need a few more notations:
\begin{enumerate}
\item Recall from \cite[XXI 1.2.1]{SGA3.3r} that the formulas 
\[
p(x)={\textstyle \sum_{\overline{\alpha}\in\overline{R}}}\left\langle x,\overline{\alpha}^{\ast}\right\rangle \overline{\alpha}^{\ast}\quad\mbox{and}\quad\ell(x)=\left\langle x,p(x)\right\rangle \quad(x\in\overline{M})
\]
define a morphism $p:\overline{M}\rightarrow\overline{M}^{\ast}$
and a map $\ell:\overline{M}\rightarrow\mathbb{N}$ such that 
\[
\forall\overline{\alpha}\in\overline{R}:\qquad\ell(\overline{\alpha})>0\quad\mbox{and}\quad2p(\overline{\alpha})=\ell(\overline{\alpha})\overline{\alpha}^{\ast},
\]

\item The Galois group $\Gal_{k}$ acts on $\overline{M}$, $\overline{M}^{\ast}$,
$\overline{W}$, $\overline{W}_{S}$ and $\overline{W}_{S}^{0}$,
the morphisms $\mathrm{res}$ and $\mathrm{res}^{\ast}$ are $\overline{W}_{S}\rtimes\Gal_{k}$-equivariant,
the latter identifies $M^{\ast}$ with $(\overline{M}^{\ast})^{\Gal_{k}}$,
the subset $\overline{R}\subset\overline{M}$ and $\overline{R}^{\ast}\subset\overline{M}^{\ast}$
are $\overline{W}\rtimes\Gal_{k}$-stable, and 
\[
\ast:\overline{R}\rightarrow\overline{R}^{\ast},\quad p:\overline{M}\rightarrow\overline{M}^{\ast},\quad\ell:\overline{M}\rightarrow\mathbb{N}
\]
are also $\overline{W}\rtimes\Gal_{k}$-equivariant (with the trivial
action on $\mathbb{N}$). 
\item For every $\gamma\in\Gal_{k}$, there is a unique $w_{\gamma}\in\overline{W}_{S}^{0}$
such that 
\[
w_{\gamma}\gamma\overline{R}(0)_{+}=\overline{R}(0)_{+}
\]
 by \cite[XXI 3.3.7]{SGA3.3r}, in which case also $w_{\gamma}\gamma\overline{R}_{+}=\overline{R}_{+}$
since 
\[
\overline{R}_{+}\setminus\overline{R}(0)_{+}=\overline{R}\cap\mathrm{res}^{-1}(R_{+})
\]
 is already stable under $\overline{W}_{S}^{0}\rtimes\Gal_{k}$. The
twisted action of $\Gal_{k}$ on $\overline{\mathscr{R}}$~\cite[6.2]{BoTi65}
is given by $\gamma\cdot=w_{\gamma}\gamma$. The above maps are equivariant
for the twisted action, which moreover preserves $\overline{\Delta}$
and $\overline{\Delta}(0)$. For every $\alpha\in\Delta$, the twisted
action is transitive on $\overline{\Delta}(\alpha)=\mathrm{res}^{-1}(\alpha)\cap\overline{\Delta}$
by \cite[6.4.2 \& 6.8]{BoTi65}. 
\item A root $\overline{\alpha}\in\overline{R}$ maps to $\alpha\in\Delta$
if and only if it is the sum of a (unique) simple root $\overline{\delta}\in\overline{\Delta}(\alpha)$
and some element of $\mathbb{N}\overline{R}(0)_{+}$. This yields
a partition of $\overline{R}(\alpha)=\overline{R}\cap\mathrm{res}^{-1}(\alpha)$
indexed by $\overline{\Delta}(\alpha)$, whose parts are permuted
transitively by the twisted action of $\Gal_{k}$. It follows that
\[
{\textstyle \sum}_{\overline{\alpha}\in\overline{R}(\alpha)}\overline{\alpha}=n_{\alpha}\cdot{\textstyle \sum}_{\overline{\delta}\in\overline{\Delta}(\alpha)}\overline{\delta}+\widetilde{\alpha}_{0}\quad\mbox{in}\quad\overline{M}
\]
with $n_{\alpha}\in\mathbb{N}^{\times}$ and $\widetilde{\alpha}_{0}\in\mathbb{N}\overline{R}(0)_{+}$.
Applying the morphism $2p$, we obtain
\[
{\textstyle \sum}_{\overline{\alpha}\in\overline{R}(\alpha)}\ell(\overline{\alpha})\overline{\alpha}^{\ast}=n_{\alpha}\ell_{\alpha}\cdot{\textstyle \sum}_{\overline{\delta}\in\overline{\Delta}(\alpha)}\overline{\delta}^{\ast}+2p(\widetilde{\alpha}_{0})\quad\mbox{in}\quad\overline{M}^{\ast}
\]
with $p(\widetilde{\alpha}_{0})\in\mathbb{N}\overline{R}(0)_{+}^{\ast}$
and $\ell_{\alpha}=\ell(\overline{\delta})\in\mathbb{N}^{\times}$
for any $\overline{\delta}\in\overline{\Delta}(\alpha)$. 
\end{enumerate}

\subsection{~}

Fix $\alpha\in\Delta$ and change $(G,S,T,P,B)$ to $(H,S,T,P\cap H,B\cap H_{k^{s}})$
where $H$ is the unique reductive subgroup of $G$ containing $Z_{G}(S)$
with 
\[
\Lie(H)=\Lie(G)_{0}\oplus\oplus_{\beta\in\mathbb{Z}\alpha\cap R}\Lie(G)_{\beta}
\]
This changes our absolute and relative based root data to respectively
\[
\left(\overline{M},\overline{R}_{\alpha},\overline{M}^{\ast},\overline{R}_{\alpha}^{\ast};\overline{\Delta}(\alpha)\cup\overline{\Delta}(0)\right)\quad\mbox{and}\quad\left(M,\mathbb{Z}\alpha\cap R,M^{\ast},(\mathbb{Z}\alpha\cap R)^{\ast};\{\alpha\}\right)
\]
where $\overline{R}_{\alpha}=\left\{ \overline{\beta}\in\overline{R}:\mathrm{res}(\overline{\beta})\in\mathbb{Z}\alpha\right\} $.
We thus already know that 
\[
\mathrm{res}^{\ast}\left(\mathrm{ind}(\alpha^{\ast})\right)={\textstyle \sum_{\overline{\delta}\in\overline{\Delta}(\alpha)}}\lambda_{\overline{\delta}}\overline{\delta}^{\ast}+\widetilde{\alpha}_{0}^{\ast}\quad\mbox{in}\quad\mathbb{N}\overline{R}_{+}\subset\overline{M}^{\ast}
\]
with $\lambda_{\overline{\delta}}\in\mathbb{N}$ and $\widetilde{\alpha}_{0}^{\ast}\in\mathbb{N}\overline{R}(0)_{+}^{\ast}$.
Since $\mathrm{res}^{\ast}(\mathrm{ind}(\alpha^{\ast}))$ is fixed
by the twisted action, the coefficient map $\overline{\delta}\mapsto\lambda_{\overline{\delta}}$
is constant on the (twisted) $\Gal_{k}$-orbit $\Delta(\alpha)$,
thus
\begin{equation}
\mathrm{res}^{\ast}\left(\mathrm{ind}(\alpha^{\ast})\right)=\lambda_{\alpha}\cdot{\textstyle \sum_{\overline{\delta}\in\overline{\Delta}(\alpha)}}\overline{\delta}^{\ast}+\widetilde{\alpha}_{0}^{\ast}\quad\mbox{in}\quad\overline{M}^{\ast}\label{eq:formulaforres*ind*alpha*}
\end{equation}
with $\lambda_{\alpha}\in\mathbb{N}$, therefore also 
\[
n_{\alpha}\ell_{\alpha}\cdot\mathrm{res}^{\ast}\left(\mathrm{ind}(\alpha^{\ast})\right)=\lambda_{\alpha}\cdot{\textstyle \sum}_{\overline{\alpha}\in\overline{R}(\alpha)}\ell(\overline{\alpha})\overline{\alpha}^{\ast}+\left(n_{\alpha}\ell_{\alpha}\cdot\widetilde{\alpha}_{0}^{\ast}-2\lambda_{\alpha}\cdot p(\widetilde{\alpha}_{0})\right).
\]
Since $\mathrm{res}^{\ast}\left(\mathrm{ind}(\alpha^{\ast})\right)$
and ${\textstyle \sum}_{\overline{\alpha}\in\overline{R}(\alpha)}\ell(\overline{\alpha})\overline{\alpha}^{\ast}$
are fixed by the usual (untwisted) action of $\Gal_{k}$ on $\overline{M}^{\ast}$,
so is the remaining term, which thus belongs to $\mathrm{res}^{\ast}(M^{\ast})$.
But 
\[
\mathrm{res}^{\ast}(M^{\ast})\cap\mathbb{Z}\overline{R}(0)^{\ast}=0
\]
 since any element of $\mathrm{res}^{\ast}(M^{\ast})$ pairs trivially
with all of $\overline{R}(0)$ while the restriction of the pairing
$\overline{M}\times\overline{M}^{\ast}\rightarrow\mathbb{Z}$ to $\mathbb{Z}\overline{R}(0)\times\mathbb{Z}\overline{R}(0)^{\ast}$
is non-degenerate by~\cite[XXI 1.2.5]{SGA3.3r}. We thus obtain the
following equalities in $\overline{M}^{\ast}$: $n_{\alpha}\ell_{\alpha}\cdot\widetilde{\alpha}_{0}^{\ast}=2\lambda_{\alpha}\cdot p(\widetilde{\alpha}_{0})$
and
\[
n_{\alpha}\ell_{\alpha}\cdot\mathrm{res}^{\ast}\left(\mathrm{ind}(\alpha^{\ast})\right)=\lambda_{\alpha}\cdot{\textstyle \sum}_{\overline{\alpha}\in\overline{R}(\alpha)}\ell(\overline{\alpha})\overline{\alpha}^{\ast}.
\]
In particular, $\lambda_{\alpha}\in\mathbb{N}^{\times}$ since $\mathrm{res}^{\ast}\left(\mathrm{ind}(\alpha^{\ast})\right)\neq0$.

\subsection{~}

Suppose now that $x\in M^{\ast}\otimes\mathbb{Q}$ is such that 
\[
\mathrm{res}^{\ast}(x)={\textstyle \sum_{\overline{\delta}\in\overline{\Delta}}}y_{\overline{\delta}}\cdot\overline{\delta}^{\ast}\quad\mbox{with}\quad y_{\overline{\delta}}\in\mathbb{Q}_{+}.
\]
Since the left hand side is invariant under the twisted action of
$\Gal_{k}$,
\[
\mathrm{res}^{\ast}(x)={\textstyle \sum}_{\alpha\in\Delta}y_{\alpha}\cdot{\textstyle \sum_{\overline{\delta}\in\overline{\Delta}(\alpha)}\overline{\delta}^{\ast}+\widetilde{x}_{0}^{\ast}\quad\mbox{with}\quad y_{\alpha}\in\mathbb{N},\quad\widetilde{x}_{0}^{\ast}\in\mathbb{Q}_{+}\overline{R}(0)_{+}^{\ast}.}
\]
Using (\ref{eq:formulaforres*ind*alpha*}) and $\mathbb{Q}\overline{R}(0)^{\ast}\cap\mathrm{res}^{\ast}(M^{\ast}\otimes\mathbb{Q})=0$,
we obtain 
\[
\left(\widetilde{x}_{0}^{\ast}-{\textstyle \sum}_{\alpha\in\Delta}y_{\alpha}\lambda_{\alpha}^{-1}\cdot\widetilde{\alpha}_{0}^{\ast}\right)=\mathrm{res}^{\ast}\left(x-{\textstyle \sum}_{\alpha\in\Delta}{\textstyle y_{\alpha}\lambda_{\alpha}^{-1}}\cdot\mathrm{ind}(\alpha^{\ast})\right)=0
\]
thus $x=\sum_{\alpha\in\Delta}y_{\alpha}\lambda_{\alpha}^{-1}\cdot\mathrm{ind}(\alpha^{\ast})$
belongs to $\mathbb{Q}_{+}R_{+}$. It follows that
\[
\left(\mathrm{res}^{\ast}\right)^{-1}\left(\left(\mathbb{N}\overline{R}_{+}^{\ast}\right)_{sat}\right)\subset\left(\mathbb{N}R_{+}^{\ast}\right)_{sat}\quad\mbox{in}\quad M^{\ast},
\]
which completes the proof of lemma~\ref{lem:comprelativeandabsolutecoroot}
and proposition~\ref{prop:compdomorders}.

\chapter{The Tannakian formalism}

Let $G$ be an affine and flat group scheme over $S$ and let $\Gamma=(\Gamma,+,\leq)$
be a non-trivial, totally ordered commutative group. We will define
below an equivariant diagram of fpqc sheaves $(\Sch/S)^{\circ}\rightarrow\Group$\nomenclature[Group]{$\Group$}{Category of groups.}
or $(\Sch/S)^{\circ}\rightarrow\Set$:
\[
\xyC{1,1pc}\xymatrix{G\ar[d]^{\iota} & \mbox{acting on} & \mathbb{G}^{\Gamma}(G)\ar[d]^{\iota}\ar[r]^{\Fi} & \mathbb{F}^{\Gamma}(G)\ar[d]^{\iota}\\
\Aut^{\otimes}(V)\ar[d] & \cdots & \mathbb{G}^{\Gamma}(V)\ar[d]\ar[r]^{\Fi} & \mathbb{F}^{\Gamma}(V)\ar[d]\\
\Aut^{\otimes}(V^{\circ})\,\mbox{ or }\Aut^{\otimes}(\omega)\ar[d] & \cdots & \mathbb{G}^{\Gamma}(V^{\circ})\,\mbox{ or }\mathbb{G}^{\Gamma}(\omega)\ar[d]\ar[r]^{\Fi} & \mathbb{F}^{\Gamma}(V^{\circ})\,\mbox{ or }\mathbb{F}^{\Gamma}(\omega)\ar[d]\\
\Aut^{\otimes}(\omega^{\circ}) & \cdots & \mathbb{G}^{\Gamma}(\omega^{\circ})\ar[r]^{\Fi} & \mathbb{F}^{\Gamma}(\omega^{\circ})
}
\]
The main result of this chapter will then be the following theorem:
\begin{thm}
\label{thm:MainTan}If $G$ is a reductive group over $S$, then 
\[
\begin{array}{ccccccc}
G & = & \Aut^{\otimes}(V) & = & \Aut^{\otimes}(V^{\circ}) & = & \Aut^{\otimes}(\omega)\\
\mathbb{G}^{\Gamma}(G) & = & \mathbb{G}^{\Gamma}(V) & = & \mathbb{G}^{\Gamma}(V^{\circ}) & = & \mathbb{G}^{\Gamma}(\omega)\\
\mathbb{F}^{\Gamma}(G) & = & \mathbb{F}^{\Gamma}(V) & = & \mathbb{F}^{\Gamma}(V^{\circ}) & \subset & \mathbb{F}^{\Gamma}(\omega)
\end{array}
\]
If moreover $G$ is isotrivial and $S$ quasi-compact, then also
\[
G=\Aut^{\otimes}(\omega^{\circ}),\quad\mathbb{G}^{\Gamma}(G)=\mathbb{G}^{\Gamma}(\omega^{\circ})\quad\mbox{and}\quad\mathbb{F}^{\Gamma}(G)=\mathbb{F}^{\Gamma}(\omega)=\mathbb{F}^{\Gamma}(\omega^{\circ}).
\]

\end{thm}
\noindent More precisely, we will first show that for any affine
flat group scheme $G$ over $S$, 
\[
\begin{array}{ccccc}
G & = & \Aut^{\otimes}(V) & = & \Aut^{\otimes}(\omega)\\
\mathbb{G}^{\Gamma}(G) & = & \mathbb{G}^{\Gamma}(V) & = & \mathbb{G}^{\Gamma}(\omega)\\
\mathbb{F}^{\Gamma}(G) & \subset & \mathbb{F}^{\Gamma}(V) & \subset & \mathbb{F}^{\Gamma}(\omega)
\end{array}
\]
Then, under technical assumptions which are satisfied by all reductive
groups (resp. all isotrivial reductive groups over quasi-compact bases),
we will also establish that
\[
\begin{array}{ccc}
\Aut^{\otimes}(V) & = & \Aut^{\otimes}(V^{\circ})\\
\mathbb{G}^{\Gamma}(V) & = & \mathbb{G}^{\Gamma}(V^{\circ})\\
\mathbb{F}^{\Gamma}(V) & \subset & \mathbb{F}^{\Gamma}(V^{\circ})
\end{array}\quad\left(\mbox{resp. }\begin{array}{ccc}
\Aut^{\otimes}(\omega) & = & \Aut^{\otimes}(\omega^{\circ})\\
\mathbb{G}^{\Gamma}(\omega) & = & \mathbb{G}^{\Gamma}(\omega^{\circ})\\
\mathbb{F}^{\Gamma}(\omega),\mathbb{F}^{\Gamma}(V^{\circ}) & \subset & \mathbb{F}^{\Gamma}(\omega^{\circ})
\end{array}\right).
\]
We will finally show that for $G$ reductive and isotrivial over a
quasi-compact $S$, the morphism $\mathbb{G}^{\Gamma}(G)\rightarrow\mathbb{F}^{\Gamma}(\omega^{\circ})$
is an epimorphism of fpqc sheaves on $S$. Thus 
\[
\mathbb{F}^{\Gamma}(G)=\mathbb{F}^{\Gamma}(V)=\mathbb{F}^{\Gamma}(V^{\circ})=\mathbb{F}^{\Gamma}(\omega)=\mathbb{F}^{\Gamma}(\omega^{\circ})
\]
in this case, and the remaining statement, namely 
\[
\mathbb{F}^{\Gamma}(G)=\mathbb{F}^{\Gamma}(V)=\mathbb{F}^{\Gamma}(V^{\circ})
\]
for a reductive group $G$ over an arbitrary $S$ easily follows.
\begin{rem}
As will be clear from the definitions below, the assertions about
$\omega^{\circ}$ and $V$ correspond to the two extreme cases of
a variety of possible statements about filtrations on fiber functors.
These cases were not clearly distinguished in \cite[Chapitre IV]{SaRi72},
which lead us to revisit its proofs. Our definition of isotriviality
for reductive groups in section~\ref{sub:DefOfIsotriv} is tailor-made
to fit the $\omega^{\circ}$-case: it is not even local for the Zariski
topology on the base. The corresponding Zariski-local notion was defined
in \cite[XXIV 4.1.2]{SGA3.3r}. For a locally isotrivial reductive
group, the above theorem works with a suitably (Zariski) localized
version of the fiber functor $\omega^{\circ}$.  
\end{rem}

\section{$\Gamma$-graduations and $\Gamma$-filtrations on quasi-coherent
sheaves}

\subsection{~}

Let $\mathcal{M}$ be a quasi-coherent sheaf on a scheme $X$. 
\begin{defn}
A $\Gamma$-graduation on $\mathcal{M}$ is a collection $\mathcal{G}=(\mathcal{G}_{\gamma})_{\gamma\in\Gamma}$
of quasi-coherent subsheaves of $\mathcal{M}$ such that $\mathcal{M}=\oplus_{\gamma\in\Gamma}\mathcal{G}_{\gamma}$.
A $\Gamma$-filtration on $\mathcal{M}$ is a collection $\mathcal{F}=(\mathcal{F}^{\gamma})_{\gamma\in\Gamma}$
of quasi-coherent subsheaves of $\mathcal{M}$ such that, locally
on $X$ for the fpqc topology, there exists a $\Gamma$-graduation
$\mathcal{G}=(\mathcal{G}_{\gamma})_{\gamma\in\Gamma}$ on $\mathcal{M}$
for which $\mathcal{F}^{\gamma}=\oplus_{\eta\geq\gamma}\mathcal{G}_{\eta}$.
We call any such $\mathcal{G}$ a splitting of $\mathcal{F}$ and
write $\mathcal{F}=\Fi(\mathcal{G})$.\nomenclature[Fil(Gcal)]{$\Fi (\mathcal{G})$}{$\Gamma$-filtration induced by a $\Gamma$-graduation $\mathcal{G}$, page \nomrefpage}\nomenclature[Gr_F(M)]{$\Gr_\mathcal{F}(\mathcal{M})$}{$\Gamma$-graded quasi-coherent sheaf associated with a $\Gamma$-filtration $\mathcal{F}$ on a quasi-coherent sheaf $\mathcal{M}$, page \nomrefpage}
We set 
\[
\mathcal{F}_{+}^{\gamma}=\cup_{\eta>\gamma}\mathcal{F}^{\eta}\quad\mbox{and}\quad\Gr_{\mathcal{F}}^{\gamma}\mathcal{M}=\mathcal{F}^{\gamma}/\mathcal{F}_{+}^{\gamma}.
\]
\end{defn}
\begin{lem}
\label{lem:BasicPropFil}Let $\mathcal{F}$ be a $\Gamma$-filtration
on $\mathcal{M}$. Then $\gamma\mapsto\mathcal{F}^{\gamma}$ is non-increasing,
exhaustive ($\cup\mathcal{F}^{\gamma}=\mathcal{M}$), separated ($\cap\mathcal{F}^{\gamma}=0$),
and for every $\gamma\in\Gamma$, 
\[
0\rightarrow\mathcal{F}^{\gamma}\rightarrow\mathcal{M}\rightarrow\mathcal{M}/\mathcal{F}^{\gamma}\rightarrow0\quad\mbox{and}\quad0\rightarrow\mathcal{F}_{+}^{\gamma}\rightarrow\mathcal{F}^{\gamma}\rightarrow\Gr_{\mathcal{F}}^{\gamma}(\mathcal{M})\rightarrow0
\]
are pure exact sequences of quasi-coherent sheaves (see~\ref{sub:AppendixPureSub}).\end{lem}
\begin{proof}
This is local in the fpqc topology on $X$, trivial if $\mathcal{F}$
has a splitting.
\end{proof}

\subsection{~}

These definitions give rise to a diagram of fpqc stacks over $\Sch$\nomenclature[Gr^GammaQCoh]{$\Gra^\Gamma \QCoh$}{Category of $\Gamma$-graded quasi-coherent sheaves on schemes, page \nomrefpage}\nomenclature[Fil^GammaQCoh]{$\Fil^\Gamma \QCoh$}{Category of $\Gamma$-filtered quasi-coherent sheaves on schemes, page \nomrefpage}\nomenclature[QCoh]{$\QCoh$}{Category of quasi-coherent sheaves on schemes, page \nomrefpage}\nomenclature[Fil]{$\Fi$}{Functor $\Fi : \Gra ^\Gamma \QCoh \rightarrow \Fil ^\Gamma \QCoh$, page \nomrefpage}\nomenclature[Gr]{$\Gr$}{Functor $\Gr : \Fil ^\Gamma \QCoh \rightarrow \Gra ^\Gamma \QCoh$, page \nomrefpage}
\[
\xyC{3pc}\xymatrix{\Gra^{\Gamma}\QCoh\ar@<2pt>[r]^{\Fi} & \Fil^{\Gamma}\QCoh\ar@<2pt>[l]^{\Gr}\ar[r]\sp(0.55){\mathrm{forg}} & \QCoh}
\]
whose fiber over a scheme $X$ is the diagram of exact $\otimes$-functors\nomenclature[Gr^GammaQCoh(X)]{$\Gra^\Gamma \QCoh(X)$}{Category of $\Gamma$-graded quasi-coherent sheaves on $X$, page \nomrefpage}\nomenclature[Fil^GammaQCoh(X)]{$\Fil^\Gamma \QCoh(X)$}{Category of $\Gamma$-filtered quasi-coherent sheaves on $X$, page \nomrefpage}\nomenclature[QCoh(X)]{$\QCoh(X)$}{Category of quasi-coherent sheaves on $X$, page \nomrefpage}
\[
\xymatrix{\Gra^{\Gamma}\QCoh(X)\ar@<2pt>[r]^{\Fi} & \Fil^{\Gamma}\QCoh(X)\ar@<2pt>[l]^{\Gr}\ar[r]\sp(0.55){\mathrm{forg}} & \QCoh(X)}
\]
where $\QCoh(X)$ is the abelian $\otimes$-category of quasi-coherent
sheaves $\mathcal{M}$ on $X$, $\Gra^{\Gamma}\QCoh(X)$ is the abelian
$\otimes$-category of $\Gamma$-graded quasi-coherent sheaves $(\mathcal{M},\mathcal{G})$
on $X$, and $\Fil^{\Gamma}\QCoh(X)$ is the exact (in Quillen's sense)
$\otimes$-category of $\Gamma$-filtered quasi-coherent sheaves $(\mathcal{M},\mathcal{F})$
on $X$. The morphisms in these last two categories are the morphisms
of the underlying quasi-coherent sheaves which preserve the given
collections of subsheaves, and the $\otimes$-products are given by
the usual formulas
\[
\begin{array}{rcl}
(\mathcal{M}_{1},\mathcal{G}_{1})\otimes(\mathcal{M}_{2},\mathcal{G}_{2})=(\mathcal{M}_{1}\otimes\mathcal{M}_{2},\mathcal{G}) & \mbox{with} & \mathcal{G}_{\gamma}=\oplus_{\gamma_{1}+\gamma_{2}=\gamma}\mathcal{G}_{1,\gamma_{1}}\otimes\mathcal{G}_{2,\gamma_{2}},\\
(\mathcal{M}_{1},\mathcal{F}_{1})\otimes(\mathcal{M}_{2},\mathcal{F}_{2})=(\mathcal{M}_{1}\otimes\mathcal{M}_{2},\mathcal{F}) & \mbox{with} & \mathcal{F}^{\gamma}=\sum_{\gamma_{1}+\gamma_{2}=\gamma}\mathcal{F}_{1}^{\gamma_{1}}\otimes\mathcal{F}_{2}^{\gamma_{2}}.
\end{array}
\]
The second formula makes sense by the purity mentioned above, and
indeed defines a $\Gamma$-filtration on $\mathcal{M}_{1}\otimes\mathcal{M}_{2}$:
if $\mathcal{G}_{i}$ splits $\mathcal{F}_{i}$ for $i\in\{1,2\}$,
then $\mathcal{G}$ splits $\mathcal{F}$. We have
\begin{eqnarray*}
\mathcal{F}_{+}^{\gamma} & = & {\textstyle \sum_{\gamma_{1}+\gamma_{2}>\gamma}}\mathcal{F}_{1}^{\gamma_{1}}\otimes\mathcal{F}_{2}^{\gamma_{2}}\\
\mbox{and}\quad\Gr_{\mathcal{F}}^{\gamma}(\mathcal{M}_{1}\otimes\mathcal{M}_{2}) & \simeq & \oplus_{\gamma_{1}+\gamma_{2}=\gamma}\Gr_{\mathcal{F}_{1}}^{\gamma_{1}}(\mathcal{M}_{1})\otimes\Gr_{\mathcal{F}_{2}}^{\gamma_{2}}(\mathcal{M}_{2}).
\end{eqnarray*}
The first formula is trivial and gives the morphism (from right to
left) in the second formula, which is easily seen to be an isomorphism
by localization to an fpqc cover of $X$ over which $\mathcal{F}_{1}$
and $\mathcal{F}_{2}$ both acquire a splitting. The neutral objects
for $\otimes$ are
\[
1_{X}=(\mathcal{O}_{X},\mathcal{G}\mbox{ of }\mathcal{F})\mbox{ with }\mathcal{G}_{\gamma}=\begin{cases}
\mathcal{O}_{X} & \mbox{for }\gamma=0,\\
0 & \mbox{otherwise}
\end{cases}\quad\mbox{and}\quad\mathcal{F}^{\gamma}=\begin{cases}
\mathcal{O}_{X} & \mbox{for }\gamma\leq0,\\
0 & \mbox{otherwise.}
\end{cases}
\]
A morphism $(\mathcal{M}_{1},\mathcal{F}_{1})\rightarrow(\mathcal{M}_{2},\mathcal{F}_{2})$
is strict if $\mathrm{Im}(\mathcal{F}_{1}^{\gamma})=\mathcal{F}_{2}^{\gamma}\cap\mathrm{Im}(\mathcal{M}_{1})$
in $\mathcal{M}_{2}$ for every $\gamma\in\Gamma$. The short exact
sequences of $\Fil^{\Gamma}\QCoh(X)$ are those made of strict arrows
whose underlying sequence of sheaves is short exact. The formulas
\[
\Fi(\mathcal{M},\mathcal{G})=(\mathcal{M},\Fi(\mathcal{G})),\quad\Gr(\mathcal{M},\mathcal{F})=\oplus_{\gamma}\Gr_{\mathcal{F}}^{\gamma}\mathcal{M}\quad\mbox{and}\quad\mathrm{forg(\mathcal{M},-)=\mathcal{M}}
\]
define the exact $\otimes$-functors between our three categories.
Finally the ``base change functors'' defining the fibered category
structures on $\Gra^{\Gamma}\QCoh$ and $\Fil^{\Gamma}\QCoh$ are
induced by the base change functors on $\QCoh$ (thanks to the purity
of the subsheaves). It is well-known that $\QCoh$ is an fpqc stack
over $\Sch$ (see for instance~\cite[Theorem 4.23]{Vi05}) and it
follows rather formally from their definitions that the other two
fibered categories are also fpqc stacks over $\Sch$. We denote by
\[
\xymatrix{\Gra^{\Gamma}\QCoh/S\ar@<2pt>[r]^{\Fi} & \Fil^{\Gamma}\QCoh/S\ar@<2pt>[l]^{\Gr}\ar[r]\sp(0.58){\mathrm{forg}} & \QCoh}
/S
\]
the corresponding stacks over $\Sch/S$ where $S$ is any base scheme.

\section{$\Gamma$-graduations and $\Gamma$-filtrations on fiber functors}

\subsection{~}

Let $s:G\rightarrow S$ be an affine and flat group scheme. We denote
by \nomenclature[Rep(G)]{$\Rep (G)$}{Fibered category of algebraic representations of $G$ on quasi-coherent sheaves, page \nomrefpage}$\Rep(G)$
the fpqc stack over $\Sch/S$ whose fiber over $T\rightarrow S$ is
the abelian $\otimes$-category $\Rep(G)(T)$ \nomenclature[Rep(G)(X)]{$\Rep (G)(X)$}{Category of algebraic representations of $G$ on quasi-coherent sheaves over $X$, page \nomrefpage}of
quasi-coherent $G_{T}$\emph{-}$\mathcal{O}_{T}$-modules as defined
in \cite[I 4.7.1]{SGA3.1r}. Then\nomenclature[A(G)]{$\mathcal{A} (G)$}{Hopf algebra of $G$.}
\[
\mathcal{A}(G)=s_{\ast}\mathcal{O}_{G}
\]
is a quasi-coherent Hopf algebra over $S$ and $\Rep(G)(T)$ is $\otimes$-equivalent
to the category of quasi-coherent $\mathcal{A}(G_{T})$-comodules
where $\mathcal{A}(G_{T})=\mathcal{A}(G)_{T}$. Let\nomenclature[V]{$V$}{Fiber functor $V: \Rep (G) \rightarrow \QCoh / S$, page \nomrefpage}
\[
V:\Rep(G)\rightarrow\QCoh/S
\]
be the forgetful functor. For any $S$-scheme $q:T\rightarrow S$,
we denote by\nomenclature[omega_X]{$\omega _X$}{Fiber functor $\omega _X: \Rep (G)(S) \rightarrow \QCoh (X) $, page \nomrefpage}
\[
V_{T}:\Rep(G_{T})\rightarrow\QCoh/T\quad\mbox{and}\quad\omega_{T}:\Rep(G)(S)\rightarrow\QCoh(T)
\]
the induced morphism of fpqc stack over $\Sch/T$ and fiber functor.
Note that $\omega_{T}$ is a right exact $\otimes$-functor. It also
commutes with arbitrary colimits and preserves pure monomorphisms
and pure short exact sequences, where purity in $\Rep(G)(S)$ refers
to purity of the underlying objects in $\QCoh(S)$.

\subsection{~}

A $\Gamma$-graduation $\mathcal{G}$ on $V_{T}:\Rep(G_{T})\rightarrow\QCoh/T$
is a factorization
\[
\xymatrix{\Rep(G_{T})\ar[r]\sp(0.45){\mathcal{G}} & \Gra^{\Gamma}\QCoh/T\ar[r]\sp(0.55){\mathrm{forg}} & \QCoh/T}
\]
 of $V_{T}$ such that if $\mathcal{G}_{\gamma}:\Rep(G_{T})\rightarrow\QCoh/T$
is the $\gamma$-component of $\mathcal{G}$, 
\begin{lyxlist}{MMM}
\item [{(G0)}] For every $T$-morphism $f:X\rightarrow Y$, $\rho\in\Rep(G)(Y)$
and $\gamma\in\Gamma$,
\[
f^{\ast}(\mathcal{G}_{\gamma}(\rho))=\mathcal{G}_{\gamma}(f^{\ast}\rho).
\]

\item [{(G1)}] For every $T$-scheme $X\rightarrow T$, $\rho_{1},\rho_{2}\in\Rep(G)(X)$
and $\gamma\in\Gamma$, 
\[
\mathcal{G}_{\gamma}(\rho_{1}\otimes\rho_{2})=\oplus_{\gamma_{1}+\gamma_{2}=\gamma}\mathcal{G}_{\gamma_{1}}(\rho_{1})\otimes\mathcal{G}_{\gamma_{2}}(\rho_{2}).
\]

\end{lyxlist}
Thus (G0) says that each $\mathcal{G}_{\gamma}$ is a morphism of
fibered categories over $\Sch/T$. Then (G1) implies that $\mathcal{G}_{0}(\rho)=\mathcal{M}$
and $\mathcal{G}_{\gamma}(\rho)=0$ for $\gamma\neq0$ when $\rho$
is the trivial representation of $G_{X}$ on $\mathcal{M}\in\QCoh(X)$
(one proves it first for $\mathcal{M}=\mathcal{O}_{X}$).

\subsection{~}

A $\Gamma$-graduation $\mathcal{G}$ on $\omega_{T}:\Rep(G)(S)\rightarrow\QCoh(T)$
is a factorization
\[
\xymatrix{\Rep(G)(S)\ar[r]\sp(0.45){\mathcal{G}} & \Gra^{\Gamma}\QCoh(T)\ar[r]\sp(0.55){\mathrm{forg}} & \QCoh(T)}
\]
of $\omega_{T}$ such that if $\mathcal{G}_{\gamma}:\Rep(G)(S)\rightarrow\QCoh(T)$
is the $\gamma$-component of $\mathcal{G}$, 
\begin{lyxlist}{MMM}
\item [{(G1)}] For every $\rho_{1},\rho_{2}\in\Rep(G)(S)$ and $\gamma\in\Gamma$,
\[
\mathcal{G}_{\gamma}(\rho_{1}\otimes\rho_{2})=\oplus_{\gamma_{1}+\gamma_{2}=\gamma}\mathcal{G}_{\gamma_{1}}(\rho_{1})\otimes\mathcal{G}_{\gamma_{2}}(\rho_{2}).
\]

\item [{(G2)}] For the trivial representation $\rho$ of $G$ on $\mathcal{M}\in\QCoh(S)$,
\[
\mathcal{G}_{0}(\rho)=\mathcal{M}\quad\mbox{and}\quad\mathcal{G}_{\gamma}(\rho)=0\mbox{\,\ if }\gamma\neq0.
\]

\end{lyxlist}
Note that each $\mathcal{G}_{\gamma}$ is right exact, commutes with
arbitrary colimits and preserves pure monomorphisms and pure short
exact sequences.

\subsection{~}

A $\Gamma$-filtration $\mathcal{F}$ on $V_{T}:\Rep(G_{T})\rightarrow\QCoh/T$
is a factorization
\[
\xymatrix{\Rep(G_{T})\ar[r]\sp(0.45){\mathcal{F}} & \Fil^{\Gamma}\QCoh/T\ar[r]\sp(0.55){\mathrm{forg}} & \QCoh/T}
\]
of $V_{T}$ such that if $\mathcal{F}^{\gamma}:\Rep(G_{T})\rightarrow\QCoh/T$
is the $\gamma$-component of $\mathcal{F}$, 
\begin{lyxlist}{MMM}
\item [{(F0)}] For every $T$-morphism $f:X\rightarrow Y$, $\rho\in\Rep(G)(Y)$
and $\gamma\in\Gamma$,
\[
f^{\ast}(\mathcal{F}^{\gamma}(\rho))=\mathcal{F}^{\gamma}(f^{\ast}\rho).
\]

\item [{(F1)}] For every $X\rightarrow T$, $\rho_{1},\rho_{2}\in\Rep(G)(X)$
and $\gamma\in\Gamma$, 
\[
\mathcal{F}^{\gamma}(\rho_{1}\otimes\rho_{2})={\textstyle \sum_{\gamma_{1}+\gamma_{2}=\gamma}}\mathcal{F}^{\gamma_{1}}(\rho_{1})\otimes\mathcal{F}^{\gamma_{2}}(\rho_{2}).
\]

\item [{(F3)}] For every $X\rightarrow T$ and $\gamma\in\Gamma$, $\mathcal{F}^{\gamma}:\Rep(G)(X)\rightarrow\QCoh(X)$
is exact.
\end{lyxlist}
Thus (F0) says that each $\mathcal{F}^{\gamma}$ is a morphism of
fibered categories over $\Sch/T$. Then again (F1) and (F3) imply
that $\mathcal{F}^{\gamma}(\rho)=\mathcal{M}$ for $\gamma\leq0$
and $\mathcal{F}^{\gamma}(\rho)=0$ for $\gamma>0$ when $\rho$ is
the trivial representation of $G$ on $\mathcal{M}\in\QCoh(X)$. \nomenclature[Gr_F^g(tau)]{$\Gr _\mathcal{F} ^\gamma (\tau)$}{Graded piece of the $\Gamma$-filtration $\mathcal{F}(\tau)$ on $V(\tau)$.}

\subsection{~}

A $\Gamma$-filtration $\mathcal{F}$ on $\omega_{T}:\Rep(G)(S)\rightarrow\QCoh(T)$
is a factorization
\[
\xymatrix{\Rep(G)(S)\ar[r]\sp(0.45){\mathcal{F}} & \Fil^{\Gamma}\QCoh(T)\ar[r]\sp(0.55){\mathrm{forg}} & \QCoh(T)}
\]
of $\omega_{T}$ such that if $\mathcal{F}^{\gamma}:\Rep(G)(S)\rightarrow\QCoh(T)$
is the $\gamma$-component of $\mathcal{F}$, 
\begin{lyxlist}{MMM}
\item [{(F1)}] For every $\rho_{1},\rho_{2}\in\Rep(G)(S)$ and $\gamma\in\Gamma$,
\[
\mathcal{F}^{\gamma}(\rho_{1}\otimes\rho_{2})={\textstyle \sum_{\gamma_{1}+\gamma_{2}=\gamma}}\mathcal{F}^{\gamma_{1}}(\rho_{1})\otimes\mathcal{F}^{\gamma_{2}}(\rho_{2}).
\]

\item [{(F2)}] For the trivial representation $\rho$ of $G$ on $\mathcal{M}\in\QCoh(S)$,
\[
\mathcal{F}^{\gamma}(\rho)=\mathcal{M}\mbox{ if }\gamma\leq0\quad\mbox{and}\quad\mathcal{F}^{\gamma}(\rho)=0\mbox{ if }\gamma>0.
\]

\item [{(F3)}] For every $\gamma\in\Gamma$, $\mathcal{F}^{\gamma}:\Rep(G)(S)\rightarrow\QCoh(T)$
is right exact.
\end{lyxlist}
Since $\mathcal{F}^{\gamma}$ preserves arbitrary direct sums (as
a subfunctor of $\omega_{T}$ which does), this last axiom implies
that $\mathcal{F}^{\gamma}$ commutes with arbitrary colimits. It
also preserves pure monomorphisms and pure short exact sequences.

\subsection{~}

We may now introduce a diagram of fpqc sheaves $(\Sch/S)^{\circ}\rightarrow\Set$,\nomenclature[G^Gamma(V)]{$\mathbb{G}^\Gamma (V)$}{Sheaf of $\Gamma$-graduations on $V$, page \nomrefpage}\nomenclature[G^Gamma(omega)]{$\mathbb{G}^\Gamma (\omega )$}{Sheaf of $\Gamma$-graduations on $\omega$, page \nomrefpage}\nomenclature[F^Gamma(V)]{$\mathbb{F}^\Gamma (V)$}{Sheaf of $\Gamma$-filtrations on $V$, page \nomrefpage}\nomenclature[F^Gamma(omega)]{$\mathbb{F}^\Gamma (\omega )$}{Sheaf of $\Gamma$-filtrations on $\omega$, page \nomrefpage}\nomenclature[Fil]{$\Fi$}{Morphism $\Fi : \mathbb{G}^\Gamma (X) \rightarrow \mathbb{F}^\Gamma (X)$ for $X=V$ or $\omega$, page \nomrefpage}
\[
\xyC{2pc}\xymatrix{\mathbb{G}^{\Gamma}(V)\ar[r]^{\mathrm{res}}\ar[d]_{\Fi} & \mathbb{G}^{\Gamma}(\omega)\ar[d]_{\Fi}\\
\mathbb{F}^{\Gamma}(V)\ar[r]^{\mathrm{res}} & \mathbb{F}^{\Gamma}(\omega)
}
\]
The four presheaves map an $S$-scheme $T$ to the corresponding set
of $\Gamma$-graduations or $\Gamma$-filtrations on $V_{T}$ or $\omega_{T}$,
the $\Fil$-morphisms are given by post-composition with the eponymous
functors, and the $\mathrm{res}$ morphisms map $\mathcal{G}$ or
$\mathcal{F}$ on $V_{T}$ to 
\[
\Rep(G)(S)\rightarrow\Rep(G)(T)\stackrel{\mathcal{G}_{T}}{\longrightarrow}\Gra^{\Gamma}\QCoh(T)\quad\mbox{or}\quad\cdots\stackrel{\mathcal{F}_{T}}{\longrightarrow}\Fil^{\Gamma}\QCoh(T).
\]
The fact that all four presheaves are actually fpqc sheaves on $S$
is essentially a formal consequence of the fact that the corresponding
fibered categories of $\Gamma$-graded and $\Gamma$-filtered quasi-coherent
sheaves are fpqc stacks over $\Sch/S$.

\subsection{~}

The above diagram is equivariant with respect to a morphism\nomenclature[Aut^o(V)]{$\Aut^{\otimes}(V)$}{Sheaf of tensor automorphisms of $V$, page \nomrefpage}\nomenclature[Aut^o(omega)]{$\Aut^{\otimes}(\omega )$}{Sheaf of tensor automorphisms of $\omega$, page \nomrefpage}
\[
\xymatrix{\Aut^{\otimes}(V)\ar[r]^{\mathrm{res}} & \Aut^{\otimes}(\omega)}
\]
of fpqc sheaves of groups on $\Sch/S$, with $\Aut^{\otimes}(\star)$
acting on $\mathbb{G}^{\Gamma}(\star)$ and $\mathbb{F}^{\Gamma}(\star)$
and mapping an $S$-scheme $T$ to a group $\Aut^{\otimes}(\star_{T})$
defined as follows: $\Aut^{\otimes}(V_{T})$ is the group of all automorphisms
$\eta:V_{T}\rightarrow V_{T}$ such that:
\begin{lyxlist}{MMM}
\item [{(A0)}] For every $T$-morphism $f:X\rightarrow Y$ and $\rho\in\Rep(G)(Y)$,
\[
\eta_{f^{\ast}(\rho)}=f^{\ast}(\eta_{\rho}).
\]

\item [{(A1)}] For every $T$-scheme $X\rightarrow T$ and $\rho_{1},\rho_{2}\in\Rep(G)(X)$,
\[
\eta_{\rho_{1}\otimes\rho_{2}}=\eta_{\rho_{1}}\otimes\eta_{\rho_{2}}.
\]

\end{lyxlist}
These conditions imply as above that $\eta_{\rho}=\mathrm{Id}_{\mathcal{M}}$
when $\rho$ is the trivial representation of $G_{X}$ on a quasi-coherent
$\mathcal{O}_{X}$-module $\mathcal{M}$. Similarly, $\Aut^{\otimes}(\omega_{T})$
is the group of all automorphisms $\eta:\omega_{T}\rightarrow\omega_{T}$
such that: 
\begin{lyxlist}{MMM}
\item [{(A1)}] For every $\rho_{1},\rho_{2}\in\Rep(G)(S)$, 
\[
\eta_{\rho_{1}\otimes\rho_{2}}=\eta_{\rho_{1}}\otimes\eta_{\rho_{2}}.
\]

\item [{(A2)}] For the trivial representation $\rho$ of $G$ on $\mathcal{M}\in\QCoh(S)$,
\[
\eta_{\rho}=\mathrm{Id}_{\mathcal{M}}.
\]

\end{lyxlist}
The fact that these two presheaves are actually fpqc sheaves on $S$
is essentially a formal consequence of the fact that $\QCoh/S$ is
a stack over $\Sch/S$. The morphism between them sends $\eta\in\Aut^{\otimes}(V_{T})$
to the automorphism of $\omega_{T}$ which maps $\rho$ in $\Rep(G)(S)$
to the automorphism $\eta_{\rho_{T}}$ of $V(\rho_{T})=\omega_{T}(\rho)$,
the actions mentioned above are the obvious ones, and the claimed
equivariance is equally straightforward.

\subsection{~\label{sub:Aut(F)Aut(G)GrF}}

For $\star\in\{V,\omega\}$ and $\mathcal{X}\in\mathbb{G}^{\Gamma}(\star)(T)$
or $\mathbb{F}^{\Gamma}(\star)(T)$, we denote by\nomenclature[Aut^o(G)]{$\Aut^{\otimes}(\mathcal{G})$}{Sheaf of tensor automorphisms of a fiber functor $\mathcal{X}$ preserving a $\Gamma$-graduation $\mathcal{G}$ on $\mathcal{X}$, page \nomrefpage}\nomenclature[Aut^o(F)]{$\Aut^{\otimes}(\mathcal{F})$}{Sheaf of tensor automorphisms of a fiber functor $\mathcal{X}$ preserving a $\Gamma$-filtration $\mathcal{F}$ on $\mathcal{X}$, page \nomrefpage}
\[
\Aut^{\otimes}(\mathcal{X}):(\Sch/T)^{\circ}\rightarrow\Group
\]
the stabilizer of $\mathcal{X}$ in the restriction $\Aut^{\otimes}(\star)\vert_{T}$
of $\Aut^{\otimes}(\star)$ to $\Sch/T$. It is an fpqc subsheaf of
$\Aut^{\otimes}(\star)\vert_{T}$. For $\mathcal{X}=\mathcal{F}$
in $\mathbb{F}^{\Gamma}(\star)(T)$, there is also a morphism\nomenclature[Gr^b]{$\Gr ^\bullet $}{Morphism $\Gr ^\bullet : \Aut ^\otimes (\mathcal{F}) \rightarrow \Aut ^\otimes (\Gr ^\bullet _\mathcal{F})$, page \nomrefpage}\nomenclature[Gr^bb]{$\mathrm{Grr}$}{Growling sound indicative of frustration with useless generalities.}
\[
\Gr^{\bullet}:\Aut^{\otimes}(\mathcal{F})\rightarrow\Aut^{\otimes}(\Gr_{\mathcal{F}}^{\bullet}).
\]
Here $\Aut^{\otimes}(\Gr_{\mathcal{F}}^{\bullet})$ is an fpqc sheaf
of groups on $\Sch/T$ which maps $X\rightarrow T$ to a group of
automorphisms of $\Gr_{\mathcal{F}_{X}}^{\bullet}=\Gr^{\bullet}\circ\mathcal{F}_{X}$
subject to conditions whose precise formulation will be left to the
reader. The kernel of this morphism is an fpqc sheaf
\[
\Aut^{\otimes!}(\mathcal{F}):(\Sch/T)^{\circ}\rightarrow\Group.
\]
If $\mathcal{G}$ is a splitting of $\mathcal{F}$, then $\Gr_{\mathcal{F}}^{\bullet}\simeq\mathcal{G}$,
thus $\Aut^{\otimes}(\Gr_{\mathcal{F}}^{\bullet})\simeq\Aut^{\otimes}(\mathcal{G})$
and 
\[
\Aut^{\otimes}(\mathcal{F})\simeq\Aut^{\otimes!}(\mathcal{F})\rtimes\Aut^{\otimes}(\mathcal{G}).
\]

\subsection{~\label{sub:defiota}}

There is finally another equivariant diagram of fpqc sheaves on $S$,
\[
\xyR{4pt}\xymatrix{G\ar[dd]^{\iota} &  & \mathbb{G}^{\Gamma}(G)\ar[r]^{\Fi}\ar[dd]^{\iota} & \mathbb{F}^{\Gamma}(G)\ar[dd]^{\iota}\\
 & \mbox{acting on}\\
\Aut^{\otimes}(V) &  & \mathbb{G}^{\Gamma}(V)\ar[r]^{\Fi} & \mathbb{F}^{\Gamma}(V)
}
\]
The morphism $\iota:G\rightarrow\Aut^{\otimes}(V)$ sends $g\in G(T)$
to the automorphism $\iota(g)$ of $V_{T}$ which maps $\rho\in\Rep(G)(X)$
to the automorphism $\rho(g_{X})$ of $V(\rho)$ -- for an $S$-scheme
$T$ and a $T$-scheme $X$. The morphism $\iota:\mathbb{F}^{\Gamma}(G)\hookrightarrow\mathbb{F}^{\Gamma}(V)$
is the image of
\[
\xymatrix{\mathbb{G}^{\Gamma}(G)\ar[r]^{\iota} & \mathbb{G}^{\Gamma}(V)\ar[r]^{\Fi} & \mathbb{F}^{\Gamma}(V)}
\]
where $\iota:\mathbb{G}^{\Gamma}(G)\rightarrow\mathbb{G}^{\Gamma}(V)$
is defined as follows. Recall from \cite[I 4.7.3]{SGA3.1r} that the
fpqc stacks $\Gra^{\Gamma}\QCoh$ and $\Rep\,\mathbb{D}(\Gamma)$
over $\Sch$ are $\otimes$-equivalent: A $\Gamma$-graded quasi-coherent
sheaf $\mathcal{M}=\oplus_{\gamma\in\Gamma}\mathcal{G}_{\gamma}$
on a scheme $X$ is mapped to the unique representation $\rho$ of
$\mathbb{D}_{X}(\Gamma)$ on $\mathcal{M}$ such that for every $f:Y\rightarrow X$
and \nomenclature[Gamma(X,F)]{$\Gamma (X,\mathcal{F})$}{Sections of a sheaf $\mathcal{F}$ over $X$.}$\alpha:\Gamma\rightarrow\Gamma(Y,\mathcal{O}_{Y}^{\ast})$
in $\mathbb{D}_{X}(\Gamma)(Y)$, $\rho(\alpha)(x)$ equals $\alpha(\gamma)\cdot x$
for every $\gamma\in\Gamma$ and $x\in\Gamma(Y,f^{\ast}\mathcal{G}_{\gamma})$.
Conversely, a representation $\rho$ of $\mathbb{D}_{X}(\Gamma)$
on a quasi-coherent $\mathcal{O}_{X}$-module $\mathcal{M}$ is sent
to the $\Gamma$-grading on $\mathcal{M}$ defined by the eigenspace
decomposition of $\rho$. Then $\iota$ maps a morphism $\chi:\mathbb{D}_{T}(\Gamma)\rightarrow G_{T}$
in $\mathbb{G}^{\Gamma}(G)(T)$ to the $\Gamma$-graduation on $V_{T}$
defined by
\[
\xyC{2.5pc}\xymatrix{\Rep(G_{T})\ar[r]\sp(0.3){-\circ\chi} & \Rep(\mathbb{D}_{T}(\Gamma))\simeq\Gra^{\Gamma}\QCoh/T\ar[r]\sp(0.65){\mathrm{forg}} & \QCoh/T.}
\]

\begin{rem}
We will show in corollary~\ref{cor:F(G)same} that for a \emph{reductive}
group $G$, the definition of the fpqc sheaf $\mathbb{F}^{\Gamma}(G)$
on $\Sch/S$ given here (image of $\mathbb{G}^{\Gamma}(G)\rightarrow\mathbb{F}^{\Gamma}(V)$)
coincides with the definition of section~\ref{sub:FiltrationsGroup}
(image of $\mathbb{G}^{\Gamma}(G)\rightarrow\mathbb{G}^{\Gamma}(R_{\mathbb{P}(G)})$). 
\end{rem}

\section{The subcategories of rigid objects}

We briefly discuss the $-^{\circ}$ variants of the above definitions,
mostly mentioning the new features.

\subsection{Finite locally free sheaves}

Let $\LF\rightarrow\Sch$ be the fibered category \nomenclature[Gr^GammaLF]{$\Gra^\Gamma \LF$}{Category of $\Gamma$-graded finite locally free sheaves on schemes, page \nomrefpage}\nomenclature[Fil^GammaLF]{$\Fil^\Gamma \LF$}{Category of $\Gamma$-filtered finite locally free sheaves on schemes, page \nomrefpage}\nomenclature[LF]{$\LF$}{Category of finite locally free sheaves on schemes, page \nomrefpage}\nomenclature[Fil]{$\Fi$}{Functor $\Fi : \Gra ^\Gamma \LF \rightarrow \Fil ^\Gamma \LF$, page \nomrefpage}\nomenclature[Gr]{$\Gr$}{Functor $\Gr : \Fil ^\Gamma \LF \rightarrow \Gra ^\Gamma \LF$, page \nomrefpage}whose
fiber over $X$ is the full subcategory $\LF(X)$ of $\QCoh(X)$ whose
objects are the finite locally free sheaves on $X$. \nomenclature[Gr^GammaLF(X)]{$\Gra^\Gamma \LF(X)$}{Category of $\Gamma$-graded finite locally free sheaves on $X$, page \nomrefpage}\nomenclature[Fil^GammaLF(X)]{$\Fil^\Gamma \LF(X)$}{Category of $\Gamma$-filtered finite locally free sheaves on $X$, page \nomrefpage}\nomenclature[LF(X)]{$\LF(X)$}{Category of finite locally free sheaves on $X$, page \nomrefpage}Then
$\LF$ is a substack of $\QCoh$ by \cite[2.5.2]{EGA4.2}. Pulling
back through $\LF\hookrightarrow\QCoh$, we obtain a diagram of fpqc
stacks over $\Sch$, 

\[
\xyC{3pc}\xymatrix{\Gra^{\Gamma}\LF\ar@<2pt>[r]^{\Fi} & \Fil^{\Gamma}\LF\ar@<2pt>[l]^{\Gr}\ar[r]\sp(0.55){\mathrm{forg}} & \LF}
\]
whose fiber over a scheme $X$ is a diagram of exact (in Quillen's
sense) $\otimes$-functors
\[
\xymatrix{\Gra^{\Gamma}\LF(X)\ar@<2pt>[r]^{\Fi} & \Fil^{\Gamma}\LF(X)\ar@<2pt>[l]^{\Gr}\ar[r]\sp(0.55){\mathrm{forg}} & \LF(X).}
\]
An alternative and useful description of the objects of $\Fil^{\Gamma}\LF(X)$
is provided by proposition~\ref{prop:FilonLF} below, which also
implies that the $\Gr$ functor is indeed well-defined. Over a base
scheme $S$, there is the corresponding diagram of fpqc stacks: 
\[
\xyC{3pc}\xymatrix{\Gra^{\Gamma}\LF/S\ar@<2pt>[r]^{\Fi} & \Fil^{\Gamma}\LF/S\ar@<2pt>[l]^{\Gr}\ar[r]\sp(0.55){\mathrm{forg}} & \LF/S}
\]

\subsection{~}

These categories have compatible inner Hom's and duals given by
\[
\underline{\Hom}(x,y)=x^{\vee}\otimes y\quad\mbox{with}\quad(\mathcal{M},\mathcal{G})^{\vee}=(\mathcal{M}^{\vee},\mathcal{G}^{\vee})\quad\mbox{and}\quad(\mathcal{M},\mathcal{F})^{\vee}=(\mathcal{M}^{\vee},\mathcal{F}^{\vee})
\]
where $\mathcal{M}^{\vee}$ is the dual of $\mathcal{M}$, $(\mathcal{G}^{\vee})_{\gamma}=(\mathcal{G}_{-\gamma})^{\vee}$
and $(\mathcal{F}^{\vee})^{\gamma}=(\mathcal{F}_{+}^{-\gamma})^{\bot}=(\mathcal{M}/\mathcal{F}_{+}^{-\gamma})^{\vee}$.
Thus if $\mathcal{G}$ is a splitting of $\mathcal{F}$, then $\mathcal{G}^{\vee}$
is a splitting of $\mathcal{F}^{\vee}$. Moreover, we have
\[
(\mathcal{F}^{\vee})_{+}^{\gamma}=(\mathcal{F}^{-\gamma})^{\bot}\simeq(\mathcal{M}/\mathcal{F}^{-\gamma})^{\vee}\quad\mbox{and}\quad\Gr_{\mathcal{F}^{\vee}}^{\gamma}(\mathcal{M}^{\vee})\simeq\Gr_{\mathcal{F}}^{-\gamma}(\mathcal{M})^{\vee}.
\]
For the inner Homs, we obtain the following formula:
\[
\Gr_{\mathcal{F}}^{\gamma}\left(\underline{\Hom}(\mathcal{M}_{1},\mathcal{M}_{2})\right)\simeq\oplus_{\gamma_{2}-\gamma_{1}=\gamma}\underline{\Hom}\left(\Gr_{\mathcal{F}_{1}}^{\gamma_{1}}(\mathcal{M}_{1}),\Gr_{\mathcal{F}_{2}}^{\gamma_{2}}(\mathcal{M}_{2})\right).
\]

\subsection{$\Gamma$-filtrations on finite locally free sheaves}
\begin{prop}
\label{prop:FilonLF}Let $\mathcal{M}$ be a finite locally free sheaf
on $X$. Let $(\mathcal{F}^{\gamma})_{\gamma\in\Gamma}$ be a non-increasing
collection of quasi-coherent subsheaves of $\mathcal{M}$. Then the
following conditions are equivalent:
\begin{enumerate}
\item For every affine open subset $U$ of $X$, there is a $\Gamma$-graduation
\[
\mathcal{M}_{U}=\oplus_{\gamma\in\Gamma}\mathcal{G}_{\gamma}
\]
 such that $\mathcal{F}_{U}^{\gamma}=\oplus_{\eta\geq\gamma}\mathcal{G}_{\eta}$
for every $\gamma\in\Gamma$. 
\item Locally on $X$ for the Zariski topology, there is a $\Gamma$-graduation
\[
\mathcal{M}=\oplus_{\gamma\in\Gamma}\mathcal{G}_{\gamma}
\]
 such that $\mathcal{F}^{\gamma}=\oplus_{\eta\geq\gamma}\mathcal{G}_{\eta}$
for every $\gamma\in\Gamma$.
\item Locally on $X$ for the fpqc topology, there exists a $\Gamma$-graduation
\[
\mathcal{M}=\oplus_{\gamma\in\Gamma}\mathcal{G}_{\gamma}
\]
 such that $\mathcal{F}^{\gamma}=\oplus_{\eta\geq\gamma}\mathcal{G}_{\eta}$
for every $\gamma\in\Gamma$, i.e. $\mathcal{F}$ is a $\Gamma$-filtration
on $\mathcal{M}$.
\item For every $\gamma\in\Gamma$, $\Gr_{\mathcal{F}}^{\gamma}(\mathcal{M})$
is finite locally free and for every $x\in X$,
\[
\dim_{k(x)}\mathcal{M}(x)={\textstyle \sum_{\gamma}}\dim_{k(x)}\Gr_{\mathcal{F}(x)}^{\gamma}(\mathcal{M}(x)).
\]

\end{enumerate}
In $(4)$, $\mathcal{F}(x)$ is the image of $\mathcal{F}$ in $\mathcal{M}(x)=\mathcal{M}\otimes k(x)$
and $\Gr_{\mathcal{F}}^{\gamma}(\mathcal{M})$, $\Gr_{\mathcal{F}(x)}^{\gamma}(\mathcal{M}(x))$
are defined as usual. Under the above equivalent conditions, for all
$\gamma\in\Gamma$: $\mathcal{F}^{\gamma}$, $\mathcal{F}_{+}^{\gamma}$
and $\Gr_{\mathcal{F}}^{\gamma}(\mathcal{M})$ are finite locally
free sheaves on $X$ and for every $x\in X$, 
\[
\mathcal{F}^{\gamma}(x)\simeq\mathcal{F}^{\gamma}\otimes k(x),\quad\mathcal{F}_{+}^{\gamma}(x)\simeq\mathcal{F}_{+}^{\gamma}\otimes k(x),\quad\Gr_{\mathcal{F}(x)}^{\gamma}(\mathcal{M}(x))\simeq\Gr_{\mathcal{F}}^{\gamma}(\mathcal{M})\otimes k(x).
\]
\end{prop}
\begin{proof}
Plainly $(1)\Rightarrow(2)\Rightarrow(3)$. Moreover $(3)\Rightarrow(4)$
is easy (using~\cite[2.5.2.iii]{EGA4.2}) and the last assertions
follow from $(1)$. To prove that $(4)\Rightarrow(1)$, we may assume
that $X=U$ is affine. Since $\Gr_{\mathcal{F}}^{\gamma}(\mathcal{M})$
is finite locally free by assumption, it is then projective in $\QCoh(X)$
by \cite[Corollary of 7.12]{Ma89}. Therefore, there exists a quasi-coherent
subsheaf $\mathcal{G}_{\gamma}$ of $\mathcal{F}_{\gamma}$ such that
$\mathcal{F}^{\gamma}=\mathcal{G}_{\gamma}\oplus\mathcal{F}_{+}^{\gamma}$.
We will show that
\[
\mathcal{M}=\oplus_{\gamma\in\Gamma}\mathcal{G}_{\gamma}\quad\mbox{and}\quad\forall\gamma:\,\mathcal{F}^{\gamma}=\oplus_{\eta\geq\gamma}\mathcal{G}_{\eta}.
\]
This being now a local question in the Zariski topology of $X$, we
may assume that the rank of $\mathcal{M}$ is constant on $X$, and
also nonzero. Fix $x\in X$ and define
\[
\Gamma(x)=\{\gamma:\Gr_{\mathcal{F}(x)}^{\gamma}(\mathcal{M}(x))\neq0\}=\{\gamma_{1}<\cdots<\gamma_{r}\}.
\]
Define $U_{0}=\mathrm{Supp}(\mathcal{M}/\mathcal{F}^{\gamma_{1}})^{c}$,
$U_{i}=\mathrm{Supp}(\mathcal{F}_{+}^{\gamma_{i}}/\mathcal{F}^{\gamma_{i+1}})^{c}\cap U_{i-1}$
for $0<i<r$ and $U_{r}=\mathrm{Supp}(\mathcal{F}_{+}^{\gamma_{r}})^{c}\cap U_{r-1}$.
Since $\mathcal{M}$ is finite locally free, $\mathcal{M}/\mathcal{F}^{\gamma_{1}}$
is finitely generated and $U_{0}$ is open in $X$. Since $\mathcal{M}=\mathcal{F}^{\gamma_{1}}$
over $U_{0}$ and $\mathcal{F}^{\gamma_{1}}=\mathcal{F}_{+}^{\gamma_{1}}\oplus\mathcal{G}_{\gamma_{1}}$
over $X$, $\mathcal{M}=\mathcal{F}_{+}^{\gamma_{1}}\oplus\mathcal{G}_{\gamma_{1}}$
over $U_{0}$. Therefore $\mathcal{F}_{+}^{\gamma_{1}}$ is finite
locally free over $U_{0}$. Repeating this argument successively with
$(\mathcal{M},X)$ replaced by $(\mathcal{F}_{+}^{\gamma_{1}},U_{0})$,
$(\mathcal{F}_{+}^{\gamma_{2}},U_{1})$ etc\ldots{} we obtain: $U_{r}$
is open in $X$, $\mathcal{M}=\oplus_{i}\mathcal{G}_{\gamma_{i}}$
and $\mathcal{F}^{\gamma}=\oplus_{i:\gamma_{i}\geq\gamma}\mathcal{G}_{\gamma_{i}}$
over $U_{r}$ for every $\gamma\in\Gamma$, with everyone finite locally
free over $U_{r}$. All we have to do now is to show that the formula
of $(4)$ implies that $x$ belongs to $U_{r}$. The formula is equivalent
to: 
\[
\mathcal{F}^{\gamma}(x)=\begin{cases}
\mathcal{M}(x) & \mbox{if }\gamma\leq\gamma_{1},\\
\mathcal{F}^{\gamma_{i+1}}(x) & \mbox{if }\gamma\in]\gamma_{i},\gamma_{i+1}],\\
0 & \mbox{if }\gamma>\gamma_{r}.
\end{cases}
\]
Since $\mathcal{M}$ is finitely generated over $X$, $\mathcal{F}^{\gamma_{1}}(x)=\mathcal{M}(x)$
implies $\mathcal{F}_{x}^{\gamma_{1}}=\mathcal{M}_{x}$ by Nakayama's
lemma, thus $x$ belongs to $U_{0}$. Since $\mathcal{M}=\mathcal{F}^{\gamma_{1}}=\mathcal{F}_{+}^{\gamma_{1}}\oplus\mathcal{G}_{\gamma_{1}}$
over $U_{0}$, $\mathcal{F}_{+}^{\gamma_{1}}(x)=\mathcal{F}^{\gamma_{2}}(x)$
in $\mathcal{M}(x)$ implies $\mathcal{F}_{+,x}^{\gamma_{1}}=\mathcal{F}_{x}^{\gamma_{2}}$
by Nakayama's lemma, therefore $x$ belongs to $U_{1}$. Repeating
the argument, we find that indeed $x$ belongs to $U_{r}$. \end{proof}
\begin{rem}
The whole proof becomes much simpler over a Noetherian base.\end{rem}
\begin{lem}
\label{lem:SumsOfFilsOnLF}Let $\mathcal{M}_{\alpha}$ be a finite
collection of locally free sheaves of finite rank on $X$ and for
each $\alpha$, let $(\mathcal{F}_{\alpha}^{\gamma})_{\gamma\in\Gamma}$
be a non-increasing collection of quasi-coherent subsheaves of $\mathcal{M}_{\alpha}$.
Set $\mathcal{M}=\oplus\mathcal{M}_{\alpha}$ and $\mathcal{F}^{\gamma}=\oplus\mathcal{F}_{\alpha}^{\gamma}$.
Then $(\mathcal{M},(\mathcal{F}^{\gamma}))$ satisfies the above equivalent
conditions if and only if each $(\mathcal{M}_{\alpha},(\mathcal{F}_{\alpha}^{\gamma}))$
does. \end{lem}
\begin{proof}
For every $\gamma\in\Gamma$ and $x\in X$, $\Gr_{\mathcal{F}}^{\gamma}(\mathcal{M})=\oplus_{\alpha}\Gr_{\mathcal{F}^{\alpha}}^{\gamma}(\mathcal{M}_{\alpha})$
and 
\[
\mathcal{M}(x)=\oplus_{\alpha}\mathcal{M}_{\alpha}(x),\quad\Gr_{\mathcal{F}(x)}^{\gamma}(\mathcal{M}(x))=\oplus_{\alpha}\Gr_{\mathcal{F}_{\alpha}(x)}^{\gamma}(\mathcal{M}_{\alpha}(x)).
\]
Moreover for every $\alpha$ and $x\in X$,
\[
\dim_{k(x)}\mathcal{M}_{\alpha}(x)\geq{\textstyle \sum_{\gamma}}\dim_{k(x)}\Gr_{\mathcal{F}_{\alpha}(x)}^{\gamma}(\mathcal{M}_{\alpha}(x)).
\]
The lemma easily follows.
\end{proof}

\subsection{~}

Let $\Rep^{\circ}(G)\rightarrow\Sch/S$ be the substack \nomenclature[Repo(G)]{$\Rep ^{\circ}  (G)$}{Fibered category of algebraic representations of $G$ on finite locally free sheaves, page \nomrefpage}of
$\Rep(G)\rightarrow\Sch/S$ whose fiber over $T\rightarrow S$ is
the exact, rigid, full sub-$\otimes$-category $\Rep^{\circ}(G)(T)$
\nomenclature[Repo(G)(X)]{$\Rep ^{\circ} (G)(X)$}{Category of algebraic representations of $G$ on finite locally free sheaves over $X$, page \nomrefpage}of
$\Rep(G)(T)$ whose objects are the representations of $G_{T}$ on
finite locally free sheaves on $T$. We write\nomenclature[Vo]{$V^\circ$}{Fiber functor $V^\circ: \Rep ^\circ (G) \rightarrow \LF / S$, page \nomrefpage}
\[
V^{\circ}:\Rep^{\circ}(G)\rightarrow\LF/S
\]
for the forgetful functor. For an $S$-scheme $T\rightarrow S$, we
denote by\nomenclature[omega_X^o]{$\omega ^\circ _X$}{Fiber functor $\omega ^\circ _X: \Rep ^\circ (G)(S) \rightarrow \LF (X) $, page \nomrefpage}
\[
V_{T}^{\circ}:\Rep^{\circ}(G_{T})\rightarrow\LF/T\quad\mbox{and}\quad\omega_{T}^{\circ}:\Rep^{\circ}(G)(S)\rightarrow\LF(T)
\]
the induced morphism of fpqc stack over $\Sch/T$ and fiber functor.
Note that $\omega_{T}^{\circ}$ is now an exact $\otimes$-functor,
since all short exact sequences in $\Rep^{\circ}(G)(S)$ are pure.

\subsection{~}

We obtain yet another equivariant diagram of fpqc sheaves on $S$,
\nomenclature[G^Gamma(Vo)]{$\mathbb{G}^\Gamma (V^\circ )$}{Sheaf of $\Gamma$-graduations on $V^\circ$, page \nomrefpage}\nomenclature[G^Gamma(omegao)]{$\mathbb{G}^\Gamma (\omega ^\circ )$}{Sheaf of $\Gamma$-graduations on $\omega ^\circ$, page \nomrefpage}\nomenclature[F^Gamma(Vo)]{$\mathbb{F}^\Gamma (V^\circ )$}{Sheaf of $\Gamma$-filtrations on $V^\circ $, page \nomrefpage}\nomenclature[F^Gamma(omegao)]{$\mathbb{F}^\Gamma (\omega^\circ  )$}{Sheaf of $\Gamma$-filtrations on $\omega^\circ $, page \nomrefpage}\nomenclature[Fil]{$\Fi$}{Morphism $\Fi : \mathbb{G}^\Gamma (X) \rightarrow \mathbb{F}^\Gamma (X)$ for $X=V^\circ $ or $\omega^\circ $, page \nomrefpage}\nomenclature[Aut^o(Vo)]{$\Aut^{\otimes}(V^\circ )$}{Sheaf of tensor automorphisms of $V^\circ $, page \nomrefpage}\nomenclature[Aut^o(omegao)]{$\Aut^{\otimes}(\omega^\circ  )$}{Sheaf of tensor automorphisms of $\omega^\circ $, page \nomrefpage}
\[
\xyC{2pc}\xyR{4pt}\xymatrix{\Aut^{\otimes}(V^{\circ})\ar[dd]^{\mathrm{res}} &  & \mathbb{G}^{\Gamma}(V^{\circ})\ar[r]^{\Fi}\ar[dd]^{\mathrm{res}} & \mathbb{F}^{\Gamma}(V^{\circ})\ar[dd]^{\mathrm{res}}\\
 & \mbox{acting on}\\
\Aut^{\otimes}(\omega^{\circ}) &  & \mathbb{G}^{\Gamma}(\omega^{\circ})\ar[r]^{\Fi} & \mathbb{F}^{\Gamma}(\omega^{\circ})
}
\]
where everything is defined as before, using $V^{\circ}$ and $\omega^{\circ}$
instead of $V$ and $\omega$. The only differences worth mentioning
are as follows: for any $S$-scheme $T$, the $\Gamma$-graduations
or $\Gamma$-filtrations on $\omega_{T}^{\circ}$ are automatically
compatible with inner Homs and duals, and there $\gamma$-components
are exact functors. We also have equivariant diagrams
\[
\xyR{4pt}\xymatrix{\Aut^{\otimes}(V)\ar[dd]^{\mathrm{res}} &  & \mathbb{G}^{\Gamma}(V)\ar[r]^{\Fi}\ar[dd]^{\mathrm{res}} & \mathbb{F}^{\Gamma}(V)\ar[dd]^{\mathrm{res}}\\
 & \mbox{acting on}\\
\Aut^{\otimes}(V^{\circ}) &  & \mathbb{G}^{\Gamma}(V^{\circ})\ar[r]^{\Fi} & \mathbb{F}^{\Gamma}(V^{\circ})
}
\]
and similarly for $\omega$ and $\omega^{\circ}$, where all the vertical
maps are induced by pre-composition with the full embedding $\Rep^{\circ}(G)\hookrightarrow\Rep(G)$.

\subsection{~}

Finally, the definitions of $\Aut^{\otimes}(\mathcal{G})$, $\Aut^{\otimes}(\mathcal{F})$,
$\Aut^{\otimes!}(\mathcal{F})$ and $\Aut^{\otimes}(\Gr_{\mathcal{F}}^{\bullet})$
given in section~\ref{sub:Aut(F)Aut(G)GrF} carry over to the situation
considered here.

\section{Skalar extensions}

The whole diagram at the beginning of this section has now been defined.
It is covariantly functorial in $G$ but not entirely compatible with
base change on $S$: if $\tilde{S}\rightarrow S$ is any morphism,
$\tilde{G}=G\times_{S}\tilde{S}$ and $\tilde{V}$, $\tilde{\omega}$
\ldots{} are the relevant functors for $\tilde{G}$, then $\mathbb{G}^{\Gamma}(\tilde{G})=\mathbb{G}^{\Gamma}(G)\vert_{\tilde{S}}$,
$\mathbb{F}^{\Gamma}(\tilde{G})=\mathbb{F}^{\Gamma}(G)\vert_{\tilde{S}}$
and 
\[
\Aut^{\otimes}(\tilde{X})=\Aut^{\otimes}(X)\vert_{\tilde{S}},\quad\mathbb{G}^{\Gamma}(\tilde{X})=\mathbb{G}^{\Gamma}(X)\vert_{\tilde{S}}\quad\mbox{and}\quad\mathbb{F}^{\Gamma}(\tilde{X})=\mathbb{F}^{\Gamma}(X)\vert_{\tilde{S}}
\]
for $X\in\{V,V^{\circ}\}$, but the natural morphisms of fpqc sheaves
on $\tilde{S}$,
\[
\Aut^{\otimes}(\tilde{Y})\rightarrow\Aut^{\otimes}(Y)\vert_{\tilde{S}},\quad\mathbb{G}^{\Gamma}(\tilde{Y})\rightarrow\mathbb{G}^{\Gamma}(Y)\vert_{\tilde{S}}\quad\mbox{and}\quad\mathbb{F}^{\Gamma}(\tilde{Y})\rightarrow\mathbb{F}^{\Gamma}(Y)\vert_{\tilde{S}}
\]
may not be isomorphisms for $Y\in\{\omega,\omega^{\circ}\}$. We investigate
this issue.

\subsection{~}

When $\mathsf{C}$ is a category and $\mathcal{B}$ is a ring object
in $\mathsf{C}$, we can form the category $\mathsf{C}(\mathcal{B})$
of (left) $\mathcal{B}$-modules in $\mathsf{C}$. Here $\mathsf{C}$
will be an additive $\otimes$-category and the ring object will be
given by its multiplication morphism $\mu:\mathcal{B}\otimes\mathcal{B}\rightarrow\mathcal{B}$
and unit $1\rightarrow\mathcal{B}$, where $1$ is the neutral object
for the tensor product, the abelian group structure on $\mathcal{B}$
being provided by the additive structure of $\mathsf{C}$. Then $\mathsf{C}(\mathcal{B})$
is the category of pairs $(\mathcal{M},\nu)$ where $\mathcal{M}$
is an object of $\mathsf{C}$ and $\nu:\mathcal{B}\otimes\mathcal{M}\rightarrow\mathcal{M}$
is a morphism in $\mathsf{C}$ subject to certain natural conditions.
There is an adjunction 
\[
f^{\ast}:\mathsf{C}\leftrightarrow\mathsf{C}(\mathcal{B}):f_{\ast}\quad\mbox{given by}\quad f_{\ast}(\mathcal{M},\nu)=\mathcal{M}\mbox{ and }f^{\ast}(\mathcal{N})=(\mathcal{B}\otimes\mathcal{N},\mu\otimes\mathrm{Id}).
\]
In many cases, it is also possible to equip $\mathsf{C}(\mathcal{B})$
with a $\otimes$-product inherited from the $\otimes$-product on
$\mathsf{C}$, with $(\mathcal{B},\mu)$ as neutral object. Instead
of trying to develop this formal theory more rigorously, let us list
some of the relevant examples:

\subsection*{$\mathsf{C}=\QCoh(S)$ and $\mathcal{B}=f_{\ast}\mathcal{O}_{T}$
where $f:T\rightarrow S$ is an affine morphism. }

There is an equivalence of $\otimes$-categories $\mathsf{C}(\mathcal{B})\simeq\QCoh(T)$
which is compatible with the usual adjunctions $f^{\ast}:\QCoh(S)\leftrightarrow\QCoh(T):f_{\ast}$,
see \cite[1.4]{EGA2}.

\subsection*{$\mathsf{C}=\Gra^{\Gamma}\QCoh(S)$ and $\mathcal{B}$ as above with
the trivial $\Gamma$-graduation.}

The first example induces an equivalence of $\otimes$-categories
$\mathsf{C}(\mathcal{B})\simeq\Gra^{\Gamma}\QCoh(T)$ which is again
compatible with the natural adjunctions.

\subsection*{$\mathsf{C}=\Fil^{\Gamma}\QCoh(S)$ and $\mathcal{B}$ as above with
the trivial $\Gamma$-filtration. }

The first example now only induces a fully faithful exact $\otimes$-functor
$\mathsf{C}(\mathcal{B})\hookrightarrow\Fil^{\Gamma}\QCoh(T)$. The
essential image is made of those $\Gamma$-filtered quasi-coherent
sheaves $(\mathcal{M},\mathcal{F})$ on $T$ such that, locally on
$S$ (as opposed to $T$) for the fpqc topology, $\mathcal{F}$ has
a splitting.

\subsection*{$\mathsf{C}=\Rep(G)(S)$ and $\mathcal{B}$ as above with the trivial
action of $G$. }

The first example again induces an equivalence of $\otimes$-categories
$\mathsf{C}(\mathcal{B})\simeq\Rep(G)(T)$ which is compatible with
the adjunctions given on the comodules by the following formulas:
\begin{eqnarray*}
f^{\ast}\left(V(\rho)\stackrel{c_{\rho}}{\longrightarrow}V(\rho)\otimes_{\mathcal{O}_{S}}\mathcal{A}(G)\right) & = & \left(V(f^{\ast}\rho)\stackrel{c_{f^{\ast}\rho}}{\longrightarrow}V(f^{\ast}\rho)\otimes_{\mathcal{O}_{T}}\mathcal{A}(G_{T})\right),\\
f_{\ast}\left(V(\rho)\stackrel{c_{\rho}}{\longrightarrow}V(\rho)\otimes_{\mathcal{O}_{T}}\mathcal{A}(G_{T})\right) & = & \left(V(f_{\ast}\rho)\stackrel{c_{f_{\ast}\rho}}{\longrightarrow}V(f_{\ast}\rho)\otimes_{\mathcal{O}_{S}}\mathcal{A}(G)\right).
\end{eqnarray*}

\subsection*{$\mathsf{C}=\LF(S)$ and $\mathcal{B}=f_{\ast}\mathcal{O}_{T}$ where
$f:T\rightarrow S$ is a finite étale morphism. }

The first example induces an equivalence of $\otimes$-categories
$\mathsf{C}(\mathcal{B})\simeq\LF(T)$. We have to show that for a
quasi-coherent sheaf $\mathcal{M}$ on $T$, $\mathcal{M}$ is a finite
locally free $\mathcal{O}_{T}$-module if and only if $f_{\ast}\mathcal{M}$
is a finite locally free $\mathcal{O}_{S}$-module (the direct implication
is easy, and only requires $f$ to be finite and locally free). By~\cite[2.5.2]{EGA4.2},
our claim is local in the fpqc topology on $S$. But, locally on $S$
for the étale topology, our finite étale morphism $f$ is simply a
finite disjoint union of open and closed embeddings (this follows
from \cite[17.9.3]{EGA4.4}), for which the claim is now obvious. 

Combining this last example with the previous three, we obtain:

\subsection*{$\mathsf{C}=\Gra^{\Gamma}\LF(S)$ and $\mathcal{B}$ as above with
the trivial $\Gamma$-graduation.}

Then 
\[
\mathsf{C}(\mathcal{B})\simeq\Gra^{\Gamma}\LF(T).
\]

\subsection*{$\mathsf{C}=\Fil^{\Gamma}\LF(S)$ and $\mathcal{B}$ as above with
the trivial $\Gamma$-filtration.}

Then 
\[
\mathsf{C}(\mathcal{B})\simeq\Fil^{\Gamma}\LF(T).
\]

\subsection*{$\mathsf{C}=\Rep^{\circ}(G)(S)$ and $\mathcal{B}$ as above with
the trivial action. }

Then 
\[
\mathsf{C}(\mathcal{B})\simeq\Rep^{\circ}(G)(T).
\]

\subsection{~}

The point of this abstract nonsense is that, if $\alpha:\mathsf{C}\rightarrow\mathsf{D}$
is a $\otimes$-functor and $\mathcal{B}$ is a ring object in $\mathsf{C}$,
then $\alpha(\mathcal{B})$ is a ring object in $\mathsf{D}$ and
$\alpha$ extends to a $\otimes$-functor $\alpha(\mathcal{B}):\mathsf{C}(\mathcal{B})\rightarrow\mathsf{D}(\alpha(\mathcal{B}))$
which we call the skalar extension of $\alpha$. Similarly, if $\eta$
is a $\otimes$-automorphism of $\alpha$ such that $\eta_{\mathcal{B}}$
is the identity of $\alpha(\mathcal{B})$, then $\eta$ extends to
a $\otimes$-automorphism $\eta(\mathcal{B})$ of $\alpha(\mathcal{B})$
which we call the skalar extension of $\eta$. 
\begin{prop}
\label{prop:SkalarExt}$(1)$ Let $f:\widetilde{S}\rightarrow S$
be a finite étale morphism and denote by $\tilde{\omega}$ the fiber
functors for $\tilde{G}=G_{\tilde{S}}$. Then we have isomorphisms
of fpqc sheaves on $\tilde{S}$:
\[
\Aut^{\otimes}(\omega^{\circ})\vert_{\widetilde{S}}=\Aut^{\otimes}(\tilde{\omega}^{\circ}),\quad\mathbb{G}^{\Gamma}(\omega^{\circ})\vert_{\widetilde{S}}=\mathbb{G}^{\Gamma}(\tilde{\omega}^{\circ})\quad\mbox{and}\quad\mathbb{F}^{\Gamma}(\omega^{\circ})\vert_{\widetilde{S}}=\mathbb{F}^{\Gamma}(\tilde{\omega}^{\circ}).
\]
$(2)$ If $f$ is merely affine, then $\mathbb{F}^{\Gamma}(\omega)\vert_{\tilde{S}}=\mathbb{F}^{\Gamma}(\tilde{\omega})$. \end{prop}
\begin{proof}
$(1)$ Let $T$ be an $\tilde{S}$-scheme. We have to define mutually
inverse maps 
\[
\alpha:\begin{array}{rcl}
\Aut^{\otimes}(\tilde{\omega}^{\circ})(T) & \longleftrightarrow & \Aut^{\otimes}(\omega^{\circ})(T)\\
\mathbb{G}^{\Gamma}(\tilde{\omega}^{\circ})(T) & \longleftrightarrow & \mathbb{G}^{\Gamma}(\omega^{\circ})(T)\\
\mathbb{F}^{\Gamma}(\tilde{\omega}^{\circ})(T) & \longleftrightarrow & \mathbb{F}^{\Gamma}(\omega^{\circ})(T)
\end{array}:\beta
\]
functorial in $T$. The $\alpha$ maps are induced by precomposition
with the base change map $\Rep^{\circ}(G)(S)\rightarrow\Rep^{\circ}(G)(\widetilde{S})$.
The $\beta$ maps are defined by composing the skalar extension maps
with the base change maps for the $\tilde{S}$-section $\iota:T\rightarrow\tilde{T}$
of the projection $f_{T}:\tilde{T}=T\times_{S}\tilde{S}\rightarrow T$
given by the structural morphism $T\rightarrow\tilde{S}$:
\[
\beta:\begin{array}{ccccc}
\Aut^{\otimes}(\omega^{\circ})(T) & \longrightarrow & \Aut^{\otimes}(\tilde{\omega}^{\circ})(\tilde{T}) & \longrightarrow & \Aut^{\otimes}(\tilde{\omega}^{\circ})(T)\\
\mathbb{G}^{\Gamma}(\omega^{\circ})(T) & \longrightarrow & \mathbb{G}^{\Gamma}(\tilde{\omega}^{\circ})(\tilde{T}) & \longrightarrow & \mathbb{G}^{\Gamma}(\tilde{\omega}^{\circ})(T)\\
\mathbb{F}^{\Gamma}(\omega^{\circ})(T) & \longrightarrow & \mathbb{F}^{\Gamma}(\tilde{\omega}^{\circ})(\tilde{T}) & \longrightarrow & \mathbb{F}^{\Gamma}(\tilde{\omega}^{\circ})(T)
\end{array}
\]
Explicitly, for $\eta$, $\mathcal{G}$ and $\mathcal{F}$ in the
source sets and $\tilde{\rho}\in\Rep^{\circ}(\tilde{G})(\tilde{S})$,
we first view $f_{\ast}\tilde{\rho}$ as a $\mathcal{B}$-module in
$\Rep^{\circ}(G)(S)$ where $\mathcal{B}=f_{\ast}\mathcal{O}_{\tilde{S}}$
with trivial $G$-action. Then:
\begin{itemize}
\item $\eta_{f_{\ast}\tilde{\rho}}$ is a $\mathcal{B}_{T}$-linear isomorphism
of $\omega_{T}^{\circ}(f_{\ast}\tilde{\rho})=(f_{T})_{\ast}\tilde{\omega}_{\widetilde{T}}^{\circ}(\tilde{\rho})$.
It thus corresponds to an isomorphism of $\tilde{\omega}_{\tilde{T}}^{\circ}(\tilde{\rho})$
whose pull-back to $\iota^{\ast}\tilde{\omega}_{\tilde{T}}^{\circ}(\tilde{\rho})=\tilde{\omega}_{T}^{\circ}(\tilde{\rho})$
is an isomorphism $\beta(\eta)_{\tilde{\rho}}$. By construction,
there is a commutative diagram
\[
\xyR{2pc}\xymatrix{\omega_{T}^{\circ}(f_{\ast}\tilde{\rho})=\tilde{\omega}_{T}^{\circ}(f^{\ast}f_{\ast}\tilde{\rho})\ar@{->>}[r]\ar[d]_{\eta_{f_{\ast}\tilde{\rho}}} & \tilde{\omega}_{T}^{\circ}(\tilde{\rho})\ar[d]^{\beta(\eta)_{\tilde{\rho}}}\\
\omega_{T}^{\circ}(f_{\ast}\tilde{\rho})=\tilde{\omega}_{T}^{\circ}(f^{\ast}f_{\ast}\tilde{\rho})\ar@{->>}[r] & \tilde{\omega}_{T}^{\circ}(\tilde{\rho})
}
\]
where the horizontal map comes from the adjunction morphism 
\[
f^{\ast}f_{\ast}\tilde{\rho}\rightarrow\tilde{\rho}.
\]

\item $\mathcal{G}(f_{\ast}\tilde{\rho})$ is a $\mathcal{B}_{T}$-stable
$\Gamma$-graduation on $(f_{T})_{\ast}\tilde{\omega}_{\tilde{T}}^{\circ}(\tilde{\rho})$,
giving a $\Gamma$-graduation on $\tilde{\omega}_{\tilde{T}}^{\circ}(\tilde{\rho})$
whose pull-back is a $\Gamma$-graduation $\beta(\mathcal{G})(\tilde{\rho})$
on $\tilde{\omega}_{T}^{\circ}(\tilde{\rho})$. Thus $\beta(\mathcal{G})_{\gamma}(\tilde{\rho})$
is the image of $\mathcal{G}_{\gamma}(f_{\ast}\tilde{\rho})$ under
the adjunction $\omega_{T}^{\circ}(f_{\ast}\tilde{\rho})\twoheadrightarrow\tilde{\omega}_{T}^{\circ}(\tilde{\rho})$. 
\item $\mathcal{F}(f_{\ast}\tilde{\rho})$ is a $\mathcal{B}_{T}$-stable
$\Gamma$-filtration on $(f_{T})_{\ast}\tilde{\omega}_{\tilde{T}}^{\circ}(\tilde{\rho})$,
giving a $\Gamma$-filtration on $\tilde{\omega}_{\tilde{T}}^{\circ}(\tilde{\rho})$
whose pull-back is a $\Gamma$-filtration $\beta(\mathcal{F})(\tilde{\rho})$
on $\tilde{\omega}_{T}^{\circ}(\tilde{\rho})$. Thus $\mathcal{B}(\mathcal{F})^{\gamma}(\tilde{\rho})$
is the image of $\mathcal{F}^{\gamma}(f_{\ast}(\tilde{\rho}))$ under
the adjunction $\omega_{T}^{\circ}(f_{\ast}\tilde{\rho})\twoheadrightarrow\tilde{\omega}_{T}^{\circ}(\tilde{\rho})$. 
\end{itemize}
One checks easily that $\alpha\circ\beta=\mathrm{Id}$ and $\beta\circ\alpha=\mathrm{Id}$.
The proof of $(2)$ is similar.\end{proof}
\begin{rem}
We have not mentioned $\Aut^{\otimes}(\omega)$ and $\mathbb{G}^{\Gamma}(\omega)$
in part $(2)$ of the proposition, because we will establish a stronger
result for them in the next section.
\end{rem}

\section{The regular representation}

The single most important representation of $G$ is the regular representation
$\rho_{\mathrm{reg}}$. We shall use it to establish the classical:
\begin{thm}
\label{thm:RecThm}The above morphisms of fpqc sheaves induce isomorphisms
\[
G\simeq\Aut^{\otimes}(V)\simeq\Aut^{\otimes}(\omega)\quad\mbox{and}\quad\mathbb{G}^{\Gamma}(G)\simeq\mathbb{G}^{\Gamma}(V)\simeq\mathbb{G}^{\Gamma}(\omega).
\]

\end{thm}

\subsection{~}

The regular representation $\rho_{\mathrm{reg}}$ \nomenclature[rho_reg]{$\rho _ \mathrm{reg}$}{Regular representation of $G$ on $\mathcal{A}(G)$, page \nomrefpage}of
$G$ on $V(\rho_{\mathrm{reg}})=\mathcal{A}(G)$ is defined by 
\[
(g\cdot a)(h)=a(hg)
\]
for $T\rightarrow S$, $a\in\Gamma(T,\mathcal{A}(G)_{T})=\Gamma(G_{T},\mathcal{O}_{G_{T}})$
and $g,h\in G(T)$. The corresponding $\mathcal{A}(G)$-comodule structure
morphism is the comultiplication map:\nomenclature[mu^nat_G]{$\mu^\natural _G$}{Comultiplication $\mu ^\natural _{G}: \mathcal{A}(G) \rightarrow \mathcal{A}(G) \otimes \mathcal{A}(G)$ of $\mathcal{A}(G)$.}
\[
\left(V(\rho_{\mathrm{reg}})\stackrel{c_{\mathrm{reg}}}{\longrightarrow}V(\rho_{\mathrm{reg}})\otimes_{\mathcal{O}_{S}}\mathcal{A}(G)\right)=\left(\mathcal{A}(G)\stackrel{\mu^{\natural}}{\longrightarrow}\mathcal{A}(G)\otimes_{\mathcal{O}_{S}}\mathcal{A}(G)\right)
\]
The $\mathcal{O}_{S}$-algebra structure morphisms on $\mathcal{A}(G)$,
namely the unit $\mathcal{O}_{S}\rightarrow\mathcal{A}(G)$ and the
multiplication $\mathcal{A}(G)\otimes\mathcal{A}(G)\rightarrow\mathcal{A}(G)$
correspond to $G$-equivariant morphisms 
\[
1_{S}\rightarrow\rho_{\mathrm{reg}}\quad\mbox{and}\quad\rho_{\mathrm{reg}}\otimes\rho_{\mathrm{reg}}\rightarrow\rho_{\mathrm{reg}}.
\]
For any $\rho\in\Rep(G)(S)$, we denote by $\rho_{0}\in\Rep(G)(S)$
the trivial representation \nomenclature[rho_0]{$\rho_0$}{Trivial representation of $G$ on $V(\rho)$.}of
$G$ on $V(\rho_{0})=V(\rho)$. We may then view the $\mathcal{A}(G)$-comodule
structure morphism $c_{\rho}:V(\rho)\rightarrow V(\rho)\otimes_{\mathcal{O}_{S}}\mathcal{A}(G)$
of $\rho$ as a $G$-equivariant morphism in $\Rep(G)(S)$\nomenclature[c_rho]{$c_\rho$}{Comodule structure morphism $V(\rho ) \rightarrow V(\rho ) \otimes \mathcal{A}(G)$ of $\rho$.}
\[
c_{\rho}:\rho\rightarrow\rho_{0}\otimes\rho_{\mathrm{reg}}
\]
The underlying morphism of quasi-coherent sheaves on $S$ is a split
monomorphism since $(\mathrm{Id}\otimes1_{G}^{\natural})\circ c_{\rho}=\mathrm{Id}$
on $V(\rho)$ where $1_{G}^{\natural}:\mathcal{A}(G)\rightarrow\mathcal{O}_{S}$
is the counit of $\mathcal{A}(G)$.\nomenclature[1^nat_G]{$1^\natural _G$}{Counit $1_{G}^{\natural}:\mathcal{A}(G)\rightarrow\mathcal{O}_{S}$ of $\mathcal{A}(G)$.}

\subsection{~\label{sub:G(w)F(w)detbyrhoreg}}

It follows that any $\eta\in\Aut^{\otimes}(\omega_{T})$, $\mathcal{G}\in\mathbb{G}^{\Gamma}(\omega_{T})$
or $\mathcal{F}\in\mathbb{F}^{\Gamma}(\omega_{T})$ is uniquely determined
by its value $\eta_{\mathrm{reg}}$, $\mathcal{G}_{\mathrm{reg}}$
or $\mathcal{F}_{\mathrm{reg}}$ on $\rho_{\mathrm{reg}}$. Indeed
for any $\rho\in\Rep(G)(S)$, $\eta_{\rho}$, $\mathcal{G}(\rho)$
and $\mathcal{F}(\rho)$ will then be the automorphism, $\Gamma$-graduation
and $\Gamma$-filtration on 
\[
\xyC{4pc}\xymatrix{\omega_{T}(\rho)\ar@{^{(}->}[r]\sp(0.35){\omega_{T}(c_{\rho})} & \omega_{T}(\rho_{0})\otimes\omega_{T}(\rho_{\mathrm{reg}})}
\]
which are respectively induced by the corresponding objects for $\rho_{0}\otimes\rho_{\mathrm{reg}}$,
namely
\begin{eqnarray*}
\eta_{\rho_{0}\otimes\rho_{\mathrm{reg}}} & = & \mathrm{Id}\otimes\eta_{\mathrm{reg}},\\
\mathcal{G}(\rho_{0}\otimes\rho_{\mathrm{reg}}) & = & \omega_{T}(\rho_{0})\otimes\mathcal{G}_{\mathrm{reg}},\\
\mathcal{F}(\rho_{0}\otimes\rho_{\mathrm{reg}}) & = & \omega_{T}(\rho_{0})\otimes\mathcal{F}_{\mathrm{reg}}.
\end{eqnarray*}
We have here used the defining axioms (A1) and (A2) for $\eta$, (G1)
and (G2) for $\mathcal{G}$ and (F1) and (F2) for $\mathcal{F}$,
as well as the fact that for every $\gamma\in\Gamma$, the functors
$\mathcal{G}_{\gamma}$ and $\mathcal{F}^{\gamma}:\Rep(G)(S)\rightarrow\QCoh(T)$
both preserve pure short exact sequences.

\subsection{~\label{sub:injofVtoomega}}

By the same token, we find that the morphisms of fpqc sheaves
\[
\Aut^{\otimes}(V)\rightarrow\Aut^{\otimes}(\omega),\quad\mathbb{G}^{\Gamma}(V)\rightarrow\mathbb{G}^{\Gamma}(\omega)\quad\mbox{and}\quad\mathbb{F}^{\Gamma}(V)\rightarrow\mathbb{F}^{\Gamma}(\omega)
\]
are monomorphisms. For instance if $\eta\in\Aut^{\otimes}(V_{T})$
induces the identity of $\omega_{T}$, then for any $f:X\rightarrow T$
and $\rho\in\Rep(G)(X)$, $\eta_{\rho}$ is the identity of $V(\rho)$
because 
\[
\eta_{\rho_{0}\otimes\rho_{\mathrm{reg},X}}=\eta_{\rho_{0}}\otimes\eta_{\rho_{\mathrm{reg},X}}=\mathrm{Id}_{V(\rho_{0})}\otimes f^{\ast}(\eta_{\rho_{\mathrm{reg},T}})
\]
and $\eta_{\rho_{\mathrm{reg},T}}$ is the trivial automorphism of
$V(\rho_{\mathrm{reg},T})=\omega_{T}(\rho_{\mathrm{reg}})$.

\subsection{\label{sub:RecThG}~}

We show that $G=\Aut^{\otimes}(\omega)$. Fix an $S$-scheme $T$
and $\eta\in\Aut^{\otimes}(\omega_{T})$. Recall that $\eta_{\mathrm{reg}}$
is the $\mathcal{O}_{T}$-linear automorphism of $\omega_{T}(\rho_{\mathrm{reg}})=\mathcal{A}(G_{T})$
induced by $\eta$. Since $\eta_{1_{S}}=\mathrm{Id}_{\mathcal{O}_{T}}$
on $\omega_{T}(1_{S})=\mathcal{O}_{T}$ by (A2) and $\eta_{\rho_{\mathrm{reg}}\otimes\rho_{\mathrm{reg}}}=\eta_{\mathrm{reg}}\otimes\eta_{\mathrm{reg}}$
on 
\[
\omega_{T}(\rho_{\mathrm{reg}}\otimes\rho_{\mathrm{reg}})=\mathcal{A}(G_{T})\otimes\mathcal{A}(G_{T})
\]
by (A1), the functoriality of $\eta$ applied to $1_{S}\rightarrow\rho_{\mathrm{reg}}$
and $\rho_{\mathrm{reg}}\otimes\rho_{\mathrm{reg}}\rightarrow\rho_{\mathrm{reg}}$
implies that $\eta_{\mathrm{reg}}$ is an automorphism of the quasi-coherent
$\mathcal{O}_{T}$-algebra $\mathcal{A}(G_{T})$. Similarly for any
$\rho\in\Rep(G)(S)$, the $G$-equivariant morphism $c_{\rho}:\rho\rightarrow\rho_{0}\otimes\rho_{\mathrm{reg}}$
induces a commutative diagram of quasi-coherent $\mathcal{O}_{T}$-modules
\[
\xyC{3pc}\xymatrix{\omega_{T}(\rho)\ar[r]\sp(.34){(c_{\rho})_{T}}\ar[d]_{\eta_{\rho}} & \omega_{T}(\rho_{0})\otimes_{\mathcal{O}_{T}}\mathcal{A}(G_{T})\ar[d]^{\mathrm{Id}\otimes\eta_{\mathrm{reg}}}\\
\omega_{T}(\rho)\ar[r]\sp(.34){(c_{\rho})_{T}} & \omega_{T}(\rho_{0})\otimes_{\mathcal{O}_{T}}\mathcal{A}(G_{T})
}
\]
Composing $\eta_{\mathrm{reg}}$ with the counit $1_{G,T}^{\natural}:\mathcal{A}(G)_{T}\rightarrow\mathcal{O}_{T}$,
we obtain a morphism of $\mathcal{O}_{T}$-algebras $s(\eta)^{\natural}:\mathcal{A}(G)_{T}\rightarrow\mathcal{O}_{T}$,
i.e.~a $T$-valued point $s(\eta)\in G(T)$. Now for any $g\in G(T)$
corresponding to $g^{\natural}:\mathcal{A}(G)_{T}\rightarrow\mathcal{O}_{T}$
and mapping to $\iota(g)\in\Aut^{\otimes}(\omega_{T})$, the automorphism
$\iota(g)_{\rho}=\rho_{T}(g)$ of $\omega_{T}(\rho)$ is obtained
by composing $(c_{\rho})_{T}$ with 
\[
\mathrm{Id}\otimes g^{\natural}:\omega_{T}(\rho)\otimes_{\mathcal{O}_{T}}\mathcal{A}(G)_{T}\rightarrow\omega_{T}(\rho).
\]
We thus find that $s\circ\iota(g)=g$ since {\small{
\[
s\circ\iota(g)^{\natural}=1_{G,T}^{\natural}\circ\iota(g)_{\mathrm{reg}}=1_{G_{T}}^{\natural}\circ(\mathrm{Id}\otimes g^{\natural})\circ\mu_{T}^{\natural}=(1_{G_{T}}^{\natural}\otimes g^{\natural})\circ\mu_{T}^{\natural}=(1_{G_{T}}\cdot g)^{\natural}=g^{\natural}.
\]
}}On the other hand, $\iota\circ s(\eta)=\eta$ since for any $\rho\in\Rep(G)(S)$,
\begin{eqnarray*}
(\iota\circ s)(\eta)_{\rho} & = & \left(\mathrm{Id}\otimes1_{G,T}^{\natural}\right)\circ\left(\mathrm{Id}\otimes\eta_{\mathrm{reg}}\right)\circ(c_{\rho})_{T}\\
 & = & \left(\mathrm{Id}\otimes1_{G,T}^{\natural}\right)\circ(c_{\rho})_{T}\circ\eta_{\rho}\\
 & = & \rho(1_{G})_{T}^{\natural}\circ\eta_{\rho}=\eta_{\rho}.
\end{eqnarray*}
Thus $G=\Aut^{\otimes}(\omega)$ and by \ref{sub:injofVtoomega},
also $G=\Aut^{\otimes}(V)$.

\subsection{~\label{sub:PrG(G)=00003DG(omega)}}

We show that $\mathbb{G}^{\Gamma}(G)=\mathbb{G}^{\Gamma}(\omega)$.
Let $T$ be an $S$-scheme, $\mathcal{G}\in\mathbb{G}^{\Gamma}(\omega_{T})$.
Then for any $T$-scheme $X$, the $\Gamma$-graduation $\mathcal{G}$
on $\omega_{T}$ and the $\otimes$-equivalence 
\[
\Gra^{\Gamma}\QCoh(T)\simeq\Rep(\mathbb{D}_{T}(\Gamma))(T)
\]
together induce a factorization 
\[
\omega_{X}^{1}:\Rep(G)(S)\stackrel{\mathcal{G}'}{\longrightarrow}\Rep(\mathbb{D}_{T}(\Gamma))(T)\stackrel{\omega_{X}^{2}}{\longrightarrow}\QCoh(X)
\]
of the fiber functor $\omega_{X}^{1}$ for the group scheme $G$ over
$S$ through the fiber functor $\omega_{X}^{2}$ for the group scheme
$\mathbb{D}_{T}(\Gamma)$ over $T$. Moreover $\mathcal{G}'$ is a
$\otimes$-functor preserving trivial representations by (G1) and
(G2). It thus induces a group homomorphism
\[
\mathbb{D}_{T}(\Gamma)(X)\stackrel{\ref{sub:RecThG}}{\simeq}\Aut^{\otimes}(\omega_{X}^{2})\rightarrow\Aut^{\otimes}(\omega_{X}^{1})\stackrel{\ref{sub:RecThG}}{\simeq}G(X).
\]
The latter being functorial in $X$ gives a morphism $s(\mathcal{G}):\mathbb{D}_{T}(\Gamma)\rightarrow G_{T}$
of group schemes over $T$, i.e.~an element $s(\mathcal{G})$ of
$\mathbb{G}^{\Gamma}(G)(T)$. Since $\mathcal{G}\mapsto s(\mathcal{G})$
is itself functorial in $T$, it gives a morphism of fpqc sheaves
$s:\mathbb{G}^{\Gamma}(\omega)\rightarrow\mathbb{G}^{\Gamma}(G)$
which is the inverse of $\iota:\mathbb{G}^{\Gamma}(G)\rightarrow\mathbb{G}^{\Gamma}(\omega)$.
Thus $\mathbb{G}^{\Gamma}(G)=\mathbb{G}^{\Gamma}(\omega)$ and by
\ref{sub:injofVtoomega}, also $\mathbb{G}^{\Gamma}(G)=\mathbb{G}^{\Gamma}(V)$.

\section{Relating $\Rep(G)(S)$ and $\Rep^{\circ}(G)(S)$}

While $\Rep(G)(S)$ already contains the interesting regular representation,
it could be that $\Rep^{\circ}(G)(S)$ contains no representations
beyond the trivial ones, in which case $\Aut^{\otimes}(\omega^{\circ})$,
$\mathbb{G}^{\Gamma}(\omega^{\circ})$ and $\mathbb{F}^{\Gamma}(\omega^{\circ})$
are the trivial sheaves represented by $S$. For instance, let $S$
be one of the two curves considered in \cite[X 6.4]{SGA3.2}, whose
enlarged fundamental group equals $\mathbb{Z}$. Let $n\geq2$ and
$A\in GL_{n}(\mathbb{Z})$ be any matrix with no roots of unity as
eigenvalue. Then by \cite[X 7.1]{SGA3.2}, this determines an $n$-dimensional
torus $G$ over $S$, and all representations $\rho\in\Rep^{\circ}(G)(S)$
are trivial because $\mathbb{Z}^{n}$ contains no finite $A$-orbit
except $\{0\}$. 

When $S$ is quasi-compact, we also consider the intermediate full
subcategory\nomenclature[Repop(G)(S)]{$\Rep ' (G)(S)$}{Full sub-category of $\Rep (G)(S)$ defined on page \nomrefpage}
\[
\Rep^{\circ}(G)(S)\subset\Rep'(G)(S)\subset\Rep(G)(S)
\]
whose objects are the representations $\rho$ for which $\rho=\underrightarrow{\lim}\,\tau$
where $\tau$ runs through the partially ordered set $X(\rho)$ \nomenclature[X(rho)]{$X(\rho)$}{Filtered set of subrepresentations of $\rho$ on finite locally free subsheaves of $V(\rho)$, page \nomrefpage}of
all subrepresentations of $\rho$ which belong to $\Rep^{\circ}(G)(S)$.
For such $\rho$'s, $V(\rho)=\underrightarrow{\lim}\, V(\tau)$ is
a flat $\mathcal{O}_{S}$-module and the quasi-compactness of $S$
implies that $X(\rho)$ is a filtered set. This subcategory is stable
under tensor product and the $\rho\mapsto\rho_{0}$ construction,
it contains $\Rep^{\circ}(G)(S)$ as a full subcategory, and for any
$\rho_{1},\rho_{2}\in\Rep'(G)(S)$,
\begin{equation}
\Hom_{\Rep(G)}(\rho_{1},\rho_{2})=\underleftarrow{\lim}_{\tau_{1}\in X(\rho_{1})}\underrightarrow{\lim}_{\tau_{2}\in X(\rho_{2})}\Hom_{\Rep(G)}(\tau_{1},\tau_{2}).\label{eq:HomsAndLims}
\end{equation}
We denote by $\omega'_{T}:\Rep'(G)(S)\rightarrow\QCoh(T)$ the restriction
of $\omega_{T}$ to $\Rep'(G)(S)$ and define the fpqc sheaf $\Aut^{\otimes}(\omega'):(\Sch/S)^{\circ}\rightarrow\Group$
as before, with automorphisms of $\omega'_{T}$ satisfying the axioms
(A1) and (A2), thus obtaining a factorization 
\[
\Aut^{\otimes}(\omega)\rightarrow\Aut^{\otimes}(\omega')\rightarrow\Aut^{\otimes}(\omega^{\circ}).
\]
On the other hand, it is obvious that $\Aut^{\otimes}(\omega')=\Aut^{\otimes}(\omega^{\circ})$.

\subsection{~}

The following assumption implies that $\Rep^{\circ}(G)(S)$ is sufficiently
big:
\begin{lyxlist}{00.00.0000}
\item [{$\mathrm{HYP}(\omega^{\circ})$}] There exists a covering $\{S_{i}\rightarrow S\}$
by finite étale morphisms such that for every $i$, $G_{S_{i}}/S_{i}$
satisfies $\mathrm{HYP}'(\omega^{\circ})$ where:
\item [{$\mathrm{HYP}'(\omega^{\circ})$}] $S$ is quasi-compact and $\rho_{\mathrm{reg}}$
belongs to $\Rep'(G)(S)$.\end{lyxlist}
\begin{prop}
\label{Pro:Aut(w0)G(w0)F(w0)}If $G/S$ satisfies $\mathrm{HYP}(\omega^{\circ})$,
then 
\[
G=\Aut^{\otimes}(\omega^{\circ}),\quad\mathbb{G}^{\Gamma}(G)=\mathbb{G}^{\Gamma}(\omega^{\circ})\quad\mbox{and}\quad\mathbb{F}^{\Gamma}(\omega)\subset\mathbb{F}^{\Gamma}(\omega^{\circ}).
\]
\end{prop}
\begin{proof}
These being fpqc sheaves on $S$, it is sufficient to establish the
proposition for their restriction to the $S_{i}$'s, which by proposition~\ref{prop:SkalarExt}
reduces us to the case where $S$ is quasi-compact and $\rho_{\mathrm{reg}}$
belongs to $\Rep'(G)(S)$. The proof of theorem~\ref{thm:RecThm}
then shows that $G=\Aut^{\otimes}(\omega'_{T})$. Thus $G=\Aut^{\otimes}(\omega^{\circ})$.
To prove that $\mathbb{G}^{\Gamma}(G)=\mathbb{G}^{\Gamma}(\omega^{\circ})$,
we may test this on quasi-compact schemes, and then the proof of section~\ref{sub:PrG(G)=00003DG(omega)}
carries over to this case. Finally: a $\Gamma$-filtration $\mathcal{F}$
on $\omega_{T}$ is uniquely determined by its value on $\rho_{\mathrm{reg}}$
by~\ref{sub:injofVtoomega}, thus $\mathbb{F}^{\Gamma}(\omega)\subset\mathbb{F}^{\Gamma}(\omega^{\circ})$
since $\rho_{\mathrm{reg}}\in\Rep'(G)(S)$.
\end{proof}

\subsection{~}

For the $V^{\circ}$ variants of these, one needs a weaker assumption:
\begin{lyxlist}{00.00.0000}
\item [{$\mathrm{HYP}(V^{\circ})$}] Locally on $S$ for the fpqc topology,
$\rho_{\mathrm{reg}}$ belongs to $\Rep'(G)(S)$. \end{lyxlist}
\begin{prop}
\label{Pro:Aut(V0)G(V0)F(V0)}If $G/S$ satisfies $\mathrm{HYP}(V^{\circ})$,
then
\[
G=\Aut^{\otimes}(V^{\circ}),\quad\mathbb{G}^{\Gamma}(G)=\mathbb{G}^{\Gamma}(V^{\circ})\quad\mbox{and}\quad\mathbb{F}^{\Gamma}(V)\subset\mathbb{F}^{\Gamma}(V^{\circ}).
\]
\end{prop}
\begin{proof}
This being local in the fpqc topology on $S$, we may assume that
$S$ is quasi-compact and $\rho_{\mathrm{reg}}$ is in $\Rep'(G)(S)$,
then $G_{T}/T$ satisfies $\mathrm{HYP}'(\omega^{\circ})$ for every
quasi-compact $T$ over $S$ and the proposition easily follows from
the previous one.
\end{proof}

\subsection{~\label{sub:DefOfIsotriv}}

It remains to give some cases where our assumptions are met.
\begin{defn}
A reductive group $G$ over $S$ is called isotrivial if and only
if there exists a covering $\{S_{i}\rightarrow S\}$ by finite étale
morphisms such that each $G_{S_{i}}$ is splittable. 
\end{defn}
\noindent For tori, this definition is slightly more general than
that given in \cite[IX 1.1]{SGA3.2}, which requires a single finite
étale cover $S'\rightarrow S$. If $S$ is quasi-compact or connected,
both notions coincide. For arbitrary reductive groups, \cite[XXIV 4.1]{SGA3.3r}
only defines local and semi-local isotriviality. If $S$ is local,
these two notions coincide with ours.
\begin{prop}
\label{pro:LocalUnibImpliesIsotriv}If $S$ is local and either geometrically
unibranch or henselian, then every reductive group $G$ over $S$
is isotrivial.\end{prop}
\begin{proof}
We may assume that $G$ is a torus by \cite[XXIV 4.1.5]{SGA3.3r}.
We then have to show that the connected components of $R=\underline{\Hom}_{S}(G,\mathbb{G}_{m,S})$
are open and finite over $S$ by \cite[X 5.11]{SGA3.1r}, and this
follows from proposition~\ref{prop:RepGspCaseTorus} and lemma~\ref{lem:StructQuasIsoTwistedSch}.
The henselian case also follows directly from~\cite[X 4.6]{SGA3.2}
or \cite[XXIV 1.21]{SGA3.3r}.\end{proof}
\begin{prop}
\label{Pro:Wedhorn}$(1)$ If $S=\Spec(A)$ for a Prüfer domain $A$
and $\rho\in\Rep(G)(S)$,
\[
\rho\in\Rep'(G)(S)\iff V(\rho)\mbox{\,\ is a flat }\mathcal{O}_{S}\mbox{-module.}
\]

$(2)$ A split reductive group over a quasi-compact $S$ satisfies
$\mathrm{HYP}'(\omega^{\circ})$.

$(3)$ An isotrivial reductive group over a quasi-compact $S$ satisfies
$\mathrm{HYP}(\omega^{\circ})$.

$(4)$ A reductive group over any $S$ satisfies $\mathrm{HYP}(V^{\circ})$.\end{prop}
\begin{proof}
$(1)$ is exactly \cite[Corollary 5.10]{We04}. For $(2)$, we may
assume that $G$ is of constant type \cite[XXII 2.8]{SGA3.3r}, thus
isomorphic \cite[XXIII 5.2]{SGA3.3r} to the base change of a reductive
group $G_{0}$ over $\Spec(\mathbb{Z})$ \cite[XXV 1.2]{SGA3.3r}
to which $(1)$ now applies. Obviously $(2)\Rightarrow(3)$, and $(2)\Rightarrow(4)$
by \cite[XXII 2.3]{SGA3.3r}.
\end{proof}

\subsection{\label{sub:extF}}

Together with theorem$\,$\ref{thm:RecThm}, proposition~\ref{Pro:Aut(w0)G(w0)F(w0)}
and \ref{Pro:Aut(V0)G(V0)F(V0)} give many cases where automorphisms
or $\Gamma$-graduations automatically extend from $\omega^{\circ}$
or $V^{\circ}$ to $\omega$ or $V$. Assuming that $S$ is quasi-compact,
we will now do something similar for $\Gamma$-filtrations.

\subsection{~}

Let $\mathcal{F}$ be a $\Gamma$-filtration on $\omega_{T}^{\circ}$.
For each $\gamma\in\Gamma$, we may extend 
\[
\mathcal{F}^{\gamma}:\Rep^{\circ}(G)(S)\rightarrow\LF(T)\quad\mbox{to}\quad\mathcal{F}^{\gamma}:\Rep'(G)(S)\rightarrow\QCoh(T)
\]
by the formula $\mathcal{F}^{\gamma}(\rho)={\textstyle \underrightarrow{\lim}}\,\mathcal{F}^{\gamma}(\tau)$,
where $\tau$ runs through $X(\rho)$. It defines a functor by (\ref{eq:HomsAndLims}),
and gives back $\mathcal{F}^{\gamma}(\rho)=\mathcal{F}^{\gamma}(\tau)$
when $\rho=\tau$ belongs to $\Rep^{\circ}(G)(S)$. In general, $\mathcal{F}^{\gamma}(\rho)$
is a pure quasi-coherent subsheaf of $V(\rho)_{T}=\underrightarrow{\lim}\, V(\tau)_{T}$
since filtered colimits are exact and commute with base change. While
$\gamma\rightarrow\mathcal{F}^{\gamma}(\rho)$ is non-increasing,
it may not be a $\Gamma$-filtration on $V(\rho)_{T}$ in our sense.
However:
\begin{lem}
\label{lem:PropOfExtFil}We have the following properties:
\begin{lyxlist}{MMM}
\item [{\emph{(F1)}}] For every $\rho_{1},\rho_{2}\in\Rep'(G)(S)$ and
$\gamma\in\Gamma$, 
\[
\mathcal{F}^{\gamma}(\rho_{1}\otimes\rho_{2})={\textstyle \sum_{\gamma_{1}+\gamma_{2}=\gamma}}\mathcal{F}^{\gamma_{1}}(\rho_{1})\otimes\mathcal{F}^{\gamma_{2}}(\rho_{2}).
\]

\item [{\emph{(F2)}}] For a trivial representation $\rho\in\Rep'(G)(S)$
on $\mathcal{M}\in\QCoh(S)$,
\[
\mathcal{F}^{\gamma}(\rho)=\mathcal{M}\mbox{ if }\gamma\leq0\quad\mbox{and}\quad\mathcal{F}^{\gamma}(\rho)=0\mbox{ if }\gamma>0.
\]

\item [{\emph{(F3r)}}] If $\rho\twoheadrightarrow\tau$ is an epimorphism
with $\rho\in\Rep'(G)(S)$ and $\tau\in\Rep^{\circ}(G)(S)$, then
$\mathcal{F}^{\gamma}(\rho)\twoheadrightarrow\mathcal{F}^{\gamma}(\tau)$
is an epimorphism in $\QCoh(T)$ for every $\gamma\in\Gamma$. 
\item [{\emph{(F3l)}}] If $\rho_{\mathrm{reg}}$ belongs to $\Rep'(G)(S)$
and $\rho_{1}\hookrightarrow\rho_{2}$ is a pure monomorphism in $\Rep'(G)(S)$,
then $\mathcal{F}^{\gamma}(\rho_{1})=\mathcal{F}^{\gamma}(\rho_{2})\cap V_{T}(\rho_{1})$
in $V_{T}(\rho_{2})$ for every $\gamma\in\Gamma$. 
\end{lyxlist}
\end{lem}
\begin{proof}
(F2) is obvious and (F1), (F3r) follow from the eponymous properties
of $\mathcal{F}$ on $\omega_{T}^{\circ}$ because, since $S$ is
quasi-compact, $\{\tau_{1}\otimes\tau_{2}:(\tau_{1},\tau_{2})\in X(\rho_{1})\times X(\rho_{2})\}$
and $\{\tau'\in X(\rho):\tau'\twoheadrightarrow\tau\}$ are respectively
cofinal in $X(\rho_{1}\otimes\rho_{2})$ and $X(\rho)$. For (F3l),
we first treat the special case of the pure monomorphism $c_{\rho}:\rho\hookrightarrow\rho_{0}\otimes\rho_{\mathrm{reg}}$
for an arbitrary $\rho\in\Rep'(G)(S)$. Given (F1) and (F2), we have
to show that
\[
\mathcal{F}^{\gamma}(\rho)=\ker\left[\omega_{T}(\rho)\stackrel{\omega_{T}(c_{\rho})}{\longrightarrow}\omega_{T}(\rho_{0})\otimes\left(\omega_{T}(\rho_{\mathrm{reg}})/\mathcal{F}^{\gamma}(\rho_{\mathrm{reg}})\right)\right].
\]
Since both sides are filtered limits over $\tau\in X(\rho)$, we may
assume that $\rho$ belongs to $\Rep^{\circ}(G)(S)$. The right hand
side is then the filtered limit of 
\[
\ker\left[\omega_{T}(\rho)\stackrel{\omega_{T}(c_{\rho,\tau})}{\longrightarrow}\omega_{T}(\rho_{0})\otimes\left(\omega_{T}(\tau)/\mathcal{F}^{\gamma}(\tau)\right)\right]=\mathcal{F}^{\gamma}(\rho,\tau)
\]
where $\tau$ ranges through the cofinal set $X'$ of $X(\rho_{\mathrm{reg}})$
defined by 
\[
X'=\left\{ \tau:c_{\rho}\mbox{ factors as }\rho\stackrel{c_{\rho,\tau}}{\longrightarrow}\rho_{0}\otimes\tau\hookrightarrow\rho_{0}\otimes\rho_{\mathrm{reg}}\right\} .
\]
Note that $\rho_{0}\otimes\tau\hookrightarrow\rho_{0}\otimes\rho_{\mathrm{reg}}$
since $V(\rho_{0})$ is a flat $\mathcal{O}_{S}$-module. For each
$\tau$ in $X'$, the cokernel $\sigma_{\rho,\tau}$ of $c_{\rho,\tau}:\rho\hookrightarrow\rho_{0}\otimes\tau$
is an object of $\Rep^{\circ}(G)(S)$: the counit $1_{G}^{\natural}:\mathcal{A}(G)\rightarrow\mathcal{O}_{S}$
gives a retraction of $V(c_{\rho,\tau})$, whose kernel is a direct
factor of $V(\rho_{0}\otimes\tau)$ isomorphic to $V(\sigma_{\rho,\tau})$.
Since $\mathcal{F}^{\gamma}$ is exact on $\Rep^{\circ}(G)(S)$, it
follows that 
\[
\mathcal{F}^{\gamma}(\rho)=\ker\left[\omega_{T}(\rho)\stackrel{\omega_{T}(c_{\rho,\tau})}{\longrightarrow}\omega_{T}(\rho_{0}\otimes\tau)/\mathcal{F}^{\gamma}(\rho_{0}\otimes\tau)\right]=\mathcal{F}^{\gamma}(\rho,\tau)
\]
for every $\tau\in X'$, which proves our claim. For any morphism
$\rho_{1}\rightarrow\rho_{2}$ in $\Rep'(G)(S)$ and any $\gamma\in\Gamma$,
we now have a commutative diagram with exact rows
\[
\begin{array}{ccccccc}
0 & \rightarrow & \mathcal{F}^{\gamma}(\rho_{1}) & \rightarrow & \omega_{T}(\rho_{1}) & \rightarrow & \omega_{T}(\rho_{1,0})\otimes\omega_{T}(\rho_{\mathrm{reg}})/\mathcal{F}^{\gamma}(\rho_{\mathrm{reg}})\\
 &  & \downarrow &  & \downarrow &  & \downarrow\\
0 & \rightarrow & \mathcal{F}^{\gamma}(\rho_{2}) & \rightarrow & \omega_{T}(\rho_{2}) & \rightarrow & \omega_{T}(\rho_{2,0})\otimes\omega_{T}(\rho_{\mathrm{reg}})/\mathcal{F}^{\gamma}(\rho_{\mathrm{reg}})
\end{array}
\]
If $V(\rho_{1})\rightarrow V(\rho_{2})$ is a pure monomorphism, the
vertical maps are monomorphisms, therefore $\mathcal{F}^{\gamma}(\rho_{1})=\mathcal{F}^{\gamma}(\rho_{2})\cap\omega_{T}(\rho_{1})$
in $\omega_{T}(\rho_{2})$: this proves (F3l).
\end{proof}

\subsection{~}

As before, for every $\rho\in\Rep'(G)(S)$ and $\gamma\in\Gamma$,
we define 
\[
\mathcal{F}_{+}^{\gamma}(\rho)=\cup_{\eta>\gamma}\mathcal{F}^{\eta}(\rho)\quad\mbox{and}\quad\Gr_{\mathcal{F}}^{\gamma}(\rho)=\mathcal{F}^{\gamma}(\rho)/\mathcal{F}_{+}^{\gamma}(\rho).
\]
Since again filtered limits are exact, we find that 
\[
\mathcal{F}_{+}^{\gamma}(\rho)=\underrightarrow{\lim}\,\mathcal{F}_{+}^{\gamma}(\tau)\quad\mbox{and}\quad\Gr_{\mathcal{F}}^{\gamma}(\rho)=\underrightarrow{\lim}\,\Gr_{\mathcal{F}}^{\gamma}(\tau)
\]
where $\tau$ ranges through $X(\rho)$. In particular, the formula
\[
\Gr_{\mathcal{F}}^{\gamma}(\rho_{1}\otimes\rho_{2})\simeq\oplus_{\gamma_{1}+\gamma_{2}=\gamma}\Gr_{\mathcal{F}}^{\gamma_{1}}(\rho_{1})\otimes\Gr_{\mathcal{F}}^{\gamma_{2}}(\rho_{2})
\]
also holds for $\rho_{1}$ and $\rho_{2}$ in $\Rep'(G)(S)$. All
of the above constructions commute with arbitrary base change on $T$.
Finally if the original $\Gamma$-filtration $\mathcal{F}$ on $\omega_{T}^{\circ}$
already was the restriction of some $\Gamma$-filtration $\mathcal{F}'$
on $\omega_{T}$, the restriction of the latter is equal to the extension
of the former on $\omega'_{T}$ since $\mathcal{F}^{\prime\gamma}$
commutes with arbitrary colimits.

\subsection{~}

We first use the above device to show that:
\begin{prop}
\label{Pro:F(V0)2F(w0)inj}If $G/S$ satisfies $\mathrm{HYP}(\omega^{\circ})$,
then $\mathbb{F}^{\Gamma}(V^{\circ})\hookrightarrow\mathbb{F}^{\Gamma}(\omega^{\circ})$.\end{prop}
\begin{proof}
By proposition~\ref{prop:SkalarExt}, we may assume: $S$ is quasi-compact
and $\rho_{\mathrm{reg}}$ is in $\Rep'(G)(S)$. We have to show that
for an $S$-scheme $T$ and $\mathcal{F}\in\mathbb{F}^{\Gamma}(V_{T}^{\circ})$
with image $\tilde{\mathcal{F}}\in\mathbb{F}^{\Gamma}(\omega_{T}^{\circ})$,
for any $U\rightarrow T$, the $\Gamma$-filtration $\mathcal{F}_{U}$
on $\Rep^{\circ}(G_{U})(U)\rightarrow\LF(U)$ induced by $\mathcal{F}$
is determined by $\tilde{\mathcal{F}}$. We may assume that $T$ and
$U$ are quasi-compact. Then: $\mathcal{F}_{U}$ is determined by
its extension to $\Rep'(G_{U})(U)\rightarrow\QCoh(U)$, which itself
is determined by its value on $\rho_{\mathrm{reg},U}\in\Rep'(G_{U})(U)$
thanks to (F1-2) and (F3l) applied to the pure monomorphisms $c_{\rho}:\rho\rightarrow\rho_{0}\otimes\rho_{\mathrm{reg},U}$
for all $\rho$'s in $\Rep'(G_{U})(U)$. Since $U$ is quasi-compact,
$X(\rho_{\mathrm{reg}})_{U}=\left\{ \tau_{U}:\tau\in X(\rho_{\mathrm{reg}})\right\} $
is cofinal in $X(\rho_{\mathrm{reg},U})$, thus $\mathcal{F}{}_{U}(\rho_{\mathrm{reg,U}})$
is determined by the restriction of $\mathcal{F}_{U}$ to $X(\rho_{\mathrm{reg}})_{U}$.
By the axiom (F0) for $\mathcal{F}$, the latter is determined by
the values of $\mathcal{F}_{T}$ on $X(\rho_{\mathrm{reg}})_{T}$,
which are the values of $\tilde{\mathcal{F}}$ on $X(\rho_{\mathrm{reg}})$.
Thus $\tilde{\mathcal{F}}$ determines $\mathcal{F}_{U}$ and $\mathcal{F}$
uniquely.
\end{proof}

\subsection{~\label{sub:LinearGroups}}

Here is another useful assumption: we say that $G/S$ is linear if
there exists $\tau\in\Rep^{\circ}(G)(S)$ inducing a closed immersion
$\tau:G\hookrightarrow GL(V(\tau))$. Note that upon replacing $\tau$
with $\tau\oplus(\det\tau)^{-1}$, we may then also assume that $\det\tau=1$. 
\begin{lem}
\label{lem:AssLin}The affine and flat group $G$ over $S$ is linear
in the following cases:
\begin{enumerate}
\item $G$ is of finite type over a noetherian regular $S$ with $\dim S\leq2$.
\item $\mathrm{HYP}(\omega^{\circ})$ holds and $S$ is quasi-compact and
quasi-separated.
\item $G$ is an isotrivial reductive group over a quasi-compact $S$.
\item $G$ is a reductive group of adjoint type over any $S$.
\end{enumerate}
\end{lem}
\begin{proof}
$(1)$ is \cite[VIB 13.2]{SGA3.1r}. For $(2)$, let $f:S'\rightarrow S$
be a finite étale cover such that $\mathrm{HYP}'(\omega^{\circ})$
holds for $G'=G_{S'}$. Then $S'$ is also quasi-compact and quasi-separated,
thus by \cite[1.7.9]{EGA4.1}, the finitely generated quasi-coherent
$\mathcal{O}_{S'}$-algebra $\mathcal{A}(G')$ is generated by a finitely
generated quasi-coherent $\mathcal{O}_{S'}$-submodule $\mathcal{E}$.
By assumption $\mathrm{HYP}'(\omega^{\circ})$ for $G'$, we may replace
$\mathcal{E}$ by a larger $V(\tau')$ for some $\tau'$ in $X(f^{\ast}\rho_{\mathrm{reg}})$.
The proof of \cite[VIB 13.2]{SGA3.1r} then shows that $\tau':G'\rightarrow GL(V(\tau'))$
is a closed immersion. Put $\tau=f_{\ast}\tau'$, so that $\tau$
belongs to $\Rep^{\circ}(G)(S)$. Then $\tau:G\rightarrow GL(V(\tau))$
is a closed immersion. Indeed, it is sufficient to show that $f^{\ast}\tau:G'\rightarrow GL(V(f^{\ast}\tau))$
is a closed immersion by \cite[2.7.1]{EGA4.2}. But $f^{\ast}\tau=\rho\otimes\tau'$
in $\Rep^{\circ}(G')(S')$, where $\rho=f^{\ast}f_{\ast}1_{G'}$ is
the trivial representation on $V(\rho)=f^{\ast}f_{\ast}\mathcal{O}_{S'}$,
i.e.~$f^{\ast}\tau$ is the composition
\[
\xyC{3pc}\xymatrix{G'\ar[r]\sp(0.4){\rho'} & GL(V(\tau'))\ar[r]\sp(0.4){\mathrm{Id}\otimes-} & GL(V(\rho)\otimes V(\tau'))}
\]
of two closed immersions, therefore itself a closed immersion. For
$(3)$: it is well-known that the Chevalley groups over $\Spec\,\mathbb{Z}$
are linear (a complete overkill: use $(1)$), so are therefore also
the split reductive groups over any base by \cite[XXII 2.8, XXIII 5.2 and XXV 1.2]{SGA3.3r},
to which one reduces as in $(2)$. For $(4)$, simply take $\tau$
to be the adjoint representation $\rho_{\mathrm{ad}}$ of $G$ on
its Lie algebra $\Lie(G)=\mathfrak{g}=V(\rho_{\mathrm{ad}})$.
\end{proof}

\section{The stabilizer of a $\Gamma$-filtration, I}

\subsection{~}

Let now $G$ be a reductive group over $S$ and let $\rho_{\mathrm{ad}}\in\Rep^{\circ}(G)(S)$
be the adjoint representation of $G$ on $V(\rho_{\mathrm{ad}})=\mathfrak{g}=\Lie(G)$.
\nomenclature[rho_ad]{$\rho _\mathrm{ad}$}{Adjoint representation of $G$ on $\Lie (G)$.}Let
$T$ be an $S$-scheme.
\begin{thm}
\label{thm:StabIsParV}Let $\mathcal{F}$ be a $\Gamma$-filtration
on $V_{T}$. Then $\Aut^{\otimes}(\mathcal{F})$ is a parabolic subgroup
$P_{\mathcal{F}}$ of $G_{T}$ with unipotent radical $U_{\mathcal{F}}\subset\Aut^{\otimes!}(\mathcal{F})$.
Moreover, 
\[
\Lie(U_{\mathcal{F}})=\mathcal{F}_{+}^{0}(\rho_{\mathrm{ad}})\quad\mbox{and}\quad\Lie(P_{\mathcal{F}})=\mathcal{F}^{0}(\rho_{\mathrm{ad}})\quad\mbox{in}\quad V_{T}(\rho_{\mathrm{ad}})=\mathfrak{g}_{T}.
\]
\end{thm}
\begin{rem}
Let $\chi:\mathbb{D}_{T}(\Gamma)\rightarrow G_{T}$ be a morphism,
$\mathcal{G}$ the corresponding $\Gamma$-graduation and $\mathcal{F}$
the induced $\Gamma$-filtration. Let $P_{\chi}=U_{\chi}\rtimes L_{\chi}$
be the subgroups of $G_{T}$ defined in proposition~\ref{prop:DefPGroupCase}.
Since $\Aut^{\otimes}(\mathcal{F})=\Aut^{\otimes!}(\mathcal{F})\rtimes\Aut^{\otimes}(\mathcal{G})$
with $\Aut^{\otimes}(\mathcal{G})$ equal to $L_{\chi}$ and isomorphic
to $\Aut^{\otimes}(\Gr_{\mathcal{F}}^{\bullet})$ (because $\mathcal{G}\simeq\Gr_{\mathcal{F}}^{\bullet})$,
the theorem implies 
\[
P_{\chi}=\Aut^{\otimes}(\mathcal{F}),\quad U_{\chi}=\Aut^{\otimes!}(\mathcal{F})\quad\mbox{and}\quad P_{\chi}/U_{\mathcal{\chi}}\simeq\Aut^{\otimes}(\Gr_{\mathcal{F}}^{\bullet}).
\]
\end{rem}
\begin{cor}
\label{cor:F(G)same}The quotients $\Fi:\mathbb{G}^{\Gamma}(G)\twoheadrightarrow\mathbb{F}^{\Gamma}(G)$
of $\mathbb{G}^{\Gamma}(G)$ defined in sections~\ref{sub:FiltrationsGroup}
and \ref{sub:defiota} are canonically isomorphic, and for any $\mathcal{F}\in\mathbb{F}^{\Gamma}(G)(T)$,
\[
P_{\mathcal{F}}=\Aut^{\otimes}(\iota\mathcal{F}),\quad U_{\mathcal{F}}=\Aut^{\otimes!}(\iota\mathcal{F})\quad\mbox{and}\quad P_{\mathcal{F}}/U_{\mathcal{F}}\simeq\Aut^{\otimes}(\Gr_{\iota\mathcal{F}}^{\bullet})
\]
where $\iota\mathcal{F}$ is the image of $\mathcal{F}$ in $\mathbb{F}^{\Gamma}(V_{T})$. \end{cor}
\begin{proof}
For the first assertion, we have to show that for $\chi_{1},\chi_{2}:\mathbb{D}_{T}(\Gamma)\rightarrow G_{T}$,
\[
\chi_{1}\sim_{\mathrm{Par}}\chi_{2}\iff\Fi\circ\iota(\chi_{1})=\Fi\circ\iota(\chi_{2})\mbox{ in }\mathbb{F}^{\Gamma}(V_{T}).
\]
Put $\mathcal{G}_{i}=\iota(\chi_{i})$, $\mathcal{F}_{i}=\Fi(\mathcal{G}_{i})$
and $P_{i}=\Aut^{\otimes}(\mathcal{F}_{i})=P_{\chi_{i}}$. If $\mathcal{\chi}_{1}\sim_{\mathrm{Par}}\chi_{2}$,
then $\chi_{2}=\Int(p)\circ\chi_{1}$ for some $p\in P_{1}(T)$, thus
$\mathcal{F}_{2}=p\mathcal{F}_{1}=\mathcal{F}_{1}$. If $\mathcal{F}_{1}=\mathcal{F}_{2}=\mathcal{F}$,
then $P_{1}=P_{2}=P$ and the canonical isomorphism $\mathcal{G}_{1}\simeq\Gr_{\mathcal{F}}^{\bullet}\simeq\mathcal{G}_{2}$
gives an element of $\Aut^{\otimes}(V_{T})$ preserving $\mathcal{F}$
and mapping $\mathcal{G}_{1}$ to $\mathcal{G}_{2}$, i.e.~an element
$p\in P(T)$ such that $\chi_{2}=\Int(p)\circ\chi_{1}$, thus $\chi_{1}\sim_{\mathrm{Par}}\chi_{2}$.
The remaining assertions are local in the fpqc topology on $T$ and
thus follow from the above remark. 
\end{proof}

\subsection{~}

For $\Gamma$-filtrations on $\omega_{T}$, we need a technical assumption
on $G/S$: 
\begin{lyxlist}{MMM}
\item [{TA}] There exists an fpqc cover $\{f_{i}:S_{i}\rightarrow S\}$
such that (a) each $f_{i}$ is an affine morphism, and (b) each $G_{i}=G_{S_{i}}$
is linear (\ref{sub:LinearGroups}).
\end{lyxlist}
This is true for \emph{any} reductive group $G$ over a \emph{separated}
$S$: starting from a Zariski covering of $S$ by affine $U_{i}$'s,
we pick fpqc covers $\{U_{i,j}\rightarrow U_{i}\}$ splitting $G_{U_{i}}$,
and again cover the $U_{i,j}$'s by affine $U_{i,j,k}$'s. The resulting
fpqc cover $\{U_{i,j,k}\rightarrow S\}$ satisfies our assumption:
$U_{i,j,k}\rightarrow U_{i}$ is affine as a morphism between affine
schemes, $U_{i}\hookrightarrow S$ is affine because $S$ is separated,
and $G_{U_{i,j,k}}$ is linear by lemma~\ref{lem:AssLin} since it
is split.
\begin{thm}
\label{thm:StabIsParw}Assume $TA$. Let $\mathcal{F}$ be a $\Gamma$-filtration
on $\omega_{T}$. Then $\Aut^{\otimes}(\mathcal{F})$ is a parabolic
subgroup $P_{\mathcal{F}}$ of $G_{T}$ with unipotent radical $U_{\mathcal{F}}\subset\Aut^{\otimes!}(\mathcal{F})$.
Moreover, 
\[
\Lie(U_{\mathcal{F}})=\mathcal{F}_{+}^{0}(\rho_{\mathrm{ad}})\quad\mbox{and}\quad\Lie(P_{\mathcal{F}})=\mathcal{F}^{0}(\rho_{\mathrm{ad}})\quad\mbox{in}\quad V_{T}(\rho_{\mathrm{ad}})=\mathfrak{g}_{T}.
\]

\end{thm}

\subsection{~}

If $\mathcal{F}'$ is a $\Gamma$-filtration on $V_{T}$ and $\mathcal{F}$
is the induced $\Gamma$-filtration on $\omega_{T}$, then $\Aut^{\otimes}(\mathcal{F}')=\Aut^{\otimes}(\mathcal{F})$
as subsheaves of $G_{T}$ by \ref{sub:injofVtoomega} and theorem~\ref{thm:RecThm}.
Therefore: (a) theorem~\ref{thm:StabIsParw} holds without the technical
assumption for such filtrations on $\omega_{T}$, and (b) theorem~\ref{thm:StabIsParV},
which is local on $S$, follows from theorem~\ref{thm:StabIsParw}
applied to any affine cover of $S$. We thus only have to consider
the case of a $\Gamma$-filtration $\mathcal{F}$ on $\omega_{T}$.
The technical assumption will be used only once below, in section~\ref{sub:finitegen}.

\subsection{\label{sub:adjreg}~}

The adjoint-regular representation $\rho_{\mathrm{adj}}$ of $G$
on $V(\rho_{\mathrm{adj}})=\mathcal{A}(G)$ is given by\nomenclature[rho_adj]{$\rho _{\mathrm{adj}}$}{Adjoint representation of $G$ on $\mathcal{A} (G)$, page \nomrefpage}
\[
(g\cdot a)(h)=a(g^{-1}hg)
\]
for $T\rightarrow S$, $a\in\Gamma(T,\mathcal{A}(G_{T}))$ and $g,h\in G(T)$.
The unit, counit $1_{G}^{\natural}$, multiplication, comultiplication
$\mu^{\natural}$ and inversion $\mathrm{inv}^{\natural}$ of $\mathcal{A}(G)$
define morphisms in $\Rep(G)(S)$: 
\[
1_{S}\rightarrow\rho_{\mathrm{adj}},\quad\rho_{\mathrm{adj}}\rightarrow1_{S},\quad\rho_{\mathrm{adj}}\otimes\rho_{\mathrm{adj}}\rightarrow\rho_{\mathrm{adj}},\quad\rho_{\mathrm{adj}}\rightarrow\rho_{\mathrm{adj}}\otimes\rho_{\mathrm{adj}},\quad\rho_{\mathrm{adj}}\rightarrow\rho_{\mathrm{adj}}.
\]
For any $\rho$ in $\Rep(G)(S)$, we may also view $c_{\rho}$ as
a split monomorphism
\[
c_{\rho}:\rho\rightarrow\rho\otimes\rho_{\mathrm{adj}}\quad\mbox{in}\quad\Rep(G)(S).
\]
If $\tau$ belongs to $\Rep^{\circ}(G)(S)$, $c_{\tau}$ gives a morphism
$\tau^{\vee}\otimes\tau\rightarrow\rho_{\mathrm{adj}}$ which induces
a $G$-equivariant morphism of quasi-coherent $G-\mathcal{O}_{S}$-algebras
\[
\Sym^{\bullet}(\tau^{\vee}\otimes\tau)\rightarrow\rho_{\mathrm{adj}}
\]
whose underlying morphism of quasi-coherent $\mathcal{O}_{S}$-algebras
is given by 
\[
\Sym_{\mathcal{O}_{S}}^{\bullet}(V(\tau)^{\vee}\otimes V(\tau))\hookrightarrow\Sym_{\mathcal{O}_{S}}^{\bullet}\left(\underline{\End}_{\mathcal{O}_{S}}(\tau)\right)\left[{\textstyle \frac{1}{\det}}\right]=\mathcal{A}\left(GL(V(\tau))\right)\stackrel{\tau^{\natural}}{\longrightarrow}\mathcal{A}(G)
\]
where $\tau^{\natural}$ is the morphism attached to $\tau:G\rightarrow GL(V(\rho))$.
In particular, if the latter is a closed embedding and $\det(\tau)=1$,
then $\Sym^{\bullet}(\tau^{\vee}\otimes\tau)\twoheadrightarrow\rho_{\mathrm{adj}}$
is an epimorphism.

\subsection{~\label{sub:adjregquo}}

Let $\rho_{\mathrm{adj}}^{\circ}$ be the kernel\nomenclature[rho_adj^o]{$\rho ^\circ  _\mathrm{adj}$}{Adjoint representation of $G$ on $\mathcal{I} (G)$, page \nomrefpage}
of $1_{G}^{\natural}:\rho_{\mathrm{adj}}\rightarrow1_{S}$\nomenclature[1_S]{$1_S$}{Trivial representation of $G$ on $\mathcal{O}_S$.}.
Thus $\rho_{\mathrm{adj}}=\rho_{\mathrm{adj}}^{\circ}\oplus1_{S}$
and $V(\rho_{\mathrm{adj}}^{\circ})$ is the augmentation ideal $\mathcal{I}(G)$
of $\mathcal{A}(G)$\nomenclature[I(G)]{$\mathcal{I}(G)$}{Augmentation ideal of $\mathcal{A}(G)$.}.
For any $n\geq1$, the multiplication map $\mathcal{I}(G)^{\otimes n+1}\rightarrow\mathcal{I}(G)$
defines a morphism $(\rho_{\mathrm{adj}}^{\circ})^{\otimes n+1}\rightarrow\rho_{\mathrm{adj}}^{\circ}$
in $\Rep(G)(S)$. We denote by $\rho^{n}\in\Rep^{\circ}(G)(S)$ its
cokernel\nomenclature[rho^n]{$\rho ^n$}{Adjoint representation of $G$ on $\mathcal{I} (G)/\mathcal{I} (G)^{n+1}$, page \nomrefpage},
a representation of $G$ on $V(\rho^{n})=\mathcal{I}(G)/\mathcal{I}(G)^{n+1}$,
and by $\rho_{n}=(\rho^{n})^{\vee}\in\Rep^{\circ}(G)(S)$ the dual
of $\rho^{n}$\nomenclature[rho_n]{$\rho _n$}{Dual of $\rho ^n$, page \nomrefpage}.
Thus $\rho_{1}=\rho_{\mathrm{ad}}$, the adjoint representation of
$G$ on $V(\rho_{\mathrm{ad}})=\mathfrak{g}$.

\subsection{~\label{sub:defUFPF}}

Let now $\mathcal{I}(\mathcal{F})$ and $\mathcal{J}(\mathcal{F})$
be the quasi-coherent ideals of $\mathcal{A}(G_{T})$ which are respectively
generated by the quasi-coherent subsheaves $\mathcal{F}_{+}^{0}(\rho_{\mathrm{adj}}^{\circ})$
and $\mathcal{F}^{0}(\rho_{\mathrm{adj}}^{\circ})$ of the augmentation
ideal $\mathcal{I}(G_{T})=\omega{}_{T}(\rho_{\mathrm{adj}}^{\circ})$
of $\mathcal{A}(G_{T})$. Then 
\[
U_{\mathcal{F}}\stackrel{\mathrm{def}}{=}\Spec\left(\mathcal{A}(G_{T})/\mathcal{J}(\mathcal{F})\right)\hookrightarrow P_{\mathcal{F}}\stackrel{\mathrm{def}}{=}\Spec\left(\mathcal{A}(G_{T})/\mathcal{I}(\mathcal{F})\right)
\]
are closed subgroup schemes of $G_{T}$, because $\mathcal{J}(\mathcal{F})$
and $\mathcal{I}(\mathcal{F})$ are compatible with the comultiplication
$\mu_{T}^{\natural}$ and inversion $\mathrm{inv}_{T}^{\natural}$
of $\mathcal{A}(G_{T})$, since $\mu^{\natural}:\rho_{\mathrm{adj}}\rightarrow\rho_{\mathrm{adj}}\otimes\rho_{\mathrm{adj}}$
and $\mathrm{inv}^{\natural}:\rho_{\mathrm{adj}}\rightarrow\rho_{\mathrm{adj}}$
are morphisms in $\Rep(G)(S)$. It follows from their definition that
the formation of $U_{\mathcal{F}}$ and $P_{\mathcal{F}}$ commutes
with arbitrary base change on $T$.

\subsection{~\label{sub:PropPFUF}}

Let $N(U_{\mathcal{F}})$ and $N(P_{\mathcal{F}})$ \nomenclature[N_G(x)]{$N_G(x)$}{Normalizer of $x$ in $G$.}be
the normalizers of $U_{\mathcal{F}}$ and $P_{\mathcal{F}}$ in $G_{T}$.
Then 
\[
P_{\mathcal{F}}\subset\Aut^{\otimes}(\mathcal{F})\subset N(U_{\mathcal{F}}),N(P_{\mathcal{F}})\quad\mbox{and}\quad U_{\mathcal{F}}\subset\Aut^{\otimes!}(\mathcal{F})
\]
as fpqc subsheaves of $G_{T}$. We have to check this on sections
over an arbitrary $T$-scheme $X$, but we may assume that $X=T$.
Since $G=\Aut^{\otimes}(\omega)$ by theorem~\ref{thm:RecThm}, 
\begin{eqnarray*}
\Aut^{\otimes}(\mathcal{F})(T) & = & \left\{ g\in G(T)\vert\,\forall\rho,\gamma:\rho(g)\left(\mathcal{F}^{\gamma}(\rho)\right)=\mathcal{F}^{\gamma}(\rho)\right\} .
\end{eqnarray*}
On the other hand, for any $\rho$ in $\Rep(G)(S)$, the morphism
$c_{\rho}:\rho\rightarrow\rho\otimes\rho_{\mathrm{adj}}$ gives a
morphism $\omega_{T}(c_{\rho}):\omega_{T}(\rho)\rightarrow\omega_{T}(\rho)\otimes\omega_{T}(\rho_{\mathrm{adj}})$
in $\QCoh(T)$ mapping $\mathcal{F}^{\gamma}(\rho)$ into
\[
\mathcal{F}^{\gamma}(\rho\otimes\rho_{\mathrm{adj}})={\textstyle \sum_{\alpha+\beta=\gamma}}\mathcal{F}^{\alpha}(\rho)\otimes\mathcal{F}^{\beta}(\rho_{\mathrm{adj}}).
\]

(a) For $g$ in $\Aut^{\otimes}(\mathcal{F})(T)$, $\rho_{\mathrm{adj}}^{\circ}(g)$
fixes $\mathcal{F}_{+}^{0}(\rho_{\mathrm{adj}}^{\circ})=\cup_{\gamma>0}\mathcal{F}^{\gamma}(\rho_{\mathrm{adj}}^{\circ})$
as well as the $\mathcal{A}(G_{T})$-ideal $\mathcal{I}(\mathcal{F})$
which it spans. It follows that the inner automorphism of $G_{T}$
defined by $g$ fixes $P_{\mathcal{F}}$. Thus $g$ belongs to $N(P_{\mathcal{F}})(T)$.
Similarly, $g\in N(U_{\mathcal{F}})(T)$.

(b) For $g$ in $P_{\mathcal{F}}(T)$, $g^{\natural}:\mathcal{A}(G_{T})\rightarrow\mathcal{O}_{T}$
is trivial on $\mathcal{F}^{\beta}(\rho_{\mathrm{adj}})$ for every
$\beta>0$ and thus $\rho(g)=(\mathrm{Id}\otimes g^{\natural})\circ\omega_{T}(c_{\rho})$
maps $\mathcal{F}^{\gamma}(\rho)$ into $\sum_{\alpha\geq\gamma}\mathcal{F}^{\alpha}(\rho)=\mathcal{F}^{\gamma}(\rho)$.
Since $g^{-1}$ also belongs to $P_{\mathcal{F}}(T)$, $\rho(g)$
fixes $\mathcal{F}^{\gamma}(\rho)$. Thus $g$ belongs to $\Aut^{\otimes}(\mathcal{F})(T)$. 

(c) For $g$ in $U_{\mathcal{F}}(T)$, $g^{\natural}-1^{\natural}:\mathcal{A}(G_{T})\rightarrow\mathcal{O}_{T}$
is trivial on $\mathcal{F}^{0}(\rho_{\mathrm{adj}})=\mathcal{O}_{T}\oplus\mathcal{F}^{0}(\rho_{\mathrm{adj}}^{\circ})$,
thus $\rho(g)-\rho(1)=\left(\mathrm{Id}\otimes\left(g^{\natural}-1^{\natural}\right)\right)\circ\omega_{T}(c_{\rho})$
maps $\mathcal{F}^{\gamma}(\rho)$ into $\sum_{\alpha>\gamma}\mathcal{F}^{\alpha}(\rho)=\mathcal{F}_{+}^{\gamma}(\rho)$.
Therefore $g$ belongs to $\Aut^{\otimes!}(\mathcal{F})(T)$.

\subsection{~}

We will establish below that the neutral components \cite[VIB 3.1]{SGA3.1r}
$U_{\mathcal{F}}^{\circ}$ and $P_{\mathcal{F}}^{\circ}$ of $U_{\mathcal{F}}$
and $P_{\mathcal{F}}$ are smooth over $S$, using the following criterion:\nomenclature[H^o]{$H^\circ$}{Neutral component of a group scheme $H$.}
\begin{prop}
Let $G$ be affine smooth over $S$, $\mathcal{A}=\mathcal{A}(G)$
and $\mathcal{I}=\mathcal{I}(G)$. Let $H\subset G$ be a closed subgroup
defined by a quasi-coherent ideal $\mathcal{J}$ of $\mathcal{A}$
such that
\begin{enumerate}
\item $\mathcal{J}$ is finitely generated,
\item $\mathcal{J}\cap\mathcal{I}^{2}=\mathcal{I}\cdot\mathcal{J}$ in $\mathcal{A}$,
and
\item $\mathcal{I}/\mathcal{J}+\mathcal{I}^{2}$ is finite locally free
on $S$.
\end{enumerate}
Then $H^{\circ}$ is representable by a smooth open subgroup scheme
of $H$. \end{prop}
\begin{proof}
By \cite[VIB 3.10]{SGA3.1r}, we have to show that $H$ is smooth
at all points of its unit section. Let thus $x\in H$ be the image
of $s\in S$ under $1_{H}:S\rightarrow H$\nomenclature[1_G]{$1_G$}{Unit section $1_G:S\rightarrow G$ of a group scheme $G$ over $S$.}.
By~\cite[1.4.3 and 1.4.5]{EGA4.1}, we already know from $(1)$ that
$H$ is locally of finite presentation over $S$. Thus by \cite[17.5.1]{EGA4.4}
and the Jacobian criterion \cite[0$_{IV}$ 22.6.4]{EGA4.1}, we have
to show that $\mathcal{J}_{x}/\mathcal{J}_{x}^{2}\otimes_{\mathcal{O}_{H,x}}k\rightarrow\Omega_{\mathcal{O}_{G,x}/\mathcal{O}_{S,s}}^{1}\otimes_{\mathcal{O}_{G,x}}k$
is injective, where $k$ is the common residue field of $s$ and $x,$
and the morphism is induced by the universal derivation $d:\mathcal{O}_{G,x}\rightarrow\Omega_{\mathcal{O}_{G,x}/\mathcal{O}_{S,s}}^{1}$.
This map factors through the corresponding map for $\mathcal{I}_{x}$,
namely $\mathcal{I}_{x}/\mathcal{I}_{x}^{2}\otimes_{\mathcal{O}_{S,s}}k\rightarrow\Omega_{\mathcal{O}_{G,x}/\mathcal{O}_{S,s}}^{1}\otimes_{\mathcal{O}_{G,x}}k$,
which is injective (because $\mathcal{O}_{G,x}/\mathcal{I}_{x}=\mathcal{O}_{S,s}$
is formally smooth over itself!). We thus have to show that 
\[
\mathcal{J}_{x}/\mathcal{J}_{x}^{2}\otimes_{\mathcal{O}_{H,x}}k=\mathcal{J}_{x}/m_{x}\mathcal{J}_{x}\rightarrow\mathcal{I}_{x}/m_{x}\mathcal{I}_{x}=\mathcal{I}_{x}/\mathcal{I}_{x}^{2}\otimes_{\mathcal{O}_{S,s}}k
\]
is injective, where $m_{x}$ is the maximal ideal of $\mathcal{O}_{G,x}$.
The latter map is base-changed from the morphism $\mathcal{J}_{x}/\mathcal{J}_{x}\mathcal{I}_{x}\rightarrow\mathcal{I}_{x}/\mathcal{I}_{x}^{2}$,
which itself is the localization at $x$ of the morphism $\mathcal{J}/\mathcal{J}\mathcal{I}\rightarrow\mathcal{I}/\mathcal{I}^{2}$,
which is a pure monomorphism by assumption. 
\end{proof}

\subsection{~\label{sub:finitegen}}

We show that $\mathcal{I}(\mathcal{F})$ and $\mathcal{J}(\mathcal{F})$
are finitely generated, focusing on $\mathcal{I}(\mathcal{F})$ to
simplify the exposition. Let $\{S_{i}\rightarrow S\}$ be an fpqc
cover as in assumption (TA), $\{f_{i}:T_{i}\rightarrow T\}$ the corresponding
fpqc cover of $T$, $\omega_{i}$ the fiber functor for $G_{i}=G_{S_{i}}$
and $\mathcal{F}_{i}$ the extension of $\mathcal{F}_{T_{i}}$ to
a $\Gamma$-filtration on $\omega_{i,T_{i}}$ -- which exists by~proposition~\ref{prop:SkalarExt}
since $f_{i}$ is affine. By \cite[2.5.2]{EGA4.2}, it is sufficient
to show that $f_{i}^{\ast}\mathcal{I}(\mathcal{F})$ is finitely generated.
Since $f_{i}$ is flat, $f_{i}^{\ast}\mathcal{I}(\mathcal{F})=\mathcal{I}(\mathcal{F}_{T_{i}})$
and obviously $\mathcal{I}(\mathcal{F}_{T_{i}})=\mathcal{I}(\mathcal{F}_{i})$.
We may thus assume that $G$ is linear over $S$: there exists $\tau\in\Rep^{\circ}(G)(S)$
inducing a closed embedding $\tau:G\hookrightarrow GL(V(\tau))$ with
$\det\tau\equiv1$, thus also an epimorphism $S^{\bullet}(\tau)=\Sym^{\bullet}(\tau^{\vee}\otimes\tau)\twoheadrightarrow\rho_{\mathrm{adj}}$
of quasi-coherent $G$-$\mathcal{O}_{S}$-algebras. By the axiom (F3)
for $\mathcal{F}$, $\mathcal{I}(\mathcal{F})$ is the image of the
ideal $\mathcal{I}(\tau)$ spanned by $\mathcal{F}_{+}^{0}(S^{\bullet}(\tau))$
in $V(S^{\bullet}(\tau))_{T}$. Using~proposition~\ref{prop:FilonLF},
we may assume that there is a splitting $V(\tau^{\vee}\otimes\tau)_{T}=\oplus_{\gamma}\mathcal{G}_{\gamma}$
of $\mathcal{F}$ on $\tau^{\vee}\otimes\tau$. By the axioms (F1)
and (F3), it induces a splitting of $\mathcal{F}$ on $S^{\bullet}(\tau)$,
\[
V(S^{n}(\tau))_{T}=\oplus_{\gamma}\oplus_{\gamma_{1}+\cdots+\gamma_{n}=\gamma}\mathcal{G}_{\gamma_{1}}\cdots\mathcal{G}_{\gamma_{n}}.
\]
It follows easily that $\mathcal{I}(\tau)$ is spanned by the finite
locally free subsheaf $\oplus_{\gamma>0}\mathcal{G}_{\gamma}$ of
$V(S^{1}(\tau))_{T}=V(\tau^{\vee}\otimes\tau)_{T}$, therefore $\mathcal{I}(\tau)$
and $\mathcal{I}(\mathcal{F})$ are indeed finitely generated.

\subsection{~\label{sub:intersecprop}}

We show that $\mathcal{I}(\mathcal{F})\cap\mathcal{I}(G_{T})^{2}=\mathcal{I}(\mathcal{F})\cdot\mathcal{I}(G_{T})$
-- the proof for $\mathcal{J}(\mathcal{F})$ is similar. Plainly,
$\mathcal{I}(\mathcal{F})\cdot\mathcal{I}(G_{T})\subset\mathcal{I}(\mathcal{F})\cap\mathcal{I}(G_{T})^{2}$.
For the other inclusion, we may assume that $T$ is affine and work
with global sections. Let thus $s$ be a (global) section of $\mathcal{I}(\mathcal{F})$,
so that $s=a+b$ with $a$ a section of $\mathcal{F}_{+}^{0}(\rho_{\mathrm{adj}}^{\circ})$
and $b$ a section of 
\[
\mathcal{I}(G_{T})\cdot\mathcal{F}_{+}^{0}(\rho_{\mathrm{adj}}^{\circ})\subset\mathcal{I}(G_{T})\cdot\mathcal{I}(\mathcal{F})\subset\mathcal{I}(G_{T})^{2}.
\]
Then $s$ belongs to $\mathcal{I}(G_{T})^{2}$ if and only $a$ does,
i.e.~$a$ is a section of $\mathcal{F}_{+}^{0}(\rho_{\mathrm{adj}}^{\circ})\cap\mathcal{I}(G_{T})^{2}$.
The pure short exact sequence and epimorphism of quasi-coherent sheaves
on $S$ 
\[
0\rightarrow\mathcal{I}(G)^{2}\rightarrow\mathcal{I}(G)\rightarrow\mathcal{I}(G)/\mathcal{I}(G)^{2}\rightarrow0\quad\mbox{and}\quad\mathcal{I}(G)^{\otimes2}\twoheadrightarrow\mathcal{I}(G)^{2}
\]
correspond to a pure short exact sequence and epimorphism in $\Rep(G)(S)$,
\[
0\rightarrow\rho_{\mathrm{adj}}^{\circ(2)}\rightarrow\rho_{\mathrm{adj}}^{\circ}\rightarrow\rho^{1}\rightarrow0\quad\mbox{and}\quad(\rho_{\mathrm{adj}}^{\circ})^{\otimes2}\twoheadrightarrow\rho_{\mathrm{adj}}^{\circ(2)}
\]
which together give, using the axioms (F1) and (F3) for $\mathcal{F}$,
\[
\mathcal{F}_{+}^{0}(\rho_{\mathrm{adj}}^{\circ})\cap\mathcal{I}(G_{T})^{2}=\mathcal{F}_{+}^{0}(\rho_{\mathrm{adj}}^{\circ(2)})={\textstyle \sum_{\gamma_{1}+\gamma_{2}>0}}\mathcal{F}^{\gamma_{1}}(\rho_{\mathrm{adj}}^{\circ})\cdot\mathcal{F}^{\gamma_{2}}(\rho_{\mathrm{adj}}^{\circ})
\]
which is contained in $\mathcal{I}(\mathcal{F})\cdot\mathcal{I}(G_{T})$,
thus $\mathcal{I}(\mathcal{F})\cap\mathcal{I}(G_{T})^{2}\subset\mathcal{I}(\mathcal{F})\cdot\mathcal{I}(G_{T})$.

\subsection{~\label{sub:finlocfree}}

We show that $\mathcal{I}(G_{T})/\mathcal{I}(\mathcal{F})+\mathcal{I}(G_{T})^{2}$
is finite locally free -- the proof for $\mathcal{J}(\mathcal{F})$
is similar. By the axiom (F3), $\mathcal{I}(\mathcal{F})+\mathcal{I}(G_{T})^{2}/\mathcal{I}(G_{T})^{2}$
is the $\mathcal{A}(G_{T})$-submodule of $\mathcal{I}(G_{T})/\mathcal{I}(G_{T})^{2}=\omega_{T}(\rho^{1})$
generated by $\mathcal{F}_{+}^{0}(\rho^{1})$, i.e.~this $\mathcal{O}_{T}$-submodule
itself since $\mathcal{A}(G_{T})$ acts on $\mathcal{I}(G_{T})/\mathcal{I}(G_{T})^{2}$
through $\mathcal{O}_{T}$. We are thus claiming that $\omega_{T}(\rho^{1})/\mathcal{F}_{+}^{0}(\rho^{1})$
is finite locally free, which follows from proposition~\ref{prop:FilonLF}.

\subsection{~\label{sub:LieUFPF}}

We have just established that $U_{\mathcal{F}}^{\circ}$ and $P_{\mathcal{F}}^{\circ}$
are representable by smooth open subschemes of $U_{\mathcal{F}}$
and $P_{\mathcal{F}}$. They are also finitely presented over $T$:
they are separated over $T$ as compositions of affine morphisms and
open immersions, and they are quasi-compact over $T$ by \cite[VIB 3.9]{SGA3.1r},
since $U_{\mathcal{F}}$ and $P_{\mathcal{F}}$ are finitely presented
over $S$, being locally of finite presentation by~\ref{sub:finitegen}
and \cite[1.4.5]{EGA4.1}, and affine by definition. From~\ref{sub:PropPFUF},
we obtain the following chain of inclusions
\[
\begin{array}{ccccc}
U_{\mathcal{F}}^{\circ} & \subset & U_{\mathcal{F}} & \subset & \Aut^{\otimes!}(\mathcal{F})\\
\cap &  & \cap &  & \cap\\
P_{\mathcal{F}}^{\circ} & \subset & P_{\mathcal{F}} & \subset & \Aut^{\otimes}(\mathcal{F})
\end{array}\quad\mbox{and}\quad\begin{array}{ccccc}
\Aut^{\otimes}(\mathcal{F}) & \subset & N(P_{\mathcal{F}}) & \subset & N(P_{\mathcal{F}}^{\circ})\\
\parallel\\
\Aut^{\otimes}(\mathcal{F}) & \subset & N(U_{\mathcal{F}}) & \subset & N(U_{\mathcal{F}}^{\circ})
\end{array}
\]
The Lie algebras of $U_{\mathcal{F}}^{\circ}\subset U_{\mathcal{F}}$
and $P_{\mathcal{F}}^{\circ}\subset P_{\mathcal{F}}$ are respectively
given by
\[
\begin{array}{rccccl}
 & \Lie(U_{\mathcal{F}}^{\circ}) & = & \Lie(U_{\mathcal{F}}) & = & \left(\mathcal{I}(G_{T})/\mathcal{J}(\mathcal{F})+\mathcal{I}(G_{T})^{2}\right)^{\vee}\\
\mbox{and} & \Lie(P_{\mathcal{F}}^{\circ}) & = & \Lie(P_{\mathcal{F}}) & = & \left(\mathcal{I}(G_{T})/\mathcal{I}(\mathcal{F})+\mathcal{I}(G_{T})^{2}\right)^{\vee}
\end{array}
\]
As quasi-coherent $\mathcal{O}_{T}$-submodules of 
\[
\Lie(G_{T})=\mathfrak{g}_{T}=\left(\mathcal{I}(G_{T})/\mathcal{I}(G_{T})^{2}\right)^{\vee}
\]
they correspond to the $\mathcal{O}_{T}$-linear forms on $\omega_{T}(\rho^{1})=\mathcal{I}(G_{T})/\mathcal{I}(G_{T})^{2}$
vanishing on 
\[
\mathcal{F}^{0}(\rho^{1})=\mathcal{J}(\mathcal{F})+\mathcal{I}(G_{T})^{2}/\mathcal{I}(G_{T})^{2}\quad\mbox{and}\quad\mathcal{F}_{+}^{0}(\rho^{1})=\mathcal{I}(\mathcal{F})+\mathcal{I}(G_{T})^{2}/\mathcal{I}(G_{T})^{2}
\]
respectively. We thus find that, as $\mathcal{O}_{T}$-submodules
of $\mathfrak{g}_{T}=\omega_{T}(\rho_{\mathrm{ad}})=\omega_{T}(\rho_{1})$,
\[
\Lie(U_{\mathcal{F}}^{\circ})=\Lie(U_{\mathcal{F}})=\mathcal{F}_{+}^{0}(\rho_{\mathrm{ad}})\quad\mbox{and}\quad\Lie(P_{\mathcal{F}}^{\circ})=\Lie(P_{\mathcal{F}})=\mathcal{F}^{0}(\rho_{\mathrm{ad}}).
\]

\subsection{~\label{sub:ParUFPF}}

We show that $P_{\mathcal{F}}^{\circ}$ is a parabolic subgroup of
$G_{T}$ with unipotent radical $U_{\mathcal{F}}^{\circ}$. Since
both groups are finitely presented and smooth over $T$ with $P_{\mathcal{F}}^{\circ}\subset N(U_{\mathcal{F}}^{\circ})$,
we may assume that $T=\Spec(k)$ for some algebraically closed field
$k$ by~\cite[XXVI 1.1 and 1.6]{SGA3.3r}. Since then $T\rightarrow S$
is affine, we may \emph{also }assume that $S=\Spec(k)$ by part $(2)$
of proposition~\ref{prop:SkalarExt}, in which case $G$ is linear
by lemma~\ref{lem:AssLin}. Using the criterion of \cite[IV 2.4.3.1]{SaRi72},
we now have to verify that 
\[
(a)\,\dim U_{\mathcal{F}}^{\circ}=\dim G/P_{\mathcal{F}}^{\circ}\quad\mbox{and}\quad(b)\, U_{\mathcal{F}}^{\circ}\mbox{ is unipotent.}
\]
The equality of dimensions follows from proposition~\ref{pro:trivkappaGF}
below since 
\[
\dim U_{\mathcal{F}}^{\circ}=\dim_{k}\Lie(U_{\mathcal{F}}^{\circ})=\dim_{k}\mathcal{F}_{+}^{0}(\rho_{\mathrm{ad}})={\textstyle \sum_{\gamma>0}}\dim_{k}\Gr_{k}^{\gamma}(\rho_{\mathrm{ad}})
\]
and 
\[
\dim G/P_{\mathcal{F}}^{\circ}=\dim_{k}\mathfrak{g}/\mathcal{F}^{0}(\rho_{\mathrm{ad}})=\dim_{k}\mathcal{F}_{+}^{0}(\rho_{\mathrm{ad}}^{\vee})={\textstyle \sum_{\gamma>0}}\dim_{k}\Gr_{k}^{\gamma}(\rho_{\mathrm{ad}}).
\]
For $(b)$, pick a finite dimensional faithful representation $\tau$
of $G$. Then 
\[
U_{\mathcal{F}}^{\circ}\subset U_{\mathcal{F}}\subset\Aut^{\otimes!}(\mathcal{F})\subset U(\mathcal{F}(\tau))
\]
where $U(\mathcal{F}(\tau))$ is the unipotent subgroup of $GL(V(\tau))$
defined by the $\Gamma$-filtration $\mathcal{F}(\tau)$ on $V(\tau)$.
Therefore $U_{\mathcal{F}}^{\circ}$ is unipotent by \cite[XVII 2.2.ii]{SGA3.2}.

\subsection{~\label{sub:EndProfAut(F)Parab}}

By \cite[XXII 5.8.5]{SGA3.3r}, $P_{\mathcal{F}}^{\circ}=N(P_{\mathcal{F}}^{\circ})$,
therefore also
\[
P_{\mathcal{F}}^{\circ}=P_{\mathcal{F}}=\Aut^{\otimes}(\mathcal{F})=N(P_{\mathcal{F}})=N(P_{\mathcal{F}}^{\circ}).
\]
On the other hand, the above proof of $(b)$ shows that $U_{\mathcal{F}}$
has unipotent geometric fibers, and then so does its quotient $U_{\mathcal{F}}/U_{\mathcal{F}}^{\circ}$
by \cite[XVII 2.2.iii]{SGA3.2}. Since $U_{\mathcal{F}}/U_{\mathcal{F}}^{\circ}$
is also normal in the reductive group $P_{\mathcal{F}}^{\circ}/U_{\mathcal{F}}^{\circ}$,
it must be trivial, thus $U_{\mathcal{F}}=U_{\mathcal{F}}^{\circ}$
and this finishes the proof of our theorem. Note that we can not say
much more about $\Aut^{\otimes!}(\mathcal{F})$ at this point -- we
do not even know that it is actually representable.

\section{Grothendieck groups\label{sub:GrothendieckGroups}}

Let again $G$ be affine and flat over $S$. Let $T$ be an $S$-scheme
and let $\mathcal{F}$ be a $\Gamma$-filtration on $\omega_{T}^{\circ}$.
Since $\mathcal{F}$ and $\Gr$ are exact $\otimes$-functors, 
\[
\Gr_{\mathcal{F}}^{\bullet}:\Rep^{\circ}(G)(S)\stackrel{\mathcal{F}}{\longrightarrow}\Fil^{\Gamma}\LF(T)\stackrel{\Gr}{\longrightarrow}\Gra^{\Gamma}\LF(T)
\]
is also an exact $\otimes$-functor. It thus induces a morphism between
the Grothendieck ring $K_{0}(G)$ \nomenclature[K0(G)]{$K_0(G)$}{Grothendieck ring of $\Rep^\circ (G)(S)$, page \nomrefpage}of
$\Rep^{\circ}(G)(S)$ and the Grothendieck ring of $\Gra^{\Gamma}\LF(T)$.
The rank function on finite locally free sheaves over $T$ defines
a morphism from the latter ring to the ring $\mathcal{C}(T,\mathbb{Z}[\Gamma])$
of locally constant functions on $T$ with values in the group ring
$\mathbb{Z}[\Gamma]$ of $\Gamma$. The $\Gamma$-filtration $\mathcal{F}$
on $\omega_{T}^{\circ}$ thus defines a ring homomorphism 
\[
\kappa(\mathcal{F}):K_{0}(G)\rightarrow\mathcal{C}(T,\mathbb{Z}[\Gamma])
\]
which maps the class of $\rho\in\Rep^{\circ}(G)(S)$ in $K_{0}(G)$
to the function 
\[
t\mapsto{\textstyle \sum_{\gamma\in\Gamma}}\dim_{k(t)}\left(\Gr_{\mathcal{F}}^{\gamma}(\rho)\otimes k(t)\right)\cdot e^{\gamma}
\]
where $e^{\gamma}$ is the basis element of $\mathbb{Z}[\Gamma]$
corresponding to $\gamma$. We have:
\[
\forall z\in K_{0}(G):\qquad\kappa(\mathcal{F})(z^{\vee})=\kappa(\mathcal{F})(z)^{\vee}
\]
where the involutions $z\mapsto z^{\vee}$ are induced by the duality
$\rho\mapsto\rho^{\vee}$ \nomenclature[rho^vee]{$\rho ^\vee$}{Dual of $\rho$.}on
$\Rep^{\circ}(G)(S)$ and by $\sum x_{\lambda}e^{\lambda}\mapsto\sum x_{\lambda}e^{-\lambda}$
on $\mathbb{Z}[\Gamma]$. When $G$ is smooth over $S$, we define
\begin{eqnarray*}
\kappa(G) & = & [\rho_{\mathrm{ad}}]-[\rho_{\mathrm{ad}}^{\vee}]\in K_{0}(G)\\
\mbox{and}\quad\kappa(G,\mathcal{F}) & = & \mbox{image of }\kappa(G)\mbox{ in }\mathcal{C}(T,\mathbb{Z}[\Gamma])
\end{eqnarray*}
The formation of $\kappa(G,\mathcal{F})$ is compatible with arbitrary
base change on $T$.
\begin{prop}
\label{pro:trivkappaGF}If $(1)$ $G$ is an isotrivial reductive
group over a quasi-compact $S$, or $(2)$ $G$ is a reductive group
over $S$ and $\mathcal{F}$ comes from a filtration on $\omega_{T}$,
then 
\[
\kappa(G,\mathcal{F})=0\quad\mbox{in}\quad\mathcal{C}(T,\mathbb{Z}[\Gamma]).
\]
\end{prop}
\begin{proof}
$(1)$ Let $\{S_{i}\rightarrow S\}$ be a covering of $S$ by finite
étale morphisms such that each $G_{i}=G_{S_{i}}$ splits. Let $\{T_{i}\rightarrow T\}$
be the corresponding covering of $T$. By part $(1)$ of proposition~\ref{prop:SkalarExt},
$\mathcal{F}_{T_{i}}$ extends to a $\Gamma$-filtration $\mathcal{F}_{i}$
on $\omega_{i}^{\circ}:\Rep^{\circ}(G_{i})(S_{i})\rightarrow\LF(T_{i})$,
and obviously $\kappa(G_{i},\mathcal{F}_{i})=\kappa(G,\mathcal{F})\circ(T_{i}\rightarrow T)$.
We may thus assume that $G$ splits over $S$, in which case the proposition
follows from lemma~\ref{lem:trivkappaG} below. The proof of $(2)$
is similar: let $\{t\rightarrow T\}$ be a covering of $T$ by geometric
points, thus $G_{t}$ splits. By part $(2)$ of proposition~\ref{prop:SkalarExt},
$\mathcal{F}_{t}$ extends to a $\Gamma$-filtration on $\omega_{t}^{\circ}:\Rep(G_{t})(t)\rightarrow\LF(t)$
which we also denote by $\mathcal{F}_{t}$, and obviously $\kappa(G_{t},\mathcal{F}_{t})=\kappa(G,\mathcal{F})\circ(t\rightarrow T)$. \end{proof}
\begin{lem}
\label{lem:trivkappaG}If $G$ is a split reductive group over a quasi-compact
$S$, then 
\[
\kappa(G)=0\quad\mbox{in}\quad K_{0}(G).
\]
\end{lem}
\begin{proof}
By~\cite[XXII 2.8]{SGA3.3r}, there is a decomposition $S=\coprod_{i\in I}S_{i}$
into open and closed subschemes $S_{i}\neq\emptyset$ of $S$ such
that for each $i\in I$, $G_{S_{i}}$ is of constant type, thus isomorphic
\cite[XXIII 5.2]{SGA3.3r} to the base change of a split reductive
group $G_{0,i}$ over $\Spec(\mathbb{Z})$ \cite[XXV 1.2]{SGA3.3r}.
Since $S$ is quasi-compact, the indexing set $I$ is finite and $K_{0}(G)\simeq\otimes_{i\in I}K_{0}(G_{S_{i}})$
with $\kappa(G)=\sum_{i\in I}\kappa(G)_{i}$ where $\kappa(G)_{i}$
is the image of $\kappa(G_{0,i})$ under $K_{0}(G_{0,i})\rightarrow K_{0}(G_{S_{i}})\rightarrow K_{0}(G)$.
We may thus assume that $S=\Spec(A)$ where $A$ a principal ideal
domain. By \cite[Théorème 5]{Se68b}, we may even assume that $A=K$
is a field. Let $H$ be a split maximal torus in $G$, with character
group $M$ and Weyl group $W$. The restriction $\Rep^{\circ}(G)\rightarrow\Rep^{\circ}(H)$
induces a ring homomorphism $K_{0}(G)\rightarrow K_{0}(H)\simeq\mathbb{Z}[M]$
which yields an isomorphism from $K_{0}(G)$ to $\mathbb{Z}[M]^{W}$
by \cite[Théorème 4]{Se68b}. Let $R\subset M$ be the set of roots
of $H$ in the Lie algebra $\mathfrak{g}=V(\rho_{\mathrm{ad}})$.
The weight decomposition of $\rho_{\mathrm{ad}}\vert H$ is then given
by $\mathfrak{g}=\mathfrak{g}_{0}\oplus\oplus_{\alpha\in R}\mathfrak{g}_{\alpha}$
with $\dim_{K}\mathfrak{g}_{\alpha}=1$ for $\alpha\in R$ and $\mathfrak{g}_{0}=\mathfrak{h}$
is the Lie algebra of $H$. Since $R=-R$, we find that $\rho_{\mathrm{ad}}\vert H\simeq\rho_{\mathrm{ad}}^{\vee}\vert H$.
Thus indeed $\kappa(G)=0$ in $K_{0}(G)$. 
\end{proof}

\section{The stabilizer of a $\Gamma$-filtration, II}

We have the following variant of theorem~\ref{thm:StabIsParV} and
\ref{thm:StabIsParw}. Let $G$ be an isotrivial reductive group over
a quasi-compact $S$. 
\begin{thm}
\label{thm:Stab0IsPar}For an $S$-scheme $T$ and a $\Gamma$-filtration
$\mathcal{F}$ on $V_{T}^{\circ}$ or $\omega_{T}^{\circ}$, $\Aut^{\otimes}(\mathcal{F})$
is a parabolic subgroup $P_{\mathcal{F}}$ of $G_{T}$ with unipotent
radical $U_{\mathcal{F}}\subset\Aut^{\otimes!}(\mathcal{F})$. Moreover,
\[
\Lie(U_{\mathcal{F}})=\mathcal{F}_{+}^{0}(\rho_{\mathrm{ad}})\quad\mbox{and}\quad\Lie(P_{\mathcal{F}})=\mathcal{F}^{0}(\rho_{\mathrm{ad}})\quad\mbox{in}\quad V_{T}(\rho_{\mathrm{ad}})=\mathfrak{g}_{T}.
\]
\end{thm}
\begin{cor}
For any $S$-scheme $T$ and $\mathcal{F}\in\mathbb{F}^{\Gamma}(G)(T)$,
\[
P_{\mathcal{F}}=\Aut^{\otimes}(\iota\mathcal{F}),\quad U_{\mathcal{F}}=\Aut^{\otimes!}(\iota\mathcal{F})\quad\mbox{and}\quad P_{\mathcal{F}}/U_{\mathcal{F}}=\Aut^{\otimes}(\Gr_{\iota\mathcal{F}}^{\bullet})
\]
where $\iota\mathcal{F}$ stands for the image of $\mathcal{F}$ in
either $\mathbb{F}^{\Gamma}(V_{T}^{\circ})$ or $\mathbb{F}^{\Gamma}(\omega_{T}^{\circ})$. 
\end{cor}
\noindent The proof of the corollary is identical to that of its
earlier counterpart.

\subsection{~}

By~propositions~\ref{Pro:Wedhorn}, \ref{Pro:Aut(w0)G(w0)F(w0)},
\ref{Pro:Aut(V0)G(V0)F(V0)} and \ref{Pro:F(V0)2F(w0)inj}, it is
sufficient to establish the theorem for a $\Gamma$-filtration $\mathcal{F}$
on $\omega_{T}^{\circ}$. For any $T$-scheme $X$, we have {\small{
\begin{eqnarray*}
\Aut^{\otimes}(\mathcal{F})(X) & = & \left\{ g\in G(X)\vert\forall\tau,\gamma\in\Rep^{\circ}(G)(S)\times\Gamma:\,\rho_{X}(g)\left(\mathcal{F}^{\gamma}(\tau)_{X}\right)=\mathcal{F}^{\gamma}(\tau)_{X}\right\} ,\\
 & = & \left\{ g\in G(X)\vert\forall\rho,\gamma\in\Rep^{\prime}(G)(S)\times\Gamma:\,\rho_{X}(g)\left(\mathcal{F}^{\gamma}(\rho)_{X}\right)=\mathcal{F}^{\gamma}(\rho)_{X}\right\} .
\end{eqnarray*}
}}We have to show that the fpqc subsheaf $\Aut^{\otimes}(\mathcal{F}):(\Sch/T)^{0}\rightarrow\Set$
of $G_{T}$ is representable by a parabolic subgroup with the specified
Lie algebra: this is a local question in the fpqc topology on $T$.
Let $\{S_{i}\rightarrow S\}$ be a covering of $S$ by finite étale
morphisms such that $G_{i}=G_{S_{i}}$ is split, let $\{T_{i}\rightarrow T\}$
be the induced covering of $T$, let $\omega_{i}$ denote the fiber
functors for $G_{i}$ and let $\mathcal{F}_{i}$ be the unique extension
of $\mathcal{F}_{T_{i}}$ to a $\Gamma$-graduation on $\omega_{i,T_{i}}^{\circ}$.
Going back to its actual definition in the proof of proposition~\ref{prop:SkalarExt},
one checks easily that $\Aut^{\otimes}(\mathcal{F})\vert_{T_{i}}=\Aut^{\otimes}(\mathcal{F}_{T_{i}})$
is equal to $\Aut^{\otimes}(\mathcal{F}_{i})$ as a subsheaf of $G\vert_{T_{i}}=\Aut^{\otimes}(\omega^{\circ})\vert_{T_{i}}=\Aut^{\otimes}(\omega_{i}^{\circ})\vert_{T_{i}}$.
We may (and do) therefore assume that $G$ is a split reductive group
over a quasi-compact $S$. By~\cite[XXII 2.8, XXIII 5.2 and XXV 1.2]{SGA3.3r},
we then have a finite partition of $S=\coprod S_{i}$ into open and
closed subschemes such that each $G_{i}=G_{S_{i}}$ arises from a
split group over $\Spec(\mathbb{Z})$, and repeating the above argument
with that covering, we may thus also assume that $G$ is the base
change of a split reductive group $G_{0}$ over $\Spec(\mathbb{Z})$.

\subsection{~}

In particular, the proof of part $(2)$ of proposition~\ref{Pro:Wedhorn}
now shows that with $\rho_{\mathrm{reg}}$, also $\rho_{\mathrm{adj}}$
and $\rho_{\mathrm{adj}}^{\circ}$ belong to $\Rep'(G)(S)$, to which
we have extended $\mathcal{F}$ in section~\ref{sub:extF}. We may
thus define subschemes $U_{\mathcal{F}}$ and $P_{\mathcal{F}}$ of
$G_{T}$ as in section~\ref{sub:defUFPF}, and try to follow from
there on the subsequent steps of the proof of theorem~\ref{thm:StabIsParV}.
Of course, we have to check that we are only using our filtration
where it is defined, namely on $\Rep'(G)(S)$, and that whenever the
axiom (F3) was used, we could have replaced it with the weaker left
and right properties (F3l) or (F3r).

\subsection{~\label{sub:dealingRightExact}}

In \ref{sub:finitegen} and \ref{sub:intersecprop}, we used the right
exactness of $\mathcal{F}$ for (respectively) 
\[
A:S^{\bullet}(\tau)=\Sym^{\bullet}(\tau^{\vee}\otimes\tau)\twoheadrightarrow\rho_{\mathrm{adj}}\quad\mbox{and}\quad B:(\rho_{\mathrm{adj}}^{\circ})^{\otimes2}\twoheadrightarrow\rho_{\mathrm{adj}}^{\circ(2)}.
\]
To deal with the first one, it would be sufficient to know that there
is a cofinal set $\Sigma\in X(\rho_{\mathrm{adj}})$ such that for
all $\sigma\in\Sigma$, $A^{-1}(\sigma)$ is still in $\Rep'(G)(S)$:
then 
\begin{eqnarray*}
\mathcal{F}^{\gamma}(\rho_{\mathrm{adj}}) & = & \underrightarrow{\lim}\,\mathcal{F}^{\gamma}(\sigma)\stackrel{\mathrm{F3r}}{=}\underrightarrow{\lim}\, A(\mathcal{F}^{\gamma}(A^{-1}(\sigma)))\\
 & = & A(\mathcal{F}^{\gamma}(\underrightarrow{\lim}\, A^{-1}(\sigma)))=A(\mathcal{F}^{\gamma}(S^{\bullet}(\tau))).
\end{eqnarray*}
Over a Dedekind domain, we have Wedhorn's criterion: a $\rho$ is
in $\Rep'(G)(S)$ if and only $V(\rho)$ is flat, i.e.~torsion free:
thus over such a domain, $A^{-1}(\sigma)$ still belongs to $\Rep'(G)(S)$
for any $\sigma\in X(\rho_{\mathrm{adj}})$. Applying this to $G_{0}$
and choosing $\tau$ in \ref{sub:finitegen} to also be defined over
$\Spec(\mathbb{Z})$ settles the case of $A$, and that of $B$ is
similar.

\subsection{~}

Everything then goes through up to \ref{sub:ParUFPF}: $U_{\mathcal{F}}^{\circ}$
and $P_{\mathcal{F}}^{\circ}$ are smooth subgroups of $G_{T}$ with
the good Lie algebras, etc\ldots{} In \ref{sub:ParUFPF}, we may
still reduce to the case where $T=\Spec(k)$ for some algebraically
closed field $k$ and use the criterion of \cite[IV 2.4.3.1]{SaRi72},
but we can not change $S$ to $\Spec(k)$. However, since we have
already reduced to the split case, proposition~\ref{pro:trivkappaGF}
(or lemma~\ref{lem:trivkappaG}) deals perfectly well with condition
$(a)$, and lemma~\ref{lem:AssLin} with condition $(b)$.

\section{Splitting filtrations}

We now come to the main statement of theorem~\ref{thm:MainTan}.
Let thus $G$ be a reductive isotrivial group over a quasi-compact
$S$, let $T$ be an $S$-scheme and let $\mathcal{F}$ be a $\Gamma$-filtration
on $\omega_{T}^{\circ}$. We will then show that: locally on $T$
for the étale topology, $\mathcal{F}$ has a splitting $\chi:\mathbb{D}_{T}(\Gamma)\rightarrow G_{T}$.

\subsection{~\label{sub:red2split}}

Let $f:\tilde{S}\rightarrow S$ be a finite étale cover splitting
$G$ and denote by $\tilde{\mathcal{F}}$ the unique extension of
$\mathcal{F}_{\tilde{T}}$ to a $\Gamma$-filtration on $\tilde{\omega}_{\tilde{T}}^{\circ}$
(see proposition~\ref{prop:SkalarExt}), where $\tilde{T}=T_{\tilde{S}}$
and $\tilde{\omega}$ is the fiber functor for $\tilde{G}=G_{\tilde{S}}$.
If $\mathcal{\chi}:\mathbb{D}_{\tilde{T}}(\Gamma)\rightarrow G_{\tilde{T}}$
is a splitting of $\tilde{\mathcal{F}}$, it is \emph{a fortiori}
a splitting of $\mathcal{F}_{\tilde{T}}$: we may thus assume that
$G$ splits over $S$.

\subsection{~}

For a positive integer $k$, there is a cartesian diagram of fpqc
sheaves on $S$, 
\[
\xymatrix{\mathbb{G}^{\Gamma}(G)\ar@{=}[r]^{\mathrm{Prop.\,}\ref{Pro:Aut(w0)G(w0)F(w0)}}\ar[d]_{k_{1}} & \mathbb{G}^{\Gamma}(\omega^{\circ})\ar[r]^{\Fi}\ar[d]_{k_{2}} & \mathbb{F}^{\Gamma}(\omega^{\circ})\ar[d]_{k_{3}}\\
\mathbb{G}^{\Gamma}(G)\ar@{=}[r]^{\mathrm{Prop.\,}\ref{Pro:Aut(w0)G(w0)F(w0)}} & \mathbb{G}^{\Gamma}(\omega^{\circ})\ar[r]^{\Fi} & \mathbb{F}^{\Gamma}(\omega^{\circ})
}
\]
where the $k_{i}$'s map $\chi$, $\mathcal{G}$ and $\mathcal{F}$
to respectively $k_{1}(\chi)=\chi\circ\mathbb{D}_{T}(k)=\chi^{k}$,
\[
k_{2}(\mathcal{G})_{\gamma}(\rho)=\begin{cases}
0 & \mbox{if }\gamma\notin k\Gamma,\\
\mathcal{G}_{\eta}(\rho) & \mbox{if }\gamma=k\eta,
\end{cases}\quad\mbox{and}\quad k_{3}(\mathcal{F})^{\gamma}(\rho)=\cup_{k\eta\geq\gamma}\mathcal{F}^{\eta}(\rho).
\]
They are all obviously well-defined monomorphisms, and the image of
$k_{2}$ is the subsheaf of $\mathbb{G}^{\Gamma}(\omega^{\circ})$
made of those $\Gamma$-graduation $\mathcal{G}'$ for which $\mathcal{G}'_{\gamma}\equiv0$
for $\gamma\notin k\Gamma$. The diagram is cartesian because if $\mathcal{G}'$
splits $k_{3}(\mathcal{F})$, then $\mathcal{G}'_{\gamma}\simeq\Gr_{k_{3}(\mathcal{F})}^{\gamma}\equiv0$
for $\gamma\notin k\Gamma$, thus $\mathcal{G}'=k_{2}(\mathcal{G})$
for a unique $\mathcal{G}$, which has to also split $\mathcal{F}$
since 
\[
k_{3}(\mathcal{F})=\Fi(\mathcal{G}')=\Fi\left(k_{2}(\mathcal{G})\right)=k_{3}(\Fi(\mathcal{G})).
\]

\subsection{~\label{sub:red2prod}}

For a central isogeny $f:G\rightarrow G'$, there is a commutative
diagram
\[
\xymatrix{\mathbb{G}^{\Gamma}(G)\ar@{=}[r]^{\mathrm{Prop.\,}\ref{Pro:Aut(w0)G(w0)F(w0)}}\ar[d]_{f_{1}} & \mathbb{G}^{\Gamma}(\omega^{\circ})\ar[r]^{\Fi}\ar[d]_{f_{2}} & \mathbb{F}^{\Gamma}(\omega^{\circ})\ar[d]_{f_{3}}\\
\mathbb{G}^{\Gamma}(G')\ar@{=}[r]^{\mathrm{Prop.\,}\ref{Pro:Aut(w0)G(w0)F(w0)}} & \mathbb{G}^{\Gamma}(\omega^{\prime\circ})\ar[r]^{\Fi} & \mathbb{F}^{\Gamma}(\omega^{\prime\circ})
}
\]
where $\omega'=\omega\circ f^{\ast}$ denotes the fiber functor for
$G'$ and the $f_{i}$'s map $\chi$, $\mathcal{G}$ and $\mathcal{F}$
to respectively $f_{1}(\chi)=f\circ\chi$, $f_{2}(\mathcal{G})=\mathcal{G}\circ f^{\ast}$
and $f_{3}(\mathcal{F})=\mathcal{F}\circ f^{\ast}$, with
\[
f^{\ast}:\Rep(G')(S)\rightarrow\Rep(G)(S)\qquad f^{\ast}(\rho)=\rho\circ f.
\]
We claim that $(1)$ all $f_{i}$'s are monomorphisms, and $(2)$
the diagram is cartesian. This is local in the finite étale topology
on $S$ by proposition~\ref{prop:SkalarExt}, and we may thus assume
that the kernel $C$ of $f$ is isomorphic to $\mathbb{D}_{S}(X)$
for some finite commutative group $X$. We fix an $S$-scheme $T$
and consider sections of the above sheaves over $T$. If $f\circ\chi_{1}=f\circ\chi_{2}$,
then $\chi_{1}^{-1}\chi_{2}$ is a morphism $\mathbb{D}_{T}(\Gamma)\rightarrow C_{T}$,
which has to be trivial since $X$ is finite and $\Gamma$ torsion
free: $f_{1}$ is injective. Any $\rho\in\Rep^{\circ}(G)(S)$ has
a finite sum decomposition $\rho=\oplus\rho(x)$ according to the
characters $x\in X$ of $C$, and $C$ acts trivially on $\rho(x)^{\otimes k(x)}$
where $k(x)\geq1$ is the order of $x$ in $X$. If two $\Gamma$-filtrations
$\mathcal{F}_{1}$ and $\mathcal{F}_{2}$ on $\omega_{T}^{\circ}$
induce the same $\Gamma$-filtration on $\omega{}_{T}^{\prime\circ}$,
then $\mathcal{F}_{1}(\rho)=\mathcal{F}_{2}(\rho)$ for every $\rho$
on which $C$ acts trivially, thus $\mathcal{F}_{1}(\rho(x))=\mathcal{F}_{2}(\rho(x))$
for every $\rho$ and $x$ by lemma~\ref{lem:FiltOntensorprod} below,
therefore $\mathcal{F}_{1}(\rho)=\mathcal{F}_{2}(\rho)$ since $\rho=\oplus\rho(x)$:
$f_{3}$ is injective. Similarly: $f_{2}$ is injective. Finally,
suppose that $\mathcal{G}'$ splits $f_{3}(\mathcal{F})$. Let $\chi':\mathbb{D}_{T}(\Gamma)\rightarrow G'_{T}$
be the corresponding morphism. Fix $k\geq1$ such that $k_{1}(\chi')$
lifts to $\chi_{k}:\mathbb{D}_{T}(\Gamma)\rightarrow G_{T}$, giving
a $\Gamma$-graduation $\mathcal{G}_{k}$ and a $\Gamma$-filtration
$\mathcal{F}_{k}$ on $\omega_{T}^{\circ}$. They respectively map
to 
\[
f_{2}(\mathcal{G}_{k})=f_{2}\circ\iota(\chi_{k})=\iota\circ f_{1}(\chi_{k})=\iota\circ k_{1}(\chi')=k_{2}\circ\iota(\chi')=k_{2}(\mathcal{G}')
\]
where $\iota$ is the isomorphism $\mathbb{G}^{\Gamma}(G)\simeq\mathbb{G}^{\Gamma}(\omega^{\circ})$,
and 
\[
f_{3}(\mathcal{F}_{k})=f_{3}\circ\Fi(\mathcal{G}_{k})=\Fi\circ f_{2}(\mathcal{G}_{k})=\Fi\circ k_{2}(\mathcal{G}')=k_{3}\circ\Fi(\mathcal{G}')=k_{3}\circ f_{3}(\mathcal{F}).
\]
Thus $f_{3}(\mathcal{F}_{k})=f_{3}\circ k_{3}(\mathcal{F})$ and $\mathcal{F}_{k}=k_{3}(\mathcal{F})$
since $f_{3}$ is a monomorphism. Since $\mathcal{G}_{k}$ splits
$k_{3}(\mathcal{F})$, there is a unique $\mathcal{G}$ such that
$\mathcal{F}=\Fi(\mathcal{G})$ and $k_{2}(\mathcal{G})=\mathcal{G}_{k}$
by the cartesian diagram of the previous subsection. Moreover $f_{2}(\mathcal{G})=\mathcal{G}'$
since 
\[
k_{2}\circ f_{2}(\mathcal{G})=f_{2}\circ k_{2}(\mathcal{G})=f_{2}(\mathcal{G}_{k})=k_{2}(\mathcal{G}')
\]
and $k_{2}$ is a monomorphism: our diagram is indeed cartesian. 
\begin{lem}
\label{lem:FiltOntensorprod}Let $\mathcal{M}$ be a finite locally
free sheaf on a scheme $S$, $k\geq1$.
\begin{enumerate}
\item Let $\mathcal{F}_{1}$ and $\mathcal{F}_{2}$ be local direct factors
of $\mathcal{M}$. Then:
\[
\mathcal{F}_{1}^{\otimes k}=\mathcal{F}_{2}^{\otimes k}\mbox{ in }\mathcal{M}^{\otimes k}\,\Longrightarrow\,\mathcal{F}_{1}=\mathcal{F}_{2}\mbox{ in }\mathcal{M}.
\]

\item Let $\mathcal{F}_{1}$ and $\mathcal{F}_{2}$ be $\Gamma$-filtrations
on $\mathcal{M}$. Then:
\[
\mathcal{F}_{1}^{\otimes k}=\mathcal{F}_{2}^{\otimes k}\mbox{ on }\mathcal{M}^{\otimes k}\,\Longrightarrow\,\mathcal{F}_{1}=\mathcal{F}_{2}\mbox{ on }\mathcal{M}.
\]

\end{enumerate}
\end{lem}
\begin{proof}
$(1)$ Fix $s\in S$ with residue field $k(s)$. We have to show that
$\mathcal{F}_{1}=\mathcal{F}_{2}$ in a neighborhood of $s$. Shrinking
$S$ if necessary, we may assume that $\mathcal{F}_{1}$ and $\mathcal{F}_{2}$
are free of constant rank $n_{1}$ and $n_{2}$. By assumption, $n_{1}^{k}=n_{2}^{k}$,
therefore $n_{1}=n_{2}=n$. If $n=0$, $\mathcal{F}_{1}=0=\mathcal{F}_{2}$
and we are done. Suppose $n>0$, and choose a linear form $f:\mathcal{M}(s)\rightarrow k(s)$
which is non-zero on $\mathcal{F}_{1}(s)$ and $\mathcal{F}_{2}(s)$.
Shrinking $S$ further, we may lift $f$ to an $\mathcal{O}_{S}$-linear
map $f:\mathcal{M}\rightarrow\mathcal{O}_{S}$ such that $f(\mathcal{F}_{1})=\mathcal{O}_{S}=f(\mathcal{F}_{2})$.
Then for the $\mathcal{O}_{S}$-linear map $F=\mathrm{Id}\otimes f^{k-1}:\mathcal{M}^{\otimes k}\rightarrow\mathcal{M}$,
we have 
\[
\mathcal{F}_{1}=F(\mathcal{F}_{1}^{\otimes k})=F(\mathcal{F}_{2}^{\otimes k})=\mathcal{F}_{2}.
\]
$(2)$ The question is local for the Zariski topology on $S$. By
proposition~\ref{prop:FilonLF}, we may thus assume that both filtrations
split, say 
\[
\mathcal{F}_{1}^{\gamma}=\oplus_{\eta\geq\gamma}\mathcal{G}_{1}^{\eta}\quad\mbox{and}\quad\mathcal{F}_{2}^{\gamma}=\oplus_{\eta\geq\gamma}\mathcal{G}_{2}^{\eta}
\]
with $\mathcal{G}_{i}^{\gamma}$ locally free of constant rank $n_{i}^{\gamma}$
for every $i\in\{1,2\}$ and $\gamma\in\Gamma$. We then argue by
induction on the constant rank $n=\sum n_{1}^{\gamma}=\sum n_{2}^{\gamma}$
of $\mathcal{M}$. For $n=0$, there is nothing to prove. Suppose
$n>0$. By assumption, for every $\gamma\in\Gamma$, 
\[
\sum_{a_{1}+\cdots+a_{k}=\gamma}\mathcal{F}_{1}^{a_{1}}\otimes\cdots\otimes\mathcal{F}_{1}^{a_{k}}=\sum_{a_{1}+\cdots+a_{k}=\gamma}\mathcal{F}_{2}^{a_{1}}\otimes\cdots\otimes\mathcal{F}_{2}^{a_{k}}
\]
 which means that
\[
\bigoplus_{a_{1}+\cdots+a_{k}\geq\gamma}\mathcal{G}_{1}^{a_{1}}\otimes\cdots\otimes\mathcal{G}_{1}^{a_{k}}=\bigoplus_{a_{1}+\cdots+a_{k}\geq\gamma}\mathcal{G}_{2}^{a_{1}}\otimes\cdots\otimes\mathcal{G}_{2}^{a_{k}}
\]
Let $\gamma_{i}$ be the largest element of the (non-empty!) finite
set $\{a:\mathcal{G}_{i}^{a}\neq0\}$. Then 
\[
\oplus_{a_{1}+\cdots+a_{k}\geq\gamma}\mathcal{G}_{i}^{a_{1}}\otimes\cdots\otimes\mathcal{G}_{i}^{a_{k}}=\begin{cases}
0 & \mbox{if }\gamma>k\gamma_{i},\\
\mathcal{G}_{i}^{\gamma_{i}}\otimes\cdots\otimes\mathcal{G}_{i}^{\gamma_{i}} & \mbox{for }\gamma=k\gamma_{i},\\
\neq0 & \mbox{if }\gamma\leq k\gamma_{i}.
\end{cases}
\]
Thus $k\gamma_{1}=k\gamma_{2}$, $\gamma_{1}=\gamma_{2}=\gamma_{0}$
and $\mathcal{G}_{1}^{\gamma_{0}}\otimes\cdots\otimes\mathcal{G}_{1}^{\gamma_{0}}=\mathcal{G}_{2}^{\gamma_{0}}\otimes\cdots\otimes\mathcal{G}_{2}^{\gamma_{0}}$
in $\mathcal{M}^{\otimes k}$, therefore $\mathcal{F}_{1}^{\gamma_{0}}=\mathcal{G}_{1}^{\gamma_{0}}=\mathcal{G}_{2}^{\gamma_{0}}=\mathcal{F}_{2}^{\gamma_{0}}=\mathcal{N}$
in $\mathcal{M}$ by the previous lemma. We conclude by our induction
hypothesis applied to the images of $\mathcal{F}_{1}$ and $\mathcal{F}_{2}$
in $\mathcal{M}/\mathcal{N}$. 
\end{proof}

\subsection{~\label{sub:redFromprod}}

Suppose that $G=G_{1}\times_{S}G_{2}$. Let $\mathcal{F}$ be a $\Gamma$-filtration
on $\omega_{T}^{\circ}$. Then $\mathcal{F}$ induces a $\Gamma$-filtration
$\mathcal{F}_{i}$ on the fiber functor $\omega_{i,T}^{\circ}$ for
$G_{i}$ by the formulas: 
\[
\mathcal{F}_{1}^{\gamma}(\rho_{1})=\mathcal{F}^{\gamma}(\rho_{1}\boxtimes1_{G_{2}})\quad\mbox{and}\quad\mathcal{F}_{2}^{\gamma}(\rho_{2})=\mathcal{F}^{\gamma}(1_{G_{1}}\boxtimes\rho_{2})
\]
We claim that if $\chi_{i}$ splits $\mathcal{F}_{i}$, then $\chi=(\chi_{1},\chi_{2})$
splits $\mathcal{F}$. Indeed, we may as above assume that $G_{1}$
and $G_{2}$ are split, and we extend $\mathcal{F}$ to $\Rep'(G)(S)$.
We then have to show that the $\Gamma$-filtration $\mathcal{F}'$
associated to $\chi$ equals $\mathcal{F}$ on $\rho_{\mathrm{reg}}$.
Since 
\[
\rho_{\mathrm{reg}}=\rho_{1,\mathrm{reg}}\boxtimes\rho_{2,\mathrm{reg}}=\underrightarrow{\lim}\,\tau_{1}\boxtimes\tau_{2}
\]
where $\rho_{i,\mathrm{reg}}$ is the regular representation of $G_{i}$
and the colimit is over $\tau_{i}\in X(\rho_{i,\mathrm{reg}})$, it
is also sufficient to establish that $\mathcal{F}'$ equals $\mathcal{F}$
on $\rho=\tau_{1}\boxtimes\tau_{2}$, $\tau_{i}\in\Rep^{\circ}(G)(S)$.
Note that $\rho=\rho_{1}\otimes\rho_{2}$ where $\rho_{1}=\tau_{1}\boxtimes1_{G_{2}}$
and $\rho_{2}=1_{G_{1}}\boxtimes\tau_{2}$. We thus find
\begin{eqnarray*}
\mathcal{F}^{\gamma}(\rho) & = & {\textstyle \sum_{\gamma_{1}+\gamma_{2}=\gamma}}\mathcal{F}^{\gamma_{1}}(\rho_{1})\otimes\mathcal{F}^{\gamma_{2}}(\rho{}_{2})\\
 & = & {\textstyle \sum_{\gamma_{1}+\gamma_{2}=\gamma}}\mathcal{F}_{1}^{\gamma_{1}}(\tau_{1})\otimes\mathcal{F}_{2}^{\gamma_{2}}(\tau{}_{2})\\
 & = & {\textstyle \oplus_{\gamma_{1}+\gamma_{2}\geq\gamma}}\mathcal{G}_{1}^{\gamma_{1}}(\tau_{1})\otimes\mathcal{G}_{2}^{\gamma_{2}}(\tau_{2})\\
 & = & \oplus_{\eta\geq\gamma}\mathcal{G}^{\eta}(\tau_{1}\boxtimes\tau_{2})\\
 & = & \mathcal{F}^{\prime\gamma}(\rho)
\end{eqnarray*}
where $\mathcal{G}$ and the $\mathcal{G}_{i}$'s are the $\Gamma$-graduations
induced by $\chi$ and the $\chi_{i}$'s.

\subsection{~}

Applying \ref{sub:red2split}, \ref{sub:red2prod} with the central
isogeny from $G$ to the product of its adjoint group and its coradical,
and finally \ref{sub:redFromprod}, we may assume that $G$ is either
a split torus or a split reductive group of adjoint type.

\subsection{~\label{sub:Splitting:caseoftori}}

Let thus first $G=\mathbb{D}_{S}(M)$ for some $M\simeq\mathbb{Z}^{d}$
and let $\mathcal{F}$ be a $\Gamma$-filtration on $\omega_{T}^{\circ}$
for an $S$-scheme $T$, which we may assume to be (absolutely) affine.
Let $\rho_{m}$ be the representation of $G$ on $V(\rho_{m})=\mathcal{O}_{S}$
on which $G$ acts by the character $m\in M$. By proposition~\ref{prop:FilonLF},
there exists a $\Gamma$-graduation $\mathcal{O}_{T}=\oplus_{\gamma}\mathcal{I}_{\gamma}(m)$
such that 
\[
\forall\gamma\in\Gamma:\qquad\mathcal{F}^{\gamma}(\rho_{m})=\oplus_{\eta\geq\gamma}\mathcal{I}_{\eta}(m).
\]
Let $T_{\gamma}(m)$ be the support of $\mathcal{I}_{\gamma}(m)$,
so that $T=\coprod_{\gamma}T_{\gamma}(m)$ and $T_{\gamma}(m)$ is
open and closed in $T$. For $t\in T$ and $m\in M$, we denote by
$f(t)(m)$ the unique element $\gamma$ in $\Gamma$ such that $t$
belongs to $T_{\gamma}(m)$. Thus $\mathcal{F}_{t}^{\gamma}(\rho_{m})=k(t)$
if $\gamma\leq f(t)(m)$ and $0$ otherwise, where $k(t)$ is the
residue field at $t$. Since $\rho_{0}=1_{G}$, $f(t)(0)=0$ by the
axiom (F2) for $\mathcal{F}$. Since $\rho_{m_{1}}\otimes\rho_{m_{2}}=\rho_{m_{1}+m_{2}}$,
$f(t)(m_{1}+m_{2})=f(t)(m_{1})+f(t)(m_{2})$ by the axiom (F1) for
$\mathcal{F}$. Therefore $f(t):M\rightarrow\Gamma$ is a group homomorphism.
Since $M$ is finitely generated, $f:T\rightarrow\Hom_{\Group}(M,\Gamma)$
is locally constant, and thus corresponds to a global section $\chi:\mathbb{D}_{T}(\Gamma)\rightarrow G_{T}$
of the locally constant sheaf (see \cite[VIII 1.5]{SGA3.2})
\[
\Hom(M,\Gamma)_{T}=\underline{\Hom}(M_{T},\Gamma_{T})=\underline{\Hom}(\mathbb{D}_{T}(\Gamma),\mathbb{D}_{T}(M))=\underline{\Hom}(\mathbb{D}_{T}(\Gamma),G_{T}).
\]
Let $\mathcal{F}'$ be the corresponding $\Gamma$-filtration on $\omega_{T}$.
For any morphism $\phi:M\rightarrow\Gamma$, let $T(\phi)$ be the
open and closed subset of $T$ where $f\equiv\phi$, so that $T=\coprod T(\phi)$
and 
\[
T(\phi)=\cap_{m\in M}T_{\phi(m)}(m)=\cap_{i=1}^{r}T_{\phi(m_{i})}(m_{i})
\]
if $\{m_{1},\ldots,m_{r}\}\subset M$ spans $M$. On $T(\phi)$, we
find that 
\[
\mathcal{F}_{T(\phi)}^{\prime\gamma}(\rho_{m})=\left\{ \begin{array}{ll}
\mathcal{O}_{T(\phi)} & \mbox{if }\gamma\leq\phi(m)\\
0 & \mbox{if }\gamma>\phi(m)
\end{array}\right\} =\mathcal{F}_{T(\phi)}^{\gamma}(\rho_{m}).
\]
Thus $\mathcal{F}'(\rho_{m})=\mathcal{F}(\rho_{m})$ for every $m$.
Extending $\mathcal{F}$ as in~\ref{sub:extF}, also $\mathcal{F}'(\rho_{\mathrm{reg}})=\mathcal{F}(\rho_{\mathrm{reg}})$
since $\rho_{\mathrm{reg}}=\oplus_{m\in M}\rho_{m}$ . Finally $\mathcal{F}'(\rho)=\mathcal{F}(\rho)$
for any $\rho$ by (F3l) applied to $c_{\rho}$. Therefore $\chi$
is a splitting of $\mathcal{F}$ -- it is in fact the unique such
splitting.

\subsection{~\label{sub:splitonadisenough}}

Suppose finally that $G$ is a split reductive group of adjoint type
over $S$, let $T$ be an $S$-scheme, and let $\mathcal{F}$ be a
$\Gamma$-filtration on $\omega_{T}^{\circ}$. We have just recalled
that $\mathcal{F}$ is uniquely determined by the value of its extension
to $\Rep'(G)(S)$ on $\rho_{\mathrm{reg}}$, but we now also have
this: there is at most one $\mbox{\ensuremath{\Gamma}}$-filtration
$\mathcal{F}'$ on $\omega_{T}$ which equals $\mathcal{F}$ on the
adjoint representation $\rho_{\mathrm{ad}}$ of $G$ on $V(\rho_{\mathrm{ad}})=\mathfrak{g}=\Lie(G)$.
In particular, any morphism $\chi:\mathbb{D}_{T}(\Gamma)\rightarrow G_{T}$
inducing $\mathcal{F}$ on $\rho_{\mathrm{ad}}$ is a splitting of
$\mathcal{F}$. To establish our claim, we consider the $G$-equivariant
epimorphism of quasi-coherent $G$-$\mathcal{O}_{S}$-algebras 
\[
f:\Sym_{\mathcal{O}_{S}}^{\bullet}(\rho_{\mathrm{ad},0}^{\vee}\otimes\rho_{\mathrm{ad}})\twoheadrightarrow\rho_{\mathrm{reg}}
\]
which is defined as in section~\ref{sub:adjreg}, starting from $c_{\mathrm{ad}}:\rho_{\mathrm{ad}}\rightarrow\rho_{\mathrm{ad},0}\otimes\rho_{\mathrm{reg}}$
for the closed embedding $\rho_{\mathrm{ad}}:G\rightarrow GL(\mathfrak{g})$.
If $\mathcal{F}'$ equals $\mathcal{F}$ on $\rho_{\mathrm{ad}}$,
they are also equal on $\Sym^{\bullet}\left(\rho_{\mathrm{ad},0}^{\vee}\otimes\rho_{\mathrm{ad}}\right)$
by the axioms (F1-3) for $\Gamma$-filtrations on $\omega_{T}^{\circ}$,
thus also
\[
\mathcal{F}^{\prime\gamma}(\rho_{\mathrm{reg}})\subset\mathcal{F}^{\gamma}(\rho_{\mathrm{reg}})
\]
for every $\gamma\in\Gamma$ by the axiom (F3) for the $\Gamma$-filtration
$\mathcal{F}'$ on $\omega_{T}$ --- it is not yet known to be satisfied
by the extension of $\mathcal{F}$ to $\Rep'(G)(S)$, unless we appeal
to the arguments of section~\ref{sub:dealingRightExact}, which is
not necessary: then $\mathcal{F}^{\prime\gamma}(\rho)\subset\mathcal{F}^{\gamma}(\rho)$
for every $\rho$ in $\Rep^{\circ}(G)(S)$ by (F3l) with $c_{\rho}$,
therefore also $\mathcal{F}_{+}^{\prime\gamma}(\rho)\subset\mathcal{F}_{+}^{\gamma}(\rho)$;
applying the latter inclusion to $\rho^{\vee}$ and dualizing gives
$\mathcal{F}^{\gamma}(\rho)\subset\mathcal{F}^{\prime\gamma}(\rho)$.
Thus $\mathcal{F}=\mathcal{F}'$ on $\omega_{T}^{\circ}$.

\subsection{~}

By theorem~\ref{thm:Stab0IsPar}, $P_{\mathcal{F}}=\Aut^{\otimes}(\mathcal{F})$
is a parabolic subgroup of $G_{T}$. Since our problem is local for
the étale topology on $T$, we may assume that $T$ is affine and
the pair $(G_{T},P_{\mathcal{F}})$ has an épinglage $\mathcal{E}=(H,M,R,\cdots)$
\cite[XXVI 1.14]{SGA3.3r}. Thus $H=\mathbb{D}_{T}(M)$ is a trivialized
split maximal torus of $G_{T}$ contained in $P_{\mathcal{F}}$, $R\subset M$
is the set of roots of $H$ in $\mathfrak{g}_{T}$ and if $\mathfrak{g}_{T}=\mathfrak{g}_{0}\oplus\oplus_{\alpha\in R}\mathfrak{g}_{\alpha}$
is the corresponding weight decomposition (so that $\mathfrak{g}_{0}=\Lie(H)$),
then $\Lie(P_{\mathcal{F}})=\mathfrak{g}_{0}\oplus\oplus_{\alpha\in R'}\mathfrak{g}_{\alpha}$
for some subset $R'$ of $R$ as in \cite[XXVI 1.4]{SGA3.3r}. The
maximal torus $H\subset P_{\mathcal{F}}$ gives rise to a Levi decomposition
$P_{\mathcal{F}}=U_{\mathcal{F}}\rtimes L_{\mathcal{F}}$ with $H\subset L_{\mathcal{F}}$,
$\Lie(L_{\mathcal{F}})=\mathfrak{g}_{0}\oplus\oplus_{\alpha\in R'_{1}}\mathfrak{g}_{\alpha}$
and $\Lie(U_{\mathcal{F}})=\oplus_{\alpha\in R'_{2}}\mathfrak{g}_{\alpha}$
where $R'_{1}=\{\alpha\in R':-\alpha\in R'\}$ and $R'_{2}=\{\alpha\in R':-\alpha\notin R'\}$
\cite[XXII 5.11.3]{SGA3.3r}. We will then show that $\mathcal{F}$
has a splitting $\chi:\mathbb{D}_{T}(\Gamma)\rightarrow G_{T}$.

\subsection{~}

Since $H\subset P_{\mathcal{F}}=\Aut^{\otimes}(\mathcal{F})$, the
$\Gamma$-filtration $\mathcal{F}$ is stable under $H$ and 
\[
\forall\gamma\in\Gamma,\rho\in\Rep^{\circ}(G)(S):\qquad\mathcal{F}^{\gamma}(\rho)=\oplus_{m\in M}\mathcal{F}_{m}^{\gamma}(\rho)
\]
where $\mathcal{F}_{m}^{\gamma}(\rho)$ is the $m$-th eigenspace
of $\mathcal{F}^{\gamma}(\rho)$, viewed as a representation of $H$.
Since $\Lie(U_{\mathcal{F}})=\mathcal{F}_{+}^{0}(\rho_{\mathrm{ad}})$
and $\Lie(P_{\mathcal{F}})=\mathcal{F}^{0}(\rho_{\mathrm{ad}})$ by
theorem~\ref{thm:Stab0IsPar}, $\mathcal{F}_{\alpha}^{\gamma}(\rho_{\mathrm{ad}})=0$
for ($\gamma>0$ and $\alpha\notin R_{2}^{\prime}$) or ($\gamma=0$
and $\alpha\notin R'\cup\{0\}$) while $\mathcal{F}_{\alpha}^{\gamma}(\rho_{\mathrm{ad}})=\mathfrak{g}_{\alpha}$
when $\gamma\leq0$ and $\alpha\in R'\cup\{0\}$. This determines
$\mathcal{F}_{\alpha}^{\gamma}(\rho_{\mathrm{ad}})$ for $\alpha\in R'_{1}\cup\{0\}$:
\[
\forall\alpha\in R'_{1}\cup\{0\}:\quad\mathcal{F}_{\alpha}^{\gamma}(\rho_{\mathrm{ad}})=\begin{cases}
\mathfrak{g}_{\alpha} & \mbox{if }\gamma\leq0,\\
0 & \mbox{if }\gamma>0.
\end{cases}
\]
For the remaining $\alpha$'s (those in $\pm R_{2}^{\prime}$), $\mathfrak{g}_{\alpha}$
is free of rank $1$. Using lemma~\ref{lem:SumsOfFilsOnLF}, we obtain
a partition $T=\coprod T(f)$ into non-empty open and closed subschemes
$T(f)$ of $T$ indexed by certain functions $f:\pm R'_{2}\rightarrow\Gamma$
such that, over $T(f)$, 
\[
\forall\alpha\in\pm R'_{2}:\quad\mathcal{F}_{\alpha}^{\gamma}(\rho_{\mathrm{ad}})=\begin{cases}
\mathfrak{g}_{\alpha} & \mbox{if }\gamma\leq f(\alpha),\\
0 & \mbox{if }\gamma>f(\alpha).
\end{cases}
\]
We extend these functions to $R\cup\{0\}$ by setting $f(R'_{1}\cup\{0\})=0$.
Thus over $T(f)$,
\[
\mathcal{F}^{\gamma}(\rho_{\mathrm{ad}})=\oplus_{\alpha\in R\cup\{0\}:f(\alpha)\geq\gamma}\mathfrak{g}_{\alpha}
\]
Moreover $f(\alpha)>0$ (resp. $<0$) if and only if $\alpha\in R'_{2}$
(resp. $-R'_{2}$).

\subsection{~}

We will establish below that each of these $f$'s extends to a group
homomorphism $f:M\rightarrow\Gamma$. The locally constant function
$T\rightarrow\Hom(M,\Gamma)$ mapping $t\in T(f)$ to $f$ thus defines
a morphism $\chi:\mathbb{D}_{T}(\Gamma)\rightarrow\mathbb{D}_{T}(M)=H\hookrightarrow G_{T}$.
By construction, $\chi$ splits $\mathcal{F}$ on $\rho_{\mathrm{ad}}$,
therefore $\chi$ splits $\mathcal{F}$ everywhere by~\ref{sub:splitonadisenough}.

\subsection{~}

To show that $f$ extends to a group homomorphism $f:M\rightarrow\Gamma$,
we may assume that $T=T(f)=\Spec(k)$ where $k$ is a field. By the
definition of adjoint groups in \cite[XXII 4.3.3]{SGA3.3r} and using
\cite[XXI 3.5.5]{SGA3.3r}, we have to show that 
\begin{enumerate}
\item $f(-\alpha)=-f(\alpha)$ for every $\alpha\in R$ and 
\item $f(\alpha+\beta)=f(\alpha)+f(\beta)$ for every $\alpha,\beta\in R$
such that also $\alpha+\beta\in R$. 
\end{enumerate}

\subsection{~\label{sub:useofK(G,F)=00003D0}}

Since $H\subset P_{\mathcal{F}}=\Aut^{\otimes}(\mathcal{F})$ fixes
$\mathcal{F}$, there is a factorization of $\Gr_{\mathcal{F}}^{\bullet}$:
\[
\xyC{2pc}\xymatrix{\Rep^{\circ}(G)(S)\ar[r] & \Gra^{\Gamma}\Rep^{\circ}(H)(k)\ar[r] & \Gra^{\Gamma}\LF(k)}
\]
where $\Gra^{\Gamma}\Rep^{\circ}(H)(k)$ is the abelian $\otimes$-category
of $\Gamma$-graded objects in $\Rep^{\circ}(H)(k)$. Both functors
are exact $\otimes$-functors, and we thus obtain a factorization
of $\kappa(\mathcal{F})$:
\[
\xymatrix{K_{0}(G)\ar[r]\sp(0.35){\kappa} & K_{0}(H)[\Gamma]=\mathbb{Z}[M][\Gamma]\ar[r] & \mathbb{Z}[\Gamma]}
\]
The morphism $\kappa$ maps the class of $\rho\in\Rep^{\circ}(G)(S)$
to 
\[
\kappa[\rho]={\textstyle \sum_{m,\gamma}}x_{m}^{\gamma}[\rho]\cdot\epsilon^{m}e^{\gamma}
\]
where $\epsilon^{m}\in\mathbb{Z}[M]$ and $e^{\gamma}\in\mathbb{Z}[\Gamma]$
are the basis elements corresponding to $m\in M$ and $\gamma\in\Gamma$
and $x_{m}^{\gamma}[\rho]$ is the dimension of the $m$-th eigenspace
of $\Gr_{\mathcal{F}}^{\gamma}(\rho)$. Thus 
\[
\kappa[\rho_{\mathrm{ad}}]=\left(\dim_{k}(\mathfrak{g}_{0})\cdot\epsilon^{0}+\sum_{\alpha\in R'_{1}}\epsilon^{\alpha}\right)\cdot e^{0}+\sum_{\alpha\in\pm R'_{2}}\epsilon^{\alpha}e^{f(\alpha)}.
\]
Since the above functors are compatible with dualities, 
\begin{eqnarray*}
\kappa[\rho_{\mathrm{ad}}^{\vee}] & = & \left(\dim_{k}(\mathfrak{g}_{0})\cdot\epsilon^{0}+\sum_{\alpha\in R'_{1}}\epsilon^{-\alpha}\right)\cdot e^{0}+\sum_{\alpha\in\pm R'_{2}}\epsilon^{-\alpha}e^{-f(\alpha)}\\
 & = & \left(\dim_{k}(\mathfrak{g}_{0})\cdot\epsilon^{0}+\sum_{\alpha\in R'_{1}}\epsilon^{\alpha}\right)\cdot e^{0}+\sum_{\alpha\in\pm R'_{2}}\epsilon^{\alpha}e^{-f(-\alpha)}.
\end{eqnarray*}
Since $[\rho_{\mathrm{ad}}]=[\rho_{\mathrm{ad}}^{\vee}]$ in $K_{0}(G)$
by lemma~\ref{lem:trivkappaG}, $f(-\alpha)=-f(\alpha)$ for every
$\alpha\in R$.

\subsection{~}

We have already defined the dual $\rho_{n}$ of 
\[
\rho^{n}=\mathrm{Coker}((\rho_{\mathrm{adj}}^{\circ})^{\otimes n+1}\rightarrow\rho_{\mathrm{adj}}^{\circ})
\]
in section~\ref{sub:adjregquo}. These representations act compatibly
(as $n$ varies), functorialy (as $\rho$ varies) and $G$-equivariantly
on any representation $\rho\in\Rep(G)(S)$ by
\[
\xymatrix{\rho_{n}\otimes\rho\ar[r]\sp(0.4){\mathrm{Id}\otimes c_{\rho}} & \rho_{n}\otimes\rho\otimes\rho_{\mathrm{adj}}\ar@{->>}[r]^{\mathrm{Id}\otimes\mathrm{proj}} & \rho_{n}\otimes\rho\otimes\rho_{\mathrm{adj}}^{\circ}\ar@{->>}[r]^{\mathrm{Id}\otimes\mathrm{proj}} & \rho_{n}\otimes\rho\otimes\rho^{n}\ar@{->>}[r]\sp(0.67){\mathrm{eval}_{n}} & \rho}
\]
For $n=1$, we retrieve the usual adjoint $G$-equivariant action
\[
\mathrm{ad}(\rho):\rho_{\mathrm{ad}}\otimes\rho\rightarrow\rho
\]
of $\mathfrak{g}$ on $V(\rho)$, which for $\rho=\rho_{\mathrm{ad}}$
is nothing but the usual Lie bracket
\[
[-,-]:\rho_{\mathrm{ad}}\otimes\rho_{\mathrm{ad}}\rightarrow\rho_{\mathrm{ad}}.
\]
We also denote by $[-,-]:\rho_{n}\otimes\rho_{\mathrm{ad}}\rightarrow\rho_{\mathrm{ad}}$
the above actions on $\rho_{\mathrm{ad}}$. Thus 
\[
\forall\gamma\in\Gamma,\,\forall\alpha,\beta\in M:\qquad[\mathcal{F}_{\alpha}^{\gamma}(\rho_{n}),\mathfrak{g}_{\beta}]\subset\mathcal{F}_{\alpha+\beta}^{\gamma+f(\beta)}(\rho_{\mathrm{ad}}).
\]
In particular, $[\mathcal{F}_{\alpha}^{\gamma}(\rho_{n}),\mathfrak{g}_{\beta}]\neq0$
implies $\alpha+\beta,\beta\in R\cup\{0\}$ and 
\[
f(\alpha+\beta)\geq\gamma+f(\beta).
\]

\subsection{~}

Suppose that $\alpha$, $\beta$ and $\alpha+\beta$ all belong to
$R$, with $\ell(\alpha)\leq\ell(\beta)$ where $\ell$ is the length.
Let $q$ and $p$ be the positive integers (with $2\leq p+q\leq4$)
such that 
\[
\left\{ \beta+n\alpha\in R:n\in\mathbb{Z}\right\} =\left\{ \beta-(p-1)\alpha,\cdots,\beta,\beta+\alpha,\cdots,\beta+q\alpha\right\} 
\]
see~\cite[XXI 2.3.5 and 1]{SGA3.3r}. By Chevalley's rule \cite[XXIII 6.5]{SGA3.3r},
\[
[\mathfrak{g}_{\alpha},\mathfrak{g}_{\beta}]=p\mathfrak{g}_{\alpha+\beta}\quad\mbox{and}\quad[\mathfrak{g}_{-\alpha},\mathfrak{g}_{-\beta}]=p\mathfrak{g}_{-\alpha-\beta}.
\]
Thus if $p\neq0$ in $k$, $[\mathfrak{g}_{\alpha},\mathfrak{g}_{\beta}]\neq0$
and $[\mathfrak{g}_{-\alpha},\mathfrak{g}_{-\beta}]\neq0$, therefore
\[
f(\alpha+\beta)\geq f(\alpha)+f(\beta)\quad\mbox{and}\quad f(-\alpha-\beta)\geq f(-\alpha)+f(-\beta)
\]
which implies $(2)$ by $(1)$, i.e. 
\[
f(\alpha+\beta)=f(\alpha)+f(\beta).
\]
If $q=1$, Chevalley's rule gives $[\mathfrak{g}_{\alpha},\mathfrak{g}_{-\alpha-\beta}]\neq0$
and $[\mathfrak{g}_{-\alpha},\mathfrak{g}_{\alpha+\beta}]\neq0$,
thus again $f(\alpha+\beta)=f(\alpha)+f(\beta)$. This leaves a single
case: $p=q=2=\mathrm{char}(k)$, where the same method already gives
$f(\beta)=f(\beta-\alpha)+f(\alpha)$. We will see below that also
\[
[\mathcal{F}_{2\alpha}^{2f(\alpha)}(\rho_{2}),\mathfrak{g}_{\beta-\alpha}]=\mathfrak{g}_{\alpha+\beta}\quad\mbox{and}\quad[\mathcal{F}_{-2\alpha}^{-2f(\alpha)}(\rho_{2}),\mathfrak{g}_{\alpha+\beta}]=\mathfrak{g}_{\beta-\alpha}.
\]
Therefore $f(\alpha+\beta)=2f(\alpha)+f(\beta-\alpha)$, thus again
$f(\alpha+\beta)=f(\alpha)+f(\beta)$.

\subsection{~\label{sub:patchG2car21}}

The pure short exact sequences of finite locally free sheaves on $S$
\[
\begin{array}{ccccccccc}
0 & \rightarrow & \Sym_{\mathcal{O}_{S}}^{2}\left(\frac{\mathcal{I}(G)}{\mathcal{I}(G)^{2}}\right) & \rightarrow & \frac{\mathcal{I}(G)}{\mathcal{I}(G)^{3}} & \rightarrow & \frac{\mathcal{I}(G)}{\mathcal{I}(G)^{2}} & \rightarrow & 0\\
0 & \rightarrow & \mathrm{ker} & \rightarrow & \left(\frac{\mathcal{I}(G)}{\mathcal{I}(G)^{2}}\right)^{\otimes2} & \rightarrow & \Sym_{\mathcal{O}_{S}}^{2}\left(\frac{\mathcal{I}(G)}{\mathcal{I}(G)^{2}}\right) & \rightarrow & 0
\end{array}
\]
give rise to pure short exact sequences in $\Rep^{\circ}(G)(S)$ which
dualize to
\[
\begin{array}{ccccccccc}
0 & \rightarrow & \rho_{\mathrm{ad}} & \rightarrow & \rho_{2} & \rightarrow & \Gamma^{2}(\rho_{\mathrm{ad}}) & \rightarrow & 0\\
0 & \rightarrow & \Gamma^{2}(\rho_{\mathrm{ad}}) & \rightarrow & \rho_{\mathrm{ad}}^{\otimes2} & \rightarrow & \Lambda^{2}(\rho_{\mathrm{ad}}) & \rightarrow & 0
\end{array}
\]
where $\Gamma^{2}(\rho)=\Sym^{2}(\rho^{\vee})^{\vee}=\ker\left(\rho^{\otimes2}\rightarrow\Lambda^{2}(\rho)\right)$.
Therefore
\[
[\rho_{2}]=[\rho_{\mathrm{ad}}]+[\rho_{\mathrm{ad}}]^{2}-[\Lambda^{2}(\rho_{\mathrm{ad}})]\quad\mbox{in}\quad K_{0}(G).
\]
Since $\mathfrak{g}_{2\alpha}=0=\Lambda^{2}(\mathfrak{g})_{2\alpha}$,
the coefficients of $\epsilon^{2\alpha}$ in $\kappa[\rho_{2}]$ and
$\kappa[\rho_{\mathrm{ad}}^{\otimes2}]=\kappa[\rho_{\mathrm{ad}}]^{2}$
are both equal to $e^{2f(\alpha)}$. Thus if $\mathfrak{d}=\oplus\mathfrak{d}_{m}$
is the weight decomposition of $\mathfrak{d}=\omega_{k}^{\circ}(\rho_{2})$,
then $\mathfrak{d}_{2\alpha}$ is $1$-dimensional and contained in
$\mathcal{F}^{\gamma}(\rho_{2})$ if and only if $\gamma\leq2f(\alpha)$.
In particular, $\mathcal{F}_{2\alpha}^{2f(\alpha)}(\rho_{2})=\mathfrak{d}_{2\alpha}$,
and similarly for $-\alpha$. We thus want: 
\[
[\mathfrak{d}_{2\alpha},\mathfrak{g}_{\beta-\alpha}]=\mathfrak{g}_{\beta+\alpha}\quad\mbox{and}\quad[\mathfrak{d}_{-2\alpha},\mathfrak{g}_{\beta+\alpha}]=\mathfrak{g}_{\beta-\alpha}.
\]

\subsection{~\label{sub:PatchG2Char2.2}}

This now only involves the split group $G_{k}$ and its épinglage,
all of which descends to $\Spec(\mathbb{Z})$ by~\cite[XXIII 5.1 and XXV 1.2]{SGA3.3r}.
We may thus assume that $G$ and $\mathcal{E}=(H,M,R,\cdots)$ are
defined over $S=\Spec(\mathbb{Z})$. The épinglage comes along with
simple roots $\Delta\subset R$ and, for each $\alpha\in R$, a basis
$X_{\alpha}$ of $\mathfrak{g}_{\alpha}$, which extends to a Chevalley
system $\{X_{\alpha}:\alpha\in R\}$ by \cite[XXIII 6.2]{SGA3.3r},
giving rise to isomorphisms $u_{\alpha}(t)=\exp(tX_{\alpha})$ from
$\mathbb{G}_{a}=\Spec(\mathbb{Z}[t])$ to the root subgroup $U_{\alpha}$
of $\alpha\in R$. As a linear form on $\mathcal{I}(G)/I(G)^{2}$,
$X_{\alpha}$ is the composition of $u_{\alpha}^{\natural}:\mathcal{I}(G)\rightarrow\mathcal{I}(\mathbb{G}_{a})$
with the linear form on $\mathcal{I}(\mathbb{G}_{a})=t\mathbb{Z}[t]$
given by the coefficient of $t$. If instead we take the coefficient
of $t^{2}$, we obtain a linear form on $\mathcal{I}(G)/\mathcal{I}(G)^{3}$
which is a basis $X_{2\alpha}$ of $\mathfrak{d}_{2\alpha}$. The
action of $X_{\alpha}$ on the regular representation is given by
\[
\mathcal{A}(G)\rightarrow\mathcal{A}(G\times\mathbb{G}_{a})=\mathcal{A}(G)[t]\rightarrow\mathcal{A}(G)
\]
where the first map takes $f$ in $\mathcal{A}(G)$ to the function
$(g,t)\mapsto f\left(u_{\alpha}(t)gu_{\alpha}^{-1}(t)\right)$, and
the second takes the coefficient of $t$ (or evaluates $\frac{d}{dt}$
at $t=0$). The action of $X_{2\alpha}$ is obtained by replacing
the second map with the coefficient of $t^{2}$, thus $2X_{2\alpha}=X_{\alpha}^{2}$
on $\rho_{\mathrm{reg}}$, therefore $2X_{2\alpha}=X_{\alpha}^{2}$
on all $\rho$'s. Let us now return to our chain of roots 
\[
\{\beta-\alpha,\beta,\beta+\alpha,\beta+2\alpha\}\subset R.
\]
By Chevalley's rule \cite[XXIII 6.5]{SGA3.3r} 
\[
[X_{\alpha},X_{\beta-\alpha}]=\pm X_{\beta}\quad\mbox{and}\quad[X_{\alpha},X_{\beta}]=\pm2X_{\beta+\alpha}.
\]
Therefore $[X_{2\alpha},X_{\beta-\alpha}]=\pm X_{\beta+\alpha}$ since
(we are now over $\mathbb{Z}$!) 
\[
2[X_{2\alpha},X_{\beta-\alpha}]=[2X_{2\alpha},X_{\beta-\alpha}]=[X_{\alpha},[X_{\alpha},X_{\beta-\alpha}]]=\pm2X_{\beta+\alpha}.
\]
Similarly, $[X_{-2\alpha},X_{\beta+\alpha}]=\pm X_{\beta-\alpha}$,
and this completes our proof.

\section{Consequences\label{sec:ConsequencesTan}}

Let $G$ be a reductive group over $S$.

\subsection{Proof of theorem~\ref{thm:MainTan}}

The assertions concerning automorphisms and $\Gamma$-graduations
follow from theorem~\ref{thm:RecThm} and propositions~\ref{Pro:Aut(w0)G(w0)F(w0)},
\ref{Pro:Aut(V0)G(V0)F(V0)} and \ref{Pro:Wedhorn}. If $G$ is an
isotrivial reductive group over a quasi-compact $S$, we have monomorphisms
\[
\xyC{4pc}\xyR{0.8pc}\xymatrix{ &  & \mathbb{F}^{\Gamma}(V^{\circ})\ar@{^{(}->}[dr]^{\mathrm{Prop.\,}\ref{Pro:F(V0)2F(w0)inj}}\\
\mathbb{F}^{\Gamma}(G)\ar@{^{(}->}[r]^{\mathrm{Cor.\,\ref{cor:F(G)same}}} & \mathbb{F}^{\Gamma}(V)\ar@{^{(}->}[ur]^{\mathrm{Prop.\,}\ref{Pro:Aut(V0)G(V0)F(V0)}}\ar@{^{(}->}[dr]_{\ref{sub:injofVtoomega}} &  & \mathbb{F}^{\Gamma}(\omega^{\circ})\\
 &  & \mathbb{F}^{\Gamma}(\omega)\ar@{^{(}->}[ur]_{\mathrm{Prop.\,}\ref{Pro:Aut(w0)G(w0)F(w0)}}
}
\]
and we have just seen that $\mathbb{G}^{\Gamma}(G)\rightarrow\mathbb{F}^{\Gamma}(G)\rightarrow\mathbb{F}^{\Gamma}(\omega^{\circ})$
is an epimorphism, therefore
\[
\mathbb{F}^{\Gamma}(G)=\mathbb{F}^{\Gamma}(V)=\mathbb{F}^{\Gamma}(V^{\circ})=\mathbb{F}^{\Gamma}(\omega)=\mathbb{F}^{\Gamma}(\omega^{\circ})
\]
in this case, from which easily follows that also
\[
\mathbb{F}^{\Gamma}(G)=\mathbb{F}^{\Gamma}(V)=\mathbb{F}^{\Gamma}(V^{\circ})
\]
for any reductive group over any $S$ -- and this is contained in
$\mathbb{F}^{\Gamma}(\omega)$ by \ref{sub:injofVtoomega}.

\subsection{~}

Since the $S$-scheme $\mathbb{G}^{\Gamma}(G)$ and $\mathbb{F}^{\Gamma}(G)$
of chapter~\ref{Chapter:GroupThForm} represent the functors indicated
in theorem~\ref{thm:MainTan}, there is a universal $\Gamma$-graduation
$\mathcal{G}_{\mathrm{univ}}$ on $V_{\mathbb{G}^{\Gamma}(G)}$ (inducing
universal $\Gamma$-graduations on $V_{\mathbb{G}^{\Gamma}(G)}^{\circ}$,
$\omega_{\mathbb{G}^{\Gamma}(G)}$ and $\omega_{\mathbb{G}^{\Gamma}(G)}^{\circ}$)
and a universal $\Gamma$-filtration $\mathcal{F}_{\mathrm{univ}}$
on $V_{\mathbb{F}^{\Gamma}(G)}$ (inducing universal $\Gamma$-filtrations
on $V_{\mathbb{F}^{\Gamma}(G)}^{\circ}$, $\omega_{\mathbb{F}^{\Gamma}(G)}$
and $\omega_{\mathbb{F}^{\Gamma}(G)}^{\circ}$) from which all other
$\Gamma$-graduations or $\Gamma$-filtrations over some base $T$
can be retrieved by pull-back through unique morphisms $T\rightarrow\mathbb{G}^{\Gamma}(G)$
or $T\rightarrow\mathbb{F}^{\Gamma}(G)$ -- for the $\omega$ or $\omega^{\circ}$
variants, we have to assume that $G$ is isotrivial and $S$ quasi-compact,
or that the $\Gamma$-graduations or $\Gamma$-filtrations (over $T$)
extend to $V$ or $V^{\circ}$. The $S$-scheme $\mathbb{C}^{\Gamma}(G)$
is a coarse moduli scheme for either $\Gamma$-graduations or $\Gamma$-filtrations
(on the various fiber functors): two such objects (over $T$) are
fpqc locally (on $T$) isomorphic if and only if the induced morphisms
$T\rightarrow\mathbb{C}^{\Gamma}(G)$ are equal.

\subsection{~}

From this perspective, we may either deduce non-trivial properties
of the $S$-schemes constructed in chapter~\ref{Chapter:GroupThForm}
from easier properties of $\Gamma$-graduations and $\Gamma$-filtrations,
or non-trivial properties of the latter from already established properties
of the former. For instance, theorem~\ref{thm:MainTan} implies that
$\Gamma$-filtrations split over affine bases, a strengthening of
the splitting results that we have established:
\begin{cor}
\label{cor:SplitOnAffine}Suppose that $S$ is affine. Then every
$\Gamma$-filtration $\mathcal{F}$ on $V_{S}$ or $V_{S}^{\circ}$
splits over $S$, and so do the $\Gamma$-filtrations on $\omega_{S}$
or $\omega_{S}^{\circ}$ if $G$ is isotrivial. \end{cor}
\begin{proof}
This follows from~\cite[XXVI 2.2]{SGA3.3r} as in section~\ref{sub:DescParEqConj}.
\end{proof}

\subsection{~}

In the other direction, we obtain the expected functoriality.
\begin{cor}
\label{cor:FunctFirstLine}The fundamental sequence of section~\ref{sub:DefFondDiag}
\[
\mathbb{G}^{\Gamma}(G)\stackrel{\Fi}{\longrightarrow}\mathbb{F}^{\Gamma}(G)\stackrel{t}{\longrightarrow}\mathbb{C}^{\Gamma}(G)
\]
is covariantly functorial on the fibered category of reductive groups
over schemes and covariantly functorial in the totally ordered commutative
group $\Gamma$.\end{cor}
\begin{proof}
We have to show that for a morphism $\varphi:G_{1}\rightarrow f^{\ast}G_{2}$
over $f:T_{1}\rightarrow T_{2}$ in the former category, there is
a canonical commutative diagram of schemes
\[
\xyR{2pc}\xyC{3pc}\xymatrix{\mathbb{G}^{\Gamma}(G_{1})\ar[r]^{\Fi}\ar[d]^{\varphi} & \mathbb{F}^{\Gamma}(G_{1})\ar[r]^{t}\ar[d]^{\varphi} & \mathbb{C}^{\Gamma}(G_{1})\ar[r]^{\mathrm{struct}}\ar[d]^{\varphi} & T_{1}\ar[d]^{f}\\
\mathbb{G}^{\Gamma}(G_{2})\ar[r]^{\Fi} & \mathbb{F}^{\Gamma}(G_{2})\ar[r]^{t} & \mathbb{C}^{\Gamma}(G_{2})\ar[r]^{\mathrm{struct}} & T_{2}
}
\]
In the Tannakian point of view, the first two vertical morphisms are
induced by pre-composition with the restriction functor $\Rep(f^{\ast}G_{2})\rightarrow\Rep(G_{1})$
which maps $\tau$ to $\tau\circ\varphi$. For the third one: if $T$
is a $T_{1}$-scheme and $x$ is a $T$-valued point of $\mathbb{C}^{\Gamma}(G_{1})$,
it lifts to a $\Gamma$-filtration over an fpqc covering $\{T_{i}\rightarrow T\}$
of $T$, and two such lifts become isomorphic over a common refinement
of the corresponding fpqc coverings. The image of these lifts in $\mathbb{F}^{\Gamma}(G_{2})$
thus yield a well-defined morphism $\varphi(x):T\rightarrow\mathbb{C}^{\Gamma}(G_{2})$,
and this defines the morphism $\varphi:\mathbb{C}^{\Gamma}(G_{1})\rightarrow\mathbb{C}^{\Gamma}(G_{2})$.
The proof of the covariance in $\Gamma$ is similar, using post-composition
with the morphisms
\[
\xymatrix{\Gra^{\Gamma_{1}}\QCoh\ar[r]^{\Fi}\ar[d]_{f} & \Fil^{\Gamma_{1}}\QCoh\ar[d]_{f}\\
\Gra^{\Gamma_{2}}\QCoh\ar[r]^{\Fi} & \Fil^{\Gamma_{2}}\QCoh
}
\]
of fpqc stacks induced by $f:(\Gamma_{1},+,\leq)\rightarrow(\Gamma_{2},+,\leq)$,
which are given by 
\[
f(\mathcal{G})_{\gamma_{2}}=\oplus_{f(\gamma_{1})=\gamma_{2}}\mathcal{G}_{\gamma_{1}}\quad\mbox{and}\quad f(\mathcal{F})^{\gamma_{2}}={\textstyle \sum_{f(\gamma_{1})\geq\gamma_{2}}}\mathcal{F}^{\gamma_{1}}
\]
for $T$ over $S$, $\mathcal{G}\in\Gra^{\Gamma_{1}}\QCoh(T)$, $\mathcal{F}\in\Fil^{\Gamma_{1}}\QCoh(T)$
and $\gamma_{2}\in\Gamma_{2}$. \end{proof}
\begin{cor}
\label{cor:GL(V)case}If $G=GL(\mathcal{V})$ for some $\mathcal{V}\in\LF(S)$
of rank $r\in\mathbb{N}^{\times}$, evaluation at the tautological
representation $\tau$ of $G$ on $\mathcal{V}$ identifies 
\[
\xymatrix{\mathbb{G}^{\Gamma}(G)\ar[r]^{\Fi} & \mathbb{F}^{\Gamma}(G)\ar[r]^{t} & \mathbb{C}^{\Gamma}(G)}
\]
with 
\[
\xymatrix{\mathbb{G}^{\Gamma}(\mathcal{V})\ar[r]^{\Fi} & \mathbb{F}^{\Gamma}(\mathcal{V})\ar[r]^{t} & \mathbb{C}^{\Gamma}(\mathcal{V})}
\]
where for any $S$-scheme $T$, 
\begin{eqnarray*}
\mathbb{G}^{\Gamma}(\mathcal{V})(T) & = & \left\{ \Gamma-\mbox{graduations on }\mathcal{V}_{T}\right\} \\
\mathbb{F}^{\Gamma}(\mathcal{V})(T) & = & \left\{ \Gamma-\mbox{filtrations on }\mathcal{V}_{T}\right\} \\
\mathbb{C}^{\Gamma}(\mathcal{V})(T) & = & \left\{ \mbox{locally constant functions }f:T\rightarrow\Gamma_{\geq}^{r}\right\} 
\end{eqnarray*}
where $\Gamma_{\geq}^{r}=\{(\gamma_{1}\geq\cdots\geq\gamma_{r})\in\Gamma^{r}\}$
and $t$ sends a $\Gamma$-filtration $\mathcal{F}$ on $\mathcal{V}_{T}$
to the function which maps $x\in T$ to the $r$-tuple with $\dim_{k(x)}\Gr_{\mathcal{F}}^{\gamma}(x)$
copies of $\gamma\in\Gamma$. \end{cor}
\begin{proof}
Evaluation at $\tau$ gives the morphisms $\tau_{g}$, $\tau_{f}$
of the diagram
\[
\xymatrix{\mathbb{G}^{\Gamma}(G)\ar[r]^{\Fi}\ar[d]_{\tau_{g}} & \mathbb{F}^{\Gamma}(G)\ar[r]^{t}\ar[d]_{\tau_{f}} & \mathbb{C}^{\Gamma}(G)\ar[d]_{\tau_{c}}\\
\mathbb{G}^{\Gamma}(\mathcal{V})\ar[r]^{\Fi} & \mathbb{F}^{\Gamma}(\mathcal{V})\ar[r]^{t} & \mathbb{C}^{\Gamma}(\mathcal{V})
}
\]
and the remaining morphism $\tau_{c}$ comes along by noting that
$t\circ\tau_{f}$ is $G$-invariant. Plainly, $\tau_{g}$ is an isomorphism:
a morphism $\mathbb{D}_{T}(\Gamma)\rightarrow G_{T}$ is nothing but
a representation of $\mathbb{D}_{T}(\Gamma)$ on $\mathcal{V}_{T}$,
i.e.~a $\Gamma$-graduation on $\mathcal{V}_{T}$. Since every $\Gamma$-filtration
on $\mathcal{V}_{T}$ splits locally for the fpqc topology on $T$
by definition (and locally for the Zariski topology by~proposition~\ref{prop:FilonLF}),
$\Fi:\mathbb{G}^{\Gamma}(\mathcal{V})\rightarrow\mathbb{F}^{\Gamma}(\mathcal{V})$
is an epimorphism of fpqc sheaves on $\Sch/S$, and so is therefore
also $\tau_{f}$. If $\mathcal{F}_{1},\mathcal{F}_{2}\in\mathbb{F}^{\Gamma}(G)(T)$
induce the same filtration $\mathcal{F}_{1}(\tau)=\mathcal{F}_{2}(\tau)$
on $\mathcal{V}_{T}=V_{T}(\tau)$, they also have the same image at
$\det\tau$ (a quotient of $\tau^{\otimes r}$) and $\tau'=\tau\oplus\det(\tau)^{-1}$.
Arguing as in~section~\ref{sub:splitonadisenough}, we obtain that
both filtrations agree on $\rho_{\mathrm{reg}}$, thus actually $\mathcal{F}_{1}=\mathcal{F}_{2}$.
It follows that $\tau_{f}$ is also a monomorphism, i.e.~it is an
isomorphism. One checks easily that $t:\mathbb{F}^{\Gamma}(\mathcal{V})\rightarrow\mathbb{C}^{\Gamma}(\mathcal{V})$
is an epimorphism of fpqc sheaves on $\Sch/S$, and so is therefore
also $\mathbb{C}^{\Gamma}(G)\rightarrow\mathbb{C}^{\Gamma}(\mathcal{V})$.
If $x,y\in\mathbb{C}^{\Gamma}(G)(T)$ have the same image in $\mathbb{C}^{\Gamma}(\mathcal{V})(T)$
and $\mathcal{X},\mathcal{Y}$ are chosen lifts of $x$ and $y$ to
$\mathbb{G}^{\Gamma}(G)(T')$ for some fpqc cover $T'\rightarrow T$,
then, locally on $T'$, $\mathcal{X}_{\gamma}(\tau)$ and $\mathcal{Y}_{\gamma}(\tau)$
are free with the same rank, thus isomorphic. Gluing these isomorphisms,
we obtain a $g\in G(T')$ which maps $\mathcal{X}$ to $\mathcal{Y}$,
thus $x_{T'}=t\circ\Fi(\mathcal{X})=t\circ\Fi(\mathcal{Y})=y_{T'}$
in $\mathbb{C}^{\Gamma}(G)(T')$ by section~\ref{sub:DescParEqConj}.
But then $x=y$ in $\mathbb{C}^{\Gamma}(G)(T)$, therefore $\tau_{c}$
is also a monomorphism, i.e. it is an isomorphism.\end{proof}
\begin{rem}
For $G=GL(\mathcal{V})$ as above, the weak and strong dominance orders
on $\mathbb{C}^{\Gamma}(G)$ are equal. They correspond to the following
order on $\mathbb{C}^{\Gamma}(\mathcal{V})$: for an $S$-scheme $T$
and locally constant functions $f_{1},f_{2}:T\rightarrow\Gamma_{\geq}^{r}$,
we have 
\[
f_{1}\leq f_{2}\quad\mbox{in}\quad\mathbb{C}^{\Gamma}(\mathcal{V})(T)\quad\iff\quad\forall t\in T:\quad f_{1}(t)\leq f_{2}(t)\quad\mbox{in}\quad\Gamma_{\geq}^{r}
\]
for the usual partial dominance order on $\Gamma_{\geq}^{r}$, given
by
\begin{multline*}
\qquad\qquad(\gamma_{1}\geq\cdots\geq\gamma_{r})\leq(\gamma'_{1}\geq\cdots\geq\gamma'_{r})\\
\iff\quad\begin{cases}
\forall1\leq i\leq r-1: & \gamma_{1}+\cdots+\gamma_{i}\leq\gamma'_{1}+\cdots+\gamma'_{i},\\
\mbox{and} & \gamma_{1}+\cdots+\gamma_{r}=\gamma'_{1}+\cdots+\gamma'_{r}.
\end{cases}\qquad\qquad
\end{multline*}
For a connected $T$, we will usually identify $\mathbb{C}^{\Gamma}(\mathcal{V})(T)$
and $\Gamma_{\geq}^{r}$.
\end{rem}

\subsection{~\label{sub:FunctWeakDom}}

We have already mentioned that the monoid structure on $\mathbb{C}^{\Gamma}(G)$
is not functorial in $G$. On the other hand, the weak dominance partial
order $\leq$ on $\mathbb{C}^{\Gamma}(G)$ defined in section~\ref{sub:C(G)_dominorder}
is functorial in $G$. 
\begin{prop}
\label{prop:FunctorWeakOrder}Let $\varphi:G\rightarrow H$ be a morphism
of reductive group over $S$, let $T$ be an $S$-scheme. Then for
any $t_{1},t_{2}\in\mathbb{C}^{\Gamma}(G)(T)$, 
\[
t_{1}\leq t_{2}\mbox{ in }\mathbb{C}^{\Gamma}(G)(T)\quad\Longrightarrow\quad\varphi(t_{1})\leq\varphi(t_{2})\mbox{ in }\mathbb{C}^{\Gamma}(H)(T).
\]
In particular for any $\tau:G\rightarrow GL(\mathcal{V})$ in $\Rep^{\circ}(G)(S)$,
with $\mathcal{V}=V(\tau)\in\LF(S)$, 
\[
t_{1}\leq t_{2}\mbox{ in }\mathbb{C}^{\Gamma}(G)(T)\quad\Longrightarrow\quad t_{1}(\tau)\leq t_{2}(\tau)\mbox{ in }\mathbb{C}^{\Gamma}(\mathcal{V})(T).
\]
\end{prop}
\begin{proof}
Since $\leq$ is open in $\mathbb{C}^{\Gamma}(H)$, we may assume
that $T$ is a geometric point. The proposition then follows from
the stronger proposition~\ref{prop:CaractTanDomOrder} below.
\end{proof}

\subsection{~\label{sub:CarWeakDomTan}}

Suppose that $G$ is isotrivial over a connected $S$. Then $G$ is
split by a single finite étale cover $\pi:S'\rightarrow S$, and we
may assume that $S'$ is connected and Galois over $S$ with Galois
group $\Theta=\Aut(S'/S)$. Let $\mathscr{R}=\mathscr{R}(G)=(M,R,M^{\ast},R^{\ast})$
be the constant type of $G$ \cite[XXII 6.8]{SGA3.3r} with Weyl group
$W=W(\mathscr{R})$ and fix a system of positive roots $R_{+}\subset R$,
giving rise to a based root data $\mathscr{R}_{+}=(M,R,M^{\ast},R^{\ast};\Delta)$.
Let $G_{0}=G_{\Spec(\mathbb{Z})}^{Ep}(\mathscr{R}_{+})$ be the corresponding
pinned Chevalley group over $\Spec(\mathbb{Z})$ \cite[XXV 1.2]{SGA3.3r}
and pick an isomorphism $\gamma:G_{0,S'}\simeq G_{S'}$. It exists
by \cite[XXIII 1.1]{SGA3.3r} and corresponds to a pinning $\mathscr{E}=(T,\iota:\mathbb{D}_{S'}(M)\stackrel{\simeq}{\rightarrow}T,(X_{\alpha})_{\alpha\in\Delta})$
of type $\mathscr{R}_{+}$ of $G_{S'}$ \cite[XXIV 1.0]{SGA3.3r}
by \cite[XXIV 1.20]{SGA3.3r}. For any $\theta$ in $\Theta$, the
pull-back $\theta^{\ast}\gamma$ is another isomorphism $G_{0,S'}\simeq G_{S'}$,
corresponding to the pinning $\theta^{\ast}\mathscr{E}=(\theta^{\ast}T,\theta^{\ast}\iota,(\theta^{\ast}X_{\alpha})_{\alpha\in\Delta})$
of type $\mathscr{R}_{+}$ of $G_{S'}$. By~\cite[XXIV 1.5]{SGA3.3r},
there is a unique inner automorphism $u_{\theta}$ of $G_{S'}$ mapping
$\theta^{\ast}\mathscr{E}$ back to $\mathscr{E}$, inducing an automorphism
$v_{\theta}$ of $\mathscr{R}$. This defines an action of $\Theta$
on $\mathscr{R}_{+}$, analogous to the twisted action considered
in section~\ref{sub:DescTwistedActionRelRoots}, which itself induces
compatible actions of $\Theta$ on $M_{d}$, $\Hom^{+}(M,\Gamma)$,
$\mathbb{N}[M]^{W}$\ldots{}

For $\tau\in\Rep^{\circ}(G)(S')$, we may use our fixed pinning $\mathscr{E}$
to view the restriction of $\tau$ to the maximal torus $T$ of $G_{S'}$
as a representation of $\mathbb{D}_{S'}(M)$. We denote by $ch_{\mathscr{E}}(\tau)$
the corresponding element of $\mathbb{N}[M]^{W}$, i.e.~$ch_{\mathscr{E}}(\tau)=\sum\rank V(\tau)_{m}\cdot e^{m}$
where $e^{m}$ is the basis element of $\mathbb{Z}[M]$ corresponding
to $m\in M$ and $V(\tau)_{m}$ is the $m$-th eigenspace of $\tau\vert T\circ\iota$.
For any $\theta$ in $\Theta$, the pull-back $\theta^{\ast}\tau$
also belongs to $\Rep^{\circ}(G)(S')$ and plainly $ch_{\theta^{\ast}\mathscr{E}}(\theta^{\ast}\tau)=ch_{\mathscr{E}}(\tau)$.
On the other hand 
\[
ch_{\mathscr{E}}(\tau)=ch_{\mathscr{E}}(\tau\circ u_{\theta})=v_{\theta}\left(ch_{\sigma^{\ast}\mathscr{E}}(\tau)\right)\quad\mbox{in}\quad\mathbb{N}[M]^{W},
\]
for every $\tau$, thus $ch_{\mathscr{E}}(\theta^{\ast}\tau)=\theta\cdot ch_{\mathscr{E}}(\tau)$
in $\mathbb{N}[M]^{W}$. In particular $ch_{\mathscr{E}}(\tau)$ is
fixed by $\Theta$ if $\theta^{\ast}\tau\vert T\simeq\tau\vert T$,
for instance if $\tau$ comes from a representation in $\Rep^{\circ}(G)(S)$. 

Let $\tau_{0,\lambda,\mathbb{Q}}\in\Rep^{\circ}(G_{0})(\mathbb{Q})$
be the irreducible representation of $G_{0,\mathbb{Q}}$ with highest
weight $\lambda\in M_{d}$ \cite[Lemme 5]{Se68b}, let $\tau_{0,\lambda}\in\Rep^{\circ}(G_{0})(\mathbb{Z})$
be any extension of $\tau_{0,\lambda,\mathbb{Q}}$ to a representation
of $G_{0}$ \cite[Lemme 2]{Se68b}, let $\tau'_{\lambda}\in\Rep^{\circ}(G)(S')$
be the corresponding representation of $G_{S'}$ and set $\tau_{\lambda}=\pi_{\ast}\tau'_{\lambda}\in\Rep^{\circ}(G)(S)$.
Then
\[
\tau_{\lambda,S'}=\pi^{\ast}\tau_{\lambda}\simeq\oplus_{\theta\in\Theta}\theta^{\ast}\tau'_{\lambda}\quad\mbox{in}\quad\Rep^{\circ}(G)(S').
\]
Thus $ch_{\mathscr{E}}(\tau_{\lambda,S'})=\sum_{\theta\in\Theta}\theta\cdot ch_{\mathscr{E}}(\tau'_{\lambda})$
with $ch_{\mathscr{E}}(\tau'_{\lambda})=ch_{\mathscr{E}_{0}}(\tau_{0,\lambda})=ch_{\mathscr{E}_{0,\mathbb{Q}}}(\tau_{0,\lambda,\mathbb{Q}})$
in $\mathbb{N}[M]^{W}$, where $\mathscr{E}_{0}=(T_{0},\iota_{0},\cdots)$
is the pinning of $G_{0}$. Since the other weights of $\tau_{0,\lambda,\mathbb{Q}}$
are contained in $\lambda-\mathbb{N}\cdot R_{+}$, it follows that
for any $f\in\Hom^{+}(M,\Gamma)$, 
\[
\max f\left(ch_{\mathscr{E}}(\tau_{\lambda,S'})\right)=\max\left\{ f(\theta\cdot\lambda):\theta\in\Theta\right\} =\max\left\{ (\theta\cdot f)(\lambda):\theta\in\Theta\right\} .
\]

Our fixed pinning $\mathscr{E}$ also induces an isomorphism of partially
ordered commutative $S'$-monoid between $\mathbb{C}^{\Gamma}(G_{S'})$
and $\Hom^{+}(M,\Gamma)_{S'}$, and the resulting isomorphism $\mathbb{C}^{\Gamma}(G)(S')\simeq\Hom^{+}(M,\Gamma)$
is $\Theta$-equivariant, cf.~section~\ref{sub:C(G)_splitcase}.
If $t\in\mathbb{C}^{\Gamma}(G)(S')$ maps to $t_{\mathscr{E}}:M\rightarrow\Gamma$,
then for every $\tau\in\Rep^{\circ}(G)(S')$, we have 
\[
t(\tau)=t_{\mathscr{E}}(ch_{\mathscr{E}}(\tau))\quad\mbox{in}\quad\Gamma_{\geq}^{r(\tau)}\subset\mathbb{N}[\Gamma]
\]
under the natural identification of $\Gamma_{\geq}^{r(\tau)}$ with
the subset of $\mathbb{N}[\Gamma]$ made of those elements which have
degree $r(\tau)=\rank V(\tau)$ (if $\tau=0$, we set $\Gamma_{\geq}^{r(\tau)}=0$),
thus also
\[
\max t(\tau)=\max t_{\mathscr{E}}(ch_{\mathscr{E}}(\tau))\quad\mbox{in}\quad\Gamma.
\]
For $\tau=\tau_{\lambda,S'}$ as above we therefore obtain
\[
\max t(\tau_{\lambda})=\max t(\tau_{\lambda,S'})=\max\left\{ t_{\mathscr{E}}(\theta\cdot\lambda):\theta\in\Theta\right\} =\max\left\{ (\theta\cdot t_{\mathscr{E}})(\lambda):\theta\in\Theta\right\} .
\]
If moreover $t$ belongs to $\mathbb{C}^{\Gamma}(G)(S)$, $\theta\cdot t_{\mathscr{E}}=t_{\mathscr{E}}$
for all $\theta\in\Theta$, thus 
\[
\max t(\tau_{\lambda})=t_{\mathscr{E}}(\lambda)\quad\mbox{in}\quad\Gamma.
\]

\subsection{~}

We may now prove the following strenghtening of Proposition~\ref{prop:FunctorWeakOrder}.
\begin{prop}
\label{prop:CaractTanDomOrder}Suppose that $G$ is isotrivial over
a connected base scheme $S$. Then for every $t_{1},t_{2}\in\mathbb{C}^{\Gamma}(G)(S)$,
the following conditions are equivalent:
\begin{enumerate}
\item $t_{1}\leq t_{2}$ in $\mathbb{C}^{\Gamma}(G)(S)$.
\item For every $\tau\in\Rep^{\circ}(G)(S)$, $t_{1}(\tau)\leq t_{2}(\tau)$
in $\Gamma_{\geq}^{r(\tau)}$.
\item For every $\tau\in\Rep^{\circ}(G)(S)$, $\max t_{1}(\tau)\leq\max t_{2}(\tau)$
in $\Gamma$.
\end{enumerate}
In $(2)$, $r(\tau)$ is the constant rank of $V(\tau)$. In $(3)$,
$\max t(\tau)=0$ if $\tau=0$. \end{prop}
\begin{proof}
Let $t_{\mathscr{E},i}$ be the $\Theta$-invariant morphism in $\Hom^{+}(M,\Gamma)$
corresponding to the base change $t_{i,S'}\in\mathbb{C}^{\Gamma}(G)(S')$
of $t_{i}\in\mathbb{C}^{\Gamma}(G)(S)$. Then
\begin{eqnarray*}
t_{1}\leq t_{2}\mbox{ in }\mathbb{C}^{\Gamma}(G)(S) & \iff & t_{1,S'}\leq t_{2,S'}\mbox{ in }\mathbb{C}^{\Gamma}(G)(S')\\
 & \iff & t_{\mathscr{E},1}\leq t_{\mathscr{E},2}\mbox{ in }\Hom^{+}(M,\Gamma)\\
 & \iff & \forall\lambda\in M_{d}:\quad t_{\mathscr{E},1}(\lambda)\leq t_{\mathscr{E},2}(\lambda)\mbox{ in }\Gamma\\
 & \iff & \forall x\in\mathbb{N}[M]^{W}:\quad\max t_{\mathscr{E},1}(x)\leq\max t_{\mathscr{E},2}(x)\mbox{ in }\Gamma
\end{eqnarray*}
using~lemma~\ref{lem:CarDominantOrderByMax} for the last equivalence.
Thus $(1)\Rightarrow(3)$ with $x=ch_{\mathscr{E}}(\tau_{S'})$ and
$(3)\Rightarrow(1)$ with $\tau=\tau_{\lambda}$. Plainly $(2)\Rightarrow(3)$.
Moreover, the equivalence $(1)\Leftrightarrow(3)$ already implies
Proposition~\ref{prop:FunctorWeakOrder}, from which $(1)\Rightarrow(2)$
immediately follows. \end{proof}
\begin{rem}
For $\Gamma$-filtrations $\mathcal{F}_{1},\mathcal{F}_{2}\in\mathbb{F}^{\Gamma}(G)(S)$,
the proposition implies:
\begin{multline*}
\qquad\qquad t(\mathcal{F}_{1})\leq t(\mathcal{F}_{2})\mbox{ in }\mathbb{C}^{\Gamma}(G)(S)\quad\iff\\
\forall\tau\in\Rep^{\circ}(G)(S),\,\forall\gamma\in\Gamma:\quad\mathcal{F}_{2}^{\gamma}(\tau)=0\,\Rightarrow\,\mathcal{F}_{1}^{\gamma}(\tau)=0\qquad\qquad
\end{multline*}
Indeed $\max t(\mathcal{F}_{i})(\tau)=\max t(\mathcal{F}_{i}(\tau))=\max\{\gamma:\mathcal{F}_{i}^{\gamma}(\tau)\neq0\}$
if $\tau\neq0$.
\end{rem}

\subsection{~\label{sub:CarWeakDomTanwithSharp}}

Still assuming that $G$ is isotrivial over a connected base scheme
$S$, suppose moreover that $\Gamma$ is divisible. Let $T$ be any
connected $S$-scheme. We claim that the monomorphism $\mathbb{C}^{\Gamma}(G)(S)\hookrightarrow\mathbb{C}^{\Gamma}(G)(T)$
then has a canonical retraction
\[
\sharp:\mathbb{C}^{\Gamma}(G)(T)\twoheadrightarrow\mathbb{C}^{\Gamma}(G)(S)
\]
in the category of partially ordered commutative monoids, which is
also functorial in $T$. To see this, we first fix a geometric point
$s\rightarrow T$, giving rise to a morphism 
\[
\pi_{1}(T,s)\rightarrow\pi_{1}(S,s)
\]
between the profinite étale fundamental group which classify the finite
étale covers of $T$ and $S$. Since $\mathbb{C}^{\Gamma}(G)$ becomes
constant over the Galois cover $S'/S$, it is itself a disjoint union
of finite étale covers of $S$ (indexed by the orbits of $\Theta$
in $\Hom^{+}(M,\Gamma)$). Thus $\pi_{1}(S,s)$ (resp.~$\pi_{1}(T,s)$)
acts on $\mathbb{C}^{\Gamma}(G)(s)$ with finite orbits and fixed
point set $\mathbb{C}^{\Gamma}(G)(S)$ (resp.~$\mathbb{C}^{\Gamma}(G)(T)$).
These actions respect the auxilliary structures, and averaging over
the $\pi_{1}(S,s)$-orbits thus yields the desired retraction. If
$s'\rightarrow T$ is another geometric point, there is a non-canonical
equivariant diagram whose vertical maps are isomorphisms \cite[V 7]{SGA1r}
\[
\begin{array}{ccc}
\pi_{1}(T,s') & \rightarrow & \pi_{1}(S,s')\\
\simeq\downarrow\,\, &  & \,\,\downarrow\simeq\\
\pi_{1}(T,s) & \rightarrow & \pi_{1}(S,s)
\end{array}\quad\mbox{acting on}\quad\begin{array}{c}
\mathbb{C}^{\Gamma}(G,s')\\
\simeq\downarrow\,\,\\
\mathbb{C}^{\Gamma}(G,s)
\end{array}
\]
Our retraction is therefore independent of $s$, and thus also functorial
in $T$.
\begin{prop}
\label{prop:CarTanDomOrderwithSharp}Suppose that $\Gamma$ is divisible
and $G$ is isotrivial over a connected base scheme $S$. For every
connected $S$-scheme $T$ and $t_{1},t_{2}\in\mathbb{C}^{\Gamma}(G)(T)$,
consider the following conditions:
\begin{enumerate}
\item $t_{1}^{\sharp}\leq t_{2}^{\sharp}$ in $\mathbb{C}^{\Gamma}(G)(S)$.
\item For every $\tau\in\Rep^{\circ}(G)(S)$, $t_{1}(\tau_{T})\leq t_{2}(\tau_{T})$
in $\Gamma_{\geq}^{r(\tau)}$.
\item For every $\tau\in\Rep^{\circ}(G)(S)$, $\max t_{1}(\tau_{T})\leq\max t_{2}(\tau_{T})$
in $\Gamma$.
\end{enumerate}
Then $(2)\iff(3)\Longrightarrow(1)$ and $(1)\iff(2)\iff(3)$ if $t_{1}^{\sharp}=t_{1}$.\end{prop}
\begin{proof}
We may assume that $T=s$ is a geometric point of the connected finite
étale Galois cover $S'$ of $S$ splitting $G$, realizing $\Theta=\Aut(S'/S)$
as a quotient of $\pi_{1}(S,s)$ through which all of the above actions
factor. Let $t_{\mathscr{E},i}$ be the image of $t_{i}$ under $\mathbb{C}^{\Gamma}(G)(s)\simeq\mathbb{C}^{\Gamma}(G,S')\simeq\Hom^{+}(M,\Gamma)$.
Then $t_{i}^{\sharp}$ maps to the average of the $\Theta$-orbit
of $t_{\mathscr{E},i}$. Plainly $(2)\Rightarrow(3)$ and conversely
$(3)\Rightarrow(2)$ since
\[
t_{1}(\tau_{s})\leq t_{2}(\tau_{s})\iff\begin{cases}
\forall1\leq i\leq r(\tau) & \max t_{1}(\Lambda^{i}\tau_{s})\leq\max t_{2}(\Lambda^{i}\tau_{s}),\\
\mbox{and} & \max t_{1}(\Lambda^{r(\tau)}\tau_{s}^{\vee})\leq\max t_{2}(\Lambda^{r(\tau)}\tau_{s}^{\vee}).
\end{cases}
\]
 For the remaining implications, note that 
\begin{eqnarray*}
t_{1}^{\sharp}\leq t_{2}^{\sharp}\mbox{ in }\mathbb{C}^{\Gamma}(G)(S) & \iff & t_{1,S'}^{\sharp}\leq t_{2,S'}^{\sharp}\mbox{ in }\mathbb{C}^{\Gamma}(G)(S')\\
 & \iff & t_{\mathscr{E},1}^{\sharp}\leq t_{\mathscr{E},2}^{\sharp}\mbox{ in }\Hom^{+}(M,\Gamma)\\
 & \iff & \forall\lambda\in M_{d}:\, t_{\mathscr{E},1}^{\sharp}(\lambda)\leq t_{\mathscr{E},2}^{\sharp}(\lambda)\mbox{ in }\Gamma\\
 & \iff & \forall\lambda\in M_{d}:\, t_{\mathscr{E},1}(\lambda^{\sharp})\leq t_{\mathscr{E},2}(\lambda^{\sharp})\mbox{ in }\Gamma
\end{eqnarray*}
where $\lambda^{\sharp}\in M\otimes\mathbb{Q}$ is the average of
the $\Theta$-orbit of $\lambda$, thus also
\[
t_{1}^{\sharp}\leq t_{2}^{\sharp}\mbox{ in }\mathbb{C}^{\Gamma}(G)(S)\iff\forall\lambda\in M_{d}^{\Theta}:\quad t_{\mathscr{E},1}(\lambda)\leq t_{\mathscr{E},2}(\lambda)\mbox{ in }\Gamma
\]
since $M_{d}^{\Theta}\subset M_{d}^{\sharp}\subset\mathbb{Q}_{\geq}M_{d}^{\Theta}$.
Thus $(3)\Rightarrow(1)$ with $\tau=\tau_{\lambda}$ for $\lambda\in M_{d}^{\Theta}$,
since 
\[
\max t_{i}(\tau_{\lambda})=\max\left\{ t_{\mathscr{E},i}(\theta\cdot\lambda):\theta\in\Theta\right\} =t_{\mathscr{E},i}(\lambda)\mbox{ in }\Gamma.
\]
Suppose finally that $t_{1}^{\sharp}=t_{1}$. Then using lemma~\ref{lem:CarDominantOrderByMax},
we have
\begin{eqnarray*}
t_{1}\leq t_{2}^{\sharp}\mbox{ in }\mathbb{C}^{\Gamma}(G)(S) & \iff & \forall x\in\mathbb{N}[M]^{W}:\,\max t_{\mathscr{E},1}(x)\leq\max t_{\mathscr{E},2}^{\sharp}(x)\mbox{ in }\Gamma\\
 & \Longrightarrow & \forall x\in\mathbb{N}[M]^{W,\Theta}:\,\max t_{\mathscr{E},1}(x)\leq\max t_{\mathscr{E},2}(x)\mbox{ in }\Gamma
\end{eqnarray*}
since indeed for any $x\in\mathbb{N}[M]^{W}$ we have 
\[
\max t_{\mathscr{E},2}^{\sharp}(x)\leq{\textstyle \frac{1}{\sharp\Theta}\sum_{\theta\in\Theta}}\max(\theta\cdot t_{\mathscr{E},2})(x)={\textstyle \frac{1}{\sharp\Theta}\sum_{\theta\in\Theta}}\max t_{\mathscr{E},2}(\theta\cdot x).
\]
Thus $(1)\Rightarrow(3)$ if $t_{1}^{\sharp}=t_{1}$, with $x=ch_{\mathscr{E}}(\tau_{s})\in\mathbb{N}[M]^{W,\Theta}$
for $\tau\in\Rep^{\circ}(G)(S)$. 
\end{proof}

\subsection{~}

The results of sections~\ref{sub:FunctWeakDom}-\ref{sub:CarWeakDomTanwithSharp}
were inspired by propositions $6.3.9$ and $9.4.2$ of \cite{DaOrRa10}.
However, the latter is contradicted by the following example, which
shows that usually $(1)$ does not imply $(2)$ in proposition~\ref{prop:CarTanDomOrderwithSharp}.
Take 
\[
\Gamma=\mathbb{Q},\quad S=\Spec K,\quad T=\Spec L\quad\mbox{and}\quad G=\Res_{L/K}\mathbb{G}_{m,L}
\]
where $L$ is a quadratic extension of a field $K$. Then $\mathbb{C}^{\Gamma}(G)(L)=\mathbb{Q}^{2}$
with the trivial partial order. The non-trivial element $\iota$ of
$\Gal(L/K)$ acts by $(x,y)\mapsto(y,x)$, thus $(x,y)^{\sharp}\leq(x',y')^{\sharp}$
if and only if $x+y=x'+y'$. For $n,m\in\mathbb{Z}$, the formula
$z\mapsto z^{n}(\iota z)^{m}$ defines a $2$-dimensional representation
$\tau_{n,m}:G\rightarrow GL_{K}(L)$ which is irreducible if $m\neq n$.
It maps $(x,y)\in\mathbb{Q}^{2}$ to 
\[
(x,y)(\tau_{n,m})=(\max,\min)\{xn+ym,yn+xm\}\in\mathbb{Q}_{\geq}^{2}
\]
Thus for $t_{1}=(1,-1)$ and $t_{2}=(0,0)$, $t_{1}^{\sharp}=t_{2}^{\sharp}=0$
in $\mathbb{C}^{\Gamma}(G)(K)=\mathbb{Q}$ but for every $n,m\in\mathbb{Z}$
with $n\neq m$, $t_{1}(\tau_{n,m})=(\left|n-m\right|,-\left|n-m\right|)>(0,0)=t_{2}(\tau_{n,m})$
in $\mathbb{Q}_{\geq}^{2}$.

\subsection{~\label{sub:descPlusTan}}

The addition map of section~\ref{sub:AdditionMap} has the following
Tannakian description. For an $S$-scheme $T$, $(\mathcal{F},\mathcal{G})\in\mathbb{STD}^{\Gamma}(G)(T)$
and any $\rho\in\Rep(G)(T)$, 
\[
(\mathcal{F}+\mathcal{G})^{\gamma}(\rho)=\sum_{\gamma_{1}+\gamma_{2}=\gamma}\mathcal{F}^{\gamma_{1}}(\rho)\cap\mathcal{G}^{\gamma_{2}}(\rho).
\]
Indeed, the question is local on $T$ for the Zariski topology, thus
by definition of $\mathbb{STD}^{\Gamma}(G)$, we may assume that $P_{\mathcal{F}}\cap P_{\mathcal{G}}$
contains a maximal subtorus $H$ of $G_{T}$. Then $\mathcal{F},\mathcal{G}$
lift to $f,g:\mathbb{D}_{T}(\Gamma)\rightarrow H$ and $(\mathcal{F}+\mathcal{G})=\Fi(f+g)$.
Let $V(\rho)_{\gamma_{1},\gamma_{2}}$ be the subsheaf of $V(\rho\vert H)$
where $\mathbb{D}_{T}(\Gamma)$ acts by $\gamma_{1}$ through $f$
and $\gamma_{2}$ through $g$. Then 
\begin{eqnarray*}
\mathcal{F}^{\gamma_{1}}(\rho) & = & \oplus_{\eta\geq\gamma_{1}}\oplus_{\eta'}V(\rho)_{\eta,\eta'}\\
\mathcal{G}^{\gamma_{2}}(\rho) & = & \oplus_{\eta}\oplus_{\eta'\geq\gamma_{2}}V(\rho)_{\eta,\eta'}\\
(\mathcal{F}+\mathcal{G})^{\gamma}(\rho) & = & \oplus_{\eta+\eta'\geq\gamma}V(\rho)_{\eta,\eta'}
\end{eqnarray*}
thus indeed $(\mathcal{F}+\mathcal{G})^{\gamma}(\rho)=\sum_{\gamma_{1}+\gamma_{2}=\gamma}\mathcal{F}^{\gamma_{1}}(\rho)\cap\mathcal{G}^{\gamma_{2}}(\rho)$.
\begin{cor}
Let $\varphi:G\rightarrow H$ be a morphism of reductive groups over
$S$. Then for any $S$-scheme $T$ and $t_{1},t_{2}\in\mathbb{C}^{\Gamma}(G)(T)$,
\[
\varphi(t_{1}+t_{2})\preceq\varphi(t_{1})+\varphi(t_{2})\quad\mbox{in}\quad\mathbb{C}^{\Gamma}(H)(T).
\]
\end{cor}
\begin{proof}
We may assume that $T=s$ is a geometric point, and lift $(t_{1},t_{2})$
to a pair of $\Gamma$-filtrations $(\mathcal{F}_{1},\mathcal{F}_{2})\in\mathbb{F}^{\Gamma}(G)(s)$
in osculatory relative position. Then $\varphi(\mathcal{F}_{1})$
and $\varphi(\mathcal{F}_{2})$ also are in standard relative position
(cf.~Remark~\ref{Rk:StdOverField}) and the above formula shows
that $\varphi(\mathcal{F}_{1}+\mathcal{F}_{2})=\varphi(\mathcal{F}_{1})+\varphi(\mathcal{F}_{2})$
in $\mathbb{F}^{\Gamma}(H)(s)$. Thus
\[
\begin{array}{rclcl}
\varphi\left(t(\mathcal{F}_{1})+t(\mathcal{F}_{2})\right) & = & \varphi\left(t(\mathcal{F}_{1}+\mathcal{F}_{2})\right)\\
 & = & t\left(\varphi(\mathcal{F}_{1}+\mathcal{F}_{2})\right)\\
 & = & t\left(\varphi(\mathcal{F}_{1})+\varphi(\mathcal{F}_{2})\right) & \preceq & t\left(\varphi(\mathcal{F}_{1})\right)+t\left(\varphi(\mathcal{F}_{2})\right)
\end{array}
\]
by~proposition~\ref{prop:CompAdditionType}, i.e.~$\varphi(t_{1}+t_{2})\preceq\varphi(t_{1})+\varphi(t_{2})$
in $\mathbb{C}^{\Gamma}(H)(s)$. 
\end{proof}

\subsection{~\label{sub:descGrPTan}}

The morphism defined in section~\ref{sub:defofGrMap} has the following
Tannakian description. Let $P$ be a parabolic subgroup of $G$ with
unipotent radical $U$, and suppose that $P=P_{\mathcal{F}}$ for
some $\Gamma$-filtration $\mathcal{F}$ on $\omega_{S}$. For every
$\rho\in\Rep(G)(S)$ and $\gamma\in\Gamma$, we may view $\Gr_{\mathcal{F}}^{\gamma}(\rho)=\mathcal{F}^{\gamma}(\rho)/\mathcal{F}_{+}^{\gamma}(\rho)$
as a representation of $P/U$. Then for every $S$-scheme $T$ and
every $\Gamma$-filtration $\mathcal{G}$ on $\omega_{T}$ such that
$P_{T}$ and $P_{\mathcal{G}}$ are in standard relative position
(i.e.~$P_{T}\cap P_{\mathcal{G}}$ is a smooth subscheme of $G_{T}$),
\[
\Gr_{P}(\mathcal{G})\left(\Gr_{\mathcal{F}}^{\gamma}(\rho)\right)=\Gr_{\mathcal{F}}^{\gamma}(\mathcal{G},\rho)
\]
where $\Gr_{\mathcal{F}}^{\gamma}(\mathcal{G},\rho)$ is the $\Gamma$-filtration
on $\Gr_{\mathcal{F}}^{\gamma}(\rho)_{T}=\mathcal{F}^{\gamma}(\rho)_{T}/\mathcal{F}_{+}^{\gamma}(\rho)_{T}$
induced by the $\Gamma$-filtration $\mathcal{G}(\rho)$ on $V(\rho)_{T}$,
so that for every $\theta\in\Gamma$, 
\[
\Gr_{\mathcal{F}}^{\gamma}(\mathcal{G},\rho)^{\theta}=\left(\mathcal{F}^{\gamma}(\rho)_{T}\cap\mathcal{G}^{\theta}(\rho)+\mathcal{F}_{+}^{\gamma}(\rho)_{T}\right)/\mathcal{F}_{+}^{\gamma}(\rho)_{T}.
\]
This follows from the explicit description of $\Gr_{P}$ in the proof
of proposition~\ref{prop:ExistCanoMapGrP}.

\subsection{~\label{sub:ClosedImmersion}}

The functors $\mathbb{G}^{\Gamma}(-)$ and $\mathbb{F}^{\Gamma}(-)$
preserve closed immersions.
\begin{prop}
\label{prop:GandFpreserveClosedImm}Let $H\hookrightarrow G$ be a
closed immersion of reductive group schemes over $S$. Then the induced
morphisms 
\[
\mathbb{G}^{\Gamma}(H)\rightarrow\mathbb{G}^{\Gamma}(G)\quad\mbox{and}\quad\mathbb{F}^{\Gamma}(H)\rightarrow\mathbb{F}^{\Gamma}(G)
\]
are finitely presented closed immersions. \end{prop}
\begin{proof}
Plainly, $\mathbb{G}^{\Gamma}(H)\rightarrow\mathbb{G}^{\Gamma}(G)$
is a monomorphism. Let $x:T\rightarrow\mathbb{G}^{\Gamma}(G)$ be
a morphism corresponding to $f:\mathbb{D}(\Gamma)_{T}\rightarrow G_{T}$.
Put $Z=f^{-1}(H_{T})$, a closed subgroup scheme of $Y=\mathbb{D}(\Gamma)_{T}$.
For every morphism $a:T'\rightarrow T$, we have:
\[
\begin{array}{rl}
 & x\circ a:T'\rightarrow\mathbb{G}^{\Gamma}(G)\mbox{ factors through }\mathbb{G}^{\Gamma}(H)\\
\iff & f_{T'}:\mathbb{D}(\Gamma)_{T'}\rightarrow G_{T'}\mbox{ factors through \ensuremath{H_{T'}}}\\
\iff & Z_{T'}=Y_{T'}.
\end{array}
\]
This last condition is represented by a closed subscheme of $T$ by
\cite[VIII 6.3 \& 6.4]{SGA3.1r}. Thus $\mathbb{G}^{\Gamma}(H)\rightarrow\mathbb{G}^{\Gamma}(G)$
is relatively representable by closed immersions, i.e.~itself a closed
immersion. Since $\mathbb{G}^{\Gamma}(H)\rightarrow S$ and $\mathbb{G}^{\Gamma}(G)\rightarrow S$
are locally of finite presentation (by theorem~\ref{thm:RepGr}),
so is $\mathbb{G}^{\Gamma}(H)\hookrightarrow\mathbb{G}^{\Gamma}(G)$
by \cite[1.4.3.v]{EGA4.1}, which therefore is a finitely presented
closed immersion. 

The second morphism $\mathbb{F}^{\Gamma}(H)\rightarrow\mathbb{F}^{\Gamma}(G)$
is a monomorphism: we have seen that $\Gamma$-filtrations on $\omega_{H}$
are uniquely determined by their value on the regular reprentation
of $H$, which is a quotient of the restriction to $H$ of the regular
representation of $G$. Since $\mathbb{G}^{\Gamma}(H)\rightarrow\mathbb{G}^{\Gamma}(G)\rightarrow\mathbb{F}^{\Gamma}(G)$
is quasi-compact and $\mathbb{G}^{\Gamma}(H)\rightarrow\mathbb{F}^{\Gamma}(H)$
is surjective, $\mathbb{F}^{\Gamma}(H)\rightarrow\mathbb{F}^{\Gamma}(G)$
is quasi-compact \cite[1.1.3]{EGA4.1}. Since $\mathbb{F}^{\Gamma}(H)\rightarrow S$
and $\mathbb{F}^{\Gamma}(G)\rightarrow S$ are separated and locally
of finite presentation, so is $\mathbb{F}^{\Gamma}(H)\rightarrow\mathbb{F}^{\Gamma}(G)$
by \cite[5.5.1.v]{EGA1} and \cite[1.4.3.v]{EGA4.1}. Since moreover
$\mathbb{F}^{\Gamma}(H)\rightarrow S$ satisfies the valuative criterion
of properness, so does $\mathbb{F}^{\Gamma}(H)\rightarrow\mathbb{F}^{\Gamma}(G)$,
which thus is a proper morphism by \cite[7.3.8]{EGA2} and a (finitely
presented) closed immersion by~\cite[18.12.6]{EGA4.4}.
\end{proof}

\section{Ranks and relative positions\label{sub:RanksRelativePos}}

Let $G$ be a reductive group over $S$.

\subsection{~}

Recall from section~\ref{sub:GrothendieckGroups} that for every
$S$-scheme $T$ and $\mathcal{F}\in\mathbb{F}^{\Gamma}(G)(T)$, the
exact $\otimes$-functor $\Gr_{\mathcal{F}}^{\bullet}:\Rep^{\circ}(G)(T)\rightarrow\Gra^{\Gamma}\LF(T)$
yields a ring homomorphism 
\[
K_{0}(G_{T})\rightarrow\mathcal{C}(T,\mathbb{Z}[\Gamma])
\]
mapping the class of $\tau\in\Rep^{\circ}(G)(T)$ in $K_{0}(G_{T})$
to the function
\[
t\mapsto\sum_{\gamma}\dim_{k(t)}\left(\Gr_{\mathcal{F}}^{\gamma}(\tau)\otimes k(t)\right)\cdot e^{\gamma}
\]
where $e^{\gamma}$ is the basis element of $\mathbb{Z}[\Gamma]$
corresponding to $\gamma\in\Gamma$. This construction is functorial
in $T$ and invariant under the action of $G$ on $\mathbb{F}^{\Gamma}(G)(T)$.
It therefore induces a morphism of fpqc sheaves of commutative rings
$(\Sch/S)^{\circ}\rightarrow\Ring$,\nomenclature[Ring]{$\Ring$}{Category of commutative rings.}
\[
\kappa:\underline{K}_{0}(G)\rightarrow\underline{\Mor}\left(\mathbb{C}^{\Gamma}(G),\mathbb{Z}[\Gamma]_{S}\right)
\]
where $\underline{K}_{0}(G)$ is the fpqc sheaf associated to the
presheaf $T\mapsto K_{0}(G_{T})$ while $\underline{\Mor}\left(\mathbb{C}^{\Gamma}(G),\mathbb{Z}[\Gamma]_{S}\right)$
is the fpqc sheaf of morphisms of $S$-schemes from the cone $\mathbb{C}^{\Gamma}(G)$
to the constant sheaf of rings $\mathbb{Z}[\Gamma]_{S}$.

\subsection{~}

Let now $(\mathcal{F}_{1},\mathcal{F}_{2})\in\mathbb{STD}^{\Gamma}(G)(T)$
be a pair of $\Gamma$-filtrations in standard relative position~(cf.~\ref{sub:RelativePositions}).
Then the formula
\[
\Gr_{\mathcal{F}_{1},\mathcal{F}_{2}}^{\gamma_{1},\gamma_{2}}(\tau)=\frac{\mathcal{F}_{1}^{\gamma_{1}}(\tau)\cap\mathcal{F}_{2}^{\gamma_{2}}(\tau)}{\mathcal{F}_{1,+}^{\gamma_{1}}(\tau)\cap\mathcal{F}_{2}^{\gamma_{2}}(\tau)+\mathcal{F}_{1}^{\gamma_{1}}(\tau)\cap\mathcal{F}_{2,+}^{\gamma_{2}}(\tau)}
\]
also defines an exact $\otimes$-functor 
\[
\Gr_{\mathcal{F}_{1},\mathcal{F}_{2}}^{\bullet,\bullet}:\Rep^{\circ}(G)(T)\longrightarrow\Gra^{\Gamma\times\Gamma}\LF(T).
\]
Indeed, we have to show that $\Gr_{\mathcal{F}_{1},\mathcal{F}_{2}}^{\gamma_{1},\gamma_{2}}(\tau)$
is locally free of finite rank, exact in $\tau$, and such that for
every $\tau',\tau''\in\Rep^{\circ}(G)(T)$ and $\gamma_{1},\gamma_{2}\in\Gamma$,
the natural map 
\[
\oplus_{(\gamma'_{1},\gamma'_{2})+(\gamma''_{1},\gamma''_{2})=(\gamma_{1},\gamma_{2})}\Gr_{\mathcal{F}_{1},\mathcal{F}_{2}}^{\gamma'_{1},\gamma'_{2}}(\tau')\otimes\Gr_{\mathcal{F}_{1},\mathcal{F}_{2}}^{\gamma''_{1},\gamma''_{2}}(\tau'')\rightarrow\Gr_{\mathcal{F}_{1},\mathcal{F}_{2}}^{\gamma_{1},\gamma_{2}}(\tau'\otimes\tau'')
\]
is an isomorphism. All this is local in the fpqc topology on $T$.
We may thus assume that $P_{\mathcal{F}_{1}}\cap P_{\mathcal{F}_{2}}$
contains a maximal torus $H$ of $G$ which is split, i.e.~$H=\mathbb{D}_{T}(M)$
for some finitely generated free abelian group $M$, in which case
$\mathcal{F}_{1}$ and $\mathcal{F}_{2}$ are split by morphisms $\mathcal{G}_{1}$
and $\mathcal{G}_{2}:\mathbb{D}_{T}(\Gamma)\rightarrow\mathbb{D}_{T}(M)$.
If $V(\tau)=\oplus_{m\in M}V(\tau)_{m}$ is the $H$-eigenspace decomposition
of $\tau\vert_{H}$, we then have a canonical isomorphism 
\[
\Gr_{\mathcal{F}_{1},\mathcal{F}_{2}}^{\gamma_{1},\gamma_{2}}(\tau)\simeq\oplus_{m\in M:(m\circ\mathcal{G}_{1},m\circ\mathcal{G}_{2})=(\gamma_{1},\gamma_{2})}V(\tau)_{m}
\]
and our claim easily follows. We thus obtain a ring homomorphism
\[
K_{0}(G_{T})\rightarrow\mathcal{C}(T,\mathbb{Z}[\Gamma\times\Gamma])
\]
which maps the class of $\tau\in\Rep^{\circ}(G)(T)$ in $K_{0}(G_{T})$
to the function
\[
t\mapsto\sum_{\gamma_{1},\gamma_{2}}\dim_{k(t)}\left(\Gr_{\mathcal{F}_{1},\mathcal{F}_{2}}^{\gamma_{1},\gamma_{2}}(\tau)\otimes k(t)\right)\cdot e^{\gamma_{1}}\otimes e^{\gamma_{2}}
\]
where $e^{\gamma_{1}}\otimes e^{\gamma_{2}}$ is the basis element
of $\mathbb{Z}[\Gamma\times\Gamma]=\mathbb{Z}[\Gamma]\otimes\mathbb{Z}[\Gamma]$
corresponding to the element $(\gamma_{1},\gamma_{2})$ of $\Gamma\times\Gamma$.

\subsection{~}

The above construction is again functorial in $T$ and invariant under
the diagonal action of $G$ on $\mathbb{STD}^{\Gamma}(G)$. It therefore
induces a morphism of fpqc sheaves of commutative rings $(\Sch/S)^{\circ}\rightarrow\Ring$,
\[
\kappa:\underline{K}_{0}(G)\rightarrow\underline{\Mor}\left(\mathbb{TSTD}^{\Gamma}(G),\mathbb{Z}[\Gamma\times\Gamma]_{S}\right).
\]

\subsection{~}

If now $f:\mathbb{Z}[\Gamma\times\Gamma]_{S}\rightarrow X$ is a morphism
of $S$-schemes, we denote by
\[
\begin{array}{rrcl}
\left\langle -,-\right\rangle _{f}: & \underline{K}_{0}(G) & \rightarrow & \underline{\Mor}\left(\mathbb{STD}^{\Gamma}(G),X\right)\\
\left\langle -,-\right\rangle _{f}^{os}: & \underline{K}_{0}(G) & \rightarrow & \underline{\Mor}\left(\mathbb{C}^{\Gamma}(G)^{2},X\right)\\
\left\langle -,-\right\rangle _{f}^{tr}: & \underline{K}_{0}(G) & \rightarrow & \underline{\Mor}\left(\mathbb{C}^{\Gamma}(G)^{2},X\right)
\end{array}
\]
the morphisms of fpqc sheaves on $S$ which are obtained by post-composition
of $\kappa$ with the obvious morphisms induced by $f$ and, respectively:
the quotient map 
\[
t_{2}:\mathbb{STD}^{\Gamma}(G)\twoheadrightarrow\mathbb{TSTD}^{\Gamma}(G)
\]
and the osculatory and transverse sections 
\[
os\quad\mbox{and}\quad tr:\mathbb{C}^{\Gamma}(G)^{2}\hookrightarrow\mathbb{TSTD}^{\Gamma}(G)
\]
of section~\ref{sub:RelativePositions}. For $\tau\in\underline{K}_{0}(G)(S)$,
we thus obtain morphisms of $S$-schemes
\[
\begin{array}{rrcl}
\left\langle -,-\right\rangle _{f,\tau}: & \mathbb{STD}^{\Gamma}(G) & \rightarrow & X\\
\left\langle -,-\right\rangle _{f,\tau}^{os}: & \mathbb{C}^{\Gamma}(G)^{2} & \rightarrow & X\\
\left\langle -,-\right\rangle _{f,\tau}^{tr}: & \mathbb{C}^{\Gamma}(G)^{2} & \rightarrow & X
\end{array}
\]
By construction, for every $S$-scheme $T$ and $(\mathcal{F}_{1},\mathcal{F}_{2})\in\mathbb{GEN}^{\Gamma}(G)(T)$,
\[
\left\langle \mathcal{F}_{1},\mathcal{F}_{2}\right\rangle _{f,\tau}=\left\langle t(\mathcal{F}_{1}),t(\mathcal{F}_{2})\right\rangle _{f,\tau}^{tr}\quad\mbox{in}\quad X(T).
\]

\subsection{~}

We will only consider these constructions in the following situation:
$\Gamma$ is a subgroup of $\mathbb{R}$, $X$ is the constant scheme
$\mathbb{R}_{S}$, $f$ is induced by the bilinear form $\Gamma\times\Gamma\ni(\gamma_{1},\gamma_{2})\mapsto\gamma_{1}\gamma_{2}\in\mathbb{R}$
and $\tau$ is a genuine representation in $\Rep^{\circ}(G)(S)$.
Then for any $S$-scheme $T$ and $(\mathcal{F}_{1},\mathcal{F}_{2})\in\mathbb{STD}^{\Gamma}(G)(T)$,
$\left\langle \mathcal{F}_{1},\mathcal{F}_{2}\right\rangle _{\tau}=\left\langle \mathcal{F}_{1},\mathcal{F}_{2}\right\rangle _{f,\tau}$
is the locally constant function $T\rightarrow\mathbb{R}$ given by
\[
t\mapsto\sum_{\gamma_{1},\gamma_{2}}\dim_{k(t)}\left(\Gr_{\mathcal{F}_{1},\mathcal{F}_{2}}^{\gamma_{1},\gamma_{2}}(\tau)\otimes k(t)\right)\cdot\gamma_{1}\gamma_{2}.
\]

\section{Appendix: pure subsheaves\label{sub:AppendixPureSub}}

Let $X$ be a scheme.
\begin{lem}
\label{lem:PureSeq1stLemma}For $\mathcal{A}\rightarrow\mathcal{B}\rightarrow\mathcal{C}$
in $\QCoh(X)$, consider the following conditions: 
\begin{enumerate}
\item For every quasi-coherent sheaf $\mathcal{F}$ on $X$,
\[
\makebox[50ex][c]{\ensuremath{0\rightarrow\mathcal{A}\otimes\mathcal{F}\rightarrow\mathcal{B}\otimes\mathcal{F}\rightarrow\mathcal{C}\otimes\mathcal{F}\rightarrow0}}\mbox{is exact in }\QCoh(X).
\]

\item For every morphism $f:Y\rightarrow X$,
\[
\makebox[50ex][c]{\ensuremath{0\rightarrow f^{\ast}\mathcal{A}\rightarrow f^{\ast}\mathcal{B}\rightarrow f^{\ast}\mathcal{C}\rightarrow0}}\mbox{is exact in }\QCoh(Y).
\]

\item For every morphism $f:Y\rightarrow X$ and quasi-coherent sheaf $\mathcal{F}$
on $Y$,
\[
\makebox[50ex][c]{\ensuremath{0\rightarrow f^{\ast}\mathcal{A}\otimes\mathcal{F}\rightarrow f^{\ast}\mathcal{B}\otimes\mathcal{F}\rightarrow f^{\ast}\mathcal{C}\otimes\mathcal{F}\rightarrow0}}\mbox{is exact in }\QCoh(Y).
\]

\end{enumerate}
Then $(1)\Leftarrow(2)\Leftrightarrow(3)$ and $(1)\Leftrightarrow(2)\Leftrightarrow(3)$
if $X$ is quasi-separated. \end{lem}
\begin{proof}
Obviously $(3)\Rightarrow(1)$ and $(2)$. Suppose $(2)$ holds. Let
$f:Y\rightarrow X$ be a morphism, $\mathcal{F}$ a quasi-coherent
sheaf on $Y$, $g:Z\rightarrow Y$ the structural morphism of $Z=\Spec(\mathcal{O}_{Y}[\mathcal{F}])$
where $\mathcal{O}_{Y}[\mathcal{F}]=\mathcal{O}_{Y}\oplus\mathcal{F}$
is the quasi-coherent $\mathcal{O}_{Y}$-algebra defined by $\mathcal{F}\cdot\mathcal{F}=0$.
By assumption, $0\rightarrow h^{\ast}\mathcal{A}\rightarrow h^{\ast}\mathcal{B}\rightarrow h^{\ast}\mathcal{C}\rightarrow0$
is an exact sequence of quasi-coherent sheaves on $Z$, where $h=f\circ g$.
Since $g$ is affine, 
\[
0\rightarrow g_{\ast}h^{\ast}\mathcal{A}\rightarrow g_{\ast}h^{\ast}\mathcal{B}\rightarrow g_{\ast}h^{\ast}\mathcal{C}\rightarrow0
\]
is an exact sequence of quasi-coherent sheaves on $Y$. But 
\[
g_{\ast}h^{\ast}\mathcal{X}=g_{\ast}g^{\ast}f^{\ast}\mathcal{X}=f^{\ast}\mathcal{X}\oplus f^{\ast}\mathcal{X}\otimes\mathcal{F}
\]
for any $\mathcal{X}$ in $\QCoh(X)$, therefore 
\[
0\rightarrow f^{\ast}\mathcal{A}\otimes\mathcal{F}\rightarrow f^{\ast}\mathcal{B}\otimes\mathcal{F}\rightarrow f^{\ast}\mathcal{C}\otimes\mathcal{F}\rightarrow0
\]
 is exact and $(2)\Rightarrow(3)$. Suppose now that $X$ is quasi-separated
and $(1)$ holds. Let $f:Y\rightarrow X$ be any morphism. Let $\{X_{i}\}$
and $\{Y_{i,j}\}$ be open coverings of $X$ and $Y$ by affine schemes
such that $f(Y_{i,j})\subset X_{i}$ and let $f_{i,j}:Y_{i,j}\rightarrow X_{i}$
be the induced morphism. Since $(f^{\ast}\mathcal{X})\vert_{Y_{i,j}}=f_{i,j}^{\ast}(\mathcal{X}\vert_{X_{i}})$
for every $\mathcal{X}\in\QCoh(X)$, we have to show that $0\rightarrow f_{i,j}^{\ast}(\mathcal{A}_{i})\rightarrow f_{i,j}^{\ast}(\mathcal{B}_{i})\rightarrow f_{i,j}^{\ast}(\mathcal{C}_{i})\rightarrow0$
is exact on $Y_{i,j}$ for every $i,j$, with $\mathcal{X}_{i}=\mathcal{X}\vert_{X_{i}}$.
Since $Y_{i,j}$ and $X_{i}$ are affine, this amounts to showing
that 
\[
0\rightarrow\mathcal{A}_{i}\otimes\mathcal{O}_{i,j}\rightarrow\mathcal{B}_{i}\otimes\mathcal{O}_{i,j}\rightarrow\mathcal{C}_{i}\otimes\mathcal{O}_{i,j}\rightarrow0
\]
is exact on $X_{i}$ for every $i,j$, for the quasi-coherent sheaf
$\mathcal{O}_{i,j}=(f_{i,j})_{\ast}\mathcal{O}_{Y_{i,j}}$ on $X_{i}$.
Since $X$ is quasi-separated, the immersion $\iota_{i}:X_{i}\hookrightarrow X$
is quasi-compact and quasi-separated by \cite[1.2.2.i \& 1.2.7.b]{EGA4.1},
thus $\mathcal{F}_{i,j}=(\iota_{i})_{\ast}\mathcal{O}_{i,j}$ is a
quasi-coherent sheaf on $X$ by \cite[1.7.4]{EGA4.1} and $0\rightarrow\mathcal{A}\otimes\mathcal{F}_{i,j}\rightarrow\mathcal{B}\otimes\mathcal{F}_{i,j}\rightarrow\mathcal{C}\otimes\mathcal{F}_{i,j}\rightarrow0$
is an exact sequence on $X$ by assumption. Pulling back through the
exact restriction functor $\iota_{i}^{\ast}:\QCoh(X)\rightarrow\QCoh(X_{i})$
yields the desired result. \end{proof}
\begin{defn}
We say that the sequence $0\rightarrow\mathcal{A}\stackrel{\iota}{\rightarrow}\mathcal{B}\rightarrow\mathcal{C}\rightarrow0$
is pure exact, or that $\iota$ is a pure monomorphism, or that $\iota(\mathcal{A})$
is a pure (quasi-coherent) subsheaf of $\mathcal{B}$ if the above
condition $(2)$ holds. \end{defn}
\begin{lem}
Let $\mathcal{B}$ be a quasi-coherent sheaf on $X$. Then 
\[
\mathcal{P}:(\Sch/X)^{\circ}\rightarrow\Set\qquad T\mapsto\left\{ \mbox{pure quasi-coherent subsheaves }\mathcal{A}\mbox{ of }\mathcal{B}_{T}\right\} 
\]
is an fpqc sheaf on $\Sch/X$. \end{lem}
\begin{proof}
It is a functor: if $\mathcal{A}\in\mathcal{P}(T)$ and $\alpha:T'\rightarrow T$
is an $X$-morphism, the monomorphism $\alpha^{\ast}(\mathcal{A}\hookrightarrow\mathcal{B}_{T})$
identifies $\alpha^{\ast}(\mathcal{A})$ with a quasi-coherent subsheaf
of $\alpha^{\ast}(\mathcal{B}_{T})=\mathcal{B}_{T'}$, which is pure
since for any morphism $f':Y\rightarrow T'$, if $f=\alpha\circ f'$,
then $f^{\prime\ast}\circ\alpha^{\ast}(\mathcal{A}\hookrightarrow\mathcal{B}_{T})=f^{\ast}(\mathcal{A}\hookrightarrow\mathcal{B}_{T})$
is a monomorphism of quasi-coherent sheaves on $Y$ since $\mathcal{A}$
is pure in $\mathcal{B}_{T}$. It is an fpqc sheaf: if $\{T_{i}\rightarrow T\}$
is an fpqc cover and $\mathcal{A}_{i}\in\mathcal{P}(T_{i})$ have
the same image $\mathcal{A}_{i,j}\in\mathcal{P}(T_{i}\times_{T}T_{j})$,
then the quasi-coherent subsheaves $\mathcal{A}_{i}$ of $\mathcal{B}_{T_{i}}$
glue to a quasi-coherent subsheaf $\mathcal{A}$ of $\mathcal{B}_{T}$
which is pure since for any $f:Y\rightarrow T$, $f^{\ast}(\mathcal{A}\hookrightarrow\mathcal{B}_{T})$
is a monomorphism of quasi-coherent sheaves on $Y$ as it becomes
so in the fpqc cover $\{Y\times_{T}T_{i}\rightarrow Y\}$ of $Y$.\end{proof}
\begin{lem}
\label{lem:PureSubsAndDirFact}Let $\mathcal{A}$ be a quasi-coherent
subsheaf of $\mathcal{B}$.
\begin{enumerate}
\item Suppose that locally on $X$ for the fpqc topology, $\mathcal{A}$
is a direct factor of $\mathcal{B}$. Then $\mathcal{A}$ is a pure
subsheaf of $\mathcal{B}$. 
\item Suppose that $\mathcal{A}$ is a pure subsheaf of $\mathcal{B}$ and
$\mathcal{C}=\mathcal{B}/\mathcal{A}$ is finitely presented. Then
locally on $X$ for the Zariski topology, $\mathcal{A}$ is a direct
factor of $\mathcal{B}$. 
\end{enumerate}
\end{lem}
\begin{proof}
$(1)$ A direct factor being obviously pure, this follows from the
previous lemma. As for $(2)$: the assumptions are local in the Zariski
topology by the previous lemma, we may thus assume that $X=\Spec(R)$
for some ring $R$. Then $A=\Gamma(X,\mathcal{A})$ is a pure $R$-submodule
of $B=\Gamma(X,\mathcal{B})$ in the sense of \cite[Appendix to \S 7]{Ma89}
by $(2)\Rightarrow(1)$ of lemma~\ref{lem:PureSeq1stLemma}, and
$C=B/A$ is a finitely presented $R$-module. Therefore $A$ is a
direct factor of $B$ by \cite[Theorem 7.14]{Ma89}, i.e.~$\mathcal{A}$
is a direct factor of $\mathcal{B}$.\end{proof}

\chapter{The vectorial Tits building $\mathbf{F}^{\Gamma}(G)$\label{chap:VectTitsBuild}}

Let $\mathcal{O}$ be a local ring, $G$ a reductive group over $\Spec(\mathcal{O})$.
We shall here take a closer look at the set $\mathbf{F}^{\Gamma}(G)=\mathbb{F}^{\Gamma}(G)(\mathcal{O})$
of sections of $\mathbb{F}^{\Gamma}(G)$ over $\Spec(\mathcal{O})$.

\section{Combinatorial structures}

\subsection{~}

We say that a morphism of posets $f:(\mathbf{X},\leq)\rightarrow(\mathbf{Y},\leq)$
is nice if
\[
\forall x,y\in\mathbf{X}\times\mathbf{Y}\mbox{\,\ with }f(x)\leq y,\mbox{\,\ there is a unique }x'\in f^{-1}(y)\mbox{ with }x\leq x'.
\]
We say that it is very nice if also
\[
\forall x,y\in\mathbf{X}\times\mathbf{Y}\mbox{\,\ with }f(x)\geq y,\mbox{\,\ there is an }x'\in f^{-1}(y)\mbox{ with }x\geq x'.
\]

\subsection{~}

We will define below an $\Aut(G)$-equivariant sequence of nice surjective
morphisms of posets 
\[
\xyC{2pc}\xymatrix{\mathbf{SBP}(G)\ar@{->>}^{a}[r] & \mathbf{SP}(G)\ar@{->>}^{b}[r] & \mathbf{OPP}(G)\ar@{->>}[r]^{p_{1}} & \mathbf{P}(G)\ar@{->>}[r]^{t} & \mathbf{O}(G)}
\]
The group $\mathsf{G}=G(\mathcal{O})$ acts on it through $\Int:\mathsf{G}\rightarrow\Aut(G)$,
and we will see that 
\[
\mathsf{G}\backslash\mathbf{SBP}(G)=\mathsf{G}\backslash\mathbf{SP}(G)=\mathsf{G}\backslash\mathbf{OPP}(G)=\mathsf{G}\backslash\mathbf{P}(G)=\mathbf{O}(G).
\]

\subsection{~}

We first define our posets. We will use the following notations:\nomenclature[S(G)]{$\mathbf{S}(G)$}{Set of all maximal split tori of $G$, page \nomrefpage}\nomenclature[B(G)]{$\mathbf{B}(G)$}{Set of all minimal parabolic subgroups of $G$, page \nomrefpage}\nomenclature[P(G)]{$\mathbf{P}(G)$}{Set of all parabolic subgroups of $G$, page \nomrefpage}\nomenclature[SP(G)]{$\mathbf{SP}(G)$}{Set of all pairs $(S,P)$ in $\mathbf{S}(G) \times \mathbf{P}(G)$ with  $Z_G(S) \subset P$, page \nomrefpage}\nomenclature[SBP(G)]{$\mathbf{SBP}(G)$}{Set of triples $(S,B,P)$ in $\mathbf{S}(G) \times \mathbf{B}(G) \times \mathbf{P}(G)$ with  $Z_G(S) \subset B \subset P$, page \nomrefpage}\nomenclature[OPP(G)]{$\mathbf{OPP}(G)$}{Set of all pairs of opposed parabolic subgroups of $G$, page \nomrefpage}
\begin{eqnarray*}
\mathbf{S}(G) & = & \left\{ S:\mbox{maximal split torus of }G\right\} \\
\mathbf{B}(G) & = & \left\{ B:\mbox{minimal parabolic subgroup of }G\right\} \\
\mathbf{P}(G) & = & \left\{ P:\mbox{parabolic subgroup of }G\right\} \\
\mathbf{SP}(G) & = & \left\{ (S,P):Z_{G}(S)\subset P\right\} \\
\mathbf{SBP}(G) & = & \left\{ (S,B,P):Z_{G}(S)\subset B\subset P\right\} \\
\mathbf{OPP}(G) & = & \left\{ (P,P'):\mbox{opposed parabolic subgroups of }G\right\} 
\end{eqnarray*}
Thus $\mathbf{P}(G)=\mathbb{P}(G)(\mathcal{O})$ and $\mathbf{OPP}(G)=\mathbb{OPP}(G)(\mathcal{O})$.
In addition, we set\nomenclature[O(G)]{$\mathbf{O}(G)$}{Set of all types of parabolic subgroups of $G$, page \nomrefpage}
\[
\mathbf{O}(G)=\mbox{image of }t:\mathbb{P}(G)(\mathcal{O})\rightarrow\mathbb{O}(G)(\mathcal{O}).
\]
We endow $\mathbf{P}(G)$ and $\mathbf{O}(G)$ with their natural
partial orders and the remaining three sets $\mathbf{SBP}(G)$, $\mathbf{SP}(G)$
and $\mathbf{OPP}(G)$ with the following ones:
\begin{eqnarray*}
(S_{1},B_{1},P_{1})\leq(S_{2},B_{2},P_{2}) & \iff & S_{1}=S_{2},\, B_{1}=B_{2}\mbox{ and }P_{1}\subset P_{2}\\
(S_{1},P_{1})\leq(S_{2},P_{2}) & \iff & S_{1}=S_{2}\mbox{ and }P_{1}\subset P_{2}\\
(P_{1},P'_{1})\leq(P_{2},P'_{2}) & \iff & P_{1}\subset P_{2}\mbox{ and }P'_{1}\subset P'_{2}
\end{eqnarray*}

\subsection{~}

The morphism $t:\mathbf{P}(G)\rightarrow\mathbf{O}(G)$ maps $P$
to its type $t(P)$. It is plainly a morphism of posets. It is surjective
by definition of $\mathbf{O}(G)$, nice by \cite[XXVI 3.8]{SGA3.3r}
and even very nice by~\cite[XXVI 5.5]{SGA3.3r}. The group $\mathsf{G}$
acts trivially on $\mathbf{O}(G)$, and $\mathsf{G}\cdot P=t^{-1}t(P)$
by \cite[XXVI 5.2]{SGA3.3r}, thus $\mathsf{G}\backslash\mathbf{P}(G)=\mathbf{O}(G)$.

\subsection{~}

The morphism $p_{1}:\mathbf{OPP}(G)\rightarrow\mathbf{P}(G)$ maps
$(P,P')$ to $P$. It is plainly a morphism of posets, and it is surjective
by \cite[XXVI 2.3 \& 4.3.2]{SGA3.3r}. Consider now $(P,P')\in\mathbf{OPP}(G)$,
$Q\in\mathbf{P}(G)$ and suppose first that $P\subset Q$. Since $t$
is nice, there is a unique $Q'\in\mathbf{P}(G)$ with $P'\subset Q'$
and $t(Q')=\iota t(Q)$, where $\iota$ is the opposition involution
of $\mathbf{O}(G)$. We have $(Q,Q')\in\mathbf{OPP}(G)$ by \cite[XXVI 4.3.2 \& 4.2.1]{SGA3.3r},
thus $p_{1}$ is nice. If $Q\subset P$, then $Q_{L}=Q\cap L$ is
a parabolic subgroup of $L=P\cap P'$ and its Levi subgroups are the
Levi subgroups of $Q$ contained in $L$ by \cite[XXVI 1.20]{SGA3.3r}.
Since $p_{1}:\mathbf{OPP}(L)\rightarrow\mathbf{P}(L)$ is surjective,
there is a parabolic subgroup $Q'_{L}$ of $L$ opposed to $Q_{L}$.
Then $Q'_{L}=Q'\cap L$ for a unique parabolic subgroup $Q'$ of $G$
contained in $P'$, and $(Q,Q')\in\mathbf{OPP}(G)$ since $Q\cap Q'=Q_{L}\cap Q'_{L}$
is a Levi subgroup of $Q_{L}$ and $Q'_{L}$, thus also of $Q$ and
$Q'$. Therefore $p_{1}$ is very nice. Finally, the stabilizer of
$P$ in $\mathsf{G}$ is $\mathbf{\mathsf{P}}=P(\mathcal{O})$ by
\cite[XXVI 1.2]{SGA3.3r}, and $\mathsf{P}\cdot(P,P')=p_{1}^{-1}(P)$
by \cite[XXVI 1.8 \& 4.3.2]{SGA3.3r}, thus $\mathsf{G}\backslash\mathbf{OPP}(G)=\mathsf{G}\backslash\mathbf{P}(G)$.

\subsection{~}

The morphism $b:\mathbf{SP}(G)\rightarrow\mathbf{OPP}(G)$ maps $(S,P)$
to $(P,\iota_{S}P)$, where $\iota_{S}P$ is defined in the next lemma,
which also says that $b$ is a morphism of posets. \nomenclature[iota_S]{$\iota _S$}{Opposition involution on the apartment of $S$ in $\mathbf{P}(G)$, page \nomrefpage}
\begin{lem}
\label{lem:LeviAndAppart}For $S\in\mathbf{S}(G)$ and $P\in\mathbf{P}(G)$
with $Z_{G}(S)\subset P$, there exists a unique Levi subgroup $L$
of $P$ and a unique parabolic subgroup $\iota_{S}P$ of $G$ opposed
to $P$ with $Z_{G}(S)\subset L,\iota_{S}P$. Moreover $L=P\cap\iota_{S}P$
and $P\mapsto\iota_{S}P$ preserves inclusions. \end{lem}
\begin{proof}
By \cite[XIV 3.20]{SGA3.2}, there is a maximal torus $T$ in $Z_{G}(S)$.
It is also maximal in $G$ and $P$. By~\cite[XXVI 1.6]{SGA3.3r},
there is a unique Levi subgroup $L$ of $P$ with $T\subset L$. We
have to show that $Z_{G}(S)\subset L$. By \cite[XXVI 6.11]{SGA3.3r},
this is equivalent to $R_{sp}(L)\subset S$, where $R_{sp}(L)$ is
the split radical of $L$, i.e.~the maximal split subtorus $R(L)_{sp}$
of the radical $R(L)$ of $L$. Since $T$ is a maximal torus in $L$,
$R(L)$ is contained in $T$, thus $R_{sp}(L)$ is contained in the
maximal split subtorus $T_{sp}$ of $T$, which obviously contains
$S$ and in fact equals $S$ by maximality of $S$. This proves the
existence and uniqueness of $L$. That of $\iota_{S}P$ follows from
\cite[XXVI 4.3.2]{SGA3.3r} which also shows that $L=P\cap\iota_{S}P$.
If $P\subset Q$, there is a unique $(Q,Q')\in\mathbf{OPP}(G)$ with
$\iota_{S}P\subset Q'$ because $p_{1}$ is nice, and obviously $\iota_{S}Q=Q'$,
thus $\iota_{S}P\subset\iota_{S}Q$.
\end{proof}
Starting with $(P,P')\in\mathbf{OPP}(G)$ put $L=P\cap P'$ and let
$S$ be a maximal split torus in $G$ containing the split radical
$R_{sp}(L)$ of $L$. Then $Z_{G}(S)$ is contained in $Z_{G}(R_{sp}(L))$
which equals $L$ by \cite[XXVI 6.11]{SGA3.3r}, thus $(S,P)\in\mathbf{SP}(G)$
and $b(S,P)$ equals $(P,P')$, i.e.~$b$ is surjective. It is obviously
nice, although not \emph{very }nice. The stabilizer of $b(S,P)$ in
$\mathsf{G}$ is $\mathsf{L}=L(\mathcal{O})$ where $L=P\cap\iota_{S}P$,
and $\mathsf{L}\cdot(S,P)=b^{-1}b(S,P)$ by \cite[XXVI 6.16]{SGA3.3r},
thus $\mathsf{G}\backslash\mathbf{SP}(G)=\mathsf{G}\backslash\mathbf{OPP}(G)$.
The opposition involution $\iota(P_{1},P_{2})=(P_{2},P_{1})$ of $\mathbf{OPP}(G)$
lifts to the involution $\iota(S,P)=(S,\iota_{S}P)$ of $\mathbf{SP}(G)$.

\subsection{~}

The morphism $a:\mathbf{SBP}(G)\rightarrow\mathbf{SP}(G)$ maps $(S,B,P)$
to $(S,P)$. It is plainly a nice morphism of poset, although not
\emph{very} nice. Fix $(S,P)\in\mathbf{SP}(G)$, let $L=P\cap\iota_{S}P$.
Then \cite[XXVI 1.20]{SGA3.3r} sets up a bijection between: the set
of minimal parabolic subgroup $B$ of $G$ with $Z_{G}(S)\subset B\subset P$
(the fiber $a^{-1}(S,P)$) and the set of minimal parabolic subgroups
$B_{L}=B\cap L$ of $L$ with $Z_{G}(S)\subset B_{L}$. The latter
set is not empty by~\cite[XXVI 6.16]{SGA3.3r}, thus $a$ is surjective.
The stabilizer of $(S,P)$ in $\mathsf{G}$ equals $\mathsf{N}_{L}(S)=N_{L}(S)(\mathcal{O})$
and $\mathsf{N}_{L}(S)\cdot(S,B,P)=a^{-1}(S,P)$ by \cite[XXVI 7.2]{SGA3.3r}
applied to $Z_{G}(S)\subset L$, thus $\mathsf{G}\backslash\mathbf{SBP}(G)=\mathsf{G}\backslash\mathbf{SP}(G)$.
The stabilizer of $(S,B,P)$ in $\mathsf{G}$ is the stabilizer of
$(S,B)$, namely $\mathsf{Z}_{G}(S)=Z_{G}(S)(\mathcal{O})$ since
$Z_{G}(S)=B\cap\iota_{S}B$.

\subsection{~}

By~\cite[XXVI 5.7]{SGA3.3r}, there is a smallest element $\circ$
in $\mathbf{O}(G)$\nomenclature[o]{$\circ$}{Smallest element of $\mathbf{O}(G)$, page \nomrefpage}.
For 
\[
\mathbf{X}\quad\mbox{in}\quad\{\mathbf{SBP}(G),\mathbf{SP}(G),\mathbf{OPP}(G),\mathbf{P}(G)\},
\]
the morphism $f:\mathbf{X}\rightarrow\mathbf{O}(G)$ is very nice.
We've proved it already in the last two cases. Since $f$ is nice,
our assertion is equivalent to: $\mathbf{X}_{\mathrm{min}}=f^{-1}(\circ)$
where $\mathbf{X}_{\mathrm{min}}$ is the set of minimal elements
in $\mathbf{X}$. This is obvious for $\mathbf{SBP}(G)$, and also
for $\mathbf{SP}(G)$ since $a$ is surjective. For any $x\in\mathbf{X}_{\mathrm{min}}=f^{-1}(\circ)$,
there is then a unique section 
\[
\xymatrix{(\mathbf{X},\leq)\ar@{->>}[r]_{f} & (\mathbf{O}(G),\leq)\ar@/_{.5pc}/[l]_{s_{x}}}
\]
with $s_{x}(\circ)=x$, and these sections cover $\mathbf{X}$.

\subsection{~\label{sub:TitsBuildDiag}}

Let now $\Gamma=(\Gamma,+,\leq)$ be a non-trivial totally ordered
commutative group and form the $\Aut(G)$-equivariant cartesian diagram
of sets: 
\[
\xyR{2pc}\xymatrix{\mathbf{ACF}^{\Gamma}(G)\ar@{->>}[r]^{a}\ar[d]^{F} & \mathbf{AF}^{\Gamma}(G)\ar@{->>}[r]^{b}\ar[d]^{F} & \mathbf{G}^{\Gamma}(G)\ar@{->>}[r]^{\Fi}\ar[d]^{F} & \mathbf{F}^{\Gamma}(G)\ar@{->>}[r]^{t}\ar[d]^{F} & \mathbf{C}^{\Gamma}(G)\ar[d]^{F}\\
\mathbf{SBP}(G)\ar@{->>}^{a}[r] & \mathbf{SP}(G)\ar@{->>}^{b}[r] & \mathbf{OPP}(G)\ar@{->>}[r]^{p_{1}} & \mathbf{P}(G)\ar@{->>}[r]^{t} & \mathbf{O}(G)
}
\]
where $\mathbf{C}^{\Gamma}(G)$\nomenclature[C^Gamma(G)]{$\mathbf{C}^{\Gamma}(G)$}{Cone of types of $\Gamma$-filtrations or $\Gamma$-graduations on $G$, page \nomrefpage}
is the inverse image of $\mathbf{O}(G)$ under $F:\mathbb{C}^{\Gamma}(G)(\mathcal{O})\rightarrow\mathbb{O}(G)(\mathcal{O})$.
Of course we may and do identify $\mathbf{F}^{\Gamma}(G)$\nomenclature[F^Gamma(G)]{$\mathbf{F}^{\Gamma}(G)$}{Set of all $\Gamma$-filtrations on $G$, page \nomrefpage}
with $\mathbb{F}^{\Gamma}(G)(\mathcal{O})$ and $\mathbf{G}^{\Gamma}(G)$\nomenclature[G^Gamma(G)]{$\mathbf{G}^{\Gamma}(G)$}{Set of all $\Gamma$-graduations on $G$, page \nomrefpage}
with $\mathbb{G}^{\Gamma}(G)(\mathcal{O})$, see section~\ref{sub:DefFondDiag}.
With these identifications, we find:\nomenclature[AF^Gamma(G)]{$\mathbf{AF}^{\Gamma}(G)$}{Set of all pairs $(S,\mathcal{F})$ with $S \in \mathbf{S}(G)$ and $\mathcal{F} \in \mathbf{F}^\Gamma (S)$, page \nomrefpage}\nomenclature[ACF^Gamma(G)]{$\mathbf{ACF}^{\Gamma}(G)$}{Set of all triples $(S,B,\mathcal{F})$ with $S \in \mathbf{S}(G)$ and $\mathcal{F} \in F^{-1}(B) \subset \mathbf{F}^\Gamma (S)$, page \nomrefpage}
\begin{eqnarray*}
\mathbf{AF}^{\Gamma}(G) & = & \left\{ (S,\mathcal{F})\in\mathbf{S}(G)\times\mathbf{F}^{\Gamma}(G)\mbox{ with }Z_{G}(S)\subset P_{\mathcal{F}}\right\} \\
 & = & \left\{ (S,\mathcal{G})\in\mathbf{S}(G)\times\mathbf{G}^{\Gamma}(G)\mbox{ with }Z_{G}(S)\subset L_{\mathcal{G}}\right\} \\
 & = & \left\{ (S,\mathcal{G}):S\in\mathbf{S}(G),\,\mathcal{G}\in\mathbf{G}^{\Gamma}(S)\right\} \\
\mathbf{ACF}^{\Gamma}(G) & = & \left\{ (S,B,\mathcal{F}):\mathbf{S}(G)\times\mathbf{B}(G)\times\mathbf{F}^{\Gamma}(G)\mbox{ with }Z_{G}(S)\subset B\subset P_{\mathcal{F}}\right\} \\
 &  & \left\{ (S,B,\mathcal{G}):S\in\mathbf{S}(G),\,\mathcal{G}\in\mathbf{G}^{\Gamma}(S)\mbox{ with }Z_{G}(S)\subset B\subset P_{\mathcal{G}}\right\} .
\end{eqnarray*}
The opposition involution $\iota$ of $\mathbf{G}^{\Gamma}(G)$ lifts
to an involution of $\mathbf{AF}^{\Gamma}(G)$, given by 
\[
\iota(S,\mathcal{F})=(S,\iota_{S}\mathcal{F})\quad\mbox{or}\quad\iota(S,\mathcal{G})=(S,\iota\mathcal{G}).
\]
Here $\iota\mathcal{G}=\mathcal{G}^{-1}$ in $\mathbf{G}^{\Gamma}(S)$
and $(\Fi(\mathcal{G}),\Fi(\iota\mathcal{G}))=(\mathcal{F},\iota_{S}\mathcal{F})$.\nomenclature[iota_S]{$\iota _S$}{Opposition involution on the apartment of $S$  in $\mathbf{F}^\Gamma (G)$, page \nomrefpage}

\subsection{\label{sub:descFacet}~}

Fix $S\in\mathbf{S}(G)$. Let $M$ be its group of characters, $R\subset M$
the roots of $S$ in $\mathfrak{g}=\Lie(G)$ and $\mathfrak{g}=\mathfrak{g}_{0}\oplus\oplus_{\alpha\in R}\mathfrak{g}_{\alpha}$
the corresponding decomposition of $\mathfrak{g}$. Put\nomenclature[W_G(S)]{$W_G(S)$}{Weyl group of $S$ in $G$, $W_G(S)=N_G(S)/Z_G(S)$.}
\[
W=(N_{G}(S)/Z_{G}(S))(\mathcal{O})=N_{G}(S)(\mathcal{O})/Z_{G}(S)(\mathcal{O}).
\]
By~\cite[XXVI 7.4]{SGA3.3r}, there exists a unique root datum $\mathscr{R}=(M,R,M^{\ast},R^{\ast})$
with Weyl group $W$ and a $W$-equivariant bijection $B\leftrightarrow R_{+}$
between the set of all $B\in\mathbf{B}(G)$ with $Z_{G}(S)\subset B$
and the set of all systems of positive roots $R_{+}\subset R$, given
by
\[
\Lie(B)=\mathfrak{g}_{0}\oplus\oplus_{\alpha\in R_{+}}\mathfrak{g}_{\alpha}.
\]

Fix one such $B$ and let $\Delta\subset R_{+}$ be the corresponding
set of simple roots. By~\cite[XXVI 7.7]{SGA3.3r}, there is an inclusion
preserving bijection $P\leftrightarrow A$ between the set of all
$P\in\mathbf{P}(G)$ with $B\subset P$ and the set of all subsets
$A$ of $\Delta$, given by 
\[
\Lie(P)=\mathfrak{g}_{0}\oplus\oplus_{\alpha\in R_{A}}\mathfrak{g}_{\alpha}
\]
where $R_{A}=R_{+}\coprod(\mathbb{Z}A\cap R_{-})$ is the set of roots
in $R=R_{+}\coprod R_{-}$ which are either positive or in the group
spanned by $A$. We write $P_{A}$ for the parabolic associated to
$A$. Since $f:\mathbf{SBP}(G)\rightarrow\mathbf{O}(G)$ is (very)
nice, we obtain a poset bijection
\[
f_{S,B}:\left(\{A\subset\Delta\},\subset\right)\rightarrow(\mathbf{O}(G),\leq),\quad A\mapsto t(P_{A}).
\]

Fix one such $P=P_{A}$. Then the fiber of $F:\mathbf{ACF}^{\Gamma}(G)\rightarrow\mathbf{SBP}(G)$
above $(S,B,P)$ is the set of all $(S,B,\mathcal{G})$ with $\mathcal{G}\in\mathbf{G}^{\Gamma}(S)=\Hom(M,\Gamma)$
such that 
\[
\forall\alpha\in\Delta:\quad\begin{cases}
\mathcal{G}(\alpha)=0 & \mbox{if }\alpha\in A,\\
\mathcal{G}(\alpha)>0 & \mbox{if }\alpha\notin A.
\end{cases}
\]
Since the elements of $\Delta$ are linearly independent and $\Gamma$
is non-trivial, this fiber is not empty and $F:\mathbf{ACF}^{\Gamma}(G)\rightarrow\mathbf{SBP}(G)$
is surjective.

\subsection{~\label{sub:DefFacetsClosedOpen}}

It follows that the five $F$'s in our diagram are surjective. Their
fibers are called facets, the type of a facet is its image in $\mathbf{O}(G)$,
and all facets of the same type are canonically isomorphic. The facets
of type $\circ$ are called chambers. For any $f':\mathbf{X}'\rightarrow\mathbf{C}^{\Gamma}(G)$
over $f:\mathbf{X}\rightarrow\mathbf{O}(G)$ in our diagram, the closed
facet of $x\in\mathbf{X}$ is $F^{-1}(\overline{x})\subset\mathbf{X}'$
where $\overline{x}=\{y\geq x\}$. It is a disjoint union of finitely
many facets. Since $x=\min FF^{-1}(\overline{x})$, closed facets
have a well-defined type and those of the same type are canonically
isomorphic. We equip the set of closed facets with the partial order
given by inclusion, which is opposite to the partial order on $\mathbf{X}$.
A closed chamber is a maximal closed facet, and the set of all closed
chambers equals $\mathbf{X}_{\mathrm{min}}=f^{-1}(\circ)$. Since
$f$ is nice, every $x\in\mathbf{X}_{\mathrm{min}}$ defines compatible
sections
\[
\xymatrix{\mathbf{X}'\ar@{->>}[r]_{f'}\ar[d]_{F} & \mathbf{C}^{\Gamma}(G)\ar@/_{.5pc}/[l]_{s_{x}}\ar[d]^{F}\\
\mathbf{X}\ar@{->>}[r]_{f} & \mathbf{O}(G)\ar@/_{.5pc}/[l]_{s_{x}}
}
\]
and the closed chamber $F^{-1}(\overline{x})$ is the image of $s_{x}:\mathbf{C}^{\Gamma}(G)\rightarrow\mathbf{X}'$.
Since $f$ is very nice, any $x'\in\mathbf{X}'$ belongs to some closed
chamber. Since $\mathsf{G}\backslash\mathbf{X}'=\mathbf{C}^{\Gamma}(G)$,
any closed chamber is a fundamental domain for the action of $\mathsf{G}$
on $\mathbf{X}'$.

\subsection{~}

The facets which are minimal among the set of non-minimal facets are
called panels. A panel $F^{-1}(x)$ bounds a chamber $F^{-1}(y)$
if $F^{-1}(x)\subset F^{-1}(\overline{y})$, i.e. $y\leq x$. Any
panel bounds at least $3$ chambers. Indeed, this means that a non-minimal
parabolic subgroup $P$ of $G$ contains at least $3$ minimal parabolic
subgroups. To establish this, fix a Levi subgroup $L$ of $P$ --
which exists by \cite[XXVI 2.3]{SGA3.3r} or the surjectivity of $p_{1}$.
Then $Q\mapsto L\cap Q$ yields a bijection between the parabolic
subgroups $Q$ of $G$ contained in $P$ and the parabolic subgroups
of $L$, by \cite[XXVI 1.20]{SGA3.3r}. Since $P$ is non-minimal,
$L$ is not a minimal parabolic subgroup of itself. By \cite[XXVI 5.11]{SGA3.3r},
it contains at least $3$ such subgroups, and so does $P$.

\subsection{\label{sub:apartments}~}

The apartment attached to $S\in\mathbf{S}(G)$ is the subset $\mathbf{F}^{\Gamma}(S)$
of all $\mathcal{F}$'s in $\mathbf{F}^{\Gamma}(G)$ such that $Z_{G}(S)\subset P_{\mathcal{F}}$.
It is canonically isomorphic to $\mathbf{G}^{\Gamma}(S)$ by the map
which sends $\mathcal{G}:\mathbb{D}_{\mathcal{O}}(\Gamma)\rightarrow S$
to $\Fil(\mathcal{G})$. Our notations are thus consistent since 
\[
\mathbf{F}^{\Gamma}(S)=\mathbf{G}^{\Gamma}(S)=\mathbb{G}^{\Gamma}(S)(\mathcal{O})=\mathbb{F}^{\Gamma}(S)(\mathcal{O}).
\]
Since $F:\mathbf{A}\mathbf{F}^{\Gamma}(G)\rightarrow\mathbf{SP}$
is surjective, $\mathbf{F}^{\Gamma}(S)$ is the disjoint union of
the facets $F^{-1}(P)$ with $Z_{G}(S)\subset P$. Since $Z_{G}(S)=B\cap B'$
for some pair of opposed minimal parabolic subgroups of $G$, $\mathbf{F}^{\Gamma}(S)$
determines $Z_{G}(S)=\cap_{F^{-1}(P)\subset\mathbf{F}^{\Gamma}(S)}P$
and its split radical $S$. Thus $S\mapsto\mathbf{F}^{\Gamma}(S)$
is an $\Aut(G)$-equivariant bijection from $\mathbf{S}(G)$ onto
the set $\mathbf{A}(G)$ of apartments in $\mathbf{F}^{\Gamma}(G)$.
In particular, 
\begin{eqnarray*}
\mathbf{AF}^{\Gamma}(G) & = & \left\{ (A,\mathcal{F}):A\in\mathbf{A}(G),\mathcal{F}\in A\right\} \\
\mathbf{ACF}^{\Gamma}(G) & = & \left\{ (A,C,\mathcal{F}):\mathcal{F}\in C=\mbox{closed chamber of }A\in\mathbf{A}(G)\right\} .
\end{eqnarray*}
Since $\mathbf{AF}^{\Gamma}(G)\rightarrow\mathbf{F}^{\Gamma}(G)$
is surjective, every $\mathcal{F}\in\mathbf{F}^{\Gamma}(G)$ belongs
to some $A\in\mathbf{A}(G)$. The stabilizer of $\mathbf{F}^{\Gamma}(S)$
in $\mathsf{G}=G(\mathcal{O})$ equals $\mathsf{N}_{G}(S)=N_{G}(S)(\mathcal{O})$
and its pointwise stabilizer equals $\mathsf{Z}_{G}(S)=Z_{G}(S)(\mathcal{O})$.
Thus $\mathsf{W}_{G}(S)=\mathsf{N}_{G}(S)/\mathsf{Z}_{G}(S)$ acts
on $\mathbf{F}^{\Gamma}(S)$, and this gives the usual action of $\mathsf{W}_{G}(S)$
on $\mathbf{F}^{\Gamma}(S)=\mathbf{G}^{\Gamma}(S)=\Hom(\mathbb{D}_{\mathcal{O}}(\Gamma),S)$.

\subsection{~}

A panel bounds exactly two chambers in any apartment which contains
it. Indeed, let $F^{-1}(Q)$ be a panel in $\mathbf{F}^{\Gamma}(S)$.
Given \cite[XXVI 1.20]{SGA3.3r}, we have to show that there are exactly
two minimal parabolic subgroups of $L=Q\cap\iota_{S}Q$ containing
$Z_{G}(S)$. By assumption, $\mathbf{O}(L)=\{\circ,t(L)\}$. Our claim
then follows from~\ref{sub:descFacet}.

\subsection{~\label{sub:StdAndApartments}}

For any $\mathcal{F}_{1},\mathcal{F}_{2}\in\mathbf{F}^{\Gamma}(G)$,
there is an apartment $A\in\mathbf{A}(G)$ containing $\mathcal{F}_{1}$
and $\mathcal{F}_{2}$ if and only if $P_{\mathcal{F}_{1}}$ and $P_{\mathcal{F}_{2}}$
are in standard position \cite[XXVI 4.5]{SGA3.3r}. Indeed if $\mathcal{F}_{1},\mathcal{F}_{2}\in\mathbf{F}^{\Gamma}(S)$
for some $S\in\mathbf{S}(G)$, then $Z_{G}(S)$ contains a maximal
torus $T$ by \cite[XIV 3.20]{SGA3.2}, thus $T\subset Z_{G}(S)\subset P_{\mathcal{F}_{1}}\cap P_{\mathcal{F}_{2}}$.
If conversely $T\subset P_{\mathcal{F}_{1}}\cap P_{\mathcal{F}_{2}}$
for some maximal torus $T$ of $G$, then $\mathcal{F}_{1},\mathcal{F}_{2}\in\mathbf{F}^{\Gamma}(S)$
for any $S\in\mathbf{S}(G)$ containing the maximal split torus $T_{sp}$
of $T$: if $R_{i}$ is the split radical of the unique Levi subgroup
$L_{i}$ of $P_{\mathcal{F}_{i}}$ containing $T$ \cite[XXVI 1.6]{SGA3.3r},
then $R_{i}\subset T_{sp}\subset S$, therefore $Z_{G}(S)\subset Z_{G}(R_{i})=L_{i}\subset P_{\mathcal{F}_{i}}$
by \cite[XXVI 6.11]{SGA3.3r}. We will denote by\nomenclature[Std(G)]{$\mathbf{Std}(G)$}{Set of all pairs of parabolic subgroups of $G$ in standard relative position, page \nomrefpage}\nomenclature[Std^Gamma(G)]{$\mathbf{Std}^{\Gamma}(G)$}{Set of all pairs of $\Gamma$-filtrations on $G$ in standard relative position, page \nomrefpage}
\[
\mathbf{Std}(G)=\mathbb{STD}(G)(\mathcal{O})\quad\mbox{and}\quad\mathbf{Std}^{\Gamma}(G)=\mathbb{STD}^{\Gamma}(G)(\mathcal{O})
\]
the corresponding subsets of $\mathbf{P}(G)^{2}$ and $\mathbf{F}^{\Gamma}(G)^{2}$,
so that
\[
\mathbf{Std}^{\Gamma}(G)=F^{-1}(\mathbf{Std}(G))=\cup_{S\in\mathbf{S}(G)}\mathbf{F}^{\Gamma}(S)\times\mathbf{F}^{\Gamma}(S)\subset\mathbf{F}^{\Gamma}(G)^{2}.
\]
For any $S\in\mathbf{S}(G)$, the map $+:\mathbf{Std}^{\Gamma}(G)\rightarrow\mathbf{F}^{\Gamma}(G)$
of section~\ref{sub:AdditionMap} induces the natural commutative
group structure on $\mathbf{F}^{\Gamma}(S)=\mathbf{G}^{\Gamma}(S)=\Hom(\mathbb{D}_{\mathcal{O}}(\Gamma),S)$.

\subsection{~\label{sub:DefGrPOnPoints}}

For $P\in\mathbf{P}(G)$ with unipotent radical $U$ and Levi $L$,
we also define
\[
\mathbf{Std}^{\Gamma}(P)=\left\{ \mathcal{F}\in\mathbf{F}^{\Gamma}(G):(P,P_{\mathcal{F}})\in\mathbf{Std}(G)\right\} =\cup_{Z_{G}(S)\subset P}\mathbf{F}^{\Gamma}(S).
\]
As explained in sections~\ref{sub:FunctLevi} and \ref{sub:defofGrMap},
the functorial map $\mathbf{F}^{\Gamma}(L)\rightarrow\mathbf{F}^{\Gamma}(G)$
lands in $\mathbf{Std}^{\Gamma}(P)$ and actually defines a section
of a $\mathsf{P}$-equivariant map
\[
\Gr_{P}:\mathbf{Std}^{\Gamma}(P)\twoheadrightarrow\mathbf{F}^{\Gamma}(P/U)
\]
which may be computed as follows: starting with $\mathcal{F}\in\mathbf{Std}^{\Gamma}(P)$,
pick $S\in\mathbf{S}(G)$ such that $Z_{G}(S)\subset P\cap P_{\mathcal{F}}$,
let $\mathcal{G}\in\mathbf{G}^{\Gamma}(S)$ be the corresponding splitting
of $\mathcal{F}$ and let $\overline{\mathcal{G}}$ be the image of
$\mathcal{G}$ in $\mathbf{G}^{\Gamma}(P/U)$. Then $\Gr_{P}(\mathcal{F})=\Fi(\overline{\mathcal{G}})$
in $\mathbf{F}^{\Gamma}(P/U)$. Thus for $\mathcal{F}\in F^{-1}(P)$,
$\mathcal{F}\in\mathbf{Std}^{\Gamma}(P)$ and $\Gr_{P}(\mathcal{F})=\overline{\mathcal{F}}$
with $\overline{\mathcal{F}}\in\mathbf{G}^{\Gamma}(\overline{R}(P))$
as in~\ref{sub:NotationsPFandFbar}.
\begin{thm*}
\cite[XXVI 4.1.1]{SGA3.3r} If $\mathcal{O}=K$ is a field, then $\mathbf{Std}(G)=\mathbf{P}(G)^{2}$,
thus also $\mathbf{Std}^{\Gamma}(G)=\mathbf{F}^{\Gamma}(G)^{2}$ and
$\Gr_{P}$ is defined on the whole of $\mathbf{F}^{\Gamma}(G)=\mathbf{Std}^{\Gamma}(P)$:
\[
\Gr_{P}:\mathbf{F}^{\Gamma}(G)\twoheadrightarrow\mathbf{F}^{\Gamma}(P/U)
\]

\end{thm*}

\subsection{~}

Suppose now that $\mathcal{O}$ is a Henselian local ring with residue
field $k$. 
\begin{prop}
\label{prop:Specialisation}The specialization from $\mathcal{O}$
to $k$ induces a map from the diagram of section~\ref{sub:TitsBuildDiag}
for $G$ to the similar diagram for $G_{k}$. In the resulting commutative
diagram, all the specialization maps $\mathbf{X}(G)\rightarrow\mathbf{X}(G_{k})$
are surjective, all the squares involving two $F$'s are cartesian,
and $\mathbf{O}(G)\simeq\mathbf{O}(G_{k})$, $\mathbf{C}^{\Gamma}(G)\simeq\mathbf{C}^{\Gamma}(G_{k})$.\end{prop}
\begin{proof}
Since $\mathbb{G}^{\Gamma}$, $\mathbb{F}^{\Gamma}$, $\mathbb{C}^{\Gamma}$,
$\mathbb{OPP}$, $\mathbb{P}$ and $\mathbb{O}$ are smooth over $\Spec(\mathcal{O})$,
the specialization from $\mathcal{O}$ to $k$ induces a map from
the last two squares of our diagram for $G$ to the last two squares
of the analogous diagram for $G_{k}$, in which all specialization
maps $\mathbf{X}(G)\rightarrow\mathbf{X}(G_{k})$ are surjective by
\cite[18.5.17]{EGA4.4}. Since $\mathbb{O}$ is finite étale over
$\Spec(\mathcal{O})$, $\mathbf{O}(G)\rightarrow\mathbf{O}(G_{k})$
is also injective by \cite[18.5.4-5]{EGA4.4}, i.e.~$\mathbf{O}(G)=\mathbf{O}(G_{k})$.
It follows that $\mathbf{P}(G)\twoheadrightarrow\mathbf{P}(G_{k})$
induces $\mathbf{B}(G)\twoheadrightarrow\mathbf{B}(G_{k})$. If $S$
is a maximal split torus in $G$, then $Z_{G}(S)$ is a Levi subgroup
of a minimal parabolic subgroup $B$ of $G$, $S$ is the maximal
split subtorus of the radical $R$ of $Z_{G}(S)$, thus $R/S$ is
an anisotropic torus, i.e.~$\Hom(\mathbb{G}_{m,\mathcal{O}},R/S)=0$.
Then by proposition~\ref{prop:RepGspCaseTorus}, lemma~\ref{lem:StructQuasIsoTwistedSch}
and \cite[18.5.4-5]{EGA4.4}, also $\Hom(\mathbb{G}_{m,k},R_{k}/S_{k})=0$,
thus $S_{k}$ is the maximal split subtorus of the radical $R_{k}$
of the Levi subgroup $Z_{G}(S)_{k}=Z_{G_{k}}(S_{k})$ of the minimal
parabolic subgroup $B_{k}$ of $G_{k}$, in particular $S_{k}$ is
a maximal split subtorus of $G_{k}$ and the specialization map $\mathbf{S}(G)\rightarrow\mathbf{S}(G_{k})$
is well-defined. It is surjective: starting with $\overline{S}$ in
$\mathbf{S}(G_{k})$, choose $\overline{B}\in\mathbf{B}(G_{k})$ containing
$Z_{G_{k}}(\overline{S})$, lift $\overline{B}$ to some $B\in\mathbf{B}(G)$,
choose $S'\in\mathbf{S}(G)$ with $Z_{G}(S')\subset B$, write $\overline{S}=\Int(\overline{b})(S'_{k})$
for some $b\in B(k)$, lift $\overline{b}$ to some $b\in B(\mathcal{O})$
using \cite[18.5.17]{EGA4.4} and set $S=\Int(b)(S')$. Then $S\in\mathbf{S}(G)$
and $S_{k}=\overline{S}$. The same argument shows that $\mathbf{SBP}(G)\rightarrow\mathbf{SBP}(G_{k})$
and $\mathbf{SP}(G)\rightarrow\mathbf{SP}(G_{k})$ are well-defined
and surjective, from which follows that also $\mathbf{ACF}^{\Gamma}(G)\rightarrow\mathbf{ACF}^{\Gamma}(G_{k})$
and $\mathbf{AF}^{\Gamma}(G)\rightarrow\mathbf{AF}^{\Gamma}(G_{k})$
are well-defined. To establish all of the remaining claims, it is
sufficient to show that $\mathbf{C}^{\Gamma}(G)\twoheadrightarrow\mathbf{C}^{\Gamma}(G_{k})$
is also injective, which again follows from~\cite[18.5.4-5]{EGA4.4}
since $\mathbb{C}^{\Gamma}(G)$ is separated and étale over $\mathcal{O}$.
Alternatively, fix $(S,B)$ as above and let $s:\mathbf{C}^{\Gamma}(G)\hookrightarrow\mathbf{F}^{\Gamma}(G)$
and $s_{k}:\mathbf{C}^{\Gamma}(G_{k})\hookrightarrow\mathbf{F}^{\Gamma}(G_{k})$
be the corresponding sections. They are compatible with the specialization
maps and there images are respectively contained in the apartments
$\mathbf{F}^{\Gamma}(S)$ of $\mathbf{F}^{\Gamma}(G)$ and $\mathbf{F}^{\Gamma}(S_{k})$
of $\mathbf{F}^{\Gamma}(G_{k})$. Since $\mathbf{G}^{\Gamma}(S)\simeq\mathbf{G}^{\Gamma}(S_{k})$,
the specialization map $\mathbf{F}^{\Gamma}(G)\rightarrow\mathbf{F}^{\Gamma}(G_{k})$
restricts to a bijection $\mathbf{F}^{\Gamma}(S)\simeq\mathbf{F}^{\Gamma}(S_{k})$,
therefore $\mathbf{C}^{\Gamma}(G)\twoheadrightarrow\mathbf{C}^{\Gamma}(G_{k})$
is indeed injective.
\end{proof}

\subsection{~}

Suppose now that $\mathcal{O}$ is a valuation ring with fraction
field $K$.
\begin{prop}
\label{prop:Generisation}The generization from $\mathcal{O}$ to
$K$ induces a map from the diagram of section~\ref{sub:TitsBuildDiag}
for $G$ to the similar diagram for $G_{K}$. In the resulting commutative
diagram, all the generization maps $\mathbf{X}(G)\rightarrow\mathbf{X}(G_{K})$
are injective, they are bijective for $\mathbf{X}\in\{\mathbf{F}^{\Gamma},\mathbf{C}^{\Gamma},\mathbf{P},\mathbf{O}\}$
and all the squares involving two $F$'s are cartesian.\end{prop}
\begin{proof}
Since $\mathbb{G}^{\Gamma}$, $\mathbb{F}^{\Gamma}$, $\mathbb{C}^{\Gamma}$,
$\mathbb{OPP}$, $\mathbb{P}$ and $\mathbb{O}$ are separated over
$\Spec(\mathcal{O})$, the generization from $\mathcal{O}$ to $K$
induces a map from the last two squares of our diagram for $G$ to
the last two squares of the analogous diagram for $G_{K}$, in which
all generization maps $\mathbf{X}(G)\rightarrow\mathbf{X}(G_{K})$
are injective. Since $\mathbb{O}$ and $\mathbb{P}$ are proper over
$\Spec(\mathcal{O})$, the maps $\mathbf{P}(G)\rightarrow\mathbf{P}(G_{K})$
and $\mathbf{O}(G)\rightarrow\mathbf{O}(G_{K})$ are in fact bijective.
It follows that $\mathbf{P}(G)\simeq\mathbf{P}(G_{K})$ induces $\mathbf{B}(G)\simeq\mathbf{B}(G_{K})$.
If $S$ is a maximal split torus in $G$, then $Z_{G}(S)$ is a Levi
subgroup of a minimal parabolic subgroup $B$ of $G$, $S$ is the
maximal split subtorus of the radical $R$ of $Z_{G}(S)$, thus $R/S$
is an anisotropic torus, i.e.~$\Hom(\mathbb{G}_{m,\mathcal{O}},R/S)=0$.
Then by proposition~\ref{prop:RepGspCaseTorus} and lemma~\ref{lem:StructQuasIsoTwistedSch},
also $\Hom(\mathbb{G}_{m,K},R_{K}/S_{K})=0$, thus $S_{K}$ is the
maximal split subtorus of the radical $R_{K}$ of the Levi subgroup
$Z_{G}(S)_{K}=Z_{G_{K}}(S_{K})$ of the minimal parabolic subgroup
$B_{K}$ of $G_{K}$, in particular $S_{K}$ is a maximal split subtorus
of $G_{K}$ and the specialization map $\mathbf{S}(G)\rightarrow\mathbf{S}(G_{K})$
is well-defined. It is injective by \cite[XXII 5.8.3]{SGA3.3r}, so
are $\mathbf{SBP}(G)\rightarrow\mathbf{SBP}(G_{K})$ and $\mathbf{SP}(G)\rightarrow\mathbf{SP}(G_{K})$,
while $\mathbf{ACF}^{\Gamma}(G)\rightarrow\mathbf{ACF}^{\Gamma}(G_{K})$
and $\mathbf{AF}^{\Gamma}(G)\rightarrow\mathbf{AF}^{\Gamma}(G_{K})$
are well-defined. To establish the remaining claims, it is sufficient
to show that $\mathbf{C}^{\Gamma}(G)\hookrightarrow\mathbf{C}^{\Gamma}(G_{K})$
is also surjective, which again follows from lemma~\ref{lem:StructQuasIsoTwistedSch}
since $\mathbb{C}^{\Gamma}(G)$ is a quasi-isotrivial twisted constant
scheme over $\mathcal{O}$. Alternatively, fix $(S,B)$ as above and
let $s:\mathbf{C}^{\Gamma}(G)\hookrightarrow\mathbf{F}^{\Gamma}(G)$
and $s_{K}:\mathbf{C}^{\Gamma}(G_{k})\hookrightarrow\mathbf{F}^{\Gamma}(G_{k})$
be the corresponding sections. They are compatible with the generization
maps and there images are respectively contained in the apartments
$\mathbf{F}^{\Gamma}(S)$ of $\mathbf{F}^{\Gamma}(G)$ and $\mathbf{F}^{\Gamma}(S_{K})$
of $\mathbf{F}^{\Gamma}(G_{K})$. Since $\mathbf{G}^{\Gamma}(S)\simeq\mathbf{G}^{\Gamma}(S_{K})$,
the generization map $\mathbf{F}^{\Gamma}(G)\rightarrow\mathbf{F}^{\Gamma}(G_{K})$
restricts to a bijection $\mathbf{F}^{\Gamma}(S)\simeq\mathbf{F}^{\Gamma}(S_{K})$,
therefore $\mathbf{C}^{\Gamma}(G)\hookrightarrow\mathbf{C}^{\Gamma}(G_{K})$
is indeed surjective.\end{proof}
\begin{rem}
Under the identifications of theorem~\ref{thm:MainTan} (note that
$G$ is isotrivial over the valuation ring $\mathcal{O}$ by proposition~\ref{pro:LocalUnibImpliesIsotriv}),
the inverse of $\mathbf{F}^{\Gamma}(G)\rightarrow\mathbf{F}^{\Gamma}(G_{K})$
maps a $\Gamma$-filtration $\mathcal{F}_{K}$ on $\omega_{K}^{\circ}$
to the $\Gamma$-filtration $\mathcal{F}$ on $\omega^{\circ}$ defined
by 
\[
\forall\tau\in\Rep^{\circ}(G)(\mathcal{O}),\,\gamma\in\Gamma:\quad\mathcal{F}^{\gamma}(\tau)=\mathcal{F}_{K}^{\gamma}(\tau)\cap V(\tau)\mbox{ in }V_{K}(\tau).
\]
It is not at all obvious that this formula indeed defines a right
exact functor!
\end{rem}

\subsection{~}

Set $\mathbf{Gen}(G)=\mathbb{GEN}(G)(\mathcal{O})$, $\mathbf{Gen}^{\Gamma}(G)=\mathbb{GEN}^{\Gamma}(G)(\mathcal{O})$
and define
\begin{eqnarray*}
\mathbf{Gen}(Y) & = & \left\{ P\in\mathbf{P}(G):\{P\}\times Y\subset\mathbf{Gen}(G)\right\} \\
\mathbf{Gen}^{\Gamma}(X) & = & \left\{ \mathcal{F}\in\mathbf{F}^{\Gamma}(G):\{\mathcal{F}\}\times X\subset\mathbf{Gen}^{\Gamma}(G)\right\} 
\end{eqnarray*}
for $Y\subset\mathbf{P}(G)$ and $X\subset\mathbf{F}^{\Gamma}(G)$.
We say that $Y$ (resp. $X$) is thin if
\[
t(\mathbf{Gen}(Y))=\mathbf{O}(G)\quad(\mbox{resp. }t(\mathbf{Gen}^{\Gamma}(X))=\mathbf{C}^{\Gamma}(G)).
\]
Plainly, $F^{-1}(\mathbf{Gen}(Y))=\mathbf{Gen}^{\Gamma}(F^{-1}(Y))$
and $F(\mathbf{Gen}^{\Gamma}(X))=\mathbf{Gen}(F(X))$, thus
\[
\left(Y\mbox{ is thin}\Leftrightarrow F^{-1}(Y)\mbox{ is thin}\right)\quad\mbox{and}\quad\left(X\mbox{ is thin}\Leftrightarrow F(X)\mbox{ is thin}\right).
\]
Moreover $Y$ is thin $\Leftrightarrow$ $\circ\in t(\mathbf{Gen}(Y))$
and a thin subset of a thin set is thin.
\begin{lem}
Suppose that $\mathcal{O}$ is a strictly Henselian valuation ring
with residue field $k$ and fraction field $K$. Then any subset $Y_{K}$
of $\mathbf{P}(G_{K})$ (resp. $X_{K}$ of $\mathbf{F}^{\Gamma}(G_{K})$)
whose image in $\mathbf{P}(G_{k})$ is finite is a thin subset of
$\mathbf{P}(G_{K})$ (resp. $\mathbf{F}^{\Gamma}(G_{K})$).\end{lem}
\begin{proof}
It is sufficient to treat the case of a subset $Y_{K}$ of $\mathbf{P}(G_{K})$.
Let $Y\simeq Y_{K}$ be its pre-image in $\mathbf{P}(G)$ with finite
image $Y_{k}$ in $\mathbf{P}(G_{k})$. For every $y\in Y_{k}$, 
\[
U_{y}=t^{-1}(\circ_{k})\times\{y\}\cap\mathbb{GEN}(G_{k})
\]
is a non-empty open subscheme of the (geometrically) irreducible $k$-scheme
$t^{-1}(\circ_{k})$ by \cite[XXVI 4.2.4.iii]{SGA3.3r}, thus $(\cap_{y\in Y_{k}}U_{y})(k)=\mathbf{B}(G_{k})\cap\mathbf{Gen}(Y_{k})$
is not empty since $k$ is algebraically closed and $t^{-1}(\circ_{k})$
is of finite type over $k$. If $B\in\mathbf{B}(G)$ lifts $B_{k}\in\mathbf{B}(G_{k})\cap\mathbf{Gen}(Y_{k})$,
then $\{B\}\times Y\subset\mathbf{Gen}(G)$ since $\mathbb{GEN}(G)$
is open in $\mathbb{P}(G)^{2}$, thus also $\{B_{K}\}\times Y_{K}\subset\mathbf{Gen}(G_{K})$,
i.e.~$B_{K}\in\mathbf{Gen}(Y_{K})$. Since also $B_{K}\in\mathbf{B}(G_{K})$,
$Y_{K}$ is indeed a thin subset of $\mathbf{P}(G_{K})$.
\end{proof}

\section{Distances and angles}

Suppose from now on that $\Gamma$ is a subring of $\mathbb{R}$ with
the induced total order on the underlying commutative group.

\subsection{~}

Recall from theorem~\ref{thm:MainTan} that for $\tau\in\Rep^{\circ}(G)(\mathcal{O})$,
any $\mathcal{F}\in\mathbf{F}^{\Gamma}(G)$ defines a $\Gamma$-filtration
$\mathcal{F}(\tau)$ on the (free) $\mathcal{O}$-module $V(\tau)$.
For any $(\mathcal{F}_{1},\mathcal{F}_{2})\in\mathbf{Std}^{\Gamma}(G)$
and $\gamma_{1},\gamma_{2}\in\Gamma$, the $\mathcal{O}$-module 
\[
\Gr_{\mathcal{F}_{1},\mathcal{F}_{2}}^{\gamma_{1},\gamma_{2}}(\tau)=\frac{\mathcal{F}_{1}^{\gamma_{1}}(\tau)\cap\mathcal{F}_{2}^{\gamma_{2}}(\tau)}{\mathcal{F}_{1,+}^{\gamma_{1}}(\tau)\cap\mathcal{F}_{2}^{\gamma_{2}}(\tau)+\mathcal{F}_{1}^{\gamma_{1}}(\tau)\cap\mathcal{F}_{2,+}^{\gamma_{2}}(\tau)}
\]
is free of finite rank: if $\mathcal{F}_{i}=\Fi(\mathcal{G}_{i})$
with $\mathcal{G}_{i}\in\mathbf{G}^{\Gamma}(S)=\Hom(M,\Gamma)$ for
some $S$ in $\mathbf{S}(G)$ with $M=\Hom(S,\mathbb{G}_{m,\mathcal{O}})$,
then $\mathcal{F}_{i}(\tau)^{\gamma}=\oplus_{\mathcal{G}_{i}(m)\geq\gamma}V(\tau)_{m}$
for any $i\in\{1,2\}$ and $\gamma\in\Gamma$ where $V(\tau)=\oplus_{m\in M}V(\tau)_{m}$
is the eigenspace decomposition of $\tau\vert_{S}$, thus 
\[
\Gr_{\mathcal{F}_{1},\mathcal{F}_{2}}^{\gamma_{1},\gamma_{2}}(\tau)=\oplus_{m:\mathcal{G}_{i}(m)=\gamma_{i}}V(\tau)_{m}.
\]

\subsection{~\label{sub:constscalangdist}}

Since $\Gamma$ is a subring of $\mathbb{R}$:\nomenclature[<-,->_tau]{$\left\langle -,-\right\rangle _{\tau}$}{Scalar product on $\mathbf{Std}^\Gamma (G)$ defined page \nomrefpage}\nomenclature[<(]{$\left\Vert -\right\Vert _{\tau}$}{Length on $\mathbf{F}^\Gamma (G)$ or $\mathbf{C}^\Gamma (G)$ defined page \nomrefpage}\nomenclature[d_tau]{$d_{\tau}$}{Distance on $\mathbf{Std}^\Gamma (G)$ defined page \nomrefpage}\nomenclature[<(-,-)_tau]{$\measuredangle_{\tau}(-,-)$}{Angle on $\mathbf{Std}^\Gamma (G)$ defined page \nomrefpage}

$\bullet$ Any apartment is endowed with a canonical structure of
free $\Gamma$-module, and these structures are preserved by the action
of $\Aut(G)$ on $\mathbf{F}^{\Gamma}(G)$. Indeed, 
\[
\mathbf{F}^{\Gamma}(S)=\mathbf{G}^{\Gamma}(S)=\Hom(M(S),\Gamma)\quad\mbox{with}\quad M(S)=\Hom(S,\mathbb{G}_{m,\mathcal{O}}).
\]

$\bullet$ Any $\tau\in\Rep^{\circ}(G)(\mathcal{O})$ defines a $\mathsf{G}$-invariant
function 
\[
\left\langle -,-\right\rangle _{\tau}:\mathbf{Std}^{\Gamma}(G)\rightarrow\Gamma,\qquad\left\langle \mathcal{F}_{1},\mathcal{F}_{2}\right\rangle _{\tau}={\textstyle \sum_{\gamma_{1},\gamma_{2}}}\mathrm{rank}_{\mathcal{O}}\left(\Gr_{\mathcal{F}_{1},\mathcal{F}_{2}}^{\gamma_{1},\gamma_{2}}(\tau)\right)\cdot\gamma_{1}\gamma_{2}
\]
whose restriction to $\mathbf{F}^{\Gamma}(S)$ is bilinear, symmetric
and non-negative, given by
\[
\left\langle \mathcal{F}_{1},\mathcal{F}_{2}\right\rangle _{\tau}={\textstyle \sum_{m\in M(S)}}\mathrm{rank}_{\mathcal{O}}\left(V(\tau)_{m}\right)\cdot\mathcal{G}_{1}(m)\mathcal{G}_{2}(m)
\]
if $\mathcal{F}_{i}\in\mathbf{F}^{\Gamma}(S)$ corresponds to $\mathcal{G}_{i}\in\Hom(M(S),\Gamma)$.
Its kernel equals $\mathbf{G}^{\Gamma}(\ker(\tau\vert S))$, thus
$\left\langle -,-\right\rangle _{\tau}$ is positive definite when
$\tau$ is a faithful representation of $G$.

$\bullet$ Write $\left\Vert \mathcal{F}\right\Vert _{\tau}=\left\langle \mathcal{F},\mathcal{F}\right\rangle _{\tau}^{1/2}$.
Thus $\left\Vert -\right\Vert _{\tau}:\mathbf{F}^{\Gamma}(G)\rightarrow\mathbb{R}_{+}$
is a $\mathsf{G}$-invariant function. It descends to a $\mathsf{G}$-invariant
function $\left\Vert -\right\Vert _{\tau}:\mathbf{C}^{\Gamma}(G)\rightarrow\mathbb{R}_{+}$
with $\left\Vert \mathcal{F}\right\Vert _{\tau}=\left\Vert t(\mathcal{F})\right\Vert _{\tau}$.
We have the Cauchy-Schwartz inequality
\[
\forall(\mathcal{F}_{1},\mathcal{F}_{2})\in\mathbf{Std}^{\Gamma}(G):\qquad\left|\left\langle \mathcal{F}_{1},\mathcal{F}_{2}\right\rangle _{\tau}\right|\leq\left\Vert \mathcal{F}_{1}\right\Vert _{\tau}\left\Vert \mathcal{F}_{2}\right\Vert _{\tau}.
\]

$\bullet$ We may thus also define a $\mathsf{G}$-invariant angle
\[
\measuredangle_{\tau}(-,-):\mathbf{Std}^{\Gamma}(G)\rightarrow[0,\pi],\qquad\left\langle \mathcal{F}_{1},\mathcal{F}_{2}\right\rangle _{\tau}=\cos\left(\measuredangle_{\tau}(\mathcal{F}_{1},\mathcal{F}_{2})\right)\cdot\left\Vert \mathcal{F}_{1}\right\Vert _{\tau}\left\Vert \mathcal{F}_{2}\right\Vert _{\tau}
\]
and a $\mathsf{G}$-invariant function
\[
d_{\tau}(-,-):\mathbf{Std}^{\Gamma}(G)\rightarrow\mathbb{R}_{+},\qquad d_{\tau}(\mathcal{F}_{1},\mathcal{F}_{2})=\sqrt{\left\Vert \mathcal{F}_{1}\right\Vert _{\tau}^{2}+\left\Vert \mathcal{F}_{2}\right\Vert _{\tau}^{2}-2\left\langle \mathcal{F}_{1},\mathcal{F}_{2}\right\rangle _{\tau}}
\]
inducing the distance $d_{\tau}(\mathcal{F}_{1},\mathcal{F}_{2})=\left\Vert \mathcal{F}_{2}-\mathcal{F}_{1}\right\Vert _{\tau}$
on any apartment. 

$\bullet$ If $\tau$ is faithful, then $d_{\tau}(\mathcal{F}_{1},\mathcal{F}_{2})=0$
if and only if $\mathcal{F}_{1}=\mathcal{F}_{2}$ and the map $\mathcal{G}\mapsto(\Fi(\mathcal{G}),\Fi(\iota\mathcal{G}))$
induces a $\mathsf{G}$-equivariant bijection 
\[
\mathbf{G}^{\Gamma}(G)\simeq\left\{ (\mathcal{F}_{1},\mathcal{F}_{2})\in\mathbf{Std}^{\Gamma}(G):\left\Vert \mathcal{F}_{1}\right\Vert _{\tau}=\left\Vert \mathcal{F}_{2}\right\Vert _{\tau}\mbox{ and }\measuredangle_{\tau}(\mathcal{F}_{1},\mathcal{F}_{2})=\pi\right\} .
\]

\subsection{~}

For $(\mathcal{F}_{1},\mathcal{F}_{2})\in\mathbf{Std}^{\Gamma}(G)$,
we also have the following formula
\[
\left\langle \mathcal{F}_{1},\mathcal{F}_{2}\right\rangle _{\tau}={\textstyle \sum_{\gamma}\gamma\cdot\deg\,}\Gr_{\mathcal{F}_{1}}^{\gamma}(\mathcal{F}_{2},\tau)={\textstyle \sum_{\gamma}\gamma\cdot\deg}\,\Gr_{\mathcal{F}_{2}}^{\gamma}(\mathcal{F}_{1},\tau)
\]
where $\Gr_{\mathcal{F}}^{\gamma}(\mathcal{G},\tau)$ is the filtration
induced by $\mathcal{G}(\tau)$ on $\Gr_{\mathcal{F}}^{\gamma}(\tau)=\mathcal{F}^{\gamma}(\tau)/\mathcal{F}_{+}^{\gamma}(\tau)$,
i.e.
\[
\forall\theta\in\Gamma,\qquad\Gr_{\mathcal{F}}^{\gamma}(\mathcal{G},\tau)^{\theta}=\mathcal{F}^{\gamma}(\tau)\cap\mathcal{G}^{\theta}(\tau)+\mathcal{F}_{+}^{\gamma}(\tau)/\mathcal{F}_{+}^{\gamma}(\tau)
\]
and the degree of an $\mathbb{R}$-filtration $\mathcal{F}$ on a
finite free $\mathcal{O}$-module $V$ is given by 
\[
\deg(\mathcal{F})={\textstyle \sum_{\gamma}}\gamma\cdot\mathrm{rank}_{\mathcal{O}}(\Gr_{\mathcal{F}}^{\gamma}).
\]
Choosing a splitting of $\mathcal{F}$, one checks easily that also
\[
\deg(\mathcal{F})=\deg(\det\mathcal{F})
\]
where $\det\mathcal{F}$ is the $\mathbb{R}$-filtration on $\det V=\Lambda_{\mathcal{O}}^{r}V$,
$r=\mathrm{rank}_{\mathcal{O}}V$ defined by 
\[
(\det\mathcal{F})^{\gamma}=\mbox{span of }\left\{ v_{1}\wedge\cdots\wedge v_{r}:v_{i}\in\mathcal{F}^{\gamma_{i}},{\textstyle \sum}\gamma_{i}=\gamma\right\} .
\]

\subsection{~\label{sub:decisortho4F(G)}}

If $\Gamma$ is a $\mathbb{Q}$-vector space, the decomposition 
\[
\mathbb{F}^{\Gamma}(G)=\mathbb{F}^{\Gamma}(G^{\mathrm{der}})\times\mathbb{G}^{\Gamma}(Z(G))
\]
of section~\ref{sub:Isogenies} induces an analogous decomposition
\[
\mathbf{F}^{\Gamma}(G)=\mathbf{F}^{\Gamma}(G^{\mathrm{der}})\times\mathbf{G}^{\Gamma}(Z(G))
\]
which is orthogonal in the following sense: for $\mathcal{F}_{1},\mathcal{F}_{2}\in\mathbf{Std}^{\Gamma}(G)$
and $\mathcal{F}\in\mathbf{F}^{\Gamma}(G)$, 
\begin{eqnarray*}
\left\langle \mathcal{F}_{1},\mathcal{F}_{2}\right\rangle _{\tau} & = & \left\langle \mathcal{F}_{1}^{r},\mathcal{F}_{2}^{r}\right\rangle _{\tau}+\left\langle \mathcal{F}_{1}^{c},\mathcal{F}_{2}^{c}\right\rangle _{\tau}\\
d_{\tau}(\mathcal{F}_{1},\mathcal{F}_{2})^{2} & = & d_{\tau}(\mathcal{F}_{1}^{r},\mathcal{F}_{2}^{r})^{2}+d_{\tau}(\mathcal{F}_{1}^{c},\mathcal{F}_{2}^{c})^{2}\\
\left\Vert \mathcal{F}\right\Vert _{\tau}^{2} & = & \left\Vert \mathcal{F}^{r}\right\Vert _{\tau}^{2}+\left\Vert \mathcal{F}^{c}\right\Vert _{\tau}^{2}
\end{eqnarray*}
where $\mathcal{F}^{r}\in\mathbf{F}^{\Gamma}(G^{\mathrm{der}})\subset\mathbf{F}^{\Gamma}(G)$
and $\mathcal{F}^{c}\in\mathbf{G}^{\Gamma}(Z(G))\subset\mathbf{F}^{\Gamma}(G)$
are the components of $\mathcal{F}$. We prove the first formula.
Pick $S\in\mathbf{S}(G)$ with $\mathcal{F}_{1},\mathcal{F}_{2}\in\mathbf{F}^{\Gamma}(S)$.
Then also $\mathcal{F}_{i}^{r},\mathcal{F}_{i}^{c}\in\mathbf{F}^{\Gamma}(S)$
with $\mathcal{F}_{i}=\mathcal{F}_{i}^{r}+\mathcal{F}_{i}^{c}$ in
the apartment $\mathbf{F}^{\Gamma}(S)$, thus 
\[
\left\langle \mathcal{F}_{1},\mathcal{F}_{2}\right\rangle _{\tau}=\left\langle \mathcal{F}_{1}^{r},\mathcal{F}_{2}^{r}\right\rangle _{\tau}+\left\langle \mathcal{F}_{1}^{c},\mathcal{F}_{2}^{c}\right\rangle _{\tau}+\left\langle \mathcal{F}_{1}^{r},\mathcal{F}_{2}^{c}\right\rangle _{\tau}+\left\langle \mathcal{F}_{1}^{c},\mathcal{F}_{2}^{r}\right\rangle _{\tau}
\]
since $\left\langle -,-\right\rangle _{\tau}$ is a bilinear form
on $\mathbf{F}^{\Gamma}(S)$. It is therefore sufficient to show that
\[
\left\langle \mathcal{F},\mathcal{G}\right\rangle _{\tau}=\left\langle \mathcal{G},\mathcal{F}\right\rangle _{\tau}={\textstyle \sum_{\gamma}}\gamma\cdot\deg\Gr_{\mathcal{G}}^{\gamma}(\mathcal{F},\tau)=0
\]
for any $\mathcal{F}\in\mathbf{F}^{\Gamma}(G^{\mathrm{der}})$ and
$\mathcal{G}\in\mathbf{G}^{\Gamma}(Z(G))$. Since $\mathcal{G}:\mathbb{D}_{\mathcal{O}}(\Gamma)\rightarrow Z(G)$
is central in $G$, $\tau=\oplus\tau_{\gamma}$ with $V(\tau_{\gamma})=\mathcal{G}_{\gamma}(\tau)$
and $\Gr_{\mathcal{G}}^{\gamma}(\mathcal{F},\tau)\simeq\mathcal{F}(\tau_{\gamma})$
on $\Gr_{\mathcal{G}}^{\gamma}(\tau)\simeq\tau_{\gamma}$, thus 
\[
\deg\Gr_{\mathcal{G}}^{\gamma}(\mathcal{F},\tau)=\deg\mathcal{F}(\tau_{\gamma})=\deg\left(\det\mathcal{F}(\tau_{\gamma})\right)=\deg\left(\mathcal{F}(\det\tau_{\gamma})\right)=0
\]
because the restriction of $\det\tau_{\gamma}$ to $G^{\mathrm{der}}$
is trivial.

\subsection{~\label{sub:angleCGamma}}

For $x,y\in\mathbf{O}(G)$, there is a single $\mathsf{G}$-orbit
of $(P,Q)$'s in $t^{-1}(x)\times t^{-1}(y)$ such that $P$ and $Q$
are in osculatory (resp.~transverse) position \cite[XXVI 5.3-5]{SGA3.3r},
and this orbit is contained in $\mathbf{Std}(G)$. Thus for any $x,y\in\mathbf{C}^{\Gamma}(G)$,
there is a single $\mathsf{G}$-orbit of $(\mathcal{F}_{1},\mathcal{F}_{2})$'s
in $t^{-1}(x)\times t^{-1}(y)$ with the property that $P_{\mathcal{F}_{1}}$
and $P_{\mathcal{F}_{2}}$ are in osculatory (resp.~transverse) position,
and it is contained in $\mathbf{Std}^{\Gamma}(G)$. We set $\measuredangle_{\tau}^{os}(x,y)=\measuredangle_{\tau}(\mathcal{F}_{1},\mathcal{F}_{2})$
and $\left\langle x,y\right\rangle _{\tau}^{os}=\left\langle \mathcal{F}_{1},\mathcal{F}_{2}\right\rangle _{\tau}$
(resp.~$\measuredangle_{\tau}^{tr}(x,y)=\measuredangle_{\tau}(\mathcal{F}_{1},\mathcal{F}_{2})$
and $\left\langle x,y\right\rangle _{\tau}^{tr}=\left\langle \mathcal{F}_{1},\mathcal{F}_{2}\right\rangle _{\tau}$),
thus obtaining two other pairs of symmetric functions\nomenclature[<-,->_tau^os]{$\left\langle -,-\right\rangle _{\tau}^{os}$}{Osculatory scalar product on $\mathbf{C}^\Gamma (G)$ defined page \nomrefpage}\nomenclature[<-,->_tau^tr]{$\left\langle -,-\right\rangle _{\tau}^{tr}$}{Transverse scalar product on $\mathbf{C}^\Gamma (G)$ defined page \nomrefpage}\nomenclature[<(-,-)_tau^os]{$\measuredangle_{\tau} ^{os}(-,-)$}{Osculatory angle on $\mathbf{C}^\Gamma (G)$ defined page \nomrefpage}\nomenclature[<(-,-)_tau^tr]{$\measuredangle_{\tau} ^{tr}(-,-)$}{Transverse angle on $\mathbf{C}^\Gamma (G)$ defined page \nomrefpage}
\begin{eqnarray*}
\measuredangle_{\tau}^{os}(-,-):\mathbf{C}^{\Gamma}(G)\times\mathbf{C}^{\Gamma}(G)\rightarrow[0,\pi] & \mbox{and} & \left\langle -,-\right\rangle _{\tau}^{os}:\mathbf{C}^{\Gamma}(G)\times\mathbf{C}^{\Gamma}(G)\rightarrow\mathbb{R},\\
\measuredangle_{\tau}^{tr}(-,-):\mathbf{C}^{\Gamma}(G)\times\mathbf{C}^{\Gamma}(G)\rightarrow[0,\pi] & \mbox{and} & \left\langle -,-\right\rangle _{\tau}^{tr}:\mathbf{C}^{\Gamma}(G)\times\mathbf{C}^{\Gamma}(G)\rightarrow\mathbb{R}.
\end{eqnarray*}
They are of course related by the formulas
\begin{eqnarray*}
\left\langle x,y\right\rangle _{\tau}^{os} & = & \cos\left(\measuredangle_{\tau}^{os}(x,y)\right)\cdot\left\Vert x\right\Vert _{\tau}\left\Vert y\right\Vert _{\tau},\\
\left\langle x,y\right\rangle _{\tau}^{tr} & = & \cos\left(\measuredangle_{\tau}^{tr}(x,y)\right)\cdot\left\Vert x\right\Vert _{\tau}\left\Vert y\right\Vert _{\tau}.
\end{eqnarray*}
We also define yet another symmetric function
\[
d_{\tau}:\mathbf{C}^{\Gamma}(G)\times\mathbf{C}^{\Gamma}(G)\rightarrow\mathbb{R}_{+}\quad d_{\tau}(x,y)=\sqrt{\left\Vert x\right\Vert _{\tau}^{2}+\left\Vert y\right\Vert _{\tau}^{2}-2\left\langle x,y\right\rangle _{\tau}^{os}}.
\]
By construction, for $(\mathcal{F}_{1},\mathcal{F}_{2})$ in $\mathbf{Gen}^{\Gamma}(G)=\mathbb{GEN}^{\Gamma}(G)(\mathcal{O})$,
\[
\measuredangle_{\tau}(\mathcal{F}_{1},\mathcal{F}_{2})=\measuredangle_{\tau}^{tr}\left(t(\mathcal{F}_{1}),t(\mathcal{F}_{2})\right)\quad\mbox{and}\quad\left\langle \mathcal{F}_{1},\mathcal{F}_{2}\right\rangle _{\tau}=\left\langle \mathcal{F}_{1},\mathcal{F}_{2}\right\rangle _{\tau}^{tr}.
\]

\subsection{~}

The above constructions are merely special cases of those of section~\ref{sub:RanksRelativePos}.
In particular, our functions are induced by morphisms of schemes over
$\mathcal{O}$, and some of them (the ``bilinear forms'') still
make sense for an arbitrary $\tau$ in the Grothendieck ring $K_{0}(G)$
of $\Rep^{\circ}(G)(\mathcal{O})$, or even for a global section of
the sheafified version $\underline{K}_{0}(G)$ of that ring. However,
the positivity of these forms requires some sort of effectiveness/faithfulness
of the initial $\tau$, which we have not tried to axiomatize. From
now on until \ref{sub:EquivDistOnF}, we fix a faithful $\tau$ in
$\Rep^{\circ}(G)(\mathcal{O})$.

\subsection{~\label{sub:calcangleCGamma}}

Fix $(S,B)\in\mathbf{S}(G)\times\mathbf{B}(G)$ with $Z_{G}(S)\subset B$,
let $B'=\iota_{S}B$ be the minimal parabolic subgroup of $G$ containing
$Z_{G}(S)$ opposed to $B$, and denote by 
\[
s,s':\mathbf{C}^{\Gamma}(G)\hookrightarrow\mathbf{F}^{\Gamma}(G)
\]
the corresponding sections of $t:\mathbf{F}^{\Gamma}(G)\twoheadrightarrow\mathbf{C}^{\Gamma}(G)$.
Then for $x,y\in\mathbf{C}^{\Gamma}(G)$, 
\[
B\subset P_{s(x)}\cap P_{s(y)}\quad\mbox{and}\quad B'\subset P_{s'(x)}\cap P_{s'(y)},
\]
thus $P_{s(x)}$ and $P_{s(y)}$ are in osculatory position while
$P_{s(x)}$ and $P_{s'(y)}$ are in transverse position. It follows
that for every $x,y\in\mathbf{C}^{\Gamma}(G)$,
\[
\begin{array}{rclcrcl}
\measuredangle_{\tau}^{os}(x,y) & = & \measuredangle_{\tau}(s(x),s(y)) & \mbox{and} & \left\langle x,y\right\rangle _{\tau}^{os} & = & \left\langle s(x),s(y)\right\rangle _{\tau},\\
\measuredangle_{\tau}^{tr}(x,y) & = & \measuredangle_{\tau}(s(x),s'(y)) & \mbox{and} & \left\langle x,y\right\rangle _{\tau}^{tr} & = & \left\langle s(x),s'(y)\right\rangle _{\tau}.
\end{array}
\]
In particular, the ``scalar products'' are compatible with the monoid
structure:{\small{
\[
\begin{array}{rclcrcl}
\left\langle x_{1}+x_{2},y\right\rangle _{\tau}^{os} & = & \left\langle x_{1},y\right\rangle _{\tau}^{os}+\left\langle x_{2},y\right\rangle _{\tau}^{os} & \mbox{and} & \left\langle x,y_{1}+y_{2}\right\rangle _{\tau}^{os} & = & \left\langle x,y_{1}\right\rangle _{\tau}^{os}+\left\langle x,y_{2}\right\rangle _{\tau}^{os},\\
\left\langle x_{1}+x_{2},y\right\rangle _{\tau}^{tr} & = & \left\langle x_{1},y\right\rangle _{\tau}^{tr}+\left\langle x_{2},y\right\rangle _{\tau}^{tr} & \mbox{and} & \left\langle x,y_{1}+y_{2}\right\rangle _{\tau}^{tr} & = & \left\langle x,y_{1}\right\rangle _{\tau}^{tr}+\left\langle x,y_{2}\right\rangle _{\tau}^{tr}.
\end{array}
\]
}}Moreover, $d_{\tau}$ is a distance on $\mathbf{C}^{\Gamma}(G)$.

\subsection{~}

The following lemma is related to the angle rigidity axiom of \cite[4.1.2]{KlLe97}.
\begin{lem}
\label{lem:AngleRigid}For any $x,y\in\mathbf{C}^{\Gamma}(G)$, the
set 
\[
D_{\tau}(x,y)=\left\{ \measuredangle_{\tau}(\mathcal{F}_{1},\mathcal{F}_{2}):(\mathcal{F}_{1},\mathcal{F}_{2})\in\mathbf{Std}^{\Gamma}(G)\cap t^{-1}(x)\times t^{-1}(y)\right\} 
\]
is finite with 
\[
\min D_{\tau}(x,y)=\measuredangle_{\tau}^{os}(x,y)\quad\mbox{and}\quad\max D_{\tau}(x,y)=\measuredangle_{\tau}^{tr}(x,y).
\]
\end{lem}
\begin{proof}
Fix $(S,B)$ and $s,s':\mathbf{C}^{\Gamma}(G)\hookrightarrow\mathbf{F}^{\Gamma}(G)$
as above. Then any pair 
\[
(\mathcal{F}_{1},\mathcal{F}_{2})\in\mathbf{Std}^{\Gamma}(G)\cap t^{-1}(x)\times t^{-1}(y)
\]
 is $\mathsf{G}$-conjugated to some pair in $\mathsf{W}_{G}(S)\cdot s(x)\times\mathsf{W}_{G}(S)\cdot s(y)\subset\mathbf{F}^{\Gamma}(S)^{2}$,
thus 
\begin{eqnarray*}
D_{\tau}(x,y) & = & \left\{ \measuredangle_{\tau}(w_{1}\cdot s(x),w_{2}\cdot s(y)):\left(w_{1},w_{2}\right)\in\mathsf{W}_{G}(S)^{2}\right\} \\
 & = & \left\{ \measuredangle_{\tau}(s(x),w\cdot s(y)):w\in\mathsf{W}_{G}(S)\right\} 
\end{eqnarray*}
is finite. To establish our final claim, we have to show that 
\[
\left\langle s(x),s(y)\right\rangle _{\tau}\geq\left\langle s(x),w\cdot s(y)\right\rangle _{\tau}\geq\left\langle s(x),s'(y)\right\rangle _{\tau}
\]
for every $w\in\mathsf{W}_{G}(S)$, which follows from \cite[Proposition 18]{BoLie46}.\end{proof}
\begin{cor}
\label{cor:tIsnonexpanding}The type map $t:\mathbf{F}^{\Gamma}(G)\twoheadrightarrow\mathbf{C}^{\Gamma}(G)$
is compatible with the $d_{\tau}$'s: 
\[
\forall(\mathcal{F}_{1},\mathcal{F}_{2})\in\mathbf{Std}^{\Gamma}(G):\qquad d_{\tau}\left(t(\mathcal{F}_{1}),t(\mathcal{F}_{2})\right)\leq d_{\tau}(\mathcal{F}_{1},\mathcal{F}_{2}).
\]

\end{cor}

\subsection{~}

Let us use the above notions to show that
\begin{prop}
For a facet $F$, a chamber $C$ and apartments $A_{1},A_{2}$ in
$\mathbf{F}^{\Gamma}(G)$ with $F\cup C\subset A_{1}\cap A_{2}$,
there exists $g\in\mathsf{G}$ with $gA_{1}=A_{2}$ and $g\equiv1$
on $\overline{F}\cup\overline{C}$.\end{prop}
\begin{proof}
In group theoretical terms, this means that for $P\in\mathbf{P}(G)$,
$B\in\mathbf{B}(G)$ and $S_{1},S_{2}\in\mathbf{S}(G)$ with $Z_{G}(S_{i})\subset B\cap P$,
there is a $g\in\mathsf{G}$ such that $\Int(g)(S_{1})=S_{2}$ and
$g\in\mathsf{B}\cap\mathsf{P}$ with $\mathsf{B}=B(\mathcal{O})$,
$\mathsf{P}=P(\mathcal{O})$. This does not depend upon $\Gamma$,
and we may thus assume that $\Gamma=\mathbb{R}$. Since $(S_{1},B)$
and $(S_{2},B)\in\mathbf{SP}(G)$ have the same image in $\mathbf{O}(G)$,
there exists an element $g\in\mathsf{G}$ with $g(S_{1},B)=(S_{2},B)$,
i.e. $\Int(g)(S_{1})=S_{2}$ and $g\in\mathsf{B}$. We will show that
also $g\in\mathsf{P}$, i.e.~$g\mathcal{F}=\mathcal{F}$ for any
$\mathcal{F}\in F^{-1}(P)\subset\mathbf{F}^{\mathbb{R}}(G)$. Note
that $\mathcal{F}$, $g\mathcal{F}$ and the chamber $C=F^{-1}(B)$
are all contained in the apartment $\mathbf{F}^{\mathbb{R}}(S_{2})$.
Fix a faithful $\tau\in\Rep^{\circ}(G)(\mathcal{O})$. Then
\[
\left\langle \mathcal{F},\mathcal{F}'\right\rangle _{\tau}=\left\langle g\mathcal{F},g\mathcal{F}'\right\rangle _{\tau}=\left\langle g\mathcal{F},\mathcal{F}'\right\rangle _{\tau}
\]
for all $\mathcal{F}'\in F^{-1}(B)$, thus $\mathcal{F}=g\mathcal{F}$
because $F^{-1}(B)$ is a non-empty open subset of the Euclidean space
$\left(\mathbf{F}^{\mathbb{R}}(S_{2}),\left\langle -,-\right\rangle _{\tau}\right)$
by \ref{sub:descFacet}.
\end{proof}

\subsection{~\label{sub:FiloverFieldsisTitsBuild}}

Suppose for this and the next subsection that our local ring $\mathcal{O}=k$
is a field. Then every pair $(\mathcal{F}_{1},\mathcal{F}_{2})\in\mathbf{F}^{\Gamma}(G)$
is contained in some apartment since 
\begin{thm}
\cite[XXVI 4.1.1]{SGA3.3r}\label{thm:Std=00003DAlloverField} $\mathbf{Std}(G)=\mathbf{P}(G)^{2}$
and $\mathbf{Std}^{\Gamma}(G)=\mathbf{F}^{\Gamma}(G)^{2}$.\end{thm}
\begin{cor}
\label{cor:AxiomeR(i)4F(G)}For any apartments $A_{1},A_{2}$ in $\mathbf{F}^{\Gamma}(G)$
and facets $F_{1},F_{2}$ in $A_{1}\cap A_{2}$, there exists $g\in\mathsf{G}$
mapping $A_{1}$ to $A_{2}$ with $g\equiv1$ on $\overline{F}_{1}\cup\overline{F}_{2}$.\end{cor}
\begin{proof}
Fix closed chambers $F_{1}\subset\overline{C}_{1}\subset A_{1}$ and
$F_{2}\subset\overline{C}_{2}\subset A_{2}$ and choose an apartment
$A_{3}$ containing $C_{1}$ and $C_{2}$. The previous proposition
shows that there exists elements $g_{1},g_{2}\in\mathsf{G}$ such
that $g_{1}A_{1}=A_{3}=g_{2}A_{2}$, $g_{1}\equiv1$ on $\overline{C}_{1}\cup\overline{F}_{2}$
and $g_{2}\equiv1$ on $\overline{C}_{2}\cup\overline{F}_{1}$. Then
$g=g_{2}^{-1}g_{1}$ maps $A_{1}$ to $A_{2}$ and $g\equiv1$ on
$\overline{F}_{1}\cup\overline{F}_{2}$.\end{proof}
\begin{cor}
\label{cor:FmonoIsmono}For a monomorphism $f:G_{1}\rightarrow G_{2}$
of reductive groups over $k$, the induced map $f:\mathbf{F}^{\Gamma}(G_{1})\rightarrow\mathbf{F}^{\Gamma}(G_{2})$
is injective. \end{cor}
\begin{proof}
Fix a faithful $\tau\in\Rep^{\circ}(G_{2})(k)$. Then $f^{\ast}\tau=\tau\circ f\in\Rep^{\circ}(G_{1})(k)$
is also faithful and for every $\mathcal{F},\mathcal{F}'\in\mathbf{F}^{\Gamma}(G_{1})$,
\[
\left\langle f(\mathcal{F}),f(\mathcal{F}')\right\rangle _{\tau}=\left\langle \mathcal{F},\mathcal{F}'\right\rangle _{f^{\ast}\tau}\quad\mbox{and}\quad d_{\tau}(f(\mathcal{F}),f(\mathcal{F}'))=d_{f^{\ast}\tau}(\mathcal{F},\mathcal{F}').
\]
Therefore $f(\mathcal{F})=f(\mathcal{F}')$ implies $\mathcal{F}=\mathcal{F}'$. \end{proof}
\begin{cor}
\label{cor:Ret}Let $P$ be a parabolic subgroup of $G$ with unipotent
radical $U$ and Levi subgroup $L$. Then $\mathbf{F}^{\Gamma}(L)$
is a fundamental domain for the action of $U(k)$ on $\mathbf{F}^{\Gamma}(G)$.
Let\nomenclature[r_P,L]{$r_{P,L}$}{Retraction $r_{P,L}:\mathbf{F}^\Gamma (G) \twoheadrightarrow \mathbf{F}^\Gamma (L)$, defined page \nomrefpage}
$r=r_{P,L}:\mathbf{F}^{\Gamma}(G)\twoheadrightarrow\mathbf{F}^{\Gamma}(L)$
be the corresponding retraction. Then 
\[
\forall x,y\in\mathbf{F}^{\Gamma}(G):\qquad d_{\tau}(rx,ry)\leq d_{\tau}(x,y).
\]

\begin{cor}
\label{cor:Dist}The function $d_{\tau}:\mathbf{F}^{\Gamma}(G)\times\mathbf{F}^{\Gamma}(G)\rightarrow\mathbb{R}_{+}$
is a distance: 
\[
\forall x,y,z\in\mathbf{F}^{\Gamma}(G):\qquad d_{\tau}(x,y)\leq d_{\tau}(x,z)+d_{\tau}(z,y).
\]

\end{cor}
\end{cor}
\begin{proof}
Fix $S_{0}\in\mathbf{S}(L)$. The $\mathsf{P}=P(k)$ and $\mathsf{L}=L(k)$
orbits of $S_{0}$ in $\mathbf{S}(G)$ are respectively equal to $\mathbf{S}(G,P)=\{S\in\mathbf{S}(G):Z_{G}(S)\subset P\}$
and $\mathbf{S}(L)$. Since any $\mathcal{F}\in\mathbf{F}^{\Gamma}(G)$
belongs to $\mathbf{F}^{\Gamma}(S)$ for some $S\in\mathbf{S}(G,P)$,
we find that with $\mathsf{U}=U(k)$, 
\[
\mathbf{F}^{\Gamma}(G)=\cup_{S\in\mathbf{S}(G,P)}\mathbf{F}^{\Gamma}(S)=\mathsf{P}\cdot\mathbf{F}^{\Gamma}(S_{0})=\mathsf{U}\cdot\cup_{S\in\mathbf{S}(L)}\mathbf{F}^{\Gamma}(S)=\mathsf{U}\cdot\mathbf{F}^{\Gamma}(L).
\]
Suppose that $\mathcal{F},u\mathcal{F}\in\mathbf{F}^{\Gamma}(L)$
for some $u\in\mathsf{U}$, and choose an $S\in\mathbf{S}(L)$ with
$\mathcal{F},u\mathcal{F}\in\mathbf{F}^{\Gamma}(S)$. Since $Z_{G}(S)\subset L\subset P$,
there is a $B\in\mathbf{B}(G)$ with $Z_{G}(S)\subset B\subset P$.
Let $C=F^{-1}(B)$ be the corresponding ($G$-)chamber in $A=\mathbf{F}^{\Gamma}(S)$.
Since $U\subset B$, $uC=C$ and $\mathcal{F},C\in A\cap u^{-1}A$.
Choose $g\in\mathsf{G}$ with $gu^{-1}A=A$, $g\mathcal{F}=\mathcal{F}$
and $gC=C$. Then $g$ belongs to $\mathsf{B}=B(k)$, thus $gu^{-1}$
belongs to $\mathsf{B}\cap\mathsf{N}_{G}(S)=\mathsf{Z}_{G}(S)$ which
acts trivially on $A$. Therefore $u\mathcal{F}=gu^{-1}u\mathcal{F}=g\mathcal{F}=\mathcal{F}$
and $\mathbf{F}^{\Gamma}(L)$ is a fundamental domain for the action
of $\mathsf{U}$ on $\mathbf{F}^{\Gamma}(G)$. 

For $A\in\mathbf{A}(G)$ containing $F^{-1}(P)$, there is a unique
$A_{L}\in\mathbf{A}(L)\cap\mathsf{U}\cdot\{A\}$ such that $r(x)=ux$
for any $x\in A$ and $u\in\mathsf{U}$ such that $uA=A_{L}$. Indeed,
there is a $p=lu$ in $\mathsf{P}=\mathsf{LU}$ such that $pA$ is
an apartment of $\mathbf{F}^{\Gamma}(L)$, then $uA=l^{-1}pA\subset\mathbf{F}^{\Gamma}(L)$
and $r(x)=ux$ for every $x\in A$. Thus for $x,y\in A$, $d_{\tau}(rx,ry)=d_{\tau}(x,y)$. 

For the remaining claims, we may assume that $\Gamma=\mathbb{R}$
and use induction on the semi-simple rank $s$ of $G$. If $s=0$
everything is obvious. If $s>0$ but $G=L$, then $r$ is the identity
thus $d_{\tau}(rx,ry)=d_{\tau}(x,y)$ for every $x,y\in\mathbf{F}^{\mathbb{R}}(G)$.
If $G\neq L$, choose an apartment $A$ in $\mathbf{F}^{\mathbb{R}}(G)$
containing $x$ and $y$, let $[x,y]$ be the corresponding segment
of $A$, and write $[x,y]=\cup_{i=0}^{n-1}[x_{i},x_{i+1}]$ for consecutive
points $x_{i}\in[x,y]$ with $x_{0}=x$, $x_{n}=y$ and $]x_{i},x_{i+1}[$
contained in a facet $F_{i}\subset A$. Then there is an apartment
containing $F^{-1}(P)$ and $\{x_{i},x_{i+1}\}\subset\overline{F}_{i}$,
thus $d_{\tau}(rx_{i},rx_{i+1})=d_{\tau}(x_{i},x_{i+1})$ for every
$i\in\{0,\cdots,n-1\}$. Since $d_{\tau}$ is a distance on $\mathbf{F}^{\mathbb{R}}(L)$
by our induction hypothesis, 
\[
d_{\tau}(rx,ry)\leq{\textstyle \sum_{i=0}^{n-1}}d_{\tau}(rx_{i},rx_{i+1})={\textstyle \sum_{i=0}^{n-1}}d_{\tau}(x_{i},x_{i+1})=d_{\tau}(x,y).
\]
Finally for $x,y,z\in\mathbf{F}^{\mathbb{R}}(G)$, choose an apartment
$\mathbf{F}^{\mathbb{R}}(S)$ containing $x$, $y$ and a chamber
$F^{-1}(B)$, let $r=r_{B,Z_{G}(S)}$ be the corresponding retraction.
Then 
\[
d_{\tau}(x,y)=d_{\tau}(rx,ry)\leq d_{\tau}(rx,rz)+d_{\tau}(rz,ry)\leq d_{\tau}(x,z)+d_{\tau}(z,y).
\]
This finishes the proof of corollaries~\ref{cor:Ret} and \ref{cor:Dist}.\end{proof}
\begin{cor}
\label{Cor:F(G)completeCAT0}If $\Gamma=\mathbb{R}$, then $(\mathbf{F}^{\mathbb{R}}(G),d_{\tau})$
is a complete CAT(0)-space. \end{cor}
\begin{proof}
Plainly, $(\mathbf{C}^{\mathbb{R}}(G),d_{\tau})$ is a complete metric
space. Let $(x_{n})$ be a Cauchy sequence in $\left(\mathbf{F}^{\mathbb{R}}(G),d_{\tau}\right)$.
Then $t(x_{n})$ is a Cauchy sequence in $(\mathbf{C}^{\mathbb{R}}(G),d_{\tau})$
by corollary~\ref{cor:tIsnonexpanding}, it thus converges to some
$y\in\mathbf{C}^{\mathbb{R}}(G)$. Now choose for each $n$ a minimal
pair $(S_{n},B_{n})\in\mathbf{SP}(G)$ corresponding to a section
$s_{n}:\mathbf{C}^{\mathbb{R}}(G)\rightarrow\mathbf{F}^{\mathbb{R}}(G)$
passing through $x_{n}$, and let $y_{n}=s_{n}(y)$. Then $d_{\tau}(x_{n},y_{n})=d_{\tau}(t(x_{n}),y)$
converges to $0$ and $y_{n}$ is also a Cauchy sequence in $\mathbf{F}^{\mathbb{R}}(G)$.
But $d_{\tau}(y_{n},y_{m})$ takes finitely many values by lemma~\ref{lem:AngleRigid},
therefore $y_{n}$ is stationary and $x_{n}$ converges to its limit.
Thus $(\mathbf{F}^{\mathbb{R}}(G),d_{\tau})$ is complete, and a geodesic
space by theorem~\ref{thm:Std=00003DAlloverField}. Finally, for
any triple $x,y,z$ in $\mathbf{F}^{\mathbb{R}}(G)$, choose a minimal
pair $(S,B)\in\mathbf{SP}(G)$ such that $x,y\in\mathbf{F}^{\mathbb{R}}(S)$
and the middle point $m$ of the segment $[x,y]$ of $\mathbf{F}^{\mathbb{R}}(S)$
belongs to $F^{-1}(B)$. Let $r=r_{B,Z_{G}(S)}:\mathbf{F}^{\mathbb{R}}(G)\rightarrow\mathbf{F}^{\mathbb{R}}(S)$
be the corresponding retraction and pick $u$ in $\mathsf{U}=U(k)$
with $uz=r(z)$, where $U$ is the unipotent radical of $B$. Then{\small{
\begin{eqnarray*}
d_{\tau}(z,m)^{2}=d_{\tau}(uz,um)^{2}=d_{\tau}(r(z),m)^{2} & = & {\textstyle \frac{1}{2}}d_{\tau}(r(z),x)^{2}+{\textstyle \frac{1}{2}}d_{\tau}(r(z),y)^{2}-{\textstyle \frac{1}{4}}d_{\tau}(x,y)^{2}\\
 & \leq & {\textstyle \frac{1}{2}}d_{\tau}(z,x)^{2}+{\textstyle \frac{1}{2}}d_{\tau}(z,y)^{2}-{\textstyle \frac{1}{4}}d_{\tau}(x,y)^{2}
\end{eqnarray*}
}}thus $\mathbf{F}^{\mathbb{R}}(G)$ is a CAT(0)-space by \cite[II.1.9]{BrHa99}.\end{proof}
\begin{cor}
\label{cor:ConcavOfScalProd}For any $\mathcal{F}\in\mathbf{F}^{\mathbb{R}}(G)$,
the function $\mathcal{G}\mapsto\left\langle \mathcal{F},\mathcal{G}\right\rangle _{\tau}$
from $\mathbf{F}^{\mathbb{R}}(G)$ to $\mathbb{R}$ is homogeneous,
concave and $\left\Vert \mathcal{F}\right\Vert _{\tau}$-Lipschitzian.\end{cor}
\begin{proof}
Homogeneity means that $\left\langle \mathcal{F},t\mathcal{G}\right\rangle _{\tau}=t\left\langle \mathcal{F},\mathcal{G}\right\rangle _{\tau}$
for all $t\in\mathbb{R}_{+}$, which is obvious from the definitions.
Concavity means that for any $\mathcal{G}_{0}$, $\mathcal{G}_{1}\in\mathbf{F}^{\mathbb{R}}(G)$
and $t\in[0,1]$, if $\mathcal{G}_{t}$ is the unique point at distance
$td_{\tau}(\mathcal{G}_{0},\mathcal{G}_{1})$ from $\mathcal{G}_{0}$
on the geodesic segment $[\mathcal{G}_{0},\mathcal{G}_{1}]$ of the
uniquely geodesic space $(\mathbf{F}^{\mathbb{R}}(G),d_{\tau})$ \cite[II.1.4]{BrHa99},
then
\[
\left\langle \mathcal{F},\mathcal{G}_{t}\right\rangle _{\tau}\geq t\left\langle \mathcal{F},\mathcal{G}_{1}\right\rangle _{\tau}+(1-t)\left\langle \mathcal{F},\mathcal{G}_{0}\right\rangle _{\tau}.
\]
Let $(0,f,g_{0},g_{1})$ be a comparison tetrahedron for $(0,\mathcal{F},\mathcal{G}_{0},\mathcal{G}_{1})$
in the Euclidean space $\mathbb{R}^{3}$, by which we mean that the
lengths of the edges containing $0$ and the angles between them are
the same for both tetrahedron. Then the lengths of the other three
edges are also the same for both tetrahedron, since every triangle
$(0,\mathcal{X},\mathcal{Y})$ in $\mathbf{F}^{\mathbb{R}}(G)$ is
flat by theorem~\ref{thm:Std=00003DAlloverField}. In particular,
$(f,g_{0},g_{1})$ is a comparison triangle for $(\mathcal{F},\mathcal{G}_{0},\mathcal{G}_{1})$,
thus $d_{\tau}(\mathcal{F},\mathcal{G}_{t})\leq d(f,g_{t})$ where
$g_{t}=tg_{1}+(1-t)g_{0}$ in $\mathbb{R}^{3}$ by the previous corollary.
Since $\left\Vert \mathcal{G}_{t}\right\Vert =\left\Vert g_{t}\right\Vert $
(because $(0,\mathcal{G}_{0},\mathcal{G}_{1})$ is flat), it follows
that 
\[
\left\langle \mathcal{F},\mathcal{G}_{t}\right\rangle _{\tau}\geq\left\langle f,g_{t}\right\rangle =t\left\langle f,g_{1}\right\rangle +(1-t)\left\langle f,g_{0}\right\rangle =t\left\langle \mathcal{F},\mathcal{G}_{1}\right\rangle _{\tau}+(1-t)\left\langle \mathcal{F},\mathcal{G}_{0}\right\rangle _{\tau}.
\]
Similarly, we find that 
\[
\left|\left\langle \mathcal{F},\mathcal{G}_{1}\right\rangle _{\tau}-\left\langle \mathcal{F},\mathcal{G}_{0}\right\rangle _{\tau}\right|=\left|\left\langle f,g_{1}\right\rangle -\left\langle f,g_{0}\right\rangle \right|\leq\left\Vert f\right\Vert \left\Vert g_{1}-g_{0}\right\Vert =\left\Vert \mathcal{F}\right\Vert _{\tau}\cdot d_{\tau}(\mathcal{G}_{0},\mathcal{G}_{1})
\]
thus $\mathcal{G}\mapsto\left\langle \mathcal{F},\mathcal{G}\right\rangle _{\tau}$
is indeed $\left\Vert \mathcal{F}\right\Vert _{\tau}$-Lipschitzian.\end{proof}
\begin{cor}
\label{cor:IneqScalProdF(G)}For any $\mathcal{F},\mathcal{G},\mathcal{H}\in\mathbf{F}^{\mathbb{R}}(G)$,
\[
\left\langle \mathcal{F},\mathcal{G}+\mathcal{H}\right\rangle _{\tau}\geq\left\langle \mathcal{F},\mathcal{G}\right\rangle _{\tau}+\left\langle \mathcal{F},\mathcal{H}\right\rangle _{\tau}.
\]
\end{cor}
\begin{proof}
Apply the previous lemma to the middle point $\mathcal{G}+\mathcal{H}$
of $[2\mathcal{G},2\mathcal{H}]$.\end{proof}
\begin{rem}
We could pursue here with many further corollaries, but our knowledgeable
readers will recognize that already with corollary~\ref{cor:AxiomeR(i)4F(G)},
we have established that $\mathbf{F}^{\mathbb{R}}(G)$, together with
its collections of apartments and facets (and the function $d_{\tau}$
for some choice of a faithful $\tau$), is a (discrete) Euclidean
building in the sense of \cite[6.1]{Ro09}. It is the vectorial (Tits)
building defined in \cite[10.6]{Ro09}. But the construction given
there singles out a pair $Z_{G}(S)\subset B$ and uses more of the
finest results from~\cite{BoTi65}: $\mathbf{F}^{\mathbb{R}}(G)$
is the building associated to the saturated Tits system $(\mathsf{G},\mathsf{B},\mathsf{N})=(G,B,N_{G}(S))(k)$.
By contrast, we may retrieve some of the results of \cite{BoTi65}
using the strongly transitive and strongly type-preserving action
of $\mathsf{G}$ on our globally constructed building $\mathbf{F}^{\mathbb{R}}(G)$,
for instance the fact that $(\mathsf{G},\mathsf{B},\mathsf{N})$ is
indeed a saturated Tits system \cite[8.6]{Ro09}. The main advantage
of our construction is however that it is plainly functorial in $G$
and $k$.
\begin{rem}
Corollary~\ref{cor:FmonoIsmono} also immediately follows from proposition~\ref{prop:GandFpreserveClosedImm}
together with \cite[VIB 1.4.2]{SGA1r}.
\end{rem}
\end{rem}

\subsection{~\label{sub:EquivDistOnF}}

If $\tau'$ is another faithful representation of $G$, the distances
$d_{\tau'}$ and $d_{\tau}$ are equivalent. One checks it first on
a fixed apartment $A$, thus obtaining constants $c,C>0$ such that
$cd_{\tau}(x,y)\leq d_{\tau'}(x,y)\leq Cd_{\tau}(x,y)$ for $x,y\in A$.
Then this holds true for every $x,y\in\mathbf{F}^{\Gamma}(G)$, since
any such pair is $\mathsf{G}$-conjugated to one in $A$. We thus
obtain a canonical metrizable $\mathsf{G}$-invariant topology on
$\mathbf{F}^{\Gamma}(G)$. The $\mathsf{G}$-invariant functions of
section~\ref{sub:constscalangdist} are continuous with respect to
the canonical topology. The apartments and the ``closed facets''
of section~\ref{sub:DefFacetsClosedOpen} are topologically closed,
being complete for the induced metrics. The canonical topology on
$\mathbf{C}^{\Gamma}(G)$ is the quotient topology of the canonical
topology on $\mathbf{G}^{\Gamma}(G)$, it is compatible with the monoid
structure on $\mathbf{C}^{\Gamma}(G)$, the sections defined by the
``closed chambers'' are homeomorphisms and the functions defined
in section~\ref{sub:angleCGamma} are continuous.

\subsection{~\label{sub:specialisationOfScalarProd}}

Suppose now that our local ring $\mathcal{O}$ is an integral domain
with fraction field $K$ and residue field $k$, giving rise to morphisms
of cartesian squares$\xyR{1pc}\xyC{1.5pc}$ 
\[
\xymatrix{\mathbf{F}^{\Gamma}(G_{K})\ar@{->>}[rd]^{t}\ar@{->>}[dd]_{F} &  & \mathbf{F}^{\Gamma}(G)\ar@{->>}[rd]^{t}\ar@{-}[d]_{F}\ar[rr]\ar[ll] &  & \mathbf{F}^{\Gamma}(G_{k})\ar@{->>}[rd]^{t}\ar@{-}[d]_{F}\\
 & \mathbf{C}^{\Gamma}(G_{K})\ar@{->>}[dd]\sb(0.3){F} & \,\ar@{->>}[d] & \mathbf{C}^{\Gamma}(G)\ar@{->>}[dd]\sb(0.3){F}\ar[rr]\ar[ll] & \,\ar@{->>}[d] & \mathbf{C}^{\Gamma}(G_{k})\ar@{->>}[dd]_{F}\\
\mathbf{P}(G_{K})\ar@{->>}[rd]^{t} & \,\ar[l] & \mathbf{P}(G)\ar@{->>}[rd]^{t}\ar@{-}[r]\ar@{-}[l] & \,\ar[r] & \mathbf{P}(G_{k})\ar@{->>}[rd]^{t}\\
 & \mathbf{O}(G_{K}) &  & \mathbf{O}(G)\ar[rr]\ar[ll] &  & \mathbf{O}(G_{k})
}
\]
We write $x\mapsto x_{K}$ for the generization maps, $x\mapsto x_{k}$
for the specialization maps. 
\begin{prop}
For any faithful $\tau\in\Rep^{\circ}(G)(\mathcal{O})$ and $x,y\in\mathbf{F}^{\Gamma}(G)$,
\begin{eqnarray*}
\left\langle x_{k},y_{k}\right\rangle _{\tau_{k}} & \geq & \left\langle x_{K},y_{K}\right\rangle _{\tau_{K}}\\
\measuredangle_{\tau_{k}}\left(x_{k},y_{k}\right) & \leq & \measuredangle_{\tau_{K}}\left(x_{K},y_{K}\right)\\
d_{\tau_{k}}\left(x_{k},y_{k}\right) & \leq & d_{\tau_{K}}(x_{K},y_{K})\\
\left\Vert x_{k}\right\Vert _{\tau_{k}} & = & \left\Vert x_{K}\right\Vert _{\tau_{K}}
\end{eqnarray*}
\end{prop}
\begin{proof}
We may assume that $\Gamma=\mathbb{R}$. For $(x,y)\in\mathbf{Std}^{\mathbb{R}}(G)$,
one checks easily that all of the above inequalities are in fact equalities.
In particular, 
\[
\left\Vert x_{k}\right\Vert _{\tau_{k}}=\left\Vert x_{K}\right\Vert _{\tau_{K}}
\]
for all $x\in\mathbf{F}^{\Gamma}(G)$. For an arbitrary pair $(x,y)$
in $\mathbf{F}^{\mathbb{R}}(G)$, the facet decomposition of $\mathbf{F}^{\mathbb{R}}(G)$
induces a decomposition of the segment $[x,y]=\cup_{i=0}^{n-1}[x_{i},x_{i+1}]$
as in the proof of corollary~\ref{cor:Ret}, with $(x_{i},x_{i+1})\in\mathbf{Std}^{\mathbb{R}}(G)$
for every $i$. Thus 
\[
d_{\tau_{K}}(x_{K},y_{K})={\textstyle \sum}_{i=0}^{n-1}d_{\tau_{K}}(x_{i,K},x_{i+1,K})={\textstyle \sum}_{i=0}^{n-1}d_{\tau_{k}}(x_{i,k},x_{i+1,k})\geq d_{\tau_{k}}(x_{k},y_{k})
\]
and the other two inequalities easily follow.\end{proof}
\begin{rem}
On the other hand for any $x,y\in\mathbf{C}^{\Gamma}(G)$,
\[
\begin{array}{rclcrcl}
\left\langle x_{k},y_{k}\right\rangle _{\tau_{k}}^{os} & = & \left\langle x_{K},y_{K}\right\rangle _{\tau_{K}}^{os} & \mbox{and} & \measuredangle_{\tau_{k}}^{os}(x_{k},y_{k}) & = & \measuredangle_{\tau_{K}}^{os}(x_{K},y_{K}),\\
\left\langle x_{k},y_{k}\right\rangle _{\tau_{k}}^{tr} & = & \left\langle x_{K},y_{K}\right\rangle _{\tau_{K}}^{tr} & \mbox{and} & \measuredangle_{\tau_{k}}^{tr}(x_{k},y_{k}) & = & \measuredangle_{\tau_{K}}^{tr}(x_{K},y_{K}).
\end{array}
\]
However, it does happen that $D_{\tau_{k}}(x_{k},y_{k})\neq D_{\tau_{K}}(x_{K},y_{K})$.\end{rem}

\chapter{Affine $\mathbf{F}(G)$-buildings}

Let $G$ be a reductive group over a field $K$. From now on, $\Gamma=(\mathbb{R},+,\leq)$
and we drop it from our notations. We also fix a faithful finite dimensional
representation $\tau$ of $G$ and drop it from the notations of section~\ref{sub:constscalangdist}.
We use \textsf{sans-serif} fonts to denote the set of $K$-valued
points of a $K$-scheme, as in $\mathsf{G}=G(K)$, $\mathsf{P}=P(K)$
etc\ldots{} \nomenclature[G]{$\mathsf{G}$}{$\mathsf{G} = G(\mathcal{O})$ for a group scheme $G$ over a local ring $\mathcal{O}$.}

\section{The dominance order\label{sub:DomOrder}}

\subsection{~}

Since $\Gamma=\mathbb{R}$ is divisible, the weak and strong partial
dominance order on $\mathbb{C}^{\Gamma}(G)$ agree. We denote by $\leq$
the induced partial order on the commutative monoid $\mathsf{C}(G)=\mathbb{C}(G)(K)$
or its submonoid $\mathbf{C}(G)=t(\mathbf{F}(G))$. It is compatible
with the monoid structure on $\mathbf{C}(G)$ and related to the decomposition
\[
\mathbf{C}(G)=\mathbf{C}^{r}(G)\times\mathbf{G}(Z)\quad\mbox{with}\quad\mathbf{C}^{r}(G)=\mathbb{C}(G)^{r}(K)\mbox{ and }\mathbf{G}(Z)=\mathbb{C}(G)^{c}(K)
\]
of section~\ref{sub:Isogenies} as follows: for $x=(x^{r},x^{c})$
and $y=(y^{r},y^{c})$ in $\mathbf{C}^{r}(G)\times\mathbf{G}(Z)$,
\begin{eqnarray*}
x\leq y & \iff & x^{r}\leq y^{r}\quad\mbox{and}\quad x^{c}=y^{c}.
\end{eqnarray*}
The poset $(\mathbf{C}(G),\leq)$ is a lattice and $\mathbf{G}(Z)\subset\mathbf{C}(G)$
is its subset of minimal elements.

\subsection{~}

Choose a minimal pair $(S,B)$ in $\mathbf{SP}(G)$, giving rise to
the relative based root data $\mathscr{R}_{+}=(M,R,M^{\ast},R^{\ast};R_{+})$
with Weyl group $\mathsf{W}_{G}(S)=W_{G}(S)(K)$, and to the partial
dominance order $\leq$ on $\Hom^{+}(M,\mathbb{R})$ as defined in
section~\ref{sub:startcompabsoluterelativedomorder}. Let also $s:\mathbf{C}(G)\hookrightarrow\mathbf{F}(S)$
be the corresponding section of $t:\mathbf{G}(G)\twoheadrightarrow\mathbf{C}(G)$,
whose image $C=s(\mathbf{C}(G))$ equals $\Hom^{+}(M,\mathbb{R})$
inside $\Hom(M,\mathbb{R})=\mathbf{F}(S)$ by~section~\ref{sub:descFacet}.
Let finally $C^{\ast}$ be the dual cone of $C$ in $\mathbf{F}(S)$
with respect to the scalar product $\left\langle -,-\right\rangle $
on $\mathbf{F}(S)$ which is attached to our chosen $\tau$, so that
\[
C^{\ast}=\left\{ t\in\mathbf{F}(S):\forall c\in C,\,\left\langle t,c\right\rangle \geq0\right\} .
\]
Then for every $x,y\in\mathbf{C}(G)$,
\begin{eqnarray*}
x\leq y & \iff & s(x)\leq s(y)\mbox{\,\ in }\Hom^{+}(M,\mathbb{R}),\\
 & \iff & s(x)\mbox{\,\ belongs to the convex hull of }\mathsf{W}_{G}(S)\cdot s(y),\\
 & \iff & s(y)-s(x)\mbox{ belongs to the dual cone }C^{\ast},\\
 & \iff & \forall z\in\mathbf{C}(G):\quad\left\langle x,z\right\rangle ^{os}\leq\left\langle y,z\right\rangle ^{os},\\
 & \iff & \forall z\in\mathbf{C}(G):\quad\left\langle x,z\right\rangle ^{tr}\geq\left\langle y,z\right\rangle ^{tr}.
\end{eqnarray*}
The first equivalence follows from~Proposition~\ref{prop:compdomorders},
the second from section~\ref{sub:CaractDomByConvEnv}, the third
one from \cite[12.14]{AtBo83} and the last two from the formulas
of section \ref{sub:calcangleCGamma}. The equivalence of the first
and third line on the right is actually a tautology, since in fact
$C^{\ast}=\mathbb{R}_{+}R_{+}^{\ast}$ in $\mathbf{F}(S)=\Hom(M,\mathbb{R})$.
Indeed for $\alpha\in R$, let $\alpha_{v}$ be the unique element
of $\mathbf{F}(S)$ such that $\left\langle x,\alpha_{v}\right\rangle =x(\alpha)$
for every $x\in\mathbf{F}(S)$. Then $s_{\alpha}$ is the orthogonal
reflection of $\mathbf{F}(S)$ with respect to the hyperplane $\alpha_{v}^{\bot}$,
thus $\alpha^{\ast}=\frac{\alpha_{v}}{\left\langle \alpha_{v},\alpha_{v}\right\rangle }$
in $\mathbf{F}(S)$. Since $C=\left\{ x\in\mathbf{F}(S):\forall\alpha\in R_{+},\,\left\langle x,\alpha_{v}\right\rangle \geq0\right\} $,
its dual cone $C^{\ast}$ is spanned by the $\alpha_{v}$'s for $\alpha\in R_{+}$
and thus $C^{\ast}=\mathbb{R}_{+}R_{+}^{\ast}$.

\subsection{~}

For every $x,y\in\mathbf{C}(G)$, we have 
\begin{eqnarray*}
x\leq y & \Longrightarrow & \left\Vert x\right\Vert ^{2}\leq\left\langle x,y\right\rangle ^{os}\leq\left\Vert y\right\Vert ^{2}\\
x=y & \iff & x\leq y\mbox{ and }\left\Vert x\right\Vert =\left\Vert y\right\Vert .
\end{eqnarray*}
Indeed the first implication follows from the equivalence 
\[
x\leq y\iff\forall z\in\mathbf{C}(G):\quad\left\langle x,z\right\rangle ^{os}\leq\left\langle y,z\right\rangle ^{os}
\]
with $z=x$ or $y$, and with $s$ as above, it says that 
\[
x\leq y\Longrightarrow\left\Vert s(x)\right\Vert ^{2}\leq\left\Vert s(x)\right\Vert \left\Vert s(y)\right\Vert \cos\measuredangle(s(x),s(y))\leq\left\Vert s(y)\right\Vert ^{2}.
\]
Thus if $x\leq y$ and $\left\Vert x\right\Vert =\left\Vert y\right\Vert $,
$s(x)=s(y)$ and $x=y$.

\subsection{~}

The next proposition slightly refines proposition~\ref{prop:CompAdditionType}.
\begin{prop}
\label{prop:TRinF(G)}For every $\mathcal{F},\mathcal{G}\in\mathbf{F}(G)$,
we have 
\[
t(\mathcal{F}+\mathcal{G})\leq t(\mathcal{F})+t(\mathcal{G})\quad\mbox{in}\quad(\mathbf{C}(G),\leq)
\]
with equality if and only if $\mathcal{F},\mathcal{G}\in C$ for some
closed chamber $C$ of $\mathbf{F}(G)$.\end{prop}
\begin{proof}
With notations as above, we may choose $(S,B)$ such that $\mathcal{F},\mathcal{G}\in\mathbf{F}(S)$
with $\mathcal{F}+\mathcal{G}\in C$, $C=s(\mathbf{C}(G))$. Set $\mathcal{F}'=s\circ t(\mathcal{F})$
and $\mathcal{G}'=s\circ t(\mathcal{G})$, so that
\[
\mathcal{F}+\mathcal{G}=s\left(t(\mathcal{F}+\mathcal{G})\right)\quad\mbox{and}\quad\mathcal{F}'+\mathcal{G}'=s\left(t(\mathcal{F})+t(\mathcal{G})\right).
\]
The (acute) dual cone $C^{\ast}$ defines a partial order $\leq$
on $\mathbf{F}(S)$, given by 
\[
x\leq y\iff y-x\in C^{\ast}.
\]
Since $\mathcal{F}'\in\left(\mathsf{W}_{G}(S)\cdot\mathcal{F}\right)\cap C$
and $\mathcal{G}'\in(\mathsf{W}_{G}(S)\cdot\mathcal{G})\cap C$, we
have 
\[
\mathcal{F}\leq\mathcal{F}'\quad\mbox{and}\quad\mathcal{G}\leq\mathcal{G}'
\]
by \cite[VI \S 1 Proposition 18]{BoLie46} (or lemma~\ref{lem:CarDominantByDomOrder}).
Thus $\mathcal{F}+\mathcal{G}\leq\mathcal{F}'+\mathcal{G}'$ with
equality if and only if $\mathcal{F}=\mathcal{F}'$ and $\mathcal{G}=\mathcal{G}'$,
i.e.~$\mathcal{F}$ and $\mathcal{G}$ belong to $C$.
\end{proof}

\subsection{~}

The above inequality can also be established and somehow refined as
follows. For every $z\in\mathbf{C}(G)$, there is an $\mathcal{H}\in\mathbf{F}(G)$
with $t(\mathcal{H})=z$ such that $\mathcal{H}$ and $\mathcal{F}+\mathcal{G}$
are in (relative) transverse position, see~\ref{sub:angleCGamma}.
For any such $\mathcal{H}$, \ref{sub:calcangleCGamma}, lemma~\ref{lem:AngleRigid}
and corollary~\ref{cor:IneqScalProdF(G)} together imply that
\begin{eqnarray*}
\left\langle t(\mathcal{F})+t(\mathcal{G}),z\right\rangle ^{tr} & = & \left\langle t(\mathcal{F}),z\right\rangle ^{tr}+\left\langle t(\mathcal{G}),z\right\rangle ^{tr}\\
 & \leq & \left\langle \mathcal{F},\mathcal{H}\right\rangle +\left\langle \mathcal{G},\mathcal{H}\right\rangle \\
 & \leq & \left\langle \mathcal{F}+\mathcal{G},\mathcal{H}\right\rangle \\
 & = & \left\langle t(\mathcal{F}+\mathcal{G}),z\right\rangle ^{tr}.
\end{eqnarray*}
Thus indeed $t(\mathcal{F}+\mathcal{G})\leq t(\mathcal{F})+t(\mathcal{G})$
in $(\mathbf{C}(G),\leq)$.

\section{Affine $\mathbf{F}(G)$-spaces and buildings}

\subsection{~}

Affine $\mathbf{F}(G)$-spaces interact with the vectorial Tits building
$\mathbf{F}(G)$ in the same way as affine spaces do with their underlying
vector space. 
\begin{defn}
An affine $\mathbf{F}(G)$-space is a set $\mathbf{X}(G)$ equipped
with: 
\begin{itemize}
\item a left action $\mathsf{G}\times\mathbf{X}(G)\rightarrow\mathbf{X}(G)$,
written $(g,x)\mapsto g\cdot x$ or $gx$,
\item a $\mathsf{G}$-equivariant \emph{pull} map $\mathbf{X}(G)\times\mathbf{F}(G)\rightarrow\mathbf{X}(G)$,
written $(x,\mathcal{F})\mapsto x+\mathcal{F}$,
\item a $\mathsf{G}$-equivariant \emph{apartment} map $\mathbf{S}(G)\rightarrow\mathcal{P}(\mathbf{X}(G))$,
written $S\mapsto\mathbf{X}(S)$,
\end{itemize}
such that for (one or) every $S\in\mathbf{S}(G)$, the pull map sends
$\mathbf{X}(S)\times\mathbf{F}(S)$ to $\mathbf{X}(S)$ and induces
a structure of affine $\mathbf{G}(S)$-space (in the usual sense)
on $\mathbf{X}(S)$. 
\end{defn}

\subsection{~}

The group $\mathsf{N}_{G}(S)$ thus acts on $\mathbf{X}(S)$ by affine
morphisms, the vectorial part of this action equals $\nu_{S}^{v}:\mathsf{N}_{G}(S)\twoheadrightarrow\mathsf{W}_{G}(S)\subset\Aut(\mathbf{G}(S))$
and the kernel $\mathsf{Z}_{G}(S)$ of $\nu_{S}^{v}$ acts on $\mathbf{X}(S)$
by translations, through a $\mathsf{W}_{G}(S)$-equivariant morphism
\[
\nu_{\mathbf{X},S}:\mathsf{Z}_{G}(S)\rightarrow\mathbf{G}(S).
\]
For any other $S'\in\mathbf{S}(G)$, there is commutative diagram
\[
\xyC{2pc}\xyR{2pc}\xymatrix{\mathsf{Z}_{G}(S)\ar[r]^{\nu_{\mathbf{X},S}}\ar[d]_{\simeq} & \mathbf{G}(S)\ar[d]^{\simeq}\\
\mathsf{Z}_{G}(S')\ar[r]^{\nu_{\mathbf{X},S'}} & \mathbf{G}(S')
}
\]
where the vertical maps are induced by $\Int(g)$ for any $g\in\mathsf{G}$
with $\Int(g)(S)=S'$. Set $\mathsf{W}_{G}=\underleftarrow{\lim}\mathsf{W}_{G}(S)$.
The type of $\mathbf{X}(G)$ is the $\mathsf{W}_{G}$-equivariant
morphism\nomenclature[nu_X,S]{$\nu _{\mathbf{X},S}$}{Morphism $\mathsf{Z}_G (S) \rightarrow \mathbf{G}(S)$ defined page \nomrefpage}\nomenclature[nu_X]{$\nu _{\mathbf{X}}$}{Type morphism of an $\mathbf{F}(G)$-building $\mathbf{X}(G)$ defined page \nomrefpage}
\[
\nu_{\mathbf{X}}=\underleftarrow{\lim}\nu_{\mathbf{X},S}:\underleftarrow{\lim}\mathsf{Z}_{G}(S)\rightarrow\underleftarrow{\lim}\mathbf{G}(S)
\]
which is obtained from these diagrams by taking the limits over all
$S\in\mathbf{S}(G)$. We say that $\mathbf{X}(G)$ is discrete when
the image of $\nu_{\mathbf{X}}$ is a discrete subgroup of the real
vector space $\underleftarrow{\lim}\mathbf{G}(S)$. Equivalently:
$\mathbf{X}(G)$ is discrete when the image of $\nu_{\mathbf{X},S}$
is a discrete subgroup of $\mathbf{G}(S)$ for one or every $S\in\mathbf{S}(G)$.

\subsection{~}

Affine $\mathbf{F}(G)$-buildings are affine $\mathbf{F}(G)$-spaces
satisfying a long list of axioms, which shall be gradually introduced
below. The following definition picks up an (hopefully minimal) subset
of these axioms, from which all others will be derived in due time,
along with various properties.
\begin{defn}
An affine $\mathbf{F}(G)$-building is an affine $\mathbf{F}(G)$-space
which satisfies the axioms $L(s)$, $R(s)$, $R(i)$, $\mathcal{C}^{\circ}$,
$NE$, $UN$, $CO$ and $UG$ listed below. 
\end{defn}

\subsection{Example}

\label{Example:F(G)Itself}The Tits building $\mathbf{F}(G)$ itself,
equipped with its left action of $\mathsf{G}$, the addition map of
section~\ref{sub:AdditionMap} and the apartment map of section~\ref{sub:apartments}
is a discrete affine $\mathbf{F}(G)$-space with trivial type $\nu_{\mathbf{F}}=0$.
We will see that it satisfies all of the required axioms, thus $\mathbf{F}(G)=(\mathbf{F}(G),+,\mathbf{F}(-))$
is an affine $\mathbf{F}(G)$-building.

\subsection{Many apartments\label{sub:Many-apartments}}

An affine $\mathbf{F}(G)$-building $\mathbf{X}(G)$ satisfies
\begin{lyxlist}{XXXX}
\item [{$L(s)$}] For every $x\in\mathbf{X}(G)$ and $\mathcal{F}\in\mathbf{F}(G)$,
\[
\mathbf{S}(x,\mathcal{F})=\left\{ S\in\mathbf{S}(G):x\in\mathbf{X}(S)\mbox{ and }\mathcal{F}\in\mathbf{F}(S)\right\} 
\]
is not empty.
\item [{$R(s)$}] For every $x,y\in\mathbf{X}(G)$, 
\[
\mathbf{S}(x,y)=\left\{ S\in\mathbf{S}(G):x,y\in\mathbf{X}(S)\right\} 
\]
is not empty. 
\item [{$T(s)$}] For every $x,y\in\mathbf{X}(G)$, 
\[
\mathbf{F}(x,y)=\left\{ \mathcal{F}\in\mathbf{F}(G):y=x+\mathcal{F}\right\} 
\]
is not empty.
\end{lyxlist}
Note that $R(s)$ implies $T(s)$ while $L(s)$ is equivalent to 
\begin{lyxlist}{XXXX}
\item [{$L'(s)$}] For every $x\in\mathbf{X}(G)$ and every closed chamber
$C$ of $\mathbf{F}(G)$, 
\[
\mathbf{S}(x,C)=\left\{ S\in\mathbf{S}(G):x\in\mathbf{X}(S)\mbox{ and }C\subset\mathbf{F}(S)\right\} 
\]
is not empty.
\end{lyxlist}
This in turn implies that the pull map is well-behaved:
\begin{lyxlist}{XXXX}
\item [{$AC$}] (\emph{Action}) For every closed chamber $C$ of $\mathbf{F}(G)$,
the map 
\[
\mathbf{X}(G)\times C\hookrightarrow\mathbf{X}(G)\times\mathbf{F}(G)\stackrel{+}{\longrightarrow}\mathbf{X}(G)
\]
defines an action of the commutative monoid $(C,+)$ on $\mathbf{X}(G)$.
\end{lyxlist}
Thus for any $x\in\mathbf{X}(G)$ and $\mathcal{F},\mathcal{G}\in\mathbf{F}(G)$,
$x+0=x$ and 
\[
(x+\mathcal{F})+\mathcal{G}=x+(\mathcal{F}+\mathcal{G})=(x+\mathcal{G})+\mathcal{F}
\]
if $P_{\mathcal{F}}$ and $P_{\mathcal{G}}$ are in osculatory position.
In particular, 
\[
(x+\lambda\mathcal{F})+\mu\mathcal{F}=x+(\lambda+\mu)\mathcal{F}
\]
for every $\lambda,\mu\geq0$ and $\mathcal{F}\in\mathbf{F}(G)$.

\subsection{Strong transitivity\label{sub:Strong-transitivity}}

An affine $\mathbf{F}(G)$-building $\mathbf{X}(G)$ satisfies
\begin{lyxlist}{XXXX}
\item [{$L(i)$}] For every $x\in\mathbf{X}(G)$ and $\mathcal{F}\in\mathbf{F}(G)$,
\[
\mathsf{G}_{x,\mathcal{F}}=\left\{ g\in\mathsf{G}:gx=x\quad\mbox{and}\quad g\mathcal{F}=\mathcal{F}\right\} 
\]
acts transitively on $\mathbf{S}(x,\mathcal{F})$. 
\item [{$R(i)$}] For every $x,y\in\mathbf{X}(G)$, 
\[
\mathsf{G}_{x,y}=\left\{ g\in\mathsf{G}:gx=x\quad\mbox{and}\quad gy=y\right\} 
\]
acts transitively on $\mathbf{S}(x,y)$. 
\item [{$T(i)$}] For every $x,y\in\mathbf{X}(G)$, $\mathsf{G}_{x,y}$
acts transitively on $\mathbf{F}(x,y)$.
\end{lyxlist}

\subsection{~}

The labels of the $L$, $R$, or $T$-axioms reflect their equivalence
with the surjectivity or injectivity of the relevant maps in the commutative
diagram 
\[
\xymatrix{\mathsf{G}\backslash\left(\mathbf{X}(G)\times\mathbf{F}(G)\right)\ar[r]^{T} & \mathsf{G}\backslash\left(\mathbf{X}(G)\times\mathbf{X}(G)\right)\\
\mathsf{\mathsf{N}}_{G}(S)\backslash\left(\mathbf{X}(S)\times\mathbf{F}(S)\right)\ar[u]^{L}\ar[r] & \mathsf{\mathsf{N}}_{G}(S)\backslash\left(\mathbf{X}(S)\times\mathbf{X}(S)\right)\ar[u]_{R}
}
\]
which is induced by the equivariant commutative diagram 
\[
\xyC{5pc}\xymatrix{\mathbf{X}(G)\times\mathbf{F}(G)\ar[r]^{(x,\mathcal{F})\mapsto(x,x+\mathcal{F})} & \mathbf{X}(G)\times\mathbf{X}(G)\\
\mathbf{X}(S)\times\mathbf{F}(S)\ar@{^{(}->}[u]\ar[r] & \mathbf{X}(S)\times\mathbf{X}(S)\ar@{^{(}->}[u]
}
\]
The bottom map in each diagram is always bijective since $\mathbf{X}(S)$
is an affine $\mathbf{F}(S)$-space. Thus $R(i)\Rightarrow T(i)$,
and $R(i)+L(s)+T(s)$ imply all of the above axioms.

\subsection{The vectorial distance\label{sub:Vectorial-Distance}}

It follows from the axioms already introduced that for an affine $\mathbf{F}(G)$-building
$\mathbf{X}(G)$, there is a unique $\mathsf{G}$-invariant map\nomenclature[d]{$\mathbf{d}$}{Vectorial distance on an affine $\mathbf{F}(G)$-building, defined page \nomrefpage}
\[
\mathbf{d}:\mathbf{X}(G)\times\mathbf{X}(G)\rightarrow\mathbf{C}(G)
\]
such that for every $x\in\mathbf{X}(G)$ and $\mathcal{F}\in\mathbf{F}(G)$,
\[
\mathbf{d}(x,x+\mathcal{F})=t(\mathcal{F})\quad\mbox{in}\quad\mathbf{C}(G).
\]
The following properties are easily established: for $x,y\in\mathbf{X}(G)$,
\[
\mathbf{d}(y,x)=\mathbf{d}(x,y)^{\iota}\quad\mbox{and}\quad\mathbf{d}(x,y)=0\iff x=y.
\]
Moreover for $x\in\mathbf{X}(G)$, $\mathcal{F}\in\mathbf{F}(G)$
and $0\leq\lambda\leq\lambda'$, 
\[
\mathbf{d}(x+\lambda\mathcal{F},x+\lambda'\mathcal{F})=(\lambda'-\lambda)\cdot t(\mathcal{F}).
\]
This \emph{vectorial distance} $\mathbf{d}$ may also satisfy the
following properties -- and it does for affine $\mathbf{F}(G)$-buildings,
by lemma~\ref{lem:EquivAxiomsTR} and proposition~\ref{prop:distretrac}
below:
\begin{lyxlist}{XXX}
\item [{$\star$}] For every $x,y\in\mathbf{X}(G)$ and $\mathcal{F},\mathcal{G}\in\mathbf{F}(G)$,
\[
\mathbf{d}(x+\mathcal{F},y+\mathcal{G})\leq\mathbf{d}(x,y)+\mathbf{d}(\mathcal{F},\mathcal{G})\quad\mbox{in}\quad\mathbf{C}(G).
\]

\item [{$TR$}] (\emph{Triangle inequality}) For every $x,y,z\in\mathbf{X}(G)$,
\[
\mathbf{d}(x,z)\leq\mathbf{d}(x,y)+\mathbf{d}(y,z)\quad\mbox{in}\quad\mathbf{C}(G).
\]

\item [{$TR'$}] For every $y\in\mathbf{X}(G)$ and $\mathcal{F},\mathcal{G}\in\mathbf{F}(G)$,
\[
\mathbf{d}(y+\mathcal{F},y+\mathcal{G})\leq\mathbf{d}(\mathcal{F},\mathcal{G})\quad\mbox{in}\quad\mathbf{C}(G).
\]

\item [{$NE$}] (\emph{Non expanding}) For every $x,y\in\mathbf{X}(G)$
and $\mathcal{F}\in\mathbf{F}(G)$, 
\[
\mathbf{d}(x+\mathcal{F},y+\mathcal{F})\leq\mathbf{d}(x,y)\quad\mbox{in}\quad\mathbf{C}(G).
\]

\item [{$\mathcal{C}^{\circ}$}] (\emph{Continuity}) For every sequences
$(x_{n})$, $(y_{n})$ and points $x$, $y$ in $\mathbf{X}(G)$,
\[
\left(\begin{array}{c}
\mathbf{d}(x_{n},x)\rightarrow0\\
\mathbf{d}(y_{n},y)\rightarrow0
\end{array}\mbox{ in }\mathbf{C}(G)\right)\Longrightarrow\left(\mathbf{d}(x_{n},y_{n})\rightarrow\mathbf{d}(x,y)\mbox{ in }\mathbf{C}(G)\right).
\]

\end{lyxlist}
Note that $\star$ and $TR'$ also involve the vectorial distance
$\mathbf{d}$ for $\mathbf{F}(G)$ -- it follows from~\ref{sub:FiloverFieldsisTitsBuild}
that the affine $\mathbf{F}(G)$-space $\mathbf{F}(G)$ indeed satisfies
the required axioms for the existence of $\mathbf{d}$: $L(s)=R(s)$
is theorem~\ref{thm:Std=00003DAlloverField} and $L(i)=R(i)$ is
its corollary~\ref{cor:AxiomeR(i)4F(G)}. 
\begin{lem}
\label{lem:EquivAxiomsTR}The above properties of $\mathbf{d}$ are
related as follows:
\[
\star\iff TR+TR'+NE\quad\mbox{and}\quad TR\iff TR'\Longrightarrow\mathcal{C}^{\circ}.
\]
\end{lem}
\begin{proof}
($TR+TR'+NE\Rightarrow\star$). For $x,y\in\mathbf{X}(G)$ and $\mathcal{F},\mathcal{G}\in\mathbf{F}(G)$,
we find
\[
\mathbf{d}(x+\mathcal{F},y+\mathcal{G})\leq\mathbf{d}(x+\mathcal{F},y+\mathcal{F})+\mathbf{d}(y+\mathcal{F},y+\mathcal{G})\leq\mathbf{d}(x,y)+\mathbf{d}(\mathcal{F},\mathcal{G})
\]
using $TR$ for the first inequality, $NE$ and $TR'$ for the second.

($\star\Rightarrow TR+TR'+NE$). Taking $x=y$ (resp.~$\mathcal{F}=\mathcal{G}$)
in $\star$ yields $TR'$ (resp. $NE$). Taking $\mathcal{F}=0$ and
$\mathcal{G}\in\mathbf{F}(y,z)$ (using $T(s)$) yields $TR$.

($TR\Rightarrow TR'$). For $x\in\mathbf{X}(G)$, $\mathcal{F},\mathcal{G}\in\mathbf{F}(G)$
and $\lambda\in[0,1]$, set 
\[
\mathcal{F}(\lambda)=(1-\lambda)\mathcal{F}+\lambda\mathcal{G}\mbox{ in }\mathbf{F}(G)\quad\mbox{and}\quad x(\lambda)=x+\mathcal{F}(\lambda)\mbox{ in }\mathbf{X}(G).
\]
Pick $S\in\mathbf{S}(G)$ with $\mathcal{F},\mathcal{G}\in\mathbf{F}(S)$.
There is a subdivision $0=\lambda_{0}<\cdots<\lambda_{n}=1$ of $[0,1]$
and for each $i\in\{1,\cdots,n\}$, a closed chamber $C_{i}$ of $\mathbf{F}(S)$
such that $\mathcal{F}(\lambda)\in C_{i}$ for all $\lambda\in[\lambda_{i-1},\lambda_{i}]$.
By $L'(s)$, there is an $S_{i}\in\mathbf{S}(G)$ such that $x\in\mathbf{X}(S_{i})$
and $C_{i}\subset\mathbf{F}(S_{i})$, thus also $x(\lambda)\in\mathbf{X}(S_{i})$
for every $\lambda\in[\lambda_{i-1},\lambda_{i}]$. Then 
\[
\mathbf{d}(x+\mathcal{F},x+\mathcal{G})\leq{\textstyle \sum}_{i=1}^{n}\mathbf{d}(x(\lambda_{i-1}),x(\lambda_{i}))={\textstyle \sum}_{i=1}^{n}\mathbf{d}(\mathcal{F}(\lambda_{i-1}),\mathcal{F}(\lambda_{i}))=\mathbf{d}(\mathcal{F},\mathcal{G})
\]
using respectively $TR$ in $\mathbf{X}(G)$ and trivial computations
in $\mathbf{X}(S_{i})$ and $\mathbf{F}(G)$. 

($TR'\Rightarrow TR$). For $x,y,z\in\mathbf{X}(G)$, pick $\mathcal{F},\mathcal{G}\in\mathbf{F}(G)$
with 
\[
x=y+\mathcal{F}\quad\mbox{and}\quad z=y+\mathcal{G}
\]
using $T(s)$. Choose $S\in\mathbf{S}(G)$ such that $\mathcal{F},\mathcal{G}\in\mathbf{F}(S)$,
and set $\mathcal{F}'=\iota_{S}\mathcal{F}$. Then 
\[
\mathbf{d}(x,z)\leq\mathbf{d}(\mathcal{F},\mathcal{G})=t(\mathcal{F}'+\mathcal{G})\leq t(\mathcal{F}')+t(\mathcal{G})=\mathbf{d}(x,y)+\mathbf{d}(y,z)
\]
using respectively $TR'$ in $\mathbf{X}(G)$, proposition~\ref{prop:TRinF(G)}
and 
\[
t(\mathcal{F}')=t(\mathcal{F})^{\iota}=\mathbf{d}(y,x)^{\iota}=\mathbf{d}(x,y),\quad t(\mathcal{G})=\mathbf{d}(y,z).
\]

($TR\Rightarrow\mathcal{C}^{\circ}$). Suppose that $\mathbf{d}(x_{n},x)\rightarrow0$
and $\mathbf{d}(y_{n},y)\rightarrow0$ for sequences $(x_{n})$, $(y_{n})$
and points $x$, $y$ in $\mathbf{X}(G)$. Then also $\mathbf{d}(x,x_{n})\rightarrow0$
and $\mathbf{d}(y,y_{n})\rightarrow0$ in $\mathbf{C}(G)$. Let $c$
be a limit point of $\mathbf{d}(x_{n},y_{n})$ in the Alexandrov compactification
$\mathbf{C}(G)\cup\{\infty\}$ of the locally compact space $\mathbf{C}(G)$.
We have to show that $c=\mathbf{d}(x,y)$, for then $\mathbf{d}(x_{n},y_{n})\rightarrow\mathbf{d}(x,y)$
in $\mathbf{C}(G)$. By the triangle inequality $TR$, 
\begin{eqnarray*}
\mathbf{d}(x,y) & \leq & \mathbf{d}(x,x_{n})+\mathbf{d}(x_{n},y_{n})+\mathbf{d}(y_{n},y)\\
\mbox{and}\quad\mathbf{d}(x_{n},y_{n}) & \leq & \mathbf{d}(x_{n},x)+\mathbf{d}(x,y)+\mathbf{d}(y,y_{n})
\end{eqnarray*}
for every $n\ge0$, thus $c\in\mathbf{C}(G)$ and $\mathbf{d}(x,y)\leq c\leq\mathbf{d}(x,y)$,
i.e.~$c=\mathbf{d}(x,y)$. 
\end{proof}

\subsection{The classical distance\label{sub:The-classical-distance}}

These axioms imply that the composition
\[
d:\mathbf{X}(G)\times\mathbf{X}(G)\rightarrow\mathbb{R}_{+},\qquad x\mapsto\left\Vert \mathbf{d}(x)\right\Vert 
\]
of the vectorial distance $\mathbf{d}$ with the norm $\left\Vert -\right\Vert :\mathbf{C}(G)\rightarrow\mathbb{R}_{+}$
attached to our chosen $\tau$ is a genuine $\mathsf{G}$-invariant
distance on $\mathbf{X}(G)$. Its restriction to any apartment is
Euclidean and $(\mathbf{X}(G),d)$ is a geodesic space: for $x,y\in\mathbf{X}(G)$
and any apartment $\mathbf{X}(S)$ containing $x$ and $y$, the unique
geodesic from $x$ to $y$ in $\mathbf{X}(S)$ is a geodesic from
$x$ to $y$ in $\mathbf{X}(G)$. For any sequence $(x_{n})$ in $\mathbf{X}(G)$
and $x$ in $\mathbf{X}(G)$, we have 
\[
x_{n}\rightarrow x\mbox{ in }(\mathbf{X}(G),d)\iff d(x_{n},x)\rightarrow0\mbox{ in }\mathbb{R}_{+}\iff\mathbf{d}(x_{n},x)\rightarrow0\mbox{ in }\mathbf{C}(G).
\]
The induced metrizable topology on $\mathbf{X}(G)$ thus does not
depend upon $\tau$ (see also \ref{sub:EquivDistOnF}). We call it
the canonical topology of $\mathbf{X}(G)$. The apartments are closed,
being complete for the induced metric. The vectorial distance and
pull map 
\[
\mathbf{d}:\mathbf{X}(G)\times\mathbf{X}(G)\rightarrow\mathbf{C}(G)\qquad\mbox{and}\qquad+:\mathbf{X}(G)\times\mathbf{F}(G)\rightarrow\mathbf{X}(G)
\]
are continuous for the canonical topologies on $\mathbf{X}(G)$, $\mathbf{C}(G)$
and $\mathbf{F}(G)$ by $\mathcal{C}^{\circ}$ and $\star$.

\subsection{The retractions\label{sub:The-retractions}}

For an affine $\mathbf{F}(G)$-building, we also require:
\begin{lyxlist}{XXX}
\item [{$UN$}] (\emph{Unipotent}) For every $x\in\mathbf{X}(G)$, $\mathcal{F}\in\mathbf{F}(G)$
and $u\in\mathsf{U}_{\mathcal{F}}$, 
\[
\lim_{s\rightarrow+\infty}\mathbf{d}(x+s\mathcal{F},ux+s\mathcal{F})=0.
\]

\end{lyxlist}
For $\mathcal{F}\in\mathbf{F}(G)$ and any Levi subgroup $L$ of $P_{\mathcal{F}}$,
we denote by $\mathcal{F}_{L}^{\iota}$ the unique filtration opposed
to $\mathcal{F}$ with $P_{\mathcal{F}}\cap P_{\mathcal{F}_{L}^{\iota}}=L$.
Thus $\mathcal{F}_{L}^{\iota}=\Fi(\iota\mathcal{G})$ where $\mathcal{G}\in\mathbf{G}(G)$
is the unique splitting of $\mathcal{F}$ with $L_{\mathcal{G}}=L$.
We have $\mathcal{F}_{L}^{\iota}=\iota_{S}\mathcal{F}$ for any $S\in\mathbf{S}(L)$.\nomenclature[F^iota_L]{$\mathcal{F}^\iota _L$}{Filtration opposed to $\mathcal{F}$ with respect to a Levi  $L$ of $P_\mathcal{F}$, page \nomrefpage}
\begin{prop}
\label{prop:distretrac}Let $\mathbf{X}(G)$ be an affine $\mathbf{F}(G)$-space
satisfying the axioms of sections~\ref{sub:Many-apartments} and~\ref{sub:Strong-transitivity},
together with $\mathcal{C}^{\circ}$, $NE$ and $UN$. Then it also
satisfy $TR$. Moreover, for every parabolic subgroup $P$ of $G$
with Levi decomposition $P=U\rtimes L$, 
\[
\mathbf{X}(L)=\cup_{S\in\mathbf{S}(L)}\mathbf{X}(S)
\]
is a fundamental domain for the action of $\mathsf{U}$ on $\mathbf{X}(G)$
and the induced retraction\nomenclature[r_P,L]{$r_{P,L}$}{Retraction $r_{P,L}:\mathbf{X}(G) \twoheadrightarrow \mathbf{X}(L)$, defined page \nomrefpage}
\[
r_{P,L}:\mathbf{X}(G)\twoheadrightarrow\mathbf{X}(L)
\]
is non-expanding for $\mathbf{d}$: for every $x,y\in\mathbf{X}(G)$,
\[
\mathbf{d}(r_{P,L}(x),r_{P,L}(y))\leq\mathbf{d}(x,y)\quad\mbox{in}\quad\mathbf{C}(G).
\]
Finally, for any $\mathcal{F}\in F^{-1}(P)$, if $\mathcal{F}'=\mathcal{F}_{L}^{\iota}$,
then for all $x\in\mathbf{X}(G)$, 
\[
r_{P,L}(x)=\lim_{s\rightarrow\infty}\left(x+s\mathcal{F}\right)+s\mathcal{F}'\quad\mbox{in}\quad\left(\mathbf{X}(G),d\right).
\]
\end{prop}
\begin{proof}
Fix $P$, $L$, $\mathcal{F}$ and $\mathcal{F}'=\mathcal{F}_{L}^{\iota}$
as above. For any $x\in\mathbf{X}(G)$, there is by $L(s)$ an $S'\in\mathbf{S}(G)$
such that $x\in\mathbf{X}(S')$ and $\mathcal{F}\in\mathbf{F}(S')$,
i.e.~$Z_{G}(S')\subset P_{\mathcal{F}}$. Let $L'$ be the unique
levi subgroup of $P_{\mathcal{F}}$ containing $Z_{G}(S')$ and let
$u$ be the unique element of $\mathsf{U}$ such that $\Int(u)(L')=L$.
Then $S=\Int(u)(S')$ belongs to $\mathbf{S}(L)$ and $ux$ belongs
to $\mathbf{X}(S)\subset\mathbf{X}(L)$, thus $\mathsf{U}\cdot\mathbf{X}(L)=\mathbf{X}(G)$.
For $s\geq0$ and $x\in\mathbf{X}(G)$, set 
\[
r_{s}(x)=\left(x+s\mathcal{F}\right)+s\mathcal{F}'\quad\mbox{in}\quad\mathbf{X}(G).
\]
Then $x\mapsto r_{s}(x)$ is non-expanding for $\mathbf{d}$ by $NE$
and for any $u\in\mathsf{U}$, 
\[
\lim_{s\rightarrow\infty}\mathbf{d}(r_{s}(x),r_{s}(u\cdot x))=0\quad\mbox{in}\quad\mathbf{C}(G)
\]
by $UN$ and $NE$. If $x$ belongs to $\mathbf{X}(L)$, say $x\in\mathbf{X}(S)$
for some $S\in\mathbf{S}(L)$, then $\mathcal{F},\mathcal{F}'\in\mathbf{F}(S)$
with $\mathcal{F}+\mathcal{F}'=0$ in $\mathbf{F}(S)$, thus $r_{s}(x)=x$
for all $s\geq0$ since $\mathbf{X}(S)$ is an affine $\mathbf{F}(S)$-space.
If $x$ and $u\cdot x$ belong to $\mathbf{X}(L)$, $\mathbf{d}(x,u\cdot x)=\mathbf{d}(r_{s}(x),r_{s}(u\cdot x))$
for all $s\geq0$, thus $\mathbf{d}(x,u\cdot x)=0$ and $x=u\cdot x$.
In particular, $\mathbf{X}(L)$ is indeed a fundamental domain for
the action of $\mathsf{U}$ on $\mathbf{X}(G)$. Let $r:\mathbf{X}(G)\twoheadrightarrow\mathbf{X}(L)$
be the corresponding retraction. For $x\in\mathbf{X}(G)$, pick $u\in\mathsf{U}$
such that $r(x)=u\cdot x$. Then 
\[
\mathbf{d}(r_{s}(x),r(x))=\mathbf{d}(r_{s}(x),u\cdot x)=\mathbf{d}(r_{s}(x),r_{s}(u\cdot x))\rightarrow0.
\]
Applying this to $x,y\in\mathbf{X}(G)$ and using $\mathcal{C}^{\circ}$,
we find that 
\[
\lim_{s\rightarrow\infty}\mathbf{d}(r_{s}(x),r_{s}(y))=\mathbf{d}(r(x),r(y)),
\]
thus $\mathbf{d}(r(x),r(y))\leq\mathbf{d}(x,y)$ since $\mathbf{d}(r_{s}(x),r_{s}(y))\leq\mathbf{d}(x,y)$
for all $s\geq0$.

Turning now to the proof of $TR$, first note that by proposition~\ref{prop:TRinF(G)},
the triangle inequality holds whenever $x,y,z$ belong to $\mathbf{X}(S)$
for some $S\in\mathbf{S}(G)$. For a general triple $x,y,z$ in $\mathbf{X}(G)$,
choose $S\in\mathbf{S}(G)$ with $x,z\in\mathbf{X}(S)$ using $R(s)$,
pick a minimal parabolic subgroup $B$ of $G$ with Levi subgroup
$L=Z_{G}(S)$ and let $r:\mathbf{X}(G)\twoheadrightarrow\mathbf{X}(S)$
be the corresponding retraction. Then $r(x)=x$ and $r(z)=z$, thus
indeed
\[
\mathbf{d}(x,z)=\mathbf{d}(r(x),r(z))\leq\mathbf{d}(r(x),r(y))+\mathbf{d}(r(y),r(z))\leq\mathbf{d}(x,y)+\mathbf{d}(y,z)
\]
since the triangle inequality holds on $\mathbf{X}(S)$ and $r$ is
non-expanding for $\mathbf{d}$. \end{proof}
\begin{cor}
The apartment map $S\mapsto\mathbf{X}(S)$ is then uniquely determined
by the pull map $+:\mathbf{X}(G)\times\mathbf{F}(G)\rightarrow\mathbf{X}(G)$:
for every $S\in\mathbf{S}(G)$, 
\[
\mathbf{X}(S)=\left\{ x\in\mathbf{X}(G):\forall\mathcal{F},\mathcal{F}'\in\mathbf{F}(S),\,(x+\mathcal{F})+\mathcal{F}'=x+(\mathcal{F}+\mathcal{F}')\right\} .
\]
\end{cor}
\begin{proof}
Let $\mathbf{X}'(S)$ be the right hand side. Plainly, $\mathbf{X}(S)\subset\mathbf{X}'(S)$.
Conversely, pick a minimal parabolic subgroup $B$ of $G$ with Levi
$Z_{G}(S)$, let $r:\mathbf{X}(G)\twoheadrightarrow\mathbf{X}(S)$
be the corresponding retraction, choose $\mathcal{F}\in F^{-1}(B)$
and set $\mathcal{F}'=\iota_{S}\mathcal{F}$. For $x$ in $\mathbf{X}'(S)$,
$(x+\lambda\mathcal{F})+\lambda\mathcal{F}'=x$ for all $\lambda\geq0$,
thus $r(x)=x$ belongs to $\mathbf{X}(S)$.
\end{proof}

\subsection{Standard geodesics}

For $x\in\mathbf{X}(G)$ and $\mathcal{F}\in\mathbf{F}(G)$, the function
\[
[0,1]\rightarrow\mathbf{X}(G)\quad\left(\mbox{resp. }\mathbb{R}_{+}\rightarrow\mathbf{X}(G)\right)\qquad t\mapsto x+t\mathcal{F}
\]
is a geodesic segment (resp.~geodesic ray) in $(\mathbf{X}(G),d)$.
We refer to these geodesics as the \emph{standard }ones. Thus a geodesic
(segment or ray) is standard precisely when it is contained in some
apartment, and the set of all standard geodesics does not depend upon
the choice of $\tau$. If $(\mathbf{X}(G),d)$ is uniquely geodesic,
then every geodesic segment is standard, but there might still be
some non-standard geodesic rays.

\subsection{Convexity}

An affine $\mathbf{F}(G)$-building satisfies all of the above axioms,
together with the following convexity axiom:
\begin{lyxlist}{XXX}
\item [{$CO^{+}$}] For every pair of geodesics $x,y:[0,1]\rightarrow\mathbf{X}(G)$
in $(\mathbf{X}(G),d)$, the function 
\[
f:[0,1]\rightarrow\mathbf{C}(G),\qquad f(t)=\mathbf{d}(x(t),y(t))
\]
is convex, i.e.~for every $\lambda$ and $t_{1}\leq t_{2}$ in $[0,1]$,
\[
f\left((1-\lambda)t_{1}+\lambda t_{2}\right)\leq(1-\lambda)f(t_{1})+\lambda f(t_{2})\quad\mbox{in}\quad\mathbf{C}(G).
\]

\end{lyxlist}
This implies that the metric space $(\mathbf{X}(G),d)$ itself is
convex in the sense of \cite[II.1.3]{BrHa99}. In particular, it is
uniquely geodesic and for every $x\in\mathbf{X}(G)$ and $\mathcal{F},\mathcal{G}\in\mathbf{F}(G)$,
\[
x+\mathcal{F}=x+\mathcal{G}\Longrightarrow\forall t\in[0,1]:\quad x+t\mathcal{F}=x+t\mathcal{G}.
\]

\begin{prop}
\label{prop:r=00003Dr'}Let $\mathbf{X}(G)$ be an affine $\mathbf{F}(G)$-building.
Let $(P,P')$ be a pair of opposed parabolic subgroups of $G$ with
common Levi subgroup $L=P\cap P'$. Let 
\[
r,r':\mathbf{X}(G)\twoheadrightarrow\mathbf{X}(L)
\]
be the corresponding retractions, as in proposition~\ref{prop:distretrac}.
Then 
\[
\mathbf{X}(L)=\left\{ x\in\mathbf{X}(G):r(x)=r'(x)\right\} .
\]
\end{prop}
\begin{proof}
For $x\in\mathbf{X}(L)$, $r(x)=x=r'(x)$, thus $x$ belongs to 
\[
\mathbf{X}'(L)=\left\{ x\in\mathbf{X}(G):r(x)=r'(x)\right\} .
\]
Suppose conversely that $x\in\mathbf{X}'(L)$ and set $y=r(x)=r'(x)$.
Pick a pair of opposed filtrations $(\mathcal{F},\mathcal{F}')$ with
$P_{\mathcal{F}}=P$, $P_{\mathcal{F}'}=P'$ and $\left\Vert \mathcal{F}\right\Vert =\left\Vert \mathcal{F}'\right\Vert =1$.
For $t\in\mathbb{R}$, set 
\[
X(t)=\begin{cases}
x+\left|t\right|\mathcal{F} & \mbox{if }t\geq0,\\
x+\left|t\right|\mathcal{F}' & \mbox{if }t\leq0,
\end{cases}\quad\mbox{and}\quad Y(t)=\begin{cases}
y+\left|t\right|\mathcal{F} & \mbox{if }t\geq0,\\
y+\left|t\right|\mathcal{F}' & \mbox{if }t\leq0.
\end{cases}
\]
Plainly, $Y:\mathbb{R}\rightarrow\mathbf{X}(G)$ is a geodesic line
and $d(X(t),Y(t))\rightarrow0$ when $\left|t\right|\rightarrow\infty$.
Note that for any $0\leq t_{1},t_{2}\leq t$, $d\left(Y(-t),Y(t)\right)=2t$
is not greater than {\footnotesize{
\[
d\left(Y(-t),X(-t)\right)+d\left(X(-t),X(-t_{1})\right)+d\left(X(-t_{1}),X(t_{2})\right)+d\left(X(t_{2}),X(t)\right)+d\left(X(t),Y(t)\right).
\]
}}The second and fourth term sum to $2t-(t_{1}+t_{2})$, thus $t_{1}+t_{2}\leq d\left(X(-t_{1}),X(t_{2})\right)$.
Since also $d\left(X(-t_{1}),X(t_{2})\right)\leq t_{1}+t_{2}$, it
follows that $X:\mathbb{R}\rightarrow\mathbf{X}(G)$ is a geodesic
line as well. Since the metric $d$ is convex, the function $t\mapsto d(X(t),Y(t))$
is convex. Since it is also bounded, it must be constant, thus actually
trivial. In particular, $d(x,y)=d(X(0),Y(0))=0$, therefore $x=y$
belongs to $\mathbf{X}(L)$. \end{proof}
\begin{defn}
The enclosure of $x,z\in\mathbf{X}(G)$ is defined by 
\begin{eqnarray*}
\lozenge(x,z) & = & \left\{ y\in\mathbf{X}(G):\mathbf{d}(x,z)=\mathbf{d}(x,y)+\mathbf{d}(y,z)\right\} \\
 & = & \left\{ y\in\mathbf{X}(G):\mathbf{d}(x,z)\geq\mathbf{d}(x,y)+\mathbf{d}(y,z)\right\} 
\end{eqnarray*}
\end{defn}
\begin{cor}
\label{cor:eqinTR}For any $S\in\mathbf{S}(G)$ and $x,x'\in\mathbf{X}(S)$,
let $F$ and $F'$ be the pair of opposed facets in $\mathbf{F}(S)$
such that $x'\in x+F$ and $x\in x'+F'$. Then
\[
\lozenge(x,x')=(x+\overline{F})\cap(x'+\overline{F}').
\]
In particular for any $x,z\in\mathbf{X}(G)$, the enclosure $\lozenge(x,z)$
is a closed and convex subset of $\mathbf{X}(G)$ which is contained
in any apartment containing $x$ and $z$.\end{cor}
\begin{proof}
For $y\in\mathbf{X}(S)$, write $y=x+a$ and $x'=y+b$ with $a,b\in\mathbf{F}(S)$,
so that
\[
\mathbf{d}(x,x')=t(a+b),\quad\mathbf{d}(x,y)=t(a)\quad\mbox{and}\quad\mathbf{d}(y,x')=t(b).
\]
Thus $y$ belongs to $\lozenge(x,x')$ if and only if there exists
a closed chamber $C$ in $\mathbf{F}(S)$ containing $a$ and $b$
by~proposition~\ref{prop:TRinF(G)}, which occurs precisely when
$a$ and $b$ both belong to the closure $\overline{F}$ of the facet
$F$ of $\mathbf{F}(S)$ which contains $c=a+b$. Hence 
\[
\lozenge(x,x')\cap\mathbf{X}(S)=(x+\overline{F})\cap(x'+\overline{F}').
\]
In particular, the function $y\mapsto\mathbf{d}(x,y)$ is injective
on $\lozenge(x,x')\cap\mathbf{X}(S)$. Now pick a pair of opposed
minimal parabolic subgroups $(B,B')$ of $G$ with $B\cap B'=Z_{G}(S)$.
Let $r,r':\mathbf{X}(G)\twoheadrightarrow\mathbf{X}(S)$ be the corresponding
retractions. For any $y\in\lozenge(x,x')$, 
\[
\mathbf{d}(x,r(y))=\mathbf{d}(r(x),r(y))\leq\mathbf{d}(x,y)\quad\mbox{and}\quad\mathbf{d}(r(y),x')=\mathbf{d}(r(y),r(x'))\leq\mathbf{d}(y,x')
\]
since $r$ is non-expanding for $\mathbf{d}$, therefore 
\[
\mathbf{d}(x,x')\leq\mathbf{d}(x,r(y))+\mathbf{d}(r(y),x')\leq\mathbf{d}(x,y)+\mathbf{d}(y,x')=\mathbf{d}(x,x').
\]
Thus $r(y)$ belongs to $\lozenge(x,x')\cap\mathbf{X}(S)$ and $\mathbf{d}(x,r(y))=\mathbf{d}(x,y)$.
Since the same conclusion holds for $r'(y)$, we obtain $r(y)=r'(y)$.
Hence $y$ belongs to $\mathbf{X}(S)$ and indeed $\lozenge(x,x')=(x+\overline{F})\cap(x'+\overline{F}')$.
The remaining assertions easily follow.
\end{proof}

\subsection{Unique Geodesics}

It may seem that the validity of the axiom $CO^{+}$ for a given affine
$\mathbf{F}(G)$-space $\mathbf{X}(G)$ depends upon the chosen $\tau$,
but it does not. In fact, $CO^{+}$ is plainly equivalent to the conjonction
of the following two axioms:
\begin{lyxlist}{XXX}
\item [{$CO$}] For any pair of standard geodesics $x,y:[0,1]\rightarrow\mathbf{X}(G)$
in $\mathbf{X}(G)$, the function 
\[
f:[0,1]\rightarrow\mathbf{C}(G),\qquad f(t)=\mathbf{d}(x(t),y(t))
\]
is convex, i.e.~for every $\lambda$ and $t_{1}\leq t_{2}$ in $[0,1]$,
\[
f\left((1-\lambda)t_{1}+\lambda t_{2}\right)\leq(1-\lambda)f(t_{1})+\lambda f(t_{2})\quad\mbox{in}\quad\mathbf{C}(G).
\]

\item [{$UG$}] The metric space $(\mathbf{X}(G),d)$ is uniquely geodesic.
\end{lyxlist}
Now $CO$ plainly does not depend upon the choice of $\tau$, and
$UG$ also does not. Indeed, suppose that $\mathbf{X}(G)$ satisfies
all of the above axioms (using $\tau$ in $UG$) and let $\tau'$
be another faithful representation of $G$, giving rise to a distance
$d'$ on $\mathbf{X}(G)$. We have to show that every geodesic segment
$c:[0,1]\rightarrow\mathbf{X}(G)$ in $(\mathbf{X}(G),d')$ is standard,
for then $CO$ implies $UG$ for $(\mathbf{X}(G),d')$. Now for all
$t\in[0,1]$, we have
\[
\begin{array}{rrcll}
 & \mathbf{d}(c(0),c(1)) & \leq & \mathbf{d}(c(0),c(t))+\mathbf{d}(c(t),c(1)) & \mbox{in}\quad\mathbf{C}(G)\\
\mbox{and} & d'(c(0),c(1)) & = & d'(c(0),c(t))+d'(c(t),c(1)) & \mbox{in}\quad\mathbb{R}_{+}
\end{array}
\]
from which easily follows that actually 
\[
\mathbf{d}(c(0),c(1))=\mathbf{d}(c(0),c(t))+\mathbf{d}(c(t),c(1))\quad\mbox{in}\quad\mathbf{C}(G).
\]
Thus $c(t)$ belongs to $\lozenge(c(0),c(1))$ for all $t\in[0,1]$
and $c$ is standard, being indeed contained in any apartment which
contains $c(0)$ and $c(1)$ by corollary~\ref{cor:eqinTR}.

\subsection{~}

By the usual dyadic, reparametrization and triangulation tricks, it
is sufficient to test the inequalities in $CO$ or $CO^{+}$ for $(t_{1},t_{2},\lambda)=(0,1,\frac{1}{2})$,
for pairs of geodesics issuing from the same point. Thus $CO$ is
equivalent to either one of 
\begin{lyxlist}{XXX}
\item [{$CO'$}] For every $x\in\mathbf{X}(G)$, $\mathcal{F},\mathcal{G}\in\mathbf{F}(G)$
and $\lambda\in[0,1]$, 
\[
\mathbf{d}(x+\lambda\mathcal{F},x+\lambda\mathcal{G})\leq\lambda\mathbf{d}(x+\mathcal{F},x+\mathcal{G})\quad\mbox{in}\quad\mathbf{C}(G).
\]

\item [{$CO''$}] For every $x\in\mathbf{X}(G)$ and $\mathcal{F},\mathcal{G}\in\mathbf{F}(G)$,
\[
\mathbf{d}(x+{\textstyle \frac{1}{2}}\mathcal{F},y+{\textstyle \frac{1}{2}}\mathcal{G})\leq{\textstyle \frac{1}{2}}\mathbf{d}(x+\mathcal{F},x+\mathcal{G})\quad\mbox{in}\quad\mathbf{C}(G).
\]

\end{lyxlist}

\subsection{~\label{sub:segments}}

There is a unique, $\mathsf{G}$-equivariant and continuous map
\[
\mathbf{X}(G)\times\mathbf{X}(G)\times[0,1]\rightarrow\mathbf{X}(G),\quad(x,y,\lambda)\mapsto(1-\lambda)x+\lambda y
\]
such that $(1-\lambda)x+\lambda y=x+\lambda\mathcal{F}$ for any $\mathcal{F}\in\mathbf{F}(x,y)$.
We set
\[
[x,y]=\left\{ (1-\lambda)x+\lambda y:\lambda\in[0,1]\right\} 
\]
and call it the \emph{segment} between $x$ and $y$. It is contained
in the enclosure $\lozenge(x,y)$, thus also contained in any apartment
$\mathbf{X}(S)$ which contains $x$ and $y$. In particular, the
intersection of two apartments is a convex subset of both apartments.
A \emph{subdivision} of $[x,y]$ is a finite collection $x=x_{0},\cdots,x_{n}=y$
of points in $[x,y]$ such that 
\[
x_{i}=(1-\lambda_{i})x+\lambda_{i}y,\quad0\leq\lambda_{0}\leq\cdots\leq\lambda_{n}=1.
\]
Thus $[x,y]=\cup_{i=1}^{n}[x_{i-1},x_{i}]$ and 
\[
\mathbf{d}(x,y)={\textstyle \sum}_{i=1}^{n}\mathbf{d}(x_{i-1},x_{i}).
\]

\subsection{~}

For an affine $\mathbf{F}(G)$-building $\mathbf{X}(G)$, we denote
by 
\[
\mathbf{d}^{r}:\mathbf{X}(G)\times\mathbf{X}(G)\rightarrow\mathbf{C}^{r}(G)\quad\mbox{and}\quad\mathbf{d}^{c}:\mathbf{X}(G)\times\mathbf{X}(G)\rightarrow\mathbf{G}(Z)
\]
the components of $\mathbf{d}$. These are $\mathsf{G}$-invariant
functions. For $x,y,z\in\mathbf{X}(G)$, 
\[
\mathbf{d}^{r}(x,z)\leq\mathbf{d}^{r}(x,y)+\mathbf{d}^{r}(y,z)\quad\mbox{and}\quad\mathbf{d}^{c}(x,z)=\mathbf{d}^{c}(x,y)+\mathbf{d}^{c}(y,z).
\]
The function $g\mapsto\mathbf{d}^{c}(x,gx)$ thus does not depend
upon $x$ and defines a morphism
\[
\nu_{\mathbf{X}}^{c}:\mathsf{G}\rightarrow\mathbf{G}(Z).
\]

\subsection{~}

A morphism of affine $\mathbf{F}(G)$-spaces $f:\mathbf{X}(G)\rightarrow\mathbf{Y}(G)$
is a $\mathsf{G}$-equivariant map between the underlying sets which
is compatible with their structure maps: 
\[
f(\mathbf{X}(S))\subset\mathbf{Y}(S)\quad\mbox{and}\quad f(x+\mathcal{F})=f(x)+\mathcal{F}
\]
for every $S\in\mathbf{S}(G)$, $x\in\mathbf{X}(G)$ and $\mathcal{F}\in\mathbf{F}(G)$.
If $\mathbf{Y}(G)$ is an $\mathbf{F}(G)$-building, it is sufficent
to require the second condition. A morphism of affine $\mathbf{F}(G)$-buildings
is a morphism of the underlying affine $\mathbf{F}(G)$-spaces. Any
such morphism is an automorphism: it is bijective on any appartment,
thus globally bijective by $R(s)$. It is compatible with the $\mathbf{d}$-maps,
and an isometry of the underlying metric spaces.

\subsection{~\label{sub:AutoOfBuild}}

An automorphism $\theta$ of an affine $\mathbf{F}(G)$-building $\mathbf{X}(G)$
acts on the apartment $\mathbf{X}(S)$ by an $\mathsf{N}_{G}(S)$-equivariant
translation, which is thus given by a vector $\theta_{S}$ in $\mathbf{G}(Z)=\mathbf{G}(S)^{\mathsf{W}_{G}(S)}$,
where $Z=Z(G)$. The $\mathsf{G}$-equivariance of $\theta$ then
implies that $S\mapsto\theta_{S}$ is also $\mathsf{G}$-equivariant,
thus constant. It follows that 
\[
\Aut(\mathbf{X}(G))=\mathbf{G}(Z)
\]
with $\mathcal{G}\in\mathbf{G}(Z)$ acting on $\mathbf{X}(G)$ by
$x\mapsto x+\mathcal{G}$.

\subsection{~\label{sub:decer4affbuil}}

For an affine $\mathbf{F}(G)$-building $\mathbf{X}(G)$, we define
\[
\mathbf{X}^{r}(G)=\mathbf{X}(G)/\mathbf{G}(Z)\quad\mbox{and}\quad\mathbf{X}^{e}(G)=\mathbf{X}^{r}(G)\times\mathbf{G}(Z).
\]
The group $\mathsf{G}$ acts: on the quotient $\mathbf{X}^{r}(G)$
of $\mathbf{X}(G)$, on $\mathbf{G}(Z)$ by translations through the
morphism $\nu_{\mathbf{X}}^{c}:\mathsf{G}\rightarrow\mathbf{G}(Z)$,
and on $\mathbf{X}^{e}(G)$ diagonally. Then, the formulas 
\[
\mathbf{X}^{r}(S)=\mathbf{X}(S)/\mathbf{G}(Z)\quad\mbox{and}\quad\mathbf{X}^{e}(S)=\mathbf{X}^{r}(S)\times\mathbf{G}(Z)
\]
yield $\mathsf{G}$-equivariant maps $\mathbf{X}^{r}:\mathbf{S}(G)\rightarrow\mathcal{P}(\mathbf{X}^{r}(G))$
and $\mathbf{X}^{e}:\mathbf{S}(G)\rightarrow\mathcal{P}(\mathbf{X}^{e}(G))$,
the pull map on $\mathbf{X}(G)$ descends to a $\mathsf{G}$-equivariant
map $+:\mathbf{X}^{r}(G)\times\mathbf{F}^{r}(G)\rightarrow\mathbf{X}^{r}(G)$,
which together with the addition map on $\mathbf{G}(Z)$ yields a
$\mathsf{G}$-equivariant map 
\[
+:\mathbf{X}^{e}(G)\times\mathbf{F}(G)\rightarrow\mathbf{X}^{e}(G)\qquad([x],\theta)+\mathcal{F}=([x]+\mathcal{F}^{r},\theta+\mathcal{F}^{c}).
\]
The resulting triple $\mathbf{X}^{e}(G)$ is yet another affine $\mathbf{F}(G)$-building,
with $\nu_{\mathbf{X}^{e}}=\nu_{\mathbf{X}}$. In fact, any point
$x_{0}\in\mathbf{X}(G)$ defines an isomorphism of affine $\mathbf{F}(G)$-buildings
\[
\mathbf{X}(G)\simeq\mathbf{X}^{e}(G)\qquad x\mapsto([x],\mathbf{d}^{c}(x_{0},x)).
\]
Thus $\mathbf{X}^{e}(G)$ appears as a rigidified version of $\mathbf{X}(G)$:
there are no non-trivial automorphisms of $\mathbf{X}^{e}(G)$ preserving
the subspace $\mathbf{X}^{r}(G)\simeq\mathbf{X}^{r}(G)\times\{0\}$
of $\mathbf{X}^{e}(G)$. The decomposition $\mathbf{X}^{e}(G)=\mathbf{X}^{r}(G)\times\mathbf{G}(Z)$
is orthogonal in the following sense:
\[
\forall(x,\theta),(x',\theta')\in\mathbf{X}^{e}(G):\qquad d((x,\theta),(x',\theta'))^{2}=d(x,x')^{2}+d(\theta,\theta')^{2}.
\]
This follows from the analogous result for $\mathbf{F}(G)$, see \ref{sub:decisortho4F(G)}.

\subsection{~\label{sub:RenormalizationOfBuildings}}

If $\mathbf{X}(G)=(\mathbf{X}(G),\mathbf{X}(-),+)$ is an affine $\mathbf{F}(G)$-space
or building, then so is $\mathbf{X}_{\lambda}(G)=(\mathbf{X}(G),\mathbf{X}(-),+_{\lambda})$
for any $\lambda>0$ in $\mathbb{R}$, where $x+_{\lambda}\mathcal{F}=x+\lambda\mathcal{F}$.
The types $\nu_{\mathbf{X}}$ of $\mathbf{X}(G)$ and $\nu_{\mathbf{X}_{\lambda}}$
of $\mathbf{X}_{\lambda}(G)$ are related by $\nu_{\mathbf{X}}=\lambda\cdot\nu_{\mathbf{X}_{\lambda}}$.

\section{Further axioms}

Let $\mathbf{X}(G)$ be an affine $\mathbf{F}(G)$-space.

\subsection{The axiom $L(s)^{+}$~}

The following is a sharp strengthening of $L(s)$:
\begin{lyxlist}{XXXX}
\item [{$L(s)^{+}$}] \noindent For any $x\in\mathbf{X}(G)$ and $\mathcal{F},\mathcal{G}\in\mathbf{F}(G)$,
there exists $S\in\mathbf{S}(G)$ and $\epsilon>0$ such that $\mathcal{F}\in\mathbf{F}(S)$
and $x+\lambda\mathcal{G}\in\mathbf{X}(S)$ for every $\lambda\in[0,\epsilon]$.\end{lyxlist}
\begin{prop}
\label{prop:L(s)+implBuild}If $\mathbf{X}(G)$ satisfies $L(s)^{+}$,
$R(s)$, $R(i)$ and $UN$, then it is an affine $\mathbf{F}(G)$-building
and $(\mathbf{X}(G),d)$ is a $CAT(0)$-space.
\end{prop}
Suppose that $\mathbf{X}(G)$ satisfies $L(s)^{+}$, $R(s)$, $R(i)$
and $UN$. Then it already satisfies all the axioms of section~\ref{sub:Many-apartments}
and \ref{sub:Strong-transitivity}, giving rise to the vectorial distance
$\mathbf{d}$ which is the subject of the remaining axioms. We do
not yet know that $\mathbf{d}$ satisfies $TR$, thus $d=\left\Vert \mathbf{d}\right\Vert $
may not be a distance on $\mathbf{X}(G)$. But for any apartment $\mathbf{X}(S)$,
the restriction of $\mathbf{d}$ to $\mathbf{X}(S)$ satisfies $TR$
and $d$ is a Euclidean distance on $\mathbf{X}(S)$.
\begin{lem}
\label{lem:ReformL(s)+}For $x\in\mathbf{X}(G)$ and $\mathcal{F},\mathcal{G}\in\mathbf{F}(G)$,
there exists $S\in\mathbf{S}(G)$, $\mathcal{G}^{\ast}\in\mathbf{F}(G)$
and $\epsilon>0$ such that $x\in\mathbf{X}(S)$, $\mathcal{F},\mathcal{G}^{\ast}\in\mathbf{F}(S)$
and 
\[
\forall\lambda\in[0,\epsilon]:\quad x+\lambda\mathcal{G}=x+\lambda\mathcal{G}^{\ast}\in\mathbf{X}(S).
\]
\end{lem}
\begin{proof}
By $L(s)^{+}$, there exists $S\in\mathbf{S}(G)$ and $\epsilon>0$
such that $\mathcal{F}\in\mathbf{F}(S)$ and $x(\lambda)=x+\lambda\mathcal{G}\in\mathbf{X}(S)$
for $\lambda\in[0,\epsilon]$. For any $0\leq\lambda\leq\lambda'$,
$x(\lambda')=x(\lambda)+(\lambda'-\lambda)\mathcal{G}$ by $AC$,
thus $\mathbf{d}(x(\lambda),x(\lambda'))=(\lambda'-\lambda)\cdot t(\mathcal{G})$
in $\mathbf{C}(G)$ and $d(x(\lambda),x(\lambda'))=(\lambda'-\lambda)\cdot\left\Vert \mathcal{G}\right\Vert $
in $\mathbb{R}_{+}$. In particular, $x(-):[0,\epsilon]\rightarrow\mathbf{X}(S)$
is a geodesic segment in $\left(\mathbf{X}(S),d\vert_{\mathbf{X}(S)}\right)$.
There is thus a unique $\mathcal{G}^{\ast}\in\mathbf{F}(S)$ such
that $x(\lambda)=x+\lambda\mathcal{G}^{\ast}$ for $\lambda\in[0,\epsilon]$. \end{proof}
\begin{lem}
For any $x,y\in\mathbf{X}(G)$ and $\mathcal{G}\in\mathbf{F}(G)$,
the function
\[
c:\mathbb{R}_{+}\rightarrow\mathbf{C}(G),\qquad c(\lambda)=\mathbf{d}(y,x+\lambda\mathcal{G})
\]
is continuous for the canonical topologies on $\mathbb{R}_{+}$ and
$\mathbf{C}(G)$.\end{lem}
\begin{proof}
Pick $\mathcal{F}\in\mathbf{F}(G)$ with $y=x+\mathcal{F}$ using
$T(s)$. By the previous lemma, there exists $S\in\mathbf{S}(G)$,
$\mathcal{G}^{\ast}\in\mathbf{F}(S)$ and $\epsilon>0$ such that
$x\in\mathbf{X}(S)$, $\mathcal{F}\in\mathbf{F}(S)$ and $x+\lambda\mathcal{G}=x+\lambda\mathcal{G}^{\ast}$
for $\lambda\in[0,\epsilon]$. Since $x,y\in\mathbf{X}(S)$ and $\mathcal{G}^{\ast}\in\mathbf{F}(S)$,
the function $\lambda\mapsto\mathbf{d}(y,x+\lambda\mathcal{G}^{\ast})$
is plainly continuous on $\mathbb{R}_{+}$, thus $c$ is continuous
on $[0,\epsilon]$. Changing $x$ to $x+\lambda\mathcal{G}$, we find
that $c$ is right continous on $\mathbb{R}_{+}$. By $L(s)$, there
is an $S'\in\mathbf{S}(G)$ with $x\in\mathbf{X}(S')$, $\mathcal{G}\in\mathbf{F}(S')$.
Set $\mathcal{G}'=\iota_{S}\mathcal{G}$. Then for $\lambda'\geq\lambda\geq0$,
\[
\begin{array}{rrcll}
 & x(\lambda) & = & x(\lambda')+(\lambda'-\lambda)\cdot\mathcal{G}' & \mbox{in}\quad\mathbf{X}(G),\\
\mbox{thus} & c(\lambda) & = & \mathbf{d}\left(y,x(\lambda')+(\lambda'-\lambda)\cdot\mathcal{G}'\right) & \mbox{in}\quad\mathbf{C}(G).
\end{array}
\]
It follows that $c$ is also left continuous on $\mathbb{R}_{+}$.\end{proof}
\begin{lem}
For any $x\in\mathbf{X}(G)$ and $\mathcal{F},\mathcal{G}\in\mathbf{F}(G)$,
\[
x+\mathcal{F}=x+\mathcal{G}\quad\Longrightarrow\quad\forall\lambda\in[0,1]:\, x+\lambda\mathcal{F}=x+\lambda\mathcal{G}.
\]
\end{lem}
\begin{proof}
Suppose $x+\mathcal{F}=x+\mathcal{G}$, put $x(\lambda)=x+\lambda\mathcal{F}$,
$y(\lambda)=x+\lambda\mathcal{G}$ and define 
\[
\lambda_{0}=\inf\{1,\lambda\in[0,1]\mbox{ such that }x(\lambda)\neq y(\lambda)\}.
\]
Suppose that $\lambda_{0}\in[0,1[$. If $\lambda_{0}\neq0$, then
since $x(\lambda)=y(\lambda)$ for all $\lambda\in[0,\lambda_{0}[$,
\[
\mathbf{d}(x(\lambda_{0}),y(\lambda_{0}))=\lim_{s\rightarrow s_{0}^{-}}\mathbf{d}(x(\lambda_{0}),y(\lambda))=\lim_{s\rightarrow s_{0}^{-}}\mathbf{d}(x(\lambda_{0}),x(\lambda))=0
\]
by the previous lemma, thus $x(\lambda_{0})=y(\lambda_{0})$. Changing
$(x,\mathcal{F},\mathcal{G})$ to 
\[
\left(x(\lambda_{0}),(1-\lambda_{0})\mathcal{F},(1-\lambda_{0})\mathcal{G}\right),
\]
we may assume that $\lambda_{0}=0$. By lemma~\ref{lem:ReformL(s)+},
there exists $S\in\mathbf{S}(G)$, $\mathcal{G}^{\ast}\in\mathbf{F}(S)$
and $\epsilon>0$ such that $x\in\mathbf{X}(S)$, $\mathcal{F}\in\mathbf{F}(S)$
and $y(\lambda)=x+\lambda\mathcal{G}^{\ast}$ for $\lambda\in[0,\epsilon]$.
Since $x+\epsilon\mathcal{G}=x+\epsilon\mathcal{G}^{\ast}$ and $x+\mathcal{G}=x+\mathcal{F}$
in $\mathbf{X}(G)$, $t(\mathcal{G}^{\ast})=t(\mathcal{G})=t(\mathcal{F})$
in $\mathbf{C}(G)$ with $\mathcal{F},\mathcal{G}^{\ast}\in\mathbf{F}(S)$,
thus $\mathcal{G}^{\ast}=w\mathcal{F}$ for some $w\in\mathsf{W}_{G}(S)$.
In the affine $\mathbf{F}(S)$-space $\mathbf{X}(S)$, 
\[
x(1)=x+\mathcal{F}=(x+\lambda\mathcal{G}^{\ast})+(\mathcal{F}-\lambda\mathcal{G}^{\ast})=y(\lambda)+(\mathcal{F}-\lambda w\mathcal{F})
\]
for all $\lambda\in[0,\epsilon]$. Since $x(1)=y(1)$, we thus find
that for $\lambda\in[0,\epsilon]$, 
\[
t((1-\lambda)\mathcal{F})=(1-\lambda)t(\mathcal{G})=\mathbf{d}(y(\lambda),y(1))=\mathbf{d}(y(\lambda),x(1))=t(\mathcal{F}-\lambda w\mathcal{F}).
\]
Let $C$ be a closed chamber in $\mathbf{F}(S)$ such that $\mathcal{F}-\lambda w\mathcal{F}\in C$
for all $\lambda\in[0,\epsilon]$ (shrinking $\epsilon$ if necessary).
Since $t$ is injective on $C$, $(1-\lambda)\mathcal{F}=\mathcal{F}-\lambda w\mathcal{F}$
in $C\subset\mathbf{F}(S)$ for all $\lambda\in[0,\epsilon]$, thus
$\mathcal{F}=w\mathcal{F}=\mathcal{G}^{\ast}$. But then $x(\lambda)=y(\lambda)$
for all $\lambda\in[0,\epsilon]$, a contradiction. Therefore $\lambda_{0}=1$,
i.e.~$x(\lambda)=y(\lambda)$ for all $\lambda\in[0,1]$.
\end{proof}
Using $R(s)$ and the previous lemma, we may now define segments in
$\mathbf{X}(G)$ and their subdivisions as in section~\ref{sub:segments},
with $[x,y]\subset\mathbf{X}(S)$ if $x,y\in\mathbf{X}(S)$.
\begin{lem}
\label{Lem:Ls+impliesSubdiv}For every $x,y\in\mathbf{X}(G)$ and
$z\in\mathbf{X}(G)$ (resp. $\mathcal{F}\in\mathbf{F}(G)$), there
exists a subdivision $x=x_{0},\cdots,x_{n}=y$ of the segment $[x,y]$
and for $i\in\{1,\cdots,n\}$, an $S_{i}\in\mathbf{S}(G)$ such that
$[x_{i-1},x_{i}]\subset\mathbf{X}(S_{i})$ and $z\in\mathbf{X}(S_{i})$
(resp.~$\mathcal{F}\in\mathbf{F}(S_{i})$).\end{lem}
\begin{proof}
By $R(s)$, there is an $S\in\mathbf{S}(G)$ such that $x,y\in\mathbf{X}(S)$,
so that 
\[
y=x+\mathcal{G}^{+}\quad\mbox{and}\quad x=y+\mathcal{G}^{-}\quad\mbox{with}\quad\mathcal{G}^{\pm}\in\mathbf{F}(S),\quad\mathcal{G}^{+}+\mathcal{G}^{-}=0.
\]
For $\lambda\in[0,1]$, set $x(\lambda)=x+\lambda\mathcal{G}^{+}$
and choose $\mathcal{F}_{\lambda}\in\mathbf{F}(G)$ such that $z=x(\lambda)+\mathcal{F}_{\lambda}$
using $T(s)$. By $L(s)^{+}$, there exists $\epsilon_{\lambda}>0$
and $S_{\lambda}^{\pm}\in\mathbf{S}(G)$ such that $\mathcal{F}_{\lambda}\in\mathbf{F}(S_{\lambda}^{\pm})$
(resp.~$\mathcal{F}\in\mathbf{F}(S_{\lambda}^{\pm})$) and $x(\lambda)+\mu\mathcal{G}^{\pm}\in\mathbf{X}(S_{\lambda}^{\pm})$
for all $\mu\in[0,\epsilon_{\lambda}]$. Pick a finite set $\mathcal{S}\subset[0,1]$
such that $[0,1]\subset\cup_{\lambda\in\mathcal{S}}]\lambda-\epsilon_{\lambda},\lambda+\epsilon_{\lambda}[$
and let $x=x_{0},\cdots,x_{n}=y$ be the subdivision of $[x,y]$ defined
by $[x,y]\cap\{x,y,x(\lambda),x(\lambda)\pm\epsilon_{\lambda}:\lambda\in\mathcal{S}\}$.
Then for each $i\in\{1,\cdots,n\}$, there exists an $S_{i}\in\{S_{\lambda}^{\pm}:\lambda\in\mathcal{S}\}$
such that $[x_{i-1},x_{i}]\in\mathbf{X}(S_{i})$ and $z\in\mathbf{X}(S_{i})$
(resp.~$\mathcal{F}\in\mathbf{F}(S_{i})$).\end{proof}
\begin{lem}
For a minimal parabolic subgroup $B=U\rtimes Z_{G}(S)$ of $G$, the
apartment $\mathbf{X}(S)$ is a fundamental domain for the action
of $\mathsf{U}$ on $\mathbf{X}(G)$ and the corresponding retraction
$r:\mathbf{X}(G)\twoheadrightarrow\mathbf{X}(S)$ is non-expanding
for $\mathbf{d}$.\end{lem}
\begin{proof}
First, $\mathbf{X}(G)=\mathsf{U}\cdot\mathbf{X}(S)$ by $L(s)$. For
$x,y\in\mathbf{X}(S)$ and any $\mathcal{F}\in\mathbf{F}(S)$, $\mathbf{d}(x,y)=\mathbf{d}(x+\mathcal{F},y+\mathcal{F})$
in $\mathbf{C}(G)$. If $y=ux$ with $u\in\mathsf{U}$ and $P_{\mathcal{F}}=B$,
then also 
\[
\lim_{s\rightarrow\infty}\mathbf{d}(x+s\mathcal{F},y+s\mathcal{F})=0
\]
by $UN$, thus $\mathbf{d}(x,y)=0$ and $x=y$, i.e.~$\mathbf{X}(S)$
is a fundamental domain for the action of $\mathsf{U}$ on $\mathbf{X}(G)$.
Let $r:\mathbf{X}(G)\twoheadrightarrow\mathbf{X}(S)$ be the corresponding
retraction. For $x,y\in\mathbf{X}(G)$, there exists by the previous
lemma a subdivision $x=x_{0},\cdots,x_{n}=y$ of $[x,y]$ and for
each $i\in\{1,\cdots,n\}$, an $S_{i}\in\mathbf{S}(G)$ such that
$[x_{i-1},x_{i}]\subset\mathbf{X}(S_{i})$ and $Z_{G}(S_{i})\subset B$.
Then, there is a unique $u_{i}\in\mathsf{U}$ such that $\Int(u_{i})(S_{i})=S$,
in which case also $u_{i}\cdot\mathbf{X}(S_{i})=\mathbf{X}(S)$ and
$r(z)=u_{i}z$ for all $z\in\mathbf{X}(S_{i})$. We thus obtain
\begin{eqnarray*}
\mathbf{d}(x,y) & = & {\textstyle \sum}_{i=1}^{n}\mathbf{d}(x_{i-1},x_{i})\\
 & = & {\textstyle \sum}_{i=1}^{n}\mathbf{d}(u_{i}x_{i-1},u_{i}x_{i})\\
 & = & {\textstyle \sum}_{i=1}^{n}\mathbf{d}\left(r(x_{i-1}),r(x_{i})\right)\\
 & \geq & \mathbf{d}\left(r(x),r(y)\right)
\end{eqnarray*}
in $\mathbf{C}(G)$ by the known triangle inequality for $\mathbf{d}$
in $\mathbf{X}(S)$. \end{proof}
\begin{lem}
The vectorial distance $\mathbf{d}$ satisfies $TR$, $NE$, $CO$
and $(\mathbf{X}(G),d)$ is a $CAT(0)$-metric space -- thus $\mathbf{X}(G)$
also satisfies $UG$.\end{lem}
\begin{proof}
For $x,y,z\in\mathbf{X}(G)$, choose $S\in\mathbf{S}(G)$ with $x,z\in\mathbf{X}(S)$
using $R(s)$, pick a minimal parabolic subgroup $B$ of $G$ with
Levi $Z_{G}(S)$ and let $r:\mathbf{X}(G)\twoheadrightarrow\mathbf{X}(S)$
be the corresponding retraction. Then 
\[
\mathbf{d}(x,z)\leq\mathbf{d}\left(x,r(y)\right)+\mathbf{d}\left(r(y),z\right)\leq\mathbf{d}(x,y)+\mathbf{d}(y,z)
\]
by the triangle inequality in $\mathbf{X}(S)$ and the previous lemma.
This proves $TR$. 

For $x,y\in\mathbf{X}(G)$ and $\mathcal{F}\in\mathbf{F}(G)$, pick
a subdivision $x=x_{0},\cdots,x_{n}=y$ of $[x,y]$ and for each $i\in\{1,\cdots,n\}$,
an $S_{i}\in\mathbf{S}(G)$ such that $[x_{i-1},x_{i}]\in\mathbf{X}(S_{i})$
and $\mathcal{F}\in\mathbf{F}(S_{i})$, using lemma~\ref{Lem:Ls+impliesSubdiv}.
Then 
\[
\mathbf{d}(x+\mathcal{F},y+\mathcal{F})\leq{\textstyle \sum}_{i=1}^{n}\mathbf{d}(x_{i-1}+\mathcal{F},x_{i}+\mathcal{F})={\textstyle \sum}_{i=1}^{n}\mathbf{d}(x_{i-1},x_{i})=\mathbf{d}(x,y)
\]
by the triangle inequality in $\mathbf{X}(G)$ that we have just proven
and a trivial computation in the affine $\mathbf{F}(S_{i})$-space
$\mathbf{X}(S_{i})$. This proves $NE$. 

For $x\in\mathbf{X}(G)$ and $\mathcal{F},\mathcal{G}\in\mathbf{F}(G)$,
set $y=x+\mathcal{F}$, $z=x+\mathcal{G}$. By lemma~\ref{Lem:Ls+impliesSubdiv},
there is a subdivision $y=x_{0},\cdots,x_{n}=z$ of the segment $[y,z]$
and for each $i\in\{1,\cdots,n\}$, an $S_{i}\in\mathbf{S}(G)$ such
that $[x_{i-1},x_{i}]\in\mathbf{X}(S_{i})$ and $x\in\mathbf{X}(S_{i})$.
For $i\in\{0,\cdots,n\}$ and $\lambda\in[0,1]$, set $x_{i}(\lambda)=(1-\lambda)x+\lambda x_{i}$
in $[x,x_{i}]$. Then 
\[
\mathbf{d}(x_{0}(\lambda),x_{n}(\lambda))\leq{\textstyle \sum}_{i=1}^{n}\mathbf{d}(x_{i-1}(\lambda),x_{i}(\lambda))={\textstyle \sum}_{i=1}^{n}\lambda\mathbf{d}(x_{i-1},x_{i})=\lambda\mathbf{d}(y,z)
\]
by the triangle inequality in $\mathbf{X}(G)$ and a trivial computation
in the affine $\mathbf{F}(S_{i})$-space $\mathbf{X}(S_{i})$. This
proves $CO'$, from which $CO$ follows. 

To establish the $CAT(0)$-property, imagine a rigid comparison triangle
$(\tilde{x},\tilde{y},\tilde{z})$ for $(x,y,z)$, lying on a Euclidean
$2$-plane $E$. Add flex points $(\tilde{x}_{1},\cdots,\tilde{x}_{n-1})$
on the segment $[\tilde{y},\tilde{z}]$ corresponding to $(x_{1},\cdots,x_{n-1})$,
and push them (inward or outward) one by one, so that each $(\tilde{x},\tilde{x}_{i-1},\tilde{x}_{i})$
becomes a comparison triangle for $(x,x_{i-1},x_{i})$ (with $\tilde{x}_{0}=\tilde{y}$
and $\tilde{x}_{n}=\tilde{z}$). If a last outward move occurs at
the $i$-th step, then in the final configuration, the chord between
$\tilde{x}_{i-1}$ and $\tilde{x}_{i+1}$ intersects the radius between
$\tilde{x}$ and $\tilde{x}_{i}$ at some point $\tilde{y}=(1-\nu)\tilde{x}+\nu\tilde{x}_{i}$,
with $\nu\in[0,1[$. For the corresponding point $y=(1-\nu)x+\nu x_{i}$
on the segment $[x,x_{i}]\subset\mathbf{X}(S_{i})\cap\mathbf{X}(S_{i+1})$,
we would have: 
\begin{eqnarray*}
d(x_{i-1},y)+d(y,x_{i+1}) & = & d(\tilde{x}_{i-1},\tilde{y})+d(\tilde{y},\tilde{x}_{i+1})\\
 & < & d(\tilde{x}_{i-1},\tilde{x}_{i})+d(\tilde{x}_{i},\tilde{x}_{i+1})\\
 & = & d(x_{i-1},x_{i})+d(x_{i},x_{i+1})\\
 & = & d(x_{i-1},x_{i+1})
\end{eqnarray*}
which contradicts the triangle inequality for $d$ in $\mathbf{X}(G)$.
It follows that there is no last outward move, i.e.~no outward move
at all. Thus for any $\lambda\in[0,1]$, if $\tilde{x}(\lambda)$
is the point corresponding to $x(\lambda)=(1-\lambda)y+\lambda z\in[y,z]$
on the articulated segment $[\tilde{y},\tilde{z}]$ of our comparison
triangle, the distance between $\tilde{x}$ and $\tilde{x}(\lambda)$
is not greater in the final configuration than it was initially. Since
the final distance is the actual distance between $x$ and $x(\lambda)$
in $(\mathbf{X}(G),d)$, this proves the required $CAT(0)$ inequality
for $x$ and the \emph{standard }geodesic segment $x(-):[0,1]\rightarrow\mathbf{X}(G)$
from $y$ to $z$. However, we still have to check that our metric
space $(\mathbf{X}(G),d)$ is unically geodesic in the usual sense.
Suppose therefore that $x'(-):[0,1]\rightarrow\mathbf{X}(G)$ is another
geodesic segment between $y$ and $z$. For $\lambda\in[0,1]$, the
$CAT(0)$-inequality that we have just established for the point $x'(\lambda)$
and the standard geodesic $x(-):[0,1]\rightarrow\mathbf{X}(G)$ implies
that $x(\lambda)=x'(\lambda)$, thus indeed $x(-)=x'(-)$.
\end{proof}

\subsection{Discrete Buildings}

The following axiom is a strengthening of $R(i)$:
\begin{lyxlist}{MM}
\item [{$R(i)^{+}$}] For $S,S'\in\mathbf{S}(G)$, there is a $g\in\mathsf{G}$
with 
\[
\Int(g)(S)=S'\quad\mbox{and}\quad g\equiv\mathrm{Id}\mbox{ on }\mathbf{X}(S)\cap\mathbf{X}(S').
\]
\end{lyxlist}
\begin{lem}
A discrete affine $\mathbf{F}(G)$-building $\mathbf{X}(G)$ satisfies
$R(i)^{+}$.\end{lem}
\begin{proof}
We may assume that $Z=\mathbf{X}(S)\cap\mathbf{X}(S')\neq\emptyset$.
Then $Z$ is a non-empty closed convex subset of the affine $\mathbf{F}(S)$-space
$\mathbf{X}(S)$, therefore $Z$ has non-empty interior as a subset
of its affine span $A$ in $\mathbf{X}(S)$. Let $\sim$ be the equivalence
relation on $Z$ defined by $x\sim y$ if and only if $x$ and $y$
have the same stabilizer in $\mathsf{N}_{G}(S)$. Since $\mathbf{X}(G)$
is discrete, there are countably many equivalence classes, thus one
of them at least, say $E\subset Z$, has the property that the closure
of $E$ has a non-empty interior in $A$. Then $A$ is also the affine
span of $\overline{E}$ or $E$ in $\mathbf{X}(S)$. Let $C\subset\mathsf{N}_{G}(S)$
be the common stabilizer of the points of $E$. Now for any $g_{1},g_{2}\in\mathsf{G}$
such that $\Int(g_{i})(S)=S'$ and $g_{1}x=g_{2}x$ for some $x\in E$,
$g_{2}=g_{1}c$ for some $c\in C$, thus $g_{1}\equiv g_{2}$ on $E$,
$A$ and $Z$. Fix $x\in E$. Then for any $y\in Z$, there exists
by $R(i)$ some $g_{y}\in\mathsf{G}$ such that $\Int(g_{y})(S)=S'$
and $g_{y}x=x$, $g_{y}y=y$. For $y,z\in Z$, we have just seen that
$g_{y}\equiv g_{z}$ on $E$, $A$ and $Z$, thus $g_{y}z=g_{z}z=z$
and $g_{y}\equiv\mathrm{Id}$ on $Z$. \end{proof}
\begin{lem}
\label{lem:DiscImpliesComplete}The metric of a discrete affine $\mathbf{F}(G)$-building
is complete.\end{lem}
\begin{proof}
Suppose that $\mathbf{X}(G)$ is discrete and equip $\mathcal{C}=\mathsf{G}\backslash\mathbf{X}(G)$
with 
\[
\overline{d}(\alpha,\beta)=\inf D(\alpha,\beta)\quad\mbox{where}\quad D(\alpha,\beta)=\left\{ d(x,y):x\in\alpha,y\in\beta\right\} .
\]
Fix $S\in\mathbf{S}(G)$. Then by $R(i)$ and $R(s)$, also $\mathcal{C}=\mathsf{N}_{G}(S)\backslash\mathbf{X}(S)$
and 
\[
D(\alpha,\beta)=\left\{ d(a,n\cdot b):n\in\mathsf{N}_{G}(S)\right\} 
\]
if $(a,b)$ lifts $(\alpha,\beta)$ in $\mathbf{X}(S)$. Since $\mathsf{N}_{G}(S)\cdot b$
is discrete in the Euclidean space $\mathbf{X}(S)$ by assumption,
if follows that there is a constant $\epsilon(\alpha,\beta)>0$ such
that 
\[
\forall(x,y)\in\alpha\times\beta:\mbox{ }d(x,y)\leq\overline{d}(\alpha,\beta)+\epsilon(\alpha,\beta)\Longrightarrow d(x,y)=\overline{d}(\alpha,\beta).
\]
In particular, $\overline{d}$ is a distance on $\mathcal{C}$. Moreover,
$(\mathcal{C},\overline{d})$ is complete: if $(\alpha_{n})$ is a
Cauchy sequence in $\mathcal{C}$, it lifts to a bounded sequence
$(a_{n})$ in $\mathbf{X}(S)$, the latter has a subsequence $(a_{\varphi(n)})$
converging to some $a$ in $\mathbf{X}(S)$, whose image in $\mathcal{C}$
is then a limit of $(\alpha_{n})$. Let now $(x_{n})$ be a Cauchy
sequence in $(\mathbf{X}(G),d)$. Its image $(\alpha_{n})$ is a Cauchy
sequence in $(\mathcal{C},\overline{d})$, which thus converges to
some $\alpha$ in $\mathcal{C}$. For each $n$, lift $\alpha$ to
some $y_{n}\in\mathbf{X}(G)$ with $d(x_{n},y_{n})=\overline{d}(\alpha_{n},\alpha)$.
Then $(y_{n})$ is also a Cauchy sequence in $\mathbf{X}(G)$, hence
$d(y_{n},y_{m})\leq\epsilon(\alpha,\alpha)$ for $n,m\gg0$, which
implies that $(y_{n})$ is actually stationary and $(x_{n})$ converges
to its limit: $(\mathbf{X}(G),d)$ is complete.
\end{proof}

\section{Walls and tight buildings}

Let again $\mathbf{X}(G)$ be an affine $\mathbf{F}(G)$-space.

\subsection{~\label{sub:NotationsForRoots}}

For $S\in\mathbf{S}(G)$, let $\Phi(S,G)$\nomenclature[Phi(S,G)]{$\Phi (S,G)$}{Roots of $S$ in $\Lie (G)$.}
be the set of roots of $S$ in the Lie algebra 
\[
\Lie(G)=\mathfrak{g}=\mathfrak{g}_{0}\oplus\oplus_{a\in\Phi(S,G)}\mathfrak{g}_{a}.
\]
We denote by $U_{a}\subset G$\nomenclature[U_a]{$U_a$}{Root subgroup of $G$ for $a\in \Phi (G,S)$, page \nomrefpage}
the root subgroup corresponding to some $a\in\Phi(G,S)$: if $S_{a}$
denotes the neutral component of the kernel of $a:S\rightarrow\mathbb{G}_{m,K}$,
then $U_{a}$ is the unipotent radical of the unique parabolic subgroup
of $Z_{G}(S_{a})$ containing $Z_{G}(S)$ with Lie algebra $\mathfrak{g}_{0}\oplus\oplus_{b\in\mathbb{N}a\cap\Phi(S,G)}\mathfrak{g}_{b}$.
If $2a\in\Phi(S,G)$, then $U_{2a}\subset U_{a}$.

\subsection{~\label{sub:defWalls4affBuild}}

For any $u\in\mathsf{U}_{a}\setminus\{1\}$, there exists a unique
triple $(u_{1},u_{2},m(u))$ with 
\[
u_{1}uu_{2}=m(u),\quad u_{1},u_{2}\in\mathsf{U}_{-a}\setminus\{1\}\quad\mbox{and}\quad m(u)\in\mathsf{N}_{G}(S).
\]
Moreover, $\nu_{S}^{v}(m(u))$ is the symmetry $s_{a}\in\mathsf{W}_{G}(S)$
attached to $a$, given by 
\[
s_{a}:\mathbf{G}(S)\rightarrow\mathbf{G}(S),\quad s_{a}(x)=x-a(x)a^{\vee}
\]
where $a^{\vee}:\mathbb{G}_{m,K}\rightarrow S$ is the coroot corresponding
to $a$ and $a(x)=a\circ x$ in 
\[
\mathbb{R}=\Hom(\mathbb{D}_{K}(\mathbb{R}),\mathbb{G}_{m,K}).
\]
This follows from \cite[\S 5]{BoTi65} by \cite[6.1.2.2 \& 6.1.3.c]{BrTi72}.
Considering the action of $m(u)$ on $\mathbf{X}(S)$, we thus obtain
a unique affine hyperplane $\mathbf{X}(S,u)$ in $\mathbf{X}(S)$
which is preserved by $m(u)$. The underlying vector space is the
fixed point set of $s_{a}$, namely 
\[
\mathbf{G}(S_{a})=\{x\in\mathbf{G}(S):s_{a}(x)=x\}=\{x\in\mathbf{G}(S):a(x)=0\}
\]
and $m(u)$ acts on $\mathbf{X}(S,u)$ by $x\mapsto x+\nu_{\mathbf{X}}(S,u)$
for some $v_{\mathbf{X}}(S,u)\in\mathbf{G}(S_{a})$.
\begin{example}
For the affine $\mathbf{F}(G)$-space $\mathbf{X}(G)=\mathbf{F}(G)$,
$\mathbf{X}(S,u)=\mathbf{G}(S_{a})$ is the fixed point set of $m(u)$
acting on $\mathbf{G}(S)=\mathbf{F}(S)$ and $\nu_{\mathbf{X}}(S,u)=0$. 
\end{example}

\subsection{~}

Of course $m(u)$ fixes $\mathbf{X}(S,u)$ if and only if $\nu_{\mathbf{X}}(S,u)=0$,
and this happens when $m(u)$ already has finite order in $\mathsf{N}_{G}(S)$,
which holds true for any $u\in\mathsf{U}_{a}\setminus\{1\}$ if $2a\notin\Phi(S,G)$.
Indeed, set $\Phi'(S,G)=\{b\in\Phi(S,G):2b\notin\Phi(S,G)\}$. This
is again a root system and $U_{b}\simeq\mathbb{G}_{a,K}^{n(b)}$ for
some $n(b)\geq1$ for all $b\in\Phi'(S,G)$. Choose a set of simple
roots $\Delta'$ of $\Phi'(S,G)$ containing $a$ and choose for each
$b\in\Delta'$ a $1$-dimensional $K$-subspace $\mathsf{U}'_{b}$
in $\mathsf{U}_{b}\simeq K^{n(b)}$, with $u\in\mathsf{U}'_{a}$.
Then by~\cite[7.2]{BoTi65}, there is a unique split reductive subgroup
$G'$ of $G$ containing $S$ with $\Phi(S,G')=\Phi'(S,G)$ such that
the root subgroup $U'_{b}$ of $b\in\Delta'$ in $G'$ is the subgroup
of $U_{b}$ determined by $\mathsf{U}'_{b}$, i.e.~$\mathsf{U}'_{b}=U'_{b}(K)$.
Then $(Z_{G'}(S_{a}),S,a)$ is an elementary system in the sense of
\cite[XX 1.3]{SGA3.3r} by~\cite[XIX 3.9]{SGA3.3r}. Let $f:S_{\mathcal{L}}\rightarrow Z_{G'}(S_{a})$
be the corresponding morphism constructed in~\cite[XX 5.8]{SGA3.3r}
and let $X\neq0$ be the unique element of $\mathcal{L}=\Lie(U'_{a})$
with $f(\begin{smallmatrix}1 & X\\
0 & 1
\end{smallmatrix})=u$. Since
\[
\left(\begin{array}{cc}
1 & 0\\
-X^{-1} & 1
\end{array}\right)\left(\begin{array}{cc}
1 & X\\
0 & 1
\end{array}\right)\left(\begin{array}{cc}
1 & 0\\
-X^{-1} & 1
\end{array}\right)=\left(\begin{array}{cc}
0 & X\\
-X^{-1} & 0
\end{array}\right)
\]
in $S_{\mathcal{L}}(K)$, we find that 
\[
m(u)=f\left(\begin{array}{cc}
0 & X\\
-X^{-1} & 0
\end{array}\right),\quad m(u)^{2}=f\left(\begin{array}{cc}
-1 & 0\\
0 & -1
\end{array}\right)\quad\mbox{and}\quad m(u)^{4}=1.
\]
On the other hand if $2a\in\Phi(G,S)$, then \cite[1.15]{Ti79} provides
examples where $m(u)$ has infinite order. Note also that $m(u)$
fixes $\mathbf{X}(S,u)$ when there is a $z\in\mathsf{Z}_{G}(S)$
such that $zuz^{-1}=u^{-1}$: since $m(u^{-1})=m(u)^{-1}$ and $m(zuz^{-1})=zm(u)z^{-1}$,
\[
\begin{array}{rcl}
\mathbf{X}(S,u^{-1}) & = & \mathbf{X}(S,u)\\
\nu_{\mathbf{X}}(S,u^{-1}) & = & -\nu_{\mathbf{X}}(S,u)
\end{array}\qquad\mbox{and}\qquad\begin{array}{rcl}
\mathbf{X}(S,zuz^{-1}) & = & \mathbf{X}(S,u)+\nu_{\mathbf{X},S}(z)\\
\nu_{\mathbf{X}}(S,zuz^{-1}) & = & \nu_{\mathbf{X}}(S,u)
\end{array}
\]
therefore $zuz^{-1}=u^{-1}$ implies $\nu_{\mathbf{X}}(S,u)=0$. Note
also that since
\[
\left(m(u)u_{2}m(u)^{-1}\right)u_{1}u=m(u)\quad\mbox{and}\quad uu_{2}\left(m(u)^{-1}u_{1}m(u)\right)=m(u)
\]
we find that $m(u_{1})=m(u_{2})=m(u)$, thus 
\[
\mathbf{X}(S,u_{1})=\mathbf{X}(S,u_{2})=\mathbf{X}(S,u)\quad\mbox{and}\quad\nu_{\mathbf{X}}(S,u_{1})=\nu_{\mathbf{X}}(S,u_{2})=\nu_{\mathbf{X}}(S,u).
\]

\subsection{~}

For a subset $\Omega\neq\emptyset$ of $\mathbf{X}(S)$, we denote
by $\mathsf{G}_{\Omega}$\nomenclature[G_Omega]{$\mathsf{G}_\Omega$}{Pointwise stabilizer of $\Omega$ in $\mathsf{G}$, page \nomrefpage}
the pointwise stabilizer of $\Omega$ in $\mathsf{G}$ and by $\mathsf{G}_{S,\Omega}$\nomenclature[G_S,Omega]{$\mathsf{G}_{S,\Omega}$}{Apartment based avatar of $\mathsf{G}_{\Omega}$, page \nomrefpage}
the subgroup of $\mathsf{G}$ spanned by $\mathsf{N}_{G}(S)_{\Omega}=\mathsf{G}_{\Omega}\cap\mathsf{N}_{G}(S)$
and 
\[
\left\{ u\in\mathsf{U}_{a}\setminus\{1\}:a\in\Phi(G,S),\,\Omega\subset\mathbf{X}^{+}(S,u)\right\} 
\]
where for any $a\in\Phi(G,S)$ and $u\in\mathsf{U}_{a}\setminus\{1\}$,
\[
\mathbf{X}^{+}(S,u)=\mathbf{X}(S,u)+\left\{ \mathcal{F}\in\mathbf{F}(S):a(\mathcal{F})\geq0\right\} .
\]
When $\Omega=\{x\}$, we simply write $\mathsf{G}_{x}=\mathsf{G}_{\{x\}}$
and $\mathsf{G}_{S,x}=\mathsf{G}_{S,\{x\}}$. Thus 
\[
\mathsf{G}_{\Omega}=\cap_{x\in\Omega}\mathsf{G}_{x}\quad\mbox{and}\quad\mathsf{G}_{S,\Omega}\subset\cap_{x\in X}\mathsf{G}_{S,x}.
\]
For $\emptyset\neq\Omega'\subset\Omega\subset\mathbf{X}(S)$, $\mathsf{G}_{\Omega}\subset\mathsf{G}_{\Omega'}$
and $\mathsf{G}_{S,\Omega}\subset\mathsf{G}_{S,\Omega'}$. Finally
for $g\in\mathsf{G}$ and $S'=\Int(g)(S)$, $\Omega'=g\cdot\Omega$,
one checks easily that 
\[
\Int(g)(\mathsf{G}_{\Omega})=\mathsf{G}_{\Omega'}\quad\mbox{and}\quad\Int(g)(\mathsf{G}_{S,\Omega})=\mathsf{G}_{S',\Omega'}.
\]

\subsection{Example\label{Example:F(G)tight}}

For the affine $\mathbf{F}(G)$-space $\mathbf{X}(G)=\mathbf{F}(G)$
and $u\in\mathsf{U}_{a}\setminus\{1\}$, 
\[
\mathbf{X}^{+}(S,u)=\left\{ \mathcal{F}\in\mathbf{F}(S):a(\mathcal{F})\geq0\right\} .
\]
For $\mathcal{F}\in\mathbf{F}(S)$, $\mathsf{G}_{\mathcal{F}}=\mathsf{P}_{\mathcal{F}}$
and $\mathsf{G}_{S,\mathcal{F}}$ is therefore the group spanned by
$\mathsf{N}_{G}(S)\cap\mathsf{P}_{\mathcal{F}}$ and the $\mathsf{U}_{a}$'s
for $a\in\Phi(G,S)$, $a(\mathcal{F})\geq0$. The $\mathsf{U}_{a}$'s
with $a(\mathcal{F})>0$ span $\mathsf{U}_{\mathcal{F}}$ by \cite[3.11]{BoTi65}.
Moreover, the group $\mathsf{N}_{G}(S)\cap\mathsf{P}_{\mathcal{F}}=\mathsf{N}_{L}(S)$
and the $\mathsf{U}_{a}$'s with $a(\mathcal{F})=0$ together span
$\mathsf{L}=L(K)$ where $L$ is the Levi subgroup of $P_{\mathcal{F}}$
which contains $Z_{G}(S)$, by the Bruhat decomposition of $\mathsf{L}$,
see~\cite[5.15]{BoTi65}. Therefore 
\[
\mathsf{G}_{S,\mathcal{F}}=\mathsf{P}_{\mathcal{F}}=\mathsf{G}_{\mathcal{F}}.
\]

\subsection{~}

We next consider the following axioms:
\begin{lyxlist}{MMM}
\item [{$ST$}] (\emph{Stabilizers}) For some (or every) $S\in\mathbf{S}(G)$,
\[
\forall x\in\mathbf{X}(S):\quad\mathsf{G}_{S,x}=\mathsf{G}_{x}.
\]

\item [{$ST^{-}$}] For some (or every) $S\in\mathbf{S}(G)$, 
\[
\forall x\in\mathbf{X}(S):\quad\mathsf{G}_{S,x}\subset\mathsf{G}_{x}.
\]

\item [{$ST_{1}^{-}$}] For some (or every) $S\in\mathbf{S}(G)$, 
\[
\forall\emptyset\neq\Omega\in\mathbf{X}(S):\quad\mathsf{G}_{S,\Omega}\subset\mathsf{G}_{\Omega}.
\]

\item [{$ST_{2}^{-}$}] For some (or every) $S\in\mathbf{S}(G)$ and any
$a\in\Phi(G,S)$, $u\in\mathsf{U}_{a}\setminus\{1\}$, 
\[
\exists x\in\mathbf{X}(S,u):\qquad ux=x.
\]

\item [{$UN^{+}$}] For $x\in\mathbf{X}(G)$, $\mathcal{F}\in\mathbf{F}(G)$
and $u\in\mathsf{U}_{\mathcal{F}}$, 
\[
\forall t\gg0:\qquad u(x+t\mathcal{F})=x+t\mathcal{F}.
\]
\end{lyxlist}
\begin{lem}
\label{lem:RelationBetweenAxiomsST}These axioms are related as follows:
\[
\begin{array}{cc}
 & ST\Longrightarrow ST^{-}\iff ST_{1}^{-}\iff ST_{2}^{-}\\
\mbox{and}\quad & ST^{-}+L(s)\,\Longrightarrow\, UN^{+}\,\Longrightarrow UN.
\end{array}
\]
Under $ST^{-}$, for every $S\in\mathbf{S}(G)$, $a\in\Phi(G,S)$
and $u\in\mathsf{U}_{a}\setminus\{1\}$, $\nu_{\mathbf{X}}(S,u)=0$
and 
\[
\mathbf{X}^{+}(S,u)=\left\{ x\in\mathbf{X}(S):ux=x\right\} =\left\{ x\in\mathbf{X}(S):ux\in\mathbf{X}(S)\right\} .
\]
\end{lem}
\begin{proof}
Plainly, $ST\Rightarrow ST^{-}$, $ST^{-}\Leftrightarrow ST_{1}^{-}$
and $UN^{+}\Rightarrow UN$. Fix $S\in\mathbf{S}(G)$, $a\in\Phi(G,S)$,
$u\in\mathsf{U}_{a}\setminus\{1\}$. Since $x\in\mathbf{X}(S,u)$
implies $u\in\mathsf{G}_{S,x}$, $ST^{-}\Rightarrow ST_{2}^{-}$.
On the other hand if $u$ fixes some $x\in\mathbf{X}(S)$, it also
fixes $x+\mathcal{F}$ for every $\mathcal{F}\in\mathbf{F}(S)$ with
$a(\mathcal{F})\geq0$ because $a(\mathcal{F})\geq0\iff U_{a}\subset P_{\mathcal{F}}$,
thus $ST_{2}^{-}\Rightarrow ST^{-}$. Under $ST_{2}^{-}$, 
\[
\mathbf{X}^{+}(S,u)\subset\left\{ x\in\mathbf{X}(S):ux=x\right\} \subset\left\{ x\in\mathbf{X}(S):ux\in\mathbf{X}(S)\right\} .
\]
Applying this to $u_{1},u_{2}\in\mathsf{U}_{-a}\setminus\{1\}$, we
obtain: $u_{1}\equiv u_{2}\equiv\mathrm{Id}$ on $\mathbf{X}(S)\setminus\mathbf{X}^{+}(S,u)$.
If $x$ and $ux$ belong to $\mathbf{X}(S)$ but $x$ does not belong
to $\mathbf{X}^{+}(S,u)$, then $ux$ also does not belong to $\mathbf{X}^{+}(S,u)$,
however $m(u)x=u_{1}uu_{2}x=u_{1}ux=ux$ does, a contradiction. This
proves the required displayed equality, and $\nu_{\mathbf{X}}(S,u)=0$
since $m(u)=u_{1}uu_{2}$ with $u,u_{1},u_{2}\equiv\mathrm{Id}$ on
$\mathbf{X}(S,u)=\mathbf{X}(S,u_{1})=\mathbf{X}(S,u_{2})$. Suppose
finally that $ST^{-}$ and $L(s)$ hold. For $x\in\mathbf{X}(G)$,
$\mathcal{F}\in\mathbf{F}(G)$ and any $u\in\mathsf{U}_{\mathcal{F}}$
with $u\neq1$, pick $S\in\mathbf{S}(G)$ with $x\in\mathbf{X}(S)$
and $\mathcal{F}\in\mathbf{F}(S)$ using $L(s)$. Since $\mathsf{U}_{\mathcal{F}}$
is spanned by the $\mathsf{U}_{a}$'s with $a\in\Phi(G,S)$, $a(\mathcal{F})>0$
(as in Example~\ref{Example:F(G)tight}), we may write $u=u_{1}\cdots u_{n}$
with $u_{i}\in\mathsf{U}_{a_{i}}\setminus\{1\}$ for some $a_{i}\in\Phi(G,S)$,
$a_{i}(\mathcal{F})>0$. For any sufficiently large $t\geq0$, $x+t\mathcal{F}\in\mathbf{X}(S)$
then belongs to $\mathbf{X}^{+}(S,u_{i})$ for all $i\in\{1,\cdots,n\}$,
thus $u_{i}(x+t\mathcal{F})=x+t\mathcal{F}$ by $ST^{-}$ and $u(x+t\mathcal{F})=x+t\mathcal{F}$,
which proves $UN^{+}$. 
\end{proof}

\subsection{~}

The next axiom is related to alcove-based retractions, see \cite[1.4]{Pa99}.
\begin{lyxlist}{MMM}
\item [{$HA$}] (\emph{Half-apartments}) For $S_{1},S_{2},S_{3}\in\mathbf{S}(G)$,
if $\mathbf{X}(S_{i})\cap\mathbf{X}(S_{j})$ contains an half-subspace
of $\mathbf{X}(S_{i})$ for every pair $(i,j)$ in $\{1,2,3\}^{2}$,
then 
\[
\mathbf{X}(S_{1})\cap\mathbf{X}(S_{2})\cap\mathbf{X}(S_{3})\neq\emptyset.
\]
\end{lyxlist}
\begin{lem}
\label{lem:3appandtightness}The axioms $UN^{+}$ and $HA$ imply
$ST^{-}$.\end{lem}
\begin{proof}
Fix $S\in\mathbf{S}(G)$, $a\in\Phi(G,S)$ and $u\in\mathsf{U}_{a}\setminus\{1\}$
and first note that 
\[
\left\{ x\in\mathbf{X}(S):ux\in\mathbf{X}(S)\right\} =\left\{ x\in\mathbf{X}(S):ux=x\right\} .
\]
Indeed if $x$ and $ux$ belong to $\mathbf{X}(S)$, pick $\mathcal{F}\in\mathbf{F}(S)$
with $a(\mathcal{F})>0$. Then $U_{a}\subset U_{\mathcal{F}}$, thus
$ux+t\mathcal{F}=u(x+t\mathcal{F})=x+t\mathcal{F}$ for $t\gg0$ by
$UN^{+}$ and $ux=x$ since $\mathbf{X}(S)$ is an affine $\mathbf{F}(S)$-space.
For every $t\in\mathbb{R}$, we define 
\begin{eqnarray*}
\mathbf{X}(S,u,t) & = & \mathbf{X}(S,u)+\{\mathcal{F}\in\mathbf{F}(S):a(\mathcal{F})=t\}\\
\mathbf{X}^{+}(S,u,t) & = & \mathbf{X}(S,u,t)+\left\{ \mathcal{F}\in\mathbf{F}(S):a(\mathcal{F})\geq0\right\} \\
\mathbf{X}^{-}(S,u,t) & = & \mathbf{X}(S,u,t)+\left\{ \mathcal{F}\in\mathbf{F}(S):a(\mathcal{F})\leq0\right\} 
\end{eqnarray*}
If $u$ fixes some $x\in\mathbf{X}(S,u,t)$, then also $u\equiv\mathrm{Id}$
on $\mathbf{X}^{+}(S,u,t)$ since 
\[
\forall\mathcal{F}\in\mathbf{F}(S):\qquad a(\mathcal{F})\geq0\iff U_{a}\subset P_{\mathcal{F}}.
\]
By \textbf{$UN^{+}$}, $u$ fixes some point in $\mathbf{X}(S)$,
thus $u\equiv\mathrm{Id}$ on $\mathbf{X}^{+}(S,u,t)$ for $t\gg0$.
Let us now write $u_{1}uu_{2}=m(u)$ with $u_{1},u_{2}\in\mathsf{U}_{-a}\setminus\{1\}$.
Then similarly for $i\in\{1,2\}$, 
\[
\left\{ x\in\mathbf{X}(S):u_{i}x\in\mathbf{X}(S)\right\} =\left\{ x\in\mathbf{X}(S):u_{i}x=x\right\} 
\]
and $u_{i}$ fixes $\mathbf{X}^{+}(S,u_{i},t)=\mathbf{X}^{-}(S,u,-t)$
for $t\gg0$. Choose $T>0$ such that 
\[
u\equiv\mathrm{Id}\mbox{ on }\mathbf{X}^{+}(S,u,T)\quad\mbox{and}\quad u_{1}\equiv u_{2}\equiv\mathrm{Id}\mbox{ on }\mathbf{X}^{-}(S,u,-T).
\]
Then: $\mathbf{X}(S)$ and $u\mathbf{X}(S)$ contain the half-subspace
$\mathbf{X}^{+}(S,u,T)$, $\mathbf{X}(S)$ and $u_{1}^{-1}\mathbf{X}(S)$
contain $\mathbf{X}^{-}(S,u,-T)$, while $u\mathbf{X}(S)$ and $u_{1}^{-1}\mathbf{X}(S)$
contain 
\[
u\mathbf{X}^{-}(S,u,-T)=uu_{2}\mathbf{X}^{-}(S,u,-T)=u_{1}^{-1}m(u)\mathbf{X}^{-}(S,u,-T)=u_{1}^{-1}\mathbf{X}^{+}(S,u,T).
\]
Thus by $HA$, there is a point $x\in\mathbf{X}(S)\cap u\mathbf{X}(S)\cap u_{1}^{-1}\mathbf{X}(S)$.
Any such point is fixed by $u^{-1}$ and $u_{1}$, thus also by $m(u)u_{2}^{-1}=u_{1}u$.
In particular $u_{2}^{-1}(x)=m(u)^{-1}(x)$ also belongs to $\mathbf{X}(S)$,
so that again $x$ is fixed by $u_{2}$, as well as $m(u)=u_{1}uu_{2}$.
But then $x$ belongs to $\mathbf{X}(S,u)=\{x\in\mathbf{X}(S):m(u)(x)=x\}$
and it is fixed by $u$, which proves $ST_{2}^{-}$, from which $ST$
follows by the previous lemma.
\end{proof}

\subsection{~}

An affine $\mathbf{F}(G)$-building is \emph{tight} if it satisfies
$ST$. It then also satisfies the conclusion of lemma~\ref{lem:RelationBetweenAxiomsST},
and it is determined by its type. More precisely:
\begin{lem}
\label{lem:UnicityTight}Suppose that $\mathbf{X}(G)$ is a tight
affine $\mathbf{F}(G)$-building and $\mathbf{Y}(G)$ is an affine
$\mathbf{F}(G)$-building which satisfies $ST^{-}$. Then $\nu_{\mathbf{X}}=\nu_{\mathbf{Y}}\iff\mathbf{X}(G)\simeq\mathbf{Y}(G)$.\end{lem}
\begin{proof}
We have to show that $\nu_{\mathbf{X}}=\nu_{\mathbf{Y}}$ implies
$\mathbf{X}^{e}(G)\simeq\mathbf{Y}^{e}(G)$. Suppose therefore that
$\nu_{\mathbf{X}}=\nu_{\mathbf{Y}}$. Pick $S\in\mathbf{S}(G)$. By
\cite[2.1.9]{Ro77}, there is a finite subgroup of $\mathsf{N}_{G}(S)$
which maps surjectively onto $\mathsf{W}_{G}(S)$, and which thus
has unique fixed points $x_{S}$ in $\mathbf{X}^{r}(S)$ and $y_{S}$
in $\mathbf{Y}^{r}(S)$. Let $\theta_{S}:\mathbf{X}^{e}(G)\rightarrow\mathbf{Y}^{e}(G)$
be the unique isomorphism of affine $\mathbf{F}(S)$-spaces mapping
$(x_{S},0)$ to $(y_{S},0)$. Then $\theta_{S}$ is $\mathsf{N}_{G}(S)$-equivariant,
and it is the unique $\mathsf{N}_{G}(S)$-equivariant isomorphism
of affine $\mathbf{F}(S)$-spaces from $\mathbf{X}^{e}(S)$ to $\mathbf{Y}^{e}(S)$
mapping $\mathbf{X}^{r}(S)$ to $\mathbf{Y}^{r}(S)$. If $\Int(g)(S)=S'$,
then $g\circ\theta_{S}=\theta_{S'}\circ g$. For $x\in\mathbf{X}^{e}(S)\cap\mathbf{X}^{e}(S')$,
there is such a $g$ in $\mathsf{G}_{x}$ by $R(i)$ for $\mathbf{X}(G)$.
Thus $g$ belongs to $\mathsf{G}_{S,x}$ by $ST$ for $\mathbf{X}(G)$,
which equals $\mathsf{G}_{S,\theta_{S}(x)}$ by definition. Then $g\in\mathsf{G}_{\theta_{S}(x)}$
by our assumption on $\mathbf{Y}(G)$, thus $\theta_{S'}(x)=\theta_{S'}(gx)=g\theta_{S}(x)=\theta_{S}(x)$.
Our isomorphisms $\theta_{S}$ therefore glue to $\theta:\mathbf{X}^{e}(G)\rightarrow\mathbf{Y}^{e}(G)$,
which is the desired isomorphism. \end{proof}
\begin{rem}
\label{Rk:ReconstBuildFromAppt}A tight affine $\mathbf{F}(G)$-building
$\mathbf{X}(G)$ can be retrieved from any apartment $\mathbf{X}(S)$
together with its $\mathsf{N}_{G}(S)$-action. It is the quotient
of $\mathsf{G}\times\mathbf{X}(S)$ for the equivalence relation $\sim$
induced by $(g,x)\mapsto gx$, which indeed only depends upon the
apartment: $(g,x)\sim(g',x')$ if and only if $g'=gkn$ and $x'=n^{-1}x$
for some $k\in\mathsf{G}_{S,x}$ and $n\in\mathsf{N}_{G}(S)$.
\end{rem}

\section{Metric properties}

Let $\mathbf{X}(G)$ be an affine $\mathbf{F}(G)$-building. We shall
here relate our mostly algebraic formalism to various notions pertaining
to the non-canonical metric $d=d_{\tau}$: rays, tangent spaces and
Busemann functions. For simplicity, we furthermore assume that $(\mathbf{X}(G),d)$
is a $CAT(0)$-space.

\subsection{~\label{sub:Rays}}

There is a $\mathsf{G}$-equivariant commutative diagram
\[
\xyR{1pc}\xymatrix{\mathbf{X}(G)\times\mathbf{F}(G)\ar@{^{(}->}[r]\sp(0.56){\alpha}\ar@(u,u)[rr]^{\mathrm{Id}\times\iota}\ar@(d,l)[dr]^{+} & \mathbf{R}\mathbf{X}(G)\ar@{^{(}->}[r]\sp(0.36){\beta}\ar[d]^{\mathrm{ev}_{1}} & \mathbf{X}(G)\times\mathcal{C}(\partial\mathbf{X}(G))\\
 & \mathbf{X}(G)
}
\]
where the various new sets and maps are defined as follows:\nomenclature[RX(G)]{$\mathbf{RX}(G)$}{Rays in $\mathbf{X}(G)$, page \nomrefpage}\nomenclature[dX(G)]{$\partial \mathbf{X}(G)$}{Visual boundary of $\mathbf{X}(G)$, page \nomrefpage}\nomenclature[C(dX(G))]{$\mathcal{C}(\partial \mathbf{X}(G))$}{Cone on the visual boundary of $\mathbf{X}(G)$, page \nomrefpage}
\begin{itemize}
\item $\mathbf{R}\mathbf{X}(G)$ is the set of all functions $f:\mathbb{R}^{+}\rightarrow\mathbf{X}(G)$
such that 
\[
\exists c_{f}\geq0\mbox{ s.t. }\forall t,u\in\mathbb{R}^{+}:\qquad d(f(t),f(u))=c_{f}\left|u-t\right|.
\]

\item $\partial\mathbf{X}(G)$ is the visual boundary $\left\{ f\in\mathbf{R}\mathbf{X}(G):c_{f}=1\right\} /\sim$
of $\mathbf{X}(G)$, where the equivalence relation $\sim$ is defined
on the whole of $\mathbf{R}\mathbf{X}(G)$ by 
\[
f\sim g\iff t\mapsto d\left(f(t),g(t)\right)\mbox{ is bounded}.
\]

\item $\mathcal{C}(\partial\mathbf{X}(G))$ is the cone $\left(\mathbb{R}^{+}\times\partial\mathbf{X}(G)\right)/\approx$
where the equivalence relation $\approx$ just collapses $\{0\}\times\partial\mathbf{X}(G)$
to a single point $0\in\mathcal{C}(\partial\mathbf{X}(G))$, so that
\[
f\mapsto[f]=\begin{cases}
(c_{f},\mbox{class of }f(c_{f}^{-1}-)) & \mbox{if }c_{f}\neq0,\\
0 & \mbox{if }c_{f}=0,
\end{cases}
\]
identifies the quotient $\mathbf{R}\mathbf{X}(G)/\sim$ with the cone
$\mathcal{C}(\partial\mathbf{X}(G))$.
\item $\alpha(x,\mathcal{F})(t)=x+t\mathcal{F}$, $\beta(f)=(f(0),[f])$
and $\mathrm{ev}_{1}(f)=f(1)$.
\end{itemize}
By the axiom $NE$ for $\mathbf{X}(G)$, $\beta\circ\alpha=\mathrm{Id}_{\mathbf{X}(G)}\times\iota$
for some $\mathsf{G}$-equivariant map 
\[
\iota:\mathbf{F}(G)\hookrightarrow\mathcal{C}(\partial\mathbf{X}(G)).
\]
The latter is injective: suppose that $x+t\mathcal{F}\sim y+t\mathcal{G}$
and pick $z\in\mathbf{X}(S)$ for some $S\in\mathbf{S}(G)$ such that
$\mathcal{F},\mathcal{G}\in\mathbf{F}(S)$. Then $z+t\mathcal{F}\sim x+t\mathcal{F}\sim y+t\mathcal{G}\sim z+t\mathcal{G}$,
thus $\mathcal{F}=\mathcal{G}$ since $z+t\mathcal{F}\sim z+t\mathcal{G}$
in the affine $\mathbf{F}(S)$-space $\mathbf{X}(S)$. It follows
that $\alpha$ is also injective. Finally $\beta$ is injective by
convexity of the CAT(0)-distance $d$, and it is also surjective when
$(\mathbf{X}(G),d)$ is complete \cite[II.8.2]{BrHa99} (for instance
in the discrete case, by lemma~\ref{lem:DiscImpliesComplete}). By
the axiom $L(s)$, $\alpha$ is bijective precisely when every geodesic
ray in $\mathbf{X}(G)$ is standard, in which case $\iota$ and $\beta$
are also bijective.
\begin{rem}
The injectivity of $\iota$ implies that the apartment map $S\mapsto\mathbf{X}(S)$
is injective: $\mathbf{X}(S)$ determines $\mathcal{C}(\partial\mathbf{X}(S))=\iota(\mathbf{F}(S))$,
thus also $\mathbf{F}(S)$ and $S\in\mathbf{S}(G)$. 
\end{rem}

\subsection{~\label{sub:settingforanglesinbuild}}

Fix $x\in\mathbf{X}(G)$ and $0\neq\mathcal{F},\mathcal{G}\in\mathbf{F}(G)$,
set $y=x+\mathcal{F}$, $z=x+\mathcal{G}$. We may then define the
following five different types of angles
\[
0\leq\measuredangle_{x}(\mathcal{F},\mathcal{G})\leq\measuredangle_{x}(\overrightarrow{xy},\mathcal{G})\leq\measuredangle_{x}^{c}(y,z)\leq\measuredangle(\overrightarrow{xy},\mathcal{G})\leq\measuredangle^{x}(\mathcal{F},\mathcal{G})\leq\pi.
\]
First, $\measuredangle_{x}^{c}(y,z)$\nomenclature[<(y,z)_x^c]{$\measuredangle_{x}^{c}(y,z)$}{Angle at $x$ in a comparison Euclidean triangle for $(x,y,z)$.}
is the angle at $x$ in a comparison triangle for $(x,y,z)$, so that
\[
d(y,z)=\left(d(x,y)^{2}+d(x,z)^{2}-2d(x,y)d(x,z)\cos\measuredangle_{x}^{c}(y,z)\right)^{1/2}.
\]
More generally for every $(t,u)\in\mathbb{R}_{+}$, the distance $d(x+t\mathcal{F},x+u\mathcal{G})$
equals 
\[
\left(t^{2}\left\Vert \mathcal{F}\right\Vert ^{2}+u^{2}\left\Vert \mathcal{G}\right\Vert ^{2}-2tu\left\Vert \mathcal{F}\right\Vert \left\Vert \mathcal{G}\right\Vert \cos\measuredangle_{x}^{c}(x+t\mathcal{F},x+u\mathcal{G})\right)^{1/2}.
\]
By~\cite[II.3.1]{BrHa99}, the comparison angle function 
\[
(t,u)\in\mathbb{R}_{>}^{2}\mapsto\measuredangle_{x}^{c}(x+t\mathcal{F},x+u\mathcal{G})\in[0,\pi]
\]
is non-decreasing in both variables. We define\nomenclature[<(F,G)_x]{$\measuredangle_{x}(\mathcal{F},\mathcal{G})$}{Alexandrov angle at $x$ between $x+t \mathcal{F}$ and $x+t \mathcal{G}$, page \nomrefpage}\nomenclature[<(y,z,G)]{$\measuredangle(\overrightarrow{xy},\mathcal{G})$}{Busemann angle at $x$ between $y$ and $x+t \mathcal{G}$, page \nomrefpage}{\small{
\[
\begin{array}{lclcl}
\measuredangle_{x}(\mathcal{F},\mathcal{G}) & = & \inf\left\{ \measuredangle_{x}^{c}(x+t\mathcal{F},x+u\mathcal{G}):t,u>0\right\}  & = & \lim_{t,u\rightarrow0}\measuredangle_{x}^{c}(x+t\mathcal{F},x+u\mathcal{G})\\
\measuredangle_{x}(\overrightarrow{xy},\mathcal{G}) & = & \inf\left\{ \measuredangle_{x}^{c}(y,x+u\mathcal{G}):u>0\right\}  & = & \lim_{u\rightarrow0}\measuredangle_{x}^{c}(y,x+u\mathcal{G})\\
\measuredangle(\overrightarrow{xy},\mathcal{G}) & = & \sup\left\{ \measuredangle_{x}^{c}(y,x+u\mathcal{G}):u>0\right\}  & = & \lim_{u\rightarrow\infty}\measuredangle_{x}^{c}(y,x+u\mathcal{G})\\
\measuredangle^{x}(\mathcal{F},\mathcal{G}) & = & \sup\left\{ \measuredangle_{x}^{c}(x+t\mathcal{F},x+u\mathcal{G}):t,u>0\right\}  & = & \lim_{t,u\rightarrow\infty}\measuredangle_{x}^{c}(x+t\mathcal{F},x+u\mathcal{G})
\end{array}
\]
}}We will also use the notations $\measuredangle_{x}(y,z)=\measuredangle_{x}(\mathcal{F},\mathcal{G})=\measuredangle(\mathcal{F}_{x},\mathcal{G}_{x}).$

\subsection{~}

Let us immediately observe that:
\begin{lem}
\label{lem:AnglesAreEqualWhenGcentral}If $\mathcal{G}$ belongs to
$\mathbf{G}(Z)\subset\mathbf{F}(G)$, then 
\[
\measuredangle_{x}(\mathcal{F},\mathcal{G})=\measuredangle_{x}(\overrightarrow{xy},\mathcal{G})=\measuredangle_{x}^{c}(y,z)=\measuredangle(\overrightarrow{xy},\mathcal{G})=\measuredangle^{x}(\mathcal{F},\mathcal{G}).
\]
\end{lem}
\begin{proof}
Pick $S\in\mathbf{S}(G)$ with $x\in\mathbf{X}(S)$, $\mathcal{F}\in\mathbf{F}(S)$
using the axiom $L(s)$ for $\mathbf{X}(G)$. Then also $\mathcal{G}\in\mathbf{F}(S)$,
thus everything stays in the flat Euclidean affine $\mathbf{F}(S)$-space
$\mathbf{X}(S)$ on which all of our angles plainly agree.
\end{proof}

\subsection{~\label{sub:TangentSpaceInBuild}}

The smallest of these angles, also denoted by $\measuredangle(\mathcal{F}_{x},\mathcal{G}_{x})$,
is the Alexandrov angle at $x$ between the rays $x+t\mathcal{F}$
and $x+t\mathcal{G}$ \cite[I.12]{BrHa99}. It satisfies a triangle
inequality: if $\mathcal{H}\in\mathbf{F}(G)$ is yet another nonzero
filtration, then by \cite[I.14]{BrHa99}, 
\[
\measuredangle_{x}(\mathcal{F},\mathcal{H})\leq\measuredangle_{x}(\mathcal{F},\mathcal{G})+\measuredangle_{x}(\mathcal{G},\mathcal{H}).
\]
The tangent cone at $x$ is the quotient $\mathbf{T}_{x}\mathbf{X}(G)=\mathbf{F}(G)/\sim_{x}$\nomenclature[T_xX(G)]{$\mathbf{T}_{x}\mathbf{X}(G)$}{Tangent space at $x$ in $\mathbf{X}(G)$, defined page \nomrefpage},
where $\mathcal{F}\sim_{x}\mathcal{G}$ if and only if $\left\Vert \mathcal{F}\right\Vert =\left\Vert \mathcal{G}\right\Vert $
and $\measuredangle_{x}(\mathcal{F},\mathcal{G})=0$. This definition
agrees with \cite[II.3.18]{BrHa99} by the axiom $R(s)$ for $\mathbf{X}(G)$.
We denote by $\loc_{x}(\mathcal{F})=\mathcal{F}_{x}$\nomenclature[loc_x]{$\loc _x$}{Projection $\mathbf{F}(G) \twoheadrightarrow \mathbf{T}_x \mathbf{X}(G)$, page \nomrefpage}
the class of $\mathcal{F}$ in $\mathbf{T}_{x}\mathbf{X}(G)$. The
norm $\left\Vert -\right\Vert $ and Alexandrov angle $\measuredangle_{x}(-,-)$
on $\mathbf{F}(G)$ descend to a norm and angle on $\mathbf{T}_{x}\mathbf{X}(G)$,
thereby justifying our notation $\measuredangle(\mathcal{F}_{x},\mathcal{G}_{x})=\measuredangle_{x}(\mathcal{F},\mathcal{G})$.
We also define a scalar product and a distance function on $\mathbf{T}_{x}\mathbf{X}(G)$
by the usual formulas:
\begin{eqnarray*}
\left\langle \mathcal{F}_{x},\mathcal{G}_{x}\right\rangle  & = & \left\Vert \mathcal{F}_{x}\right\Vert \left\Vert \mathcal{G}_{x}\right\Vert \cos\measuredangle(\mathcal{F}_{x},\mathcal{G}_{x})\\
d(\mathcal{F}_{x},\mathcal{G}_{x}) & = & \sqrt{\left\Vert \mathcal{F}_{x}\right\Vert ^{2}+\left\Vert \mathcal{G}_{x}\right\Vert ^{2}-2\left\langle \mathcal{F}_{x},\mathcal{G}_{x}\right\rangle }.
\end{eqnarray*}
By definition of the Alexandrov angle, 
\begin{eqnarray*}
d(\mathcal{F}_{x},\mathcal{G}_{x}) & = & \lim_{t\rightarrow0}{\textstyle \frac{1}{t}}d(x+t\mathcal{F},x+t\mathcal{G}).
\end{eqnarray*}
These formulas for $d$ respectively show that $d(\mathcal{F}_{x},\mathcal{G}_{x})=0$
if and only if $\mathcal{F}_{x}=\mathcal{G}_{x}$, and that $d(\mathcal{F}_{x},\mathcal{H}_{x})\leq d(\mathcal{F}_{x},\mathcal{G}_{x})+d(\mathcal{G}_{x},\mathcal{H}_{x})$.
Thus $d$ is indeed a distance on $\mathbf{T}_{x}\mathbf{X}(G)$ and
$\mathcal{F}\sim_{x}\mathcal{G}$ if and only if $\lim_{t\rightarrow0}\frac{1}{t}d(x+t\mathcal{F},x+t\mathcal{G})=0$.

\subsection{~}

By the very definition of $\mathbf{T}_{x}\mathbf{X}(G)$, there is
a commutative diagram 
\[
\xyR{2pc}\xymatrix{\mathbf{F}(G)\ar@{->>}[rr]^{\mathcal{F}\mapsto x+\mathcal{F}}\ar@{->>}[dr]_{\loc_{x}} &  & \mathbf{X}(G)\ar@{->>}[dl]^{\loc_{x}^{a}}\\
 & \mathbf{T}_{x}\mathbf{X}(G)
}
\]
We may thus also define\nomenclature[loc_x^a]{$\loc_x^a$}{Localization $\mathbf{X}(G) \twoheadrightarrow \mathbf{T}_x \mathbf{X}(G)$, page \nomrefpage}
\begin{eqnarray*}
\measuredangle_{x}(y,z) & = & \measuredangle\left(\loc_{x}^{a}(y),\loc_{x}^{a}(z)\right),\\
\left\langle y,z\right\rangle _{x} & = & \left\langle \loc_{x}^{a}(y),\loc_{x}^{a}(z)\right\rangle ,\\
d_{x}(y,z) & = & d\left(\loc_{x}^{a}(y),\loc_{x}^{a}(z)\right).
\end{eqnarray*}

\subsection{~}

Our second smallest angle actually equals the first one by \cite[I.1.16]{BrHa99}.
Thus for $y=x+\mathcal{F}$, we have $\measuredangle_{x}(\overrightarrow{xy},\mathcal{G})=\measuredangle(\mathcal{F}_{x},\mathcal{G}_{x})$.
Since
\[
\lim_{t\rightarrow0}{\textstyle \frac{1}{t}}\left(d(y,x)-d(y,x+t\mathcal{G})\right)=\left\Vert \mathcal{F}\right\Vert \left\Vert \mathcal{G}\right\Vert \cos\measuredangle_{x}(\overrightarrow{xy},\mathcal{G})
\]
by definition of $\measuredangle_{x}(\overrightarrow{xy},\mathcal{G})$,
it follows that 
\[
\lim_{t\rightarrow0}{\textstyle \frac{1}{t}}\left(d(y,x)-d(y,x+t\mathcal{G})\right)=\left\langle \loc_{x}^{a}(y),\loc_{x}\mathcal{G}\right\rangle .
\]

\subsection{~}

By definition of our largest angle $\measuredangle^{x}(\mathcal{F},\mathcal{G})$,
we have 
\[
\lim_{t\rightarrow\infty}{\textstyle \frac{1}{t}}d(x+t\mathcal{F},x+t\mathcal{G})=\sqrt{\left\Vert \mathcal{F}\right\Vert ^{2}+\left\Vert \mathcal{G}\right\Vert ^{2}-2\left\Vert \mathcal{F}\right\Vert \left\Vert \mathcal{G}\right\Vert \cos\measuredangle^{x}(\mathcal{F},\mathcal{G})}.
\]
For $z_{1},z_{2}\in\mathbf{X}(G)$, $x+t\mathcal{F}\sim z_{1}+t\mathcal{F}$
and $x+t\mathcal{G}\sim z_{2}+t\mathcal{G}$, thus also 
\[
\lim_{t\rightarrow\infty}{\textstyle \frac{1}{t}}d(z_{1}+t\mathcal{F},z_{2}+t\mathcal{G})=\sqrt{\left\Vert \mathcal{F}\right\Vert ^{2}+\left\Vert \mathcal{G}\right\Vert ^{2}-2\left\Vert \mathcal{F}\right\Vert \left\Vert \mathcal{G}\right\Vert \cos\measuredangle^{x}(\mathcal{F},\mathcal{G})}.
\]
In particular, $\measuredangle^{x}(\mathcal{F},\mathcal{G})$ is independent
of $x$. Taking $x\in\mathbf{X}(S)$ for some $S\in\mathbf{S}(G)$
with $\mathcal{F},\mathcal{G}\in\mathbf{F}(S)$, we find that $\measuredangle^{x}(\mathcal{F},\mathcal{G})=\measuredangle(\mathcal{F},\mathcal{G})=\measuredangle_{x}(\mathcal{F},\mathcal{G})$.
Thus 
\[
d(\mathcal{F},\mathcal{G})=\lim_{t\rightarrow\infty}{\textstyle \frac{1}{t}}d(z_{1}+t\mathcal{F},z_{2}+t\mathcal{G})
\]
for every $z_{1},z_{2}\in\mathbf{X}(G)$ and 
\begin{eqnarray*}
\measuredangle(\mathcal{F},\mathcal{G}) & = & \max\{\measuredangle(\mathcal{F}_{x},\mathcal{G}_{x}):x\in\mathbf{X}(G)\},\\
\left\langle \mathcal{F},\mathcal{G}\right\rangle  & = & \min\{\left\langle \mathcal{F}_{x},\mathcal{G}_{x}\right\rangle :x\in\mathbf{X}(G)\},\\
d(\mathcal{F},\mathcal{G}) & = & \max\{d(\mathcal{F}_{x},\mathcal{G}_{x}):x\in\mathbf{X}(G)\}.
\end{eqnarray*}

\subsection{~\label{sub:BusemannFunctions}}

Recall from~\cite[II.8.18-20]{BrHa99} that for any $y,z\in\mathbf{X}(G)$,
the function 
\[
t\mapsto d(y,z+t\mathcal{G})-t\left\Vert \mathcal{G}\right\Vert 
\]
is non-increasing and bounded, the functions 
\[
y\mapsto d(y,z+t\mathcal{G})-t\left\Vert \mathcal{G}\right\Vert 
\]
converge uniformly on bounded subsets of $\mathbf{X}(G)$ as $t\rightarrow\infty$
to
\[
y\mapsto b_{z,\mathcal{G}}(y)=\lim_{t\rightarrow\infty}\left(d(y,z+t\mathcal{G})-t\left\Vert \mathcal{G}\right\Vert \right),
\]
and (for $\mathcal{G}\neq0$) the Busemann function in two variables
\[
(x,y)\mapsto b_{\mathcal{G}}(x,y)=b_{z,\mathcal{G}}(y)-b_{z,\mathcal{G}}(x)
\]
does not depend upon $z$. Note that the proof of this last statement
in \emph{loc.~cit.}, which only uses the ``if'' part of \cite[II.8.19]{BrHa99},
does indeed not require the ambient CAT(0)-space to be complete. For
any $\mathcal{G}\in\mathbf{F}(G)$ and $x,y\in\mathbf{X}(G)$, we
set\nomenclature[<xy,G>]{$\left\langle \overrightarrow{xy},\mathcal{G}\right\rangle $}{Busemann scalar product between $[x,y]$ and $x+t\mathcal{G}$, defined page \nomrefpage }
\[
\left\langle \overrightarrow{xy},\mathcal{G}\right\rangle =\left\Vert \mathcal{G}\right\Vert \cdot\lim_{t\rightarrow\infty}\left(d(x,z+t\mathcal{G})-d(y,z+t\mathcal{G})\right)
\]
which is thus well-defined, independent of $z$, and equal to $\left\Vert \mathcal{G}\right\Vert \cdot b_{\mathcal{G}}(y,x)$
if $\mathcal{G}\neq0$. For $x,y,z\in\mathbf{X}(G)$, we have $\left\langle \overrightarrow{xz},\mathcal{G}\right\rangle =\left\langle \overrightarrow{xy},\mathcal{G}\right\rangle +\left\langle \overrightarrow{yz},\mathcal{G}\right\rangle $,
$\left\langle \overrightarrow{yx},\mathcal{G}\right\rangle +\left\langle \overrightarrow{xy},\mathcal{G}\right\rangle =0$.
Taking $z=x$ in the formula defining $\left\langle \overrightarrow{xy},\mathcal{G}\right\rangle $,
we find that 
\[
\left\langle \overrightarrow{xy},\mathcal{G}\right\rangle =d(x,y)\cdot\left\Vert \mathcal{G}\right\Vert \cdot\cos\measuredangle(\overrightarrow{xy},\mathcal{G})
\]
by definition of our second largest angle $\measuredangle(\overrightarrow{xy},\mathcal{G})$.
The function
\[
y\mapsto\left\langle \overrightarrow{xy},\mathcal{G}\right\rangle =-\left\langle \overrightarrow{yx},\mathcal{G}\right\rangle 
\]
is $\left\Vert \mathcal{G}\right\Vert $-Lipschitzian and concave
(by convexity of $d$).

\subsection{~}

Returning to $y=x+\mathcal{F}$ and $z=x+\mathcal{G}$, we obtain
\[
\left\langle \mathcal{F},\mathcal{G}\right\rangle \leq\left\langle \overrightarrow{xy},\mathcal{G}\right\rangle \leq{\textstyle \frac{1}{2}}\left(d(x,y)^{2}+d(x,z)^{2}-d(y,z)^{2}\right)\leq\left\langle \mathcal{F}_{x},\mathcal{G}_{x}\right\rangle 
\]
with absolute values bounded by $\left\Vert \mathcal{F}\right\Vert \left\Vert \mathcal{G}\right\Vert $,
as well as 
\[
d(\mathcal{F}_{x},\mathcal{G}_{x})\leq d(y,z)\leq\left(d(x,y)^{2}+\left\Vert \mathcal{G}\right\Vert ^{2}-2\left\langle \overrightarrow{xy},\mathcal{G}\right\rangle \right)^{1/2}\leq d(\mathcal{F},\mathcal{G}).
\]
In particular, the localization functions
\[
\begin{array}{ccccc}
\left(\mathbf{F}(G),d\right) & \rightarrow & \left(\mathbf{X}(G),d\right) & \rightarrow & \left(\mathbf{T}_{x}\mathbf{X}(G),d\right)\\
\mathcal{F} & \mapsto & x+\mathcal{F}=y & \mapsto & y_{x}=\mathcal{F}_{x}
\end{array}
\]
are non-expanding. For $S\in\mathbf{S}(G)$ with $x\in\mathbf{X}(S)$,
they restrict to isometries
\[
\begin{array}{ccccc}
\left(\mathbf{F}(S),d\right) & \simeq & \left(\mathbf{X}(S),d\right) & \simeq & \left(\mathbf{T}_{x}\mathbf{X}(S),d\right)\end{array}
\]
where $\mathbf{T}_{x}\mathbf{X}(S)=\loc_{x}^{a}\mathbf{X}(S)=\loc_{x}\mathbf{F}(S)$.
We refer to $\mathbf{T}_{x}\mathbf{X}(S)$ as the apartment of $S$
in $\mathbf{T}_{x}\mathbf{X}(G)$. It is a (complete thus) closed
subset of $\mathbf{T}_{x}\mathbf{X}(G)$.

\subsection{~\label{sub:GoodPropOfTangSpace}}

Suppose that any two germs of geodesic segments in $\mathbf{X}(G)$
issuing from the same point are contained in some apartment of $\mathbf{X}(G)$.
This is for instance the case when the strengthening $L(s)^{+}$ of
$L(s)$ holds for $\mathbf{X}(G)$. Then: 
\begin{enumerate}
\item The axiom $R(s)$ holds for $\mathbf{T}_{x}\mathbf{X}(G)$, i.e.~any
two elements of $\mathbf{T}_{x}\mathbf{X}(G)$ belong to $\mathbf{T}_{x}\mathbf{X}(S)$
for some $S$ in $\mathbf{S}(x)=\left\{ S\in\mathbf{S}(G):x\in\mathbf{X}(S)\right\} $;
\item $\mathcal{F}_{x}=\mathcal{G}_{x}$ if and only if $x+t\mathcal{F}=x+t\mathcal{G}$
for all sufficiently small $t\geq0$; 
\item $(\mathbf{T}_{x}\mathbf{X}(G),d)$ is a CAT(0)-space; and 
\item $v_{2}\mapsto\left\langle v_{1},v_{2}\right\rangle $ is homogeneous,
concave and $\left\Vert v_{1}\right\Vert $-Lipschitzian on $\mathbf{T}_{x}\mathbf{X}(G)$. 
\end{enumerate}
The first two properties are easy. To establish $(3)$, first note
that $(\mathbf{T}_{x}\mathbf{X}(G),d)$ is a geodesic space by $(1)$,
so it remains to establish the $CAT(0)$-inequality. Let thus $v,v_{0},v_{1}$
be three points of $\mathbf{T}_{x}\mathbf{X}(G)$ and choose $S\in\mathbf{S}(x)$
such that $v_{0},v_{1}$ belong to $\mathbf{T}_{x}\mathbf{X}(S)$.
Lift $v_{i}$ to $\mathcal{F}_{i}\in\mathbf{F}(S)$ and lift $v$
to some $\mathcal{F}\in\mathbf{F}(G)$. For $u\in[0,1]$, let $\mathcal{F}_{u}=(1-u)\mathcal{F}_{0}+u\mathcal{F}_{1}$
be the point at distance $ud(\mathcal{F}_{0},\mathcal{F}_{1})$ from
$\mathcal{F}_{0}$ on the segment $[\mathcal{F}_{0},\mathcal{F}_{1}]$
of $\mathbf{F}(S)$. Then $x_{u,t}=x+t\mathcal{F}_{u}$ is the point
at distance $ud(x_{0,t},x_{1,t})$ from $x_{0,t}$ on the segment
$[x_{0,t},x_{1,t}]$ of $\mathbf{X}(S)$ while $v_{u}=\loc_{x}\mathcal{F}_{u}$
is the point at distance $ud(v_{0},v_{1})$ from $v_{0}$ on the segment
$[v_{0},v_{1}]$ of $\mathbf{T}_{x}\mathbf{X}(S)$. Set $x_{t}=x+t\mathcal{F}$.
By the $CAT(0)$-inequality in $\mathbf{X}(G)$ applied to the triangle
$(x_{t},x_{0,t},x_{1,t})$, 
\[
d(x_{t},x_{u,t})^{2}\leq(1-u)\cdot d(x_{t},x_{0,t})^{2}+u\cdot d(x_{t},x_{1,t})^{2}-u(1-u)\cdot d(x_{0,t},x_{1,t})^{2}.
\]
Dividing by $t^{2}$ and taking the limit as $t\rightarrow0$ gives
\[
d(v,v_{u})^{2}\leq(1-u)\cdot d(v,v_{0})^{2}+u\cdot d(v,v_{1})^{2}-u(1-u)\cdot d(v_{0},v_{1})^{2}
\]
which is the $CAT(0)$-inequality for $\mathbf{T}_{x}\mathbf{X}(G)$.
Given $(3)$, the proof of $(4)$ is entirely similar to that of corollary~\ref{cor:ConcavOfScalProd}.
\begin{rem}
By $(2)$, the quotient $\mathbf{T}_{x}\mathbf{X}(G)$ of $\mathbf{F}(G)$
does not depend upon the chosen metric (i.e.~chosen $\tau$) for
buildings satisfying the above condition on germs. Assuming instead
that $(\mathbf{X}(G),d)$ is complete, \cite[II.3.19]{BrHa99} shows
that the completion of $(\mathbf{T}_{x}\mathbf{X}(G),d)$ is always
a CAT(0)-space. 
\end{rem}

\subsection{~}

If the axiom $L(s)^{+}$ holds for $\mathbf{X}(G)$, then $\mathcal{G}\mapsto\left\langle \overrightarrow{xy},\mathcal{G}\right\rangle $
is $d(x,y)$-Lipschitzian. Indeed, for $\mathcal{G}_{1},\mathcal{G}_{2}\in\mathbf{F}(G)$,
there is a subdivision $x=x_{0},\cdots,x_{n}=y$ of the segment $[x,y]$
of $\mathbf{X}(G)$ and for each $i\in\{1,\cdots,n\}$, tori $S_{i,1},S_{i,2}\in\mathbf{S}(G)$
such that $[x_{i-1},x_{i}]\subset\mathbf{X}(S_{i,j})$ and $\mathcal{G}_{j}\in\mathbf{F}(S_{i,j})$
for $j\in\{1,2\}$ by lemma~\ref{Lem:Ls+impliesSubdiv}. Then 
\[
\left\langle \overrightarrow{xy},\mathcal{G}_{1}\right\rangle -\left\langle \overrightarrow{xy},\mathcal{G}_{2}\right\rangle ={\textstyle \sum}_{i=0}^{n-1}\left\langle \overrightarrow{x_{i}x_{i+1}},\mathcal{G}_{1}\right\rangle -\left\langle \overrightarrow{x_{i}x_{i+1}},\mathcal{G}_{2}\right\rangle 
\]
with $\left\langle \overrightarrow{x_{i}x_{i+1}},\mathcal{G}_{j}\right\rangle =\left\langle \loc_{x_{i}}^{a}(x_{i+1}),\loc_{x_{i}}(\mathcal{G}_{j})\right\rangle $.
Thus 
\[
\left|\left\langle \overrightarrow{xy},\mathcal{G}_{1}\right\rangle -\left\langle \overrightarrow{xy},\mathcal{G}_{2}\right\rangle \right|\leq{\textstyle \sum}_{i=0}^{n-1}\left\Vert \loc_{x_{i}}^{a}(x_{i+1})\right\Vert \cdot d\left(\loc_{x_{i}}(\mathcal{G}_{1}),\loc_{x_{i}}(\mathcal{G}_{2})\right)
\]
because $v_{2}\mapsto\left\langle v_{1},v_{2}\right\rangle $ is $\left\Vert v_{1}\right\Vert $-Lipschitzian
on $\mathbf{T}_{x_{i}}\mathbf{X}(G)$. Since 
\[
\left\Vert \loc_{x_{i}}^{a}(x_{i+1})\right\Vert =d(x_{i},x_{i+1})\quad\mbox{and}\quad d\left(\loc_{x_{i}}(\mathcal{G}_{1}),\loc_{x_{i}}(\mathcal{G}_{2})\right)\leq d(\mathcal{G}_{1},\mathcal{G}_{2})
\]
we obtain the desired inequality:
\[
\left|\left\langle \overrightarrow{xy},\mathcal{G}_{1}\right\rangle -\left\langle \overrightarrow{xy},\mathcal{G}_{2}\right\rangle \right|\leq{\textstyle \sum}_{i=0}^{n-1}d(x_{i},x_{i+1})\cdot d(\mathcal{G}_{1},\mathcal{G}_{2})=d(x,y)\cdot d(\mathcal{G}_{1},\mathcal{G}_{2}).
\]

\subsection{Convex projections}

Let $C$ be a closed convex subset of $\mathbf{X}(G)$ which is complete
in the induced topology. Then for every $x\in\mathbf{X}(G)$, there
is a unique point $p(x)$ in $C$ such that $d(x,p(x))=d(x,C)=\inf\left\{ d(x,y):y\in C\right\} $.
We call 
\[
p:\mathbf{X}(G)\twoheadrightarrow C
\]
the convex projection onto $C$. It is non-expanding, constant on
the segment $[x,p(x)]$, the map $H:\mathbf{X}(G)\times[0,1]\rightarrow\mathbf{X}(G)$
associating to $(x,t)$ the unique point at distance $td(x,p(x))$
from $x$ on the segment $[x,p(x)]$ is a continuous homotopy from
$\mathrm{Id}_{\mathbf{X}(G)}$ to $p$, and $\measuredangle_{p(x)}(x,y)\geq\frac{\pi}{2}$
for every $y\in C$ by \cite[II.2.4]{BrHa99}, thus also 
\[
\left\langle x,y\right\rangle _{p(x)}\leq0.
\]
For any $\mathcal{F}\in\mathbf{F}(G)$ such that $p(x)+t\mathcal{F}$
belongs to $C$ for all sufficiently small $t>0$, {\small{
\[
\left\langle \overrightarrow{p(x)x},\mathcal{F}\right\rangle \leq{\textstyle \frac{1}{2}}\left(d(x,C)^{2}+\left\Vert \mathcal{F}\right\Vert ^{2}-d(x,p(x)+\mathcal{F})^{2}\right)\leq\left\langle \loc_{p(x)}^{a}(x),\loc_{p(x)}(\mathcal{F})\right\rangle \leq0.
\]
}}{\small \par}

\section{The affine $\mathbf{F}(P/U)$-space $\mathbf{T}_{P}^{\infty}\mathbf{X}(G)$}

\nomenclature[T_P^ooX(G)]{$\mathbf{T}_{P}^{\infty}\mathbf{X}(G)$}{Quotient of $\mathbf{X}(G)$ by the unipotent radical $\mathsf{U}$ of $\mathsf{P}$, page \nomrefpage}Let
$\mathbf{X}(G)$ be an affine $\mathbf{F}(G)$-building. Fix a parabolic
subgroup $P$ of $G$ and let $U$ be the unipotent radical of $P$.

\subsection{~}

We have already seen that a Levi subgroup $L$ of $P$ determines: 
\begin{enumerate}
\item a parabolic subgroup $P_{L}^{\iota}$ of $G$ opposed to $P$ with
$P\cap P_{L}^{\iota}=L$,
\item a splitting map $F^{-1}(P)\hookrightarrow\mathbf{G}(L)$,
\item an opposition map $F^{-1}(P)\ni\mathcal{F}\mapsto\mathcal{F}_{L}^{\iota}\in F^{-1}(P_{L})$,
\item a section $\mathbf{F}(P/U)\ni\mathcal{H}\mapsto\mathcal{H}_{L}\in\mathbf{F}(L)$
of $\Gr_{P}:\mathbf{F}(G)\twoheadrightarrow\mathbf{F}(P/U)$,
\item a fundamental domain $\mathbf{X}(L)=\cup_{S\in\mathbf{S}(L)}\mathbf{X}(S)$
for the $\mathsf{U}$-action on $\mathbf{X}(G)$,
\item an $\mathsf{L}$-equivariant, $\mathsf{U}$-invariant retraction $\mathbf{X}(G)\ni x\mapsto x_{L}\in\mathbf{X}(L)$.
\end{enumerate}
The splitting map takes $\mathcal{F}$ to its unique splitting $\mathcal{G}$
with $Z_{G}(\mathcal{G})=L$. We have $\mathcal{F}_{L}^{\iota}=\Fi(\iota\mathcal{G})$
and $P_{\mathcal{F}_{L}^{\iota}}=P_{L}^{\iota}$. The $\Gr_{P}$-map
and its section are discussed in \ref{sub:StdAndApartments}, and
$\Gr_{P}$ is defined everywhere by theorem~\ref{thm:Std=00003DAlloverField}.
Finally $(5)$ and $(6)$ come from proposition~\ref{prop:distretrac},
which also says that for any\emph{ }$\mathcal{F}$ in the facet $F^{-1}(P)$,
\[
x_{L}=\lim_{t\rightarrow\infty}(x+t\mathcal{F})+t\mathcal{F}_{L}^{\iota}\quad\mbox{in }\mathbf{X}(G).
\]

\subsection{~\label{sub:orbitsOfU}}

For any $x,y\in\mathbf{X}(G)$, the following conditions are equivalent:
\begin{enumerate}
\item $\mathsf{U}\cdot x=\mathsf{U}\cdot y$,
\item $\lim_{t\rightarrow\infty}\mathbf{d}(x+t\mathcal{F},y+t\mathcal{F})=0$
for some (or every) $\mathcal{F}\in F^{-1}(P)$,
\item $\lim_{t\rightarrow\infty}d(x+t\mathcal{F},y+t\mathcal{F})=0$ for
some (or every) $\mathcal{F}\in F^{-1}(P)$,
\item $x_{L}=y_{L}$ for some (or every) Levi subgroup $L$ of $P$
\end{enumerate}
Indeed $(1)\Rightarrow(2)$ by the axiom $UN$, $(2)\Rightarrow(3)$
is trivial, $(3)\Rightarrow(4)$ because 
\[
(3)\stackrel{NE}{\Longrightarrow}\lim_{t\rightarrow\infty}d\left(\left(x+t\mathcal{F}\right)+t\mathcal{F}_{L}^{\iota},\left(y+t\mathcal{F}\right)+t\mathcal{F}_{L}^{\iota}\right)=0\Longrightarrow(4)
\]
and $(4)\Rightarrow(1)$ is obvious. If $\mathbf{X}(G)$ satisfies
$UN^{+}$, they are also equivalent to: 

$\,\,\,(5)$ $x+t\mathcal{F}=y+t\mathcal{F}$ for $t\gg0$. 

\noindent Indeed $(1)\Rightarrow(5)$ by $UN^{+}$ and plainly $(5)\Rightarrow(3)$.

\subsection{~}

For any $\overline{S}\in\mathbf{S}(P/U)$, there is a $\mathsf{U}$-equivariant
bijection between the set $\mathbf{S}(P,\overline{S})$ of all $S\in\mathbf{S}(G)$
with $Z_{G}(S)\subset P$ such that $P\rightarrow P/U$ induces an
isomorphism from $S$ to $\overline{S}$, and the set of all Levi
subgroups $L$ of $P$. It maps $S$ to the unique Levi subgroup $L_{S}$
containing $Z_{G}(S)$ and $L$ to the unique lift $S_{L}$ of $\overline{S}$
in $L\simeq P/U$, see lemma~\ref{lem:LeviAndAppart}. In particular,
$\mathbf{S}(P,\overline{S})$ is a $\mathsf{U}$-torsor.

\subsection{~\label{sub:TangentAtInftyInBuild}}

There is a structure of affine $\mathbf{F}(P/U)$-space on $\mathsf{U}\backslash\mathbf{X}(G)$,
\[
\mathbf{T}_{P}^{\infty}\mathbf{X}(G)=\left(\mathsf{U}\backslash\mathbf{X}(G),+,\mathbf{T}_{P}^{\infty}\mathbf{X}(-)\right).
\]
The $\mathsf{P}/\mathsf{U}$-equivariant apartment map is defined
by 
\begin{eqnarray*}
\mathbf{T}_{P}^{\infty}\mathbf{X}(\overline{S}) & = & \mathsf{U}\backslash\cup_{S\in\mathbf{S}(P,\overline{S})}\mathbf{X}(S)\\
 & = & \mbox{image of }\mathbf{X}(S)\mbox{ in }\mathsf{U}\backslash\mathbf{X}(G)\mbox{\,\ for any }S\in\mathbf{S}(P,\overline{S}).
\end{eqnarray*}
The $\mathsf{P}/\mathsf{U}$-equivariant pull map takes $x\in\mathsf{U}\backslash\mathbf{X}(G)$
and $\mathcal{H}\in\mathbf{F}(P/U)$ to
\[
x+\mathcal{H}=\mbox{image of }x_{L}+\mathcal{H}_{L}\mbox{ in }\mathsf{U}\backslash\mathbf{X}(G)
\]
where $L$ is any Levi subgroup of $P$: if $L'$ is another one,
there is a unique $u\in\mathsf{U}$ such that $\Int(u)(L)=L'$. Then
$x_{L'}=ux_{L}$, $\mathcal{H}_{L'}=u\mathcal{H}_{L}$, thus $x_{L'}+\mathcal{H}_{L'}=u(x_{L}+\mathcal{H}_{L})$
and $x_{L}+\mathcal{H}_{L}$ have the same image in $\mathsf{U}\backslash\mathbf{X}(G)$.
This defines an affine $\mathbf{F}(P/U)$-space: for any $\overline{S}\in\mathbf{S}(P/U)$,
$x\mapsto x_{L}$ and $\mathcal{H}\mapsto\mathcal{H}_{L}$ yield bijections
$\mathbf{T}_{P}^{\infty}\mathbf{X}(\overline{S})\rightarrow\mathbf{X}(S_{L})$
and $\mathbf{F}(\overline{S})\rightarrow\mathbf{F}(S_{L})$, thus
$+$ indeed induces a structure of affine $\mathbf{F}(\overline{S})$-space
on $\mathbf{T}_{P}^{\infty}\mathbf{X}(\overline{S})$.

\subsection{~\label{sub:defDistOnTangentSpaceAtInfty}}

There is a $\mathsf{P}/\mathsf{U}$-invariant distance on $\mathbf{T}_{P}^{\infty}\mathbf{X}(G)$,
given by the formulas:
\begin{eqnarray*}
d(x,y) & = & d(x_{L},y_{L})\\
 & = & \lim_{t\rightarrow\infty}d(x_{0}+t\mathcal{F},y_{0}+t\mathcal{F})\\
 & = & \inf\left\{ d(x',y'):(x',y')\in\mathbf{X}(G)^{2},\mbox{ }(x',y')\mapsto(x,y)\right\} 
\end{eqnarray*}
In the first formula, $L$ is any Levi subgroup of $P$. In the second
formula, $\mathcal{F}\in\mathbf{F}(G)$ belongs to the facet $F^{-1}(P)$
and $(x_{0},y_{0})\in\mathbf{X}(G)^{2}$ lifts $(x,y)$. The three
formulas agree: writing $d_{i}$ for the function defined by the $i$-th
formula, first note that $d_{2}$ is well-defined (by $NE$), independent
of the chosen lift (by $UN$), and not greater than $d(x_{0},y_{0})$
(by $NE$). Therefore $d_{2}(x,y)\leq d_{3}(x,y)\leq d_{1}(x,y)$.
But
\begin{eqnarray*}
d_{1}(x,y) & = & d\left(\lim_{t\rightarrow\infty}\left((x_{0}+t\mathcal{F})+t\mathcal{F}_{L}^{\iota}\right),\lim_{t\rightarrow\infty}\left((y_{0}+t\mathcal{F})+t\mathcal{F}_{L}^{\iota}\right)\right)\\
 & = & \lim_{t\rightarrow\infty}d\left(\left((x_{0}+t\mathcal{F})+t\mathcal{F}_{L}^{\iota}\right),\left((y_{0}+t\mathcal{F})+t\mathcal{F}_{L}^{\iota}\right)\right)\\
 & \leq & \lim_{t\rightarrow\infty}d\left(x_{0}+t\mathcal{F},y_{0}+t\mathcal{F}\right)=d_{2}(x,y)
\end{eqnarray*}
by $NE$, so that indeed $d=d_{1}=d_{2}=d_{3}$. It is obviously a
$\mathsf{P}/\mathsf{U}$-invariant distance, and it restricts to a
Euclidean norm on any apartment $\mathbf{T}_{P}^{\infty}\mathbf{X}(\overline{S})$.
The projection
\[
\Gr_{P}^{\infty}:\mathbf{X}(G)\twoheadrightarrow\mathbf{T}_{P}^{\infty}\mathbf{X}(G)
\]
is non-expanding and restricts to an isometry on $\mathbf{X}(L)$.\nomenclature[Gr_P^oo]{$\Gr_{P}^{\infty}$}{Projection $\mathbf{X}(G) \twoheadrightarrow \mathbf{T}_{P}^{\infty}\mathbf{X}(G)$, page \nomrefpage}

\subsection{~}

If $\mathbf{T}_{P}^{\infty}\mathbf{X}(G)$ satisfies $R(s)$ and $(\mathbf{X}(G),d)$
is a $CAT(0)$-metric space, then so is $(\mathbf{T}_{P}^{\infty}\mathbf{X}(G),d)$.
Indeed, it is a geodesic space by $R(s)$, so it remains to establish
the $CAT(0)$-inequality. Let thus $v,v_{0},v_{1}$ be three points
of $\mathbf{T}_{P}^{\infty}\mathbf{X}(G)$ and choose $\overline{S}\in\mathbf{S}(P/U)$
such that $v_{0},v_{1}$ belong to $\mathbf{T}_{P}^{\infty}\mathbf{X}(\overline{S})$.
Lift $\overline{S}$ to $S\in\mathbf{S}(P,\overline{S})$ and $v_{i}$
to $x_{i}\in\mathbf{X}(S)$, and lift $v$ to $x\in\mathbf{X}(G)$.
For $u\in[0,1]$, let $x_{u}=(1-u)x_{0}+ux_{1}$ be the point at distance
$ud(x_{0},x_{1})$ from $x_{0}$ on the segment $[x_{0},x_{1}]$ of
$\mathbf{X}(S)$. Fix $\mathcal{F}\in\mathbf{F}(S)$ with $P_{\mathcal{F}}=P$.
Then for every $t\geq0$, $x_{u,t}=x_{u}+t\mathcal{F}$ is the point
at distance $ud(x_{0,t},x_{1,t})$ from $x_{0,t}$ on the segment
$[x_{0,t},x_{1,t}]$ of $\mathbf{X}(S)$ and $v_{u}=\Gr_{P}^{\infty}(x_{u})$
is the point at distance $ud(v_{0},v_{1})$ from $v_{0}$ on the segment
$[v_{0},v_{1}]$ of $\mathbf{T}_{P}^{\infty}\mathbf{X}(\overline{S})$.
Set $x_{t}=x+t\mathcal{F}$. By the $CAT(0)$-inequality in $\mathbf{X}(G)$,
\[
d(x_{t},x_{u,t})^{2}\leq(1-u)\cdot d(x_{t},x_{0,t})^{2}+u\cdot d(x_{t},x_{1,t})^{2}-u(1-u)\cdot d(x_{0,t},x_{1,t})^{2}.
\]
Taking the limit as $t\rightarrow\infty$ gives
\[
d(v,v_{u})^{2}\leq(1-u)\cdot d(v,v_{0})^{2}+u\cdot d(v,v_{1})^{2}-u(1-u)\cdot d(v_{0},v_{1})^{2}
\]
which is the $CAT(0)$-inequality for $\mathbf{T}_{P}^{\infty}\mathbf{X}(G)$.
\begin{rem}
Suppose that for any $x,y\in\mathbf{X}(G)$ and $\mathcal{F}\in\mathbf{F}(G)$
there is an $S\in\mathbf{S}(G)$ such that $x+t\mathcal{F}$ and $y+t\mathcal{F}$
belong to $\mathbf{X}(S)$ for $t\gg0$. Then also $\mathcal{F}\in\mathbf{F}(S)$
and the axiom $R(s)$ holds for $\mathbf{T}_{P}^{\infty}\mathbf{X}(G)$.
If $\mathbf{X}(G)$ satisfies $UN^{+}$, this condition on $\mathbf{X}(G)$
is actually equivalent to the axiom $R(s)$ for $\mathbf{T}_{P}^{\infty}\mathbf{X}(G)$.
\end{rem}

\subsection{~}

We shall always equip $\mathbf{F}(P/U)$ with the scalar product,
distance, norm\ldots{} which are induced by the representation $\Gr_{\mathcal{F}}^{\bullet}(\tau)$
of $P/U$. Here $\mathcal{F}$ is any filtration in the facet $F^{-1}(P)$,
and we view $\Gr_{\mathcal{F}}^{\bullet}(\tau)=\oplus_{\gamma}\Gr_{\mathcal{F}}^{\gamma}(\tau)$
as a representation of $P/U$. If $L$ is a Levi subgroup of $P$
and $\mathcal{G}$ is the corresponding splitting of $\mathcal{F}$,
the restriction of $\tau$ to $L$ splits as $\tau\vert L=\oplus\tau_{\gamma}$
with $V(\tau_{\gamma})=\mathcal{G}_{\gamma}(\tau)$ and the isomorphism
$L\simeq P/U$ maps $\tau_{\gamma}$ to the representation $\Gr_{\mathcal{F}}^{\gamma}(\tau)$
of $P/U$, thus $\tau\vert L$ to $\Gr_{\mathcal{F}}^{\bullet}(\tau)$.
In particular, $\Gr_{\mathcal{F}}^{\bullet}(\tau)$ is indeed a faithful
representation of $P/U$ and its isomorphism class $\Gr_{P}^{\bullet}(\tau)$
does not depend upon $\mathcal{F}$. It follows from this conventions
that the isomorphism $\mathbf{F}(L)\simeq\mathbf{F}(P/U)$ is compatible
with the scalar products, distances, norms\ldots{} which are induced
on $\mathbf{F}(L)$ and $\mathbf{F}(P/U)$ by the chosen faithful
representation $\tau$ of $G$.

\subsection{~\label{sub:QuotientAndBuildMetricAgree}}

If $\mathbf{T}_{P}^{\infty}\mathbf{X}(G)$ satisfies $L(s)$, then
$d(x,x+\mathcal{H})=\left\Vert \mathcal{H}\right\Vert $ for $x\in\mathbf{T}_{P}^{\infty}\mathbf{X}(G)$
and $\mathcal{H}\in\mathbf{F}(P/U)$. Indeed, choose $\overline{S}\in\mathbf{S}(P/U)$
with $x\in\mathbf{T}_{P}^{\infty}\mathbf{X}(\overline{S})$ and $\mathcal{H}\in\mathbf{F}(\overline{S})$.
Then $x_{L}\in\mathbf{X}(S_{L})$ and $\mathcal{H}_{L}\in\mathbf{F}(S_{L})$,
thus also $x_{L}+\mathcal{H}_{L}\in\mathbf{X}(S_{L})$. In particular,
$(x+\mathcal{H})_{L}=x_{L}+\mathcal{H}_{L}$, thus $d(x,x+\mathcal{H})=d(x_{L},x_{L}+\mathcal{H}_{L})=\left\Vert \mathcal{H}_{L}\right\Vert =\left\Vert \mathcal{H}\right\Vert $:
if the affine $\mathbf{F}(P/U)$-space $\mathbf{T}_{P}^{\infty}\mathbf{X}(G)$
actually is an affine $\mathbf{F}(P/U)$-building, its ``quotient''
distance defined above agrees with its ``building'' distance defined
in section~\ref{sub:The-classical-distance}.

\subsection{~\label{sub:ComputingBusemanAtInfinity}}

Suppose again that $(\mathbf{X}(G),d)$ is a $CAT(0)$-space. The
Busemann scalar product of section~\ref{sub:BusemannFunctions} on
$\mathbf{X}(G)^{2}\times\mathbf{F}(G)$, namely 
\[
\left\langle \overrightarrow{xy},\mathcal{F}\right\rangle =\left\Vert \mathcal{F}\right\Vert \cdot\lim_{t\rightarrow\infty}\left(d(x,z+t\mathcal{F})-d(y,z+t\mathcal{F})\right)
\]
induces a function on $\mathbf{T}_{P}^{\infty}\mathbf{X}(G)^{2}\times F^{-1}(P)$.
Indeed for $u,v\in\mathsf{U}$ and $\mathcal{F}\in F^{-1}(P)$, {\small{
\begin{eqnarray*}
\lim_{t\rightarrow\infty}\left(d(ux,z+t\mathcal{F})-d(vy,z+t\mathcal{F})\right) & = & \lim_{t\rightarrow\infty}\left(d(x,u^{-1}z+t\mathcal{F})-d(y,v^{-1}z+t\mathcal{F})\right)\\
 & = & \lim_{t\rightarrow\infty}\left(d(x,z+t\mathcal{F})-d(y,z+t\mathcal{F})\right)
\end{eqnarray*}
}}by the triangle inequality and the axiom $UN$ for $\mathbf{X}(G)$.
If $\mathbf{T}_{P}^{\infty}\mathbf{X}(G)$ moreover satisfies $R(s)$,
the resulting function depends only on the image $\overline{\mathcal{F}}=\Gr_{P}\mathcal{F}$
of $\mathcal{F}$ in 
\[
\mathbf{G}(\overline{R}(P))=\mathbf{G}(Z(P/U))\subset\mathbf{F}(P/U).
\]
In fact, it is simply the corresponding Busemann scalar product
\[
\left\langle \overrightarrow{xy},\overline{\mathcal{F}}\right\rangle =\left\Vert \overline{\mathcal{F}}\right\Vert \cdot\lim_{t\rightarrow\infty}\left(d(x,z+t\overline{\mathcal{F}})-d(y,z+t\mathcal{\overline{F}})\right)
\]
on the CAT(0)-space $(\mathbf{T}_{P}^{\infty}\mathbf{X}(G),d)$. Indeed,
pick $\overline{S}\in\mathbf{S}(P/U)$ with $x,z\in\mathbf{T}_{P}^{\infty}\mathbf{X}(\overline{S})$.
Then $x_{L},z_{L}\in\mathbf{X}(S_{L})$, $\mathcal{F}$ equals $(\mathcal{\overline{F}})_{L}$
and belongs to $\mathbf{F}(S_{L})$, $z_{L}+t\mathcal{F}$ belongs
to $\mathbf{X}(S_{L})$, therefore $(z+t\overline{\mathcal{F}})_{L}=(z_{L}+t(\overline{\mathcal{F}})_{L})_{L}=(z_{L}+t\mathcal{F})_{L}=z_{L}+t\mathcal{F}$,
and finally
\[
d(x,z+t\overline{\mathcal{F}})=d(x_{L},z_{L}+t\mathcal{F})
\]
for all $t\geq0$, which proves our claim since also $\left\Vert \mathcal{F}\right\Vert =\left\Vert (\mathcal{\overline{F}})_{L}\right\Vert =\left\Vert \overline{\mathcal{F}}\right\Vert $.
However, $z\mapsto z+t\overline{\mathcal{F}}$ is now a Clifford translation
of $\mathbf{T}_{P}^{\infty}\mathbf{X}(G)$ \cite[II 6.14]{BrHa99}:
we have just seen that it corresponds to $z_{L}\mapsto z_{L}+t\mathcal{F}$
on the isometric space $\mathbf{X}(L)$, and the latter map is non-expanding
by $NE$ with non-expanding inverse $z_{L}\mapsto z_{L}+t\mathcal{F}_{L}^{\iota}$.
It is therefore an isometry, and a Clifford translation since $d(z_{L},z_{L}+t\mathcal{F})=t\left\Vert \mathcal{F}\right\Vert $
for all $z_{L}\in\mathbf{X}(L)$. If $\mathcal{F}\neq0$, it now follows
from~\cite[II 6.15]{BrHa99} that the comparison angle
\[
t\mapsto\measuredangle_{x}^{c}(y,x+t\overline{\mathcal{F}})
\]
is constant. Taking $z=x$ in the defining formula for $\left\langle \overrightarrow{xy},\overline{\mathcal{F}}\right\rangle $,
we thus find that
\[
\left\langle \overrightarrow{xy},\overline{\mathcal{F}}\right\rangle =d(x,y)\cdot\left\Vert \mathcal{F}\right\Vert \cdot\cos\left(\lim_{t\rightarrow\infty}\measuredangle_{x}^{c}(y,x+t\overline{\mathcal{F}})\right)
\]
equals
\[
\left\langle \overrightarrow{xy},\overline{\mathcal{F}}_{x}\right\rangle =d(x,y)\cdot\left\Vert \mathcal{F}\right\Vert \cdot\cos\left(\lim_{t\rightarrow0}\measuredangle_{x}^{c}(y,x+t\overline{\mathcal{F}})\right).
\]
Note that using~\cite[I.1.16]{BrHa99} again, this is also equal
to 
\[
\left\langle \loc_{x}^{a}(y),\overline{\mathcal{F}}_{x}\right\rangle =d(x,y)\cdot\left\Vert \mathcal{F}\right\Vert \cdot\cos\left(\lim_{t\rightarrow0}\measuredangle_{x}^{c}(y_{t},x+t\overline{\mathcal{F}})\right)
\]
where $y_{t}=ty+(1-t)x$ is the point at distance $td(x,y)$ from
$x$ on the segment $[x,y]$ of the $CAT(0)$-space $\mathbf{T}_{P}^{\infty}\mathbf{X}(G)$.
In any case, we obtain yet another series of formulas for the relevant
Busemann function on $\mathbf{T}_{P}^{\infty}\mathbf{X}(G)$ or $\mathbf{X}(G)$. 
\begin{rem}
If the affine $\mathbf{F}(P/U)$-space $\mathbf{T}_{P}^{\infty}\mathbf{X}(G)$
is an affine $\mathbf{F}(P/U)$-building, we may also directly apply
lemma~\ref{lem:AnglesAreEqualWhenGcentral} to $\mathcal{G}=\overline{\mathcal{F}}=\Gr_{P}(\mathcal{F})$,
thereby obtaining
\[
\left\langle \overrightarrow{xy},\mathcal{F}\right\rangle =\left\langle \loc_{\Gr_{P}^{\infty}(x)}^{a}\left(\Gr_{P}^{\infty}(y)\right),\loc_{\Gr_{P}^{\infty}(x)}\left(\Gr_{P}(\mathcal{F})\right)\right\rangle 
\]
where the second scalar product is in the tangent space $\mathbf{T}_{\Gr_{P}(x)}\left(\mathbf{T}_{P}^{\infty}\mathbf{X}(G)\right)$.
\end{rem}

\subsection{~}

If $\mathbf{X}(L)+\mathbf{F}(L)\subset\mathbf{X}(L)$, then $\mathbf{X}(L)$
inherits from $\mathbf{X}(G)$ a structure of affine $\mathbf{F}(L)$-space,
and the restriction of $\Gr_{P}^{\infty}:\mathbf{X}(G)\rightarrow\mathbf{T}_{P}^{\infty}\mathbf{X}(G)$
to $\mathbf{X}(L)$ is an isomorphism of affine $\mathbf{F}(L)$-spaces
-- viewing $\mathbf{T}_{P}^{\infty}\mathbf{X}(G)$ as an affine $\mathbf{F}(L)$-space
through $\mathbf{F}(L)\simeq\mathbf{F}(P/U)$. Most of the above discussion
then becomes much easier.

\section{Example: $\mathbf{F}(G)$ as a tight affine $\mathbf{F}(G)$-building}

\subsection{~}

Recall from example~\ref{Example:F(G)Itself} that $(\mathbf{F}(G),+,\mathbf{F}(-))$
is a discrete affine $\mathbf{F}(G)$-space with trivial type. Under
the identification $\mathbf{F}(G)=\mathbf{F}(\omega_{K}^{\circ})$,
the pull map may be computed as follows: for $\mathcal{F}_{1},\mathcal{F}_{2}\in\mathbf{F}(G)$,
$\rho\in\Rep^{\circ}(G)(K)$ and $\gamma\in\mathbb{R}$, 
\[
(\mathcal{F}_{1}+\mathcal{F}_{2})^{\gamma}(\rho)=\sum_{\gamma_{1}+\gamma_{2}=\gamma}\mathcal{F}_{1}^{\gamma_{1}}(\rho)\cap\mathcal{F}_{2}^{\gamma_{2}}(\rho).
\]
We have already mentioned that $\mathbf{F}(G)$ satisfies $L(s)=R(s)$
by~theorem~\ref{thm:Std=00003DAlloverField} and $L(i)=R(i)$ by
corollary~\ref{cor:AxiomeR(i)4F(G)}. Actually for $\mathcal{F},\mathcal{G}\in\mathbf{F}(G)$,
choosing $S$ in $\mathbf{S}(G)$ with $\mathcal{F},\mathcal{G}\in\mathbf{F}(S)$,
we find using~\ref{sub:descFacet} that $\mathcal{F}+\eta\mathcal{G}$
belongs to a fixed closed chamber of $\mathbf{F}(S)$ for all sufficiently
small $\eta\geq0$, from which the stronger axiom $L(s)^{+}$ easily
follows. We have also seen in example~\ref{Example:F(G)tight} that
$\mathbf{F}(G)$ satisfies the axiom $ST$. It thus satisfies $UN^{+}$
and $UN$ by lemma~\ref{lem:RelationBetweenAxiomsST}, and it is
therefore a (tight) affine $\mathbf{F}(G)$-building by proposition~\ref{prop:L(s)+implBuild}.
It also trivially satisfies the axiom $HA$, because every apartment
contains the origin $0\in\mathbf{F}(G)$. The latter is fixed by $\mathsf{G}$,
and it follows from lemma~\ref{lem:UnicityTight} that $\mathbf{F}(G)$
is, up to isomorphism, the unique affine $\mathbf{F}(G)$-building
with a point fixed by $\mathsf{G}$. Indeed, any such building has
trivial type and satisfies $ST_{2}^{-}$, thus also $ST^{-}$ by lemma~\ref{lem:RelationBetweenAxiomsST}.

\subsection{~}

The retractions of corollary~\ref{cor:Ret} and proposition~\ref{prop:distretrac}
agree, and so do the decompositions of~sections~\ref{sub:Isogenies}
and \ref{sub:decer4affbuil} (with base point $0\in\mathbf{F}(G)$).

\subsection{~}

The distance $d=d_{\tau}$ of section~\ref{sub:FiloverFieldsisTitsBuild}
is equal to the corresponding distance on the affine $\mathbf{F}(G)$-building
$\mathbf{F}(G)$ defined in section~\ref{sub:The-classical-distance}.
For $\mathcal{F},\mathcal{G}\in\mathbf{F}(G)$ and $t\in[0,1]$, the
unique point at distance $td(\mathcal{F},\mathcal{G})$ from $\mathcal{F}$
on the segment $[\mathcal{F},\mathcal{G}]$ in the CAT(0)-space $(\mathbf{F}(G),d)$
is equal to the sum $t\mathcal{G}+(1-t)\mathcal{F}$, as defined above,
of the rescaled filtrations $t\mathcal{G}$ and $(1-t)\mathcal{F}$
of $\mathbf{F}(G)$: this is obvious in any apartment.

\subsection{~\label{sub:Gr_PforF(G)}}

For a parabolic subgroup $P=U\rtimes L$ of $G$ with unipotent radical
$U$ and Levi subgroup $L$, there is a commutative diagram
\[
\xyR{2pc}\xyC{2pc}\xymatrix{ & \mathbf{F}(L)\ar@<1ex>@{^{(}->}[d]^{\iota_{L,G}}\ar@(r,u)[ddr]^{\simeq}\ar@(l,u)[ddl]_{\simeq}^{\iota_{L}}\\
 & \mathbf{F}(G)\ar@<1ex>@{->>}[u]^{r_{P,L}}\ar@{->>}[dl]_{\Gr_{P}}\ar@{->>}[dr]^{\Gr_{P}^{\infty}}\\
\mathbf{F}(P/U)\ar[rr]^{\simeq} &  & \mathbf{T}_{P}^{\infty}\mathbf{F}(G)
}
\]
where $\iota_{L}:\mathbf{F}(L)\simeq\mathbf{F}(P/U)$ and $\iota_{L,G}:\mathbf{F}(L)\hookrightarrow\mathbf{F}(G)$
are the $\mathsf{L}$-equivariant maps functorially induced by the
isomorphism $L\simeq P/U$ and the embedding $L\hookrightarrow G$,
$r_{P,L}:\mathbf{F}(G)\twoheadrightarrow\mathbf{F}(L)$ is the $\mathsf{U}$-invariant
$\mathsf{L}$-equivariant retraction of corollary~\ref{cor:Ret},
$\Gr_{P}$ is the $\mathsf{P}$-equivariant morphism of section~\ref{sub:defofGrMap}
(which is defined on the whole of $\mathbf{F}(G)$ by theorem~\ref{thm:Std=00003DAlloverField})\nomenclature[Gr_P]{$\Gr _P$}{$\mathsf{P}$-equivariant map $\mathbf{F}(G) \twoheadrightarrow \mathbf{F}(P/U)$, page \nomrefpage}
and $\Gr_{P}^{\infty}:\mathbf{F}(G)\twoheadrightarrow\mathbf{T}_{P}^{\infty}\mathbf{F}(G)$
is the $\mathsf{P}$-equivariant quotient map onto $\mathbf{T}_{P}^{\infty}\mathbf{F}(G)=\mathsf{U}\backslash\mathbf{F}(G)$.
Since $\Gr_{P}$ is $\mathsf{P}$-equivariant (thus $\mathsf{U}$-invariant)
and $\Gr_{P}\circ\iota_{L,G}=\iota_{L}$, also $\Gr_{P}=\iota_{L}\circ r_{P,L}$.
The right hand side triangles are plainly commutative, and this implies
the existence of the bottom map bijection. One checks easily that
it is an isomorphism of affine $\mathbf{F}(P/U)$-spaces. In particular:
$\mathbf{T}_{P}^{\infty}\mathbf{F}(G)$ is an affine $\mathbf{F}(P/U)$-building,
its ``quotient'' and ``building'' metric agree by~\ref{sub:QuotientAndBuildMetricAgree},
thus $\mathbf{F}(P/U)\rightarrow\mathbf{T}_{P}^{\infty}\mathbf{F}(G)$
is an isometry while $\Gr_{P}:\mathbf{F}(G)\twoheadrightarrow\mathbf{F}(P/U)$
is non-expanding. This gives the following formula: for any $\mathcal{F},\mathcal{G}_{1},\mathcal{G}_{2}\in\mathbf{F}(G)$,
\[
\lim_{t\rightarrow\infty}{\textstyle d\left(\mathcal{G}_{1}+t\mathcal{F},\mathcal{G}_{2}+t\mathcal{F}\right)=d\left(\Gr_{\mathcal{F}}(\mathcal{G}_{1}),\Gr_{\mathcal{F}}(\mathcal{G}_{2})\right)}\leq d(\mathcal{G}_{1},\mathcal{G}_{2})
\]
where $\Gr_{\mathcal{F}}=\Gr_{P_{\mathcal{F}}}:\mathbf{F}(G)\twoheadrightarrow\mathbf{F}(P_{\mathcal{F}}/U_{\mathcal{F}})$.
Also, 
\begin{eqnarray*}
\left\langle \Gr_{P}(\mathcal{G}_{1}),\Gr_{P}(\mathcal{G}_{2})\right\rangle  & \geq & \left\langle \mathcal{G}_{1},\mathcal{G}_{2}\right\rangle \\
\measuredangle\left(\Gr_{P}(\mathcal{G}_{1}),\Gr_{P}(\mathcal{G}_{2})\right) & \leq & \measuredangle\left(\mathcal{G}_{1},\mathcal{G}_{2}\right)
\end{eqnarray*}
since $\Gr_{P}$ contracts the distances and preserves the norms. 
\begin{rem}
Here is a more direct proof of the fact that $\Gr_{P}$ is non-expanding:
starting with $\mathcal{G}_{1},\mathcal{G}_{2}\in\mathbf{F}(G)$,
cut the segment $[\mathcal{G}_{1},\mathcal{G}_{2}]$ along its facet
decomposition, going from $\mathcal{H}_{0}=\mathcal{G}_{1}$ to $\mathcal{H}_{n}=\mathcal{G}_{2}$
with $F$ constant on $]\mathcal{H}_{i-1},\mathcal{H}_{i}[$. Then
observe that $\Gr_{P}$ restrict to an isometry on any facet $F^{-1}(Q)$,
$Q\in\mathbf{P}(G)$: it restricts to an isometry on any apartment
containing $F^{-1}(P)$, and there is at least one such apartment
which also contains $F^{-1}(Q)$ (along with its closure). Thus 
\begin{eqnarray*}
d\left(\mathcal{G}_{1},\mathcal{G}_{2}\right) & = & {\textstyle \sum}_{i=1}^{n}d\left(\mathcal{H}_{i-1},\mathcal{H}_{i}\right)\\
 & = & {\textstyle \sum}_{i=1}^{n}d\left(\Gr_{P}(\mathcal{H}_{i-1}),\Gr_{P}(\mathcal{H}_{i})\right)\\
 & \geq & d\left(\Gr_{P}(\mathcal{G}_{1}),\Gr_{P}(\mathcal{G}_{2})\right)
\end{eqnarray*}
by the triangle inequality in $\mathbf{F}(P/U)$. One can also probably
establish the inequalities using the explicit formulas for the scalar
products, but this involves playing around with three filtrations.
In any case, these approaches do not yield an exact formula relating
the distances on $\mathbf{F}(P/U)$ and $\mathbf{F}(G)$. 
\end{rem}

\subsection{~}

For $\mathcal{F},\mathcal{G}\in\mathbf{F}(G)$, choose $S\in\mathbf{S}(G)$
with $\mathcal{F},\mathcal{G}\in\mathbf{F}(S)$. Then by~\ref{sub:descFacet},
there is a facet $F\subset\mathbf{F}(S)$ of $\mathbf{F}(G)$ such
that $\mathcal{F}+t\mathcal{G}$ belongs to $F$ for every sufficiently
small $t>0$. If $P\supset Z_{G}(S)$ is the corresponding parabolic
subgroup, then $P\subset P_{\mathcal{F}}$ since $\mathcal{F}$ belongs
to the closure of $F$, thus also $U_{\mathcal{F}}\subset U$ where
$U$ is the unipotent radical of $P$. In particular, $\mathcal{F}+tu\mathcal{G}=u(\mathcal{F}+t\mathcal{G})=\mathcal{F}+t\mathcal{G}$
for every $u\in\mathsf{U}_{\mathcal{F}}$, thus 
\[
\loc_{\mathcal{F}}:\mathbf{F}(G)\twoheadrightarrow\mathbf{T}_{\mathcal{F}}\mathbf{F}(G)\qquad\mathcal{G}\mapsto\mbox{germ of }(t\mapsto\mathcal{F}+t\mathcal{G})
\]
is $\mathsf{U}_{\mathcal{F}}$-invariant. Since $\Gr_{\mathcal{F}}$
induces a bijection $\mathsf{U}_{\mathcal{F}}\backslash\mathbf{F}(G)\simeq\mathbf{F}(P_{\mathcal{F}}/U_{\mathcal{F}})$,
it follows that there is a canonical $\mathsf{P}_{\mathcal{F}}$-equivariant
commutative diagram
\[
\xyC{2pc}\xymatrix{ & \mathbf{F}(G)\ar@{->>}[dl]_{\Gr_{\mathcal{F}}}\ar@{->>}[dr]^{\loc_{\mathcal{F}}}\\
\mathbf{F}(P_{\mathcal{F}}/U_{\mathcal{F}})\ar[rr]^{\varphi} &  & \mathbf{T}_{\mathcal{F}}\mathbf{F}(G)
}
\]
For $S\in\mathbf{S}(G)$ with $\mathcal{F}\in\mathbf{F}(S)$, it restrict
to a commutative diagram of isometries
\[
\xymatrix{ & (\mathbf{F}(S),d)\ar[dl]_{\simeq}\ar[dr]^{\simeq}\\
(\mathbf{F}(\overline{S}),d)\ar[rr]^{\simeq} &  & (\mathbf{T}_{\mathcal{F}}\mathbf{F}(S),d)
}
\]
where $\overline{S}$ is the image of $S$ in $P_{\mathcal{F}}/U_{\mathcal{F}}$.
Since any two elements $x,y\in\mathbf{F}(P_{\mathcal{F}}/U_{\mathcal{F}})$
are contained in one such $\mathbf{F}(\overline{S})$, it follows
that $\varphi:\mathbf{F}(P_{\mathcal{F}}/U_{\mathcal{F}})\rightarrow\mathbf{T}_{\mathcal{F}}\mathbf{F}(G)$
is an isometry. It is therefore also compatible with the relevant
norms, angles and scalar products. This gives the following explicit
formulas: 
\[
d\left(\Gr_{\mathcal{F}}(\mathcal{G}_{1}),\Gr_{\mathcal{F}}(\mathcal{G}_{2})\right)=\lim_{t\rightarrow0}{\textstyle \frac{1}{t}}{\textstyle d(\mathcal{G}_{1}+t\mathcal{F},\mathcal{G}_{2}+t\mathcal{F})}\leq d(\mathcal{G}_{1},\mathcal{G}_{2})
\]
and 
\[
\lim_{t\rightarrow0}{\textstyle \frac{1}{t}\left(d(\mathcal{F}+\mathcal{G}_{1},\mathcal{F})-d(\mathcal{F}+\mathcal{G}_{1},\mathcal{F}+t\mathcal{G}_{2})\right)}=\left\langle \Gr_{\mathcal{F}}(\mathcal{G}_{1}),\Gr_{\mathcal{F}}(\mathcal{G}_{2})\right\rangle 
\]
for every $\mathcal{G}_{1},\mathcal{G}_{2}\in\mathbf{F}(G)$, with
$\left\Vert \Gr_{\mathcal{F}}(\mathcal{G})\right\Vert =\left\Vert \mathcal{G}\right\Vert $
and 
\[
\left\langle \Gr_{\mathcal{F}}(\mathcal{G}_{1}),\Gr_{\mathcal{F}}(\mathcal{G}_{2})\right\rangle ={\textstyle \sum_{\gamma}}\left\langle \Gr_{\mathcal{F}}^{\gamma}(\mathcal{G}_{1},\tau),\Gr_{\mathcal{F}}^{\gamma}(\mathcal{G}_{2},\tau)\right\rangle 
\]
Also: $\loc_{\mathcal{F}}(\mathcal{G}_{1})=\loc_{\mathcal{F}}(\mathcal{G}_{2})$
if and only if $\mathsf{U}_{\mathcal{F}}\cdot\mathcal{G}_{1}=\mathsf{U}_{\mathcal{F}}\cdot\mathcal{G}_{2}$.
\begin{rem}
The previous results yield a $\mathsf{P}_{\mathcal{F}}$-equivariant
isometry between the tangent space $\mathbf{T}_{\mathcal{F}}\mathbf{F}(G)$
at $\mathcal{F}$ viewed as a point in the affine building and the
``tangent space'' $\mathbf{T}_{\mathcal{F}}^{\infty}\mathbf{F}(G)$
at $\mathcal{F}$ viewed as a boundary point. This reflects the homogeneity
of the vectorial Tits building $\mathbf{F}(G)$: our isometry is induced
by 
\[
(t\,\mbox{small})\quad\mathcal{F}+t\mathcal{G}\mapsto t^{-1}(\mathcal{F}+t\mathcal{G})=\mathcal{G}+t^{-1}\mathcal{F}\quad(t^{-1}\,\mbox{large})
\]

\end{rem}

\subsection{~}

There is also the localization map $\loc_{\mathcal{F}}^{a}:\mathbf{F}(G)\rightarrow\mathbf{T}_{\mathcal{F}}\mathbf{F}(G)$,
which sends $\mathcal{G}$ to $\loc_{\mathcal{F}}^{a}(\mathcal{G})=\loc_{\mathcal{F}}(\mathcal{H})$
if $\mathcal{G}=\mathcal{F}+\mathcal{H}$. Define $\Gr_{\mathcal{F}}^{a}=\varphi^{-1}\circ\loc_{\mathcal{F}}^{a}$,
so that 
\[
\xymatrix{ & \mathbf{F}(G)\ar@{->>}[dl]_{\Gr_{\mathcal{F}}^{a}}\ar@{->>}[dr]^{\loc_{\mathcal{F}}^{a}}\\
\mathbf{F}(P_{\mathcal{F}}/U_{\mathcal{F}})\ar[rr]^{\varphi} &  & \mathbf{T}_{\mathcal{F}}\mathbf{F}(G)
}
\]
One checks easily in $\mathbf{F}(S)\ni\mathcal{F},\mathcal{G}$ that
$\Gr_{\mathcal{F}}^{a}(\mathcal{G})+\overline{\mathcal{F}}=\Gr_{\mathcal{F}}(\mathcal{G})$,
where 
\[
\overline{\mathcal{F}}=\Gr_{\mathcal{F}}(\mathcal{F})\in\mathbf{G}(Z(P_{\mathcal{F}}/U_{\mathcal{F}}))=\Aut\left(\mathbf{F}(P_{\mathcal{F}}/U_{\mathcal{F}})\right)
\]
is the automorphism determined by $\mathcal{F}$ (see~\ref{sub:DefFondDiag},
\ref{sub:DomOrder} and \ref{sub:AutoOfBuild}). Thus
\begin{eqnarray*}
\left\langle \Gr_{\mathcal{F}}^{a}(\mathcal{G}),\Gr_{\mathcal{F}}(\mathcal{H})\right\rangle  & = & \left\langle \Gr_{\mathcal{F}}(\mathcal{G}),\Gr_{\mathcal{F}}(\mathcal{H})\right\rangle -\left\langle \overline{\mathcal{F}},\Gr_{\mathcal{F}}(\mathcal{H})\right\rangle \\
 & = & {\textstyle \sum_{\gamma}}\left\langle \Gr_{\mathcal{F}}^{\gamma}(\mathcal{G},\tau),\Gr_{\mathcal{F}}^{\gamma}(\mathcal{H},\tau)\right\rangle -\gamma\deg(\Gr_{\mathcal{F}}^{\gamma}(\mathcal{H},\tau))
\end{eqnarray*}
This gives an explicit formula for 
\[
\lim_{t\rightarrow0}{\textstyle \frac{1}{t}\left(d(\mathcal{G},\mathcal{F})-d(\mathcal{G},\mathcal{F}+t\mathcal{H})\right)}=\left\langle \Gr_{\mathcal{F}}^{a}(\mathcal{G}),\Gr_{\mathcal{F}}(\mathcal{H})\right\rangle .
\]

\subsection{~}

It is finally very easy to compute the Busemann functions:
\[
b_{0,\mathcal{G}}(\mathcal{F})=\lim_{t\rightarrow\infty}\left(d(\mathcal{F},t\mathcal{G})-t\left\Vert \mathcal{G}\right\Vert \right)=-\left\Vert \mathcal{F}\right\Vert \cos\measuredangle(\mathcal{F},\mathcal{G}).
\]
It follows that for any $\mathcal{F},\mathcal{G},\mathcal{H}\in\mathbf{F}(G)$,
\[
\left\langle \overrightarrow{\mathcal{F}\mathcal{G}},\mathcal{H}\right\rangle =\left\langle \mathcal{G},\mathcal{H}\right\rangle -\left\langle \mathcal{F},\mathcal{H}\right\rangle .
\]
Thus if $C$ is a closed convex subset of $\mathbf{F}(G)$, $\mathcal{F}\in C$
is the convex projection of some $\mathcal{G}\in\mathbf{F}(G)$ and
$\mathcal{H}\in\mathbf{F}(G)$ satisfies $\mathcal{F}+t\mathcal{H}\in C$
for all sufficiently $t>0$, then 
\[
\left\langle \mathcal{G},\mathcal{H}\right\rangle \leq\left\langle \mathcal{F},\mathcal{H}\right\rangle .
\]

\section{Example: a symmetric space\label{sub:ExampSymSpac}}

Let $K=\mathbb{R}$ and $G=GL(V)$, where $V$ is an $\mathbb{R}$-vector
space of dimension $n\in\mathbb{N}$.

\subsection{~}

By corollary~\ref{cor:GL(V)case}, the tautological representation
$V$ of $G$ identifies $\mathbf{G}(G)$ and $\mathbf{F}(G)$ with
the sets $\mathbf{G}(V)$ and $\mathbf{F}(V)$ of all $\mathbb{R}$-graduations
and $\mathbb{R}$-filtrations on $V$. Similarly, the action of $\mathsf{G}$
on the set $\mathbb{P}^{1}(V)$ of $\mathbb{R}$-lines in $V$ identifies
$\mathbf{S}(G)$ with 
\[
\mathbf{S}(V)=\left\{ \mathcal{S}\subset\mathbb{P}^{1}(V):V=\oplus_{L\in\mathcal{S}}L\right\} .
\]
We denote by $\mathbf{F}(\mathcal{S})$ the apartment of $\mathbf{F}(V)$
corresponding to $\mathcal{S}\in\mathbf{S}(V)$. An $\mathbb{R}$-filtration
$\mathcal{F}$ on $V$ thus belongs to  $\mathbf{F}(\mathcal{S})$
if and only if 
\[
\forall\gamma\in\mathbb{R}:\qquad\mathcal{F}^{\gamma}=\oplus_{L\in\mathcal{S},\,\mathcal{F}^{\sharp}(L)\geq\gamma}L\quad\mbox{where}\quad\mathcal{F}^{\sharp}(L)=\sup\{\lambda:L\subset\mathcal{F}^{\lambda}\}.
\]
We also identify $\mathbf{C}(G)$ with $\mathbb{R}_{\leq}^{n}=\{\gamma_{1}\leq\cdots\leq\gamma_{n}:\gamma_{i}\in\mathbb{R}\}$
by the map which sends $t(\mathcal{F})$ to $\underline{t}(\mathcal{F})=(t_{i}(\mathcal{F}))_{i=1}^{n}$,
with $\sharp\{i:t_{i}(\mathcal{F})=\gamma\}=\dim_{\mathbb{R}}\Gr_{\mathcal{F}}^{\gamma}(V)$
for $\gamma\in\mathbb{R}$. The dominance order on $\mathbf{C}(G)$
defined in section~\ref{sub:DomOrder} corresponds to 
\[
(\gamma_{i})_{i=1}^{n}\leq(\gamma'_{i})_{i=1}^{n}\iff\begin{cases}
\sum_{j=1}^{n}\gamma_{j}=\sum_{j=1}^{n}\gamma'_{j} & \mbox{and}\\
\sum_{j=i}^{n}\gamma_{j}\leq\sum_{j=i}^{n}\gamma'_{j} & \mbox{for }2\leq i\leq n.
\end{cases}
\]
The length $\left\Vert -\right\Vert :\mathbf{C}(G)\rightarrow\mathbb{R}_{+}$
attached to the tautological faithful representation $V$ of $G$
in~\ref{sub:constscalangdist} corresponds to the function $\left\Vert -\right\Vert :\mathbb{R}_{\leq}^{n}\rightarrow\mathbb{R}_{+}$
given by 
\[
\left\Vert \gamma_{1}\leq\cdots\leq\gamma_{n}\right\Vert =\sqrt{\gamma_{1}^{2}+\cdots+\gamma_{n}^{2}}.
\]

\subsection{~}

The exponential $\exp:\mathbb{R}\rightarrow\mathbb{R}^{\times}$ defines
an $\mathbb{R}$-valued section $\exp$ of the multiplicative group
$\mathbb{D}(\mathbb{R})$, whose evaluation at the character $\gamma\in\mathbb{R}$
is given by 
\[
\gamma(\exp)=\exp(\gamma)\in\mathbb{R}^{\times}.
\]
For $\mathcal{G}\in\mathbf{G}(V)$, we denote by $\mathcal{G}^{\flat}$
the endomorphism of $V$ which acts by $\gamma\in\mathbb{R}$ on the
direct summand $\mathcal{G}_{\gamma}$ of $V$. Viewing $\mathcal{G}$
as an $\mathbb{R}$-morphism $\mathbb{D}(\mathbb{R})\rightarrow G$,
we thus have $\exp(\mathcal{G}^{\flat})=\mathcal{G}(\exp)$ in $\mathsf{G}=G(\mathbb{R})$.
The maps $\mathcal{G}\mapsto\mathcal{G}^{\flat}\mapsto\exp(\mathcal{G}^{\flat})$
yield $\mathsf{G}$-equivariant bijections between $\mathbf{G}(V)$,
the set of diagonalizable endomorphisms of $V$ and the set of diagonalizable
elements of $\mathsf{G}$ with positive eigenvalues.

\subsection{~}

Let $\mathbf{B}(V)$ be the space of all Euclidean norms $\alpha$
on $V$, i.e.~functions 
\[
\alpha:V\rightarrow\mathbb{R}_{+}
\]
whose square $\alpha^{2}$ is a positive definite quadratic form on
$V$ -- thus $\mathbf{B}(V)$ may also be viewed as the space of all
scalar products on $V$, or as the space of all ellipsoids in $V$.
The group $\mathsf{G}$ acts transitively on $\mathbf{B}(V)$ by $(g\cdot\alpha)(v)=\alpha(g^{-1}v)$
and the stabilizer of $\alpha$ is the orthogonal group $\mathsf{O}(\alpha)\subset\mathsf{G}$,
thus $\mathbf{B}(V)\simeq\mathsf{G}/\mathsf{O}(\alpha)$ is a smooth
variety. We denote by $\mathbf{G}(V,\alpha)\subset\mathbf{G}(V)$
the set of all $\alpha$-orthogonal $\mathbb{R}$-graduations on $V$,
by $\mathbf{Sym}(V,\alpha)=\mathbf{G}(V,\alpha)^{\flat}$ the set
of all $\alpha$-symmetric endomorphisms of $V$, and by $\mathsf{G}(\alpha)=\exp(\mathbf{Sym}(V,\alpha))$
the set of all $\alpha$-symmetric automorphisms of $V$ with positive
eigenvalues. Thus $\mathbf{Sym}(V,\alpha)$ is the tangent space of
$\mathbf{B}(V)$ at $\alpha$ and the polar decomposition in $\mathsf{G}$
yields $\mathsf{G}=\mathsf{G}(\alpha)\cdot\mathsf{O}(\alpha)$. For
$\mathcal{F}\in\mathbf{F}(V)$ and $\gamma\in\mathbb{R}$, we denote
by $\mathcal{G}_{\alpha}(\mathcal{F})_{\gamma}$ the $\alpha$-orthogonal
complement of $\mathcal{F}_{+}^{\gamma}$ in $\mathcal{F}^{\gamma}$.
Then $\mathcal{G}_{\alpha}(\mathcal{F})$ is the unique splitting
of $\mathcal{F}$ in $\mathbf{G}(V,\alpha)$ and $\mathcal{G}_{\alpha}:\mathbf{F}(V)\rightarrow\mathbf{G}(V,\alpha)$
is an $\mathsf{O}(\alpha)$-equivariant section of $\Fi:\mathbf{G}(V)\twoheadrightarrow\mathbf{F}(V)$.
We obtain a sequence of $\mathsf{O}(\alpha)$-equivariant bijections
\[
\xymatrix{\mathbf{F}(V)\ar[r]\sp(0.45){\mathcal{G}_{\alpha}} & \mathbf{G}(V,\alpha)\ar[r]\sp(0.44){\flat} & \mathbf{Sym}(V,\alpha)\ar[r]\sp(0.58){\exp} & \mathsf{G}(\alpha)\ar[r]\sp(0.45){-\cdot\alpha} & \mathbf{B}(V).}
\]
We set $g_{\alpha}(\mathcal{F})=\exp(\mathcal{G}_{\alpha}(\mathcal{F})^{\flat})=\mathcal{G}_{\alpha}(\mathcal{F})(\exp)\in\mathsf{G}(\alpha)$
and define 
\[
\alpha+\mathcal{F}=g_{\alpha}(\mathcal{F})\cdot\alpha\quad\mbox{in}\quad\mathbf{B}(V).
\]
Thus for any $\alpha\in\mathbf{B}(V)$, $\mathcal{F}\in\mathbf{F}(G)$
and $v\in V$,
\[
(\alpha+\mathcal{F})(v)=\alpha\left({\textstyle \sum}_{\gamma}e^{-\gamma}v_{\gamma}\right):\quad v={\textstyle \sum_{\gamma}}v_{\gamma},\quad v_{\gamma}\in\mathcal{G}_{\alpha}(\mathcal{F})_{\gamma}.
\]
For $\mathcal{S}\in\mathbf{S}(V)$, we denote by $\mathbf{B}(\mathcal{S})$
the set of $\alpha$'s in $\mathbf{B}(V)$ for which $V=\oplus_{L\in\mathcal{S}}L$
is an orthogonal decomposition. Thus for $\alpha\in\mathbf{B}(\mathcal{S})$
and $\mathcal{F}\in\mathbf{F}(\mathcal{S})$, we find that 
\[
(\alpha+\mathcal{F})^{2}(v)={\textstyle \sum_{L\in\mathcal{S}}}(\alpha+\mathcal{F})^{2}(v_{L})={\textstyle \sum_{L\in\mathcal{S}}}(e^{-\mathcal{F}^{\sharp}(L)}\alpha)^{2}(v_{L})
\]
where $v=\sum_{L\in\mathcal{S}}v_{L}$ with $v_{L}\in L$, therefore
also $\alpha+\mathcal{F}\in\mathbf{B}(\mathcal{S})$.

\subsection{~}

The above formulas show that $\mathbf{B}(V)=\left(\mathbf{B}(V),\mathbf{B}(-),+\right)$
is an affine $\mathbf{F}(V)$-space. It is well-known that it satisfies
$R(s)$, and $L(s)$ follows from the existence of the $\alpha$-orthogonal
splittings. Moreover for any $\alpha\in\mathbf{B}(V)$, the pull map
\[
\mathbf{F}(V)\rightarrow\mathbf{B}(V),\qquad\mathcal{F}\mapsto\alpha+\mathcal{F}
\]
is an $\mathsf{O}(\alpha)$-equivariant bijection. The Fischer-Courant
theory tells us that the orbits of the diagonal action of $\mathsf{G}$
on $\mathbf{B}(V)\times\mathbf{B}(V)$ are classified by a $\mathsf{G}$-equivariant
map
\[
\underline{\mathbf{d}}:\mathbf{B}(V)\times\mathbf{B}(V)\rightarrow\mathbb{R}_{\leq}^{n}
\]
whose $i$-th component $\mathbf{d}_{i}:\mathbf{B}(V)\times\mathbf{B}(V)\rightarrow\mathbb{R}$
is given by
\[
\mathbf{d}_{i}(\alpha,\beta)=-\log\left(\max\left\{ \min\left\{ \frac{\beta(x)}{\alpha(x)}:x\in W\setminus\{0\}\right\} :W\subset V,\,\dim_{\mathbb{R}}W=i\right\} \right).
\]
Suppose that $\alpha,\beta\in\mathbf{B}(\mathcal{S})\cap\mathbf{B}(\mathcal{S}')$
and choose $\mathbb{R}$-basis $e=(e_{i})_{i=1}^{n}$ and $e'=(e'_{i})_{i=1}^{n}$
of $V$ such that $\mathcal{S}=\{\mathbb{R}e_{i}:i=1,\cdots,n\}$,
$\mathcal{S}'=\{\mathbb{R}e'_{i}:i=1,\cdots,n\}$, $e$ and $e'$
are orthonormal for $\alpha$, and $\beta(e_{1})\geq\cdots\geq\beta(e_{n})$,
$\beta(e'_{1})\geq\cdots\geq\beta(e'_{n})$. Then necessarily 
\[
\forall i\in\{1,\ldots,n\}:\qquad\beta(e_{i})=\exp\left(-\mathbf{d}_{i}(\alpha,\beta)\right)=\beta(e'_{i})
\]
The element $g\in\mathsf{G}$ mapping $e$ to $e'$ satisfies $g\mathcal{S}=\mathcal{S}'$,
$g\alpha=\alpha$ and $g\beta=\beta$, which proves $R(i)$. The resulting
vectorial distance $\mathbf{d}$ equals $\underline{\mathbf{d}}$
under the identification $\mathbf{C}(G)\simeq\mathbb{R}_{\leq}^{n}$,
i.e.~for every $\alpha\in\mathbf{B}(V)$ and $\mathcal{F}\in\mathbf{F}(V)$,
$\underline{\mathbf{d}}(\alpha,\alpha+\mathcal{F})=\underline{t}(\mathcal{F}).$
Indeed for any $X\in\mathbf{Sym}(V,\alpha)$ and $\beta=\exp(X)\cdot\alpha$,
we have 
\[
\underline{\mathbf{d}}(\alpha,\beta)=(\gamma_{1},\cdots,\gamma_{n})\quad\mbox{in}\quad\mathbb{R}_{\leq}^{n}
\]
where $\gamma_{1}\leq\cdots\leq\gamma_{n}$ are the eigenvalues of
$X$ counted with multiplicities.

\subsection{~}

Define $\mathbf{d}^{i}(\alpha,\beta)=\sum_{j=0}^{i-1}\mathbf{d}_{n-j}(\alpha,\beta)$,
so that
\begin{eqnarray*}
\mathbf{d}^{i}(\alpha,\beta) & = & \max\left\{ \mathbf{d}^{i}(\alpha\vert W,\beta\vert W):W\subset V,\,\dim_{\mathbb{R}}W=i\right\} \\
 & = & \log\max\left\{ \frac{\Lambda^{i}(\alpha)(v)}{\Lambda^{i}(\beta)(v)}:v\in\Lambda^{i}(V)\setminus\{0\}\right\} 
\end{eqnarray*}
where $\Lambda^{i}(\alpha)$ is the Euclidean norm on $\Lambda^{i}(V)$
induced by $\alpha$. We have 
\[
\mathbf{d}^{n}(\alpha,\beta)=\log\left(\frac{\int_{\beta(v)\leq1}dv}{\int_{\alpha(v)\leq1}dv}\right)
\]
for any Borel measure $dv$ on $V$, therefore 
\begin{eqnarray*}
\mathbf{d}^{n}(\alpha,\gamma) & = & \mathbf{d}^{n}(\alpha,\beta)+\mathbf{d}^{n}(\beta,\gamma),\\
\mathbf{d}^{n}(\alpha,g\alpha) & = & \log\left|\det(g)\right|,\\
\mathbf{d}^{n}(\alpha,\alpha+\mathcal{F}) & = & {\textstyle \sum}_{\gamma}\gamma\dim_{\mathbb{R}}\Gr_{\mathcal{F}}^{\gamma}.
\end{eqnarray*}
In particular, if $\mathbf{d}^{i}(\alpha,\gamma)=\mathbf{d}^{i}(\alpha\vert W,\gamma\vert W)$
for some $W\subset V$, $\dim_{\mathbb{R}}W=i$, then
\[
\mathbf{d}^{i}(\alpha,\gamma)=\mathbf{d}^{i}(\alpha\vert W,\beta\vert W)+\mathbf{d}^{i}(\beta\vert W,\gamma\vert W)\leq\mathbf{d}^{i}(\alpha,\beta)+\mathbf{d}^{i}(\beta,\gamma)
\]
i.e.~$\underline{\mathbf{d}}$ satisfies the triangle inequality
$TR$.

\subsection{~}

We next show that for any $\alpha\in\mathbf{B}(V)$ and $\mathcal{F},\mathcal{G}\in\mathbf{F}(V)$,
\[
2\cdot\underline{\mathbf{d}}(\alpha+\mathcal{F},\alpha+\mathcal{G})\leq\underline{\mathbf{d}}(\alpha+2\mathcal{F},\alpha+2\mathcal{G})\quad\mbox{in}\quad\mathbb{R}_{\leq}^{n}.
\]
Put $f=g_{\alpha}(\mathcal{F})$, $g=g_{\alpha}(\mathcal{G})$. Then
$f^{2}=g_{\alpha}(2\mathcal{F})$, $g^{2}=g_{\alpha}(2\mathcal{G})$
and we have to show
\[
2\cdot\underline{\mathbf{d}}(f\alpha,g\alpha)\leq\underline{\mathbf{d}}(f^{2}\alpha,g^{2}\alpha).
\]
Let $h\mapsto h^{\ast}$ be the involution of $\mathsf{G}$ defined
by $\alpha$, so that $f^{\ast}=f$, $g^{\ast}=g$ and $f^{-2}g^{2}$
is conjugated to $gf^{-2}g=(gf^{-1})(gf^{-1})^{\ast}$. For $1\leq i\leq n$
and $h\in\mathsf{G}$, write $\lambda_{i}(h)$ for the largest real
eigenvalue of $h$ acting on $\Lambda^{i}(V)$ and denote by $\left\langle -,-\right\rangle _{\alpha,i}$
the scalar product on $\Lambda^{i}(V)$ attached to its Euclidean
norm $\Lambda^{i}(\alpha)$. Then
\begin{eqnarray*}
\exp\left(\mathbf{d}^{i}(f^{2}\alpha,g^{2}\alpha)\right) & = & \max\left\{ \frac{\Lambda^{i}(\alpha)(f^{-2}v)}{\Lambda^{i}(\alpha)(g^{-2}v)}:v\in\Lambda^{i}(V)\setminus\{0\}\right\} \\
 & = & \max\left\{ \frac{\Lambda^{i}(\alpha)(f^{-2}g^{2}v)}{\Lambda^{i}(\alpha)(v)}:v\in\Lambda^{i}(V)\setminus\{0\}\right\} \\
 & \geq & \lambda_{i}(f^{-2}g^{2})\quad=\quad\lambda_{i}(gf^{-2}g)\\
 & = & \max\left\{ \frac{\left\langle gf^{-2}gx,x\right\rangle _{\alpha,i}}{\left\langle x,x\right\rangle _{\alpha,i}}:x\in\Lambda^{i}(V)\setminus\{0\}\right\} \\
 & = & \max\left\{ \frac{\left\langle f^{-1}x,f^{-1}x\right\rangle _{\alpha,i}}{\left\langle g^{-1}x,g^{-1}x\right\rangle _{\alpha,i}}:x\in\Lambda^{i}(V)\setminus\{0\}\right\} \\
 & = & \exp\left(2\mathbf{d}^{i}(f\alpha,g\alpha)\right)
\end{eqnarray*}
with equality for $i=n$, which proves our claim. Thus $\underline{\mathbf{d}}$
satisfies $CO''$ and $CO$.

\subsection{~\label{sub:A5forEuclideanNorms}}

For $\alpha\in\mathbf{B}(V)$ and $\mathcal{F}\in\mathbf{F}(V)$,
we denote by $\Gr_{\mathcal{F}}(\alpha)$ the Euclidean norm on $\Gr_{\mathcal{F}}^{\bullet}(V)$
induced by $\alpha$ through the isomorphism $V\simeq\Gr_{\mathcal{F}}(V)$
provided by the $\alpha$-orthogonal splitting $\mathcal{G}_{\alpha}(\mathcal{F})$
of $\mathcal{F}$. We claim that for every $\alpha,\beta\in\mathbf{B}(V)$,
\[
\lim_{t\rightarrow\infty}\underline{\mathbf{d}}(\alpha+t\mathcal{F},\beta+t\mathcal{F})=\underline{\mathbf{d}}(\Gr_{\mathcal{F}}(\alpha),\Gr_{\mathcal{F}}(\beta))\quad\mbox{in}\quad\mathbb{R}_{\leq}^{n}.
\]
Indeed, choosing an isomorphism $\left(\mathcal{G}_{\alpha}(\mathcal{F})_{\gamma},\alpha\vert\mathcal{G}_{\alpha}(\mathcal{F})_{\gamma}\right)\simeq\left(\mathcal{G}_{\beta}(\mathcal{F})_{\gamma},\beta\vert\mathcal{G}_{\beta}(\mathcal{F})_{\gamma}\right)$
for every $\gamma\in\mathbb{R}$, we obtain an element $g\in\mathsf{G}$
which fixes $\mathcal{F}$ and maps $\alpha$ to $\beta$. It then
also maps $\alpha+t\mathcal{F}=g_{\alpha}(t\mathcal{F})\cdot\alpha$
to $\beta+t\mathcal{F}=gg_{\alpha}(t\mathcal{F})\cdot\alpha$, so
that 
\[
\underline{\mathbf{d}}(\alpha+t\mathcal{F},\beta+t\mathcal{F})=\underline{\mathbf{d}}\left(\alpha,g_{\alpha}^{-1}(t\mathcal{F})gg_{\alpha}(t\mathcal{F})\cdot\alpha\right).
\]
Let $L_{\alpha}(\mathcal{F})$ be the centralizer of $\mathcal{G}_{\alpha}(\mathcal{F})$,
so that $P_{\mathcal{F}}=U_{\mathcal{F}}\rtimes L_{\alpha}(\mathcal{F})$.
Write $g=u\cdot\ell$ with $u\in\mathsf{U}_{\mathcal{F}}$ and $\ell\in\mathsf{L}_{\alpha}(\mathcal{F})$,
so that $g_{\alpha}^{-1}(t\mathcal{F})gg_{\alpha}(t\mathcal{F})=g_{\alpha}^{-1}(t\mathcal{F})ug_{\alpha}(t\mathcal{F})\cdot\ell$.
Let then $\mathfrak{u}_{\mathcal{F}}=\oplus_{\gamma>0}\mathfrak{u}_{\gamma}$
be the weight decomposition of $\mathfrak{u}_{\mathcal{F}}=\Lie(U_{\mathcal{F}})(\mathbb{R})$
induced by 
\[
\mathrm{ad}\circ\mathcal{G}_{\alpha}(\mathcal{F}):\mathbb{D}(\mathbb{R})\rightarrow G\rightarrow GL(\mathfrak{g})
\]
where $\mathfrak{g}=\Lie(G)(\mathbb{R})$. Then $g_{\alpha}(t\mathcal{F})$
acts on $\mathfrak{u}_{\gamma}$ by $\exp(t\gamma)$, from which easily
follows that $g_{\alpha}^{-1}(t\mathcal{F})ug_{\alpha}(t\mathcal{F})$
converges to $1$ in $\mathsf{U}_{\mathcal{F}}$ (for the real topology).
It follows that
\[
\lim_{t\rightarrow\infty}\underline{\mathbf{d}}(\alpha+t\mathcal{F},\beta+t\mathcal{F})=\underline{\mathbf{d}}\left(\alpha,\ell\alpha\right)=\underline{\mathbf{d}}\left(\Gr_{\mathcal{F}}(\alpha),\Gr_{\mathcal{F}}(\beta)\right).
\]
Taking $\beta=u\alpha$ with $u\in\mathsf{U}_{\mathcal{F}}$, we obtain
$UN$. On the other hand for any $\beta$, since 
\[
\mathbb{R}_{+}\ni t\mapsto\underline{\mathbf{d}}(\alpha+t\mathcal{F},\beta+t\mathcal{F})\in\mathbb{R}_{\leq}^{n}
\]
is convex and bounded, it is non-increasing, which proves $NE$.

\subsection{~\label{sub:B(V)isCAT(0)}}

Let $d=\left\Vert \underline{\mathbf{d}}\right\Vert $ be the $\mathsf{G}$-invariant
distance on $\mathbf{B}(V)$ attached to the faithful representation
$V$ of $\mathsf{G}$, as in~\ref{sub:The-classical-distance}. We
claim that the metric space $(\mathbf{B}(V),d)$ is $CAT(0)$. In
particular, it is uniquely geodesic, thus $\mathbf{B}(V)$ also satisfies
$UG$. To establish our claim, fix $\alpha\in\mathbf{B}(V)$, choose
an $\alpha$-orthonormal basis $(e_{i})_{i=1}^{n}$ of $V$ and use
it to identify $\mathsf{G}$ with $GL(n,\mathbb{R})$, $\mathbf{Sym}(V,\alpha)$
with the vector space $S(n,\mathbb{R})$ of symmetric matrices in
$M(n,\mathbb{R})$ and $\mathsf{G}(\alpha)$ with the open cone $P(n,\mathbb{R})\subset S(n,\mathbb{R})$
of positive definite matrices. Let $\left\langle -,-\right\rangle $
be the scalar product on $V$ attached to $\alpha$ and $g\mapsto g^{\ast}$
the corresponding involution of $\mathsf{G}$. For $p\in\mathsf{G}(\alpha)$,
$g\in\mathsf{G}$ and $v\in V$, set $g\cdot p=gpg^{\ast}$ and $\alpha_{p}(v)=\left\langle pv,v\right\rangle ^{1/2}$.
This defines an action of $\mathsf{G}$ on $\mathsf{G}(\alpha)$ and
$p\mapsto\alpha_{p}$ is an isomorphism of differentiable manifold
$\mathsf{G}(\alpha)\rightarrow\mathbf{B}(V)$ such that 
\[
\alpha_{g\cdot p}=(g^{\ast})^{-1}\cdot\alpha_{p}\quad\mbox{and}\quad\alpha+\mathcal{F}=\alpha_{g_{\alpha}(\mathcal{F})^{-2}}\quad\mbox{in}\quad\mathbf{B}(V)
\]
for any $g\in\mathsf{G}$, $p\in\mathsf{G}(\alpha)$ and $\mathcal{F}\in\mathbf{F}(V)$.
In \cite[II.10.31]{BrHa99}, $\mathsf{G}(\alpha)$ is equipped with
a $\mathsf{G}$-invariant Riemannian structure. Let $d_{\alpha}$
be the corresponding $\mathsf{G}$-invariant Riemannian metric on
$\mathsf{G}(\alpha)$ or $\mathbf{B}(V)$. For $X\in\mathbf{Sym}(V,\alpha)$
and $p=\exp(X)\in\mathsf{G}(\alpha)$, 
\[
d_{\alpha}^{2}(\alpha,\alpha_{p})=\mathrm{Tr}(X^{2})
\]
by \cite[II.10.42.(2)]{BrHa99}. Thus for any $\mathcal{F}\in\mathbf{F}(V)$,
\[
d_{\alpha}^{2}(\alpha,\alpha+\mathcal{F})=4\mathrm{Tr}\left(\left(\mathcal{G}_{\alpha}(\mathcal{F})^{\flat}\right)^{2}\right)=4\left\Vert \mathcal{F}\right\Vert ^{2}=4d^{2}(\alpha,\alpha+\mathcal{F})
\]
since $\alpha+\mathcal{F}=\alpha_{p}$ with $p=\exp(-2\mathcal{G}_{\alpha}(\mathcal{F})^{\flat})$
in $\mathsf{G}(\alpha)$. Therefore $d_{\alpha}(\alpha,\beta)=2d(\alpha,\beta)$
for any $\beta\in\mathbf{B}(V)$ by $R(s)$ and $d_{\alpha}=2d$ on
$\mathbf{B}(V)$ since $\mathsf{G}$ acts transitively on $\mathbf{B}(V)$.
Since the metric space $(\mathbf{B}(V),d_{\alpha})$ is $CAT(0)$
by \cite[II.10.39]{BrHa99}, so is $(\mathbf{B}(V),d)$.

\subsection{~}

We have thus established that $\left(\mathbf{B}(V),\mathbf{B}(-),+\right)$
is an affine $\mathbf{F}(V)$-building. If $S\in\mathbf{S}(G)$ corresponds
to $\mathcal{S}\in\mathbf{S}(V)$, the type map $\nu_{\mathbf{B},S}:\mathsf{S}\rightarrow\mathbf{G}(S)$
maps $s\in\mathsf{S}$ to the unique morphism $\mathbb{D}_{\mathbb{R}}(\mathbb{R})\rightarrow S$
whose composite with the character $\chi_{L}$ through which $S$
acts on $L\in\mathcal{S}$ is the character $\log\left|\chi_{L}(s)\right|\in\mathbb{R}$
of $\mathbb{D}(\mathbb{R})$.

\subsection{~}

The computations of section~\ref{sub:A5forEuclideanNorms} show that
$\Gr_{\mathcal{F}}$ induces an isometry
\[
\Gr_{\mathcal{F}}:\mathbf{T}_{P_{\mathcal{F}}}^{\infty}\mathbf{B}(V)\simeq{\textstyle \prod_{\gamma}}\mathbf{B}\left(\Gr_{\mathcal{F}}^{\gamma}(V)\right).
\]
On the other hand, the tangent space $\mathbf{T}_{\alpha}\mathbf{B}(V)$
as defined in section~\ref{sub:TangentSpaceInBuild} is equal to
the corresponding tangent space $\mathbf{Sym}(V,\alpha)$ of the differential
manifold $\mathbf{B}(V)$, its scalar product is given by $\left\langle X,Y\right\rangle =\mathrm{Tr}(XY)$
and the localization map
\[
\loc_{\alpha}:\mathbf{F}(V)\twoheadrightarrow\mathbf{T}_{\alpha}\mathbf{B}(V)
\]
maps $\mathcal{F}\in\mathbf{F}(V)$ to the $\alpha$-symmetric endomorphism
\[
\loc_{\alpha}(\mathcal{F})=\left(\frac{d}{dt}g_{\alpha}(t\mathcal{F})\right)_{t=0}=\mathcal{G}_{\alpha}(\mathcal{F})^{\flat}\quad\mbox{in}\quad\mathbf{Sym}(V,\alpha).
\]

\chapter{Bruhat-Tits buildings}

We first keep the assumptions and notations of the previous chapter.
Thus $G$ will be a reductive group over a field $K$, $\mathsf{G}=G(K)$
and $\Gamma=(\mathbb{R},+,\leq)$. In addition, we assume that $K$
is equipped with a non-trivial, non-archimedean absolute value
\[
\left|-\right|:K\rightarrow\mathbb{R}^{+}.
\]
However, we will eventually return to the setting of chapter~\ref{chap:VectTitsBuild},
with $G$ a reductive group over the valuation ring $\mathcal{O}_{K}=\{x\in K:\left|x\right|\leq1\}$
of $K$ and $\Gamma=\mathbb{R}$. Note that then $\mathbf{P}(G)=\mathbf{P}(G_{K})$,
$\mathbf{F}(G)=\mathbf{F}(G_{K})$ but $\mathbf{S}(G)\subsetneq\mathbf{S}(G_{K})$
and $\mathbf{G}(G)\subsetneq\mathbf{G}(G_{K})$.

\section{The Bruhat-Tits building of $GL(V)$\label{sub:BruhatGL(V)}}

\subsection{~}

Let $G=GL(V)$, where $V\neq0$ is a $K$-vector space of dimension
$n\in\mathbb{N}$. As in section~\ref{sub:ExampSymSpac}, we thus
have $\mathsf{G}$-equivariant bijections
\[
\begin{array}{rcccl}
\mathbf{S}(G) & \simeq & \mathbf{S}(V) & = & \left\{ \mathcal{S}\subset\mathbb{P}^{1}(V)(K):V=\oplus_{L\in\mathcal{S}}L\right\} ,\\
\mathbf{F}(G) & \simeq & \mathbf{F}(V) & = & \left\{ \mathbb{R}-\mbox{filtrations on }V\right\} ,\\
\mathbf{G}(G) & \simeq & \mathbf{G}(V) & = & \left\{ \mathbb{R}-\mbox{graduations on }V\right\} .
\end{array}
\]

\subsection{~}

A $K$-norm (or simply: norm) on $V$ is a function $\alpha:V\rightarrow\mathbb{R}^{+}$
such that 
\begin{enumerate}
\item $\alpha(v)=0$ if and only if $v=0$,
\item $\alpha(\lambda v)=\left|\lambda\right|\alpha(v)$ for every $\lambda\in K$
and $v\in V$, and 
\item $\alpha(u+v)\leq\max\left\{ \alpha(u),\alpha(v)\right\} $ for every
$u,v\in V$. 
\end{enumerate}
The $K$-norm $\alpha$ is split by $\mathcal{S}\in\mathbf{S}(V)$
if and only if 
\[
\forall v\in V:\quad\alpha(v)=\max\left\{ \alpha(v_{L}):L\in\mathcal{S}\right\} \quad\mbox{where }v={\textstyle \sum}_{L\in\mathcal{S}}v_{L},\, v_{L}\in L.
\]
It is splittable if it is split by $\mathcal{S}$ for some $\mathcal{S}\in\mathbf{S}(V)$.
If $K$ is locally compact, every $K$-norm on $V$ is splittable
by~\cite[Proposition 1.1]{GoIw63}.

\subsection{~}

We denote by $\mathbf{B}(V)$ the set of all splittable $K$-norms
on $V$, by $\mathbf{B}(\mathcal{S})$ the subset of all $K$-norms
split by $\mathcal{S}$. We let $\mathsf{G}$ act on $\mathbf{B}(V)$
by $(g\cdot\alpha)(v)=\alpha(g^{-1}v)$, and define the pull map $+:\mathbf{B}(V)\times\mathbf{F}(V)\rightarrow\mathbf{B}(V)$
by 
\[
(\alpha+\mathcal{F})(v)=\min\left\{ \max\left\{ e^{-\gamma}\alpha(v_{\gamma}):\gamma\in\mathbb{R}\right\} :v={\textstyle \sum_{\gamma\in\mathbb{R}}}v_{\gamma},\, v_{\gamma}\in\mathcal{F}^{\gamma}\right\} 
\]
where the sums $\sum_{\gamma\in\mathbb{R}}v_{\gamma}$ have finite
support. We have to verify that this operation is well-defined. Note
first that the axiom $L(s)$ follows from the second proof of \cite[1.5.ii]{BrTi84b}:
for any $\alpha\in\mathbf{B}(V)$ and $\mathcal{F}\in\mathbf{F}(V)$,
there is an $\mathcal{S}\in\mathbf{S}(V)$ with $\alpha\in\mathbf{B}(\mathcal{S})$
and $\mathcal{F}\in\mathbf{F}(\mathcal{S})$. Let us then identify
$\mathbf{F}(\mathcal{S})$ with $\mathbb{R}^{\mathcal{S}}$ by $\mathcal{F}\mapsto\mathcal{F}^{\sharp}$
where
\[
\mathcal{F}^{\sharp}(L)=\max\left\{ \gamma\in\mathbb{R}:L\subset\mathcal{F}^{\gamma}\right\} ,\qquad\mathcal{F}^{\gamma}=\oplus_{L:\mathcal{F}^{\sharp}(L)\geq\gamma}L.
\]
Then for $v=\sum_{L\in\mathcal{S}}v_{L}$ in $V=\oplus_{L\in\mathcal{S}}L$,
we find that
\[
\inf\left\{ \max\left\{ e^{-\gamma}\alpha(v_{\gamma}):\gamma\in\mathbb{R}\right\} \left|\begin{smallmatrix}v=\sum_{\gamma}v_{\gamma}\\
v_{\gamma}\in\mathcal{F}^{\gamma}
\end{smallmatrix}\right.\right\} =\max\left\{ e^{-\mathcal{F}^{\sharp}(L)}\alpha(v_{L}):L\in\mathcal{S}\right\} .
\]
Indeed for $v=\sum_{\gamma}v_{\gamma}$ with $v_{\gamma}=\sum_{L}v_{\gamma,L}$,
$v_{\gamma,L}\in L$ and $v_{\gamma,L}=0$ if $\gamma>\mathcal{F}^{\sharp}(L)$,
\begin{eqnarray*}
\max\left\{ e^{-\gamma}\alpha(v_{\gamma}):\gamma\in\mathbb{R}\right\}  & = & \max\left\{ e^{-\gamma}\alpha(v_{\gamma,L}):\gamma\in\mathbb{R},\, L\in\mathcal{S}\right\} \\
 & \geq & \max\left\{ e^{-\mathcal{F}^{\sharp}(L)}\alpha(v_{\gamma,L}):\gamma\in\mathbb{R},\, L\in\mathcal{S}\right\} \\
 & \geq & \max\left\{ e^{-\mathcal{F}^{\sharp}(L)}\alpha(v_{L}):L\in\mathcal{S}\right\} 
\end{eqnarray*}
since $\alpha\in\mathbf{B}(\mathcal{S})$ (for the first equality)
and $v_{L}=\sum_{\gamma}v_{\gamma,L}$ (for the last inequality),
which provides the non-trivial required inequality in the displayed
formula. Thus
\begin{equation}
(\alpha+\mathcal{F})(v)=\max\left\{ e^{-\mathcal{F}^{\sharp}(L)}\alpha(v_{L}):L\in\mathcal{S}\right\} \label{eq:formulaAlpha+Ftori}
\end{equation}
from which follows that $\alpha+\mathcal{F}$ is well-defined and
again belongs to $\mathbf{B}(\mathcal{S})$.

\subsection{~}

The apartment and pull maps are plainly $\mathsf{G}$-equivariant,
and the above formula shows that the latter turns $\mathbf{B}(\mathcal{S})$
into an affine $\mathbf{F}(\mathcal{S})$-space, thus 
\[
\mathbf{B}(V)=(\mathbf{B}(V),+,\mathbf{B}(-))
\]
 is an affine $\mathbf{F}(G)$-space. If $S\in\mathbf{S}(G)$ corresponds
to $\mathcal{S}\in\mathbf{S}(V)$, the type map 
\[
\nu_{\mathbf{B},S}:\mathsf{S}\rightarrow\mathbf{G}(S)
\]
maps $s$ to the unique $\mathcal{F}\in\mathbf{F}(S)$ with $\gamma_{L}(\mathcal{F})=\log\left|\chi_{L}(s)\right|$
for all $L\in\mathcal{S}$, where $\chi_{L}:S\rightarrow\mathbb{G}_{m,k}$
is the character through which $S$ acts on $L$.

\subsection{~}

In~\cite[\S 3]{Pa99}, Parreau shows that the closely related set
$\Delta=\mathbb{R}_{+}^{\times}\backslash\mathbf{B}(V)$ is an affine
building in the sense of \cite[1.1]{Pa99} (see also \cite{BrTi84b,GoIw63}).
The axioms $R(s)$, $R(i)^{+}$, $HA$ and $L(s)^{+}$ for $\mathbf{B}(V)$
respectively follow from the axioms $A3$, $A2$, $A5$ and proposition~$1.8$
for $\Delta$ in \cite{Pa99}. For $\alpha\in\mathbf{B}(V)$, $\mathcal{F}\in\mathbf{F}(V)$
and $u\in\mathsf{U}_{\mathcal{F}}$, pick $\mathcal{S}\in\mathbf{S}(V)$
such that $\alpha\in\mathbf{B}(\mathcal{S})$ and $\mathcal{F}\in\mathbf{F}(\mathcal{S})$
using $L(s)$. Write $\mathcal{S}=\{Kv_{1},\cdots,Kv_{n}\}$ with
$i\mapsto\gamma_{i}=\mathcal{F}^{\sharp}(Kv_{i})$ non-increasing
and identify $\mathbf{B}(\mathcal{S})$ with $\mathbb{R}^{n}$ by
$\alpha\mapsto(\alpha_{1},\cdots,\alpha_{n})$ where $\alpha_{i}=-\log(\alpha(v_{i}))$.
Then for $t\geq0$, $t\mathcal{F}\in\mathbf{F}(\mathcal{S})$ acts
on $\mathbf{B}(\mathcal{S})\simeq\mathbb{R}^{n}$ by 
\[
(\alpha_{1},\cdots,\alpha_{n})\mapsto(\alpha_{1}+t\gamma_{1},\cdots,\alpha_{n}+t\gamma_{n})
\]
and the matrix $(u_{i,j})$ of $u\in\mathsf{U_{\mathcal{F}}}$ in
the basis $(v_{1},\cdots,v_{n})$ of $V$ satisfies $u_{i,i}=1$ and
$u_{i,j}\neq0$ if and only if $\gamma_{i}>\gamma_{j}$ for $i\neq j$.
Moreover, $u$ fixes $\alpha$ if and only if $\alpha_{j}-\alpha_{i}\leq-\log\left|u_{i,j}\right|$
for all $1\leq i,j\leq n$ by~\cite[3.5]{Pa99}. Therefore $u$ fixes
$\alpha+t\mathcal{F}$ for all $t\gg0$, i.e. $\mathbf{B}(V)$ satisfies
$UN^{+}$. Thus by proposition~\ref{prop:L(s)+implBuild}, $\mathbf{B}(V)$
is an affine $\mathbf{F}(G)$-building whose underlying metric space
is $CAT(0)$.

\subsection{~}

Let $Z\simeq\mathbb{G}_{m,K}$ be the center of $G$, so that $\mathbf{G}(Z)\simeq\mathbb{R}$
by the isomorphism which maps $\mathcal{G}:\mathbb{D}_{K}(\mathbb{R})\rightarrow Z$
to the unique weight $\mathcal{G}^{\sharp}\in\mathbb{R}$ of the corresponding
representation of $\mathbb{D}_{K}(\mathbb{R})$ on $V$. For $S\in\mathbf{S}(G)$
corresponding to $\mathcal{S}\in\mathbf{S}(V)$, the projection from
$\mathbf{F}(S)=\mathbf{G}(S)$ to $\mathbf{G}(Z)$ then maps $\mathcal{F}\in\mathbf{F}(S)$
to the unique $\mathcal{G}$ with 
\[
\mathcal{G}^{\sharp}={\textstyle \frac{1}{n}\sum_{L\in\mathcal{S}}}\mathcal{F}^{\sharp}(L).
\]
It follows that the projection 
\[
\mathbf{d}^{c}:\mathbf{B}(V)\times\mathbf{B}(V)\rightarrow\mathbf{G}(Z)
\]
of the distance $\mathbf{d}:\mathbf{B}(V)\times\mathbf{B}(V)\rightarrow\mathbf{C}(G)$
maps $(\alpha,\beta)$ to the unique $\mathcal{G}$ with 
\[
\mathcal{G}^{\sharp}={\textstyle \frac{1}{n}\sum_{i=1}^{n}}\log\alpha(v_{i})-\log\beta(v_{i})
\]
for any $K$-basis $(v_{1},\cdots,v_{n})$ of $V$ such that $\alpha,\beta\in\mathbf{B}(\mathcal{S})$
with $\mathcal{S}=\{Kv_{1},\cdots,Kv_{n}\}$. From~\cite[3.2]{Pa99},
we then deduce that the morphism 
\[
\nu_{\mathbf{B}}^{c}:\mathsf{G}\rightarrow\mathbf{G}(Z)
\]
maps $g$ to the unique $\mathcal{G}$ with $\mathcal{G}^{\sharp}=\frac{1}{n}\log\left|\det g\right|$.
In particular, $\left|\det\mathsf{G}_{\alpha}\right|=1$ for every
$\alpha\in\mathbf{B}(V)$, and then \cite[3.5]{Pa99} implies $ST$:
$\mathsf{G}_{\alpha}=\mathsf{G}_{S,\alpha}$ for all $\alpha\in\mathbf{B}(S)$,
$S\in\mathbf{S}(G)$. Therefore $\mathbf{B}(V)$ is a tight affine
$\mathbf{F}(G)$-building.

\subsection{~}

If the valuation of $K$ is discrete, the map 
\[
\alpha\mapsto L=\{x\in V:\alpha(x)\leq1\}
\]
identifies the subset $\mathbf{B}^{\circ}(V)\subset\mathbf{B}(V)$
of $K$-norms $\alpha$ on $V$ such that 
\[
\alpha(V\setminus\{0\})=\left|K^{\times}\right|
\]
with the set $\mathcal{L}(V)$ of $\mathcal{O}_{K}$-lattices in $V$.
Then $\mathbf{B}^{\circ}(V)$ is stable under the action of $\mathsf{G}$
and of the subset $\mathbf{F}^{\log\left|K^{\times}\right|}(V)\subset\mathbf{F}(V)$
of $\log\left|K^{\times}\right|$-filtrations on $V$. It is then
convenient to either normalize the valuation by requiring that $\left|K^{\times}\right|=e^{\mathbb{Z}}$,
or to rescale the pull map as in~\ref{sub:RenormalizationOfBuildings}.
Then $\mathbf{F}^{\mathbb{Z}}(V)$ acts on $\mathbf{B}^{\circ}(V)\simeq\mathcal{L}(V)$
by 
\[
L+\mathcal{F}={\textstyle \sum_{i\in\mathbb{Z}}}\pi^{-i}L\cap\mathcal{F}^{i}
\]
for $L\in\mathcal{L}(V)$ and $\mathcal{F}\in\mathbf{F}^{\mathbb{Z}}(V)$,
where $\pi\in\mathcal{O}_{K}$ is any uniformizer.

\subsection{~}

The space $\mathbf{B}(V)$ is known to be a realization of the Bruhat-Tits
building of $G$; for a more general case, see \cite{BrTi84b}.

\section{The Bruhat-Tits building of $G$\label{sub:BTbuildIsAffine}}

\subsection{~\label{sub:Assumption4ExistBT}}

For a general reductive group $G$ over $K$, we have to make some
assumption on the triple $(G,K,\left|-\right|)$: the existence of
a \emph{valuation} on the root datum $(\mathsf{Z}_{G}(S),(\mathsf{U}_{a})_{a\in\Phi(G,S)})$
of $\mathsf{G}=G(K)$, in the sense of \cite[6.2.1]{BrTi72}. Here
$S$ is a fixed element of $\mathbf{S}(G)$ and the notations are
taken from section~\ref{sub:NotationsForRoots}.

\subsection{~}

Let then $\mathbf{B}^{r}(G)$ and $\mathbf{B}^{e}(G)=\mathbf{B}^{r}(G)\times\mathbf{G}(Z)$\nomenclature[B^r(G)]{$\mathbf{B}^r (G)$}{Reduced Bruhat-Tits building of $G$, page \nomrefpage}\nomenclature[B^e(G)]{$\mathbf{B}^e (G)$}{Extended Bruhat-Tits building of $G$, page \nomrefpage}
be respectively the reduced and extended Bruhat-Tits buildings of
$G$, as defined in \cite[\S 7]{BrTi72} and \cite[4.2.16 \& 5.1.29]{BrTi84}.
These two sets have compatible actions of $\mathsf{G}$, they are
covered by apartments $\mathbf{B}^{r}(S)$ and $\mathbf{B}^{e}(S)=\mathbf{B}^{r}(S)\times\mathbf{G}(Z)$
which are $\mathsf{G}$-equivariantly parametrized by $\mathbf{S}(G)$,
$\mathbf{B}^{e}(S)$ is an affine $\mathbf{F}(S)$-space on which
$\mathsf{N}_{G}(S)$ acts by affine transformations with linear part
$\nu_{S}^{v}:\mathsf{N}_{G}(S)\twoheadrightarrow\mathsf{W}_{G}(S)$
and the resulting action of $\mathsf{Z}_{G}(S)$ is given by a morphism
$\nu_{\mathbf{B},S}:\mathsf{Z}_{G}(S)\rightarrow\mathbf{G}(S)$ which
is uniquely characterized by the following property: for every morphism
$\chi:Z_{G}(S)\rightarrow\mathbb{G}_{m,K}$, the induced morphism
\[
\mathbf{G}(\chi\vert_{S})\circ\nu_{\mathbf{B},S}:\mathsf{Z}_{G}(S)\longrightarrow\mathbf{G}(\mathbb{G}_{m,K})
\]
maps $z$ in $\mathsf{Z}_{G}(S)$ to $\log\left|\chi(z)\right|$ in
$\mathbb{R}=\mathbf{G}(\mathbb{G}_{m,K})$. Similarly, the action
of $\mathsf{G}$ on $\mathbf{G}(Z)$ is given by a morphism $\nu_{\mathbf{B}}^{c}:\mathsf{G}\rightarrow\mathbf{G}(Z)$
which is uniquely characterized by the following property: for every
morphism $\chi:G\rightarrow\mathbb{G}_{m,K}$, the induced morphism
\[
\mathbf{G}(\chi\vert Z)\circ\nu_{\mathbf{B}}^{c}:\mathsf{G}\rightarrow\mathbf{G}(\mathbb{G}_{m,K})
\]
maps $g$ in $\mathsf{G}$ to $\log\left|\chi(g)\right|$ in $\mathbb{R}=\mathbf{G}(\mathbb{G}_{m,K})$.
There is a $\mathsf{G}$-invariant distance 
\[
d:\mathbf{B}^{e}(G)\times\mathbf{B}^{e}(G)\rightarrow\mathbb{R}^{+}
\]
inducing a Euclidean distance on each apartment, which turns $\mathbf{B}^{e}(G)$
into a CAT(0)-space. Finally, $\mathbf{B}^{e}(G)$ already satisfies
our axiom $R(s)$ by \cite[7.4.18.i]{BrTi72} as well as the following
strengthening of $ST$ and $R(i)^{+}$:
\begin{quote}
For every subset $\Omega\neq\emptyset$ of $\mathbf{B}^{e}(S)$, the
pointwise stabilizer $\mathsf{G}_{\Omega}\subset\mathsf{G}$ of $\Omega$
equals $\mathsf{G}_{S,\Omega}$ by \cite[7.4.4]{BrTi72}, and it acts
transitively on the set of apartments containing $\Omega$ by \cite[7.4.9]{BrTi72}.
\end{quote}
We denote by $+_{S}:\mathbf{B}^{e}(S)\times\mathbf{F}(S)\rightarrow\mathbf{B}^{e}(S)$
the given structure of affine $\mathbf{F}(S)$-space on $\mathbf{B}^{e}(S)$.
These maps are compatible in the following sense:
\[
g\cdot(x+_{S}\mathcal{F})=(g\cdot x+_{g\cdot S}g\cdot\mathcal{F}).
\]

\subsection{~}

Let us first prove $L(s)$: starting with $x\in\mathbf{B}^{e}(G)$
and $\mathcal{F}\in\mathbf{F}(G)$, choose a minimal parabolic subgroup
$B'\subset P_{\mathcal{F}}$ with Levi $Z_{G}(S')$, pick $c\in\mathbf{B}^{e}(S')$
and form the sector $C'=c'+_{S'}F^{-1}(B')$ in $\mathbf{B}^{e}(S')$.
By \cite[7.4.18.ii]{BrTi72}, there is another apartment $\mathbf{B}^{e}(S)$
containing $x$ and a subsector $C$ of $C'$, which \emph{a priori}
is of the form $C=c+_{S}F^{-1}(B)$ for some minimal parabolic $B$
with Levi $Z_{G}(S)$. Since $C\subset\mathbf{B}^{e}(S)\cap\mathbf{B}^{e}(S')$,
there is a $g\in\mathsf{G}$ fixing $C$ and mapping $S$ to $S'$.
Then $g\in\mathsf{B}$ since $\mathsf{G}_{C}=\mathsf{G}_{S,C}\subset\mathsf{B}\cap\mathsf{G}_{c}$,
thus $\Int(g)(B)=B$ and $Z_{G}(S')\subset B$. Moreover 
\[
C=gC=g(c+_{S}F^{-1}(B))=c+_{S'}F^{-1}(B)
\]
thus actually $B=B'$ because 
\[
C=c+_{S'}F^{-1}(B)\subset c'+_{S'}F^{-1}(B')=C'
\]
in the affine $\mathbf{F}(S')$-space $\mathbf{B}^{e}(S')$. Now $x\in\mathbf{B}^{e}(S)$
and also $\mathcal{F}\in\mathbf{F}(S)$ since 
\[
Z_{G}(S)\subset B=B'\subset P_{\mathcal{F}}.
\]
This proves \textbf{$L(s)$}. Note also that for any $\mathcal{G}\in\mathbf{F}(S)\cap\mathbf{F}(S')$
in the closure of $F^{-1}(B)$, 
\[
c+_{S}\mathcal{G}=g(c+_{S}\mathcal{G})=c+_{S'}\mathcal{G}
\]
since $g$ fixes $c$, $c+_{S}\mathcal{G}$ and $\mathcal{G}$. In
particular, $c+_{S}t\mathcal{F}=c+_{S'}t\mathcal{F}$ for all $t\in\mathbb{R}^{+}$.

\subsection{~}

Suppose now that $x\in\mathbf{B}^{e}(S_{1})\cap\mathbf{B}^{e}(S_{2})$
and $\mathcal{F}\in\mathbf{F}(S_{1})\cap\mathbf{F}(S_{2})$ with $S_{1},S_{2}\in\mathbf{S}(G)$.
We now show that $x+_{S_{1}}\mathcal{F}=x+_{S_{2}}\mathcal{F}$ in
$\mathbf{B}^{e}(G)$. For $t\geq0$, put 
\[
x_{i}(t)=x+_{S_{i}}t\mathcal{F}\quad\mbox{in}\quad\mathbf{B}^{e}(S_{i}).
\]
The $CAT(0)$-property of $d$ implies that $t\mapsto d(x_{1}(t),x_{2}(t))$
is a convex function, and it is therefore sufficient to show that
it is also bounded. Let us choose minimal parabolic subgroups $Z_{G}(S_{i})\subset B_{i}\subset P_{\mathcal{F}}$,
and form the corresponding sectors 
\[
C_{i}=x+_{S_{i}}F^{-1}(B_{i})\quad\mbox{in}\quad\mathbf{B}^{e}(S_{i}).
\]
By \cite[7.4.18.iii]{BrTi72}, there is another apartment $\mathbf{B}^{e}(S)$
which contains subsectors 
\[
C'_{1}\subset C_{1}\quad\mbox{and}\quad C'_{2}\subset C_{2}.
\]
We have just seen that then $Z_{G}(S)\subset B_{i}$ (thus $\mathcal{F}\in\mathbf{F}(S)$)
and $C'_{i}=y_{i}+_{S}F^{-1}(B_{i})$ in $\mathbf{B}^{e}(S)$ for
some $y_{i}$'s in $\mathbf{B}^{e}(S)$, with moreover 
\[
y_{i}(t)=y_{i}+_{S}t\mathcal{F}=y_{i}+_{S_{i}}t\mathcal{F}
\]
for $t\geq0$. Then $t\mapsto d(y_{1}(t),y_{2}(t))$ and $t\mapsto d(x_{i}(t),y_{i}(t))$
are constant by elementary computations in $\mathbf{B}^{e}(S)$ and
$\mathbf{B}^{e}(S_{i})$ respectively, thus $t\mapsto d(x_{1}(t),x_{2}(t))$
is indeed bounded by the triangle inequality in $(\mathbf{B}^{e}(G),d)$.

\subsection{~}

We may at last define our pull map: for $x\in\mathbf{B}^{e}(G)$ and
$\mathcal{F}\in\mathbf{F}(G)$, choose $S\in\mathbf{S}(G)$ with $x\in\mathbf{B}^{e}(S)$
and $\mathcal{F}\in\mathbf{F}(S)$ and set $x+\mathcal{F}=x+_{S}\mathcal{F}$
in $\mathbf{B}^{e}(G)$: this does not depend upon the chosen $S$.
Our pull map is plainly $\mathsf{G}$-equivariant, and induces the
given structure of affine $\mathbf{F}(S)$-space on $\mathbf{B}^{e}(S)$.
Therefore 
\[
\mathbf{B}^{e}(G)=\left(\mathbf{B}^{e}(G),+,\mathbf{B}^{e}(-)\right)
\]
is an affine $\mathbf{F}(S)$-space.

\subsection{~}

For $x\in\mathbf{B}^{e}(G)$ and $\mathcal{F},\mathcal{G}\in\mathbf{F}(G)$,
choose $S\in\mathbf{S}(G)$ with $x\in\mathbf{B}^{e}(S)$, $\mathcal{F}\in\mathbf{F}(S)$,
let $F$ be the ``facet'' in $\mathbf{B}^{e}(S)$ denoted by $\gamma(x,E)$
in \cite[7.2.4]{BoTi65} with $E=\{t\mathcal{F}:t>0\}$, let $C$
be a ``chamber'' of $\mathbf{B}^{e}(S)$ containing $F$ in its
closure. Using \cite[7.4.18.ii]{BrTi72} as above, we find that there
is an apartment $\mathbf{B}^{e}(S')$ containing $C$ with $\mathcal{G}\in\mathbf{F}(S')$.
It then also contains $F$ by \cite[7.4.8]{BrTi72}, which means that
for some $\epsilon>0$, it contains $x+\eta\mathcal{G}$ for every
$\eta\in[0,\epsilon]$: this proves $L(s)^{+}$.

\subsection{~}

We already have the axioms $R(s)$, $R(i)^{+}$, $L(s)^{+}$ and $ST$,
thus $\mathbf{B}^{e}(G)$ is a tight affine $\mathbf{F}(G)$-building
by proposition~\ref{prop:L(s)+implBuild} and lemma~\ref{lem:RelationBetweenAxiomsST}.
Note that the axiom $HA$ also holds for $\mathbf{B}^{e}(G)$ by \cite[1.4]{Pa99}
and \cite[7.4.19]{BrTi72}. If $G=GL(V)$, then $\mathbf{B}^{e}(G)\simeq\mathbf{B}(V)$
by lemma~\ref{lem:UnicityTight} (see also \cite{GoIw63,BrTi84b,Pa99}).

\subsection{~}

The $CAT(0)$-distance $d$ used above may be chosen to be one of
our $d_{\tau}$'s, for some faithful representation $\tau$ of $G$.
The affine $\mathbf{F}(G)$-space $\mathbf{B}^{e}(G)$ is discrete
when $(K,\left|-\right|)$ is discrete, in which case $(\mathbf{B}^{e}(G),d)$
is a complete metric space by lemma~\ref{lem:DiscImpliesComplete}
or \cite[2.5.12]{BrTi72}. If $(K,\left|-\right|)$ is complete, then
every geodesic ray or line in $\mathbf{B}^{e}(G)$ is contained in
some apartment by \cite[2.3.8]{Ro77} and $R(s)$. Thus with the notations
of section~\ref{sub:Rays}, $\mathbf{F}(G)\simeq\mathcal{C}(\partial\mathbf{B}^{e}(G))$
if $(K,\left|-\right|)$ is complete.

\subsection{~\label{sub:changevalforBT}}

The Bruhat-Tits building $\mathbf{B}^{e}(G)=\mathbf{B}^{e}(G,\left|-\right|)$
depends upon the choice of the valuation $\left|-\right|$ on $K$,
and so does its structure of affine $\mathbf{F}(G)$-building. However
for $\nu>0$ , there is a $G(K)$-equivariant commutative diagram
\[
\begin{array}{ccccc}
\mathbf{B}^{e}(G,\left|-\right|) & \times & \mathbf{F}(G_{K}) & \stackrel{+}{\longrightarrow} & \mathbf{B}^{e}(G,\left|-\right|)\\
a\downarrow &  & b\downarrow &  & a\downarrow\\
\mathbf{B}^{e}(G,\left|-\right|^{\nu}) & \times & \mathbf{F}(G_{K}) & \stackrel{+}{\longrightarrow} & \mathbf{B}^{e}(G,\left|-\right|^{\nu})
\end{array}
\]
where $a$ is a canonical $G(K)$-equivariant map and $b(\mathcal{F})=\nu\mathcal{F}$.

\section{Functoriality for Bruhat-Tits buildings\label{sub:FunctorialityForBTBuildings}}

\subsection{~}

Suppose for this section that the valuation ring of $(K,\left|-\right|)$,
namely 
\[
\mathcal{O}_{K}=\{x\in K:\left|x\right|\leq1\}
\]
is Henselian. Then for every algebraic extension $L$ of $K$, there
is a unique absolute value $\left|-\right|:L\rightarrow\mathbb{R}^{+}$
on $L$ which extends $\left|-\right|:K\rightarrow\mathbb{R}^{+}$,
and its valuation ring 
\[
\mathcal{O}_{L}=\{x\in L:\left|x\right|\leq1\}
\]
is the integral closure of $\mathcal{O}_{K}$ in $L$, also Henselian.
We say that $L/K$ has a geometric property $\mathcal{P}$ over $\mathcal{O}_{K}$
if the corresponding morphism $\Spec(\mathcal{O}_{L})\rightarrow\Spec(\mathcal{O}_{K})$
does. 
\begin{prop}
\label{prop:ConstrCanPtBT}Let $G$ be a reductive group over $\mathcal{O}_{K}$.
\begin{lyxlist}{MM}
\item [{$(1)$}] There is an extension $L/K$, finite étale and Galois
over $\mathcal{O}_{K}$, splitting $G$. 
\item [{$(2)$}] The Bruhat-Tits building $\mathbf{B}^{e}(G_{K})$ exists
and contains a canonical point\nomenclature[o_G^e]{$\circ ^e _{G}$}{Canonical point of $\mathbf{B}^e (G_K)$ attached to $G$ over $\mathcal{O}_K$, page \nomrefpage}\nomenclature[o_G^r]{$\circ ^r _{G}$}{Canonical point of $\mathbf{B}^r (G_K)$ attached to $G$ over $\mathcal{O}_K$, page \nomrefpage}
\[
\circ_{G,K}^{e}=\circ_{G}^{e}=(\circ_{G}^{r},0)\in\mathbf{B}^{e}(G_{K})=\mathbf{B}^{r}(G_{K})\times\mathbf{G}(Z(G_{K}))
\]
with stabilizer $G(\mathcal{O}_{K})$ in $G(K)$. The projection $\circ_{G}^{r}$
of $\circ_{G}^{e}$ is the unique fixed point of $G(\mathcal{O}_{K})$
in $\mathbf{B}^{r}(G_{K})$ if the residue field of $\mathcal{O}_{K}$
is neither $\mathbb{F}_{2}$ nor $\mathbb{F}_{3}$. 
\item [{$(3)$}] The apartments of $\mathbf{B}^{e}(G_{K})$ containing
$\circ_{G}^{e}$ are the $\mathbf{B}^{e}(S_{K})$'s for $S\in\mathbf{S}(G)$.
\end{lyxlist}
\end{prop}
\begin{proof}
Let $S$ be a maximal split torus of $G$ and let $T$ be a maximal
torus of $Z_{G}(S)$ \cite[XIV 3.20]{SGA3.2}. Then $G$ and $T$
are isotrivial by proposition~\ref{pro:LocalUnibImpliesIsotriv},
thus split by a finite étale cover of $\Spec(\mathcal{O}_{K})$ which
we may assume to be connected and Galois, i.e.~of the form $\Spec(\mathcal{O}_{L})\rightarrow\Spec(\mathcal{O}_{K})$
where $\mathcal{O}_{L}$ is the normalization of $\mathcal{O}_{K}$
in a finite étale Galois extension $L/K$ over $\mathcal{O}_{K}$
by \cite[18.10.12]{EGA4.4}. Since $\mathcal{O}_{K}$ is Henselian,
$\mathcal{O}_{L}$ is also the valuation ring of $(L,\left|-\right|)$.
Let $(x_{\alpha})$ be a Chevalley system for $(G_{\mathcal{O}_{L}},T_{\mathcal{O}_{L}})$,
as defined in \cite[XXIII 6.2]{SGA3.3r}, giving rise to a Chevalley
valuation $\varphi_{L}$ for $G_{L}$, as explained in \cite[6.2.3.b]{BrTi72}
and \cite[4.2.1]{BrTi84}, thus also to the reduced Bruhat-Tits building
$\mathbf{B}^{r}(G_{L})$ with its distinguished apartment $\mathbf{B}^{r}(T_{L})$
and the distinguished point $\circ_{G}^{r}\equiv\varphi_{L}$ in $\mathbf{B}^{r}(T_{L})$,
as defined in \cite[\S 7]{BrTi72}. For $f=0$, the group schemes
$\mathfrak{G}_{f}^{0}\subset\mathfrak{G}_{f}\subset\hat{\mathfrak{G}}_{f}\subset\mathfrak{G}_{f}^{\dagger}$
constructed in \cite[4.3-6]{BrTi84} are all equal to $G_{\mathcal{O}_{L}}$
\cite[4.6.22]{BrTi84}. Thus by~\cite[4.6.28]{BrTi84}, $G(\mathcal{O}_{L})$
is the stabilizer of the distinguished point $\circ_{G}^{e}=(\circ_{G}^{r},0)$
of $\mathbf{B}^{e}(T_{L})\subset\mathbf{B}^{e}(G_{L})$ in $G(L)$,
and $\circ_{G}^{r}$ is the unique fixed point of $G(\mathcal{O}_{L})$
in $\mathbf{B}^{r}(G_{L})$ by \cite[5.1.39]{BrTi84} if the residue
field of $\mathcal{O}_{L}$ is not equal to $\mathbb{F}_{2}$ or $\mathbb{F}_{3}$,
which we can always assume. 

The pair $(G_{K},K)$ satisfies the conditions of the pair denoted
by $(H,K^{\natural})$ in \cite[5.1.1]{BrTi84}. The Galois group
$\Sigma=\Gal(L/K)$ acts compatibly on $G(L)$ and $\mathbf{B}^{e}(G_{L})$.
It therefore fixes $\circ_{G}^{e}$, which thus belongs to $\mathbf{B}^{e}(T_{L})^{\Sigma}=\mathbf{B}^{e}(S_{K})$.
Applying this to $Z_{G}(S)$ instead of $G$, we see that $(G_{K},K)$
also satisfies the assumption $(DE)$ of \cite[5.1.5]{BrTi84}. Then
by \cite[5.1.20]{BrTi84}, the valuation $\varphi_{L}$ descends to
a valuation $\varphi$ for $G_{K}$. The corresponding building $\mathbf{B}^{e}(G_{K})$
is the fixed point set of $\Sigma$ in $\mathbf{B}^{e}(G_{L})$ by
\cite[5.1.25]{BrTi84}. The stabilizer of $\circ_{G}^{e}\in\mathbf{B}^{e}(G_{K})$
in $G(K)$ equals $G(\mathcal{O}_{K})=G(K)\cap G(\mathcal{O}_{L})$
and again by \cite[5.1.39]{BrTi84}, $\circ_{G}^{r}$ is the unique
fixed point of $G(\mathcal{O}_{K})$ in $\mathbf{B}^{r}(G_{K})$ if
the residue field of $\mathcal{O}_{K}$ is not equal to $\mathbb{F}_{2}$
or $\mathbb{F}_{3}$. By construction, $\circ_{G}^{e}$ belongs to
$\mathbf{B}^{e}(S_{K})$. Therefore \cite[7.4.9]{BrTi72} proves our
last claim, since $G(\mathcal{O}_{K})$ also acts transitively on
$\mathbf{S}(G)$.
\end{proof}

\subsection{~}

We denote by $\mathbf{B}^{e}(G,K,\left|-\right|)$\nomenclature[B^e(G,K)]{$\mathbf{B}^e (G,K)$}{Pointed extended Bruhat-Tits building for $G$ over $\mathcal{O}_K$, page \nomrefpage}
the pointed affine $\mathbf{F}(G_{K})$-building 
\[
\mathbf{B}^{e}(G,K,\left|-\right|)=(\mathbf{B}^{e}(G_{K}),\circ_{G}^{e})
\]
attached to a reductive group $G$ over $\mathcal{O}_{K}$. It easily
follows from \cite[5.1.41]{BrTi84} that this construction is functorial
in the Henselian pair $(K,\left|-\right|)$. More precisely, let $\HV$\nomenclature[HV]{$\HV$}{Category of Henselian valued fields, page \nomrefpage}
be the category whose objects are pairs $(K,\left|-\right|)$ where
$K$ is a field and $\left|-\right|:K\rightarrow\mathbb{R}^{+}$ is
a non-trivial, non-archimedean absolute value whose valuation ring
$\mathcal{O}_{K}$ is Henselian. Then for every morphism $f:(K,\left|-\right|)\rightarrow(L,\left|-\right|)$
in $\HV$ and every reductive group $G$ over $\mathcal{O}_{K}$,
there is a canonical morphism $f:\mathbf{B}^{e}(G_{K})\rightarrow\mathbf{B}^{e}(G_{L})$
such that 
\[
f(\circ_{G}^{e})=\circ_{G}^{e},\quad f(gx)=f(g)f(x)\quad\mbox{and}\quad f(x+\mathcal{F})=f(x)+f(\mathcal{F})
\]
for every $x\in\mathbf{B}^{e}(G_{K})$, $g\in G(K)$ and $\mathcal{F}\in\mathbf{F}(G_{K})$.
The first and last property already determine this morphism uniquely:
by the axiom $T(s)$ for $\mathbf{B}^{e}(G_{K})$, any element $x$
of $\mathbf{B}^{e}(G_{K})$ equals $\circ_{G}^{e}+\mathcal{F}$ for
some for $\mathcal{F}\in\mathbf{F}(G_{K})$. 
\begin{rem}
The above functoriality amounts to saying that the mapping 
\[
\mathbf{B}^{e}(G_{K})\ni\circ_{G}^{e}+\mathcal{F}\mapsto\circ_{G}^{e}+f(\mathcal{F})\in\mathbf{B}^{e}(G_{L})
\]
is well-defined and equivariant with respect to $G(K)\rightarrow G(L)$.
This indeed implies the equivariance with respect to $f:\mathbf{F}(G_{K})\rightarrow\mathbf{F}(G_{L})$
as follows. For $S\in\mathbf{S}(G)$ mapping into $S'\in\mathbf{S}(G_{\mathcal{O}_{L}})$,
the above mapping restricts to a well-defined map $\mathbf{B}^{e}(S_{K})\rightarrow\mathbf{B}^{e}(S'_{L})$
which is equivariant with respect to $f:\mathbf{F}(S_{K})\rightarrow\mathbf{F}(S'_{K})$;
by the axiom $L(s)$ for $\mathbf{B}^{e}(G_{K})$ and proposition~\ref{prop:Generisation},
any pair $(x,\mathcal{F})$ in $\mathbf{B}^{e}(G_{K})\times\mathbf{F}(G_{K})$
is conjugated by some $g\in G(K)$ to one in $\mathbf{B}^{e}(S_{K})\times\mathbf{F}(S_{K})$,
thus
\[
f\left(x+\mathcal{F}\right)=f(g^{-1})f\left(gx+g\mathcal{F}\right)=f(g^{-1})\left(f(gx)+f(g\mathcal{F})\right)=f(x)+f(\mathcal{F}).
\]
\end{rem}
\begin{thm}
\label{thm:FonctBuild}The pointed affine $\mathbf{F}(G)$-building
$\mathbf{B}^{e}(G,K,\left|-\right|)$ is also functorial in the reductive
group $G$ over $\mathcal{O}_{K}$: for every morphism $f:G\rightarrow H$
of reductive groups over $\mathcal{O}_{K}$, there is a unique morphism
$f:\mathbf{B}^{e}(G_{K})\rightarrow\mathbf{B}^{e}(H_{K})$ such that
\[
f(\circ_{G}^{e})=\circ_{H}^{e},\quad f(gx)=f(g)f(x)\quad\mbox{and}\quad f(x+\mathcal{F})=f(x)+f(\mathcal{F})
\]
for every $x\in\mathbf{B}^{e}(G_{K})$, $g\in G(K)$ and $\mathcal{F}\in\mathbf{F}(G_{K})$. 
\end{thm}
\noindent This essentially follows from Landvogt's work in \cite{La00},
which has no assumptions on the reductive groups over $K$ but requires
$(K,\left|-\right|)$ to be quasi-local, in particular discrete. The
main difficulty there is the construction of base points with good
properties, which is here trivialized by the given points $\circ_{G}^{e}$
and $\circ_{H}^{e}$. Note that again, the uniqueness of $f:\mathbf{B}^{e}(G_{K})\rightarrow\mathbf{B}^{e}(H_{K})$
follows from the first and last displayed requirements, and its existence
amounts to showing that the mapping
\[
\mathbf{B}^{e}(G_{K})\ni\circ_{G}^{e}+\mathcal{F}\mapsto\circ_{H}^{e}+f(\mathcal{F})\in\mathbf{B}^{e}(H_{K})
\]
is well-defined and equivariant with respect to $f:G(K)\rightarrow H(K)$.
Given the identification $\mathbf{B}^{e}(GL(V))\simeq\mathbf{B}(V)$,
this theorem is closely related to the Tannakian theorem~\ref{thm:TanDescBuild}
below. We will prove the former as a corollary of the latter.

\subsection{~\label{sub:Gr_PforBTBuildings}}

Assuming theorem~\ref{thm:FonctBuild}, we may work out an analog
of the discussion of section~\ref{sub:Gr_PforF(G)} for the pointed
Bruhat-Tits building $\mathbf{B}^{e}(G,K)$. First, recall that $\mathbf{P}(G_{K})=\mathbf{P}(G)$
since $\mathbb{P}(G)$ is projective over $\mathcal{O}_{K}$. Let
thus $P\in\mathbf{P}(G)$ be a parabolic subgroup of $G$ with unipotent
radical $U$. For every Levi subgroup $L$ of $P$, there is a canonical
commutative diagram 
\[
\xyC{1pc}\xymatrix{ & \mathbf{B}^{e}(L,K)=\mathbf{B}^{e}(L_{K})\ar@<1ex>@{^{(}->}[d]^{\iota_{L,G}}\ar@(r,u)[ddr]^{\simeq}\ar@(l,u)[ddl]_{\simeq}\\
 & \mathbf{B}^{e}(G,K)=\mathbf{B}^{e}(G_{K})\ar@<1ex>@{->>}[u]^{r_{P,L}}\ar@{->>}[dl]_{\Gr_{P}}\ar@{->>}[dr]^{\Gr_{P}^{\infty}}\\
\mathbf{B}^{e}(P/U,K)\ar[rr]_{\simeq}^{\psi} &  & \mathbf{T}_{P}^{\infty}\mathbf{B}^{e}(G_{K})
}
\]
where $\iota_{L}:\mathbf{B}^{e}(L,K)\simeq\mathbf{B}^{e}(P/U,K)$
and $\iota_{L,G}:\mathbf{B}^{e}(L,K)\hookrightarrow\mathbf{B}^{e}(G,K)$
are the $L(K)$-equivariant maps functorially induced by $L\simeq P/U$
and $L\hookrightarrow G$, $r_{P,L}$ is the $U(K)$-invariant, $L(K)$-equivariant
retraction of proposition~\ref{prop:distretrac} onto the image $\cup_{S\in\mathbf{S}(L_{K})}\mathbf{B}^{e}(S)$
of $\iota_{L,G}$, $\Gr_{P}=\iota_{L}\circ r_{P,L}$ is a $P(K)$-equivariant
map\nomenclature[Gr_P]{$\Gr _P$}{$P(K)$-equivariant map $\mathbf{B}^e(G,K) \twoheadrightarrow \mathbf{B}^e(P/U,K)$, page \nomrefpage},
and the right hand side triangle comes from~\ref{sub:TangentAtInftyInBuild}.
Both $\Gr_{P}$ and $\Gr_{P}^{\infty}$ identify there codomain with
$U(K)\backslash\mathbf{B}^{e}(G_{K})$, which yields the existence
and unicity of the $P(K)$-equivariant bijection $\psi:\mathbf{B}^{e}(P/U,K)\simeq\mathbf{T}_{P}^{\infty}\mathbf{B}^{e}(G_{K})$
at the bottom of our diagram. Neither $\psi$ nor $\Gr_{P}$ depends
upon the choice of $L$: if $L'$ is another Levi subgroup of $P$,
there is a $u\in U(\mathcal{O}_{K})$ such that $L'=uLu^{-1}$. The
automorphism $\Int(u):G\rightarrow G$ then induces by functoriality
a commutative diagram
\[
\xyC{2pc}\xyR{2pc}\xymatrix{\mathbf{B}^{e}(G,K)\ar[d]^{\Int(u)}\ar@{->>}[r]\sp(0.40){r_{P,L}} & \cup_{S\in\mathbf{S}(L_{K})}\mathbf{B}^{e}(S)\ar[d]^{\Int(u)} & \mathbf{B}^{e}(L,K)\ar[l]\sb(0.4){\iota_{L,G}}\ar[d]^{\Int(u)}\ar[r]\sp(0.45){\iota_{L}} & \mathbf{B}^{e}(P/U,K)\ar[d]^{\mathrm{Id}}\\
\mathbf{B}^{e}(G,K)\ar@{->>}[r]\sp(0.38){r_{P,L'}} & \cup_{S'\in\mathbf{S}(L'_{K})}\mathbf{B}^{e}(S') & \mathbf{B}^{e}(L',K)\ar[l]\sb(0.4){\iota_{L',G}}\ar[r]\sp(0.45){\iota_{L'}} & \mathbf{B}^{e}(P/U,K)
}
\]
The first vertical map is also equal to the multiplication by $u$
map on $\mathbf{B}^{e}(G,K)$: 
\[
\Int(u)(\circ_{G}^{e}+\mathcal{F})=\circ_{G}^{e}+u\mathcal{F}=u(\circ_{G}^{e}+\mathcal{F})
\]
for all $\mathcal{F}\in\mathbf{F}(G)$ since $u\in G(\mathcal{O}_{K})$
fixes $\circ_{G}^{e}$. Thus $\Gr_{P}$ and $\psi$ indeed do not
depend upon the choice of $L$. One checks easily that $\psi$ is
an isomorphism of affine $\mathbf{F}(P_{K}/U_{K})$-spaces. In particular:
$\mathbf{T}_{P}^{\infty}\mathbf{B}^{e}(G_{K})$ is an affine $\mathbf{F}(P_{K}/U_{K})$-building,
its ``quotient'' and ``building'' metric agree by~\ref{sub:QuotientAndBuildMetricAgree},
thus $\psi:\mathbf{B}^{e}(P/U,K)\rightarrow\mathbf{T}_{P}^{\infty}\mathbf{B}^{e}(G_{K})$
is an isometry while $\Gr_{P}:\mathbf{B}^{e}(G,K)\twoheadrightarrow\mathbf{B}^{e}(P/U,K)$
is non-expanding when everyone is equipped with the metrics induced
by a chosen faithful representation $\tau$ of $G_{K}$. This gives
the following formula: for every $x,y\in\mathbf{B}^{e}(G_{K})$ and
$\mathcal{F}\in\mathbf{F}(G)$, 
\[
\lim_{t\rightarrow\infty}d_{\tau}(x+t\mathcal{F},y+t\mathcal{F})=d_{\Gr_{\mathcal{F}}^{\bullet}(\tau)}(\Gr_{\mathcal{F}}(x),\Gr_{\mathcal{F}}(y))\leq d_{\tau}(x,y)
\]
where $\Gr_{\mathcal{F}}=\Gr_{P_{\mathcal{F}}}:\mathbf{B}^{e}(G,K)\twoheadrightarrow\mathbf{B}^{e}(P_{\mathcal{F}}/U_{\mathcal{F}},K)$.
Also:
\[
\left\langle \overrightarrow{xy},\mathcal{F}\right\rangle =\left\langle \overrightarrow{\Gr_{\mathcal{F}}(x)\Gr_{\mathcal{F}}(y)},\overline{\mathcal{F}}\right\rangle 
\]
for every $x,y\in\mathbf{B}^{e}(G,K)$, with $\overline{\mathcal{F}}=\Gr_{\mathcal{F}}(\mathcal{F})$
in $\mathbf{G}(Z(P_{\mathcal{F}}/U_{\mathcal{F}}))=\mathbf{G}(\overline{R}(P_{\mathcal{F}}))$.

\subsection{~}

We may also establish some partial functoriality results when no base
point is given, as in Landvogt's work. Fix a quasi-local (discrete,
Henselian) pair $(K,\left|-\right|)$. For any reductive group $G$
over $K$, there is a finite Galois extension $L$ of $K$ splitting
$G$ such that the reduced building $\mathbf{B}^{r}(G_{L})$ contains
a special point $\circ$ fixed by $\Gal(L/K)$. Indeed, let first
$L_{1}$ be a finite Galois extension of $K$ splitting $G$, and
choose a facet $F$ of $\mathbf{B}^{r}(G_{L_{1}})$ fixed by $\Gal(L_{1}/K)$,
for instance one which intersects $\mathbf{B}^{r}(G_{K})$. Then the
barycenter $\circ$ of $F$ is also fixed by $\Gal(L_{1}/K)$, and
just like any barycenter of a facet of the Bruhat-Tits building of
a split group, it becomes special over a sufficiently ramified extension
$L$ of $L_{1}$, which we may assume to be Galois over $K$. Write
$\circ_{G}^{e}=(\circ,0)$ for the corresponding $\Gal(L/K)$-invariant
point of $\mathbf{B}^{e}(G_{L})$ and let $G_{\circ}$ be the reductive
group over $\mathcal{O}_{L}$ with generic fiber $G_{L}$ such that
$G_{\circ}(\mathcal{O}_{L})$ is the stabilizer of $\circ_{G}^{e}$
in $G(L)$. Since $\circ_{G}^{e}$ is fixed by $\Gal(L/K)$, the Hopf
$\mathcal{O}_{L}$-sub-algebra $\mathcal{A}(G_{\circ})$ of $\mathcal{A}(G_{L})=\mathcal{A}(G)_{L}$
is fixed by the action of $\Gal(L/K)$. 

Let now $\tau$ be a finite dimensional $K$-representation of $G$,
corresponding to a morphism $f:G\rightarrow H$, with $H=GL(V)$,
$V=V(\tau)$. By~\cite[1.5]{Se68b}, every finitely generated $\mathcal{O}_{L}$-submodule
$M$ of $V_{L}$ is contained in some $\mathcal{A}(G_{\circ})$-sub-comodule
$F$ of $V_{L}$ which is finitely generated (hence free) over $\mathcal{O}_{L}$.
Since $\mathcal{A}(G_{\circ})$ is flat over $\mathcal{O}_{L}$, there
is a smallest such $F$, which we denote by $F(M)$. Since $\mathcal{A}(G_{\circ})$
is stabilized by $\Gal(L/K)$, the map $M\mapsto F(M)$ is $\Gal(L/K)$-equivariant.
Thus starting with a $\Gal(L/K)$-stable $\mathcal{O}_{L}$-lattice
$M$ of $V_{L}$, for instance the base change of an $\mathcal{O}_{K}$-lattice
of $V$, we obtain an $\mathcal{O}_{L}$-model $f:G_{\circ}\rightarrow H_{\circ}$
of $f_{L}:G_{L}\rightarrow H_{L}$, with $H_{\circ}=GL(F(M))$, such
that the point $\circ_{H}^{e}=(\circ,0)$ corresponding to $H_{\circ}$
in $\mathbf{B}^{e}(H_{L})$ is also fixed by $\Gal(L/K)$. Applying
now the previous functoriality results to this $\mathcal{O}_{L}$-morphism
$f:G_{\circ}\rightarrow H_{\circ}$, we obtain: for every extension
$(L',\left|-\right|)$ of $(L,\left|-\right|)$ in $\HV$, there is
a unique morphism $f:\mathbf{B}^{e}(G_{L'})\rightarrow\mathbf{B}^{e}(H_{L'})$
such that 
\[
f(\circ_{G}^{e})=\circ_{H}^{e},\quad f(gx)=f(g)f(x)\quad\mbox{and}\quad f(x+\mathcal{F})=f(x)+f(\mathcal{F})
\]
for every $x\in\mathbf{B}^{e}(G_{L'})$, $g\in G(L')$ and $\mathcal{F}\in\mathbf{F}(G_{L'})$.
Moreover, for every $K$-linear morphism $\sigma:(L',\left|-\right|)\rightarrow(L'',\left|-\right|)$
between two such extensions, 
\[
f(\sigma x)=\sigma f(x)\quad\mbox{in}\quad\mathbf{B}^{e}(G_{L''})
\]
for every $x\in\mathbf{B}^{e}(G_{L'})$. Indeed if $x=\circ_{G}^{e}+\mathcal{F}$
with $\mathcal{F}\in\mathbf{F}(G_{L'})$, then 
\begin{eqnarray*}
f(\sigma x) & = & f(\sigma\circ_{G}^{e}+\sigma\mathcal{F})=f(\circ_{G}^{e}+\sigma\mathcal{F})=\circ_{H}^{e}+f(\sigma\mathcal{F})\\
 & = & \sigma\circ_{H}^{e}+\sigma f(\mathcal{F})=\sigma(\circ_{H}^{e}+f(\mathcal{F}))=\sigma f(x).
\end{eqnarray*}

\section{A Tannakian formalism for Bruhat-Tits buildings}

\subsection{~}

Let again $(K,\left|-\right|)$ be a field with a non-trivial, non-archimedean
absolute value $\left|-\right|:K\rightarrow\mathbb{R}^{+}$, with
valuation ring $\mathcal{O}_{K}$ and residue field $k$. We denote
by $\Norm^{\circ}(K,\left|-\right|)$\nomenclature[Norm]{$\Norm ^\circ (K)$}{Category of splittable normed finite $K$-vector spaces, page \nomrefpage}
the category whose objects are pairs $(V,\alpha)$ where $V$ is a
finite dimensional $K$-vector space and $\alpha:V\rightarrow\mathbb{R}^{+}$
is a splittable $K$-norm on $V$. A morphism $f:(V,\alpha)\rightarrow(V',\alpha')$
is a $K$-linear morphism $f:V\rightarrow V'$ such that $\alpha'(f(x))\leq\alpha(x)$
for every $x\in V$. This defines an $\mathcal{O}_{K}$-linear rigid
$\otimes$-category with neutral object $1_{K}=(K,\left|-\right|)$.
The $\otimes$-products, inner homs and duals 
\begin{eqnarray*}
\left(V_{1},\alpha_{1}\right)\otimes\left(V_{2},\alpha_{2}\right) & = & \left(V_{1}\otimes V_{2},\alpha_{1}\otimes\alpha_{2}\right)\\
\Hom\left(\left(V_{1},\alpha_{1}\right),\left(V_{2},\alpha_{2}\right)\right) & = & \left(\Hom\left(V_{1},V_{2}\right),\Hom\left(\alpha_{1},\alpha_{2}\right)\right)\\
\left(V,\alpha\right)^{\ast} & = & \left(V^{\ast},\alpha^{\ast}\right)
\end{eqnarray*}
are respectively given by : $\alpha_{1}\otimes\alpha_{2}=\Hom(\alpha_{1}^{\ast},\alpha_{2})$
under $V_{1}\otimes V_{2}=\Hom(V_{1}^{\ast},V_{2})$, 
\begin{eqnarray*}
\Hom\left(\alpha_{1},\alpha_{2}\right)(f) & = & \sup\left\{ \frac{\alpha_{2}(f(x))}{\alpha_{1}(x)}:x\in V_{1}\setminus\{0\}\right\} ,\\
\alpha^{\ast}(f) & = & \sup\left\{ \frac{\left|f(x)\right|}{\alpha(x)}:x\in V\setminus\{0\}\right\} .
\end{eqnarray*}
In addition, $\Norm^{\circ}(K,\left|-\right|)$ is an exact category
in Quillen's sense: a short sequence
\[
(V_{1},\alpha_{1})\stackrel{f_{1}}{\longrightarrow}(V_{2},\alpha_{2})\stackrel{f_{2}}{\longrightarrow}(V_{3},\alpha_{3})
\]
is exact precisely when the underlying sequence of $K$-vector spaces
is exact and 
\[
\alpha_{1}(x)=\alpha_{2}(f_{1}(x)),\qquad\alpha_{3}(z)=\inf\left\{ \alpha_{2}(y):y\in f_{2}^{-1}(z)\right\} 
\]
for every $x\in V_{1}$ and $z\in V_{3}$. For $\gamma\in\mathbb{R}$
and $(V,\alpha)\in\Norm^{\circ}(K,\left|-\right|)$, we set\nomenclature[B(alpha,gamma)]{$B(\alpha,\gamma)$}{Open ball of radius $\exp (-\gamma)$ for $\alpha$, page \nomrefpage}\nomenclature[B(alpha,gamma)b]{$\overline{B}(\alpha,\gamma)$}{Closed ball of radius $\exp (-\gamma)$ for $\alpha$, page \nomrefpage}
\begin{eqnarray*}
B(\alpha,\gamma) & = & \left\{ x\in V:\alpha(x)<\exp(-\gamma)\right\} \}\\
\overline{B}(\alpha,\gamma) & = & \left\{ x\in V:\alpha(x)\leq\exp(-\gamma)\right\} \}
\end{eqnarray*}
These are $\mathcal{O}_{K}$-submodules of $V$ and the functors $(V,\alpha)\mapsto B(\alpha,\gamma)$
are easily seen to be exact. However, $(V,\alpha)\mapsto\overline{B}(\alpha,\gamma)$
is \emph{also }exact, because in fact every exact sequence in $\Norm^{\circ}(K)$
is split by \cite[1.5.ii + Appendix]{BrTi84b}! If $M$ is an $\mathcal{O}_{K}$-lattice
in $V$ (by which we mean a finitely generated, thus free, $\mathcal{O}_{K}$-submodule
spanning $V$), we denote by $\alpha_{M}$ the splittable $K$-norm
on $V$ with $\overline{B}(\alpha_{M},0)=M$ defined by
\[
\alpha_{M}(x)=\inf\left\{ \left|\lambda\right|:\lambda\in K,\, x\in\lambda M\right\} =\min\left\{ \left|\lambda\right|:\lambda\in K,\, x\in\lambda M\right\} .
\]

\subsection{~}

For $(K,\left|-\right|)\rightarrow(L,\left|-\right|)$, there is an
exact $\mathcal{O}_{K}$-linear $\otimes$-functor
\[
-\otimes L:\Norm^{\circ}(K,\left|-\right|)\rightarrow\Norm^{\circ}(L,\left|-\right|)
\]
defined by $(V,\alpha)\otimes L=(V_{L},\alpha_{L})$ where $V_{L}=V\otimes L$
and 
\begin{eqnarray*}
\alpha_{L}(v) & = & \inf\left\{ \max\{\left|x_{k}\right|\alpha(v_{k})\}:v={\textstyle \sum}v_{k}\otimes x_{k},\mbox{ }v_{k}\in V,\, x_{k}\in L\right\} ,\\
 & = & \min\left\{ \max\{\left|x_{k}\right|\alpha(v_{k})\}:v={\textstyle \sum}v_{k}\otimes x_{k},\mbox{ }v_{k}\in V,\, x_{k}\in L\right\} .
\end{eqnarray*}
For $(V,\alpha)\in\Norm^{\circ}(K,\left|-\right|)$, $\gamma\in\mathbb{R}$
and $x\in V$, 
\[
B(\alpha_{L},\gamma)=B(\alpha,\gamma)\otimes\mathcal{O}_{L},\quad\overline{B}(\alpha_{L},\gamma)=\overline{B}(\alpha,\gamma)\otimes\mathcal{O}_{L}\quad\mbox{and}\quad\alpha=\alpha_{L}\vert V.
\]
If $M$ is an $\mathcal{O}_{K}$-lattice in $V$, then $\alpha_{M,L}=\alpha_{M\otimes\mathcal{O}_{L}}$.

\subsection{~}

We shall also consider the category $\Norm'(K)$\nomenclature[Norm(K)']{$\Norm '(K)$}{Category of normed $K$-spaces with a lattice, defined page \nomrefpage}
whose objects are triples $(V,\alpha,M)$ where $(V,\alpha)$ is an
object of $\Norm^{\circ}(K)$ and $M$ is an $\mathcal{O}_{K}$-lattice
in $V$, with the obvious morphisms. It is again an $\mathcal{O}_{K}$-linear
$\otimes$-category. The formula
\[
\loc^{\gamma}(V,\alpha,M)=\mbox{image of }\overline{B}(\alpha,\gamma)\cap M\mbox{ in }M_{k}=M\otimes_{\mathcal{O}_{K}}k
\]
defines an $\mathcal{O}_{K}$-linear $\otimes$-functor with values
in $\Fil(k)=\Fil^{\mathbb{R}}\LF(k)$,\nomenclature[loc]{$\loc$}{Functor $\Norm '(K) \rightarrow \Fil (k)$, defined page \nomrefpage}
\[
\loc:\Norm'(K)\rightarrow\Fil(k).
\]
Indeed by the axiom $R(s)$ for $\mathbf{B}(V)$, there is an $\mathcal{O}_{K}$-basis
$(e_{1},\cdots,e_{n})$ of $M$ adapted to $\alpha$, thus $\alpha(\sum x_{i}e_{i})=\max\left\{ \left|x_{i}\right|e^{-\gamma_{i}}\right\} $
where $\gamma_{i}=-\log\alpha(e_{i})$ and 
\[
\loc^{\gamma}(V,\alpha,M)=\oplus_{\gamma_{i}\geq\gamma}ke_{i}
\]
from which easily follows that $\loc$ is well-defined and compatible
with $\otimes$-products.

\subsection{~}

For an extension $(K,\left|-\right|)\rightarrow(L,\left|-\right|)$
and a reductive group $G$ over $\mathcal{O}_{K}$, we denote by $\mathbf{B}'(\omega_{G}^{\circ},L,\left|-\right|)$
or simply $\mathbf{B}'(\omega_{G}^{\circ},L)$\nomenclature[B(omega_G^o,L)']{$\mathbf{B}'(\omega ^\circ _G ,L)$}{Space of all $L$-norms on $\omega ^\circ _{G,L}$, defined page \nomrefpage}
the set of all factorizations
\[
\Rep^{\circ}(G)(\mathcal{O}_{K})\stackrel{\alpha}{\longrightarrow}\Norm^{\circ}(L,\left|-\right|)\stackrel{\mathrm{forg}}{\longrightarrow}\Vect(L)
\]
of the fiber functor $\omega_{G,L}^{\circ}$ through an $\mathcal{O}_{K}$-linear
$\otimes$-functor $\alpha$. For $\tau\in\Rep^{\circ}(G)(\mathcal{O}_{K})$
and $\alpha\in\mathbf{B}'(\omega_{G}^{\circ},L)$, we denote by $\alpha(\tau)$
the corresponding $L$-norm on $V_{L}(\tau)$.

\subsection{~}

For $g\in G(L)$ and $\mathcal{F}\in\mathbf{F}(G_{L})$, the following
formulas
\[
(g\cdot\alpha)(\tau)=\tau_{L}(g)\cdot\alpha(\tau)\quad\mbox{and}\quad(\alpha+\mathcal{F})(\tau)=\alpha(\tau)+\mathcal{F}(\tau)
\]
respectively define an action of $G(L)$ on $\mathbf{B}'(\omega_{G}^{\circ},L)$
and a $G(L)$-equivariant map
\[
+:\mathbf{B}'(\omega_{G}^{\circ},L)\times\mathbf{F}(G_{L})\rightarrow\mathbf{B}'(\omega_{G}^{\circ},L).
\]

\subsection{~}

We define the canonical $L$-norm $\alpha_{G,L}$\nomenclature[alpha_G,L]{$\alpha _{G,L}$}{Canonical $L$-norm on $\omega ^\circ _{G,L}$ defined page \nomrefpage}
on $\omega_{G,L}^{\circ}$ by the formula 
\[
\alpha_{G,L}(\tau)=\alpha_{V_{\mathcal{O}_{L}}(\tau)}=\alpha_{V(\tau),L}.
\]
By~propositions~\ref{Pro:Aut(w0)G(w0)F(w0)} and \ref{pro:LocalUnibImpliesIsotriv},
$G(\mathcal{O}_{L})$ is the stabilizer of $\alpha_{G,L}$ in $G(L)$.
We set\nomenclature[B(omega_G^o,L)]{$\mathbf{B}(\omega ^\circ _G ,L)$}{Space of all good $L$-norms on $\omega ^\circ _{G,L}$, defined page \nomrefpage}
\[
\mathbf{B}(\omega_{G}^{\circ},L)\stackrel{\mathrm{def}}{=}\alpha_{G,L}+\mathbf{F}(G_{L}).
\]
This is a $G(\mathcal{O}_{L})$-stable subset of $\mathbf{B}'(\omega_{G}^{\circ},L)$
equipped with a $G(\mathcal{O}_{L})$-equivariant map\nomenclature[can]{$\can$}{Map $\mathbf{F}(G_L) \twoheadrightarrow \mathbf{B}(\omega ^\circ _G ,L)$, defined page \nomrefpage}
\[
\can:\mathbf{F}(G_{L})\twoheadrightarrow\mathbf{B}(\omega_{G}^{\circ},L),\qquad\can(\mathcal{F})=\alpha_{G,L}+\mathcal{F}.
\]

\subsection{~}

Any $L$-norm $\alpha$ on $\omega_{G,L}^{\circ}$ induces an $\mathcal{O}_{K}$-linear
$\otimes$-functor 
\[
\alpha':\Rep^{\circ}(G)(\mathcal{O}_{K})\rightarrow\Norm'(L)
\]
by the formula $\alpha'(\tau)=\left(V_{L}(\tau),\alpha(\tau),V_{\mathcal{O}_{L}}(\tau)\right)$,
thus also an $\mathcal{O}_{K}$-linear $\otimes$-functor 
\[
\loc(\alpha):\Rep^{\circ}(G)(\mathcal{O}_{K})\rightarrow\Fil(k_{L}),\quad\loc(\alpha)=\loc\circ\alpha'
\]
where $k_{L}$ is the residue field of $\mathcal{O}_{L}$. We may
thus define\nomenclature[B(omega_G^o,L)?]{$\mathbf{B}^{?}(\omega ^\circ _G ,L)$}{Space of all nice $L$-norms on $\omega ^\circ _{G,L}$, defined page \nomrefpage}
\[
\mathbf{B}^{?}(\omega_{G}^{\circ},L)=\left\{ \alpha\in\mathbf{B}'(\omega_{G}^{\circ},L):\loc(\alpha)\mbox{\,\ is exact}\right\} .
\]
This is a $G(\mathcal{O}_{L})$-stable subset of $\mathbf{B}'(\omega_{G}^{\circ},L)$
equipped with a $G(\mathcal{O}_{L})$-equivariant map\nomenclature[loc]{$\loc$}{Map $\mathbf{B}^{?}(\omega ^\circ _G ,L) \rightarrow \mathbf{F}(G_{k_L})$, defined page \nomrefpage}
\[
\loc:\mathbf{B}^{?}(\omega_{G}^{\circ},L)\rightarrow\mathbf{F}(G_{k_{L}}).
\]

\subsection{~ }

All of the above constructions are functorial in $G$, $(K,\left|-\right|)$
and $(L,\left|-\right|)$, using pre- or post-composition with the
obvious exact $\otimes$-functors
\[
\begin{array}{rcl}
\Rep^{\circ}(G_{2})(\mathcal{O}_{K}) & \longrightarrow & \Rep^{\circ}(G_{1})(\mathcal{O}_{K})\\
\Rep^{\circ}(G)(\mathcal{O}_{K_{1}}) & \longrightarrow & \Rep^{\circ}(G)(\mathcal{O}_{K_{2}})\\
\Norm^{\circ}(L_{1},\left|-\right|_{1}) & \longrightarrow & \Norm^{\circ}(L_{2},\left|-\right|_{2})
\end{array}\mbox{ for }\begin{array}{rcl}
G_{1} & \rightarrow & G_{2}\\
(K_{1},\left|-\right|_{1}) & \rightarrow & (K_{2},\left|-\right|_{2})\\
(L_{1},\left|-\right|_{1}) & \rightarrow & (L_{2},\left|-\right|_{2})
\end{array}
\]

\begin{lem}
\label{lem:loccan}For any reductive group $G$ over $\mathcal{O}_{K}$,
we have 
\[
\mathbf{B}(\omega_{G}^{\circ},L)\subset\mathbf{B}^{?}(\omega_{G}^{\circ},L)\subset\mathbf{B}'(\omega_{G}^{\circ},L)
\]
and the composition $\loc\circ\can:\mathbf{F}(G_{L})\rightarrow\mathbf{F}(G_{k_{L}})$
is the reduction map 
\[
\mathbf{F}(G_{L})\stackrel{\simeq}{\longleftarrow}\mathbf{F}(G_{\mathcal{O}_{L}})\stackrel{\mathrm{red}}{\longrightarrow}\mathbf{F}(G_{k_{L}}).
\]
For any $S\in\mathbf{S}(G_{\mathcal{O}_{L}})$, the functorial map
$\mathbf{B}'(\omega_{S}^{\circ},L)\rightarrow\mathbf{B}'(\omega_{G}^{\circ},L)$
is injective.\end{lem}
\begin{proof}
By proposition~\ref{prop:Generisation}, any $\mathcal{F}\in\mathbf{F}(G_{L})$
belongs to $\mathbf{F}(S_{L})$ for some $S$ in $\mathbf{S}(G_{\mathcal{O}_{L}})$.
Pre-composing with $\Rep^{\circ}(G)(\mathcal{O}_{K})\rightarrow\Rep^{\circ}(S)(\mathcal{O}_{L})$
yields the vertical maps of the commutative diagram 
\[
\xymatrix{\mathbf{F}(S_{L})\ar@{->>}[r]\sp(0.4){\can}\ar@{_{(}->}[d] & \mathbf{B}(\omega_{S}^{\circ},L)\ar@{^{(}->}[r]\ar[d] & \mathbf{B}'(\omega_{S}^{\circ},L)\ar[d] & \mathbf{B}^{?}(\omega_{S}^{\circ},L)\ar@{_{(}->}[l]\ar[r]^{\loc}\ar[d] & \mathbf{F}(S_{k_{L}})\ar@{^{(}->}[d]\\
\mathbf{F}(G_{L})\ar@{->>}[r]\sp(0.4){\can} & \mathbf{B}(\omega_{G}^{\circ},L)\ar@{^{(}->}[r] & \mathbf{B}'(\omega_{G}^{\circ},L) & \mathbf{B}^{?}(\omega_{G}^{\circ},L)\ar@{_{(}->}[l]\ar[r]^{\loc} & \mathbf{F}(G_{k_{L}})
}
\]
which reduces us to the case $K=L$, $G=S$ treated below.\end{proof}
\begin{lem}
\label{lem:ApptNorm}Suppose that $G=S$ is a split torus. Then all
maps in
\[
\xymatrix{\mathbf{F}(S_{L})\ar@{->>}[r]\sp(0.4){\can} & \mathbf{B}(\omega_{S}^{\circ},L)\ar@{^{(}->}[r] & \mathbf{B}'(\omega_{S}^{\circ},L) & \mathbf{B}^{?}(\omega_{S}^{\circ},L)\ar@{_{(}->}[l]\ar[r]^{\loc} & \mathbf{F}(S_{k_{L}})}
\]
are isomorphisms of pointed affine $\mathbf{G}(S)$-spaces. Moreover,
$S(L)$ acts on
\[
\mathbf{B}(\omega_{S}^{\circ},L)=\mathbf{B}^{?}(\omega_{S}^{\circ},L)=\mathbf{B}'(\omega_{S}^{\circ},L)
\]
by translations through the morphism 
\[
\nu_{\mathbf{B},S}:S(L)\rightarrow\mathbf{G}(S)
\]
which maps $s\in S(L)$ to the unique morphism $\nu_{\mathbf{B},S}(s):\mathbb{D}_{\mathcal{O}_{K}}(\mathbb{R})\rightarrow S$
whose composition with any character $\chi$ of $S$ is the character
$\log\left|\chi(s)\right|\in\mathbb{R}$ of $\mathbb{D}_{\mathcal{O}_{K}}(\mathbb{R})$. \end{lem}
\begin{proof}
Put $M=\Hom(S,\mathbb{G}_{m,\mathcal{O}_{K}})$ and let $\rho_{m}$
be the representation of $S$ on $\mathcal{O}_{K}$ given by the character
$m\in M$. For $\tau\in\Rep^{\circ}(S)(\mathcal{O}_{K})$, let $\tau=\oplus\tau_{m}$
be the weight decompositions of $\tau$. Recall from~section~\ref{sub:Splitting:caseoftori}
that the formulas
\[
\mathcal{F}^{\gamma}(\tau)=\oplus_{\mathcal{F}^{\sharp}(m)\geq\gamma}V(\tau_{m}),\quad\mathcal{F}^{\sharp}(m)=\sup\{\gamma:\mathcal{F}^{\gamma}(\rho_{m})\neq0\}
\]
yield isomorphisms between $\mathbf{F}(S)=\mathbf{G}(S)$ and $\Hom(M,\mathbb{R})$.
Similarly, the formulas
\[
\alpha(\tau)(x)=\max\left\{ e^{-\alpha^{\sharp}(m)}\alpha_{V_{\mathcal{O}_{L}}(\tau_{m})}(x_{m}):m\in M\right\} ,\quad\alpha^{\sharp}(m)=-\log\alpha(\rho_{m})(1_{\mathcal{O}_{K}})
\]
where $x=\sum x_{m}$ is the decomposition of $x$ in $V_{L}(\tau)=\oplus V_{L}(\tau_{m})$
yield isomorphisms between $\mathbf{B}'(\omega_{S}^{\circ},L)$ and
$\Hom(M,\mathbb{R})$. One then checks easily that 
\[
\alpha_{S,L}^{\sharp}=0,\quad(\alpha+\mathcal{F})^{\sharp}=\alpha^{\sharp}+\mathcal{F}^{\sharp}\quad\mbox{and}\quad s\cdot\alpha=\alpha+\nu_{\mathbf{B},S}(s)
\]
as well as $\loc^{\gamma}(\alpha)(\tau)=\oplus_{\alpha^{\sharp}(m)\geq\lambda}V_{k_{L}}(\tau_{m})$,
from which the lemma follows.
\end{proof}

\subsection{~}

For $S\in\mathbf{S}(G_{\mathcal{O}_{L}})$, we identify $\mathbf{B}(\omega_{S}^{\circ},L)$
with its image in $\mathbf{B}(\omega_{G}^{\circ},L)$ and call it
the apartment attached to $S$. The pull map on $\mathbf{B}'(\omega_{G}^{\circ},L)$
thus induces a structure of affine $\mathbf{F}(S_{L})$-space on $\mathbf{B}(\omega_{S}^{\circ},L)$,
and the action of $G(L)$ on $\mathbf{B}'(\omega_{G}^{\circ},L)$
restricts to an action of $S(L)$ on $\mathbf{B}(\omega_{S}^{\circ},L)$,
by translations through the above morphism $\nu_{\mathbf{B},S}:S(L)\rightarrow\mathbf{G}(S_{L})$.

\subsection{~}

We now restrict our attention to Henselian fields, so that $\mathbf{B}^{e}(G,L,\left|-\right|)$
is also well-defined, functorial in $(L,\left|-\right|)$, and equal
to $\circ_{G}^{e}+\mathbf{F}(G_{L})$ by $T(s)$. Given the functorial
properties of $\mathbf{B}(\omega_{G}^{\circ},L)$, theorem~\ref{thm:FonctBuild}
immediately follows from:
\begin{thm}
\label{thm:TanDescBuild}The formula $\circ_{G}^{e}+\mathcal{F}\mapsto\alpha_{G,L}+\mathcal{F}$
defines a functorial bijection\nomenclature[alpha]{$\boldsymbol{\alpha}$}{Functorial isomorphism $\mathbf{B}^{e}(G,L) \rightarrow \mathbf{B} (\omega_{G}^{\circ},L)$ defined page \nomrefpage}
\[
\boldsymbol{\alpha}:\mathbf{B}^{e}(G,L,\left|-\right|)\rightarrow\mathbf{B}(\omega_{G}^{\circ},L,\left|-\right|)
\]
such that for every $x\in\mathbf{B}^{e}(G_{L}),$ $g\in G(L)$ and
$\mathcal{F}\in\mathbf{F}(G_{L})$, 
\[
\boldsymbol{\alpha}(\circ_{G}^{e})=\alpha_{G},\quad\boldsymbol{\alpha}(g\cdot x)=g\cdot\boldsymbol{\alpha}(x)\quad\mbox{and}\quad\boldsymbol{\alpha}(x+\mathcal{F})=\boldsymbol{\alpha}(x)+\mathcal{F}.
\]
\end{thm}
\begin{proof}
Fix an extension $(L,\left|-\right|)\rightarrow(L',\left|-\right|)$
such that $G'=G_{\mathcal{O}_{L'}}$ splits and consider the following
diagram, where $\mathcal{F}\in\mathbf{F}(G_{L})$ and $\mathcal{F}'\in\mathbf{F}(G_{L'})=\mathbf{F}(G'_{L'})$:
\[
\xymatrix{\circ_{G}^{e}+\mathcal{F}\ar@{.>}[d]_{?} & \mathbf{B}^{e}(G,L)\ar@{.>}[d]_{\boldsymbol{\alpha}}\ar@{.>}[dr]^{\boldsymbol{\beta}}\ar[rr] &  & \mathbf{B}^{e}(G',L')\ar@{.>}[d]^{\boldsymbol{\alpha}'}\ar@{.>}[dl]_{\boldsymbol{\beta}'} & \circ_{G'}^{e}+\mathcal{F}'\ar@{.>}[d]^{?}\\
\alpha_{G,L}+\mathcal{F} & \mathbf{B}'(\omega_{G}^{\circ},L)\ar[r]^{-\otimes L'} & \mathbf{B}'(\omega_{G}^{\circ},L') & \mathbf{B}'(\omega_{G'}^{\circ},L')\ar[l]_{\mathrm{Res}} & \alpha_{G',L'}+\mathcal{F}'
}
\]
The bottom maps are respectively induced by post and pre-composition
with
\[
-\otimes L':\Norm^{\circ}(L)\rightarrow\Norm^{\circ}(L')\quad\mbox{and}\quad-\otimes\mathcal{O}_{L'}:\Rep^{\circ}(G)(\mathcal{O}_{K})\rightarrow\Rep^{\circ}(G')(\mathcal{O}_{L'}).
\]
If $\boldsymbol{\alpha}'$ is well-defined and equivariant with respect
to the operations of $G(L')$ and $\mathbf{F}(G_{L'})$, so is $\boldsymbol{\beta}'$.
Then $\boldsymbol{\beta}$ is well-defined and equivariant with respect
to the operations of $G(L)$ and $\mathbf{F}(G_{L})$. But $\mathbf{B}'(\omega_{G}^{\circ},L)\rightarrow\mathbf{B}'(\omega_{G}^{\circ},L')$
is injective, thus $\boldsymbol{\alpha}$ is also well-defined and
equivariant with respect to the operations of $G(L)$ and $\mathbf{F}(G_{L})$.
Its image equals $\mathbf{B}(\omega_{G}^{\circ},L)$ by definition,
which is thus stable under the operations of $G(L)$ and $\mathbf{F}(G_{L})$
on $\mathbf{B}'(\omega_{G}^{\circ},L)$. Since $\loc(\alpha_{G,L}+\mathcal{F})=\mathcal{F}_{k_{L}}$
for every $\mathcal{F}\in\mathbf{F}(G_{L})$, the restriction of $\boldsymbol{\alpha}$
to any apartment $\mathbf{B}^{e}(S_{L})=\circ_{G}^{e}+\mathbf{F}(S_{L})$
for $S\in\mathbf{S}(G_{\mathcal{O}_{L}}$) is injective. Since any
pair of points in $\mathbf{B}^{e}(G,L)$ is $G(L)$-conjugated to
one in such an apartment by the axiom $R(s)$ for $\mathbf{B}^{e}(G,L)$,
$\boldsymbol{\alpha}:\mathbf{B}^{e}(G,L)\rightarrow\mathbf{B}(\omega_{G}^{\circ},L)$
is a bijection. This reduces us to the case where $G$ is split over
$\mathcal{O}_{K}$ and $K=L$. 

Suppose that $\circ_{G}^{e}+\mathcal{F}_{1}=\circ_{G}^{e}+\mathcal{F}_{2}=x$
in $\mathbf{B}^{e}(G_{K})$ for some $\mathcal{F}_{1},\mathcal{F}_{2}\in\mathbf{F}(G_{K})$,
choose $S_{i}\in\mathbf{S}(G_{K})$ such that $\mathcal{F}_{i}\in\mathbf{F}(S_{i})$
and $\circ_{G}^{e}\in\mathbf{B}^{e}(S_{i})$ using $L(s)$ for $\mathbf{B}^{e}(G_{K})$,
and then choose $g\in G(K)$ fixing $\circ_{G}^{e}$ and $x$ such
that $\Int(g)(S_{1})=S_{2}$ using $R(i)$ for $\mathbf{B}^{e}(G_{K})$.
Then $S_{i}\in\mathbf{S}(G)$ and $g\in G(\mathcal{O}_{K})$ by proposition~\ref{prop:ConstrCanPtBT},
moreover $g\mathcal{F}_{1}=\mathcal{F}_{2}$ since $\mathbf{B}^{e}(S_{2})$
is an affine $\mathbf{F}(S_{2})$-space. Thus $g(\alpha_{G}+\mathcal{F}_{1})=\alpha_{G}+\mathcal{F}_{2}$
in $\mathbf{B}'(\omega_{G}^{\circ},K)$, since $G(\mathcal{O}_{K})$
fixes $\alpha_{G}$. But $g$ fixes the point $x=\circ_{G}^{e}+\mathcal{F}_{1}$
of $\mathbf{B}^{e}(S_{1})$, thus $g$ fixes $\alpha_{G}+\mathcal{F}_{1}$
in $\mathbf{B}'(\omega_{G}^{\circ},K)$ by lemma~\ref{lem:StabNorms}
below, therefore $\alpha_{G}+\mathcal{F}_{1}=\alpha_{G}+\mathcal{F}_{2}$
and our map $\boldsymbol{\alpha}:\mathbf{B}^{e}(G,K)\rightarrow\mathbf{B}'(\omega_{G}^{\circ},K)$
is indeed well-defined. 

It is plainly $G(\mathcal{O}_{K})$-equivariant. For any $S\in\mathbf{S}(G)$,
the $G(K)$-equivariant map $\boldsymbol{\alpha}_{S}$ of lemma~\ref{lem:StabNorms}
below coincides with $\boldsymbol{\alpha}$ on $\mathbf{B}^{e}(S_{K})$,
thus $\boldsymbol{\alpha}$ equals $\boldsymbol{\alpha}_{S}$ everywhere
since every point of $\mathbf{B}^{e}(G_{K})$ is conjugated to one
in $\mathbf{B}^{e}(S_{K})$ by some element in $G(\mathcal{O}_{K})$.
Therefore $\boldsymbol{\alpha}$ is $G(K)$-equivariant. Since every
pair in $\mathbf{B}^{e}(G_{K})\times\mathbf{F}(G_{K})$ is conjugated
to one in $\mathbf{B}^{e}(S_{K})\times\mathbf{F}(S_{K})$ by some
element in $G(K)$, our $\boldsymbol{\alpha}$ is also compatible
with the operations of $\mathbf{F}(G_{K})$.\end{proof}
\begin{lem}
\label{lem:StabNorms}Suppose that $G$ is split over $\mathcal{O}_{K}$
and let $(K,\left|-\right|)\rightarrow(L,\left|-\right|)$ be any
extension in $\HV$. Then for any $S\in\mathbf{S}(G)$, there is a
unique map 
\[
\boldsymbol{\alpha}_{S}:\mathbf{B}^{e}(S_{L})\rightarrow\mathbf{B}'(\omega_{S}^{\circ},L)
\]
such that for all $x\in\mathbf{B}^{e}(S_{L})$ and $\mathcal{F}\in\mathbf{F}(S_{L})$,
\[
\boldsymbol{\alpha}_{S}(\circ_{G}^{e})=\alpha_{G,L}\quad\mbox{and}\quad\boldsymbol{\alpha}_{S}(x+\mathcal{F})=\boldsymbol{\alpha}_{S}(x)+\mathcal{F}
\]
Moreover, it extends uniquely to a $G(L)$-equivariant map 
\[
\boldsymbol{\alpha}_{S}:\mathbf{B}^{e}(G_{L})\rightarrow\mathbf{B}'(\omega_{G}^{\circ},L).
\]
\end{lem}
\begin{proof}
The uniqueness of both maps is obvious. Since $\mathbf{B}^{e}(S_{L})$
and $\mathbf{B}(\omega_{S}^{\circ},L)$ are affine $\mathbf{G}(S_{L})$-spaces
on which $S(L)$ acts by translations through the same morphism $\nu_{\mathbf{B},S}:S(L)\rightarrow\mathbf{G}(S_{L})$,
the unique isomorphism of affine $\mathbf{G}(S_{L})$-spaces 
\[
\boldsymbol{\alpha}_{S}:\mathbf{B}^{e}(S_{L})\rightarrow\mathbf{B}(\omega_{S}^{\circ},L)
\]
mapping $\circ_{G}^{e}\in\mathbf{B}^{e}(S_{L})$ to $\alpha_{G,L}\in\mathbf{B}(\omega_{S}^{\circ},L)$
is $S(L)$-equivariant. Since $G(\mathcal{O}_{L})$ fixes $\circ_{G}^{e}\in\mathbf{B}^{e}(G_{L})$
and $\alpha_{G,L}\in\mathbf{B}'(\omega_{G}^{\circ},L)$, the induced
embedding 
\[
\boldsymbol{\alpha}_{S}:\mathbf{B}^{e}(S_{L})\rightarrow\mathbf{B}'(\omega_{G}^{\circ},L)
\]
is also equivariant for the actions of $N_{G}(S)(L)=N_{G}(S)(\mathcal{O}_{L})\cdot S(L)$.
To extend the latter map to a $G(L)$-equivariant morphism on the
whole tight building $\mathbf{B}^{e}(G_{L})$, it remains to establish
the following claim -- see remark~\ref{Rk:ReconstBuildFromAppt}: 
\begin{quote}
For every $x\in\mathbf{B}^{e}(S_{L})$, the $G(L)$-stabilizer of
$x\in\mathbf{B}^{e}(G_{L})$ is contained in the $G(L)$-stabilizer
of $\boldsymbol{\alpha}_{S}(x)\in\mathbf{B}'(\omega_{G}^{\circ},L)$. 
\end{quote}
This is true for $x=\circ_{G}^{e}$, where both stabilizers equal
$G(\mathcal{O}_{L})$. This is therefore also true for any $x$ in
$S(L)\cdot\circ_{G}^{e}=\circ_{G}^{e}+\nu_{\mathbf{B},S}(S(L))$ since
$\boldsymbol{\alpha}_{S}$ is $S(L)$-equivariant. To clarify the
proof, note that the base change maps from $K$ to $L$ identify 
\[
\begin{array}{rclcrcl}
F & = & \mathbf{F}(S_{K}) & \mbox{with} & \mathbf{F}(S_{L}) & \subset & \mathbf{F}(G_{L})\\
A & = & \mathbf{B}^{e}(S_{K}) & \mbox{with} & \mathbf{B}^{e}(S_{L}) & \subset & \mathbf{B}^{e}(G_{L})\\
B & = & \mathbf{B}(\omega_{S}^{\circ},K) & \mbox{with} & \mathbf{B}(\omega_{S}^{\circ},L) & \subset & \mathbf{B}'(\omega_{G}^{\circ},L)
\end{array}
\]
and the isomorphism of affine $F$-space $\boldsymbol{\alpha}_{S}:A\rightarrow B$
also does not depend upon $L$. What depends upon $L$ is the subset
$\Lambda(L)=\circ_{G}^{e}+\nu_{\mathbf{B},S}(\mathbf{S}(L))$ of $A$
on which we know the validity of our claim. So let us fix $x$ and
$\alpha=\boldsymbol{\alpha}_{S}(x)$ as above, as well as some $g\in G(L)$
such that $gx=x$. By lemma~\ref{lem:ExistGoodExtensionValued} below,
there is an extension $(L,\left|-\right|)\rightarrow(L',\left|-\right|)$
in $\HV$ such that $\log\left|L^{\prime\times}\right|=\mathbb{R}$.
Then $\Lambda(L')=A$, thus $g\alpha=\alpha$ in $\mathbf{B}'(\omega_{G}^{\circ},L')$
since $gx=x$ in $\mathbf{B}^{e}(G_{L'})$. But $\mathbf{B}'(\omega_{G}^{\circ},L)\rightarrow\mathbf{B}'(\omega_{G}^{\circ},L')$
is injective and $G(L)$-equivariant, thus also $g\alpha=\alpha$
in $\mathbf{B}'(\omega_{G}^{\circ},L)$, which proves our claim.\end{proof}
\begin{lem}
\label{lem:ExistGoodExtensionValued}Let $L$ be a field with a non-archimedean
absolute value $\left|-\right|$. There is an extension $(L',\left|-\right|)$
of $(L,\left|-\right|)$ with $L'$ algebraically closed and $\log\left|L^{\prime\times}\right|=\mathbb{R}$.\end{lem}
\begin{proof}
By \cite[VI, \S 8, Proposition 9]{BoAC56}, we may assume that $L$
is algebraically closed. Then $\log\left|L^{\times}\right|$ is a
divisible subgroup of $\mathbb{R}$, i.e. a $\mathbb{Q}$-vector space.
Let $(\delta_{i})_{i\in I}$ be a $\mathbb{Q}$-basis of $\mathbb{R}/\log\left|L^{\times}\right|$
and lift each $\delta_{i}$ to $d_{i}\in\mathbb{R}$. Let $(t_{i})_{i\in I}$
be independent variables and let $L'$ be an algebraic closure of
the purely transcendental extension $M=K((t_{i})_{i\in I})$ of $K$.
By~Zorn's lemma and \cite[VI, \S 10, Proposition 1]{BoAC56}, there
is a unique extension of $\left|-\right|$ to a non-archimedean absolute
value on $M$ such that $\log\left|t_{i}\right|=d_{i}$ for every
$i\in I$. The latter again extends to $L'$, and then $\log\left|L^{\prime\times}\right|$
equals $\mathbb{R}$, being a divisible subgroup of $\mathbb{R}$
which contains $\log\left|L^{\times}\right|$ and all $d_{i}$'s.
\end{proof}

\subsection{~}

The theorem implies various properties of $\mathbf{B}(\omega_{G}^{\circ},K)$,
for instance: $\mathbf{B}(\omega_{G}^{\circ},K)$ is a tight affine
$\mathbf{F}(G_{K})$-building. For an extension $(K,\left|-\right|)\rightarrow(L,\left|-\right|)$,
the map $\mathbf{B}(\omega_{G_{\mathcal{O}_{L}}}^{\circ},L)\rightarrow\mathbf{B}(\omega_{G}^{\circ},L)$
is an isomorphism of affine $\mathbf{F}(G_{L})$-buildings. For a
closed immersion $G_{1}\hookrightarrow G_{2}$, the map $\mathbf{B}(\omega_{G_{1}}^{\circ},K)\rightarrow\mathbf{B}(\omega_{G_{2}}^{\circ},K)$
is injective. For a central isogeny $G_{1}\twoheadrightarrow G_{2}$,
the map $\mathbf{B}(\omega_{G_{1}}^{\circ},K)\rightarrow\mathbf{B}(\omega_{G_{2}}^{\circ},K)$
is an isomorphism. Thus $\mathbf{B}(\omega_{G}^{\circ},K)$ has canonical
decompositions analogous to those of section~\ref{sub:Isogenies}.
This last property also follows from~\ref{sub:decer4affbuil}.

\subsection{~}

Fix a faithful representation $\tau$ in $\Rep^{\circ}(G)(\mathcal{O}_{K})$
and drop it from the notations for the induced distances, angles,
scalar products\ldots{} For $x,y\in\mathbf{B}^{e}(G,K)$, 
\[
d(x,y)=d\left(\boldsymbol{\alpha}(x),\boldsymbol{\alpha}(y)\right)=d\left(\boldsymbol{\alpha}(x)(\tau),\boldsymbol{\alpha}(y)(\tau)\right)
\]
where the last distance is computed in the space of $K$-norms on
$V_{K}(\tau)$.

\subsection{~}

Fix $\mathcal{F}_{1},\mathcal{F}_{2}\in\mathbf{F}(G_{K})$. Suppose
that for some $\epsilon>0$, 
\[
\forall t\in[0,\epsilon]:\qquad\alpha_{G}+t\mathcal{F}_{1}=\alpha_{G}+t\mathcal{F}_{2}\quad\mbox{in}\quad\mathbf{B}(\omega_{G}^{\circ},K).
\]
Then the reductions $\mathcal{F}_{1,k}$ and $\mathcal{F}_{2,k}$
are equal in $\mathbf{F}(G_{k})$ by lemma~\ref{lem:loccan}. Suppose
conversely that $\mathcal{F}_{1,k}=\mathcal{F}_{2,k}$, and choose
an apartment $\mathbf{B}^{e}(S)$ in $\mathbf{B}^{e}(G_{K})$ containing
the germs of $t\mapsto\circ_{G}^{e}+t\mathcal{F}_{i}$ for $i\in\{1,2\}$
-- in particular, $S$ belongs to $\mathbf{S}(G)$ since $\circ_{G}^{e}$
belongs to $\mathbf{B}^{e}(S)$. Then there are unique $\mathcal{F}_{i}^{\ast}$
in $\mathbf{F}(S)$ such that, for some $\epsilon>0$, $\circ_{G}^{e}+t\mathcal{F}_{i}=\circ_{G}^{e}+t\mathcal{F}_{i}^{\ast}$
in $\mathbf{B}^{e}(G_{K})$ for all $t\in[0,\epsilon]$. But then
also $\alpha_{G}+t\mathcal{F}_{i}=\alpha_{G}+t\mathcal{F}_{i}^{\ast}$
in $\mathbf{B}(\omega_{G}^{\circ},K)$, thus $\mathcal{F}_{i,k}=\mathcal{F}_{i,k}^{\ast}$
in $\mathbf{F}(G_{k})$, therefore $\mathcal{F}_{1,k}^{\ast}=\mathcal{F}_{2,k}^{\ast}$
and $\mathcal{F}_{1}^{\ast}=\mathcal{F}_{2}^{\ast}$ since the reduction
map is injective on $\mathbf{F}(S)$, thus again $\alpha_{G}+t\mathcal{F}_{1}=\alpha_{G}+t\mathcal{F}_{2}$
for all $t\in[0,\epsilon]$. This yields canonical identifications
\[
\xymatrix{ & \mathbf{F}(G_{K})\ar@{->>}@(l,u)[dl]_{\loc_{\circ_{G}^{e}}}\ar@{->>}[d]^{\loc_{\alpha_{G}}}\ar@{->>}@(r,u)[dr]^{\mathrm{red}}\\
\mathbf{T}_{\circ_{G}^{e}}\mathbf{B}^{e}(G_{K})\ar[r]\sp(0.45){\simeq}\ar@(d,d)[rr]_{\kappa}^{\simeq} & \mathbf{T}_{\alpha_{G}}\mathbf{B}(\omega_{G}^{\circ},K)\ar[r]\sp(0.58){\simeq} & \mathbf{F}(G_{k})
}
\]
between the localization maps of \ref{sub:TangentSpaceInBuild} and
the reduction map on $\mathbf{F}(G_{K})$. By restriction to an apartment
$\mathbf{F}(S_{K})$ with $S\in\mathbf{S}(G)$, one checks that the
isomorphism 
\[
\kappa:\mathbf{T}_{\circ_{G}^{e}}\mathbf{B}^{e}(G_{K})\stackrel{\simeq}{\longrightarrow}\mathbf{F}(G_{k})
\]
is compatible with the distances, scalar products etc\ldots{} attached
to our chosen $\tau$ as in \ref{sub:TangentSpaceInBuild} and \ref{sub:specialisationOfScalarProd},
and also that $\kappa$ fits in a commutative diagram 
\[
\xymatrix{\mathbf{B}^{e}(G,K)\ar@{->>}[d]_{\loc_{\circ_{G}^{e}}^{a}}\ar[r]^{\boldsymbol{\alpha}} & \mathbf{B}(\omega_{G}^{\circ},K)\ar@{->>}[d]^{\loc}\\
\mathbf{T}_{\circ_{G}^{e}}\mathbf{B}^{e}(G_{K})\ar[r]^{\kappa} & \mathbf{F}(G_{k})
}
\]
Thus for every $\mathcal{F},\mathcal{G}\in\mathbf{F}(G)$ and $x,y\in\mathbf{B}^{e}(G_{K})$,
\[
\measuredangle_{\circ}(\mathcal{F},\mathcal{G})=\measuredangle(\mathcal{F}_{k},\mathcal{G}_{k})\quad\mbox{and}\quad\measuredangle_{\circ}(x,y)=\measuredangle(\loc\circ\boldsymbol{\alpha}(x),\loc\circ\boldsymbol{\alpha}(y))
\]
where we have abbreviated $\circ_{G}^{e}=\circ$. In particular, 
\begin{eqnarray*}
\lim_{t\rightarrow0}{\textstyle \frac{1}{t}}d(\circ+t\mathcal{F},\circ+t\mathcal{G}) & = & d(\mathcal{F}_{k},\mathcal{G}_{k})\\
\lim_{t\rightarrow0}{\textstyle \frac{1}{t}}\left(d(x,\circ+t\mathcal{F})-d(x,\circ)\right) & = & \left\langle \loc\circ\boldsymbol{\alpha}(x),\mathcal{F}_{k}\right\rangle 
\end{eqnarray*}
As for the vector valued distance $\mathbf{d}:\mathbf{B}^{e}(G_{K})\times\mathbf{B}^{e}(G_{K})\rightarrow\mathbf{C}(G_{K})$,
we have
\[
\mathbf{d}(\circ,x)=t(\loc\circ\boldsymbol{\alpha}(x))\quad\mbox{in}\quad\mathbf{C}(G_{K})=\mathbf{C}(G_{k}).
\]

\subsection{~}

For a parabolic subgroup $P$ of $G$ with unipotent radical $U$,
the $\Gr_{P}$-map of section~\ref{sub:Gr_PforBTBuildings} induces
an analogous $P(K)$-equivariant map\nomenclature[Gr_P(alpha)]{$\Gr _P (\alpha)$}{$K$-norm on $\omega ^\circ _{P/U}$ induced by a $K$-norm $\alpha$ on $\omega ^\circ _G$, page \nomrefpage}
\[
\Gr_{P}:\mathbf{B}(\omega_{G}^{\circ},K)\rightarrow\mathbf{B}(\omega_{P/U}^{\circ},K).
\]
For $\mathcal{F}\in\mathbf{F}(G)$, set $\Gr_{\mathcal{F}}=\Gr_{P_{\mathcal{F}}}$.
For $\rho\in\Rep^{\circ}(G)(\mathcal{O}_{K})$, $\gamma\in\mathbb{R}$
and $\alpha\in\mathbf{B}'(\omega_{G}^{\circ},K)$, let $\Gr_{\mathcal{F}}^{\gamma}(\alpha,\rho)$
be the $K$-norm on $\Gr_{\mathcal{F}}^{\gamma}(\rho)_{K}$ induced
by $\alpha(\rho)$ on $V_{K}(\rho)$, i.e.
\[
\Gr_{\mathcal{F}}^{\gamma}(\alpha,\rho)(\overline{x})=\inf\left\{ \alpha(\rho)(x):x\in\mathcal{F}_{K}^{\gamma}(\rho),\, x\equiv\overline{x}\bmod\mathcal{F}_{+,K}^{\gamma}(\rho)\right\} 
\]
for every $\overline{x}$ in $\Gr_{\mathcal{F}}^{\gamma}(\rho)_{K}=\mathcal{F}_{K}^{\gamma}(\rho)/\mathcal{F}_{+,K}^{\gamma}(\rho)$.
By the axiom $L(s)$ for the $\mathbf{F}(V_{K}(\rho))$-building $\mathbf{B}(V_{K}(\rho))$
of splittable $K$-norms on $V_{K}(\rho)$, $\Gr_{\mathcal{F}}^{\gamma}(\alpha,\rho)$
is a splittable $K$-norm on $\Gr_{\mathcal{F}}^{\gamma}(\rho)_{K}$.
Viewing $\Gr_{\mathcal{F}}^{\gamma}(\rho)$ as a representation of
$P_{\mathcal{F}}/U_{\mathcal{F}}$, we have:
\[
\forall\alpha\in\mathbf{B}(\omega_{G}^{\circ},K):\qquad\Gr_{\mathcal{F}}(\alpha)\left(\Gr_{\mathcal{F}}^{\gamma}(\rho)\right)=\Gr_{\mathcal{F}}^{\gamma}(\alpha,\rho).
\]
Indeed, both sides only depend upon the $U_{\mathcal{F}}(K)$-orbit
of $\alpha$, and we may thus assume that $\alpha$ belongs to the
image of $\mathbf{B}(\omega_{L}^{\circ},K)\rightarrow\mathbf{B}(\omega_{G}^{\circ},K)$
for some fixed Levi subgroup $L$ of $P_{\mathcal{F}}$, i.e.~$\alpha=\alpha_{G}+\mathcal{H}$
for some $\mathcal{H}\in\mathbf{F}(L)$. Then $r_{P_{\mathcal{F}},L}(\alpha)=\alpha$,
thus $\Gr_{\mathcal{F}}(\alpha)=\alpha_{P/U}+\overline{\mathcal{H}}$
where $\overline{\mathcal{H}}$ is the image of $\mathcal{H}$ in
$\mathbf{F}(P_{\mathcal{F}}/U_{\mathcal{F}})$. On the other hand,
the chosen $L$ gives a splitting $\mathcal{G}\in\mathbf{G}(Z(L))$
of $\mathcal{F}$, thus also a splitting $\rho\vert_{L}=\oplus_{\gamma}\rho_{\gamma}$
with $\rho_{\gamma}\in\Rep^{\circ}(L)(\mathcal{O})$, $V(\rho_{\gamma})=\mathcal{G}_{\gamma}(\rho)$.
Since $\alpha$ is the image of $\alpha_{L}+\mathcal{H}$ in $\mathbf{B}^{e}(\omega_{G}^{\circ},K)$,
this splitting is adapted to $\alpha$: $\alpha(\rho)=\oplus\alpha_{\gamma}(\rho)$
where $\alpha_{\gamma}(\rho)=\alpha_{L}(\rho_{\gamma})+\mathcal{H}(\rho_{\gamma})$.
Thus $\Gr_{\mathcal{F}}^{\gamma}(\alpha,\rho)\simeq\alpha_{\gamma}(\rho)$
under the $K$-linear isomorphism $V_{K}(\rho_{\gamma})\simeq\Gr_{\mathcal{F}}^{\gamma}(\rho)_{K}$
induced by the $(L\rightarrow P_{\mathcal{F}}/U_{\mathcal{F}})$-equivariant
isomorphism $\rho_{\gamma}\simeq\Gr_{\mathcal{F}}^{\gamma}(\rho)$.
It follows that indeed 
\[
\Gr_{\mathcal{F}}^{\gamma}(\alpha,\rho)=\alpha_{P/U}(\Gr_{\mathcal{F}}^{\gamma}(\rho))+\overline{\mathcal{H}}(\Gr_{\mathcal{F}}^{\gamma}(\rho))=\Gr_{\mathcal{F}}(\alpha)(\Gr_{\mathcal{F}}^{\gamma}(\rho)).
\]
For the distances attached to our chosen $\tau$, we thus obtain
\[
\lim_{t\rightarrow\infty}d_{\tau}(x+t\mathcal{F},y+t\mathcal{F})=d\left(\Gr_{\mathcal{F}}^{\bullet}(\boldsymbol{\alpha}(x),\tau),\Gr_{\mathcal{F}}^{\bullet}(\boldsymbol{\alpha}(y),\tau)\right)
\]
for every $x,y\in\mathbf{B}^{e}(G_{K})$, $\mathcal{F}\in\mathbf{F}(G)$.

\subsection{~}

Combining the previous two computations, we also obtain a formula
for the Busemann scalar product on $\mathbf{B}^{e}(G_{K})$. Recall
from section~\ref{sub:Gr_PforBTBuildings} (and \ref{sub:ComputingBusemanAtInfinity})
that for any $x,y\in\mathbf{B}^{e}(G_{K})$ and $\mathcal{F}\in\mathbf{F}(G)$,
we have 
\[
\left\langle \overrightarrow{xy},\mathcal{F}\right\rangle =\left\langle \overrightarrow{\Gr_{\mathcal{F}}(x)\Gr_{\mathcal{F}}(y)},\overline{\mathcal{F}}\right\rangle =\left\langle \loc_{\Gr_{\mathcal{F}}(x)}^{a}\left(\Gr_{\mathcal{F}}(y)\right),\loc_{\Gr_{\mathcal{F}}(x)}\left(\overline{\mathcal{F}}\right)\right\rangle 
\]
where the second and third scalar product are respectively the Busemann
scalar product on $\mathbf{B}^{e}(P_{K}/U_{K})$ and the scalar product
on its tangent space at $\Gr_{\mathcal{F}}(x)$, with $(P,U)=(P_{\mathcal{F}},U_{\mathcal{F}})$.
For $x=\circ_{G}^{e}$, $\Gr_{\mathcal{F}}(x)=\circ_{P/U}^{e}$ and
we thus obtain 
\[
\left\langle \overrightarrow{xy},\mathcal{F}\right\rangle =\left\langle \loc\left(\Gr_{\mathcal{F}}(\boldsymbol{\alpha}(y))\right),\overline{\mathcal{F}}_{k}\right\rangle 
\]
with the scalar product of $\mathbf{F}(P_{k}/U_{k})$ attached to
the faithful representation 
\[
\Gr_{\mathcal{F}}^{\bullet}(\tau)=\oplus_{\gamma}\Gr_{\mathcal{F}}^{\gamma}(\tau)
\]
of $P/U$. Since $\overline{\mathcal{F}}(\Gr_{\mathcal{F}}^{\gamma}(\tau))$
is the $\mathbb{R}$-filtration with a single jump at $\gamma$, 
\[
\left\langle \overrightarrow{xy},\mathcal{F}\right\rangle ={\textstyle \sum_{\gamma}}\gamma\cdot\deg\left(\loc\left(\Gr_{\mathcal{F}}(\boldsymbol{\alpha}(y))\right)\left(\Gr_{\mathcal{F}}^{\gamma}(\tau)\right)\right).
\]
By definition of the morphism $\loc:\mathbf{B}(\omega_{G}^{\circ},K)\rightarrow\mathbf{F}(G_{k})$,
\[
\loc\left(\Gr_{\mathcal{F}}(\boldsymbol{\alpha}(y))\right)\left(\Gr_{\mathcal{F}}^{\gamma}(\tau)\right)=\loc\left(\Gr_{\mathcal{F}}^{\gamma}(\tau)_{K},\Gr_{\mathcal{F}}^{\gamma}(\boldsymbol{\alpha}(y),\tau),\Gr_{\mathcal{F}}^{\gamma}(\tau)\right).
\]
The degree of this filtration is the degree of its determinant. Since
the functors 
\[
\loc:\Norm'(K)\rightarrow\Fil(k)\quad\mbox{and}\quad\Gr_{\mathcal{F}}^{\bullet}\left(\boldsymbol{\alpha}(y)\right)':\Rep^{\circ}(P/U)(\mathcal{O}_{K})\rightarrow\Norm'(K)
\]
are exact $\otimes$-functors, they both commute with the determinant.
The degrees which occur in the last displayed formula for $\left\langle \overrightarrow{xy},\mathcal{F}\right\rangle $
are therefore given by 
\[
\deg\left(\loc\left(\Lambda_{\mathcal{F}}^{\gamma}(\tau)_{K},\Lambda_{\mathcal{F}}^{\gamma}(\boldsymbol{\alpha}(y),\tau),\Lambda_{\mathcal{F}}^{\gamma}(\tau)\right)\right)
\]
where $\Lambda_{\mathcal{F}}^{\gamma}(\tau)=\det\left(\Gr_{\mathcal{F}}^{\gamma}(\tau)\right)$
is a rank one representation of $P/U$ and 
\[
\Lambda_{\mathcal{F}}^{\gamma}(\boldsymbol{\alpha}(y),\tau)=\det\left(\Gr_{\mathcal{F}}^{\gamma}(\boldsymbol{\alpha}(y),\tau)\right)=\Gr_{\mathcal{F}}^{\bullet}\left(\boldsymbol{\alpha}(y)\right)\left(\Lambda_{\mathcal{F}}^{\gamma}(\tau)\right)
\]
is a $K$-norm on $\Lambda_{\mathcal{F}}^{\gamma}(\tau)_{K}$. For
a rank one object $(V,\alpha,L)$ in $\Norm'(K)$, the degree of $\loc(V,\alpha,L)$
is simply the largest $\gamma\in\mathbb{R}$ such that $L\subset\overline{B}(\alpha,\gamma)$.
Equivalently, 
\[
\deg\left(\loc(V,\alpha,L)\right)=-\log\left(\sup\left\{ \alpha(\ell):\ell\in L\right\} \right)=-\log\left(\alpha(\ell_{0})\right)
\]
where $L=\mathcal{O}_{K}\cdot\ell_{0}$. Thus, still assuming that
$x=\circ_{G}^{e}$, we finally obtain
\begin{eqnarray*}
\left\langle \overrightarrow{xy},\mathcal{F}\right\rangle  & = & -{\textstyle \sum_{\gamma}}\gamma\cdot\log\left(\sup\left\{ \Lambda_{\mathcal{F}}^{\gamma}\left(\boldsymbol{\alpha}(y),\tau\right)\vert\Lambda_{\mathcal{F}}^{\gamma}(\tau)\right\} \right)\\
 & = & -{\textstyle \sum_{\gamma}}\gamma\cdot\log\left(\Lambda_{\mathcal{F}}^{\gamma}\left(\boldsymbol{\alpha}(y),\tau\right)(e_{1}^{\gamma}\wedge\cdots\wedge e_{r_{\gamma}}^{\gamma})\right)
\end{eqnarray*}
where $(e_{1}^{\gamma},\cdots,e_{r_{\gamma}}^{\gamma})$ is an $\mathcal{O}_{K}$-basis
of $\Gr_{\mathcal{F}}^{\gamma}(\tau)$. For a general $x$ in $\mathbf{B}^{e}(G_{K})$,
we find:
\[
\left\langle \overrightarrow{xy},\mathcal{F}\right\rangle ={\textstyle \sum_{\gamma}}\gamma\cdot\log\left(\frac{\Lambda_{\mathcal{F}}^{\gamma}\left(\boldsymbol{\alpha}(x),\tau\right)}{\Lambda_{\mathcal{F}}^{\gamma}\left(\boldsymbol{\alpha}(y),\tau\right)}\left(e_{1}^{\gamma}\wedge\cdots\wedge e_{r_{\gamma}}^{\gamma}\right)\right).
\]
Note that if we are given some $\mathcal{G}\in\mathbf{F}(P/U)$ with
$\Gr_{\mathcal{F}}(y)=\Gr_{\mathcal{F}}(x)+\mathcal{G}$, then simply
\[
\left\langle \overrightarrow{xy},\mathcal{F}\right\rangle =\left\langle \mathcal{G},\overline{\mathcal{F}}\right\rangle ={\textstyle \sum_{\gamma}}\gamma\cdot\deg\left(\mathcal{G}\left(\Gr_{\mathcal{F}}^{\gamma}(\tau)\right)\right).
\]

\subsection{~}

For every $\nu>0$, there is a $G(K)$-equivariant commutative diagram
\[
\begin{array}{ccccc}
\mathbf{B}(\omega_{G}^{\circ},K,\left|-\right|) & \times & \mathbf{F}(G_{K}) & \stackrel{+}{\longrightarrow} & \mathbf{B}(\omega_{G}^{\circ},K,\left|-\right|)\\
a\downarrow &  & b\downarrow &  & a\downarrow\\
\mathbf{B}(\omega_{G}^{\circ},K,\left|-\right|^{\nu}) & \times & \mathbf{F}(G_{K}) & \stackrel{+}{\longrightarrow} & \mathbf{B}(\omega_{G}^{\circ},K,\left|-\right|^{\nu})
\end{array}
\]
where $a(\alpha)=\alpha^{\nu}$ and $b(\mathcal{F})=\nu\mathcal{F}$.
It is compatible with the analogous diagram of section~\ref{sub:changevalforBT}
via the relevant $\boldsymbol{\alpha}$-maps.

\subsection{~}

For $x\in\mathbf{B}^{e}(G_{K})$, the $K$-norm $\boldsymbol{\alpha}(x)\in\mathbf{B}(\omega_{G}^{\circ},K)$
is exact and extends to a $K$-norm on $\omega'_{G}$ as in \ref{sub:extF}.
Thus by proposition~\ref{Pro:Wedhorn}, it yields a $K$-norm $\boldsymbol{\alpha}(x)(\rho)$
on $V_{K}(\rho)$ for any representation $\rho$ of $G$ on a flat
$\mathcal{O}_{K}$-module $V(\rho)$. We set 
\[
\boldsymbol{\alpha}_{\mathrm{ad}}(x)=\boldsymbol{\alpha}(x)(\rho_{\mathrm{ad}}),\quad\boldsymbol{\alpha}_{\mathrm{reg}}(x)=\boldsymbol{\alpha}(x)(\rho_{\mathrm{reg}})\quad\mbox{and}\quad\boldsymbol{\alpha}_{\mathrm{adj}}(x)=\boldsymbol{\alpha}(x)(\rho_{\mathrm{adj}}).
\]

\begin{prop}
Suppose that $(K,\left|-\right|)$ is discrete, say $\left|K^{\times}\right|=q^{\mathbb{Z}}$
with $q>1$. Let $(\mathfrak{g}_{x,r})_{r\in\mathbb{R}}$ be the Moy-Prasad
filtration attached to $x$ on $\mathfrak{g}_{K}=\Lie(G_{K})$. Then
\[
\forall x\in\mathbb{R}:\qquad\mathfrak{g}_{x,r}=\left\{ v\in\mathfrak{g}_{K}:\boldsymbol{\alpha}_{\mathrm{ad}}(x)(v)\leq q^{-r}\right\} .
\]
\end{prop}
\begin{proof}
Given the definition of $\mathfrak{g}_{x,r}$ (by étale descent from
the quasi-split case) and proposition~\ref{prop:ConstrCanPtBT},
we may assume that $G$ splits over $\mathcal{O}_{K}$. Changing $\left|-\right|$
to $\left|-\right|^{\nu}$ with $\nu=\frac{1}{\log q}$, we may also
assume that $q=e$. Fix $S\in\mathbf{S}(G)$ with $x$ in $\mathbf{B}^{e}(S_{K})$
and write $x=\circ_{G}^{e}+\mathcal{F}$ for some $\mathcal{F}\in\mathbf{F}(S_{K})$,
so that also $\boldsymbol{\alpha}(x)=\alpha_{G}+\mathcal{F}$. Let
$\mathfrak{g}=\mathfrak{g}_{0}\oplus\oplus_{\beta\in\Phi(G,S)}\mathfrak{g}_{\beta}$
be the weight decomposition of $\mathfrak{g}$ and $\mathcal{F^{\sharp}}:M\rightarrow\mathbb{R}$
the morphism corresponding to $\mathcal{F}$, where $M=\Hom(S,\mathbb{G}_{m,\mathcal{O}_{K}})$.
Then for every $r\in\mathbb{R}$, 
\[
\overline{B}\left(\boldsymbol{\alpha}_{\mathrm{ad}}(x),r\right)=\mathfrak{g}_{0,r}\oplus\oplus_{\beta\in\Phi(G,S)}\mathfrak{g}_{\beta,r}
\]
where $\mathfrak{g}_{\beta,r}=\overline{B}\left(\alpha_{\mathfrak{g}_{\beta}},r-\mathcal{F}^{\sharp}(\beta)\right)$
for $\beta\in\Phi(G,S)\cup\{0\}$. For $r=0$, this is the Lie algebra
$\mathfrak{g}_{x}$ of the group scheme $\mathfrak{G}_{x}$ over $\mathcal{O}_{K}$
attached to $x$ in \cite{BrTi84}. Comparing now this formula with
the definition of $\mathfrak{g}_{x,r}$ in \cite[2.1.3]{AdDeB02}
proves our claim.
\end{proof}
\noindent Let $G_{K}^{\mathrm{an}}$ be the analytic Berkovich space
attached to $G_{K}$. In \cite[2.2]{ReThWe10}, the authors construct
a canonical map $\vartheta:\mathbf{B}^{e}(G_{K})\rightarrow G_{K}^{\mathrm{an}}$,
thus attaching to every $x\in\mathbf{B}^{e}(G_{K})$ a multiplicative
$K$-semi-norm $\vartheta(x)$ on $\mathcal{A}(G_{K})$. 
\begin{prop}
For every $x\in\mathbf{B}^{e}(G_{K})$, $\boldsymbol{\alpha}_{\mathrm{adj}}(x)=\vartheta(x)$.
In particular, the $K$-norm $\boldsymbol{\alpha}_{\mathrm{adj}}(x)$
on $\mathcal{A}(G_{K})$ is multiplicative and $\vartheta(x)$ is
a norm.\end{prop}
\begin{proof}
Equip $G_{K}^{\mathrm{an}}$ with the action of $G(K)$ induced by
$\rho_{\mathrm{adj}}$. Then $x\mapsto\vartheta(x)$ is $G(K)$-equivariant
and compatible with extensions $(K,\left|-\right|)\rightarrow(L,\left|-\right|)$
in the sense that $\vartheta(x)=\vartheta(x_{L})\vert\mathcal{A}(G_{K})$
for every $x\in\mathbf{B}^{e}(G_{K})$ \cite[Proposition 2.8]{ReThWe10}.
The map $x\mapsto\boldsymbol{\alpha}_{\mathrm{adj}}(x)$ has the same
properties. We may thus assume that $G$ splits over $\mathcal{O}_{K}$,
and again choosing $L$ with $\log\left|L^{\times}\right|=\mathbb{R}$,
we merely have to show that $\vartheta(\circ_{G}^{e})=\boldsymbol{\alpha}_{\mathrm{adj}}(\circ_{G}^{e})=\alpha_{\mathcal{A}(G)}$.
By definition: $\{\vartheta(x)\}$ is the Shilov boundary of a $K$-affinoid
subgroup $G_{x}$ of $G_{K}^{\mathrm{an}}$. For $x=\circ_{G}^{e}$,
$G_{x}$ is the affinoid group $G^{\mathrm{an}}$ attached to $G$,
and its Shilov boundary is the gauge norm attached to $\mathcal{A}(G)$,
i.e.~$\alpha_{\mathcal{A}(G)}$. 
\end{proof}
\noindent Since the multiplication on $\mathcal{A}(G)$ is a morphism
$\rho_{\mathrm{reg}}\otimes\rho_{\mathrm{reg}}\rightarrow\rho_{\mathrm{reg}}$
in $\Rep'(G)(\mathcal{O}_{K})$, the $K$-norm $\boldsymbol{\alpha}_{\mathrm{reg}}(x)$
on $\mathcal{A}(G_{K})$ is sub-multiplicative. Since for $\tau\in\Rep^{\circ}(G)(\mathcal{O}_{K})$,
the co-module map $V(\tau)\rightarrow V(\tau)\otimes\mathcal{A}(G)$
is a pure monomorphism $\tau\hookrightarrow\tau_{0}\otimes\rho_{\mathrm{reg}}$
in $\Rep'(G)(\mathcal{O}_{K})$, $\boldsymbol{\alpha}(x)(\tau)$ is
the restriction of $\alpha_{V(\tau_{0})}\otimes\boldsymbol{\alpha}_{\mathrm{reg}}(x)$
to $V_{K}(\tau)$, thus $\boldsymbol{\alpha}_{\mathrm{reg}}(x)$ determines
$\boldsymbol{\alpha}(x)$ and $\boldsymbol{\alpha}_{\mathrm{reg}}$
is a $G(K)$-equivariant embedding of $\mathbf{B}^{e}(G_{K})$ into
the space of sub-multiplicative $K$-norms on $\mathcal{A}(G)$ (equipped
with the regular action).

\subsection{~}

Some final remarks:

$(1)$ We have not given an intrinsic characterization of the subset
$\mathbf{B}(\omega_{G}^{\circ},K)$ of $\mathbf{B}'(\omega_{G}^{\circ},K)$.
We expect that $\mathbf{B}(\omega_{G}^{\circ},K)=\mathbf{B}^{?}(\omega_{G}^{\circ},K)$,
or perhaps even that $\mathbf{B}(\omega_{G}^{\circ},K)$ is equal
to the $G(K)$-stable subset of exact norms in $\mathbf{B}'(\omega_{G}^{\circ},K)$. 

$(2)$ Suppose that $\mathcal{O}$ is a valuation ring of height $>1$
with fraction field $K$. Then $\Gamma=K^{\times}/\mathcal{O}^{\times}$
is a totally ordered commutative group which can not be embedded into
$\mathbb{R}$. Let $G$ be a reductive group over $\mathcal{O}$.
Replacing $\mathbb{R}$ with $\Gamma$ in the above constructions,
it might be possible to define a ``Bruhat-Tits'' building $\mathbf{B}(\omega_{G}^{\circ},K)$
with compatible actions of $G(K)$ and $\mathbf{F}^{\Gamma}(G_{K})$,
made of factorizations of the fiber functor $\omega_{G,K}^{\circ}:\Rep^{\circ}(G)(\mathcal{O})\rightarrow\Vect(K)$
through a suitable category of ``$\Gamma$-norms''. The type maps
should be the tautological morphisms $\nu:S(K)\rightarrow\mathbf{G}^{\Gamma}(S)$
mapping $s\in S(K)$ to the unique morphism $\nu(s):\mathbb{D}_{K}(\Gamma)\rightarrow S$
whose composite with a character $\chi$ of $S$ is the image of $\chi(s)$
in $\Gamma=K^{\times}/\mathcal{O}^{\times}$.

$(3)$ There might also be a similar Tannakian formalism for the symmetric
spaces of reductive groups over $\mathbb{R}$, with factorizations
of fiber functors through a category of Euclidean spaces, using compact
forms of the adjoint groups as base point. 

\printnomenclature[2cm]{}

\bibliographystyle{amsplain}
\bibliography{MyBib}

\providecommand{\bysame}{\leavevmode\hbox to3em{\hrulefill}\thinspace}
\providecommand{\MR}{\relax\ifhmode\unskip\space\fi MR }
\providecommand{\MRhref}[2]{%
  \href{http://www.ams.org/mathscinet-getitem?mr=#1}{#2}
}
\providecommand{\href}[2]{#2}
\begin{thebibliography}{10}

\bibitem{SGA3.2}
\emph{{Sch{\'e}mas en groupes. {II}: {G}roupes de type multiplicatif, et
  structure des sch{\'e}mas en groupes g{\'e}n{\'e}raux}}, {S{\'e}minaire de
  G{\'e}om{\'e}trie Alg{\'e}brique du Bois Marie 1962/64 (SGA 3). Dirig{\'e}
  par M. Demazure et A. Grothendieck. Lecture Notes in Mathematics, Vol. 152},
  Springer-Verlag, Berlin, 1970. \MR{0274459 (43 \#223b)}

\bibitem{SGA1r}
\emph{{Rev{\^e}tements {\'e}tales et groupe fondamental ({SGA} 1)}}, {Documents
  Math{\'e}matiques (Paris) [Mathematical Documents (Paris)], 3},
  Soci{\'e}t{\'e} Math{\'e}matique de France, Paris, 2003, S{\'e}minaire de
  g{\'e}om{\'e}trie alg{\'e}brique du Bois Marie 1960--61. [Algebraic Geometry
  Seminar of Bois Marie 1960-61], Directed by A. Grothendieck, With two papers
  by M. Raynaud, Updated and annotated reprint of the 1971 original [Lecture
  Notes in Math., 224, Springer, Berlin; MR0354651 (50 \#7129)]. \MR{2017446
  (2004g:14017)}

\bibitem{AdDeB02}
Jeffrey~D. Adler and Stephen DeBacker, \emph{{Some applications of
  {B}ruhat-{T}its theory to harmonic analysis on the {L}ie algebra of a
  reductive {$p$}-adic group}}, Michigan Math. J. \textbf{50} (2002), no.~2,
  263--286. \MR{1914065 (2003g:22016)}

\bibitem{AtBo83}
M.~F. Atiyah and R.~Bott, \emph{{The {Y}ang-{M}ills equations over {R}iemann
  surfaces}}, Philos. Trans. Roy. Soc. London Ser. A \textbf{308} (1983),
  no.~1505, 523--615. \MR{702806 (85k:14006)}

\bibitem{BoTi65}
Armand Borel and Jacques Tits, \emph{{Groupes r{\'e}ductifs}}, Inst. Hautes
  {\'E}tudes Sci. Publ. Math. (1965), no.~27, 55--150. \MR{0207712 (34 \#7527)}

\bibitem{BoAC56}
N.~Bourbaki, \emph{{{\'E}l{\'e}ments de math{\'e}matique. {F}asc. {XXX}.
  {A}lg{\`e}bre commutative. {C}hapitre 5: {E}ntiers. {C}hapitre 6:
  {V}aluations}}, {Actualit{\'e}s Scientifiques et Industrielles, No. 1308},
  Hermann, Paris, 1964. \MR{0194450 (33 \#2660)}

\bibitem{BoLie46}
\bysame, \emph{{{\'E}l{\'e}ments de math{\'e}matique. {F}asc. {XXXIV}.
  {G}roupes et alg{\`e}bres de {L}ie. {C}hapitre {IV}: {G}roupes de {C}oxeter
  et syst{\`e}mes de {T}its. {C}hapitre {V}: {G}roupes engendr{\'e}s par des
  r{\'e}flexions. {C}hapitre {VI}: syst{\`e}mes de racines}}, {Actualit{\'e}s
  Scientifiques et Industrielles, No. 1337}, Hermann, Paris, 1968. \MR{0240238
  (39 \#1590)}

\bibitem{BrHa99}
M.~R. Bridson and A.~Haefliger, \emph{{Metric spaces of non-positive
  curvature}}, {Grundlehren der Mathematischen Wissenschaften [Fundamental
  Principles of Mathematical Sciences]}, vol. 319, Springer-Verlag, Berlin,
  1999. \MR{1744486 (2000k:53038)}

\bibitem{BrTi72}
F.~Bruhat and J.~Tits, \emph{{Groupes r{\'e}ductifs sur un corps local}}, Inst.
  Hautes {\'E}tudes Sci. Publ. Math. (1972), no.~41, 5--251. \MR{0327923 (48
  \#6265)}

\bibitem{BrTi84}
\bysame, \emph{{Groupes r{\'e}ductifs sur un corps local. {II}. {S}ch{\'e}mas
  en groupes. {E}xistence d'une donn{\'e}e radicielle valu{\'e}e}}, Inst.
  Hautes {\'E}tudes Sci. Publ. Math. (1984), no.~60, 197--376. \MR{756316
  (86c:20042)}

\bibitem{BrTi84b}
\bysame, \emph{{Sch{\'e}mas en groupes et immeubles des groupes classiques sur
  un corps local}}, Bull. Soc. Math. France \textbf{112} (1984), no.~2,
  259--301. \MR{788969 (86i:20064)}

\bibitem{BrTi87}
\bysame, \emph{{Sch{\'e}mas en groupes et immeubles des groupes classiques sur
  un corps local. {II}. {G}roupes unitaires}}, Bull. Soc. Math. France
  \textbf{115} (1987), no.~2, 141--195. \MR{919421}

\bibitem{CoGaPr10}
Brian Conrad, Ofer Gabber, and Gopal Prasad, \emph{{Pseudo-reductive groups}},
  {New Mathematical Monographs}, vol.~17, Cambridge University Press,
  Cambridge, 2010. \MR{2723571 (2011k:20093)}

\bibitem{Co15}
Christophe Cornut, \emph{{Mazur's Inequality and Laffaille's theorem,
  Preprint}},  (2014).

\bibitem{Co13}
\bysame, \emph{{A fixed point theorem in {E}uclidean buildings}}, Adv. Geom.
  \textbf{16} (2016), no.~4, 487--496. \MR{3595184}

\bibitem{CoNi16}
Christophe Cornut and Marc-Hubert Nicole, \emph{{Cristaux et immeubles}}, Bull.
  Soc. Math. France \textbf{144} (2016), no.~1, 125--143. \MR{3481264}

\bibitem{DaOrRa10}
J.-F. Dat, S.~Orlik, and M.~Rapoport, \emph{{Period domains over finite and
  {$p$}-adic fields}}, {Cambridge Tracts in Mathematics}, vol. 183, Cambridge
  University Press, Cambridge, 2010. \MR{2676072 (2012a:22026)}

\bibitem{De90}
P.~Deligne, \emph{{Cat{\'e}gories tannakiennes}}, {The {G}rothendieck
  {F}estschrift, {V}ol.\ {II}}, {Progr. Math.}, vol.~87, Birkh{\"a}user Boston,
  Boston, MA, 1990, pp.~111--195. \MR{1106898 (92d:14002)}

\bibitem{FoRa05}
Jean-Marc Fontaine and Michael Rapoport, \emph{{Existence de filtrations
  admissibles sur des isocristaux}}, Bull. Soc. Math. France \textbf{133}
  (2005), no.~1, 73--86. \MR{2145020 (2005m:14032)}

\bibitem{SGA3.1r}
Philippe Gille and Patrick Polo (eds.), \emph{{Sch{\'e}mas en groupes ({SGA}
  3). {T}ome {I}. {P}ropri{\'e}t{\'e}s g{\'e}n{\'e}rales des sch{\'e}mas en
  groupes}}, {Documents Math{\'e}matiques (Paris) [Mathematical Documents
  (Paris)], 7}, Soci{\'e}t{\'e} Math{\'e}matique de France, Paris, 2011,
  S{\'e}minaire de G{\'e}om{\'e}trie Alg{\'e}brique du Bois Marie 1962--64.
  [Algebraic Geometry Seminar of Bois Marie 1962--64], A seminar directed by M.
  Demazure and A. Grothendieck with the collaboration of M. Artin, J.-E.
  Bertin, P. Gabriel, M. Raynaud and J-P. Serre, Revised and annotated edition
  of the 1970 French original. \MR{2867621}

\bibitem{SGA3.3r}
Philippe Gille and Patrick Polo (eds.), \emph{{Sch{\'e}mas en groupes ({SGA}
  3). {T}ome {III}. {S}tructure des sch{\'e}mas en groupes r{\'e}ductifs}},
  {Documents Math{\'e}matiques (Paris) [Mathematical Documents (Paris)]},
  vol.~8, Soci{\'e}t{\'e} Math{\'e}matique de France, Paris, 2011,
  S{\'e}minaire de G{\'e}om{\'e}trie Alg{\'e}brique du Bois Marie 1962--64.
  [Algebraic Geometry Seminar of Bois Marie 1962--64], A seminar directed by M.
  Demazure and A. Grothendieck with the collaboration of M. Artin, J.-E.
  Bertin, P. Gabriel, M. Raynaud and J-P. Serre, Revised and annotated edition
  of the 1970 French original. \MR{2867622}

\bibitem{GoIw63}
O.~Goldman and N.~Iwahori, \emph{{The space of {$p$}-adic norms}}, Acta Math.
  \textbf{109} (1963), 137--177. \MR{0144889 (26 \#2430)}

\bibitem{EGA1}
A.~Grothendieck, \emph{{{\'E}l{\'e}ments de g{\'e}om{\'e}trie alg{\'e}brique.
  {I}. {L}e langage des sch{\'e}mas}}, Inst. Hautes {\'E}tudes Sci. Publ. Math.
  (1960), no.~4, 228. \MR{MR0163908 (29 \#1207)}

\bibitem{EGA2}
\bysame, \emph{{{\'E}l{\'e}ments de g{\'e}om{\'e}trie alg{\'e}brique. {II}.
  {\'E}tude globale {\'e}l{\'e}mentaire de quelques classes de morphismes}},
  Inst. Hautes {\'E}tudes Sci. Publ. Math. (1961), no.~8, 222. \MR{MR0163909
  (29 \#1208)}

\bibitem{EGA4.1}
\bysame, \emph{{{\'E}l{\'e}ments de g{\'e}om{\'e}trie alg{\'e}brique. {IV}.
  {\'E}tude locale des sch{\'e}mas et des morphismes de sch{\'e}mas. {I}}},
  Inst. Hautes {\'E}tudes Sci. Publ. Math. (1964), no.~20, 259. \MR{MR0173675
  (30 \#3885)}

\bibitem{EGA4.2}
\bysame, \emph{{{\'E}l{\'e}ments de g{\'e}om{\'e}trie alg{\'e}brique. {IV}.
  {\'E}tude locale des sch{\'e}mas et des morphismes de sch{\'e}mas. {II}}},
  Inst. Hautes {\'E}tudes Sci. Publ. Math. (1965), no.~24, 231. \MR{MR0199181
  (33 \#7330)}

\bibitem{EGA4.3}
\bysame, \emph{{{\'E}l{\'e}ments de g{\'e}om{\'e}trie alg{\'e}brique. {IV}.
  {\'E}tude locale des sch{\'e}mas et des morphismes de sch{\'e}mas. {III}}},
  Inst. Hautes {\'E}tudes Sci. Publ. Math. (1966), no.~28, 255. \MR{MR0217086
  (36 \#178)}

\bibitem{EGA4.4}
\bysame, \emph{{{\'E}l{\'e}ments de g{\'e}om{\'e}trie alg{\'e}brique. {IV}.
  {\'E}tude locale des sch{\'e}mas et des morphismes de sch{\'e}mas {IV}}},
  Inst. Hautes {\'E}tudes Sci. Publ. Math. (1967), no.~32, 361. \MR{MR0238860
  (39 \#220)}

\bibitem{KlLe97}
Bruce Kleiner and Bernhard Leeb, \emph{{Rigidity of quasi-isometries for
  symmetric spaces and {E}uclidean buildings}}, Inst. Hautes {\'E}tudes Sci.
  Publ. Math. (1997), no.~86, 115--197 (1998). \MR{1608566 (98m:53068)}

\bibitem{Ko85}
R.~E. Kottwitz, \emph{{Isocrystals with additional structure}}, Compositio
  Math. \textbf{56} (1985), no.~2, 201--220. \MR{809866 (87i:14040)}

\bibitem{Laf80}
Guy Laffaille, \emph{{Groupes {$p$}-divisibles et modules filtr{\'e}s: le cas
  peu ramifi{\'e}}}, Bull. Soc. Math. France \textbf{108} (1980), no.~2,
  187--206. \MR{606088 (82i:14028)}

\bibitem{La00}
E.~Landvogt, \emph{{Some functorial properties of the {B}ruhat-{T}its
  building}}, J. Reine Angew. Math. \textbf{518} (2000), 213--241. \MR{1739403
  (2001g:20029)}

\bibitem{Le42}
F.~W. Levi, \emph{{Ordered groups}}, Proc. Indian Acad. Sci., Sect. A.
  \textbf{16} (1942), 256--263. \MR{0007779 (4,192b)}

\bibitem{Ma89}
Hideyuki Matsumura, \emph{{Commutative ring theory}}, second ed., {Cambridge
  Studies in Advanced Mathematics}, vol.~8, Cambridge University Press,
  Cambridge, 1989, Translated from the Japanese by M. Reid. \MR{1011461
  (90i:13001)}

\bibitem{MoPr94}
Allen Moy and Gopal Prasad, \emph{{Unrefined minimal {$K$}-types for {$p$}-adic
  groups}}, Invent. Math. \textbf{116} (1994), no.~1-3, 393--408. \MR{1253198
  (95f:22023)}

\bibitem{Pa10}
Anne Parreau, \emph{{La distance vectorielle dans les immeubles affines et les
  espaces sym{\'e}triques}}.

\bibitem{Pa99}
\bysame, \emph{{Immeubles affines : construction par les normes et {\'e}tude
  des isom{\'e}tries.}}, {In Crystallographic groups and their generalizations
  (Kortrijk, 1999)}, {Contemp. Math.}, vol. 262, Amer. Math. Soc., Providence,
  RI, 2000, pp.~263--302.

\bibitem{ReThWe10}
B.~R{\'e}my, A.~Thuillier, and A.~Werner, \emph{{Bruhat-{T}its theory from
  {B}erkovich's point of view. {I}. {R}ealizations and compactifications of
  buildings}}, Ann. Sci. {\'E}c. Norm. Sup{\'e}r. (4) \textbf{43} (2010),
  no.~3, 461--554. \MR{2667022 (2011j:20075)}

\bibitem{Ro77}
Guy Rousseau, \emph{{Immeubles des groupes r{\'e}ductifs sur les corps
  locaux}}, U.E.R. Math{\'e}matique, Universit{\'e} Paris XI, Orsay, 1977,
  Th{\`e}se de doctorat, Publications Math{\'e}matiques d'Orsay, No. 221-77.68,
  \url{http://www.iecl.univ-lorraine.fr/~Guy.Rousseau/Textes/}. \MR{0491992 (58
  \#11158)}

\bibitem{Ro09}
\bysame, \emph{{Euclidean buildings}}, {G{\'e}om{\'e}tries {\`a} courbure
  n{\'e}gative ou nulle, groupes discrets et rigidit{\'e}s}, {S{\'e}min.
  Congr.}, vol.~18, Soc. Math. France, Paris, 2009, pp.~77--116. \MR{2655310
  (2011m:20072)}

\bibitem{SaRi72}
Neantro {Saavedra Rivano}, \emph{{Cat{\'e}gories {T}annakiennes}}, {Lecture
  Notes in Mathematics, Vol. 265}, Springer-Verlag, Berlin, 1972. \MR{0338002
  (49 \#2769)}

\bibitem{Sc12}
Daniel Sch{\"a}ppi, \emph{{A characterization of categories of coherent sheaves
  of certain algebraic stacks, Preprint}},  (2012).

\bibitem{Sc13}
\bysame, \emph{{The formal theory of {T}annaka duality}}, Ast{\'e}risque
  (2013), no.~357, viii+140. \MR{3185459}

\bibitem{Se68b}
Jean-Pierre Serre, \emph{{Groupes de {G}rothendieck des sch{\'e}mas en groupes
  r{\'e}ductifs d{\'e}ploy{\'e}s}}, Inst. Hautes {\'E}tudes Sci. Publ. Math.
  (1968), no.~34, 37--52. \MR{0231831 (38 \#159)}

\bibitem{St98}
John~R. Stembridge, \emph{{The partial order of dominant weights}}, Adv. Math.
  \textbf{136} (1998), no.~2, 340--364. \MR{1626860 (2000a:06038)}

\bibitem{Ti74}
Jacques Tits, \emph{{Buildings of spherical type and finite {BN}-pairs}},
  {Lecture Notes in Mathematics, Vol. 386}, Springer-Verlag, Berlin-New York,
  1974. \MR{0470099}

\bibitem{Ti79}
\bysame, \emph{{Reductive groups over local fields}}, {Automorphic forms,
  representations and {$L$}-functions ({P}roc. {S}ympos. {P}ure {M}ath.,
  {O}regon {S}tate {U}niv., {C}orvallis, {O}re., 1977), {P}art 1}, {Proc.
  Sympos. Pure Math., XXXIII}, Amer. Math. Soc., Providence, R.I., 1979,
  pp.~29--69. \MR{546588 (80h:20064)}

\bibitem{Vi05}
Angelo Vistoli, \emph{{Grothendieck topologies, fibered categories and descent
  theory}}, {Fundamental algebraic geometry}, {Math. Surveys Monogr.}, vol.
  123, Amer. Math. Soc., Providence, RI, 2005, pp.~1--104. \MR{2223406}

\bibitem{We04}
Torsten Wedhorn, \emph{{On {T}annakian duality over valuation rings}}, J.
  Algebra \textbf{282} (2004), no.~2, 575--609. \MR{2101076 (2005j:18007)}

\bibitem{Wi10}
Jr~Kevin~Michael Wilson, \emph{{A {T}annakian description for parahoric
  {B}ruhat-{T}its group schemes}}, ProQuest LLC, Ann Arbor, MI, 2010, Thesis
  (Ph.D.)--University of Maryland, College Park. \MR{2941465}

\bibitem{Zi12}
Paul Ziegler, \emph{{Graded and filtered fiber functors on {T}annakian
  categories}}, J. Inst. Math. Jussieu \textbf{14} (2015), no.~1, 87--130.
  \MR{3284480}

\end{thebibliography}

\end{document}